\documentclass[11pt]{book}
\usepackage{amsmath,amsthm,amssymb,amscd,fancybox,ifthen,epsfig,subfigure}
\usepackage[all]{xy}
\usepackage{hyperref,makeidx} 

\newcommand\cost{\cos\theta}
\newcommand\et{e_{\theta}}
\newcommand\eit{e^{i\theta}}
\newcommand\ett{e_{2\theta}}

\newcommand\Kimura{\operatorname{Kim}}

\newcommand\ter{\operatorname{ter}}

\newcommand\fU{\mathfrak U}
\newcommand\fW{\mathfrak W}

\newcommand\loc{\operatorname{loc}}
\newcommand\euc{\operatorname{euc}}

\newcommand\Dom{\operatorname{Dom}}

\newcommand\dist{\operatorname{dist}}
\newcommand\sdist{\operatorname{s-dist}}

\newcommand\Int{\operatorname{int}}

\newcommand\Spec{\operatorname{spec}}

\newcommand\cM{\mathcal M}

\newcommand\cO{\mathcal O}

\newcommand\cQ{\mathcal Q}
\newcommand\cK{\mathcal K}

\newcommand\tx{\tilde x}
\newcommand\tw{\tilde w}

\newcommand\ty{\tilde y}

\newcommand\ba{\boldsymbol a}
\newcommand\bb{\boldsymbol b}
\newcommand\BB{\boldsymbol B}
\newcommand\bzero{\boldsymbol 0}
\newcommand\balpha{\boldsymbol \alpha}
\newcommand\bbeta{\boldsymbol \beta}

\newcommand\bxi{\boldsymbol \xi}
\newcommand\bseta{\boldsymbol \eta}

\newcommand\tchi{\widetilde\chi}

\newcommand\tE{\widetilde{E}}

\newcommand\bBeta{\boldsymbol \beta}
\newcommand\bj{\boldsymbol j}

\newcommand\tA{\widetilde A}
\newcommand\tB{\widetilde B}
\newcommand\tC{\widetilde C}

\newcommand\tL{\widetilde L}
\newcommand\tP{\widetilde P}
\newcommand\tR{\widetilde R}

\newcommand\tbb{\boldsymbol {\tilde{b}}}
\newcommand\tbz{\boldsymbol {\tilde{z}}}

\newcommand\bx{\boldsymbol x}
\newcommand\tbx{\widetilde{\boldsymbol x}}
\newcommand\by{\boldsymbol y}
\newcommand\tby{\widetilde{\boldsymbol y}}
\newcommand\br{\boldsymbol r}

\newcommand\bone{\boldsymbol 1}

\newcommand\bX{\boldsymbol X}
\newcommand\bY{\boldsymbol Y}

\newcommand\cF{\mathcal F}

\newcommand\talpha{\widetilde{\alpha}}
\newcommand\tbeta{\widetilde{\beta}}
\newcommand\tgamma{\widetilde{\gamma}}
\newcommand\ta{\widetilde{a}}
\newcommand\tb{\widetilde{b}}
\newcommand\tc{\widetilde{c}}
\newcommand\tg{\widetilde{g}}

\newcommand\tu{\widetilde{u}}
\newcommand\tQ{\widetilde{Q}}
\newcommand\hQ{\widehat{Q}}
\newcommand\hR{\widehat{R}}

\newcommand\RR{\mathbb R}

\newcommand\supone{\sup^1}
\newcommand\bbr[1]{\left[\!\!\left| #1\right|\!\!\right]}

\newcommand\Ker{\operatorname{ker}}

\newcommand\cC{\mathcal{C}}

\newcommand\cS{\mathcal{S}}

\newcommand\cD{\mathcal{D}}

\newcommand\bz{\boldsymbol{z}}
\newcommand\bw{\boldsymbol{w}}

\renewcommand\Re{\operatorname{Re}}

\newcommand\bbC{\mathbb C}

\newcommand\bbN{\mathbb N}
\newcommand\bbP{\mathbb P}
\newcommand\bbR{\mathbb R}
\newcommand\bbZ{\mathbb Z}

\newcommand\pa{\partial}

\newcommand\restrictedto{\upharpoonright}

\newcommand\supp{\operatorname{supp}}
\newcommand\subsubset{\subset\!\subset}

\newcommand\dcC{\dot{\mathcal C}}
\newcommand\CI{{\mathcal C}^{\infty}}
\newcommand\CIc{{\mathcal C}^{\infty}_{\text{c}}}
\newcommand\Id{\operatorname{Id}}
\DeclareMathOperator{\WF}{WF}
\DeclareMathOperator{\BP}{BP}

\newcommand{\NN}{\mathbb{N}}

\newcommand{\e}{\epsilon}
\newcommand{\del}{\partial}

\newcommand{\oL}{\overline{L}}

\newcommand{\calC}{{\mathcal C}}

\newcommand{\calO}{{\mathcal O}}

\newcommand{\calQ}{{\mathcal Q}}

\newcommand{\calU}{{\mathcal U}}
\newcommand{\calV}{{\mathcal V}}

\newcommand\labelbis[1]{{\rm\bf \ref{#1}\,${}^{\prime}$}}

\newtheorem{theorem}{Theorem}[section]
\newtheorem{proposition}{Proposition}[section]
\newtheorem{corollary}{Corollary}[section]
\newtheorem{lemma}{Lemma}[section]
\newtheorem*{lemmabis}{Lemma}

\theoremstyle{definition}
\newtheorem{definition}{Definition}[section]

\theoremstyle{remark}
\newtheorem{remark}{Remark}[section]

\makeindex
\begin{document}

\title{Degenerate Diffusion Operators Arising in Population Biology} 

\author{Charles L. Epstein
and Rafe Mazzeo}
\date{September 15, 2011}

\maketitle

\tableofcontents

\chapter*{Preface}
This \emph{lange megillah} is concerned with establishing properties of a
mathematical model in population genetics that some might regard, in light of
what is being modeled, as entirely obvious.  Once written down, however, a
mathematical model has a life of its own; it must be addressed in its own
terms, and understood without reference to its origins.

The models we consider are phrased as partial differential equations, which
arise as limits of finite Markov chains. The existence of solutions to these
partial differential and their properties are suggested by the physical,
economic, or biological systems under consideration, but logically speaking,
are entirely independent of them. What is far from obvious are the regularity
properties of these solutions, and, as is so often the case, the existence of
solutions actually hinges on these very subtle properties. Using a Schauder
method, we \emph{prove} the existence of solutions to a class of degenerate
parabolic and elliptic equations that arise in population genetics and
mathematical finance.

The archetypes for these equations arise as infinite population limits
of the Wright-Fisher models in population genetics. These describe the
prevalence of a mutant allele, in a population of fixed size, under
the effects of genetic drift, mutation, migration and selection. The
formal generator of the infinite population limit acts on functions
defined on $[0,1]$ (the space of frequencies) and is given by
\begin{equation}
  L_{\WF}=x(1-x)\pa_x^2+[b_0(1-x)-b_1x+sx(1-x)]\pa_x.
\end{equation}
Processes defined by such operators were studied by Feller in the
early 1950s and used to great effect by Kimura, et al. in the 1960s
and 70s to give quantitative answers to a wide range of questions in
population genetics. Notwithstanding, a modern appreciation of the
analytic properties of these equations is only now coming into focus.

In this monograph we provide analytic foundations for equations of this type
and their natural higher dimensional generalizations. We call these operators
generalized Kimura diffusions. They act on functions defined on generalizations
of convex polyhedra, which are called manifolds-with-corners. We provide the
basic H\"older space-type estimates for operators in this class with which we
establish the existence of solutions. These operators satisfy a strong form of
the maximum principle, which implies uniqueness. 

The partial differential operators we consider are degenerate and the
underlying manifolds-with-corners are themselves singular. This inevitably
produces significant technical challenges in the analysis of such equations,
and explains, in part, the length of this text. The Markov processes defined by
these operators provide fundamental models for many Biological and Economic
situations, and it is for this reason that we feel that these operators merit
such a detailed and laborious treatment.

A large portion of this book is devoted to a careful exploration of
operators of the form
\begin{equation}
  L_{\bb,m}=\sum_{i=1}^n[x_i\pa_{x_i}^2+b_i\pa_{x_i}]+\sum_{l=1}^m\pa_{y_l}^2,
\end{equation}
acting on functions defined in $\bbR_+^n\times\bbR^m.$ Here the
coefficients $\{b_i\}$ are non-negative constants. These operators are
interesting in their own right, arising as models in population
genetics, but for us, they are largely building blocks for the
analysis of general Kimura diffusions. These are the analogues, in the
present context, of the ``constant coefficient elliptic operators'' in
classical elliptic theory. A notable feature of this family is that,
because the coefficient of $\pa_{x_i}^2$ vanishes at $x_i=0,$ we need
to retain the first order transverse term. The value of
$\bb=(b_1,\dots,b_n)$ has a pervasive effect on the behavior of the
solution. Much of our analysis rests upon  explicit formul{\ae} for the
solutions of the initial value problems:
\begin{equation}
  \pa_t v-L_{\bb,m}v=0\text{ with }v(\bx,\by,0)=f(\bx,\by).
\end{equation}

Using these models, we have succeeded in developing a rather complete
existence and regularity theory for general Kimura diffusions with
H\"older continuous data. This in turn suffices to prove the existence
of a $c_0$-semi-group acting on continuous functions, and showing that
the associated Martingale problem has a unique solution. The existence
of a strong Markov process, which in applications to genetics,
describes the statistical behavior of individual populations, follows
from this.

In special situations, such results have been established by other
authors, but without either the precise control on the regularity of
solutions, or the generality considered herein. As long as it is, this
text just begins to scratch the surface of this very rich field. We
have restricted our attention to the solutions with the best possible
regularity properties, which leads to considerable
simplifications. For applications it will be important to consider
solutions with more complicated boundary behavior; we hope that
this text will provide a solid foundation for these investigations.

\section*{Acknowledgments}
We would like to acknowledge the generous financial support and
unflagging personal support provided by Ben Mann and the DARPA FunBio
project. It is certainly the case that without Ben's encouragement, we
would never have undertaken this project. We would like to thank our
FunBio colleagues\footnote{Phil Benfey, Michael Deem, Richard Lenski, Jack
Morava, Lior Pachter, Herbert Edelsbrunner, John
Harer, Jim Damon, Peter Bates, Joshua Weitz, Konstantin Mischaikow, Gunnar
Carlsson, Bernd Sturmfels, Tim Buchman, Ary Goldberger, Jonathan
Eisen, Olivier Porquie, Thomas Fink, Ned Wingreen, Jonathan Dushoff, Peter
Nara, \emph{inter alia}.}, and administrative staff
\footnote{Mark Filipkowski, Shauna Koppel, Rachel Scholz, Matthew Clement,
Traci Kiesling, Traci Pals, \emph{inter alia}.}
who provided us with the motivation, knowledge base to pursue this
project, and Simon Levin for his leadership and inspiration.  CLE
would like to thank Josh Plotkin, Warren Ewens, Todd Parsons, and
Ricky Der, from whom he has learned most of what he knows about
population genetics.

We would both like to thank Charlie Fefferman for showing us an explicit
formula for $k^0_t(x,y),$ which set us off in the very fruitful direction
pursued herein. We would also like to thank Dan Stroock for his help with
connections to Probability Theory.

CLE would like to acknowledge the financial support of DARPA through
grants: HR0011-05-1-0057 and HR0011-09-1-0055, the NSF through the
grant: DMS06-03973, and Leslie Greengard and the Courant Institute,
where this text was completed. 

RM would like to acknowledge the financial support of DARPA through
grants: HR0011-05-1-0057 and HR0011-09-1-0055, and the NSF through the
grant: DMS08-05529.
\chapter{Introduction} 
In population genetics one frequently replaces a discrete Markov chain model,
which describes the random processes of genetic drift, with or without
selection, and mutation with a limiting, continuous time and space, stochastic
process. If there are $N+1$ possible types, then the configuration space for
the resultant continuous Markov process is typically the $N$-simplex
\begin{equation}
  \cS_{N}=\{(x_1,\dots,x_N):\: x_j\geq 0\text{ and }x_1+\cdots+x_N\leq 1\}.
\end{equation}
If a different scaling is used to define the limiting process,
different domains might also arise.  As a geometrical object the simplex is
quite complicated. Its boundary is not a smooth manifold, but has a stratified structure
with strata of codimensions $1$ through $N.$ The codimension 1 strata are
\begin{equation}
  \Sigma_{1,l}=\{x_j=0\}\cap\cS_{N}\text{ for }l=1,\dots,N,
\end{equation}
along with
\begin{equation}
   \Sigma_{1,0}=\{x_1+\cdots+x_N=1\}\cap\cS_{N}.
\end{equation}
Components of the stratum of codimension $1<l\leq N,$ arise by choosing integers
$0\leq i_1<\cdots<i_l\leq N$ and forming the intersection:
\begin{equation}\label{codimnlsimpl}
  \Sigma_{1,i_1}\cap\cdots\cap\Sigma_{1,i_l}.
\end{equation}
The simplex is an example of a \emph{manifold with corners}. The singularity of
its boundary significantly complicates the analysis of differential operators
acting functions defined in $\cS_{N}.$

In the simplest case, without mutation or selection, the limiting operator of
the Wright-Fisher process is the Kimura diffusion operator, with formal
generator:
\begin{equation}
  L_{\Kimura}=\sum_{i,j=1}^{N}x_i(\delta_{ij}-x_j)\pa_{x_i}\pa_{x_j}.
\end{equation}
This is the ``backward'' Kolmogorov operator for the limiting Markov
process. This operator is elliptic in the interior of $\cS_{N}$ but the
coefficient of the second orders normal derivative tends to zero as one
approaches a boundary. We can introduce local coordinates $(x_1,y_1,\dots,
y_{N-1})$ near the interior of a point on one of the faces of $\Sigma_{1,l},$
so that the boundary is given locally by the equation $x_1=0,$ and the operator
then takes the form
\begin{equation}
  x_1\pa_{x_1}^2+\sum_{l=1}^{N-1}x_1c_{1l}\pa_{x_1}\pa_{y_l}+
\sum_{l,m=1}^{N-1}c_{lm}\pa_{y_m}\pa_{y_l},
\end{equation}
where the matrix $c_{lm}$ is positive definite. To include the effects
of mutation, migration and selection, one typically adds a vector
field:
\begin{equation}
  V=\sum_{i=1}^{N}b_i(x)\pa_{x_i},
\end{equation}
where $V$ is inward pointing along the boundary of $\cS_{N}.$ In the
classical models, if only the effect of mutation and migration are
included, then the coefficients $\{b_i(x)\}$ can be taken to be linear
polynomials, whereas selection requires at least quadratic terms.

The most significant feature is that the coefficient of $\pa_{x_1}^2$
vanishes exactly to order $1.$ This places $L_{\Kimura}$ outside the
classes of degenerate elliptic operators that have already been
analyzed in detail.  For applications to Markov processes the
difficulty that presents itself is that it is not possible to
introduce a square root of the coefficient of the second order terms
that is Lipschitz continuous up to the boundary. Indeed the best one
can hope for is H\"older-$\frac{1}{2}.$ The uniqueness of the
solutions to either the forward Kolmogorov equation, or the
associated SDE, cannot then be concluded using standard methods.

Even in the presence of mutation and migration, the solutions
of the heat equation for this operator in one-dimension was studied by
Kimura, using the fact that $L_{\Kimura}+V$ preserves polynomials of
degree $d$ for each $d.$ In higher dimensions it was done by Karlin
and Shimakura by showing the existence of a complete basis of
polynomial eigenfunctions for this operator. This in turn leads to a
proof of the existence of a polynomial solution to the initial value
problem for $[\pa_t-(L_{\Kimura}+V)]v=0$ with polynomial initial
data. Using the maximum principle, this suffices to establish the
existence of a strongly continuous semi-group acting on $\cC^0,$ and
establish many of its basic properties, see~\cite{Shimakura2}. This
general approach has been further developed by Barbour, Etheridge and
Griffiths, see~\cite{EthierGriffiths1993,BarbourEthierGriffiths2000,
  EtheridgeGriffiths2009,Griffiths2009}.

As noted, if selection is also included, then the coefficients of $V$ are
at least quadratic polynomials, and can be quite complicated,
see~\cite{ChalubSouza}. So long as the second order part remains
$L_{\Kimura},$ then a result of Ethier, using the Trotter product
formula, makes it possible to again define a strongly continuous
semi-group, see~\cite{Ethier}. Various extensions of these results,
using a variety of functional analytic frameworks, were made by
Athreya, Barlow, Bass Perkins, Sato, Cerrai and Cl\'ement, and others,
see~\cite{AthreyaBarlowBassPerkins2002,BassPerkins2002,
  CerraiClement1,CerraiClement2,CerraiClement3}. 

For example Cerrai and Cl\'ement constructed a semi-group acting on
$\cC^0([0,1]^N),$ with the coefficient $a_{ij}$ of
$\pa_{x_i}\pa_{x_j}$ assumed to be of the form
\begin{equation}
  a_{ij}(x)=m(x)A_{ij}(x_i,x_j).
\end{equation}
Here $m(x)$ is strictly positive.
In~\cite{AthreyaBarlowBassPerkins2002,
  BassPerkins2002,BassLavrentiev2007}, Bass, Perkins along with
several collaborators, study a class of equations, similar to that
defined below. Their work has many points of contact with our own, and
we discuss it in greater detail at the end of
Section~\ref{s.app_prob}.

We have not yet said anything about boundary conditions, which would seem to be
a serious omission for a PDE on a domain with a boundary.  Indeed, one would
expect that there would be an infinite dimensional space of solutions to the
homogeneous equation. It is possible to formulate local boundary conditions
that assure uniqueness, but, in some sense this is not necessary. As a result
of the degeneracy of the principal part, uniqueness for these types of
equations can also be obtained as a consequence of regularity alone! We
illustrate this in the simplest 1-dimensional case, which is the equation, with
$b(0)\geq 0,\, b(1)\leq 0,$
\begin{equation}\label{eqn1.9.00}
 \pa_tv- [x(1-x)\pa_x^2+b(x)\pa_x]v=0\text{ and }v(x,0)=f(x).
\end{equation}
If we assume that $\pa_xv(x,t)$ extends continuously to $[0,1]\times
(0,\infty)$ and
\begin{equation}
  \lim_{x\to 0^+}x(1-x)\pa_x^2v(x,t)=
 \lim_{x\to 1^-}x(1-x)\pa_x^2v(x,t)=0,
\end{equation}
then a simple maximum principle argument shows that the solution is
unique. In our approach, such regularity conditions naturally lead to
uniqueness, and little effort is expended in the consideration of
boundary conditions. In Chapter~\ref{c.maxprin} we prove a
generalization of the Hopf boundary point maximum principle that
demonstrates, in the general case, how regularity implies uniqueness.

\section{Generalized Kimura Diffusions}
In his seminal work, Feller analyzed the most general closed extensions of
operators, like those in~\eqref{eqn1.9.00}, which generate Feller semi-groups
in 1-dimension, see~\cite{Feller1}.  Up to now very little is known, in higher
dimensions, about the analytic properties of the solution to the initial value
problem for the heat equation
\begin{equation}
  \pa_tv-(L_{\Kimura}+V)v=0\text{ in }(0,\infty)\times \cS_{N}\text{ with }v(0,x)=f(x).
\end{equation}
Indeed, if we replace $L_{\Kimura}$ with a qualitatively similar
second order part, which does not take one of the forms described
above, then even the existence of a solution is not known.  In this
monograph we introduce a very flexible analytic framework for studying
a large class of equations, which includes all standard models, of
this type appearing in population genetics, as well as the SIR model
for epidemics, see~\cite{Ewens,Shimakura2}, and models that arise
in Mathematical Finance, see~\cite{FernholzKaratzas}. Our approach is
to introduce non-isotropic H\"older spaces with respect to which we
establish sharp existence and regularity results for the solutions to
heat equations of this type, as well as the corresponding elliptic
problems.  Using the Lumer-Phillips theorem we conclude that the
$\cC^0$-graph closure of this operator generates a strongly continuous
semi-group.

In this monograph we extend our work on the 1d-case in~\cite{WF1d}.  Our
analysis applies to a class of operators that we call \emph{generalized Kimura
  diffusions}, which act on functions defined on \emph{manifolds with corners.}
Such spaces generalize the notion of a regular convex polyhedron in $\bbR^N,$
e.g. the simplex. Working in this more general context allows for a great deal
of flexibility, which proves indispensable in the proof of our basic existence
result.

Locally a manifold with corners, $P,$\index{manifold with corners} can
be described as a subset of $\bbR^N$ defined by inequalities. Let
$\{p_k(x):\: k=1,\dots,K\}$ be smooth functions in the unit ball
$B_1(0)\subset\bbR^N,$ vanishing at $0,$ with $\{dp_k(0):\:
k=1,\dots,K\}$ linearly independent; clearly $K\leq N.$ Locally $P$ is
diffeomorphic to
\begin{equation}
  \bigcap\limits_{k=1}^K\{x\in B_1(0):\: p_k(x)\geq 0\}.
\end{equation}
We let $\Sigma_k=P\cap \{x:\: p_k(x)=0\};$ suppose that $\Sigma_k$ contains a
non-empty, open $(N-1)$-dimensional hypersurface and that $dp_k$ is
non-vanishing in a neighborhood of $\Sigma_k.$ The boundary of $P$ is a
stratified space,\index{boundary stratification} where the strata of
codimension $n$ locally consists of points where the boundary is defined by the
vanishing of $n$ functions with independent gradients. The components of the
codimension 1 part of the $bP$ are called \emph{faces}\index{faces}. As
in~\eqref{codimnlsimpl}, the codimension-$n$ stratum of $bP$ is formed from
intersections of $n$ faces.

The \emph{formal} generator is a degenerate elliptic operator of the form
\begin{equation}\label{eqn1.2.00}
  L=\sum\limits_{i,j=1}^N A_{ij}(x)\pa_{x_i}\pa_{x_j}+\sum\limits_{j=1}^Nb_j(x)\pa_{x_j}.
\end{equation}
Here $A_{ij}(x)$ is a smooth, symmetric matrix valued function in $P.$
The second order term is positive definite in the interior of $P$ and
degenerates along the hypersurface boundary components in a  specific
way. For each $k$
\begin{equation}\label{eqn3}
  \sum\limits_{i,j=1}^N a_{ij}(x)\pa_{x_i}p_k(x)\pa_{x_j}p_k(x)\propto p_k(x)
  \text{ as } x\text{ approaches }\Sigma_k.
\end{equation}
On the other hand, 
\begin{equation}\label{eqn4}
  \sum\limits_{i,j=1}^N a_{ij}(x)v_iv_j>0\text{ for
  }x\in \Int \Sigma_k\text{ and }v\neq 0\in T_x\Sigma_k.
\end{equation}
The first order part of $L$ is an inward pointing vector field
\begin{equation}\label{eqn1.5.00}
  Vp_k(x)=\sum\limits_{j=1}^Nb_j(x)\pa_{x_j}p_k(x)\geq 0\text{ for }x\in \Sigma_k.
\end{equation}
We call a second order partial differential operator defined on $P,$ which is
non-degenerate elliptic in $\Int P,$  with this local description near any
boundary point a \emph{generalized Kimura diffusion}\index{generalized Kimura
  diffusion}. 

If $p$ is a point on the stratum of $bP$ of codimension $n,$ then
locally there are coordinates $(x_1,\dots,x_n;y_1,\dots,y_m)$ so that
$p$ corresponds to $(\bzero;\bzero),$ and a neighborhood, $U,$ of $p$
is given by
\begin{equation}
  U=\{(\bx;\by)\in [0,1)^n\times (-1,1)^m\}.
\end{equation}
In these local coordinates a generalized Kimura diffusion,  $L,$ takes the form
\begin{multline}
L=\sum_{i=1}^{n}[a_{ii}(\bx;\by)x_i\pa_{x_i}^2+\tb_i(\bx;\by)\pa_{x_i}]+
\sum_{1\leq i\neq j\leq  n}x_ix_j a_{ij}(\bx;\by)\pa^2_{x_i x_j} +\\
\sum_{i=1}^n\sum_{k=1}^mx_ib_{ik}(\bx;\by)\pa^2_{x_i y_k}+
\sum_{k,l=1}^mc_{kl}(\bx;\by)\pa^2_{y_k y_l}+
\sum_{k=1}^md_k(\bx;\by)\pa_{y_k};
\end{multline}
$(a_{ij})$ and $(c_{kl})$ are symmetric matrices, the matrices $(a_{ii})$
and $(c_{kl})$ are strictly positive definite. The coefficients
$\{\tb(\bx;\by)\}$ are non-negative along  $bP\cap U,$ so that first
order part is inward pointing.

Let $P$ be a compact manifold with corners and $L$ a generalized Kimura diffusion
defined on $P.$ Broadly speaking, our goal is to prove the existence, uniqueness and
regularity of solutions to the equation
\begin{equation}\label{eqn1.9.000}
\begin{split}
  &(\pa_t-L)u=g\text{ in }P\times (0,\infty)\\
&\text{with }u(p,0)=f(p),
\end{split}
\end{equation}
with certain boundary behavior along $bP\times [0,\infty),$ for data $g$ and $f$
satisfying appropriate regularity conditions. These results in turn can be used
to prove the existence of a strongly continuous semi-group acting on $\cC^0(P),$
with formal generator $L.$ This is the ``backward Kolmogorov equation.'' The
solution to the  ``forward Kolmogorov equation,''  $(\pa_t-L^*)m=\nu,$ is then
given by the adjoint semi-group, canonically defined on a domain in
$[\cC^0(P)]'=\cM(P),$ the space of finite Borel measures on $P.$
\index{backward Kolmogorov equation}\index{forward Kolmogorov equation}

\section{Model problems}
The problem of proving the existence of solutions to a class of PDEs is essentially a matter
of finding a good class of model problems, for which existence and regularity
can be established, more or less directly, and then finding a functional analytic
setting in which to do a perturbative analysis of the equations of
interest. The model operators for Kimura diffusions are the differential
operators, defined on $\bbR_+^n\times\bbR^m,$ by\index{model problem}\index{model operator}
\begin{equation}\label{eqn1.18.05}
  L_{\bb,m}=\sum_{j=1}^n[x_j\pa_{x_j}^2+b_j\pa_{x_j}]+\sum_{k=1}^m\pa_{y_k}^2.
\end{equation}
Here $\bb=(b_1,\dots,b_n)$ is a non-negative vector. 

We have not been too explicit about the boundary conditions that we impose
along $b\bbR_+^n\times\bbR^m.$ This condition can be defined by a local
Robin-type formula involving the value of the solution and its normal
derivative along each hypersurface boundary component of $bP.$
For $b>0,$ the 1-dimensional model operator, $L_b=x\pa_x^2+b\pa_x,$ has two
indicial roots\index{indicial root}
\begin{equation}
  \beta_0=0,\beta_1=1-b,
\end{equation}
that is
\begin{equation}
  L_bx^{\beta_0}=L_bx^{\beta_1}=0.
\end{equation}
The boundary condition,
\begin{equation}\label{eqn1.23.007}
  \lim_{x\to 0^+}[\pa_x(x^bu(x,t))-bx^{b-1}u(x,t)]=0,
\end{equation}
excludes the appearance of terms like $x^{1-b}$ in the asymptotic
expansion of solutions along $x=0.$ In fact, this condition insures
that $u$ is as smooth as possible along the boundary: if $g=0$ and $f$
has $m$ derivatives then the solution to~\eqref{eqn1.9.000},
satisfying~\eqref{eqn1.23.007} does as well.  This boundary condition
can be encoded as a regularity condition, that is
$u(\cdot,t)\in\cC^1([0\infty))\cap\cC^2((0,\infty)),$ with 
  \begin{equation}
    \lim_{x\to 0^+}x\pa_x^2u(x,t)=0
  \end{equation}
for $t>0.$ We call the unique
solution to a generalized Kimura diffusion, satisfying this condition,
or its analogue, the \emph{regular} solution. The vast majority of
this monograph is devoted to the study of regular solutions.

In applications to probability one often seeks solutions to equations of the
form $Lw=g,$ where $w$ satisfies a Dirichlet boundary condition:
$w\restrictedto_{P}=h.$ Our uniqueness results often imply that these equations
\emph{cannot} have a regular solution, for example, when $g\geq 0.$ In the
classical case the solutions to these problems can sometimes be written down
explicitly, and are seen to involve the non-zero indicial roots. Usually these
satisfy the other natural boundary condition, a la~\cite{Feller1}. In
1-dimension, when $b\neq 0,$ it is:
\begin{equation}
  \lim_{x\to 0^+}[\pa_x(xu(x,t))-(2-b)u]=0,
\end{equation}
and allows for solutions that are
$O(x^{1-b})$ as $x\to 0^+.$ These are not smooth up to the boundary,
even if the data is. The adjoint of $L$ is naturally defined as an
operator on $\cM(P),$ the space finite Borel measures on $P.$ It is
more common to study this operator using techniques from probability
theory, see~\cite{StrookProbPDE}. 

For a generalized Kimura diffusion in dimensions greater than 1, the coefficient
of the normal first derivative can vary as one moves along the boundary. For
example, in two-dimensions one might consider the operator
\begin{equation}
  L=x\pa_x^2+\pa_y^2+b(y)\pa_x.
\end{equation}
If $b(y)$ is not constant, then, with
the boundary condition
\begin{equation}
    \lim_{x\to 0^+}[\pa_x(xu(x,y,t))-(2-b(y))u(x,y,t)]=0,
\end{equation}
one would be faced with the very thorny issue of a varying indicial
root on the outgoing face of the heat or resolvent kernel. As it is,
we get a varying indicial root on the incoming face. A fact which already
places the analysis of this problem beyond what has been achieved
using the detailed kernel methods familiar in geometric microlocal
analysis. The natural boundary condition for the adjoint operator
includes the condition:
\begin{equation}
    \lim_{x\to 0^+}[\pa_x(xu(x,y,t))-b(y)u(x,y,t)]=0,
\end{equation}
allowing for solutions that behave like $x^{b(y)-1},$ as $x\to 0^+.$

The solution operators for the 1-dimensional model problems are given by  simple explicit
formul{\ae}. If $b>0,$ then the heat kernel is
\begin{equation}\label{eqn1.22.05}
  k^b_t(x,y)dy=\left(\frac{y}{t}\right)^be^{-\frac{x+y}{t}}\psi_b\left(\frac{xy}{t^2}\right)\frac{dy}{y},
\end{equation}
where\index{heat kernel, model problem}\index{$ k^b_t(x,y)$}
\begin{equation}
  \psi_b(z)=\sum_{j=0}^{\infty}\frac{z^j}{j!\Gamma(j+b)}.
\end{equation}
\index{$\psi_b(z)$}If $b=0$ then
\begin{equation}\label{eqn1.24.05}
  k^0_t(x,y)=e^{-\frac{x}{t}}\delta_0(y)+
\left(\frac{x}{t}\right)e^{-\frac{x+y}{t}}\psi_2\left(\frac{xy}{t^2}\right)\frac{dy}{t}.
\end{equation}
In either case $k^b_t$ is smooth as $x\to 0^+$ and displays a
$y^{b-1}$ singularity as $y\to 0^+.$ It is notable that the character
of the kernel changes dramatically as $b\to 0,$ nonetheless the
regular solutions to these heat equations satisfy uniform estimates
even as $b\to 0^+.$ This fact is quite essential for the success of
our approach.

The structure of these operators suggests that a
natural functional analytic setting in which to do the analysis might be that
provided by the anisotropic H\"older spaces defined by the singular, but
incomplete metric on $\bbR_+^n\times\bbR^m:$
\begin{equation}
  ds_{\WF}^2=\sum_{j=1}^n\frac{dx_j^2}{x_j}+\sum_{k=1}^mdy_m^2.
\end{equation}
\index{anisotropic metric}\index{Wright-Fisher metric} Similar spaces have been
introduced by other authors for problems with similar degeneracies,
see~\cite{DaskHam}. In~\cite{GoulaouicShimakura} Goulaouic and Shimakura proved
a priori estimates in a H\"older space of this general sort in a case where the
operator has this type of degeneracy, but the boundary is smooth. As was the
case in these earlier works, we introduce two families of anisotropic H\"older
spaces, which we denote by $\cC^{k,\gamma}_{\WF}(P),$ and
$\cC^{k,2+\gamma}_{\WF}(P),$ for $k\in\bbN_0,$ and $0<\gamma<1.$ In this
context, heuristically an operator $A$ is ``elliptic of second order'' if
$A^{-1}:\cC^{k,\gamma}_{\WF}(P)\to \cC^{k,2+\gamma}_{\WF}(P).$ Note that
$\cC^{k,2+\gamma}_{\WF}(P)\subseteq \cC^{k+1,\gamma}_{\WF}(P),$ but
$\cC^{k+2,\gamma}_{\WF}(P)\nsubseteq \cC^{k,2+\gamma}_{\WF}(P),$ which explains
the need for two families of spaces.

In this monograph we consider the problem in~\eqref{eqn1.9.000} for $f$ and $g$
belonging to these H\"older spaces. The results obtained suffice to prove the
existence of a semi-group on the space $\cC^0(P),$ but establishing the refined
regularity properties of this semi-group and its adjoint require the usage of
\emph{a priori} estimates. These are of a rather different character from the
analysis presented here; we will return to this question in a subsequent publication.

As manifolds with corners have non-smooth boundaries, and the Kimura diffusions
are degenerate elliptic operators, the analysis of~\eqref{eqn1.9.000} can be
expected to be rather challenging. We have already indicated a variety of problems that
arise:
\begin{enumerate}
\item The principal part of $L$ degenerates at the boundary.
\item The boundary of $P$ is not smooth.
\item The ``indicial roots'' vary with the location of the point on $bP.$
\item The character of the solution operator is quite different at points where
  the vector field is tangent to $bP.$
\end{enumerate}

Along the boundary $\{x_j=0\},$ the first and second order
terms in~\eqref{eqn1.18.05}, $b_j\pa_{x_j}$ and $x_j\pa_{x_j}^2,$
respectively are of comparable ``strength.'' It is a notable and
non-trivial fact that estimates for the solutions of these equations
can be proved in these H\"older spaces, without regard for the value of
$\bb\geq\bzero.$ As there is an explicit formula for the fundamental
solution, the analysis of these model operators, while tedious and
time consuming, is elementary. Indeed the solution of the homogeneous
Cauchy problem
\begin{equation}
\begin{split}
  &(\pa_t-L_{\bb,m})u=0\text{ in }P\times (0,\infty)\\
&\text{with }u(p,0)=f(p),
\end{split}
\end{equation}
has an analytic extension to $\Re t>0,$ which satisfies many useful
estimates. 

To obtain a gain of derivatives where $\Re t>0,$ in a manner that can
be extended beyond the model problems, one must address the
inhomogeneous problem, which has somewhat simpler analytic
properties. By this device, one can also estimate the Laplace
transform of the heat semi-group, which is the resolvent operator:
\begin{equation}
  (\mu-L_{\bb,m})^{-1}=\int\limits_{0}^{\infty}e^{tL_{\bb,m}}e^{-\mu t}dt.
\end{equation}
\index{resolvent operator}The estimates for the inhomogeneous problem show
that, in an appropriate sense, $ (\mu-L_{\bb,m})^{-1}$ gains two derivatives
and is analytic in the complement of $(-\infty,0].$ Finally one can
re-synthesize the heat operator from the resolvent, via contour integration:
\begin{equation}
  e^{tL_{\bb,m}}=\frac{1}{2\pi i}\int\limits_{C}(\mu-L_{\bb,m})^{-1}e^{\mu t}d\mu,
\end{equation}
where $C$ is of the form $|\arg\mu|=\frac{\pi}{2}+\alpha,$ for an
$0<\alpha<\frac{\pi}{2}.$  This shows that, for $t$ with positive real part,
$e^{tL_{\bb,m}}$ also gains two derivatives. 

\section{Perturbation Theory}
The next step is to use these estimates for the model problems in a
perturbative argument to prove existence and regularity for a
generalized Kimura diffusion operator $L$ on a manifold with corners,
$P.$ The boundary of a manifold with corners is a stratified space,
which produces a new set of difficulties.  To overcome this we use an
induction on the maximal codimension of the strata of $bP.$

The induction starts with the simplest case where $bP$ is just a manifold, and
$P$ is then a manifold with boundary. In this case, we can use the model
operators to build a parametrix for the solution operator to the heat equation
in a neighborhood of the boundary, $\hQ^t_b.$ It is a classical fact that there
is an exact solution operator, $\hQ^t_i,$ defined in the complement of a
neighborhood of the boundary, for, in any such  subset of $P,$ $L$ is a
non-degenerate elliptic operator. Using a partition of unity these operators
can be ``glued together'' to define a parametrix, $\hQ^t$ for the solution
operator. The Laplace transform
\begin{equation}
  \hR(\mu)=\int\limits_{0}^{\infty}e^{\mu t}\hQ^tdt
\end{equation}
is then a right parametrix for $(\mu-L)^{-1}.$ Using the estimates and analyticity
for the model problems, and the properties of the interior solution operator, we
can show that
\begin{equation}
  (\mu-L)\hR(\mu)=\Id+E(\mu),
\end{equation}
where $E(\mu)$ is analytic in $\bbC\setminus (-\infty,0],$ and the Neumann
series for $(\Id+E(\mu))^{-1}$ converges in the operator norm topology for
$\mu$ in sectors $|\arg\mu|\leq \pi-\alpha,$ for any $\alpha>0,$ if
$|\mu|$ sufficiently large. This allows us to show that
\begin{equation}
  (\mu-L)^{-1}=\hR(\mu)(\Id+E(\mu))^{-1}
\end{equation}
is analytic and satisfies certain estimates in 
\begin{equation}
  T_{\alpha,R}=\{\mu: |\arg\mu|<\pi-\alpha,\quad |\mu|>R\},
\end{equation}
for any $0<\alpha,$ and $R$ depending on $\alpha.$ \index{$T_{\alpha,R}$}

For $t$ in the right half plane we can now reconstruct the heat semi-group
acting on the H\"older spaces:
\begin{equation}\label{eqn1.22.00}
  e^{tL}=\frac{1}{2\pi i}\int\limits_{bT_{\alpha,R}}(\mu-L)^{-1}e^{\mu t}d\mu
\end{equation}
for an appropriate choice of $\alpha.$ This allows us to verify that $e^{tL}$
has an analytic continuation to $\Re t>0,$ which satisfies the desired
estimates with respect to the anisotropic H\"older spaces defined above. The
proof for the general case now proceeds by induction on the maximal codimension
of the strata of $bP.$ In all cases we use the model operators to construct a
boundary parametrix $\hQ^t_b$ near the maximal codimensional part of $bP.$ The
induction hypothesis provides an exact solution operator in the ``interior,''
$\hQ^t_i,$ with certain properties, which we once again glue together to get
$\hQ^t.$ A key step in the argument is to verify that the heat operator we
finally obtain satisfies the induction hypotheses. The representation of
$e^{tL}$ in~\eqref{eqn1.22.00} is a critical part of this argument.

\section{Main Results}\label{s.mainresults}
With these preliminaries we can state our main results.  The sharp estimates
for operators $e^{tL}$ and $(\mu-L)^{-1}$ are phrased in terms of two families
of H\"older spaces. For $k\in\bbN_0$ and $0<\gamma<1,$ we define the spaces
$\cC^{k,\gamma}_{\WF}(P),$ $\cC^{k,2+\gamma}_{\WF}(P),$ and their ``heat-space''
analogues, $\cC^{k,\gamma}_{\WF}(P\times [0,T]),$ $\cC^{k,2+\gamma}_{\WF}(P\times
[0,T]),$ see Chapter~\ref{chap.holdspces}.  For example: in the 1-dimensional
case $f\in \cC^{0,\gamma}_{\WF}([0,\infty))$ if $f$ is continuous and
\begin{equation}
  \sup_{x\neq y}\frac{|f(x)-f(y)|}{|\sqrt{x}-\sqrt{y}|^{\gamma}}<\infty;
\end{equation}
it belongs to $\cC^{0,2+\gamma}_{\WF}([0,\infty))$ if $f,\pa_xf,$ and $x\pa_x^2f$ all
belong to $\cC^{0,\gamma}_{\WF}([0,\infty)),$ with 
$$\lim_{x\to 0^+,\infty}x\pa_x^2f(x)=0.$$
For $k\in\bbN,$ we say that $f\in \cC^{k,\gamma}_{\WF}([0,\infty)),$ if $f\in
\cC^{k}([0,\infty)),$ and $\pa_x^kf\in\cC^{0,\gamma}_{\WF}([0,\infty)).$ A
function $g\in\cC^{0,\gamma}_{\WF}([0,\infty)\times [0,\infty)),$ if 
$g\in\cC^{0}([0,\infty)\times [0,\infty)),$ and
\begin{equation}
  \sup_{(x,t)\neq (y,s)} \frac{|g(x,t)-g(y,s)|}{[|\sqrt{x}-\sqrt{y}|+\sqrt{|t-s|}]^{\gamma}}<\infty,
\end{equation}
etc. 

Much of this monograph is concerned with proving detailed estimates for the
model problems with respect to these H\"older spaces and then using
perturbative arguments to obtain analogous results for a general Kimura
diffusion on an arbitrary compact manifold with corners. 

To describe the uniqueness properties for solutions to these equations, we need
to consider the geometric structure of the boundary of $P,$ and its
relationship to $L.$ As noted $bP$ is a stratified space, with hypersurface
boundary components $\{\Sigma_{1,j}:\: j=1,\dots, N_1\}.$ A boundary component
of codimension $n$ is a component of an intersection
\begin{equation}
  \Sigma_{1,i_1}\cap\cdots\cap\Sigma_{1,i_n},
\end{equation}
where $1\leq i_1<\cdots<i_n\leq N_1.$ A component of $bP$ is \emph{minimal}
if\index{minimal boundary} it is an isolated point or a positive dimensional
manifold without boundary. We denote the set of minimal components by
$bP_{\min}.$\index{$bP_{\min}$} Fix a generalized Kimura diffusion operator $L.$ Let $\{\rho_j:\:
j=1,\dots,N_1\}$ be defining functions for the hypersurface boundary
components. We say that $L$ is \emph{tangent} to $\Sigma_{1,j}$ if
$L\rho_j\restrictedto_{\Sigma_{1,j}}=0,$ and \emph{transverse} if there is a
$c>0$ so that\index{tangent to $L$}\index{transverse to $L$}
\begin{equation}
  L\rho_j\restrictedto_{\Sigma_{1,j}}>c.
\end{equation}
\begin{definition} The terminal boundary of $P$ relative to $L,$ $bP_{\ter}(L),$
  consists of elements of $bP_{\min}$ to which $L$ is tangent, along with
  boundary strata, $\Sigma$ of $P$ to which $L$ is tangent, and such that
  $L_{\Sigma}$ is transverse to all components of $b\Sigma.$\index{terminal
    boundary}
\end{definition}

For the model space we say that 
\begin{equation}
  f\in\cD^2_{\WF}(\bbR_+^n\times\bbR^m)\subset \cC^1(\bbR_+^n\times\bbR^m)\cap
\cC^2((0,\infty)^n\times\bbR^m)
\end{equation}
if the scaled second derivatives
\begin{equation}x_i\pa_{x_i}^2f(\bx;\by),\, x_ix_j\pa_{x_ix_j}^2f(\bx,\by),\,
  x_i\pa^2_{x_iy_l}f(\bx,\by),\,\pa^2_{y_ly_k}f(\bx,\by)
\end{equation}
extend continuously to $\bbR_+^n\times\bbR^m.$ We also assume that
$x_ix_j\pa_{x_ix_j}^2f(\bx,\by)$ tends to zero if either $x_i$ or $x_j$
goes to zero, and $x_i\pa_{x_i}^2f(\bx;\by)$ and
$x_i\pa^2_{x_iy_l}f(\bx,\by)$ go to zero as $x_i$ goes to zero. A
function $f\in \cC^1(P)\cap\cC^2(\Int P)$ belongs to $\cD^2_{\WF}(P)$ if it
belongs to these local spaces in each local coordinate chart. Using a variant
of the Hopf maximum principle, we can prove
\begin{theorem} Let $P$ be a compact manifold with corners, and $L$ a
  generalized Kimura diffusion defined on $P.$ Suppose that $L$ is either
  tangent or transverse to every hypersurface boundary component of $bP,$ and
  let $bP_{\ter}(L)$ denote the set of terminal components of the boundary
  stratification relative to $L.$ The cardinality of the set $bP_{\ter}(L)$
  equals the dimension of the null-space of $L$ acting on $\cD^2_{\WF}(P),$
  which is also the dimension of $\Ker \oL^*.$ The nullspace of $L$ is
  represented by smooth non-negative functions; the nullspace of $\oL^*$ by
  non-negative measures supported on the components of $bP_{\ter}(L).$
\end{theorem}

The existence and regularity results for the heat equation defined by
a general Kimura diffusion, $L,$ on a manifolds with corners, $P,$ are
summarized in the next two results:
\begin{theorem}\label{thm0.4.2.00} Let $P$ be a compact manifold with corners, 
  $L$ a generalized Kimura diffusion on $P,$ $k\in\bbN_0$ and $0<\gamma<1.$ If
  $f\in\cC^{k,\gamma}_{\WF}(P),$ then there is a unique solution
$$v\in\cC^{k,\gamma}_{\WF}(P\times [0,\infty))\cap \cC^{\infty}(P\times (0,\infty)),$$ 
to the initial value problem
\begin{equation}
  (\pa_t-L)v=0\text{ with }v(p,0)=f(p).
\end{equation}
This solution has an analytic continuation to $t$ with $\Re t>0.$
\end{theorem}

We have a similar result for the inhomogeneous problem:
\begin{theorem} Let $P$ be a compact manifold with corners, $L$ a generalized Kimura
  diffusion on $P,$ $k\in\bbN_0,$ $0<\gamma<1,$ and $T>0.$ If
  $g\in\cC^{k,\gamma}_{\WF}(P\times [0,T]),$ then there is a unique solution 
$$u\in\cC^{k,2+\gamma}_{\WF}(P\times [0,T])$$
to
\begin{equation}
  (\pa_t-L)u=g\text{ with }u(p,0)=0,
\end{equation}
which satisfies estimates of the form
\begin{equation}
  \|u\|_{\WF,k,2+\gamma,T}\leq C(1+T)\|g\|_{\WF,k,\gamma,T}.
\end{equation}
\end{theorem}

We also have a result for the resolvent of $L$ acting on the spaces
$\cC^{k,2+\gamma}_{\WF}(P),$ showing that $(\mu-L)^{-1}$ is an elliptic
operator with respect to our scales of Banach spaces.
\begin{theorem}\label{thm0.4.4.00} Let $P$ be a compact manifold with corners, 
  $L$ a generalized Kimura diffusion on $P,$ $k\in\bbN_0,$ $0<\gamma<1.$ The
  spectrum, $E,$ of the unbounded, closed operator $L,$ with domain
 $$\cC^{k,2+\gamma}_{\WF}(P)\subset  \cC^{k,\gamma}_{\WF}(P),$$ 
is independent of $k,$ and $\gamma.$ It is a discrete set lying in a
  conic neighborhood of $(-\infty,0].$ The eigenfunctions belong $\cC^{\infty}(P).$
\end{theorem}
\begin{remark} Note that  $\cC^{k,2+\gamma}_{\WF}(P)$ is \emph{not} a dense
  subspace of  $\cC^{k,\gamma}_{\WF}(P).$
\end{remark}

\section{Applications in Probability Theory}\label{s.app_prob}
The principal sources for operators of the type studied here are infinite 
population limits of Markov chains in population genetics, and certain classes
of ``linear'' models in mathematical finance. In this context the operator $L,$
acting on a dense domain in $\cC^0(P)$ is called the backward Kolmogorov
operator. Its formal adjoint, which acts on the dual space, $\cM(P)$, of finite
signed measure Borel measures on $P,$ is the forward Kolmogorov operator. The
standard way to address the adjoint operator is to study the martingale problem
associated with $L$ on $\cC^0([0,\infty);P).$ Letting
$\omega\in\cC^0([0,\infty);P),$ for each $t\in [0,\infty),$  we define 
$$x(t):\cC^0([0,\infty);P)\to P,$$
by $x(t)[\omega]=\omega(t).$ We let $\cF_{t}$ denote the $\sigma$-field
generated by $\{x(s):0\leq s\leq t\}$ and $\cF$ the $\sigma$-field generated by
$\{x(s):\: s\geq 0\}.$ For each $q\in P,$ a probability measure $\bbP_{q}$ on
$(\cC^0([0,\infty);P),\cF)$ is a solution of the martingale problem associated
with $L$ and starting from $q\in P$ at time $t=0,$ if\index{martingale problem}
\begin{equation}
  \begin{split}
    &\bbP_q(x(0)=q)=1\text{ and }\\
&\left\{f(x(t))-\int\limits_0^tLf(x(s))ds\right\}_{t\geq 0}
  \end{split}
\end{equation}
is a $\bbP_q$-martingale with respect to $\{\cF_t\}_{t\geq 0}.$

The existence results Theorems~\ref{thm0.4.2.00} or~\ref{thm0.4.4.00} suffice
to prove that the associated martingale problem has a unique solution. A
standard argument then shows that the paths for associated strong Markov
process remain, almost surely, within $P.$ From this we can deduce a wide
variety of results about the forward Kolmogorov equation, and the solutions of
the associated stochastic differential equation.  The precise nature of
these results depends on the behavior of the vector field $V$ along $bP.$ As
this analysis requires techniques quite distinct from those employed here, we
defer these questions to a future,  joint publication with Daniel Stroock.

Using the Lumer-Phillips theorem, these results also suffice to prove
that the $\cC^0(P)$-graph closure, $\oL,$ of $L$ acting on $\cC^3(P)$
is the generator of a strongly continuous contraction semi-group. At
present we have not succeeded in showing that the resolvent of $\oL$
is compact, and will return to this question in a later
publication. We have nonetheless been able to characterize the
nullspace of adjoint operator $\oL^*,$ under a natural clean intersection
condition for the vector field $V.$ This allows for an analysis of the
asymptotic behavior of the solution to $\pa_t \nu-L^*\nu=0,$ as
$t\to\infty,$ see formula~\eqref{13.47.001}, and~\eqref{eqn12.31.001},
for the asymptotics of $e^{tL}f.$

In~\cite{AthreyaBarlowBassPerkins2002, BassPerkins2002} Bass and
Perkins, et al. have employed methods, similar to our own, to study
operators of the form
\begin{equation}
  L_{\BP}=\sum_{ i,j=1}^{n}\sqrt{x_ix_j}a_{ij}(\bx)\pa^2_{x_ix_j}+\sum_{i=1}^nb_i(\bx)\pa_{x_i}
\end{equation}
acting on functions in $\cC^2_b(\bbR_+^n).$ Here $b_i(\bx)\geq 0$
along $b\bbR_+^n.$ They have also considered other degenerate
operators of this general type. Their main goal is to show the
uniqueness of the solution to the martingale Problem defined by
$L_{\BP}.$ To that end they introduce \emph{weighted} H\"older spaces,
which take the place of our anisotropic spaces. In the 1-dimensional
case, the weighted $\gamma$-semi-norm is defined by
\begin{equation}
  \bbr{f}_{\BP,\gamma}=\sup_{x\in\bbR_+;\, h>0}\left[\frac{|f(x)-f(x+h)|}
{h^{\gamma}}x^{\frac{\gamma}{2}}\right].
\end{equation}

They prove estimates for the heat kernels of  model operators,
equivalent to $L_{\bb,0},$ with respect to these H\"older
spaces. Under a smallness assumption on the off-diagonal elements of
the coefficient matrix $(a_{ij}(\bx)),$ they are able to control the error
terms introduced by replacing $L_{\BP}$ by the model operator
\begin{equation}
  L_{\BP,0}=\sum_{i=1}^n[x_ia_{ii}(\bzero)\pa_{x_i}^2+b_i(\bzero)\pa_{x_i}]
\end{equation}
well enough to construct a resolvent operator $(L_{\BP}-\mu)^{-1}$ for
$\mu>0.$  This suffices for their applications to the martingale
problem defined by $L_{\BP}.$ Notice that with this approach, only
``pure corner'' models are used, and no consideration is given to
operators of the form $L_{\bb,m}$ with $m>0.$ For domains much more
general than $\bbR_+^n$ it is difficult to see how to make such an
approach viable.

The operators we treat are somewhat more restricted, in that we take the
coefficients of the off-diagonal terms to have the form
$x_ix_ja_{ij}(\bx;\by).$ Our method could equally well be applied to operators
of the form considered by Bass and Perkins, i.e. with $x_ix_j$ replaced by
$\sqrt{x_ix_j},$ if we were to append smallness hypotheses for the off-diagonal
terms, similar to those they employ.  After slightly modifying the definitions
of the higher order H\"older norms to include certain increasing weights, many
of our results could be generalized to include certain non-compact cases.

Our aims were of a more analytic character, and take us far beyond
what is needed to show the uniqueness of the solution to the
martingale Problem. This lead us to consider such things as the higher
order regularity of solutions with smoother initial data, the analytic
extension of the semi-group in time, and the higher order mapping
properties of the resolvent operator. We also show the ellipticity of
the resolvent, with a gain of 2 derivatives with respect to the our
anisotropic H\"older norms. While this does not appear explicitly
in~\cite{BassPerkins2002}, a similar result, with respect to the
weighted H\"older norms, should follow from what they have proved.

\section{Outline of Text}
The book is divided in three parts:
\begin{enumerate}
\renewcommand{\labelenumi}{\Roman{enumi}.}
\item {\bf Wright-Fisher Geometry and the Maximum Principle:
    Chapters~\ref{c.mwc}-\ref{c.maxprin}.}  Chapter~\ref{c.mwc}
  introduces the geometric preliminaries needed to analyze generalized
  Kimura diffusions. In Chapter~\ref{s.nrmfrms} we show that
  coordinates
 $$(x_1,\dots,x_M; y_1,\dots,y_{N-M})$$ 
 can be introduced in the neighborhood of a boundary point of codimension $M$
 so that the boundary is locally given by $\{x_1=\cdots=x_M=0\}$ and the 
 second order purely normal part of $L$ takes the form
  \begin{equation}
    \sum_{j=1}^{M}x_i\pa_{x_i}^2.
  \end{equation}
This generalizes a 1-dimensional result
in~\cite{Feller1}. In Chapter~\ref{c.maxprin} we prove maximum principles for the
parabolic and elliptic equations,
\begin{equation}
  (\pa_t-L)u=g\text{ and }(\mu-L)w=f,\text{ respectively,}
\end{equation}
from which the uniqueness results follow easily. Of particular note is an
analogue of the Hopf boundary point maximum principle, which allows very
detailed analyses of the $\Ker L$ and $\Ker \oL^*.$
\item {\bf Analysis of Model Problems:
    Chapters~\ref{c.models1}--\ref{s.genmod}.}
    In Chapter~\ref{c.models1} we introduce the model problems and the solution
    operator for the associated heat equations. These operators,
    \begin{equation}
      L_{\bb,m}=\sum_{j=1}^{m}[x_j\pa_{x_j}^2+b_j\pa_{x_j}]+\sum_{l=1}^m\pa_{y_l}^2,
    \end{equation}
 act on functions defined on
    $S_{n,m}=\bbR_+^n\times\bbR^m,$ where $n+m=N,$ and give a good approximation for
    the behavior of the heat kernel  $(\pa_t-L)^{-1}$ in neighborhoods of
    different types of boundary points. We state and prove elementary features
    of these operators, that generalize results proved in~\cite{WF1d}, and show
    that the model heat operators have an analytic continuation to the right half
    plane:
    \begin{equation}
      H_+=\{t:\: \Re t>0\}.
    \end{equation}

In    Chapter~\ref{chap.holdspces} we introduce the degenerate H\"older spaces
on the spaces $S_{n,m},$ and their heat-space counterparts on 
$S_{n,m}\times [0,T].$  These are, in essence, H\"older spaces
defined by the incomplete metric on $S_{n,m}$ given by
\begin{equation}
  ds^2_{\WF}=\sum_{j=1}^n\frac{dx_i^2}{x_i}+\sum_{l=1}^m dy_l^2.
\end{equation}
We also establish the basic properties of these spaces.

Chapters~\ref{chap.1ddegen_ests}--\ref{s.genmod} are devoted to analyzing the
heat and resolvent operators for the model problems acting on the H\"older
spaces defined in Chapter~\ref{chap.holdspces}. This is a very long and tedious
process because many cases need to be considered and, in each case, many
estimates are required. Conceptually, however, these results are 
elementary. The estimates are pointwise estimates done in H\"older spaces,
which means one can vary a single variable at a time.  As the model heat
kernels are products of 1-dimensional heat kernels, this reduces essentially
every question one might want to answer to one of proving estimates for the
1-dimensional kernels. We call this the \emph{one-variable-at-a-time}
method.\index{one-variable-at-a-time method} In higher dimensions, the
resolvent kernel, which is the Laplace transform of the heat kernel, is
\emph{not} a product of 1-dimensional kernels. This makes it far more difficult
to deduce the mapping properties of the resolvent from its kernel, and explains
why we use the representation as a Laplace transform.

The proof of the estimates on the 1-dimensional heat kernels, defined by the
operators $x\pa_x^2+b\pa_x$ are given in Appendix~\ref{prfsoflems}. Analogous
results for the Euclidean heat kernel are stated in Chapter~\ref{s.eucmod}. The
proofs of these lemmas, which are  elementary, are left to the reader. A
notable feature of the estimates for the degenerate model problem is the fact
that the constants remain uniformly bounded as $b\to 0.$ This despite the fact
that the character of the heat kernel changes quite dramatically at $b=0,$
see~\eqref{eqn1.22.05} and~\eqref{eqn1.24.05}. This is also in sharp contrast to the
analysis of similar problems in~\cite{DaskHam}, where a positive lower bound is
assumed for the coefficient of the analogous vector field.

\item {\bf Analysis of Generalized Kimura Diffusions:
    Chapters~\ref{exstsoln0}--\ref{c.adjsmgrp}.}
This part of the book represents the culmination of all the work done up to 
this point. We consider a generalized Kimura diffusion operator $L$ defined on
a compact manifold with corners $P.$ In Chapter~\ref{exstsoln0} we prove the
existence of solutions to the heat equation
\begin{equation}
  (\pa_t-L)u=g\text{ in }P\times (0,T]\text{ with }u(p,0)=f(p),
\end{equation}
with data in $(g,f)\in \cC^{k,\gamma}_{\WF}(P\times [0,T])\times
\cC^{k,2+\gamma}_{\WF}(P).$ We show that the solution belongs to
$\cC^{k,2+\gamma}_{\WF}(P\times [0,T]),$

(Theorems~\ref{thm13.1} and~\ref{thm13.3}). The case $g=0$ provides a solution
to the Cauchy problem, but it is not optimal as regards either the regularity
of the solution, or the domain of the time variable, defects that are corrected
in Chapter~\ref{c.resolv}. The proof of these results is an intricate induction
argument, where we induct over the maximal codimension of $bP.$ This argument
allows us to handle one stratum at a time. The underlying geometric fact is a
``doubling theorem,'' which shows that any neighborhood, \emph{complementary} to the
highest codimension stratum of $bP,$ can be embedded into a manifold with
corners $\tP$ where the maximum codimension of $b\tP$ is one less than that of
$bP,$ (Theorem~\ref{dblthm}). This explains why we need to consider domains
well beyond those that can be easily embedded into Euclidean space.  

We first treat the lowest differentiability case ($k=0$) and then use
an extension of the contraction mapping theorem to towers of Banach
spaces (Theorem~\ref{thm14.0.1}), to obtain the mapping results for
$k>0.$ These results (even in the $k=0$ case) suffice to prove that
the graph closure of $L$ acting on $\cC^3(P)$ is the generator of
strongly continuous semi-group in $\cC^0(P).$

We next consider the operators $L_{\gamma},$ defined as $L$ acting on the
domain $\cC^{0,2+\gamma}_{\WF}(P). $ As a map from $\cC^{0,2+\gamma}_{\WF}(P)$
to $\cC^{0,\gamma}_{\WF}(P),$ $L_{\gamma}$ is a Fredholm operator of index
0. In Chapter~\ref{c.resolv} we use essentially the same parametrix
construction as used to prove Theorem~\ref{thm13.1} to prove the existence of
the resolvent operator
$$(\mu-L_{\gamma})^{-1}:\cC^{k,\gamma}_{\WF}(P)\longrightarrow \cC^{k,2+\gamma}_{\WF}(P),$$
for $\mu$ in the complement of discrete set lying in a conic neighborhood of
$(-\infty,0].$ These are the expected ``elliptic'' estimates for operators with
this type of degeneracy. In fact the spectrum of $L$ acting on
$\cC^{k,\gamma}_{\WF}(P)$ does not depend on $k$ or $\gamma,$ as the resolvent
is compact and all eigenfunctions belong to $\cC^{\infty}(P).$ Using the
analyticity properties of the resolvent, we give an alternate construction,
using a contour integral, for the semi-group, acting on
$\cC^{0,\gamma}_{\WF}(P):$
\begin{equation}
  e^{tL}=\frac{1}{2\pi i}\int\limits_{\Gamma_{\alpha}}e^{t\mu}(\mu-L)^{-1}d\mu,
\end{equation}
where $\Gamma_{\alpha}$ bounds a region of the form $|\arg\mu|>\pi-\alpha,$ and
$|\mu|>R_{\alpha},$ (Theorem~\ref{thm12.2.1.00}).  As we can take $\alpha$ to be
any positive number, this shows that the semi-group is holomorphic in the right
half plane.

Finally in Chapter~\ref{c.adjsmgrp} we give a good description of the nullspace
of $L_{\gamma},$ and show that the non-zero spectrum of $L_{\gamma}$ lies in a
half plane $\Re\mu<\eta<0.$ We also deduce various properties of the semi-group,
defined by the graph closure, $\oL,$ of $L,$ acting on $\cC^0(P).$ The adjoint
operator, $\oL^*,$ is defined on a domain $\Dom(\oL^*)\subset\cM(P),$ which is
not dense. Although we have not yet proved the compactness of the resolvent of
$\oL,$ we obtain a rather complete description of the null-space of $\oL^*.$
Using this we give the long time asymptotics for $e^{tL}f,$ assuming that
$f\in\cC^{0,\gamma}_{\WF}(P),$ for any $0<\gamma,$ as well as those for
$e^{t\oL^*}\nu,$ for $\nu$ in the closure of $\Dom(\oL^*).$

\item {\bf Proof of 1-dimensional Estimates: Appendix~\ref{prfsoflems}.} In the
  appendix we give careful proofs of the estimates for the degenerate,
  1-dimensional heat kernels used in the perturbation theory. These arguments
  are complicated by the fact that the heat kernel displays both the additive
  and multiplicative group structures on $\bbR_+:$ 
  \begin{equation}
    k^b_t(x,y)dy=
\left(\frac{y}{t}\right)^{b}e^{-\frac{x+y}{t}}\psi_b\left(\frac{xy}{t^2}\right)\frac{dy}{y}.
  \end{equation}
  The arguments involve Taylor's theorem, the asymptotic expansion of
  the heat kernel where $\frac{|\sqrt{x}-\sqrt{y}|}{\sqrt{t}}$ tends
  to infinity and Laplace's method. We obtain mapping properties for
  $b>0,$ with uniform constants as $b$ tends to zero. Using
  compactness of the embeddings, $\cC^{k,\gamma}_{\WF}\hookrightarrow
  \cC^{k,\tgamma}_{\WF},$ for $\tgamma<\gamma,$ e.g. we then extend
  these results to $b=0.$

\end{enumerate}    
\section{Notational Conventions}
We use $\bbR_+=[0,\infty),$ hence $\cC^{\infty}_c(\bbR_+)$ consists of smooth
functions on $\bbR_+,$ supported in finite intervals $[0,R].$
The right half plane is denoted\index{right half plane}
\begin{equation}
  H_+=\{t\in\bbC:\:\Re t>0\}.
\end{equation}
We let\index{$S_{n,m}$}
\begin{equation}
  S_{n,m}=\bbR_+^n\times\bbR^m.
\end{equation}
For $\phi\in (0,\frac{\pi}{2})$ we define the sector
\begin{equation}
  S_{\phi}=\{t\in\bbC:\: |\arg t|<\frac{\pi}{2}-\phi.
\end{equation}
\index{sector, $S_{\phi}$}
For $\bx,\bx'\in\bbR_+^n$ we let
\begin{equation}\index{$ \rho_s(\bx,\bx')$}
 \rho_s(\bx,\bx')= \sum_{j=1}^n |\sqrt{x_j} - \sqrt{x_j'}|;
\end{equation}
For $\by,\by'\in\bbR^m$ we let\index{$\rho_e(\by,\by')$}
\begin{equation}
\rho_e(\by,\by') = \sum_{k=1}^m |y_k - y_k'|.
\end{equation}
We then let
\begin{equation}
  \rho((\bx,\by),(\bx',\by'))= \rho_s(\bx,\bx')+\rho_e(\by,\by').
\end{equation}
We also use
\begin{equation}
  d_{\WF}((\bx,\by),(\bx',\by'))=\rho((\bx,\by),(\bx',\by'))
\end{equation}
When there is also a time variable
\begin{equation}
  d_{\WF}((\bx,\by,t),(\bx',\by',t'))=\rho((\bx,\by),(\bx',\by'))+\sqrt{|t-t'|}.
\end{equation}
\index{$d_{\WF}((\bx,\by,t),(\bx',\by',t'))$}
\noindent
If $\ba$ and $\bb$ are vectors in $\bbR^n,$ then $\ba<\bb,$ or $\ba\leq \bb$
means that
\begin{equation}
  a_j<b_j(\text{ or }a_j\leq b_j)\text{ for }j=1,\dots, n.
\end{equation}
We let $\bzero=(0,\dots,0),$ and $\bone = (1, \dots,
1)$\index{$\bzero$}\index{$\bone$} the dimension will be clear from
the context.

\part{Wright-Fisher Geometry and the Maximum Principle}
\chapter{Polyhedra and Manifolds with Corners}\label{c.mwc}
The natural domains of definition for generalized Kimura diffusions are
polyhedra in Euclidean space or, more generally, abstract manifolds with
corners. In order to set notation and fix ideas, we review this class of
objects here and discuss the main properties about them that will be needed
below. A more complete discussion can be found
in~\cite{merlosemwcrnr,MelroseAPS}.

The standard $N$-dimensional Euclidean space is denoted $\RR^N$. For any $n = 1, \ldots, N$, 
let us set $m = N - n$ and define the positive $n$-orthant in $\RR^N$ as the subset
\begin{equation}
S_{n,m} := \RR_+^{n} \times \RR^{n} = \{ (\bx,\by) \in \RR^n \times \RR^m: x_j \geq 0\ j = 1, \ldots, n\}.
\label{orthant}
\end{equation}

Recall the standard definition of a smooth manifold: A paracompact,
Hausdorff topological space, $M,$ is called an $N$-dimensional smooth
manifold if every point $p \in M$ has a neighborhood $\calU_p$ which
is identified homeomorphically with an open set $\calV_p$ around the
origin in $\RR^N.$ Here $p$ mapped to the origin, and such that the
identifications between these various subsets of $\RR^N$ are
diffeomorphisms. More specifically, if $\psi_p: \calU_p \to \calV_p$
is the homeomorphism, then for $p\neq q:$
\begin{equation}
\psi_p \circ \psi_q^{-1}: \psi_q(\calU_q \cap \calU_p) \longrightarrow
\psi_p(\calU_q \cap \calU_p)
\label{overlap}
\end{equation}
is a diffeomorphism. The mappings $\psi_p$ are sometimes called charts, and the compositions
$\psi_p \circ \psi_q^{-1}$ are called transition functions. 

Generalizing this, we say that $P$ is an $N$-dimensional manifold with corners up 
to codimension $n$ if for every $p \in P$, there is a neighborhood $\calU_p$ and a
homeomorphism $\psi_p$ from $\calU_p$ to a neighborhood of $0$ in $\RR_+^\ell
\times \RR^{N-\ell}$ for some $\ell \in \{0, \ldots, n\}$, with $\psi_p(p) =
0$, and such that the overlap maps, defined exactly as in \eqref{overlap}, are
diffeomorphisms.  (Recall that a mapping between two relatively open sets in
$\RR_+^n \times \RR^{N-n}$ is a diffeomorphism if it is the restriction of a
diffeomorphism between two absolute open sets in $\RR^N$.) If such a map
exists, we say that a point $p$ lies on a corner of codimension $\ell$. The
fact that the codimension associated to any point is well-defined is a basic
fact from differential topology known as the invariance of domain lemma.

Using these local charts, we can meaningfully define all the usual flora and
fauna of differential geometry in this setting.  For example, we can discuss
smooth functions, vector fields, differential forms, etc., simply by
identifying these objects using the charts to the familiar objects of each type
on the orthants in $\RR^N$. The fact that such objects are well-defined
follows from the fact that the transition functions between the charts are
diffeomorphisms and each of these classes of objects (smooth functions, etc.)
are preserved by diffeomorphisms.

It follows directly from this definition that the set of points $p$ lying on a
corner of codimension\index{corner of codimension $\ell$} $0<\ell\leq N$
constitute a possibly open, and possibly disconnected, smooth manifold,
$\Sigma_{\ell},$ of dimension $N-\ell.$ If $\ell$ is strictly less than the
maximal codimension, $n,$ then $\Sigma_{\ell}$ is open and
\begin{equation}
\overline{\Sigma}_{\ell} =\cup_{j=\ell}^{n}\Sigma_{j},
\end{equation}
where the union here is only over the union of corners $\Sigma_j$ such that
$\Sigma_j \cap \overline{\Sigma}_\ell \neq \emptyset$.  Each component of
$\Sigma_{\ell}$ is called a corner of $P$ of codimension $\ell.$ The corners of
codimension one are the boundary hypersurfaces, which we sometimes also call
the faces, of $P$.  We henceforth make the global hypothesis that the closure
of each connected corner of $P$, of any codimension $\ell$, is itself an
embedded manifold with corners at most up to codimension $N-\ell$.  The
important part of this hypothesis is the embeddedness of this closure. In the
sequel we consider components of the boundary stratification\index{components
of the boundary stratification} to be these closed manifolds with corners.

\begin{definition} The stratum of $bP$ of codimension $\ell$ consists of the
  closures, in $P,$ of the connected components of $\Sigma_{\ell}.$ We call
  these connected subsets the components of the stratum of $bP$ of codimension
  $\ell,$ or more briefly, components of $bP.$ \index{components of $bP$}
\index{stratum of $bP$}
\end{definition}

We now prove several useful facts about this class of objects. In the
following, fix any manifold with corners $P$, and denote by $\{H_1, \ldots, H_A
\}$ its set of boundary hypersurfaces. Note that every corner of $P$ of
codimension $\ell$ arises as a component of an intersection $H_{i_1} \cap
\ldots \cap H_{i_\ell}$. For simplicity we usually assume that $P$ is compact,
though the results below extend easily to the noncompact setting (sometimes
with with a few extra hypotheses).

\begin{lemma} If $H$ is any boundary hypersurface of $P$, then there is a
  smooth vector field $V$ defined in a neighborhood of the closure of $H$
  which is inward-pointing, nowhere vanishing, transverse to $H$, and which is
  also tangent to all other boundary faces and corners at $bH$.
\end{lemma}
\begin{proof} It is easy to construct such a vector field near the origin in
  $\RR_+^n \times \RR^{N-n}$: using the coordinates $(\bx,\by)$, normalized so
  that $H$ is locally $\{x_1=0\}.$ In this chart we let $V$ be the vector field
  $\pa_{x_1}.$ Now choose a finite number of coordinate charts $\calU_\alpha,$
  which provide an open cover of $H$ and such that each $\calU_\alpha$ is
  mapped by a chart to an relative open ball in some orthant of $\RR^N.$ Let
  $\{\chi_\alpha\}$ be a partition of unity subordinate to this open cover. For
  each $\alpha$, let $V_\alpha$ be the coordinate vector field in
  $\calU_\alpha$ defined above, and set $V = \sum \chi_\alpha V_\alpha.$ This
  vector field clearly satisfies the conclusions of the lemma.
\end{proof}

\begin{lemma} For any boundary hypersurface $H$, there is a neighborhood $\calU$ of $H$
in $P,$ which is diffeomorphic to a product $H \times [0,1)$. 
\end{lemma}
\begin{proof}
  Let $V$ be the vector field defined above relative to the boundary
  hypersurface $H$. Assuming that $H$ is compact, there exists some $\e > 0$
  such that the flow by the one-parameter family of diffeomorphisms $\Phi_t$
  associated to the vector field $V$ is defined on all of $H$ for $0 \leq t <
  \e$. This gives a diffeomorphism between $H \times [0,\e)$ and some
  neighborhood $\calU$ of $H$.  A rescaling in $t$ gives a diffeomorphism from
  $H \times [0,1)$.
\end{proof}

\begin{lemma}
  Let $K$ be a corner of codimension $\ell$ in $P$. Let $B_+^\ell = \{x: |x| <
  1\} \cap \RR_+^\ell$.  Then there is a diffeomorphism between $K \times
  B_+^\ell$ and a neighborhood $\calU$ of $K$ in $P$.
\end{lemma}
\begin{proof}
Let $H_1, \ldots, H_\ell$ be the boundary hypersurfaces which intersect along $K$. Let
$V_j$ be an inward-pointing vector field transversal to $H_j$, as constructed above.  
For any $x \in B_+^\ell$ with $|x| < \e$, let $V_x = \sum x_j V_j$ and $\Phi_{x}$ be the 
time one flow of the one-parameter family of diffeomorphisms associated to $V_x$. Then 
\[
K\times \e B_+^\ell \ni (q,x) \longmapsto \Phi_{x}(q)
\]
is the desired mapping.
\end{proof}

% We emphasize that as a consequence of this last result, some neighbourhood $\calU$
% of each corner $K$ is diffeomorphic to a \emph {product} $K \times B_+^\ell$.  There are more general
% classes of spaces, the smoothly stratified spaces, where this product structure is only
% required to be local in $K$, or in other words, the neighbourhood is the total space of a fibration over $K$. 

\begin{lemma} For each boundary face $H$ of $P$ there is a function $\rho_H$
  which is everywhere positive in $P \setminus H$ and which vanishes on $H$ and
  has nonvanishing differential there.  Any such function is called a boundary
  defining function for $H$.\index{boundary defining function}\index{defining
    function}
\end{lemma}
\begin{proof}
  This may be constructed using a partition of unity exactly as in the first
  lemma. Alternately, if $\calU$ is a neighborhood of $H,$ which has been
  identified diffeomorphically with a product $H \times [0,1)$, then we can set
  $\rho_H$ to equal the projection onto the second coordinate in this
  neighborhood. This function is then extended to the rest of $P$ as a strictly
  positive function using a partition of unity.
\end{proof}

There is one other construction which plays an important role below. We present it in a sequence of two lemmas.
\begin{lemma} For each boundary face $H$, there is a new manifold with corners
  $\widetilde{P}$ which is obtained by doubling $P$ across $H$.
\end{lemma}
\begin{proof}
Let $\widetilde{P}$ be the disjoint union of two copies of $P$ identified by the identity mapping along
$H$. If we wish to work within the setting of oriented manifolds, then one copy should be $P$ itself
and the other $-P$, i.e.\ $P$ with the opposite orientation. We give this space the structure of
a manifold with corners by specifying a collection of coordinate charts. First, use all charts of $P$ which
are disjoint from $H$.  Then, near any point $p \in \overline{H}$,
choose a chart $\psi: \calU \to (\RR_+)^\ell \times \RR^{N-\ell-1}$ for $H$ as a manifold with corners, and 
define the extended chart $\tilde{\psi}: \calU \times (-\e,\e) \to (\RR_+)^\ell \times \RR^{N-\ell-1} \times\RR$
by $\tilde{\psi}(q,t) = (\psi(q),t)$.  This provides a chart around all points of the image of $H$ in $\widetilde{P}$,
and it is clear that the transition functions are smooth. 
\end{proof}

\begin{lemma} Let $P$ be a compact manifold with corners. Let $K$ be the corner of 
maximal codimension $n$.  Then there is a new space $\widetilde{P}$, 
which is a manifold with corners only up to codimension $n-1$, obtained by
`doubling' $P$ across $K,$ such that $P\setminus K$ is identified with an open
subset of $\tP.$
\end{lemma}
\begin{proof}
Notice that $K$ is a closed, and possibly disconnected, manifold of dimension $N-n$. For simplicity we
assume that $K$ is connected, but removing this only complicates the notation slightly. 
We first define the radial blowup of $P$ along the submanifold 
$K$.  Let $B$ be the unit ball around $0$ in $\RR_+^n$, which we describe via the polar coordinates $(r,\theta)$
where $0 \leq r < 1$ and $\theta \in S^n_+ = \{\theta \in \RR^{n+1}: |\theta| = 1, \ \theta_j \geq 0\ \forall\, j\}$. 
This coordinate system is degenerate at $r=0$ since the entire spherical orthant $\{0\} \times S^n_+$ is collapsed
to a point.  We define the radial blowup of $B$ at $0$, denoted $\widehat{B} = [B,\{0\}]$, by replacing the origin by a 
copy of $S^n_+$; in other words, $\widehat{B}$ is simply a copy of the cylinder $[0,1) \times S^n_+$. 

Next, define the radial blowup of $P$ along $K$, denoted $\widehat{P} = [P;K]$. In terms of the identification 
of the tubular neighborhood $\calU$ of $K$ in $P$ as $K \times B$, we replace each $B$ by $\widehat{B}$.  
This space is still a manifold with corners up to codimension $n$; the codimension $n$
corners are now the products of the vertices of $S^n_+$ with $K$. There is a new, possibly disconnected, 
hypersurface boundary, $H_0 = K \times S^n_+$ (at $r=0$).

The final step is to define $\widetilde{P}$ to be the double of $\widehat{P}$
across the face $H_0$ in the sense of the previous lemma.  This space no longer
has any corners of codimension $n$. From the construction it is clear that
$P\setminus K$ embeds in $\tP$ as an open set.
\end{proof}

An important class of manifolds with corners is provided by the regular
polyhedra $P \subset \RR^N$.  Recall first that a polyhedron is a domain in
$\RR^N$ whose boundary lies in a union of hyperplanes and is a finite union of
regions $\{Q_i\}$, each itself a polyhedron in a hyperplane $H_i \subset
\RR^N$. Of particular interest to us here are the convex polyhedra, which are
by definition determined by a finite number of affine inequalities:
\[
P = \{z \in \RR^N: z \cdot \omega_j \geq c_j, \ j = 1, \ldots, A\},
\]
where $\{\omega_1, \ldots, \omega_A\} \subset \RR^N$ is a finite set of vectors, and the $c_j$ are
real numbers.  The various faces and corners of $P$ are the subsets determined by replacing any
subcollection of these inequalities by the corresponding equalities. 

\index{regular polyhedron}
Amongst the convex polyhedra we distinguish the subclass of regular convex
polyhedra $P$. By definition, $P$ is a regular convex polyhedron if it is
convex and if near any corner, $P$ is the intersection of no more than $N$
half-spaces with corresponding normal vectors $\omega_j$ linearly independent.  It
is clear from these definitions that any regular convex polyhedron is a
manifold with corners. Namely, if $p \in P$ lies in a corner of codimension
$\ell$, hence is an element of $\ell$ independent hyperplanes $\{z \cdot
\omega_j = c_j\}$, then there is an affine change of variables which carries a
neighborhood of $p$ in $P$ to a neighborhood of $0$ in $\RR_+^\ell \times
\RR^{N-\ell}$.  A polyhedron that is not regular or non-convex, when
endowed with the smooth structure given by the embedding into Euclidean space,
does not satisfy one or more of the defining properties of manifolds with
corners. It should be noted, that unlike convex polyhedra, which are always
contractible, manifolds with corners can have very complicated topologies.

\chapter{Normal Forms and WF-Geometry}\label{s.nrmfrms}
We are now in a position to define the general class of elliptic Kimura operators
on a manifold with corners $P$. These, and their associated heat operators, are
our main objects of study.  The definition we give is coordinate-dependent,
but we indicate  how to formulate this in a coordinate independent way.
Our goal in this chapter is to show that there is a local normal form for any 
operator $L$ in this class which shows that it can be regarded as a perturbation 
in a small enough neighborhood of one of the model Kimura operators $L_{\bb, m}$ 
introduced in the introduction. The reduction to this normal form is assisted by use
of geometric constructions with respect to a  singular Riemannian metric 
on $P$.  This metric (or any one equivalent to it) is also instrumental in the
formulation of the correct function spaces on which we let $L$ act; this is the
topic of Chapter~\ref{chap.holdspces}.  The normal form in this multi-dimensional
setting generalizes the normal form, originally introduced by Feller, which is 
fundamental in the analysis of the $1$-dimensional case in \cite{WF1d}. 

\begin{definition} Let $P$ be a manifold with corners. A second order operator
  $L$ defined on $P$ is a called a generalized Kimura diffusion operator
  \index{generalized Kimura diffusion operator} if it satisfies the following
  set of conditions:
\begin{itemize}
\item[i)] $L$ is elliptic in the interior of $P$.
\item[ii)] If $q$ is a boundary point of $P$ which lies in the interior of a 
corner of codimension $n$, then there are local coordinates $(\bx, \by)$,
$\bx = (x_1,\dots,x_n)$, $\by = (y_1,\dots,y_m)$ so that in the neighborhood
\[
\calU=\{ 0 \geq x_j < 1\ \forall\, j \leq n,\  |y_k| < 1\ \forall\, k \leq m\} 
\]
the operator takes the form
\begin{multline}
L=\sum_{j=1}^{n}a_{ii}x_i\pa_{x_i}^2+\sum_{1\leq i\neq j\leq  n}x_ix_j a_{ij}\pa^2_{x_i x_j} +\\
\sum_{i=1}^n\sum_{k=1}^mx_ib_{ik}\pa^2_{x_i y_k}+ \sum_{k,l=1}^mc_{kl}\pa^2_{y_k y_l}+V.
\end{multline}
For simplicity we assume that all coefficients lie in $\calC^\infty(P).$ We
also assume that $(a_{ij})$ and $(c_{kl})$ are symmetric matrices.
\item[iii)] The vector field $V$ is inward pointing at all boundaries and corners of $P$; 
\item[iv)] The matrices $(a_{ii})$ and $(c_{kl})$ are strictly positive definite. 
\end{itemize}
\end{definition}

The distinguishing features are the simple vanishing of the coefficients of the
second order terms normal to the boundary and the ellipticity in all other
directions. This leads to a coordinate-invariant definition. Recall that any
second order operator $L$ in the variables $\bx, \by$ has a principal symbol
\begin{multline*}
\sigma_2(L)(\bx, \by, \bxi,\bseta) =  \\
\sum_{i, j =1}^n \hat{a}_{ij}(\bx, \by) \xi_i \xi_j + \sum_{i=1}^n \sum_{k=1}^m \hat{b}_{ik}(\bx, \by) \xi_i \eta_k
+ \sum_{k, l = 1}^m \hat{c}_{kl}(\bx, \by) \eta_k \eta_l.
\end{multline*}
Here $\bxi$ and $\bseta$ are the dual (cotangent) variables associated to $\bx$ and $\by$. 
We require then that $\sigma_2(L)(\bx, \by, \bxi, \bseta)$ is nonnegative for all 
$(\bxi, \bseta)$, and strictly positive when all $x_j > 0$ and $(\bxi,\bseta)
\neq (\bzero, \bzero)$. Furthermore its characteristic set 
\[\mbox{Char}(L) = \{(\bx,\by,\bxi, \bseta): \sigma_2(L) (\bx,\by,\bxi, \bseta) = 0\}
\]
is equal to the set of all conormal vectors to $bP$, or more precisely to the set of all points 
$(q, \nu)$ where $q \in bP$ and $\nu \in T^*_qP$ vanishes on the tangent spaces of all
boundary hypersurfaces which contain $q$. Finally, we require that $\sigma_2(L)$ vanishes
precisely to first order at this set. 

It is immediate to see that the any operator which satisfies the first
definition satisfies all the coordinate-invariant conditions above. For the
converse, observe that in any local coordinate system, at a point $q$ in the
interior of the face where $x_j = 0$, the conormal is spanned by the covector
$\nu = dx_j$, i.e.\ $\xi_j = 0$ and all other $\xi_i$ and $\eta_k$ vanish.
Hence if we write $\sigma_2(L)(\bx,\by,\bxi,\bseta)$ as a quadratic form as
above, then for this $(q,\nu)$, $\hat{a}_{ij}(q) = 0$ for all $i \leq n$, and
this vanishing is simple. This gives that $\hat{a}_{jj}(\bx,\by) = x_j
a_{jj}(\bx,\by)$ and for $i \neq j$, $\hat{a}_{ij}(\bx,\by) = x_i x_j
a_{ij}(\bx,\by)$.  The rest of the verification is straightforward. As noted in
the introduction, we could enlarge our family somewhat by allowing terms of the
form
\begin{equation}
  \sqrt{x_ix_j}a_{ij}\pa^2_{x_ix_j}\text{ and }
\sqrt{x_i}b_{ik}\pa^2_{x_iy_k},
\end{equation}
if we append a smallness hypotheses on the coefficients
$\{a_{ij}(\bx;\by), b_{ik}(\bx;\by)\}.$ Operators of this type were analyzed by
Bass and Perkins. Indeed we can consider operators of this more general type
where the coefficients are smooth functions of the variables
$(\sqrt{x_1},\dots,\sqrt{x_n};y_1,\dots, y_m).$ We leave these generalizations
to the interested reader.

Let $H$ be a boundary hypersurface of $P,$ with  $\rho$ a $\cC^2$-function,
vanishing on $H.$ The
first order part $V$ of $L$ is tangent to $H$ if and only if
\begin{equation}
  V\rho\restrictedto_{\rho=0}=0.
\end{equation}
From the form of the operator it is easy to see that this is true if and only
if
\begin{equation}
  L\rho\restrictedto_{\rho=0}=0.
\end{equation}
In this case we say that \emph{$L$ is tangent to $H.$}\index{$L$ is tangent to}
If $L$ is tangent to $H,$ then there is a naturally induced operator $L_H$
acting on $\cC^{\infty}(H).$
\begin{definition} If $L$ is tangent to $H,$ then the restriction of $L$ to $H,$
  $L_H,$ is given by the prescription
\[
  L_H u:=(L\tu) \restrictedto_H \ \forall\ u \in \cC^\infty(H), 
\]
here $\tu$ is any smooth extension of $u$ to a neighborhood of $H$ in $P.$
The operator $L_H$ is a generalized Kimura diffusion operator on the manifold with
corners $H$.\index{$L_H$}\index{restriction of $L$ to $H$}
\end{definition}
As we have already mentioned, it is very helpful to consider a singular
Riemannian metric on $P$ such that the second order terms of $L$ agree with
those in the Laplacian for $g$.
%Kimura operator can be reduced, at each boundary point, to one of the form $x\pa_x^2+V.$ 
The change of variables which brings $L$ into a normal form is then simply a Fermi
coordinate system for this metric.
To do this, we regard the principal symbol of $L$ as a dual metric on the cotangent
bundle; we write this as
\begin{equation}
  g^{-1}=\left(\begin{matrix} XA_d & 0\\ 0& C\end{matrix}\right)+
\left(\begin{matrix} XA_oX &XB\\ B^tX& 0\end{matrix}\right).
\label{co-metric}
\end{equation}
Here, $B$ is an $n\times m$ matrix, $C$ is a positive definite $m\times m$ matrix, $A_o$ is
off-diagonal (and has all diagonal entries equal to zero) and $X$ and $A_d$ 
are the diagonal matrices given by
\begin{equation}
  X=\left(\begin{matrix} x_1&0&\dots&0\\ 0&x_2&\dots&0\\
&&\ddots&\\
0&0&\dots&x_n\end{matrix}\right),\quad
 A_d=\left(\begin{matrix} a_{11}&0&\dots&0\\ 0&a_{22}&\dots&0\\
&&\ddots&\\
0&0&\dots&a_{nn}\end{matrix}\right).
\end{equation}
We then compute that the inverse of this matrix, i.e.\  the metric tensor itself, takes the form:
\begin{equation}
  g=\left(\begin{matrix} X^{-1}\tA_d +\tA_o& \tB\\ \tB^t& \tC\end{matrix}\right).
\end{equation}
The block submatrices here are all smooth, with $\tA_d$ positive definite and diagonal, 
$\tA_o$ having vanishing diagonal entries and $\tC$ positive definite. 

The metric $g$ is singular when any $x_j$ vanishes, but it can be desingularized by changing variables via 
\begin{equation}
d\zeta_j=\frac{dx_j}{\sqrt{x_j}},\ \text{ i.e. }\zeta_j=2\sqrt{x_j}.
\end{equation}
In these coordinates, the metric takes the form
\begin{equation}
g =\sum_{i=1}^n\ta_{ii} \, d\zeta_i^2+ \sum_{i,j=1}^n\ta_{ij}\zeta_i\zeta_j\, d\zeta_i d\zeta_j+\sum_{i=1}^n\sum_{k=1}^m
\tb_{ik}\zeta_i \, d\zeta_i dy_k+\sum_{k,l=1}^m\tc_{kl}\, dy_kdy_l.
\end{equation}
The coefficients $\tilde{a}_{ij}$, $\tilde{b}_{ik}$ and $\tilde{c}_{kl}$ are smooth functions of 
$(\zeta_1^2,\dots,\zeta_n^2;y_1,\dots,y_m)$.  This implies that the metric $g$ can be extended by even 
reflection across each hyperplane $\zeta_j = 0$ to define a smooth, non-degenerate metric on a full 
neighborhood $\calV$ of the origin in $\RR^{n+m}$. This means that each boundary hypersurface $S_i = 
\{ \zeta_i = 0\}$ is the fixed point set of the locally defined isometry $\zeta_i \to -\zeta_i$, and hence 
$S_i$, and any intersection $S_J = \cap_{j \in J} S_{i_j}$, is totally geodesic.  Let $S$ be the intersection of 
all the $S_i$, i.e.\ the corner of maximal codimension which intersects $\calV$; for simplicity,
assume that its codimension is $n$. 

Assuming that $\calV$ is sufficiently small, then for each $p \in \calV$ there is a unique closest
point $\Pi_J(p)$ to $p$ in $S_J$; if $S_J = S$, the maximal codimension corner, then we write this 
projection as $\Pi_S$. The signed distance 
function $\rho_i(p) = \sdist_g(p,S_i)$ to each hypersurface is smooth. Abusing notation slightly,
let $\by = (y_1, \ldots, y_m)$ be the composition of the original set of local coordinates restricted to $S$ 
with the projection $\Pi_S$. Then $(\rho_1, \ldots, \rho_n, y_1, \ldots, y_m)$ is a new set of smooth
local coordinates in $\calV$.

Let us compute the metric coefficients of $g$ with respect to these coordinates. In fact,
we first compute the coefficients for the corresponding co-metric, i.e.\ the entries of the matrix
\begin{equation}
g^{-1}=\left(\begin{matrix}\langle d\rho_i,d\rho_j\rangle & \langle
      d\rho_i,d y_l\rangle\\
\langle dy_k,d\rho_j\rangle & \langle
      dy_k,d y_l\rangle\end{matrix}\right).
\end{equation}
From general considerations, since each of the $\rho_i$ are distance functions, we have 
\[
\langle d\rho_i,d\rho_i\rangle=1, \quad i=1,\dots,n,
\]
in the entire neighborhood $\calV$. 
Now, because of the reflectional symmetries, $d\rho_j$ is clearly orthogonal to $d\rho_i$ 
when either $\rho_i$ or $\rho_j$ equal $0$, and similarly, $d\rho_i$ is orthogonal to $dy_k$ 
when $\rho_i = 0$; this means that there are smooth functions $\talpha_{ij}$ and $\tbeta_{il}$ such that 
\[
\langle d\rho_i,d\rho_j\rangle=\rho_i\rho_j\talpha_{ij}, \ \ \mbox{and}\  \   \langle d\rho_i,dy_k\rangle=\rho_i\tbeta_k. 
\]

Inserting these expressions into the matrix above and changing coordinates by
setting $\rho_j=2\sqrt{x_j}$, we obtain the matrix of coefficients for this
co-metric. This has the form \eqref{co-metric} with $A_d =
\mbox{Id}_n$. Finally, taking the inverse of this matrix gives the normal form
we are seeking.  We summarize this in the
\begin{proposition}\label{p.nrmfrm} Let $q \in bP$ lie in a boundary face of codimension $n$. 
Then there is a neighborhood $\calU$ of $q$ and smooth local coordinates $(\bx, \by)$ 
in this neighborhood, with $q$ corresponding to $(\bzero;\bzero)$, in terms of which $L$ takes the form
\begin{multline}\label{Lnrmfrm}
L=\sum_{i=1}^nx_i\pa_{x_i}^2+ \sum_{1\leq k,l\leq   m}c'_{kl}\pa_{y_k}\pa_{y_l} + \\
\sum_{1\leq i\neq j\leq   n}x_ix_ja'_{ij}\pa_{x_i}\pa_{x_j}+\sum_{i=1}^n\sum_{l=1}^mx_ib'_{il}\pa_{x_i}\pa_{y_l}+V.
\end{multline}
Here $V$ is an inward pointing vector field, and $(c'_{kl}(\bx,\by))$ is a smooth family of positive definite matrices.
\end{proposition}

\begin{definition} The coordinates introduced in this proposition are called
  \emph{adapted local coordinates} centered at $q.$\index{adapted local coordinates}
\end{definition}
We have written this expression in such a way as to emphasize that the first two 
sums on the right are in some sense the principal parts of $L$, and the other second order terms should
be regarded as lower order perturbations. The body of our work below is devoted to showing
that this is exactly the case.   On the other hand, the first order term $V$ (or at least the
part of it which is not tangent to $bP$) is definitely not a lower order perturbation. 

The key point in this normal form is that all of the coefficients of the `leading' terms
$x_i\pa_{x_i}^2$ are simultaneously equal to $1$.  By a linear change of the $\by$ coordinates
we can also make $c_{kl}'(\bzero, \bzero) = \delta_{kl}$.  Hence restricting to an even smaller 
neighborhood, and writing the normal part of $V$ at $(\bzero, \bzero)$ as $\sum_{i=1}^n b_i \del_{x_i}$,
then we see that 
\begin{equation}\label{mnbbmodel}
 L_{\bb,m} :=\sum_{i=1}^n \left( x_i\pa_{x_i}^2+{b_i}\pa_{x_i} \right)+\sum_{k=1}^m \pa_{y_k}^2
\end{equation}
should provide a good model for $L$. Note that since $V$ is inward-pointed, we have
$b_i \geq 0$ for each $i$. 

\chapter[Maximum Principles]{Maximum Principles and Uniqueness Theorems}\label{c.maxprin} 
One of the most important features associated to any scalar parabolic or
elliptic problem is the use of the maximum principle.  Although this seems to
give only qualitative properties of solutions, it can actually be used to
deduce many quantitative results, including even the parabolic Schauder
estimates. On a more basic level, it is the key ingredient in proving
uniqueness of solutions to such an equation. We now develop the maximum
principle and its main consequences, both for the model operators $\del_t -
L_{b,m}$ on an open orthant, and for the general Kimura diffusion operators
$\del_t - L$ on a compact manifold with corners, as well as their elliptic
analogues.  This generalizes the results for the one-dimensional case in
\cite{WF1d}. Of particular note in this regard is a generalization of the Hopf
boundary point maximum principle, given in Lemma~\ref{hopfmaxpriple}. This
result allows us to precisely describe the nullspace of $L$ and its adjoint $L^*.$

\section{Model Problems}
We begin with maximum principles for the model operators.
\begin{proposition}\label{prop.maxPmod}
Suppose that $u$ is a subsolution of the model Kimura diffusion equation $\del_t u \leq L_{\bb,m}u$ on 
$[0,T] \times S_{n,m}$ such that
\[
u \in \calC^0( [0,T] \times S_{n,m}) \cap \calC^1( (0,T] \times S_{n,m}),
\]
and $u \in \calC^2$ away from the boundaries of $S_{n,m}$ for $t > 0$. Suppose also that
$x_j \del_{x_j}^2 u \to 0$ as $x_j \to 0$ for each $j \leq n$. Suppose finally that
\[
|u(\bx,\by,t)| \leq A e^{a (|\bx| + |\by|^2)}
\]
for some $a, A > 0$. Then
\[
\sup_{[0,T] \times S_{n,m}} u(\bx,\by,t) = \sup_{S_{n,m}} u(\bx,\by,0). 
\]
\end{proposition}
\begin{proof}
We show that if $u(\bx,\by,t) \leq 0$, then $u(\bx,\by,t) < 0$ for all $0 < t \leq T$. 

It is straightforward to check that the function $e^{x/(\tau-t)}$ on $[0,\tau)_t \times \RR_+$ is
a solution of the equation $(\del_t - L_0)u = 0$. Using the relation $\del_x^k L_0 = L_k \del_x^k$,
we obtain that
\[
\del_x^k e^{x/(\tau-t)} = \frac{1}{(\tau-t)^k} e^{x/(\tau-t)}
\]
is a solution of $(\del_t - L_k)u = 0$, and hence also, if $\mathbf{k}$ is the multi-index $(k_1, \ldots, k_n) \in \NN^n$, then
\[
U_{1,\tau,\mathbf{k}}(x,t) = \sum_{j=1}^n (\tau - t)^{-k_j-1} e^{x_j/(\tau-t)}
\]
solves $(\del_t - L_{\mathbf{k}})U_{1,\tau,\mathbf{k}} = 0$.  Suppose that $k_j > b_j$ for each $j$. Then 
\[
(\del_t - L_{\bb,m}) U_{1,\tau,\mathbf{k}} = \sum_{j=1}^n (k_j - b_j) x_j \del_{x_j} U_{1,\tau,\mathbf{k}} > 0,\ 0 \leq t < \tau.
\]
Nex define
\[
U_{2,\tau}(\by,t) = \frac{1}{(\tau-t)^{m/2}} e^{|y|^2/2(\tau-t)}.
\]

Finally, given the solution $u$ in the statement of the proposition, set
\[
v(\bx,\by,t) = u(\bx,\by,t) - \epsilon_1 U_{1,\tau,\mathbf{k}}(\bx,t) - \epsilon_2 U_{2,\tau}(\by,t) + \epsilon_3 \frac{1}{1+t}. 
\]
We see that
\[
(\del_t - L_{\bb,m}) v  < 0,
\]
so that $v$ is a strict subsolution of this equation.  Let 
\[
D_R = [0,\tau'] \times B_R(0)\ \mbox{and}\ D_R' = \{0\} \times B_R(0) \cup [0,\tau'] \times bB_R(0),
\]
$B_R(0)$ is the ball of radius $R$ around the origin in $\RR_+^n \times \RR_m$ and $0 < \tau' < \tau$. We claim
that the supremum of $v$ on $D_R$ is attained on $D_R'$ but not at any one of the boundaries
where any $x_j = 0$. The fact that this supremum must occur on $D_R'$ follows from the standard 
maximum principle.  If the supremum were to occur when $x_j = 0$, then using that $u \in \calC^1$ 
up to this boundary, we see that the corresponding derivative $\del_{x_j} v \geq 0$ at this point which
is impossible. This proves the claim.   Finally, since $u$ grows no faster than $e^{a(|x| + |y|^2)}$, we 
can choose $1/\tau < a$ and $R$ sufficiently large to ensure that $v$ is as negative as we wish on the 
entire side boundary $(0,\tau'] \times bB_R(0)$. 

We conclude that 
\[
u(\bx,\by,t) \leq \epsilon_1 U_{1,\tau,k}(\bx,0) + \epsilon_2 U_{2,\tau}(\by,0) - \epsilon_3,
\]
for any $\epsilon_1, \epsilon_2, \epsilon_3 > 0$. Now let these parameters tend to zero to see that $u \leq 0$ when $t > 0$,
as desired. 
\end{proof}

\section[Kimura Diffusion Operators]{Kimura Diffusion Operators on Manifolds with Corners}
The corresponding result for a general variable coefficient Kimura diffusion,
$L$ on a manifold with corners, $P,$
requires slightly stronger hypotheses. We let $\cD^2_{\WF}(P)$ denote a certain
subspace of $\cC^1(P)\cap\cC^2(\Int P)$ adapted to the degeneracies of $L.$ In
a neighborhood, $U,$ of a boundary point $p_0$ of codimension $M$ we introduce local
coordinates
$$(x_1,\dots,x_M,y_1,\dots,y_{N-M}),$$
so that the stratum of the boundary through $p_0$ is locally given by
\begin{equation}
  \Sigma\cap U=\{(\bzero,\by):\: |\by|<1\}.
\end{equation}
 
A function $f\in\cC^1(\overline{U})\cap\cC^2(U)$ belongs to $\cD^2_{\WF}(U)$ if, for $1\leq i,j\leq M,$ and
$1\leq l,m\leq N-M$ the functions
\begin{equation}
  x_i\pa_{x_i}^2f,\,  x_ix_j\pa_{x_i}\pa_{x_j}f,\, x_i\pa_{x_i}\pa_{y_l}f,\,
\pa_{y_l}\pa_{y_m}f
\end{equation}
extend continuously to $bP\cap U,$ with the first three types of expressions
vanishing whenever $x_i$ or $x_j$ vanishes. These conditions are clearly
coordinate invariant. 
\begin{definition}
A function in $\cC^1(P)\cap\cC^2(\Int P)$ belongs to
$\cD^2_{\WF}(P)$ if its restrictions to neighborhoods of boundary points belong
to each of these local $\cD^2_{\WF}$-spaces.
\end{definition}\index{$\cD^2_{\WF}$}

Our first result shows that on $\cD^2_{\WF}(P),$ a Kimura diffusion is a
dissipative operator.
\begin{lemma}\label{lem3.0.7.05} Let $w\in \cD^2_{\WF}(P)$, and suppose that $w$ assumes a local
  maximum at $p_0\in P,$ then $Lw(p_0)\leq 0.$
\end{lemma}
  \begin{proof} If $p_0\in\Int P,$ then this is obvious, as $L$ is strongly
    elliptic in the interior, and annihilates the constant function. Suppose
    that $p_0$ is a boundary point of codimension $n,$ and
    $(x_1,\dots,x_n,y_1,\dots,y_m)$ are adapted local coordinates. We normalize
    so that $p_0$ corresponds to $\bzero;$ the stratum $bP$ through $p_0$ is
    locally given by $\Sigma=\{x_1=\cdots=x_n=0\}.$ The regularity assumptions
    show that $w$ restricted to $\Sigma$ is locally $\cC^2$ and and the local
    form for $L$ given in~\eqref{Lnrmfrm} shows that
  \begin{equation}
    Lw(p_0)=\sum_{k,l}c'_{kl}\pa_{y_k}\pa_{y_l}w(p_0)+Vw(p_0).
  \end{equation}
Since $V(p_0)$ is inward pointing and $p_0$ is a local maximum, it is clear
that
\begin{equation}
  Vw(p_0)\leq 0.
\end{equation}
Since $w\restrictedto_{\Sigma}$ is locally $\cC^2,$ the second order part of
$Lw$ at $p_0$ is also non-positive, thus proving the lemma.
\end{proof}

In order to refine this result, we must describe in more detail the structure
of $bP$, and the relationship of $L$ to the various components of the
stratification of $bP$.  First, let $\{\Sigma_{1,j}:\: j=1,\dots, N_1\}$ denote
the connected hypersurface boundary components of $P$ and $\{\rho_{j}\}$ their
respective defining functions. If $\Sigma$ is a component of $bP$ of codimension
$n$, then there are $n$ hypersurface boundary components $\{\Sigma_{1,j_i}:\: i
= 1, \dots, n \}$ so that $\Sigma$ is a connected component of the intersection
\begin{equation}
\Sigma_{1,j_1}\cap\cdots\cap\Sigma_{1,j_n}.  
\end{equation}

The first order part $V$ is tangent to $\Sigma$ near $p_0,$ if and only if
$V\rho_{j_i}=0$ for $i=1,\dots,n$; this is evidently equivalent to the
condition $L\rho_{j_i}=0$ for this collection of indices.   If this holds at all points of
$\Sigma$, then we say that $L$ \emph{is tangent to} $\Sigma.$
\index{$L$ is tangent to $\Sigma$} If, on the other
hand, there is a $c>0$ so that
\begin{equation}
L\rho_{j_i}\restrictedto_{\Sigma}>c,\text{ for }i=1,\dots,n,
\end{equation}
then we say that $L$ is \emph{transverse to $\Sigma.$} \index{$L$ is transverse
  to $\Sigma$} These conditions are independent of the choice of defining
function. We write $bP^T(L)$ for the\index{$bP^T(L)$} union of boundary
components to which $L$ is tangent, and $bP^{\pitchfork}(L)$
\index{$bP^{\pitchfork}(L)$} for the union of boundary components to which $L$
is transverse.

The following non-degeneracy assumption about $L$ simplifies many of the global results.

\begin{definition} We say that $L$ \emph{meets $bP$ cleanly}, if for each
  $1\leq j\leq N_1$, either $L\rho_j\restrictedto_{\{\rho_j=0\}}\equiv 0,$ or
  there exists a $c_j>0,$ so that \index{$L$ meets $bP$ cleanly}
\begin{equation}
L\rho_j\restrictedto_{\{\rho_j=0\}}>c_j.
\end{equation}
\end{definition} 
\noindent

Briefly, for each $j$, the vector field is either tangent or transverse to
$\Sigma_{1,j}$, so cleanness prevents such behavior as $V$ lying tangent to
some $\Sigma_{1,j}$ only along a proper closed subset (possibly in
$b\Sigma_{1,j}$). A boundary component belongs to $bP^T(L)$ if and only if it a
component of the intersection of a collection of boundary faces to which $L$ is
tangent. A boundary component belongs to $bP^{\pitchfork}(L)$ if and only if it
a component of the intersection of a collection of boundary faces to which $L$
is transverse. There may be boundary components that belong to neither of these
extremes.

Any boundary component $\Sigma$ is itself a manifold with
corners. If $L$ is tangent to $\Sigma,$ then we write $L_{\Sigma}$ for the
generalized Kimura diffusion defined by restriction of $L$ to $\Sigma.$
% Suppose that $p_0$ lies in the interior of $\Sigma$, with $(\bx,\by)$ adapted local coordinates defined in a 
% neighborhood $U$ around $p_0$. Let $L^{(2)}$ denote the sum of the second order terms
% in~\eqref{Lnrmfrm}. The differential operator:
% \begin{equation}
%  L^{(2)}_{\Sigma\cap U}=\sum_{1\leq k,l\leq m}c'_{kl}\pa_{y_k}\pa_{y_l}
% \end{equation}
% acts on functions defined on $\Sigma\cap U,$ and 
It is clear that if $U$ is any neighborhood of a point $p_0$ in the interior of
$\Sigma$, if $w\in\cD^2_{\WF}(U),$ and if $L$ is tangent to $\Sigma,$ then
\begin{equation}\label{restrict}
[Lw]\restrictedto_{\Sigma\cap U}=L_{\Sigma}[w\restrictedto_{\Sigma\cap U}].
\end{equation}

The first basic result is the following.
\begin{lemma}\label{lem3.0.8.05} Let $w\in \cD^2_{\WF}(P)$, and suppose that
  $w$ is a subsolution of $L$, i.e.\ $Lw\geq 0.$ Then $w$ cannot assume a local
  maximum in the interior of $P$, or in the interior of any component $\Sigma
  \in bP^T(L)$, unless $w$ is constant on $P$, or on that component,
  respectively.
\end{lemma}
\begin{proof} Since $L$ is a non-degenerate elliptic operator in $\Int P,$ it
  follows from the standard strong maximum principle that $w$ does not assume a
  local maximum in $\Int P$ unless $w$ is constant.  The regularity hypothesis,
  and the assumption that $L$ is tangent to $\Sigma$ shows that
  $w\restrictedto_{\Sigma}\in\cC^{2}_{\WF}(\Sigma),$ and
 \begin{equation}
L_{\Sigma}[w\restrictedto_{\Sigma}]\geq 0.
\end{equation}
Hence the first part of this proof applies to show that $w$ cannot assume its maximum in 
$\Int\Sigma,$ unless $w\restrictedto_{\Sigma}$ is constant.
\end{proof}

In order to apply this result to determine the nullspace of $L$, we need to
discuss two further types of boundary components.  First, amongst the
collection of all components of the stratification of $bP$, certain ones are
minimal in the sense that they themselves have no boundary; these components
are either points or closed manifolds. We denote by $bP_{\min}$ the union of
all such components; the different components of this set may have different
dimensions.  These minimal components are the minimal elements in the maximal
well-ordered chains of boundary components, where the ordering is given by
containment in the closure.\index{minimal boundary}
\begin{lemma}\label{bpminlem} Every component of $bP$ is either minimal or else
  contains elements  of $bP_{\min}$ in its closure.
\end{lemma}
\begin{proof}
  This follows directly by induction on the maximal codimension of corners in
  $P$.  If $\Sigma \not\in bP_{\min}$, then  $\Sigma$ is a
  manifold with corners with maximal codimension no more than one less than
  that of $P$. Hence there is some boundary component $\Sigma'$ of $\Sigma$
  which is minimal. Clearly $\Sigma'$ is also a boundary component of $P$, and
  since it has no boundary, it must be minimal for $P$ as well.
\end{proof}

Note that if $\Sigma \in bP_{\min}^T(L)=bP^T(L)\cap bP_{\min}(L)$, then either
$\Sigma$ is a point and all coefficients of $L$ vanish at $\Sigma$, or else
$\dim\Sigma>0$ and $L_\Sigma$ is a nondegenerate elliptic operator. It follows
immediately that if $Lw \geq 0$ as above, and if $w \restrictedto_\Sigma$
attains a local maximum on $\Sigma \in bP_{\min}^T(L)$, then $w
\restrictedto_\Sigma$ is constant.

Finally, the \emph{terminal boundary} of $P$ relative to $L$, denoted
$bP_{\ter}(L)$,\index{$bP_{\ter}(L)$} consists of the union of boundary components $\Sigma \subset
bP^T(L)$ such that $L_\Sigma$ is transverse to all components of $b\Sigma$
(i.e.\ $b\Sigma \in b\Sigma^\pitchfork(L_\Sigma)$). In particular, if $L$ is
transverse to all components of $bP$ itself, then $bP_{\ter}(L)=P$.  As
elements of $bP_{\min}^T(L)$ have empty boundary, it follows that
$bP_{\min}^T(L)\subset bP_{\ter}(L).$\index{terminal boundary}

There is a version of the Hopf boundary point lemma, adapted to this setting.
\index{Hopf boundary maximum principle}
\begin{lemma}\label{hopfmaxpriple} Let $P$ be a compact, connected manifold with corners, and $L$ 
  a generalized Kimura diffusion operator that meets $bP$ cleanly. Suppose
  that $w \in \cD^2_{\WF}$ is a subsolution of $L$, $Lw \geq 0$, in a
  neighborhood, $U,$ of a point $p_0\in bP$ which lies in the interior of a
  boundary component $\Sigma \in bP^\pitchfork(L)$. If $w$ attains a local
  maximum at $p_0,$ then $w$ is constant on $U$.
\end{lemma}
This has an immediate and important consequence.
\begin{lemma}\label{terface}
  Let $P$ be a compact manifold with corners, and $L$ a generalized Kimura
  diffusion on $P,$ which meets $bP$ cleanly. Suppose that $L$ is transverse to
  every face of $bP,$ and $w \in \cD^2_{\WF}(P)$ is a subsolution of $L$. Then
  $w$ is constant.
\end{lemma}
\noindent
This follows directly from Lemma~\ref{hopfmaxpriple}.

We defer the proof of Lemma~\ref{hopfmaxpriple} momentarily and derive its main
consequence.  The following result shows that, at least when $L$ meets $bP$
cleanly, the null-space of $L$ on $\cD^2_{\WF}(P)$ is finite dimensional.
\begin{proposition}\label{prop.kdbvs} 
Let $P$ be a compact, connected manifold with corners and $L$ a generalized Kimura diffusion 
which meets $bP$ cleanly. Let $w\in\cD^2_{\WF}(P)$ be a solution to $Lw=0$. Then $w$ is 
determined by its (constant) values on the components of $bP_{\ter}(L).$ 
\end{proposition}
\begin{proof} We prove this by induction on the dimension of  $P.$  
If the dimension of $P$ is $1$, then $P$ is an interval. The statement of the
proposition in this case was established in~\cite{WF1d}

Now suppose the result has been proved for all compact manifolds with corners
$P'$ of dimension $N-1$, and all general Kimura diffusion operators $L'$ on
them.  Assume that $\dim P=N,$ and that $w \restrictedto_{\Sigma} = 0$ for all
$\Sigma \in bP_{\ter}(L)$.  Obviously, if $P$ has no boundary faces to which
$L$ is tangent, then $P$ is itself a terminal boundary and we already have that
$w \equiv 0$. We henceforth assume that $bP^T(L)\neq\emptyset.$ If $w$ does not
vanish identically, then we can assume that it is positive somewhere, and
therefore attains a positive maximum somewhere in $P.$

The induction hypothesis shows that for any boundary component $\Sigma\in
bP^T(L),$ we have a solution $w\restrictedto_{\Sigma},$ which vanishes on
$b\Sigma_{\ter}(L_{\Sigma}).$ This follows because the terminal components of $bP$
relative to $L,$ which are contained in the closure of $\Sigma,$ are the same as
the terminal components of $b\Sigma$ relative to $L_\Sigma$. Indeed, if $L$ is
tangent to some component $\Sigma_0$ of $b\Sigma$, then $L_\Sigma$ is also tangent to
$\Sigma_0$. Furthermore, the restriction of $L_\Sigma$ to $\Sigma_0$ is the
same as $L_{\Sigma_0}$, the restriction of $L$ to $\Sigma_0$. Thus the
condition that $L_{\Sigma}$ be transverse to all components of $b\Sigma$ is the
same, whether restricting from $P$ or $\Sigma$.  This means that $w
\restrictedto_{\Sigma_0}$ vanishes on all $\Sigma_0\in b\Sigma_{\ter}(L_{\Sigma})$, and
hence by induction, $w \restrictedto_{\Sigma} = 0$. Note that $w$ vanishes
on every hypersurface boundary component to which $L$ is tangent.

Lemma~\ref{hopfmaxpriple} shows that $w$ cannot attain a positive maximum on a
boundary component to which $L$ is transverse. Thus $w$ must attain its maximum
at a point $p_0$ lying in a boundary component, $\Sigma_1$ to which $L$ is
neither tangent, nor transverse. That is, $\Sigma_1$ lies in the intersection
of boundary faces some of which belong to $bP^T(L)$ and some of which $L$ is
transverse to.  This implies that there is a boundary hypersurface $\Sigma_0$
to which $L$ is tangent, and such that $\Sigma_1\subset b\Sigma_0.$ As $w=0$ on
$\Sigma_0$ and $\Sigma_1\subset \overline{\Sigma_0},$ this contradiction
establishes that $w\equiv 0$ on $P.$
\end{proof}

\begin{remark} This theorem is proved in~\cite{Shimakura1} in the special case
  of the classical Kimura diffusion, without selection, acting on the
  $N$-simplex. 
\end{remark}
Lemma~\ref{terface} is a special case of this proposition. Two  other
corollaries are:
\begin{corollary} If $L$ is a generalized Kimura diffusion on the compact
  manifold with corners and $L$ is everywhere tangent to $bP$, then any
  solution $w\in \cD^2_{\WF}(P)$ to $Lw=0$ is determined by its constant values
  on $bP_{\min}=bP_{\min}^T.$
\end{corollary}

\begin{corollary}\label{cor3.1.000}
  If $P$ is a compact manifold with corners, and $L$ a generalized Kimura
  diffusion on $P$ meeting $bP$ clearly, then the dimension of the nullspace of $L$ acting on
  $\cD^2_{\WF}(P)$ is bounded above by the cardinality of the set $bP_{\ter}(L).$
\end{corollary}
\begin{remark} These results give powerful support for our assertion
  that the regularity condition $w\in\cD^2_{\WF}(P)$ is a reasonable
  replacement for a local boundary condition, at least when $L$ meets
  $bP$ cleanly. In applications to probability, one often considers
  solutions of equations of the form $Lw=g,$ where $w$ satisfies a
  Dirichlet condition on $bP.$ Frequently $g$ is non-negative, and our
  uniqueness results easily imply that there cannot be a regular
  solution. The simplest example arises in the 1-dimensional case,
  with $L=x(1-x)\pa_x^2.$ The solution, $w$ to the equation
\begin{equation}
  Lw=-1\text{ with }w(0)=w(1)=0
\end{equation}
gives the expected time to arrive at $\{0,1\},$ for a path of the
process starting at $0<x<1.$ There cannot be a regular solution as
$x(1-x)\pa_x^2w(x)=-1,$ cannot converge to $0$ as $x\to 0^+,1^{-}.$
The solution, given by 
$$w(x)=-[x\log x+(1-x)\log(1-x)],$$ 
is plainly not regular. In applications to probability this situation often
pertains. The fact that $Lw\geq 0$ has no non-trivial regular solutions shows
that the required solutions cannot be regular, and therefore involve the
non-zero indicial root.
\end{remark}

\noindent
We now turn to the proof of the ``Hopf lemma:''

\begin{proof}[Proof of Lemma~\ref{hopfmaxpriple}]
  The proof of this lemma relies on the construction of barrier functions and a
  simple scaling argument.  To motivate the argument we first give the proof
  for a model operator and a boundary point of codimension 1. Let $(x,\by)$
  denote normalized local coordinates so that
\begin{equation}
  L=x\pa_x^2+b\pa_x+\sum_{l=1}^m\pa_{y_l}^2,\text{ with }b>0,
\end{equation}
and $w$ assumes a local max at the $(0,\bzero).$ We assume that $w$ is not
constant. The strong maximum principle implies that there is a neighborhood $U$
of $0$ so that if $(x,y)\in U$ and $x>0,$ then
\begin{equation}
  w(x,\by)<w(0,\bzero).
\end{equation}

For $R,r$ positive numbers and $\frac{1}{2}\leq\alpha<1,$
we define anisotropic balls in $\bbR_+^n\times\bbR^m:$
\begin{equation}
  B^+_{R,r,\alpha}(n,m)=\{(\bx,\by)\in\bbR_+^n\times\bbR^m:|\bx^{\alpha}-\br^{\alpha}|^2+|\by|^2\leq
  nR^{2\alpha}\}\cap\{\bx\geq \bzero\}.
\end{equation}
Here $\bx^{\alpha}=(x_1^{\alpha},\dots,x_n^{\alpha}),$ etc.
We now construct a non-negative local barrier, $v_{\lambda,\alpha,r}$ 
that satisfies
\begin{multline}
  Lv_{\lambda,r}>0 \text{ in } B^+_{r,r,\alpha}(1,m)\setminus
  B^+_{\frac{r}{2^{\beta}},r,\alpha}(1,m),\text{ and }
v_{\lambda,r}\restrictedto_{bB^+_{r,r,\alpha}(1,m)}=0,\\
\text{ where } \beta=\frac{1}{2\alpha}.
\end{multline}
Note that $B^+_{\frac{r}{2^{\beta}},r,\alpha}(1,m)$ is a compact subset of $P$
lying a positive distance from $bP.$ Figure~\ref{fig1} shows the set 
$B^+_{1,1,\frac{1}{2}}(2,0)\setminus B^+_{\frac{1}{4},1,\frac{1}{2}}(2,0)\subset \bbR_+^2.$
\begin{figure}[hh]
\centering
{\epsfig{file=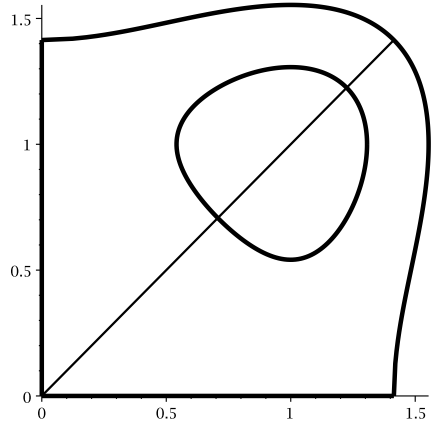,width=6cm}}
\caption{The set 
$B^+_{1,1,\frac{1}{2}}(2,0)\setminus B^+_{\frac{1}{4},1,\frac{1}{2}}(2,0)\subset \bbR_+^2$ lies between the black curves.}
\label{fig1}
\end{figure}

We first define barrier functions in the 1-codimensional case, letting
\begin{equation}
  w_{\lambda,\alpha,r}=e^{-\lambda[(x^{\alpha}-r^{\alpha})^2+|y|^2]},
\end{equation}
then
\begin{multline}
  Lw_{\lambda,r}=\Big[4\lambda^2[\alpha^2x^{2\alpha-1}(r^{\alpha}-x^{\alpha})^2+|y|^2]+\\
2\lambda\left(\alpha[b-(1-\alpha)] x^{\alpha-1}(r^{\alpha}-x^{\alpha})-
\alpha^2x^{2\alpha-1}-m\right)\Big]w_{\lambda,r}.
\end{multline}
If $1-b<\alpha<1,$ then we see that $Lw_{\lambda,r}(x,\by)$ tend to $+\infty$ as
$x$ tends to zero.  As
$[\alpha^2x^{2\alpha-1}(r^{\alpha}-x^{\alpha})+|y|^2]$ only vanishes at
$(r,\bzero)$ and $(0,\bzero)$ (if $\alpha\neq \frac 12$) it is not difficult to
see that for large enough $\lambda$ we have that
\begin{equation}
  Lw_{\lambda,\alpha,r}>0 \text{ in  }B^+_{r,r,\alpha}(1,m)\setminus
  B^+_{\frac{r}{2^{\beta}},r,\alpha}(1,m).
\end{equation}
Let 
\begin{equation}
  v_{\lambda,\alpha,r}=w_{\lambda,\alpha,r}-w_{\lambda,\alpha,r}(0,\bzero),
\end{equation}
so that the barrier vanishes on $bB^+_{r,r,\alpha}(1,m).$

We fix an $r$ so that $B^+_{r,r,\alpha}(1,m)\subset U$ and a $\lambda$ and
$\alpha$ as above. The hypothesis that $w$ is non-constant
shows that there is an $\epsilon>0$ so that
\begin{equation}
  (w+\epsilon v_{\lambda,\alpha,r})\restrictedto_{bB^+_{\frac{r}{2^{\beta}},r,\alpha}(1,m)}<w(0,\bzero),
\end{equation}
and therefore $w+\epsilon v_{\lambda,\alpha,r}$ assumes its maximum value at
$(0,\bzero).$ Thus for small positive $x$ we have
\begin{equation}
  (w+\epsilon v_{\lambda,r})(x,\bzero)-(w+\epsilon v_{\lambda,r})(0,\bzero)\leq 0.
\end{equation}
A simple application of the mean value theorem  shows that $\pa_x
w(x,\bzero)$ tends to $-\infty$ as $x\to 0^+,$ contradicting the
assumed regularity of $w.$ This completes the proof that $w$ must be constant
in this case.

To treat a general (non-model operator) we might need to dilate the coordinates
by setting:
\begin{equation}
  x=\mu X\text{ and } \by=\sqrt{\mu}\bY,\text{ for }\mu>0.
\end{equation}
Under this change of variables, the model operator $L$ becomes
\begin{equation}
  \frac{1}{\mu}[X\pa_{X}^2+b\pa_{X}+\sum_{l=1}^m\pa^2_{Y_l}].
\end{equation}
That is, up to a positive factor, this operator is invariant under these
changes of coordinate.  

If we let $W(X,\bY)=w(\mu X,\sqrt{\mu}\bY),$ then evidently $LW\geq 0,$ and $W$
attains a local maximum at $(0,\bzero).$ The ball
\begin{equation}
  (X^{\alpha}-R^{\alpha})^2+|\bY|^2\leq R^{2\alpha}\text{ where }
R=\frac{r}{\mu}.
\end{equation}
is contained in this coordinate chart. In the original coordinates we have
\begin{multline}
  L=x\pa_{x}^2+b\pa_{x}++\sum_{k=1}^m\pa_{y_k}^2+
\sum_{l=1}^mxa_{l}(x,\by)\pa_{x}\pa_{y_l}+\\
\sum_{k,l=1}^mc_{kl}(\bx,\by)\pa_{y_k}\pa_{y_l}+
\sum_{i=1}^n\tb(x,\by)\pa_{x}+\sum_{k=1}^md_k(\bx,\by)\pa_{y_k},
\end{multline}
where $\tb(0,\bzero)=c_{kl}(0,\bzero)=0.$ Letting $x=\mu X$ and
$\by=\sqrt{\mu}\bY,$ we obtain:
\begin{multline}
  L_{\mu}=\frac{1}{\mu}\Bigg\{X\pa_{X}^2+b\pa_{X}+\sum_{k=1}^m\pa_{Y_k}^2+
\sqrt{\mu}\sum_{l=1}^mXa_{l}(\mu X,\sqrt{\mu}\bY)\pa_{X}\pa_{Y_l}+\\
\sum_{k,l=1}^mc_{kl}(\mu X,\sqrt{\mu}\bY)\pa_{Y_k}\pa_{Y_l}+
\tb(\mu X,\sqrt{\mu}\bY)\pa_{X}+\sqrt{\mu}\sum_{k=1}^md_k(\mu\bX,\sqrt{\mu}\bY)\pa_{Y_k}
\Bigg\}.
\end{multline}
Even though we may let $\mu$ get very small, we can fix a positive $R.$ It is
then not hard to see, that, with a possibly larger $\alpha<1,$ by taking $\mu$
small enough we can arrange for
\begin{equation}
  L_{\mu}w_{\lambda,\alpha,R}>0\text{ in }
 B^+_{R,R,\alpha}(1,m)\setminus
  B^+_{\frac{R}{2^{\beta}},R,\alpha}(1,m).
\end{equation}
From this point the argument proceeds as before showing that if $Lw\geq 0$ and
$w$ attains a maximum at $p_0,$ then $w$ is constant.

Suppose that $p_0$ is a point on a stratum of codimension $n,$ we can choose
$(\bx,\by)$ adapted local coordinates, with $p_0$ corresponding to $(\bzero,
\bzero),$ so that the operator takes the form:
\begin{multline}
  L=\sum_{i=1}^n[x_i\pa_{x_i}^2+b_i\pa_{x_i}]+\sum_{k=1}^m\pa_{y_k}^2+\\
\sum_{i\neq j=1}^nx_ix_ja_{ij}(\bx,\by)\pa_{x_i}\pa_{x_j}+
\sum_{i=1}^n\sum_{l=1}^mx_ia_{il}(\bx,\by)\pa_{x_i}\pa_{y_l}+
\sum_{k,l=1}^mc_{kl}(\bx,\by)\pa_{y_k}\pa_{y_l}+\\
\sum_{i=1}^n\tb_i(\bx,\by)\pa_{x_i}+\sum_{k=1}^md_k(\bx,\by)\pa_{y_k},
\end{multline}
where
\begin{equation}
  c_{kl}(\bzero,\bzero)=\tb_i(\bzero,\bzero)=0.
\end{equation}

Let $\bb=(b_1,\dots,b_n)>\bzero.$ For the model operator $L_{\bb,m}$ we
consider barrier functions of the form
\begin{equation}
  w_{\lambda,\alpha,r}=\exp\left[-\lambda(|\bx^{\alpha}-\br^{\balpha}|^2+|\by|^2\right],
\end{equation}
where $\br=(r,\dots,r).$ Applying the model operator we see that
\begin{multline}
  L_{\bb,m} w_{\lambda,\alpha,r}=4\lambda^2
\left[\alpha^2\sum_{j=1}^nx_{j}^{2\alpha-1}(x_j^{\alpha}-r^{\alpha})^2+|\by|^2\right]+\\
2\lambda
\left[\alpha\sum_{j=1}^n
\left([b_j-(1-\alpha)]x_{j}^{\alpha-1}(r^{\alpha}-x_j^{\alpha})-\alpha
  x_j^{2\alpha-1}\right)-m
\right].
\end{multline}
In order for $w_{\lambda,\alpha,r}$ to be a subsolution, we need to choose
$\frac 12\leq \alpha<1,$ so that
\begin{equation}
  1-\min\{b_1,\dots,b_n\}<\alpha.
\end{equation}
With such a choice of $\alpha,$ we see that $L_{\bb,m}w_{\lambda,\alpha,r}$
tends to $+\infty$ as any $x_j$ tends to zero.  We can therefore find
$\lambda_0$ so that for $\lambda_0<\lambda$ we have
\begin{equation}
 L_{\bb,m} w_{\lambda,\alpha,r}>0\text{ in }B^+_{r,r,\alpha}(n,m)\setminus 
B^+_{\frac{r}{(2n)^{\beta}},r,\alpha}(n,m),
\end{equation}
where, as before $\beta=\frac{1}{2\alpha}.$

As before we can scale the variables $\bx=\mu\bX$ and $\by=\sqrt{\mu}\bY,$ to obtain
\begin{multline}
  L_{\mu}=\frac{1}{\mu}\Bigg\{\sum_{i=1}^n[X_i\pa_{X_i}^2+b_i\pa_{X_i}]+\sum_{k=1}^m\pa_{Y_k}^2+
\mu\sum_{i\neq j=1}^nX_iX_ja_{ij}(\mu\bX,\sqrt{\mu}\bY)\pa_{X_i}\pa_{X_j}+\\
\sqrt{\mu}\sum_{i=1}^n\sum_{l=1}^mX_ia_{il}(\mu\bX,\sqrt{\mu}\bY)\pa_{X_i}\pa_{Y_l}+
\sum_{k,l=1}^mc_{kl}(\mu\bX,\sqrt{\mu}\bY)\pa_{Y_k}\pa_{Y_l}+\\
\sum_{i=1}^n\tb_i(\mu\bX,\sqrt{\mu}\bY)\pa_{X_i}+\sqrt{\mu}\sum_{k=1}^md_k(\mu\bX,\sqrt{\mu}\bY)\pa_{Y_k}
\Bigg\}.
\end{multline}
A calculation shows that
\begin{multline}
  L_{\mu}w_{\lambda,\alpha,r}(\bX,\bY)=\frac{1}{\mu}\Bigg\{4\lambda^2\bigg[\sum_{i=1}^n\alpha^2X_i^{2\alpha-1}
(X_i^{\alpha}-r^{\alpha})^2+|\bY|^2+\\
\sum_{i\neq j}\alpha^2a_{ij}(\mu\bX,\sqrt{\mu}\bY)X_iX_j
(X_i^{\alpha}-r^{\alpha})(X_j^{\alpha}-r^{\alpha})
+\\\sqrt{\mu}\sum_{i=1}^n\sum_{l=1}^m\alpha
a_{il}(\mu\bX,\sqrt{\mu}\bY)X_i^{\alpha}(X_i^{\alpha}-r^{\alpha})Y_l+
\sum_{k,l=1}^mc_{kl}(\mu\bX,\sqrt{\mu}\bY)Y_kY_l\bigg]+\\
2\lambda\bigg[\sum_{i=1}^n\Big[\big(b_i+\tb_i(\mu\bX,\sqrt{\mu}\bY)-(1-\alpha)\big)X_i^{\alpha-1}(r^{\alpha}-X_i^{\alpha})-
\alpha^2X_i^{2\alpha-1}\Big]-\\
\sum_{l=1}^m[1+c_{ll}(\mu\bX,\sqrt{\mu}\bY)+\sqrt{\mu}d_l(\mu\bX,\sqrt{\mu}\bY)Y_l]
\bigg]\Bigg\}.
\end{multline}
If we take $r$ and $\mu$ sufficiently small, then the $O(\lambda^2)$-term is
bounded below by a positive multiple of
\begin{equation}
  \sum_{i=1}^n\alpha^2X_i^{2\alpha-1}(X_i^{\alpha}-r^{\alpha})^2+|\bY|^2
\end{equation}
By taking $\alpha<1$ a little larger, and possibly reducing $\mu$ we can assure
that
\begin{equation}
  \min\{b_i+\tb_i(\mu\bX,\sqrt{\mu}\bY)-(1-\alpha):\:
  (\bX,\bY)\in B^+_{r,r,\alpha}(n,m);\, i=1,\dots, n\}
\end{equation}
is strictly positive. With these choices, there is a $\lambda_0$ so that if
$\lambda>\lambda_0,$ then 
\begin{equation}
  L_{\mu}w_{\lambda,\alpha,r}>0\text{ in }B^+_{r,r,\alpha}(n,m)\setminus 
B^+_{\frac{r}{(2n)^{\beta}},r,\alpha}(n,m).
\end{equation}
Note also that $\pa_{X_i}w_{\lambda,\alpha,r}(\bX,\bY)$ tends to $+\infty$ as
$X_i\to 0^+.$ Finally we set
\begin{equation}
  v_{\lambda,\alpha,r}=w_{\lambda,\alpha,r}-w_{\lambda,\alpha,r}(\bzero,\bzero)
\end{equation}

The argument then proceeds as before. We  are assuming that $w$ is a
non-constant solution to $Lw\geq 0,$ which assumes a local maximum at
$(\bzero,\bzero).$  This implies that $w(\bX,\bY)<w(\bzero,\bzero)$ for
$(\bX,\bY)\in\Int P.$ Thus we can choose an $\epsilon>0$ so that 
for $(\bX,\bY)\in bB^+_{\frac{r}{(2n)^{\beta}},r,\alpha}(n,m)$ we have the
estimate
\begin{equation}
  (w+\epsilon v_{\lambda,\alpha,r})(\bX,\bY)<(w+\epsilon v_{\lambda,\alpha,r})(\bzero,\bzero).
\end{equation}
Since $v_{\lambda,\alpha,r}$ vanishes on $bB^+_{r,r,\alpha}(n,m),$ we see that
$(w+\epsilon v_{\lambda,\alpha,r})$ must assume its maximum at a point $p_0$ on
$bP\cap B^+_{r,r,\alpha}(n,m).$ As before this implies that the derivatives
$\pa_{X_j}w(p)$ tend to $-\infty$ as $p$ approaches $p_0,$ contradicting our
assumptions about the smoothness of $w.$ This completes the proof of the lemma.
\end{proof}

These results do not address the case when $L$ fails to meet $bP$
cleanly. While a different argument is needed, it seems likely that a result
like that in Proposition~\ref{prop.kdbvs} remains true. In particular, the
dimension of the null-space of $L$ acting on $\cD^2_{\WF}(P)$ should be finite
dimensional.

\section{Maximum Principles for the Heat Equation}
We now turn to maximum principles for the heat equation:
\begin{proposition}\label{prop.maxPP}
Let $u$ be a subsolution of the Kimura diffusion equation $\del_t u \leq L u$ on 
$[0,T] \times P$, where $P$ is a compact manifold with corners, such that
\[
u \in \calC^0( [0,T] \times P) \cap \calC^1( (0,T] \times P),
\]
and $u(t,\cdot) \in \cD^2_{\WF}(P)$ for $t > 0,$ then
\[
\sup_{[0,T] \times P} u(p,t) = \sup_{P} u(p,0). 
\]
\end{proposition}
\begin{proof}
This is proved in almost exactly the same way as Proposition~\ref{prop.maxPmod}. Because $P$
is compact and $u(\cdot,t)$ is continuous up to $bP$, there is no need to assume
a growth condition on $u$.  The hypotheses are such that we can verify as before that
no local maximum occurs along $bP \times (0,T]$, and by the usual maximum principle,
there is also no local maximum in the interior of $P$ when $ 0< t \leq T$. 
\end{proof}

\begin{corollary}\label{uniquehigherdimmod}
  Let $u_1$ and $u_2$ be two solutions of $(\del_t - L_{b,m})u = g$, $u(\bx,\by,0)
  = f(\bx,\by)$ in $\RR_+ \times S_{n,m}$ such that $|u_j| \leq C_T e^{a(|\bx| +
    |\by|^2)}$ for some $C > 0$ uniformly in any $[0,T] \times S_{n,m},$
  satisfying the regularity hypotheses of Proposition~\ref{prop.maxPmod}. Then
  $u_1 \equiv u_2$.

Similarly, if $u_1$ and $u_2$ are two solutions of $(\del_t - L)u = f$,
$u(\cdot,0) = f(\cdot)$ in $\RR_+ \times P$, satisfying the regularity
hypotheses of Proposition~\ref{prop.maxPP}, where $P$ is a compact manifold
with corners, then $u_1 \equiv u_2$.
\end{corollary}
\begin{remark}
  The regularity assumption up to the boundaries where $x_j = 0$ are
  fundamental. For example, if $0 < b < 1$, then $x^{1-b}$ is a stationary
  solution of $(\del_t - L_b)u = 0$ on $\RR_+ \times \RR_+$ which is certainly
  subexponential as $x \to \infty$.  However, by the results of \cite{WF1d},
  there is some other solution $w(x,t)$ to this equation with initial data
  $w(x,0) = x^{1-b}$ which is smooth up to $x = 0$ for $t > 0$, so that $w(x,t)
  \neq x^{1-b}$ for $t > 0$. Then $x^{1-b} - w(x,t)$ is a homogeneous solution
  with zero Cauchy data at $t=0$ and which has subexponential growth. It is neither
  $\calC^1$ up to $x=0$, nor does it satisfy $\lim_{x\to 0^+}x\pa_x^2u(x,t)=0.$
\end{remark}

We record one other easy extension of these results.
\begin{proposition}\label{uniquenesswithzeroorderterm}
Let $L$ be a general elliptic Kimura operator on a compact manifold with corners $P$,
and suppose that $c \in \cC^0(P)$. Suppose that $u$ is a subsolution of the diffusion
equation associated to $L+c$, i.e.\ $\del_t u \leq (L + c) u$, such that
such that
\[
u \in \calC^0( [0,T] \times P) \cap \calC^1( (0,T] \times P),
\]
and $u(\cdot,t) \in \cD^2_{\WF}(P)$ for $t > 0,$
 then
\[
\sup_{[0,T] \times P} u(p,t) \leq e^{\alpha t} \sup_{P} u(p,0), 
\]
where $\alpha = ||c||_{\infty}$. 

Consequently, if $u_1$ and $u_2$ are two solutions of $\del_t u = (L + c)u$
which satisfy the regularity conditions above and which have the same initial
condition at $t = 0$, then $u_1 \equiv u_2$.
\end{proposition}

Finally, we also state the corresponding maximum principle and uniqueness result
for elliptic Kimura equations.
\begin{proposition} \label{uniquenesselliptic}
Let $L$ be a general elliptic Kimura operator on a compact manifold
with corners $P$ and that $c \in \cC^0(P)$ is a nonpositive function. Let $f$
satisfy $0\leq(L + c)f $ and $f\in\cD^2_{\WF}(P),$
 then
\[
\sup_P f(p) = \sup_{bP}f(p)
\]

If $f_1$ and $f_2$ are any two solutions of $(L+c) f = 0$ which satisfy all the
regularity assumptions above and which agree on $bP$, then $f_1 \equiv f_2$.
\end{proposition}

There is an important special case, where a much sharper result is true.
\begin{proposition} \label{uniquenesselliptic.1}
Let $L$ be a general elliptic Kimura operator on a compact manifold
with corners $P$ and that $c \in \cC^0(P)$ is a strictly negative function. Let $u$
satisfy $(L + c)f=0,$ and also suppose  that
\[
f\in\calC^2_{\WF}( P),
\]
then $f\equiv 0.$
\end{proposition}
\begin{proof} We show that $f$ can neither attain a positive maximum not a
  negative minimum, and hence is identically zero. Since we are considering the
  equation $(L + c)f=0,$ it suffices to show that $f$ cannot attain a  negative minimum.
 We suppose that $f$ does attain a  negative minimum. The regularity
  assumptions and the compactness of $P$ show that $f$ attains its minimum
  at some point $p_0\in P.$ It is easy to see that $p_0\notin \Int P.$ Suppose
  that $p_0$ belongs to a point of the boundary of codimension $M,$ so that in
  local coordinates
  \begin{multline}
    Lf=\sum_{i=1}^M[x_ia_i(\bx,\by)\pa_{x_i}^2f+b_i(\bx,\by)\pa_{x_i}f]+
\sum_{i,j=1}^Mx_ix_ja_{ij}(\bx,\by)\pa_{x_i}\pa_{x_j}f+\\
\sum_{i=1}^M\sum_{l=1}^{n-M}x_ic_{il}(\bx,\by)\pa_{x_i}\pa_{y_l}f+
\sum_{l,m=1}^{n-M}c_{lm}(\bx,\by)\pa_{y_m}\pa_{y_l}f+
\sum_{l=1}^{n-M}d_{m}(\bx,\by)\pa_{y_l}f.
  \end{multline}
At $p_0=(\bzero,\by_0)$ we see that the regularity assumptions show that
\begin{equation}
  Lf(p_0)=\sum_{i=1}^Mb_i(\bzero,\by_0)\pa_{x_i}f+
\sum_{l,m=1}^{n-M}c_{lm}(\bzero,\by_0)\pa_{y_m}\pa_{y_l}f.
\end{equation}
As $p_0$ is a local minimum, the second order part is non-negative; since the
vector field is inward pointing, so is the first order part. We therefore
conclude again that $(L+c)f(p_0)>0,$ contradicting our assumption that $f$ is
in the null-space of $L+c.$
\end{proof}

\part{Analysis of Model Problems}
\chapter{The model solution operators}\label{c.models1}
In this chapter we introduce the model heat kernels $k_{\bb,m}^t$, i.e.\ the
solution operators for the model problems $\del_t - L_{\bb,m}$, and then prove
a sequence of basic estimates for these operators which are direct
generalizations of the estimates for the one-dimensional version of this
problem considered in \cite{WF1d}. We recall and slightly extend the $\cC^0$
and $\cC^k$ theory in the one-dimensional case proved in \cite{WF1d} and then
derive the straightforward extensions of these results to higher
dimensions. This sets the stage for the more difficult H\"older estimates for
solutions, which is carried out in the next several chapters, and which forms
the technical heart of this monograph.  

We also define the resolvent families $(L_{\bb,m} - \mu)^{-1}$, describe their
holomorphic behavior as functions of $\mu$ and relate this to the analytic
semi-group theory for the model parabolic problems.  At the end of the chapter
we describe why the estimates we prove here are not adequate for the
perturbation theoretic arguments needed to construct the solution operator for
general Kimura diffusions $\del_t - L$.

\section{The model problem in one dimension}
First recall the one-dimensional model operator, 
\begin{equation}
L_b = x\pa_x^2 +b\pa_x\qquad \mbox{on}\ \RR_+ \times \RR^+,
\end{equation}
where $b$ is any nonnegative constant, and the general inhomogeneous Cauchy problem
\begin{equation}\label{1dmdlb}
  \begin{cases}
\pa_tw-L_bw=g \quad \text{on}\  \RR^+ \times (0,T]\\ w(x,0)=f(x),\ x \in \RR^+.
\end{cases}
\end{equation}
So long as $g$ and $f$ have moderate growth, then Corollary~\ref{uniquehigherdimmod} guarantees that there is a 
unique solution with moderate growth and satisfying certain regularity hypotheses at $x=0$; it is given by the integral formula
\begin{equation}
w(x,t)=\int\limits_{0}^t\int\limits_0^{\infty}k^b_{t-s}(x,\tx)g(\tx,s) \, d\tx ds+ \int\limits_{0}^{\infty}k^b_t(x,\tx)f(\tx)\, d\tx.
\label{repform1d}
\end{equation}The precise form of the heat kernel for this problem was derived in \cite{WF1d}: for any $b>0,$
\begin{equation}
\label{kbfrm}
k^b_t(x,\tx)=\frac{\tx^{b-1}}{t^b}e^{-\frac{x+\tx}{t}}\psi_b\left(\frac{x \tx}{t^2}\right),
\end{equation}
where 
\begin{equation}\label{fndslnfrm1}
\psi_b(z)=\sum_{j=0}^{\infty}\frac{z^j}{j!\Gamma(j+b)}.
\end{equation}
When $b=0$, the Schwartz kernel takes a somewhat different form:
\begin{equation}\label{kbfrm0}
  k^0_t(x,\tx)=\left(\frac{x}{t^2}\right)e^{-\frac{x+\tx}{t}}\psi_2\left(\frac{x
      \tx}{t^2}\right)+
e^{-\frac{x}{t}}\delta(y).
\end{equation}

The defining equation for this kernel is that $(\del_t - L_b)k^b_t = 0$ for $t > 0$ (along
with the initial condition that $\lim_{t \to 0^+} k^b(x,\tx) = \delta(x-\tx)$). However, it can
be checked directly from the explicit expression that
\begin{equation}
(\del_t - L^t_{b,\tx}) k^b_t(x,\tx) = 0,\ \ \mbox{where}\ L_{b,\tx}^t=\pa_{\tx}(\pa_{\tx} \tx-b)
\label{adjeqn}
\end{equation}
is the formal adjoint operator. Note too that we can verify directly from \eqref{kbfrm} that
\begin{equation}
\lim_{\tx\to 0^+}(\pa_{\tx} \tx-b)k^b_t(x,\tx)=0\text{ and }k^b_t(x,\tx)=\calO(\tx^{b-1})
\label{adjeqn2}
\end{equation}
when $t > 0$ and $b > 0$. 

We write the solution as a sum $w = v + u$ where $v$ is a solution to the problem with $g=0$
and $u$ is a solution to the problem with $f=0$.  We often call the first of these the homogeneous
Cauchy problem and the second the inhomogeneous problem. \index{homogeneous Cauchy problem} 
\index{inhomogeneous problem}

In the following, we describe the various estimates for solutions on integer order spaces. More specifically, we
use the standard spaces of $\ell$-times continuously differentiable functions $\cC^\ell(\RR_+)$, $\ell \in \bbN$, 
and their parabolic analogues, $\cC^{\ell,\frac \ell 2}(\bbR_+^2)$, which are the closures of $\cC^{\infty}_{c}(\bbR_+^2)$ \index{$\cC^{\ell,\frac \ell 2}$}\index{$\cC^\ell$}
with respect to the norms
\begin{equation}
\|g\|_{\ell,\frac \ell 2}=\sum_{0\leq p+2q\leq \ell}\|\pa_t^q\pa_x^pg(x,t)\|_{\infty}. 
\end{equation}
\begin{remark} To make the notation less cumbersome, in the context of these
  parabolic spaces, we always take $\frac{\ell}{2}$ to mean the greatest
  integer in $\ell/2.$
\end{remark}

To keep track of behavior of solutions as $x \to \infty$, we also use $\dcC^\ell(\bbR_+)$, which
is the closure of $\CI_c(\bbR_+)$ with respect to the norm
\begin{equation}
\|f\|_{\ell}=\sum_{p=0}^\ell\|\pa_x^pf(x)\|_{\infty}. 
\end{equation}
We first discuss  solutions of the homogeneous problem and after that solutions of inhomogeneous problem.

The first result is a slight improvement of a theorem from \cite{WF1d}. 
\begin{lemma}\label{lem3.1new.0} For each $\ell\in\bbN$, if $f\in\dcC^\ell(\bbR_+),$ then $v\in
\cC^{\ell,\frac \ell 2}(\bbR_+\times\bbR_+)$. Moreover if $p+2q\leq \ell,$ then
\begin{equation}\label{eqn28new.0}
\pa_t^q\pa_x^pv(x,t)=\int\limits_{0}^{\infty}k^{b+p}_t(x,\tx)L_{b+p}^q\pa_{\tx}^pf(\tx)\, d\tx.
\end{equation}
\end{lemma}
\begin{proof} Let $v$ be defined by \eqref{repform1d} with $g = 0$. It is shown in~\cite{WF1d} that 
$$
v\in\cC^0([0,\infty)_t;\cC^{\ell}_b([0,\infty)_x)) \cap\CI((0,\infty)_t\times [0,\infty)_x),
$$
and \eqref{eqn28new.0} is established for any $p \in \bbN$ but only for $q=0.$ We establish the 
general case as follows. For $t>0$, differentiate the representation formula for $v$ (i.e.\ \eqref{repform1d} with
$g=0$) and use \eqref{adjeqn} to obtain
\begin{equation}
\pa_tv(x,t) =\int\limits_{0}^{\infty}\pa_tk^b_t(x,\tx)f(\tx)\, d\tx =\lim_{\epsilon\to 0^+}\int\limits_{\epsilon}^{\infty}L^t_{b,\tx}k^b_t(x,\tx)f(\tx)\, d\tx. 
\end{equation}
If $\ell\geq 2$ and $b>0$, we can integrate by parts twice in $\tx$ and let $\epsilon \to 0$ to obtain that
\begin{equation}\label{eqn31new.0}
\pa_tv(x,t)=\int\limits_{0}^{\infty}k^b_t(x,\tx)L_{b}f(\tx)\, d\tx,
\end{equation}
and hence $\pa_tv\in\cC^0(\bbR_+\times\bbR_+)$. In particular, if $f\in\cC^2(\bbR_+),$ then 
$v\in\cC^{2,1}(\bbR_+\times\bbR_+).$ Using~\eqref{eqn28new.0} and~\eqref{eqn31new.0} inductively shows
that $v$ has the stated regularity and also gives \eqref{eqn28new.0} for all $p,q$ with $p+2q\leq \ell.$ 
The result for $b=0$ follows from the formula~\eqref{eqn28new.0} above when $b>0$ and Proposition 7.8 in~\cite{WF1d}.
\end{proof}

Now turn to the inhomogeneous equation, $(\pa_t - L_b)u = g$, $u(\cdot,0) = 0$.  The solution is given 
by the Duhamel formula
\begin{equation}
u(x,t)=\int\limits_{0}^t\int\limits_0^{\infty}k^b_{t-s}(x,\tx)g(\tx,s)\, d\tx ds.
\end{equation}
Since $k^b_t(x,\tx) \to \delta(x-\tx)$ as $t\to 0^+,$  we understand this to mean that 
\begin{equation}\label{eqn34new.2}
u(x,t)=
\lim_{\epsilon\to 0^+}\int\limits_{0}^{t-\epsilon}\int\limits_0^{\infty}k^b_{t-s}(x,\tx) g(\tx,s)\, d\tx ds.
\end{equation}
Denote this Volterra operator by $K^b_t.$ Lemma~\ref{lem3.1new.0} implies the basic regularity result.
\begin{lemma}\label{lem3.2new} If $g\in\cC^{\ell,\frac{\ell}{2}}(\bbR_+\times [0,T])$, then $u=K^b_tg \in \cC^{\ell,\frac{\ell}{2}}(\bbR_+\times [0,T])$;
furthermore, for any $p,q$ with $p+2q\leq \ell,$ 
\begin{equation}\label{eqn217.0.4}
\pa_t^q\pa_x^p u=K^{b+p}_tL_{b+p}^q\pa_{\tx}^pg+\sum_{l=0}^{q-1}L_{b+p}^l
\pa_t^{q-l-1}\pa_x^pg(x,t),
\end{equation}
where the sum is absent when $q=0.$ 
\end{lemma}
\begin{proof}
Let
\begin{equation}
G(\tau,s,x)=\int\limits_0^{\infty}k^b_{\tau}(x,\tx)g(\tx,s)\, d\tx;
  \end{equation}
Lemma~\eqref{lem3.1new.0} implies that
$\pa_x^i\pa_t^j\pa_{\tau}^kG\in\cC^{0}(\bbR_+^3),$ provided $i+2(j+k)\leq \ell.$ From this it follows easily that
\begin{equation}
u(x,t)=\int\limits_{0}^tG(t-s,s,x) g(\tx, s)\, ds\in\cC^{\ell,\frac \ell 2}(\bbR_+\times\bbR_+). 
\end{equation}

Equation~\eqref{eqn217.0.4} with $q=0$ follows directly from Lemma~\ref{lem3.1new.0}. Without loss of 
generality, we can therefore let $p=0,$ and prove the remainder of the formula by induction. The case $q=1$ 
is simply the equation,
\begin{equation}
\pa_tu=L_bu+g.
\end{equation}
Assume that the formula holds for some $q$ with $2q\leq \ell-2.$ The results in~\cite{WF1d} show that we can 
differentiate the equation in~\eqref{eqn217.0.4} to obtain that
\begin{equation}
\pa_t^{q+1}u=\pa_t K^{b}_tL_b^qg+\sum_{l=0}^{q-1}L_b^l\pa_t^{q-l}g.
\end{equation}
The first term on the right equals $L_bK^{b}_tL_b^qg+L_b^qg.$ This completes the proof 
since $L_bK^{b}_tL_b^qg=K^{b}_tL_b^{q+1}g.$
\end{proof}

\section{The model problem in higher dimensions}
We now generalize these results and formul{\ae} to the higher dimensional model operators $L_{\bb,m}$. 
Using the multiplicative nature of heat kernels, we can immediately write the model heat kernels in
terms of the one-dimensional ones, 
\begin{equation}
k_t^{\bb,m}(\bx, \by, \tbx, \tby) = k_t^{\bb}(\bx,\tbx)k_t^{\euc,m}(\by,\tby),
\end{equation}
where \index{$k_t^{\bb,m}$}\index{$k_t^{\bb}$}\index{heat kernel, higher
  dimensional model problem}
\begin{equation}
k_t^{\bb}(\bx,\tbx)=\prod\limits_{i=1}^nk_t^{b_i}(x_i,\tx_i), \quad
k_t^{\euc,m}(\by,\tby)=\frac{1}{(4\pi t)^{\frac m2}} e^{-\frac{|\by-\tby|^2}{4t}}.
\end{equation}
Note that this makes sense, even when $b_i=0$ for some indices. This is because
there is, at most, one $\delta$-factor in each coordinate.
The general problem is
\begin{equation}\label{hdmodint}
\begin{cases}
\pa_tw-L_{\bb,m} w =g \quad & \text{on}\  S_{n,m} \times (0,T]\\ w(\bx,\by,0) =f(\bx,\by), & (\bx,\by) \in S_{n,m}.
\end{cases}
\end{equation}
Uniqueness of moderate growth solutions with appropriate regularity and moderate growth data 
is then given by Corollary~\ref{uniquehigherdimmod}, and this solution has the integral representation
\begin{equation}
\begin{aligned}
w(\bx,\by,t)=\int_{S_{n,m}} \int\limits_0^t & k^{\bb,m}_{t-s}(\bx,\by, \tbx, \tby)g(\tbx,\tby,s) \, d\tbx d\tby ds \\
+ & \int_{S_{n,m}} k^{\bb,m}_t(\bx,\by,\tbx,\tby)f(\tbx,\tby)\, d\tbx d\tby. 
\end{aligned}
\label{repformhigherd}
\end{equation}
As before we discuss the homogeneous ($g=0$) and inhomogeneous ($f=0$) problems
separately, analyzing regularity in the elementary spaces $\cC^\ell(S_{n,m})$
and $\cC^{\ell,\frac \ell2}(S_{n,m} \times[0, T])$, which are defined as the
completions of the spaces $\cC^\infty_c(S_{n,m})$ and $\cC^\infty_c(S_{n,m}
\times [0,T])$ with respect to the norms
\begin{multline}
||f||_{\ell} = \max_{|\balpha| + |\bBeta| \leq \ell} 
\|\del_{\bx}^{\balpha} \del_{\by}^{\bBeta} f\|_{L^{\infty}(S_{n,m})} \mbox{and}\\
||g||_{\ell, \frac \ell 2} = \max_{2j + |\balpha| + |\bBeta| \leq \ell}
 \|\del_t^j\del_{\bx}^{\balpha} \del_{\by}^{\bBeta}
g\|_{L^{\infty}(S_{n,m}\times [0,T])},
\end{multline}
respectively. 

The following result will be helpful below.
\begin{proposition}\label{prop3.1new.0} Fix $\bar{\bb} \in \RR_+^n$ and $f\in\cC^{0}(S_{n,m})$. For any
$\bb \in \RR_+^n$, let $v^{\bb}$ be the unique moderate growth solution of \eqref{hdmodint} with 
Cauchy data $f$ (and with $g=0$). Then $\lim_{\bb \to \bar{\bb}} v^{\bb} = v^{\bar{\bb}}$ in $\cC^{0}(S_{n,m} \times [0,T]),$ 
for every $T>0.$
\end{proposition}
\begin{proof} 
Note that the result is trivial if all entries of $\bar{\bb}$ are strictly positive since $k^{\bb,m}_t$ varies smoothly
with $\bb \in (0,\infty)^n$ and is uniformly integrable, so we assume that some entries of $\bar{\bb}$ vanish.
This result is the multi-dimensional generalization of \cite[Prop.7.8]{WF1d}, and we review the proof of this one-dimensional 
case because we wish to use this same argument inductively. That proof both starts the induction and provides the inductive step.

So, first let $f \in \cC^0(\RR_+)$. If $f(0) = 0$, then we can approximate $f$ uniformly by $f_\ell \in \cC^0_c((0,\infty))$.
Since $k^b_t$ converges smoothly to $k^0_t$ away from $x = \tx = 0$, it is clear that $k^b_t f_\ell \to k^0_t f_\ell$.
Now estimate 
\[
|k^b_t f - k^0_t f| \leq |k^b_t (f - f_\ell)| + |k^0_t (f - f_\ell)| + | (k^b_t - k^0_t)f_\ell|.
\]
Given $f$, choose $\ell$ so that $\sup |f - f_\ell| < \epsilon$; by the maximum principle, the first 
and second terms are each less than $\epsilon$. Then, for this $\ell$, choose $b$ sufficiently
small so that the third term is less than $\epsilon$ too.  

For arbitrary $f \in \cC^0(\RR_+)$, choose a smooth cutoff $\chi(x)$ which equals $1$ for $x \leq 1$
and vanishes for $x \geq 2$, and write 
\[
f(x) = f(0) + \chi(x) ( f(x) - f(0)) + (1-\chi(x))(f(x) - f(0)). 
\]
Applying $k^b_t$ to this sum, then $k^b_t f(0) = f(0)$ for all $b$, and the other two terms 
vanish at zero, so we may apply the previous reasoning to each of them.  This proves the 
one-dimensional case.

Now consider the higher dimensional case. For simplicity, assume that $b_n = 0$; write $\bb' = (b_1, \ldots, b_{n-1})$ 
and $\bar{\bb}' = (\bar{b}_1, \ldots, \bar{b}_{n-1})$, and also set $\bx' = (x_1, \ldots, x_{n-1})$.  Suppose that 
$f \in \cC^0(S_{n,m})$. Decompose $f(\bx,\by)$ as in the one-dimensional case, as
\[
f(\bx', 0, \by) + \chi(x_n) ( f(\bx,\by) - f(\bx',0,\by)) + (1 - \chi(x_n))(f(\bx, \by) -  f(\bx',0,\by)). 
\]
Since the first term is independent of $x_n$, we have
\begin{multline*}
\int_{S_{n,m}} k^{\bb,m}_t (\bx, \by, \tbx, \tby) f(\tbx',0,\tby)\, d\tbx d\tby = \\
\int_{S_{n-1,m}} k^{\bb',m}_t (\bx', \by', \tbx', \tby') f(\tbx',0,\tby)\, d\tbx' d\tby,
\end{multline*}
so we may apply the inductive hypothesis to see that this is continuous in $\bb'$ up to $\bar{\bb}'$.
The third term is supported away from $x_n = 0$ already, so the result is clear for this term. Finally,
for the second term, which we denote by $f_2$, we can argue exactly as in the one-dimensional 
case, choosing a continuous function $h_2$ supported away from $x_n = 0$ and 
such that $\sup |f_2 - h_2| < \epsilon$. Then $| k^{\bb,m}_t (f_2 -h_2)| < \epsilon$
and $| k^{\bar{\bb},m}_t (f_2 -h_2)| < \epsilon$, and we may choose $\bb$ sufficiently
close to $\bar{\bb}$ so that $\sup |(k^{\bb,m}_t - k^{\bar{\bb},m}_t) h_2| < \epsilon$ too.
\end{proof}

\begin{remark} In cases where some of the $\{b_j\}$ vanish, the solution kernel
  is quite a bit more complicated than when all the $b_j>0.$ Suppose $k<n$ and
  that $b_1=\cdots=b_k=0,$ but $b_j>0$ for $k=k+1,\dots,n.$ The heat kernel for
  $L_{\bb,0}$ takes the
form
\begin{equation}
  k^{\bb}_{t}(\bx,\tbx)=\prod_{j=1}^{k}[k^{0,D}_{t}(x_j,\tx_j)+e^{-\frac{x_j}{t}}\delta_0(\tx_j)]
\prod_{j=k+1}^{n}k^{b_j}_t(x_j,\tx_j).
\end{equation}
If $H_j=\{\tx_j=0\},$ then this kernel has a $\delta$-distribution on  the
incoming boundary strata $$\{H_{i_1}\cap\cdots\cap H_{i_l}:\: \forall 1\leq l\leq k,
\text{ and sequences: }1\leq i_1<\cdots<i_l\leq k\}.$$
A similar result is given in~\cite{Shimakura1,Shimakura2} for the case of the Fleming-Viot
operator defined on the simplex.
\end{remark}

The analogue of Lemma~\ref{lem3.1new.0} is 
\begin{proposition}\label{lem3.3new0.0} For $\ell\in\bbN,$ if $f\in\cC^\ell(S_{n,m})$, then 
$v\in\cC^{\ell,\frac \ell 2}(S_{n,m}\times [0,\infty)).$ For $j\in \bbN_0$ and any multi-index 
of nonnegative integers $\balpha$ and $\bBeta$,  if $2j+|\balpha|+|\bBeta|\leq \ell,$ then
\begin{equation}\label{eqn47new.03}
\pa_t^j\pa_{\bx}^{\balpha}\pa_{\by}^{\bBeta}v=
\int_{S_{n,m}} k^{\bb+\balpha,m}_t(\bx,\by,\tbx,\tby) L_{\bb+\balpha,m}^j \pa_{\tbx}^{\balpha}\pa_{\tby}^{\bBeta}f(\tbx,\tby)
\, d\tbx d\tby
  \end{equation}
\end{proposition}
\begin{proof} We assume that all entries $b_i >0$, since if some $b_i = 0$ then we can prove the result
for an approximating sequence $\bb^{(j)} \to \bb$ with all $b^{(j)}_i > 0$ and then apply the previous Proposition. 

For $t>0$ (and all $b_i > 0$) the kernel is smooth in $(t,\bx,\by)$ and we can differentiate under the integral 
sign.  Using \eqref{adjeqn2} and \cite[Cor. 7.4]{WF1d}, we obtain~\eqref{eqn47new.03} with $j = 0$. 

The argument needed to handle the $\del_t$ derivatives is slightly more delicate. We claim that
\begin{equation}\label{eqn6.29.4}
\pa_tk_t^b(x,\tx)=(x\pa_x^2+b\pa_x)k^b_t(x,\tx)=(\pa_{\tx}^2\tx-b\pa_{\tx})k^b_t(x,\tx),
\end{equation}
but some care is needed because $\pa_{\tx}^2\tx k^b_t$ and $\pa_{\tx} k^b_t$ are not integrable at 
$0$ when $b\leq 1.$ We overcome this issue with the following lemma.
\begin{lemma}\label{lem6.0.9} If $f\in\cC^2(S_{n,m}),$ then
\begin{multline}\label{eqn6.30.1}
L_{b_i,x_i}  \int_{S_{n,m}} k^{\bb,m}_t(\bx,\by, \tbx, \tby)  f(\tbx,\tby)\, d\tbx d\tby=\\
\int_{S_{n,m}} k^{\bb,m}_t(\bx,\by, \tbx, \tby) L_{b_i,\tx_i} f(\tbx,\tby)\, d\tbx d\tby
\end{multline}
\end{lemma}
\begin{proof}[Proof of Lemma]
To simplify notation, relabel so that $i=1$, and write $S_{n-1,m} = \RR_+^{n-1}\times \RR^m$. 
First assume that $b_1>0.$ If $f\in L^{\infty}(\bbR_+),$ then
\begin{equation}
\int\limits_{0}^{\infty}L_{b_1,x_1}k^{b_1}_t(x_1,\tx_1)f(\tx_1)\, d\tx_1\text{ and }
\int\limits_{0}^{\infty}\pa_tk^{b_1}_t(x_1,\tx_1)f(\tx_1)\, dy_1
\end{equation}
are both absolutely integrable when $t>0.$ Thus using~\eqref{eqn6.29.4}, we
re-express
\begin{multline}
L_{b_i,x_i} \int\limits_{S_{n,m}} k^{\bb,m}_t(\bx,\by,\tbx,\tby) f(\tbx,\tby)\, d\tbx d\tby =\\ 
\lim_{\epsilon\to   0^+}\int\limits_{\epsilon}^{\infty}L^t_{b_1,\tx_1}k^{b_1}_t(x_1,\tx_1)
\int\limits_{S_{n-1,m}} \prod_{i=2}^nk^{b_i}_t(x_i,\tx_i) k_t^{\euc,m}(\by,\tby)f(\tbx,\tby)\, d\tbx d\tby.
\end{multline}
Since $f\in \cC^2(S_{n,m})$ and $\lim_{\tx_1\to 0^+}(\pa_{\tx_1}\tx_1-b_1)k^{b_1}_t(x_1,\tx_1)=0$, 
we can integrate by parts in the $\tx_1$-integral and let $\epsilon\to 0,$ to
obtain~\eqref{eqn6.30.1}. The case $b_1=0$ follows by applying Proposition~\ref{prop3.1new.0}.
\end{proof}
The assertion that 
\begin{multline}
\pa_t\int\limits_{S_{n,m}} k^{\bb,m}_t(\bx,\by, \tbx, \tby) f(\tbx,\tby)\, d\tbx d\tby= \\
\int_{S_{n,m}} k^{\bb,m}_t(\bx,\by, \tbx, \tby) L_{\bb,m}f(\tbx,\tby)\, d\tbx d\tby
\end{multline}
follows directly from this lemma.

The fact that $v\in \cC^{\ell, \frac \ell 2}$ if $f \in \cC^\ell$ then follows from these formul\ae\ and the
more elementary fact that if $f\in\cC^0(S_{n,m})$, then $v \in \cC^0(S_{n,m} \times [0,\infty)).$  
\end{proof}

Finally consider the inhomogeneous problem
\begin{equation}\label{inhomprb0.4}
(\pa_t-L_{\bb,m})u=g,\qquad u(\bx,\by,0)=0,
\end{equation}
with solution given by \eqref{hdmodint}.
This is treated just as before. Let 
\begin{equation}
G(\tau,s,\bx,\by)= \int_{S_{n,m}} k^{\bb,m}_{\tau}(\bx,\by,\tbx,\tby) g(\tbx,\tby,s)\, d\tbx d\tby,
\end{equation}
and observe that
\begin{equation}\label{eqn49new.0.1}
u(\bx,\by,t)=\lim_{\epsilon\to 0^+}\int\limits_{0}^{t-\epsilon}G(t-s,s,\bx,\by)\, ds.
\end{equation}
The following result is an immediate consequence of Proposition~\ref{lem3.3new0.0} and~\eqref{eqn49new.0.1}:
\begin{proposition}\label{lem3.4new0.0} If $g\in\cC^{\ell,\frac \ell 2}(S_{n,m} \times [0,T])$ for some $\ell \in \bbN$, 
then the unique bounded solution $u$ to~\eqref{inhomprb0.4} is given by~\eqref{eqn49new.0.1} and lies in 
$\cC^{\ell,\frac \ell 2}(S_{n,m} \times [0,T]).$ If $2j+|\balpha|+|\bBeta|\leq l,$ then
\begin{multline}\label{eqn50new.03}
\pa_t^j\pa_{\bx}^{\balpha}\pa_{\by}^{\bBeta}u= \int_{S_{n,m}} k^{\bb+\balpha,m}_t(\bx,\by,\tbx,\tby)
(L_{\bb+\balpha,m})^j \pa_{\tbx}^{\balpha}\pa_{\tby}^{\bBeta}g(\tbx,\tby,s)\, d\tbx d\tby\\
+ \sum_{r=0}^{j-1}\pa_t^{j-r-1}(L_{\bb+\balpha,m})^r\pa_{\by}^{\bBeta}\pa_{\bx}^{\balpha}g
\end{multline}
\end{proposition}

We prove one final result, concerning the behavior at spatial infinity of solutions corresponding
to compactly supported data $f, g$. 
\begin{proposition}\label{prop3.4newdcyinfty} Let $w$ be given by \eqref{hdmodint} with data
$f\in\cC^0_c(S_{n,m})$ and $g\in\cC^0_c(S_{n,m} \times [0,T])$. For any sets of nonnegative integers
$\balpha,\bBeta$ and $N$, 
\begin{equation}
\begin{split}
&\lim_{|(\bx,\by)|\to\infty}(1+|(\bx,\by)|)^N|\pa_{\bx}^{\balpha}\pa_{\by}^{\bBeta}
v(\bx,\by,t)|=0\\
&\lim_{|(\bx,\by)|\to\infty}(1+|(\bx,\by)|)^N|\pa_{\bx}^{\balpha}\pa_{\by}^{\bBeta}
u(\bx,\by,t)|=0    
\end{split}
  \end{equation}
for each $t > 0$. 
\end{proposition}
\begin{proof} This follows easily from the fact that the singularities of these kernels are 
located on the diagonal at $t=0,$ and from their exponential rates of decay at spatial infinity. 
If the incoming variables $\tbx, \tby$ are confined to a compact set, this decay is uniform for $t\in [0,T]$
for any $T < \infty$.
\end{proof}

\section{Holomorphic extension}\label{s.1dholoext}
The kernel functions $k^b_t(x,y)$ extend to be analytic for $t$ lying in the right half plane $\Re t>0.$  
By the permanence of functional relations, the functional equation
\begin{equation}
\int\limits_{0}^{\infty}k^b_t(x,y)k^b_s(y,z)dy=k^b_{t+s}(x,y)
\end{equation}
holds for $s$ and $t$ in this half plane. Therefore the solution to the homogeneous Cauchy problem,
\begin{equation}
v(x,t)=\int\limits_{0}^{\infty}k^b_t(x,y)f(y)dy
\end{equation}
is analytic in $\Re t>0,$ and satisfies
\begin{equation}
  \pa_t v-L_bv=0\text{ where }\Re t>0,
\end{equation}
where $\pa_t$ is the complex derivative.

If we let $t=\tau e^{i\theta},$
where $\tau\in (0,\infty)$ and $|\theta|<\frac{\pi}{2},$ then
\begin{equation}
  k_{\tau e^{i\theta}}^b(x,y)
=\frac{y^{b-1}e^{-ib\theta}}{\tau^b}e^{-\frac{(x+y)e^{-i\theta}}{\tau}}
\psi_b\left(\frac{xye^{-2i\theta}}{\tau^2}\right).
\end{equation}
For any $\alpha>0,$ the asymptotic expansion\index{asymptotic expansion, $\psi_b$}
\begin{equation}
  \psi_b(z)\sim\frac{z^{\frac{1}{4}-\frac{b}{2}}e^{2\sqrt{z}}}{\sqrt{4\pi}}\left(
1+\sum_{j=1}^\infty\frac{c_{b,j}}{z^{\frac j2}}\right)
\end{equation}
holds uniformly for $|\arg z|\leq \pi-\alpha.$ This shows that the kernel has an
asymptotic expansion:
\begin{equation}\label{eqn6.25.00}
   k_{\tau e^{i\theta}}^b(x,y)\sim\frac{e^{-\frac{i\theta}{2}}}{\tau^b\sqrt{y}}
\left(\frac{y}{x}\right)^{\frac{b}{2}-\frac{1}{4}}
e^{-\frac{(\sqrt{x}-\sqrt{y})^2e^{-i\theta}}{\tau}} 
\left(1+\sum_{j=1}^{\infty}c_{b,j}\frac{\tau^j e^{ij\theta}}{(xy)^{\frac j2}}\right).
\end{equation}
This explicit expression shows that the qualitative behavior of this kernel, as $\tau\to
0,$ is uniform in sectors 
\begin{equation}\label{sectordef.0}
  S_{\phi}=\{t:|\arg t|<\frac{\pi}{2}-\phi\},
\end{equation}
 for any $\phi>0.$ Moreover, if $f\in \cC^0_b([0,\infty)),$ then
 \begin{equation}
   \lim_{t\to 0 \atop t \in S_{\phi}}v(x,t)=f(x)
 \end{equation}
uniformly in the $\cC^0$-topology.

In the sequel we shall also be analyzing the resolvent $R(\mu) = (L_b - \mu)^{-1}$. 
If $\mu\in (0,\infty)$ and $f\in\cC^0_b([0,\infty))$, then $R(\mu) f$ is defined by expression
\begin{equation}
R(\mu)f = \lim_{\epsilon\to 0^+} \int\limits_{\epsilon}^{\frac{1}{\epsilon}}\left(\int\limits_{0}^{\infty} k_t^b(x,y)f(y)\, dy\right)\, e^{-\mu t}\, dt.
\end{equation}
Using the asymptotics of $k^b_t$, we can apply Morera's theorem to show that
for each fixed $x\in\bbR_+,$ $R(\mu)f(x)$ is an analytic function of $\mu$ in
the right half plane.  Applying $L_b$ to the integral on the right and
integrating by parts, and using estimates proved below, we show that if $f$
is H\"older continuous and bounded on $[0,\infty),$ then
\begin{equation}\label{eqn6.29.00}
(\mu-L_b)R(\mu)f=f.
\end{equation}

Using Cauchy's theorem and the asymptotics of the heat kernel, the contour for the $t$-integral can 
be deformed to show that, for $\mu\in (0,\infty)$ and $|\arg\theta|<\frac{\pi}{2}$ we also have
\begin{equation}\label{eqn6.30.00}
R(\mu)f=\lim_{\epsilon\to 0^+} \left(\int\limits_{\epsilon}^{\frac{1}{\epsilon}} \int\limits_{0}^{\infty}
k_{\tau e^{i\theta}}^b(x,y)f(y)\, dy\right) \, e^{-\mu \tau e^{i\theta}}e^{i\theta}\, d\tau.
\end{equation}\index{resolvent operator, model problem}
This expression is analytic in $\mu$ in the region where $\Re(\mu e^{i\theta})>0,$ and hence $R(\mu)f$ 
extends analytically to $\bbC\setminus (-\infty,0]$, and this extension satisfies~\eqref{eqn6.29.00}. 

% For the purposes of analyzing these degenerate operators, we define specially adapted 
% H\"older spaces. For $0<\gamma<1,$ we say that
% $f\in\cC^{0,\gamma}_{\WF}(\bbR_+)$ if $f\in\cC^0(\bbR_+),$ tends to
% zero at infinity, and
% \begin{equation}
%   \bbr{f}_{\WF,0,\gamma}=\sup_{x\neq y}\frac{|f(x)-f(y)|}
% {|\sqrt{x}-\sqrt{y}|^{\gamma}}<\infty.
% \end{equation}
% These H\"older spaces are defined in Chapter~\ref{chap.holdspces}.
For $f \in \cC^0$, we show that the following limits exist locally uniformly for $x\in (0,\infty):$
\begin{equation}
  \begin{split}
    &\lim_{\epsilon\to 0^+}
\int\limits_{\epsilon}^{\frac{1}{\epsilon}}
\int\limits_{0}^{\infty}
\pa_xk_{\tau e^{i\theta}}^b(x,y)f(y)dye^{-\mu \tau
  e^{i\theta}}e^{i\theta}d\tau\\
&\lim_{\epsilon\to 0^+}
\int\limits_{\epsilon}^{\frac{1}{\epsilon}}
\int\limits_{0}^{\infty}
x\pa_x^2k_{\tau e^{i\theta}}^b(x,y)f(y)dye^{-\mu \tau e^{i\theta}}e^{i\theta}d\tau.
  \end{split}
\end{equation}
This demonstrates the $R(\mu) f$ is twice differentiable in $(0,\infty).$ 
If $f$ is only in $\cC^0([0,\infty)),$ then
the limits $\lim_{x\to 0^+}\pa_x R(\mu) f(x),$ and $\lim_{x\to 0^+}x\pa_x^2
R(\mu) f(x)=0$ may not exist 

For $|\theta|<\frac{\pi}{2}$ and $\epsilon>0,$ we let
\begin{equation}
  \Gamma_{\theta,\epsilon}=\{t\eit:\: t\in [\epsilon,\epsilon^{-1}]\}.
\end{equation}
For $t\in S_0$ the kernel function satisfies the equation:
\begin{equation}
  \pa_t k^b_t(x,y)=L_bk^b_t(x,y),
\end{equation}
where $\pa_t$ is the complex derivative. For $\epsilon>0$ we see that
\begin{equation}\label{resolinprts0}
  L_b \int\limits_{\Gamma_{\theta,\epsilon}}\int\limits_{0}^{\infty}
k_{t}^b(x,y)f(y)dye^{-\mu t}dydt=
\int\limits_{\Gamma_{\theta,\epsilon}}\int\limits_{0}^{\infty}
\pa_t k_{t}^b(x,y)f(y)dye^{-\mu t}dydt.
\end{equation}
For $\epsilon>0$ we can interchange the order of the integrations and integrate
by parts to obtain:
\begin{multline}\label{resolinprts1}
  L_b \int\limits_{\Gamma_{\theta,\epsilon}}\int\limits_{0}^{\infty}
k_{t}^b(x,y)f(y)dye^{-\mu t}dydt=
\mu\int\limits_{\Gamma_{\theta,\epsilon}}\int\limits_{0}^{\infty}
 k_{t}^b(x,y)f(y)dye^{-\mu t}dydt+\\
\int\limits_{0}^{\infty}
 k_{\epsilon^{-1}\eit}^b(x,y)f(y)dye^{-\mu \epsilon^{-1}\eit}dy-
\int\limits_{0}^{\infty}
 k_{\epsilon\eit}^b(x,y)f(y)dye^{-\mu \epsilon\eit}dy.
\end{multline}
Provided that $\Re (\eit\mu ) >0,$ and $f\in\cC^0(\bbR_+),$ we can 
let $\epsilon\to 0^+$ to obtain that
\begin{equation}
(\mu-L_b)R(\mu)f=f.
\end{equation}

We summarize these observations as a proposition:
\begin{proposition} The solution to the homogeneous Cauchy problem $(\pa_t  -L_b)v=0$ 
with $v(x,0)=f(x)\in\cC^0_b([0,\infty)),$ extends to an analytic function of $t$ with 
$\Re t>0.$ The resolvent operator $R(\mu)$ is analytic in the complement of $(-\infty,0],$ 
and is given by the integral   in~\eqref{eqn6.30.00} provided that $\Re(\mu e^{i\theta})>0.$ 
% If, for a
%   $\gamma>0,$  $f\in\cC^{0,\gamma}_{\WF}(\bbR_+),$  then 
Moreover, $(\mu-L_b)R(\mu)f=f,$ for $\mu\in\bbC\setminus (-\infty,0].$
\end{proposition}

From the corresponding fact for its one-dimensional factors, 
it is obvious that the kernel $k_t^{\bb,m}(\bx, \by, \tbx, \tby)$ extends analytically to $\Re t>0$ 
and hence  the solution $v^{\bb}(\bx,\by,t)$ does as well. Indeed, if $f\in\cC^0_b(S_{n,m})$ 
then $t\mapsto v^{\bb}(\cdot,\cdot,t)$ is an analytic function from the right half plane with values in
$\cC^{\infty}_b(S_{n,m}).$ The asymptotic formula~\eqref{eqn6.25.00} and the standard asymptotics for the
Euclidean heat kernel then give that for any $\phi>0,$
\begin{equation}
\lim_{t\to 0 \text{ in } S_{\phi}}v^{\bb}(\cdot,\cdot,t)=f,
\end{equation}
in the uniform topology.

For $f\in\cC^0_b(S_{n,m})$ and $\Re \mu>0$, the Laplace transform is defined, by the limit
\begin{equation}\label{eqn6.36.01}
R(\mu)f(\bx,\by)=\lim_{\epsilon\to 0^+}
\int\limits_{\epsilon}^{\frac{1}{\epsilon}} v^{\bb}(\bx,\by,t)e^{-\mu t}\, dt.
\end{equation}
Assuming that $f \in \cC^{0}_{\WF}(S_{n,m})$ then again using the analyticity
and asymptotic behavior of the kernel, we can use Cauchy's theorem to deform
the contour of integration in~\eqref{eqn6.36.01}. For $|\theta|<\frac{\pi}{2},$
and $\mu\in (0,\infty)$ we have that
\begin{equation}\label{eqn6.38.00}
R(\mu)f(\bx,\by)=\lim_{\epsilon\to  0^+}\int\limits_{\epsilon}^{\frac{1}{\epsilon}}
v^{\bb}(\bx,\by,\tau e^{i\theta}) e^{-\mu\tau e^{i\theta} }e^{i\theta}\, d\tau.
\end{equation}
The expression in~\eqref{eqn6.38.00} defines an analytic function of $\mu$
where $\Re \mu e^{i\theta}>0.$ This in turn shows that $R(\mu)f$ has an
analytic continuation to $\bbC\setminus (-\infty,0].$ 

In order to establish the identity
\begin{equation}
(\mu-L_{\bb,m}) R(\mu)f =f
\label{residmd}
\end{equation}
in the higher dimensional case, it is simpler to assume that $f$ is H\"older
continuous. Specifically, in the next chapter we shall define H\"older spaces
$\cC^{0,\gamma}_{\WF}(S_{n,m})$ which are specially adapted to this
problem. For such data one can show that the individual terms in
$L_{\bb,m}R(\mu)f$ are continuous on $S_{n,m},$ and satisfy~\eqref{residmd}. If
$f$ is only continuous, then, arguing as before one can show that the limit,  as $\epsilon\to 0^+$ of
\begin{equation}
  L_{\bb,m}\int\limits_{\epsilon}^{\frac{1}{\epsilon}}
v^{\bb}(\bx,\by,\tau e^{i\theta}) e^{-\mu\tau e^{i\theta} }e^{i\theta}\, d\tau.
\end{equation}
exists in $\cC^0(S_{n,m}).$ It satisfies the identity $(\mu-L_{\bb,m}) R(\mu)f
=f$ in the $\cC^0$-graph closure sense. Generally the individual terms of
$L_{\bb,m}R(\mu)f$ are not defined as $(\bx,\by)$ approaches the boundary of
$S_{n,m}.$

We summarize these results in a proposition.
\begin{proposition}\label{prop6.3.2.01} 
  The solution to the homogeneous Cauchy problem $(\pa_t -L_{\bb,m})v=0$ with
  $v(\bx,\by,0)=f(\bx,\by)\in\cC^0_b(S_{n,m}),$ extends to an
  analytic function of $t$ with $\Re t>0.$ The resolvent operator $R(\mu)$ is
  analytic in the complement of $(-\infty,0],$ and is given by the integral
  in~\eqref{eqn6.38.00} provided that $\Re(\mu e^{i\theta})>0.$ If $f \in \cC^{0}_{\WF}(S_{n,m})$,
  then $ (\mu- L_{\bb,m}) R(\mu)f =f,$ for $\mu\in\bbC\setminus (-\infty,0].$
\end{proposition}

\section{First steps toward perturbation theory}
Our primary goal in this monograph is to construct the solution operator $\calQ^t$ for a 
general Kimura operator $\del_t - L$ and to use it to study properties of the associated 
semi-group on various function spaces. This is done by perturbing an approximate solution
obtained by patching together the solution operators $\calQ^t_{\bb,m}$ for the models 
$\del_t - L_{\bb,m}$ associated to the normal forms of $L$ in various coordinate charts. 
This strategy works very well in the one-dimensional problem considered in \cite{WF1d}, 
but turns out to be substantially more complicated in higher dimensions.  We explain this now.

The relative simplicity of this method for operators in one dimension is not hard to explain.
As already pointed out in Chapter~\ref{s.nrmfrms}, the normal form for the second order part of 
a one-dimensional Kimura operator is {\it exactly} $x\del_x^2$ in a full neighborhood of a 
boundary point. Thus we can choose coordinates and a constant $b \geq 0$ so that 
$L = L_b + W$, where $W$ is a vector field which vanishes at the boundary point. Hence
the error term incurred by using $\calQ^t_b$ as an approximate solution operator
is $W \circ \calQ^t_b$. In an appropriate sense, this operator is smoothing of order $1$,
and restricted to data on suitably small time intervals $[0,T]$, it also has small norm
acting on continuous functions.  Hence it is easy to solve away this error term using
a convergent Neumann series.   

When carrying out the same procedure in higher dimensions, the difference between $L$ 
and any one of the models $L_{\bb,m}$ is unavoidably second order. Hence the error term
$E$ incurred by applying $\del_t - L$ to a parametrix formed by patching together these
model heat kernels is no longer smoothing, since it is the result of applying a differential
operator of order $2$ to an operator which is smoothing of order $2$. Even worse,
this error is {\it not} bounded on $\cC^0$. This is a well-known fact in classical potential
theory, that $\cC^0$ and higher $\cC^k$ spaces are ill-suited for the study of regularity
properties of elliptic and parabolic problems in higher dimensions, and that one
should use H\"older spaces instead.  

The applications of these Kimura diffusions in probability and biology demand 
that we study the semi-group for $L$ on $\cC^0$. This leaves us with a slightly unsatisfactory
state of affairs. We are only able to construct the solution operator for $\del_t - L$ on a suitable 
scale of H\"older spaces. We can still prove the existence and many properties of the
semi-group on $\cC^0$, but this must be done in an indirect fashion.

\chapter{Degenerate H\"older Spaces}\label{chap.holdspces}
The starting point to implement this perturbation theory is a description of
the various function spaces we shall be using. As described above, we seek
function spaces on the domain $P$ for which the diffusion associated to a
general Kimura diffusion operator $L$ is well posed. More pragmatically, we
wish to define spaces on which one can prove analogues of the standard
parabolic Schauder estimates, so that we can pass from the model to more
general operators. This chapter is devoted to a description of the various
spaces on which this is possible, and to an explanation of the relationships
between them.

Two familiar guiding principles when choosing the right function spaces for a problem
are that one should choose spaces which respect the natural scaling properties of the operator, and
in addition, that these spaces should be based on the geometry of an associated metric.
In the classical setting, the operator $\del_t - \sum \del_{y_j}^2$ on $\RR_+ \times \RR^m$
is homogeneous with respect to the parabolic dilations $(t, \by) \mapsto (\lambda^2 t, 
\lambda \by)$, and is naturally associated to the Euclidean metric. The first of these
principles indicates that $t$ and $y$ derivatives should be weighted differently; the second
suggests that if we formulate mapping properties in terms of H\"older spaces with
semi-norms defined using the Euclidean distance function. This is indeed the case, 
and we review the definitions of the standard parabolic H\"older spaces below. 
Other examples where these principles are applied 
include \cite{DaskHam} and \cite{Maz-edge}. 

To apply the same two principles in the present setting, we observe that 
$\del_t - L_{\bb,m}$ is homogeneous with respect to the slightly different scaling, 
\[
(t,\bx,\by) \mapsto (\lambda t, \lambda \bx, \sqrt{\lambda} \by),
\]
which indicates that derivatives with respect to $t$ and the $x_i$ should be twice as strong
as derivatives with respect to the $y_j$. On the other hand, when all $x_j$ are strictly positive,
we must revert to the standard scaling corresponding to the interior
problem. In other words, whatever function spaces we use must incorporate
both types of homogeneity. The metric naturally associated to $L_{\bb,m}$ is 
\begin{equation}
ds^2=\sum_{i=1}^n\frac{dx_i^2}{x_i}+\sum_{j=1}^mdy_j^2;
\end{equation}
note that this metric is homogeneous with respect to $(\bx,\by) \mapsto (\lambda \bx, \sqrt{\lambda} \by)$,
and that the associated Laplacian is simply $L_{\bb,m}$ with $\bb = (\frac12, \ldots, \frac12)$. 

Before embarking on the many definitions below, we make two remarks.
First, the basic definition of a H\"older semi-norm with respect to a given metric $g$  is
\begin{equation}
\bbr{u}_{g; 0,\gamma} := \sup_{z \neq z'}  \frac{ |u(z) - u(z')|}{ d_g(z,z')^\gamma},
\end{equation}
where $d_g$ is the Riemannian distance between the two points. It is very useful to
observe that instead of taking the supremum over all distinct pairs $z, z'$, it suffices
to take the supremum only over pairs with $z \neq z'$ and $d_g(z,z') \leq 1$.
This is simply because  if $d_g(z,z') > 1$, then the quotient on the right, evaluated
at $z,z'$, is bounded by $2\sup |u|$.  For this reason, we introduce the notation
\begin{equation}
\supone_{ z \neq z'}\equiv \sup\limits_{\{z\neq z':\: d(z,z')<1\}}, 
\end{equation}\index{$\supone$}
which will be used throughout the rest of this paper. This makes the semi-norm
monotonely increasing, as a function of $\gamma.$ For functions depending on
both $z$ and $t$, we use this same notation to denote the supremum over pairs
$(z,t) \neq (z',t')$ with $d_g(z,z') + \sqrt{|t-t'|} \leq 1$.

Second, although our main focus is on generalized Kimura diffusion operators $L$ on
compact regions $P$, it is convenient from certain technical points of view to
study the model operators $L_{\bb,m}$ on the unbounded region $\RR^n_+ \times
\RR^m$. In addition, there are some practical motivations for this since
certain problems arising in biological applications actually occur on such
unbounded orthants. We handle spatial infinity by defining appropriate H\"older
norms and then taking spaces which are the completions of the subspaces of
smooth compactly supported functions with respect to these norms The functions
obtained in this way must tend to zero at infinity along with an appropriate
number of scaled derivatives. This requires us to check that the solution
operators for these heat equations preserve this property. We denote the spaces
obtained by this closure procedure with a superscript dot. Thus, for example,
$\dcC^{0}(\RR^m)$ denotes the closure in $\cC^{0}(\RR^m)$ of the subspace of
compactly supported smooth functions; the space $\dcC^{0}(\RR^m)$ consists of
continuous functions which tend to zero at infinity.

\section{Standard H\"older spaces}
To be clear about notation and definitions, we briefly recall the
classical interior H\"older spaces and their parabolically scaled `heat'
analogues.  All spaces here are subspaces of $\dcC^0(\bbR^m)$. Here and in the
remainder of the book $\gamma$ denotes a number in the interval $(0,1).$

The space $\cC^{0,\gamma}(\bbR^m)$ is the subspace of $\dcC^{0}(\bbR^m)$ consisting of functions $f$ for which the norm
\begin{equation}
\|f\|_{0,\gamma}:=\|f\|_{L^{\infty}(\bbR^m)}+  \bbr{f}_{0,\gamma}
\end{equation}
is finite.  Here 
\begin{equation}
\bbr{f}_{0,\gamma}=\supone_{\by\neq \by'} \frac{|f(\by)-f(\by')|}{|\by-\by'|^{\gamma}}
\end{equation}
is the H\"older semi-norm of order $\gamma$.  Note that this is different from the so-called `little H\"older space'
$\dcC^{0,\gamma}(\RR^m)$, which is the closure of the space of smooth functions with bounded supported in
this H\"older norm and which consists of $\cC^{0,\gamma}$ functions such that
\begin{equation}
\lim_{\delta \to 0} \sup_{|\by-\by'| < \delta} \frac{|f(\by)-f(\by')|}{|\by-\by'|^{\gamma}} = 0.
\label{littleholder}
\end{equation}
Similarly, a function $f\in\dcC^k(\bbR^m)$ belongs to  $\cC^{k,\gamma}(\bbR^m)$ if the norm
\begin{equation}
\|f\|_{k,\gamma}=\|f\|_{\cC^k(\bbR^m)}+\sup_{ |\balpha|=k}\, \bbr{\pa^\alpha f}_{0,\gamma}
% \supone_{\by\neq \by'}
% \frac{|\pa^{\balpha}f(\by)-\pa^{\balpha}f(\by')|}{|\by-\by'|^{\gamma}}
\end{equation}
is finite. (This sup is over multi-indices $\balpha \in \bbN^m$, where $|\balpha| = \alpha_1 + 
\ldots + \alpha_m$.) 

Now consider functions which depend on both $\by$ and $t$. The heat H\"older spaces $\cC^{0,\gamma}(\bbR^m\times [0,T])$ 
are defined as the set of functions $g\in\dcC^0(\bbR^m\times [0,T])$ such that
\begin{equation}
\|g\|_{0,\gamma} :=\|g\|_{\infty}+ \bbr{g}_{0,\gamma} < \infty,
\end{equation}
where now 
\begin{equation}
\bbr{g}_{0,\gamma} = \supone_{(\by,t)\neq (\by',t')} \frac{|g(\by,t)-g(\by',t')|}{[|\by-\by'|+\sqrt{|t-t'|}]^{\gamma}}. 
\end{equation}
Finally, letting $\dcC^{k,\frac k2}(\bbR^m\times [0,T])$ denote the closure of $\cC^{\infty}_c(\bbR^m\times [0,T])$ with 
respect to the norm
\begin{equation}
\|g\|_{k,\frac  k2} :=\sup_{|\balpha|+2j\leq k} \|\pa_t^j\pa_{\by}^{\balpha}g\|_\infty, 
\end{equation}
then $\cC^{k,\gamma}(\bbR^m\times [0,T])$ consists of functions $g\in\dcC^{k,\frac k2}(\bbR^m\times [0,T])$ such that
\begin{equation}
\|g\|_{k,\gamma} :=\|g\|_{k,\frac k2}+\sup_{|\balpha|+2j=k}\, \bbr{\pa_t^j \pa^{\balpha}_{\by} g}_{0,\gamma} < \infty.
\end{equation}

Note that in all these cases, the Euclidean metric appears through the quantity $|\by-\by'|$ (which is comparable to 
$d_{\mathrm{euc}}(\by, \by')$) and that the parabolic scaling is reflected not only in the definition of $\cC^{k, \frac k2}$, but also by the quantity
$|\by - \by'| + \sqrt{|t-t'|}$. 

\section{WF H\"older spaces in one dimension}
We now turn to the definitions of the degenerate H\"older spaces associated to the one-dimensional
model operator $L_b$. As indicated above, one guide is the geometry 
on $\bbR_+$ with the incomplete metric
\begin{equation}
  ds^2_{\WF}=\frac{dx^2}{x}.
\end{equation}
Note that the change coordinates $\xi = 2\sqrt{x}$ transforms this
to the standard Euclidean metric $d\xi^2$, and that the model operator
$L_{1/2}$ is simply the Laplacian $\pa_\xi^2$.  This allows us to transform all
the standard H\"older theory for functions of $\xi$ (or $\xi$ and $t$) to
obtain the corresponding spaces and estimates for this particular operator
$L_{1/2}$. As we eventually show, these spaces and estimates also adapt to the
other operators $L_b$, although this requires more than a simple coordinate
transformation to verify.

This identification makes certain basic geometric formul{\ae} trivial to verify. We record these here, although they 
will not be used until a later chapter. First, the distance function has the explicit expression
\begin{equation}
\rho = d_{\WF}(x_1,x_2)=2|\sqrt{x_2}-\sqrt{x_1}|.
\end{equation}
Next, the midpoint of the interval $[x_1,x_2]$ with respect to $ds^2_{\WF}$ is $\bar{x}=(\sqrt{x_1}+\sqrt{x_2})^2/$.
Finally, the WF-ball $B_\rho(\bar{x})$ centered at the point $\bar{x}$ and with radius $\rho$ is the interval 
$[\alpha,\beta]$, where
\begin{equation}
\sqrt{\alpha}=\max\left\{0,\frac{3\sqrt{x_1}-\sqrt{x_2}}{2}\right\},\quad \sqrt{\beta}=\frac{3\sqrt{x_2}-\sqrt{x_1}}{2}.
\end{equation}

\medskip

\subsection{WF-H\"older spaces on $\RR_+$: } 
We now proceed to the definitions of the associated function spaces.
Following the dictum in the beginning of this chapter, the WF H\"older seminorm is given by
\begin{equation}
\bbr{f}_{\WF,0,\gamma} = \supone_{x\neq x'}\frac{|f(x)-f(x')|}{|\sqrt{x}-\sqrt{x'}|^{\gamma}} = 
2 \supone_{x \neq x'} \frac{|f(x)-f(x')|}{d_{\WF}(x,x')^{\gamma}}.
\index{$\cC_{\WF}^{0,\gamma}(\bbR_+)$}
\end{equation}
Then $\cC_{\WF}^{0,\gamma}(\bbR_+)$ is the subspace of $\dcC^0(\bbR_+)$ on which the norm
\begin{equation}
\|f\|_{\WF,\gamma}= ||f||_\infty +   \bbr{f}_{\WF,0,\gamma}
\end{equation}
is finite. This is clearly a Banach space. We also define $\cC^{0,1}_{\WF}(\bbR_+)$ to be the 
closure of $\cC^1(\bbR_+)$ with respect to the norm:
\begin{equation}
  \|f\|_{\WF,0,1}=\|f\|_{L^{\infty}}+\|\sqrt{x}\pa_{x}f\|_{L^{\infty}}.
\end{equation}
Note that if $f\in\cC^{0,1}_{\WF}(\bbR_+),$ then 
\begin{equation}
  \lim_{x\to 0^+}\sqrt{x}\pa_xf(x)=0,
\end{equation}
since this is true for every $f \in \cC^1$. Moreover, integration gives
\begin{equation}
  |f(x_2)-f(x_1)|\leq 2 \|f\|_{\WF,0,1}|\sqrt{x_2}-\sqrt{x_1}|,
\end{equation}
and hence, for any $0<\gamma<1,$ the inclusion
\begin{equation}
  \cC^{0,1}_{\WF}(\bbR_+)\subset \cC^{0,\gamma}_{\WF}(\bbR_+)
\end{equation}
is compact. 

Two simple facts will be used repeatedly below. First, if $f\in\cC^{0,\gamma}(\bbR_+),$ then
directly from the definition,
\begin{equation}\label{mstbscest}
|f(x)-f(x')|\leq 2\|f\|_{\WF,\gamma}|\sqrt{x}-\sqrt{x'}|^{\gamma}.
\end{equation}
Second, the basic inequality 
\[
|f(x)g(x)-f(x')g(x')|\leq |f(x)(g(x)-g(x'))|+|g(x')(f(x)-f(x'))|
\]
implies that these H\"older semi-norms satisfy a standard `Leibniz' rule: $f,g\in\cC^{0,\gamma}_{\WF}(\bbR_+),$ then
\begin{equation}\label{leibfrm0}
\bbr{fg}_{\WF,0,\gamma}\leq \|f\|_{\infty}\bbr{g}_{\WF,0,\gamma} +\|g\|_{\infty}\bbr{f}_{\WF,0,\gamma},
\end{equation}
(where $||\cdot||_\infty$ is the $L^\infty$ norm). 

There are in fact a couple of slightly different ways to define WF spaces which capture higher regularity. 
The ultimate goal is to capture the precise gain in regularity for elliptic and parabolic problems, which
leads us to the various definitions below.  

The first set of spaces is meant to capture the fact that if $L_b u = f$, then
we wish to be able to estimate $u$, $\del_x u$ and $x\del_x^2 u$ separately in
terms of $f$. Define $\dcC_{\WF}^2(\bbR_+)$ as the closure of $\cC^2_c(\bbR_+)$
with respect to the norm:
\begin{equation}
\|f\|_{\WF,2}=\|f\|_{\infty}+\|\pa_xf\|_{\infty}+\|x\pa_x^2f\|_{\infty},
\label{WFC2}
\end{equation}
and then let $\cC^{0,2+\gamma}_{\WF}(\bbR_+)$ be the subspace of
$\dcC_{\WF}^2(\bbR_+)$ on which the norm \index{$\cC_{\WF}^{0,2+\gamma}(\bbR_+)$}
\begin{equation}
\|f\|_{\WF,0,2+\gamma}=\|f\|_{\WF,0,\gamma}+\|\pa_xf\|_{\WF,0,\gamma}+\|x\pa_x^2f\|_{\WF,0,\gamma} 
\end{equation}
is finite. 

% Before proceeding, we note that there is a slight awkwardness about this definition, since the second and third terms 
% on the right in \eqref{WFC2} scale like $\lambda$ with respect to dilation $x \to \lambda x$, whereas the first term 
% scales like $\lambda^0 = 1$. We show below how to decompose these WF spaces into spaces which are fully 
% scale-invariant; this is important both conceptually and for various arguments below. 

As a matter of convention, we write $\RR_+$ for the closed half-line $[0,\infty)$ and denote the
open half-line by $(0,\infty)$.  Clearly $\cC^2_{\WF}(\RR_+) \subset \cC^2((0,\infty))\cap\cC^1(\RR_+)$. 
Furthermore, analogous to \eqref{littleholder}, since $\dcC^2_{\WF}(\RR_+)$ is the closure of $\cC^2_c(\bbR_+)$ 
with respect to \eqref{WFC2}, then for any $f \in \cC^2_{\WF}$, 
\begin{equation}\label{van2ndder0}
\lim_{x\to 0^+}x\pa_x^2f(x)=0,\quad\lim_{x\to\infty} \left(|f(x)|+|\pa_xf(x)|+|x\pa_x^2f(x)|\right)=0.
\end{equation}
The first assertion is an important part of the characterization of the domain of $L_b$ on $\cC^0$.

There is an elementary characterization of $\cC^{0,2+\gamma}_{\WF}$, which also gives a simple proof that it
is a Banach space.
 \begin{lemma}\label{lem4.1.2}
Suppose that $f\in\cC^2((0,\infty))\cap\cC^1(\RR_+)$ satisfies~\eqref{van2ndder0}, and 
that $\|f\|_{\WF,0,2+\gamma} < \infty$. Then $f\in\cC^2_{\WF}(\bbR_+).$
 \end{lemma}
 \begin{proof}
We must find a sequence $\{f_n\}$ in $\dcC^2(\bbR_+)$ such that $||f_n - f||_{\WF,0,2+\gamma} \to 0$. 
However, we know that 
\begin{equation}\label{2ndderestat0.4}
|x\pa_x^2 f(x)|\leq Mx^{\frac{\gamma}{2}}
\end{equation}
for $x < 1$ and some $M < \infty$. Letting $f_n(x)=f(x+1/n)$, then clearly 
$\|f_n-f\|_{\infty}+\|\pa_x (f_n-f)\|_{\infty} \to 0$.  

It remains to show that 
\begin{equation}
\lim_{n\to\infty}\|x\pa_x^2(f-f_n)\|_{\infty}=0.
\end{equation}
For any $0 < \delta < C$, the uniform convergence of these second derivatives on $[\delta, C]$ is
obvious; in fact, by the second part of~\eqref{van2ndder0}, there is also no difficulty as $x\to\infty$.
Now observe that 
\begin{multline*}
|x\pa_x^2f(x)-x\pa_x^2f_n(x)|=\\ 
 \left|x\pa_x^2f(x)-\left(x+\frac 1n\right)\pa_x^2f\left(x+\frac 1n\right)+\left(\frac 1n\right)\pa_x^2f\left(x+\frac 1n\right)\right|
 \end{multline*}
Using~\eqref{2ndderestat0.4} and the fact that $\|x\pa_x^2f\|_{\WF,0,\gamma} < \infty$, we see that
\begin{equation}
|x\pa_x^2(f(x)-f_n(x))|\leq \frac{1}{n^{\frac{\gamma}{2}}} \|x\pa_x^2f\|_{\WF,0,\gamma}
+  M\left(x+\frac{1}{n}\right)^{\frac{\gamma}{2}},
 \end{equation}
which implies that $\|x\pa_x^2(f(x)-f_n(x))\|_{\infty} \to 0$ as $n\to\infty.$
\end{proof}
\begin{corollary} If $0<\gamma<1$, then the topological vector space $\cC^{0,2+\gamma}_{\WF}(\bbR_+)$ is a Banach space.
\end{corollary}
\begin{proof} If $f_n$ is a sequence in $\cC^{0,2+\gamma}_{\WF}(\bbR_+)$ which converges to some $f$ in the 
$\cC^{0,2+\gamma}_{\WF}$-norm, then $f$ satisfies the hypotheses of the previous lemma. This shows that $f$ is
the $\cC^2_{\WF}$-limit of a sequence of functions in $\dcC^2([0,\infty))$ and hence $f\in \cC^{0,2+\gamma}_{\WF}(\bbR_+)$ as well.
\end{proof}

\medskip

\medskip

\subsection{Parabolic WF-H\"older spaces on $\RR_+ \times [0,T]$:}
We now introduce the parabolic (or `heat') WF-H\"older spaces $\cC_{\WF}^{0,\gamma}(\bbR_+\times [0,T])$ and
$\cC_{\WF}^{0,\gamma+2}(\bbR_+\times [0,T])$. To define these, first let $\dcC^{2,1}_{\WF}(\bbR_+\times[0,T])$ be 
the closure of $\cC^{2,1}_c(\bbR_+\times[0,T])$ with respect to the norm
\begin{equation}
\|g\|_{\WF,0,2,1}=\|g\|_{\infty}+\|\pa_xg\|_{\infty}+\|\pa_tg\|_{\infty}+\|x\pa_x^2g\|_{\infty}.
\end{equation}
As before, if $g\in \dcC^{2,1}_{\WF}(\bbR_+\times[0,T]),$ then 
\begin{gather}
g\in \cC^{1}(\bbR_+\times[0,T])\cap \cC^{2,1}((0,\infty)\times[0,T]),  \notag \\
\lim_{x\to 0^+}x\pa_x^2g(x,t)=0, \text{ and } \label{van2ndder1} \\
\lim_{x\to\infty}[|g(x,t)|+|\pa_xg(x,t)|+|\pa_tg(x,t)|+|x\pa_x^2g(x,t)|]=0,  
\ 0 \leq t \leq T. \notag
\end{gather}
\index{$\cC_{\WF}^{0,\gamma}(\bbR_+\times [0,T])$}\index{$\cC_{\WF}^{0,2+\gamma}(\bbR_+\times [0,T])$}
Next, define the $\WF$ seminorm of order $\gamma$
\[
\bbr{g}_{\WF,0,\gamma}= \supone_{(x,t)\neq (x',t')} \frac{|g(x,t)-g(x',t')|}{(|\sqrt{x}-\sqrt{x'}|+\sqrt{|t-t'|})^{\gamma}}\, ;
\]
this has a Leibniz formula,
\begin{equation}\label{leibfrm1}
\bbr{gh}_{\WF,0,\gamma}\leq \|g\|_{\infty}\bbr{h}_{\WF,0,\gamma}+\|h\|_{\infty}\bbr{g}_{\WF,0,\gamma},
\end{equation}
and provides the constant in the estimate
\begin{equation}\label{mstbscest2}
|g(x,t)-g(x',t')|\leq 2\|g\|_{\WF,0,\gamma}(|\sqrt{x}-\sqrt{x'}|+\sqrt{|t-t'|})^{\gamma}.
\end{equation}

Finally, $\cC_{\WF}^{0,\gamma}(\bbR_+\times [0,T]) \subset \dcC^0(\bbR_+ \times [0,T])$ and
$\cC_{\WF}^{0,\gamma+2}(\bbR_+\times [0,T]) \subset \dcC^{2,1}_{\WF}(\bbR_+ \times [0,T])$ are
the respective subspaces on which the norms
\begin{equation}
\begin{array}{rl}
&\|g\|_{\WF,0,\gamma}= ||g||_\infty +  \bbr{g}_{\WF,0,\gamma}, \quad \mbox{and}   \\
&\|g\|_{\WF,0,2+\gamma}=\|g\|_{\WF,0,\gamma}+\|\pa_xg\|_{\WF,0,\gamma}+\|\pa_tg\|_{\WF,0,\gamma}+\|x\pa_x^2g\|_{\WF,0,\gamma} 
\end{array}
\end{equation}
are finite.  As before, there is a characterization of elements in $\cC^{0,2+\gamma}_{\WF}(\RR_+ \times [0,T])$.
\begin{lemma}\label{lem4.2.2}
Suppose that $g\in\cC^{2,1}((0,\infty)\times [0,T])\cap\cC^1( \RR_+\times [0,T])$ satisfies~\eqref{van2ndder1}
and $\|g\|_{\WF,0,2+\gamma} < \infty$. Then $g\in\cC^{2,1}_{\WF}(\bbR_+\times[0,T]).$
\end{lemma}
\begin{proof}
The proof is essentially identical to that of Lemma~\ref{lem4.1.2}. The hypotheses imply that there is a constant 
$M$ so that
\begin{equation}
|x\pa_x^2g(x,t)|\leq Mx^{\frac{\gamma}{2}}, \qquad \mbox{when}\ x \leq 1.
 \end{equation}
This implies that $g$ is the  $\cC^{2,1}_{\WF}$-limit of 
\begin{equation}
g_n(x,t)=g\left(x+\frac 1n,t\right).
\end{equation}
\end{proof}
\begin{corollary} For $0<\gamma<1,$ the spaces $\cC_{\WF}^{0,\gamma+2}(\bbR_+\times [0,T])$ are Banach spaces.
\end{corollary}

\medskip

\subsection{Hybrid spaces:} 
For $k\in\bbN,$ we define analogues of all the spaces above which have $k$ full derivatives
in the $x$ direction. We call these hybrid since they mix ordinary with WF-regularity.

First let $\cC_{\WF}^{k,\gamma}(\bbR_+)$ be the subspace of $\dcC^{k}(\bbR_+)$  on which
\begin{equation}
  \|f\|_{\WF,k,\gamma}=\|f\|_{\cC^{k}(\bbR_+)}+\|\pa_x^kf\|_{\WF,0,\gamma} < \infty;
\end{equation}
next, $\dcC^{k,2}_{\WF}(\bbR_+)$ is the closure of $\CIc(\bbR_+)$ with respect to 
\begin{equation}
\|f\|_{\WF,k,2}=\|f\|_{\cC^{k-1}(\bbR_+)}+\|\pa_x^kf\|_{\WF,2};
\end{equation}
in terms of this, $\cC_{\WF}^{k,\gamma+2}(\bbR_+)$ is the subspace of this space on which
\begin{equation}
\|f\|_{\WF,k,2+\gamma}=\|f\|_{\cC^k(\bbR_+)}+\|\pa_x^kf\|_{\WF,0,2+\gamma} < \infty.
\end{equation}
\index{$\cC_{\WF}^{k,\gamma}(\bbR_+)$}
\index{$\cC_{\WF}^{k,2+\gamma}(\bbR_+)$}

We could equally well substitute other spaces in place of $\cC^{0,\gamma}_{\WF}$ or $\cC^{0,2+\gamma}_{\WF}$ here.
In particular, we can define $\cC^{k, \ell + \gamma}_b(\RR_+)$ to consist of all functions $f$ such that
\begin{equation}
\|f\|_{b, k, \ell + \gamma} = \sup_{j \leq k \atop i \leq \ell} ||\del_x^j (x\del_x)^i f||_{\infty} + \bbr{\del_x^k (x\del_x)^\ell}_{b,0,\gamma}.
\end{equation}
Similarly, we can define analogous parabolic versions of these hybrid spaces. For example, 
$\cC_{\WF}^{k,\gamma}(\bbR_+\times [0,T])$ and $\cC_{\WF}^{k,\gamma+2}(\bbR_+\times [0,T])$ 
are the spaces on which 
\begin{equation}
\begin{split}
\|g\|_{\WF,k,\gamma}&=\|g\|_{\cC^{k,\frac k2}(\bbR_+\times\bbR_+)}+
\sum_{2i+j= k}\|\pa_t^i\pa_x^jg\|_{\WF,0,\gamma} < \infty \\
\|g\|_{\WF,k,2+\gamma}&=\|g\|_{\cC^{k,\frac k2}(\bbR_+\times\bbR_+)}+
\sum_{2i+j= k}\|\pa_t^i\pa_x^jg\|_{\WF,0,2+\gamma} < \infty,
\end{split}
\end{equation}
respectively. These can be proved to be Banach spaces exactly exactly as for the case $k=0$. 
%Finally, the hybrid parabolic $b$-H\"older spaces are defined similarly.

\begin{remark}\label{noweights} In the one dimensional case, we use the formul{\ae}
  in~\eqref{eqn28new.0}, and~\eqref{eqn217.0.4} to deduce the higher order
  regularity of the solutions to the Cauchy and inhomogeneous problems,
  respectively, when the data has more regularity. A little thought shows that
  these formul{\ae} involve expressions of the form $x^{l}\pa_x^{l+p+q}f.$ This
  suggests that the higher order norms should include terms involving these
  weighted derivatives, i.e. terms like
  $\|x^{l}\pa_x^{l+p+q}f\|_{\WF,0,\gamma}.$ The estimates in
  Lemmas~\ref{lem10.0.1} and~\ref{lem10.0.3.3} strongly suggest that the
  desired weighted estimates are also correct.  To avoid further proliferation
  of an already very large number of cases, we have decided to omit these terms
  from our norms.

  For our applications to the analysis of generalized Kimura diffusions on
  compact manifolds with corners it suffices to assume that the data has
  support in a fixed compact set. With this assumption, the Leibniz formula
  leads to a bound on a term like $\|x^{l}\pa_x^{l+p+q}f\|_{\WF,0,\gamma}$ by a
  multiple of $\|\pa_x^{l+p+q}f\|_{\WF,0,\gamma}.$ To generalize the results in
  this monograph to the case of $P$ non-compact, it would be natural to modify
  the definitions of the higher norms spaces to include terms of this type.
\end{remark}

\medskip

\subsection{Multidimensional WF-H\"older spaces}
Following this detailed presentation of these various function spaces in one
and $1+1$ dimensions, we can follow much the same path in defining the
WF-H\"older spaces in higher dimensions. As before, we work on the model space,
either \[ S_{n,m} = \bbR_+^n\times\bbR^m,\ \mbox{or}\ \ S_{n,m} \times [0,T].
\]
We denote points in these spaces by $(\bx, \by, t)$, where $\bx = (x_1, \ldots, x_n)$
and $\by = (y_1, \ldots, y_m)$, with all $x_j \geq 0$. 

The metric on which the WF H\"older spaces are based is
\begin{equation}
  ds^2_{\WF}=\sum_{j=1}^n\frac{dx_j^2}{x_j}+\sum_{k=1}^m dy_k^2.
\end{equation}
Note that this is incomplete as any $x_j \to 0$. The Riemannian distance function is
equivalent to
\begin{equation}
d_{\WF}( (\bx, \by), (\bx', \by')) = \sum_{j=1}^n |\sqrt{x_j} - \sqrt{x_j'}| + \sum_{k=1}^m |y_k - y_k'|\, ;
\end{equation}
we sometimes write the right hand side as $\rho_s(\bx, \bx') + \rho_e(\by,\by')$. 
We also set $d_{\WF}( (\bx, \by, t), (\bx', \by', t')) = d_{\WF}((\bx, \by), (\bx', \by')) + \sqrt{|t-t'|}$. 

The function $f\in \dcC^0(S_{n,m})$ belongs to $\cC^{0,\gamma}_{\WF}(S_{n,m})$ if 
\begin{equation}
\|f\|_{\WF, 0,\gamma}=\|f\|_{\infty}+\supone_{(\bx,\by)\neq (\bx',\by')}
\frac{|f(\bx,\by)-f(\bx',\by')|}{[\rho_s(\bx,\bx')+\rho_e(\by,\by')]^{\gamma}} < \infty. 
\end{equation}
The semi-norm $\bbr{f}_{\WF,0,\gamma}$ is the second term on the right.

The space $\dcC^{2}_{\WF}(S_{n,m})$ is the closure of $\cC^{2}_c(S_{n,m})$ with respect to the norm:
\[
\|f\|_{\WF,2} =\|f\|_{\infty}+\|\nabla f\|_{\infty} + \sup_{|\alpha| + |\beta| = 2} 
\| (\sqrt{\bx}\del_{\bx} )^\alpha \del_{\by}^\beta f\|_\infty.
\]
We are introducing here the notation
\[
(\sqrt{\bx}\del_{\bx})^\alpha = (\sqrt{x_1}\del_{x_1})^{\alpha_1} \ldots (\sqrt{x_n}\del_{x_n})^{\alpha_n},
\]
and $\del_{\by}^\beta = \del_{y_1}^{\beta_1} \ldots \del_{y_m}^{\beta_m}$ as usual.  To be even more specific,
we are measuring the $L^\infty$ norms of all second derivatives
\[
\sqrt{x_i x_j} \del^2_{x_i x_j} f,\ \sqrt{x_j} \del^2_{x_j y_p}f,\ \del^2_{y_p y_q} f,\quad
i,j \leq n,\ p,q \leq m.
\]
We are also implicitly
extending any of these norms to vector-valued functions (e.g.\ $\nabla f$) in the obvious way.
A function $f\in \dcC^{2}_{\WF}(S_{n,m})$ belongs to $\cC^{,0,2+\gamma}_{\WF}(S_{n,m})$ provided 
\[
\|f\|_{\WF,0,2+\gamma}=\|f\|_{\WF,2}+ \| \nabla f \|_{\WF, 0, \gamma}+
\sum_{|\alpha| + |\beta| \leq 2} \| (\sqrt{\bx}\del_{\bx})^\alpha \del_{\by}^\beta f\|_{\WF, 0, \gamma} < \infty.
\]
% \sum_{j=1}^n\|\pa_{x_j}f\|_{\WF,0,\gamma}+
% \sum_{p=1}^m\|\pa_{y_p}f\|_{\WF,0,\gamma}+\\\sum_{q,p=1}^m\|\pa_{y_q}\pa_{y_p}f\|_{\WF,0,\gamma}+
% \sum_{j=1}^n\sum_{p=1}^m\|\sqrt{x_j}\pa_{x_j}\pa_{y_p}f\|_{\WF,0,\gamma}+
% \sum_{j,k=1}^n\|\sqrt{x_jx_k}\pa_{x_j}\pa_{x_k}f\|_{\WF,0,\gamma}
% \end{multline}

We prove once again the basic characterization lemma.
\begin{lemma}\label{4.5.2}
If $f\in\cC^1(S_{n,m}) \, \cap\, \dcC^2((0,\infty)^n\times\bbR^m)$ has 
$||f||_{\WF, 0, 2+\gamma} < \infty$ and satisfies
\begin{equation}
\begin{aligned}
\lim_{x_j \, \mathrm{or}\,  x_k\to 0^+} & \sqrt{x_jx_k}\, \pa^2_{x_j x_k}f(\bx,\by)=0 \\ & \text{  and }
 \lim_{x_j\to 0^+}\sqrt{x_j}\, \pa^2_{x_j y_p}f(\bx,\by)=0,
\end{aligned}
\end{equation}
for all $j,k\leq n,$ $p\leq m,$ and in addition, 
\begin{equation}\label{largxyhyp3}
\lim_{(\bx,\by)\to\infty} \left( |f(\bx,\by)| + |\nabla f(\bx, \by)|
+ \sum_{|\alpha| + |\beta| \leq 2} |(\sqrt{\bx}\del_{\bx})^\alpha \del_{\by}^\beta f| \right) = 0. 
% \Bigg[+\sum_{j=1}^n|\pa_{x_j}f(\bx,\by)|+
% \sum_{p=1}^m|\pa_{y_p}f(\bx,\by)|+\sum_{q,p=1}^m|\pa_{y_q}\pa_{y_p}f(\bx,\by)|+\\
% \sum_{j=1}^n\sum_{p=1}^m|\sqrt{x_j}\pa_{x_j}\pa_{y_p}f(\bx,\by)|+
% \sum_{j,k=1}^n|\sqrt{x_jx_k}\pa_{x_j}\pa_{x_k}f(\bx,\by)|\Bigg]=0.
\end{equation}
Then $f\in\cC^{2}_{\WF}(S_{n,m}).$ 
\end{lemma}
\begin{proof}  The hypotheses imply that 
 \begin{equation}\label{eqn82.0.3}
\begin{split}
&\sqrt{x_jx_k} \, |\pa^2_{x_j x_k}f(\bx,\by)| \leq  \|f\|_{\WF,0,2+\gamma}\, \min\{x_j^{\frac{\gamma}{2}},x_k^{\frac{\gamma}{2}}\}\\
&\sqrt{x_j}\, |\pa^2_{x_j y_p}f(\bx,\by)|\leq \|f\|_{\WF,0,2+\gamma} \, x_j^{\frac{\gamma}{2}},
\end{split}
\end{equation}
and in addition that each scaled second derivative has a continuous extension
to a certain part of the boundary.  For example, $\sqrt{x_j}\, \pa^2_{x_j
  y_p}f$ extends continuously to that subset of the boundary of $S_{n,m}$ where
$\{x_j>0\}.$

Let $\bone = (1, \dots, 1)$ and choose any sequence of positive numbers $\eta_i \to 0$. Then define 
\begin{equation}
f_i(\bx,\by)= f(\bx+ \eta_i \bone,\by). 
\end{equation}
The definition $||\cdot ||_{\WF, 0, 2+\gamma}$ and~\eqref{largxyhyp3} imply that
\begin{equation}
\lim_{i\to\infty}\left( \|f-f_i\|_{\infty}+\|\nabla_{\bx,\by}(f-f_i)\|_{\infty}+
\sup_{|\beta| =2}\|\pa_{\by}^\beta(f-f_i)\|_{\infty}\right)=0.
\end{equation}
Hence it remains to study the terms $|| (\sqrt{\bx}\del_{\bx})^\alpha \del_{\by}^\beta (f-f_i)||_{\infty}$
for $|\alpha| + |\beta| = 2$ and $\alpha \neq (0, \ldots, 0)$. 
% \begin{equation}
%   \sup_{1\leq  j\leq n}\sup_{1\leq p\leq
%   m}\|\sqrt{x_j}\pa_{x_j}\pa_{y_p}[f-f_i]\|_{L^{\infty}}+
% \sup_{1\leq  j,k\leq n}\|\sqrt{x_jx_k}\pa_{x_j}\pa_{x_k}[f-f_i]\|_{L^{\infty}}.
% \end{equation}

We begin with $\sqrt{x_j}\pa^2_{x_j y_p}(f-f_i).$ For $\delta>0$, define 
$$
W_{j,\delta}=\{(\bx,\by):\: \delta\leq x_j\}\subset S_{n,m}. 
$$
From the hypotheses again, it is clear that if $\delta>0$, then 
\begin{equation}\label{eqn86.0.3}
\lim_{i\to\infty}\|\sqrt{x_j}\, \pa^2_{x_j y_p}( f-f_i) \|_{\infty, W_{j,\delta}}=0,
\end{equation}
so we must only show that $|\sqrt{x_j}\, \pa^2_{x_j y_p}(f-f_i)(\bx,\by)|$ is
uniformly small when $i$ is large and $x_j$ is small. We have 
\begin{multline*}
|\sqrt{x_j}\, \pa^2_{x_j y_p}(f-f_i)(\bx,\by)| \\
\leq |\sqrt{x_j}\,\pa^2_{x_j y_p}f(\bx,\by)-(\sqrt{x_j+\eta_i})\pa^2_{x_j y_p}f(\bx+\eta_i,\by)|\\
+ \frac{|\sqrt{x_j}-\sqrt{x_j+\eta_i}|}{\sqrt{x_j+\eta_i}} |(\sqrt{x_j+\eta_i})\pa^2_{x_j y_p}f(\bx+\eta_i \bone,\by)|.
\end{multline*}
By definition of the $\cC^{0,2+\gamma}_{\WF}$-norm again, and using~\eqref{eqn82.0.3}, this gives:
\begin{equation}
|\sqrt{x_j}\, \pa^2_{x_j y_p}(f-f_i)(\bx,\by)|\leq 
\|f\|_{\WF,0,2+\gamma}\left[n^\gamma \eta_i^{\gamma/2}+(x_j+\eta_i)^{\frac{\gamma}{2}}\right].
\end{equation}
Together with~\eqref{eqn86.0.3}, this implies that
\begin{equation}
\lim_{i\to\infty}\|\sqrt{x_j}\,\pa_{x_j y_p}(f-f_i)\|_{\infty}=0.
\end{equation}

Finally, we must consider terms of the form $\sqrt{x_jx_k}\, |\pa^2_{x_j x_k}( f-f_i)|.$  Once 
again, for any $\delta>0,$  
\begin{equation}
\lim_{i\to\infty}\|\sqrt{x_jx_k}\, \pa^2_{x_j x_k}(f-f_i)\|_{\infty, W_{j,\delta} \cap W_{k,\delta}}=0
\end{equation}
Near the boundary, we have 
\begin{multline}
| \sqrt{x_jx_k}\, \pa^2_{x_j x_k} (f-f_i)(\bx,\by)| \leq\\
| \sqrt{x_jx_k}\, \pa^2_{x_j x_k}f(\bx,\by)-
\sqrt{(x_j+\eta_i)(x_k+\eta_i)}\, \pa^2_{x_j x_k}f_i(\bx+\eta_i \bone ,\by)|+\\
\frac{|\sqrt{(x_j+\eta_i)(x_k+\eta_i)}-\sqrt{x_jx_k}|}
{\sqrt{(x_j+\eta_i)(x_k+\eta_i)}}\, 
\sqrt{(x_j+\eta_i)(x_k+\eta_i)}|\pa^2_{x_j x_k}f_i(\bx+\eta_i \bone,\by)|,
\end{multline}
whence, by~\eqref{eqn82.0.3}, 
\begin{multline}
|\sqrt{x_jx_k}\,\pa^2_{x_j x_k}(f-f_i)(\bx,\by)|\leq\\ 
\|f\|_{\WF,0,2+\gamma}\left[  n^\gamma \eta_i^{\gamma/2}
\min\left\{\left|x_j+\eta_i\right|^{\frac{\gamma}{2}},
\left|x_k+\eta_i\right|^{\frac{\gamma}{2}}\right\}\right].
\end{multline}
This implies that 
\begin{equation}
\lim_{i\to\infty}\|\sqrt{x_jx_k}\, \pa^2_{x_j x_k}(f-f_i)\|_{\infty}=0,
\end{equation}
and proves the lemma.
\end{proof}

A function $f\in\dcC^k(S_{n,m})$ belongs to $\cC^{k,\gamma}_{\WF}(S_{n,m})$ if 
\index{$\cC^{k,\gamma}_{\WF}$}
\begin{multline}
\|f\|_{\WF,k,\gamma}=\|f\|_{\cC^k}+\\
\sup_{\{|\balpha|+|\bbeta|=k\}} \supone_{(\bx,\by)\neq (\bx',\by')}
\frac{|\pa_{\bx}^{\balpha}\pa_{\by}^{\bbeta}f(\bx,\by)-\pa_{\bx}^{\balpha}\pa_{\by}^{\bbeta}f(\bx',\by')|}{[\rho_s(\bx,\bx')+\rho_e(\by,\by')]^{\gamma}}
\end{multline}
is finite.   Similarly, $\dcC^{k,2}_{\WF}(S_{n,m})$ is the closure of $\cC^{k+2}_c(S_{n,m})$ with respect to the norm
\begin{equation}
  \|f\|_{\WF,k,2}=\|f\|_{\cC^{k-1}}+\sup_{\{|\balpha|+|\bBeta|=k\}} \|\pa_{\bx}^{\balpha}\pa_{\by}^{\bBeta}f\|_{\WF,2},
\end{equation}
and a function $f\in \dcC^{k,2}_{\WF}(S_{n,m})$ belongs to $\cC^{k,2+\gamma}_{\WF}(S_{n,m})$ if 
\begin{equation}
  \|f\|_{\WF,k,\gamma}=\|f\|_{\cC^k}+\sup_{\{|\balpha|+|\bBeta|=k\}} \|\pa_{\bx}^{\balpha}\pa_{\by}^{\bBeta}f\|_{\WF,0,2+\gamma} < \infty.
\end{equation}
The analogue of Lemma~\ref{4.5.2} is straightforward and shows that these are Banach spaces.

The parabolic H\"older spaces are defined similarly.  A  function $g\in \dcC^0(S_{n,m} \times [0,T])$ belongs to
$\cC^{0,\gamma}_{\WF}(S_{n,m}\times [0,T])$ provided 
\begin{equation}\label{eqn137.3}
\begin{aligned}
\|g\|_{\WF,0,\gamma}= & \|g\|_{\infty} + \\ & \supone_{(\bx,\by,t)\neq (\bx',\by' ,t' )}
\frac{|g(\bx,\by,t)-g(\bx',\by',t')|}{[\rho_s(\bx,\bx')+\rho_e(\by,\by') +\sqrt{|t-t'|}]^{\gamma}} < \infty.
\end{aligned} 
\end{equation} 
The semi-norm $\bbr{g}_{\WF, 0,\gamma}$ is the second term on the right. When
it is important to emphasize the maximum time $T,$ we use the notation
$\bbr{g}_{\WF, 0,\gamma,T}$ for this semi-norm.
\index{$\bbr{g}_{\WF, 0,\gamma}$}\index{$\bbr{g}_{\WF, 0,\gamma,T}$} A function $g\in\dcC^k(S_{n,m}\times [0,T])$ 
belongs to  $\cC^{k,\gamma}_{\WF}(S_{n,m}\times [0,T])$ if 
\begin{equation}
\begin{aligned}
& \|g\|_{\WF,k,\gamma}=\|g\|_{\cC^k} + \\
& \sup_{\{|\balpha|+|\bbeta|+2j=k\}}\supone_{(\bx,\by,t) \atop \neq (\bx',\by',t')}
\frac{|\pa_t^j\pa_{\bx}^{\balpha}\pa_{\by}^{\bbeta}g(\bx,\by,t)-\pa_t^j\pa_{\bx}^{\balpha}
\pa_{\by}^{\bbeta}g(\bx',\by',t')|}{[\rho_s(\bx,\bx')+\rho_e(\by,\by') +\sqrt{|t-t'|}]^{\gamma}} < \infty.
\end{aligned}
\end{equation}

We now list several basic estimates and facts. First,  for functions $f \in \cC^{0,\gamma}_{\WF}(S_{n,m})$ 
and $g \in \cC^{0,\gamma}_{\WF}(S_{n,m} \times [0,T])$, we have
\begin{equation}
\begin{split}
&|f(\bx,\by)-f(\bx',\by')|\leq 2\|f\|_{\WF,0,\gamma} d_{\WF}( (\bx, \by), (\bx', \by') )^{\gamma} \\
&|g(\bx,\by,t)-g(\bx',\by',t')|\leq 2\|g\|_{\WF,0,\gamma} d_{\WF}( (\bx, \by, t), (\bx', \by', t') )^\gamma
\end{split}
\label{eqmd}
\end{equation}
Furthermore, there are Leibniz formul\ae\ for these semi-norms: if $f,g\in\cC^{0,\gamma}_{\WF}(S_{n,m}),$ or
$f,g\in\cC^{0,\gamma}_{\WF}(S_{n,m} \times [0,T]),$ then 
\begin{equation}\label{leibfrmnm}
 \bbr{fg}_{\WF,0,\gamma}\leq \bbr{f}_{\WF,0,\gamma}\|g\|_{L^{\infty}}+ \bbr{g}_{\WF,0,\gamma}\|f\|_{L^{\infty}}
\end{equation}
% The following lemma, relating the semi-norms $\bbr{f}_{\WF,0,\gamma},$
% ($\bbr{g}_{\WF,0,\gamma}$) for different values of $\gamma,$ is also useful in the
% study of the behavior of solutions as $x\to\infty.$
\begin{lemma}\label{2gamslem} Let $0<\gamma'<\gamma<1$ and suppose that
$f \in \cC^{0,\gamma}_{\WF}(S_{n,m})$ and $g \in \cC^{0,\gamma}_{\WF}(S_{n,m} \times [0,T])$. Then 
\begin{equation}
\bbr{f}_{\WF,0,\gamma'}\leq  2 \bbr{f}_{\WF,0,\gamma}^{\frac{\gamma'}{\gamma}} \|f\|_{\infty}^{1-\frac{\gamma'}{\gamma}},
\label{leib1}
\end{equation}
\begin{equation}
\bbr{g}_{\WF,0,\gamma'}\leq 2 \bbr{g}_{\WF,0,\gamma}^{\frac{\gamma'}{\gamma}}\|g\|_{\infty}^{1-\frac{\gamma'}{\gamma}}.
\end{equation}
\end{lemma}
\begin{proof}
These follow directly from the identity
\begin{multline}
  \frac{|h(\bx,\by,t)-h(\bx',\by',t')|}{d_{\WF}((\bx,\by,t),(\bx',\by',t'))^{\gamma'}}
=\\
\left[\frac{|h(\bx,\by,t)-h(\bx',\by',t')|}{d_{\WF}((\bx,\by,t),(\bx',\by',t'))^{\gamma}}
\right]^{\frac{\gamma}{\gamma'}}
\left(|h(\bx,\by,t)-h(\bx',\by',t')|\right)^{1-\frac{\gamma'}{\gamma}}
\end{multline}
where $h$ is defined on $S_{n,m}\times[0,T],$ or the analogous identity for
functions defined on $S_{n,m}.$
\end{proof}
% We prove only the second inequality, since the first is essentially identical. Observe that 
% \begin{multline}
% \frac{|g(\bx,\by,t)-g(\bx',\by',t')|} {[\rho( (\bx,\by), (\bx',\by'))+\sqrt{|t-t'|}]^{\gamma'}}=\\
% \left[\frac{|g(\bx,\by,t)-g(\bx',\by',t')|} {[\rho( (\bx,\by), (\bx',\by') )+\sqrt{|t-t'|}]^{\gamma}}\right]^{\frac{\gamma'}{\gamma}}
% \left[|g(\bx,\by,t)-g(\bx',\by',t')|\right]^{1-\frac{\gamma'}{\gamma}},
% \end{multline}
% and the inequality follows directly from this. 

The space $\dcC^{2,1}_{\WF}(S_{n,m} \times [0,T])$ is the closure of $\cC^{2,1}_c(S_{n,m} \times [0,T])$ with
respect to the norm
\begin{multline}
\|g\|_{\WF,2,1}=\|g\|_{\infty}+\|\nabla_{\bx,\by,t}g\|_{\infty}+ 
\sup_{|\balpha| + |\bBeta| = 2}\|(\sqrt{\bx}\del_{\bx})^{\balpha} \del_{\by}^{\bBeta} g\|_{\infty},
\end{multline}
and $\cC^{0,2+\gamma}_{\WF}(S_{n,m} \times [0,T])$ is the subspace on which 
\begin{equation}\label{eqn146.3}
\begin{aligned}
& \|g\|_{\WF,0,2+\gamma} =  \|g\|_{\WF,2,1} \\ & \quad +\|\nabla_{\bx,\by,t}g\|_{\WF,0,\gamma}+
\sup_{|\balpha| + |\bBeta| = 2}\|(\sqrt{\bx}\del_{\bx})^{\balpha} \del_{\by}^{\bBeta} g\|_{\WF,0,\gamma} < \infty.
\end{aligned}
\end{equation}

The basic lemma now reads:
\begin{lemma}\label{lem4.9new0.0}
Let $g\in\cC^1(S_{n,m} \times   [0,T])\cap\cC^{2,1}_{\WF}((0,\infty)^n\times \bbR^m\times [0,T])$ satisfy
\[
\lim_{x_j \ \mathrm{or}\   x_k\to 0^+} \sqrt{x_jx_k}\, \pa^2_{x_j x_k}g(\bx,\by,t)=0\ \ \mbox{and}
\lim_{x_j\to 0^+}\sqrt{x_j}\, \pa^2_{x_j y_p}g(\bx,\by,t)=0
\]
for $j,k\leq n,$ $p\leq m,$  and 
\[
\lim_{(\bx,\by)\to\infty}  \left[|g(\bx,\by,t)|+|\nabla_{\bx,\by,t}g(\bx,\by,t)|+
\sup_{|\balpha| + |\bBeta| = 2 }|(\sqrt{\bx}\del_{\bx})^{\balpha} \del_{\by}^{\bBeta} g(\bx,\by,t)|\right] = 0.
\]
If $\|g\|_{\WF, 0, 2+\gamma} < \infty$, then $g\in\cC^{0,2+\gamma}_{\WF}(S_{n,m}\times [0,T]).$
\end{lemma}
\noindent
The proof is nearly identical to the one for Lemma~\ref{4.5.2}, and this implies as before that
$\cC^{0,2+\gamma}_{\WF}(S_{n,m}\times [0,T])$ is a Banach space.

We finally define the higher parabolic H\"older spaces in the expected way. Namely, 
$\dcC^{k+2,\frac k2+1}_{\WF}(S_{n,m}\times [0,T])$ is the closure of $\cC^\infty_c(S_{n,m}\times [0,T])$
with respect to the norm 
\begin{multline}
||g||_{\WF, k+2, k/2+1} =\|g\|_{\cC^{k,\frac k2}}+ 
\sup_{|\balpha| + |\bBeta| + 2j = k} 
||\del_{\bx}^{\balpha} \del_{\by}^{\bBeta} \del_t^{j} g||_{\WF,2,1}.
\end{multline}
We define $\cC^{k,2+\gamma}_{\WF}(S_{n,m}\times [0,T])$ to be the subspace of
$\dcC^{k+2,\frac k2+1}_{\WF}(S_{n,m}\times [0,T])$ on which
\[
||g||_{\WF, k, 2+\gamma} = ||g||_{\cC^{k,\frac k2}} + \sup_{|\balpha| +
  |\bBeta| + 2j = k} \bbr{\del_{\bx}^{\balpha} \del_{\by}^{\bBeta} \del_t^{j}
  g}_{\WF,0,2+\gamma} < \infty.
\]
As before, if the upper limit $T,$ for the time variable, is important we
sometimes denote these norms by $||g||_{\WF, k, \gamma,T},$ and $||g||_{\WF, k,
  2+\gamma,T},$ respectively.

These various spaces satisfy some obvious inclusions: if $k'> k,$ or $k'=k$ and
$1 \geq \gamma' >  \gamma > 0,$ then
\begin{equation}\label{bscmpincl}
\cC^{k',\gamma'}_{\WF}(S_{n,m})\subset \cC^{k,\gamma}_{\WF}(S_{n,m}), \ \ \cC^{k',2+\gamma'}_{\WF}(S_{n,m})\subset \cC^{k,2+\gamma}_{\WF}(S_{n,m})
\end{equation}
and 
\begin{equation}\label{htcmpincl}
\begin{split}
&\cC^{k',\gamma'}_{\WF}(S_{n,m} \times [0,T])\subset \cC^{k,\gamma}_{\WF}( S_{n,m} \times [0,T])\\
&\cC^{k',2+\gamma'}_{\WF}( S_{n,m} \times [0,T])\subset \cC^{k,2+\gamma}_{\WF}(S_{n,m}\times [0,T]).
\end{split}
\end{equation}

\begin{proposition}\label{prop4.1new} If $k'<k$ or $k'=k$ and $\gamma'<\gamma,$ then the
restrictions of the inclusions in~\eqref{bscmpincl} and~\eqref{htcmpincl} to  subspaces
of functions which are supported in a ball of finite radius in $S_{n,m}$ or $S_{n,m} \times [0,T]$ 
are compact.
\end{proposition}
\begin{proof}
These facts can all be deduced in a fairly straightforward manner from the Arzela-Ascoli theorem.  
We illustrate this by considering the inclusion
\[
\{u \in \cC^{0,\gamma'}_{\WF}(S_{n,m}): u = 0\ \mbox{for}\ |(\bx,\by)| > R\} \hookrightarrow
\cC^{0,\gamma}_{\WF}(S_{n,m}).
\]
If $\{u_j\}$ is a sequence in the space on the left with uniformly bounded norm, then 
by \eqref{eqmd}, this sequence is uniformly bounded and equicontinuous, hence some
subsequence converges in $\calC^0$ to a limit function $u$. Now apply \eqref{leib1}
to see that this subsequence is Cauchy in $\cC^{0,\gamma'}_{\WF}$.
\end{proof}

As described in remark~\ref{noweights} in the 1-dimensional case, the
higher order estimates in the general case are deduced by using
formul{\ae}~\eqref{eqn47new.03} and~\eqref{eqn50new.03}. Again this
suggests that the higher order norms should include weighted
derivatives. As noted above, for our applications to Kimura operators
on compact manifolds with corners we only need these results for data
with fixed bounded support.  To somewhat shorten this already long
text, we have omitted these terms from the definitions of the higher
order norms. Using the Leibniz formula we easily deduce the following
estimates:
\begin{proposition}\label{wtedest_pr} Fix an $R>0,$ a $k\in\bbN,$ a
  non-negative vector $\bzero\leq\bb,$ and a $0<\gamma<1.$ There is a
  constant $C_R$ so that
\begin{enumerate}
 \item If $f\in\cC^{k,\gamma}_{\WF}(S_{n,m})$ has support in the set
  $\{(\bx;\by):\: \|\bx\|\leq R\},$ then if $2q+|\alpha|+|\beta|\leq
  k,$ we have the estimate
  \begin{equation}
    \|L_{\bb,m}^q\pa_{\bx}^{\balpha}\pa_{\by}^{\bbeta}f\|_{\WF,0,\gamma}\leq C_R\|f\|_{\WF,k,\gamma}.
  \end{equation}
\item If $f\in\cC^{k,2+\gamma}_{\WF}(S_{n,m})$ has support in the set
  $\{(\bx;\by):\: \|\bx\|\leq R\},$ then if $2q+|\alpha|+|\beta|\leq
  k,$ we have the estimate
  \begin{equation}
    \|L_{\bb,m}^q\pa_{\bx}^{\balpha}\pa_{\by}^{\bbeta}f\|_{\WF,0,2+\gamma}\leq C_R\|f\|_{\WF,k,2+\gamma}.
  \end{equation}
\item If $g\in\cC^{k,\gamma}_{\WF}(S_{n,m}\times [0,T])$ has support in the set
  $\{(\bx;\by,t):\: \|\bx\|\leq R\},$ then if $2p+2q+|\alpha|+|\beta|\leq
  k,$ we have the estimate
  \begin{equation}
    \|\pa_{t}^pL_{\bb,m}^q\pa_{\bx}^{\balpha}\pa_{\by}^{\bbeta}g\|_{\WF,0,\gamma,T}
\leq C_R\|g\|_{\WF,k,\gamma,T}.
  \end{equation}
\item If $g\in\cC^{k,2+\gamma}_{\WF}(S_{n,m}\times [0,T])$ has support in the set
  $\{(\bx;\by,t):\: \|\bx\|\leq R\},$ then if $2p+2q+|\alpha|+|\beta|\leq
  k,$ we have the estimate
  \begin{equation}
    \|\pa_{t}^pL_{\bb,m}^q\pa_{\bx}^{\balpha}\pa_{\by}^{\bbeta}g\|_{\WF,0,2+\gamma,T}
\leq C_R\|g\|_{\WF,k,2+\gamma,T}.
  \end{equation}
\end{enumerate}
\end{proposition}

\chapter[H\"older estimates in 1 dimension]{H\"older estimates for the $1$-dimensional model problem}
\label{chap.1ddegen_ests}

In this and the following three chapters we establish  H\"older
estimates for the solutions of the model problems, i.e. $w$ such that
\begin{equation}
  (\pa_t-L_{\bb,m})w=g\text{ with } w(p,0)=f(p),
\end{equation}
where $f$ and $g$ belong to the anisotropic H\"older spaces introduced in
Chapter~\ref{chap.holdspces}. It may appear that we are taking a circuitous
path, by first considering the 1-dimensional case, then pure corner models,
$\bbR_+^n$, followed by Euclidean models ($ \bbR^m,$) before finally treating the general
case, $\bbR_+^n\times \bbR^m.$ In fact, all cases need to be treated, and in
the end nothing is really wasted. We give a detailed treatment of the
1-dimensional  case, both because it establishes a pattern that will be
followed in the subsequent cases, and because all of the higher dimensional
estimates are reduced to estimates on heat kernels for the 1-dimensional
model problems.

The derivation of parabolic Schauder estimates is now an old subject, and there
are many possible approaches to follow.  Our proof of these estimates for the
model operator $\del_t - L_b$ is elementary. It uses the explicit formula for
the heat kernel,~\eqref{eqn1.22.05}, along with standard tools of analysis,
like Taylor's formula and Laplace's method.  The paper \cite{DaskHam} considers
a similar degenerate diffusion operator in $2+1$-dimensions, and contains
proofs of parabolic Schauder estimates for that problem. We present different
arguments to derive the analogous estimates here.  This allows us to handle the
case $b=0$, which is somewhat different than the situation in \cite{DaskHam}.

It is straightforward from the definitions that for any $k\in\bbN_0$ and $0<\gamma<1,$
\[
\pa_t-L_b: \{u \in \cC^{k,2+\gamma}_{\WF}(\bbR_+\times [0,T]): u(x,0) = 0\} 
\longrightarrow \cC^{k,\gamma}_{\WF}(\bbR_+\times [0,T]).
\]
Our goal is to prove the converse, and of course also to study the regularity effects of nontrivial
initial data.  We shall prove the following two results:

\begin{proposition}\label{prop2} Fix $k \in \bbN_0$, $\gamma \in (0,1),$ $0<R$ and $b\geq 0.$
  Suppose that $f\in\cC^{k,\gamma}_{\WF}(\bbR_+),$ and let $v$ be the
  unique solution to~\eqref{1dmdlb}, with $g=0.$ If $k>0,$ then also
  assume that $f$ has support in $[0,R].$ Then
  $v\in\cC^{k,\gamma}_{\WF}(\bbR_+\times [0,T])$ for any $T > 0$ and
  there a constant $C_{k,\gamma,b,R}$ so that
\begin{equation}\label{bscest0}
\|v\|_{\WF,k,\gamma}\leq C_{k,\gamma,b,R}\|f\|_{\WF,k,\gamma}.
\end{equation}
If $0<\gamma'<\gamma,$ then
\begin{equation}\label{bscest00}
\lim_{t\to 0^+}\|v(\cdot,t)-f\|_{\WF,k,\gamma'}=0.
\end{equation}
If $f\in\cC^{k,2+\gamma}_{\WF}(\bbR_+),$ then
\begin{equation}\label{bscest2}
\|v\|_{\WF,k,2+\gamma}\leq C_{k,\gamma,b,R}\|f\|_{\WF,k,2+\gamma}.
\end{equation}
The constants $C_{k,\gamma,b,R}$ are uniformly bounded on any finite interval
$0\leq b \leq B.$ 
If $0<\gamma'<\gamma,$ then
\begin{equation}\label{bscest20}
\lim_{t\to 0^+}\|v(\cdot,t)-f\|_{\WF,k,2+\gamma'}=0.
\end{equation}
If $k=0,$ then the constants in these estimates do not depend on $R.$
\end{proposition}

\begin{proposition}\label{prop1} Fix $k \in \bbN_0$, $\gamma \in (0,1),$ $0<R,$ and $b\geq 0.$ 
  Let $u$ be the unique solution to~\eqref{1dmdlb}, with $f=0$ and
  $g\in\cC^{k,\gamma}_{\WF} (\bbR_+\times [0,T])$. If $k>0$ we assume
  that $g(x,t)$ is supported in $\{(x,t):x\leq R\}.$ Then
  $u\in\cC^{k,2+\gamma}_{\WF}(\bbR_+\times[0,T])$ and there is a
  constant $C_{k,\gamma,b,R}$ so that
\begin{equation}\label{bscest3}
\|u\|_{\WF,k,2+\gamma,T}\leq C_{k,\gamma,b,R}(1+T)\|g\|_{\WF,k,\gamma,T}.
\end{equation}
The constants $C_{k,\gamma,b,R}$ are uniformly bounded for $0\leq b \leq B.$  For
any $0<\gamma'<\gamma,$ the solution $u(\cdot,t)$ tends to zero in $\cC^{k,2+\gamma'}_{\WF}(\bbR_+).$ If $k=0,$ then the constant is independent of $R.$
\end{proposition}
The assertions about the behavior of solutions as $t\to 0^+$ follow easily from
Proposition~\ref{htcmpincl}, the following lemma, and the obvious facts that
$v(\cdot,t)$ tends to $f$ and $u(\cdot,t)$ tends to zero in $\cC^0(\bbR_+).$
\begin{lemma}\label{lem8.0.6.nu} Let $X_2\subset X_1\subset X_0$ be Banach spaces with the first
  inclusion precompact, and the second bounded. If for some $M,$ the
  family $v(t)\in X_2$ satisfies:
  \begin{equation}
    \sup_{t\in [0,1]}\|v(t)\|_{X_2}\leq M\text{ and }\lim_{t\to 0^+}\|v(t)\|_{X_0}=0,
  \end{equation}
then
\begin{equation}
  \lim_{t\to 0^+}\|v(t)\|_{X_1}=0.
\end{equation}
\end{lemma}
\begin{proof}  If $\lim_{t\to 0^+}\|v(t)\|_{X_1}\neq 0,$ then, by compactness,  we
  can choose a sequence $<t_n>,$ tending to zero so that $<v(t_n)>$ converges,
  in $X_1,$ to $v^*\neq 0.$ The boundedness of the inclusion $X_1\subset X_0,$
  implies that $<v(t_n)>$ must also converge, in $X_0,$  to $v^*,$ but then
  $v^*$ must equal $0.$
\end{proof}

Our final results concern the resolvent operator defined, for $\mu\in(0,\infty),$  by
\begin{equation}\label{eqn6.30.001}
  R(\mu)f=\lim_{\epsilon\to 0^+}
\int\limits_{\epsilon}^{\frac{1}{\epsilon}}e^{-\mu t}v(x,t)dt.
\end{equation}
As noted in Proposition~\ref{prop6.3.2.01}, $R(\mu) f$ extends to define an
analytic function for $\mu\in \bbC\setminus (-\infty,0].$ Our final proposition
gives  a more refined statement of the mapping properties of $R(\mu)$ for the
1-dimensional model problem:
\begin{proposition}\label{prop8.0.4.00}  The resolvent operator $R(\mu)$ is analytic
  in the complement of $(-\infty,0],$ and is given by the integral
  in~\eqref{eqn6.30.00} provided that $\Re(\mu e^{i\theta})>0.$ For $\alpha\in
  (0,\pi],$ there are constants $C_{b,\alpha}$
 so that if 
 \begin{equation}\label{argest001}
   \alpha-\pi\leq\arg{\mu}\leq\pi-\alpha, \end{equation}
then for $f\in\cC^0_b(\bbR_+)$ we have:
  \begin{equation}
    \|R(\mu) f\|_{L^{\infty}}\leq \frac{C_{b,\alpha}}{|\mu|}\|f\|_{L^{\infty}};
  \end{equation}
  with $C_{b,\pi}=1.$ Moreover, for $0<\gamma<1,$ there is a constant
  $C_{b,\alpha,\gamma}$ so that if $f\in\cC^{0,\gamma}_{\WF}(\bbR_+),$ then
  \begin{equation}\label{eqn8.13.00}
      \|R(\mu) f\|_{\WF,0,\gamma}\leq \frac{C_{b,\alpha,\gamma}}{|\mu|}\|f\|_{\WF,0,\gamma}.
  \end{equation}

If for a $k\in\bbN_0,$ and $0<\gamma<1,$ $f\in\cC^{k,\gamma}_{\WF}(\bbR_+),$ then
$R(\mu)f\in\cC^{k,2+\gamma}_{\WF}(\bbR_+),$ and, we have
\begin{equation}
  (\mu-L_b)R(\mu) f=f.
\end{equation}
If $f\in \cC^{0,2+\gamma}_{\WF}(\bbR_+),$ then
\begin{equation}
  R(\mu)(\mu-L_b) f=f.
\end{equation}
There are constants $C_{b,k,\alpha}$ so that, for $\mu$
satisfying~\eqref{argest001}, we have
\begin{equation}
  \|R(\mu)f\|_{\WF,k,2+\gamma}\leq 
C_{b,k,\alpha}\left[1+\frac{1}{|\mu|}\right]\|f\|_{\WF,k,\gamma}.
\end{equation}
For any $B>0,$ these constants are uniformly bounded for $0\leq b\leq B.$
\end{proposition}
\begin{remark} Unlike the results for the heat equations, the higher
  order estimates for the resolvent do not require an assumption about
  the support of the data. This is because the estimates for this
  operator only involve spatial derivatives; it is the time
  derivatives that lead to the $x^j$-weights.
\end{remark}

\section[Kernel estimates]{Kernel Estimates for  Degenerate Model Problems}\label{1ddegmods}
The proofs of the estimates in one and higher dimensions rely upon estimates
for the kernel functions $k^b_t(x,y)$ and their derivatives. These kernels are
analytic in the right half plane $\Re t>0,$ and many of these estimates are
stated and proved for this analytic continuation.  Since we often need to refer
to these results, we first list these estimates as a series of
lemmas. Most of the proofs are given in Appendix~\ref{prfsoflems}.

\emph{Throughout this book} we let $C,$ $C_{b}$ or $C_{b,*}$ where $*$
are other parameters denote positive constants that are uniformly bounded for
$0\leq b\leq B,$ and a fixed value of $\gamma.$ We often make use of the
following elementary inequalities.
\begin{lemma}\label{lem1} For each $k\in\bbN,$ and $0<\gamma<1,$ there is a constant
  $C_{k,\gamma}$ such that for non-negative numbers $\{a_1,\dots,a_k\}$ we have
  \begin{equation}\label{bscest}
    C_{k,\gamma}^{-1}\sum_{j=1}^ka_j^{\gamma}<
\left[\sum_{j=1}^ka_j\right]^{\gamma}<C_{k,\gamma}\sum_{j=1}^ka_j^{\gamma}
  \end{equation}
\end{lemma}
\begin{proof} As everything is homogeneous of degree $\gamma$ it suffices to
  consider non-negative $k$-tuples, $(a_1,\dots,a_k),$ with
  \begin{equation}
    \sum_{j=1}^ka_j=1,
  \end{equation}
for which the statement is obvious.
\end{proof}
\begin{lemma}\label{lem2} For $0<\gamma<1,$
  there is a constant $m_{\gamma}$ so that,  if $x$ and $y$ are non-negative, then
\begin{equation}\label{bscest201}
  |x^{\gamma}-y^{\gamma}|\leq m_{\gamma}|x-y|^{\gamma}.
\end{equation}
\end{lemma}
\begin{proof} We can assume that $x>y,$ and therefore the inequality is
  equivalent to the assertion that, for $1<x,$ we have
\begin{equation}\label{bscest21}
  |x^{\gamma}-1|\leq m_{\gamma}|x-1|^{\gamma}.
\end{equation}
The existence of $m_{\gamma}$ follows easily from the observation that
\begin{equation}
  x^{\gamma}-1=\gamma(x-1)+O((x-1)^2).
\end{equation}
\end{proof}

The remaining lemmas are divided according to the order of the derivative being
estimated. Proofs are given in Appendix~\ref{prfsoflems}. The reader can skip
the rest of this subsection and refer to it later, as needed. Recall that, for
$0<b,$ 
\begin{equation}
  k^b_t(x,y)=\frac{y^{b-1}}{t^b}e^{-\frac{(x+y)}{t}}\psi_b\left(\frac{xy}{t^2}\right),
\end{equation}
where
\begin{equation}
  \psi_b(z)=\sum_{j=0}^{\infty}\frac{z^j}{j!\Gamma(j+b)}.
\end{equation}
This heat kernel is a smooth function in $[0,\infty)_x\times (0,\infty)_y\times
(0,\infty)_t,$ which has an analytic extension in the $t$ variable to the right
half plane $S_{0},$ where the sectors $S_{\phi}$ are defined
in~\eqref{sectordef.0}. The asymptotic expansion~\eqref{eqn6.25.00} is valid in
any sector $S_{\phi},$ with $\phi>0.$

\subsection{Basic Kernel Estimates}
Recall that as $b\to 0^+,$ the kernels $k^b_t$ converge, in the sense of
distributions, to
\begin{equation}
  k^0_t(x,y)=k^{0,D}_t(x,y)+e^{-\frac{x}{t}}\delta_0(y),
\end{equation}
where
\begin{equation}
  k^{0,D}_t(x,y)=\left(\frac{x}{t^2}\right)e^{-\frac{x+y}{t}}\psi_2\left(\frac{x
      y}{t^2}\right),
\end{equation}
is the solution operator for the equation $\pa_tv=x\pa_x^2v$ with $v(0,t)=0.$
As we will see, the solutions to the equations $\pa_tu-L_bu=g,$ and their
higher dimensional analogues satisfy H\"older estimates with constants
uniformly bounded as $b\to 0^+.$ The kernel estimates are therefore proved for
$0<b,$ and the properties of solutions to the PDE with $b=0$ are obtained by
taking limits of solutions.

A trivial but crucial fact is the following:
\begin{lemma}\label{lem9.1.3.00}
For $t\in S_0$  and $b>0$ we have:
\begin{equation}\label{intkbtis1}
  \int\limits_0^{\infty}k^b_t(x,y)dy=1.
\end{equation}
There is a constant $C_{\phi}$ so that, for $t\in S_{\phi}$
\begin{equation}\label{intkbtis12}
  \int\limits_0^{\infty}|k^b_t(x,y)|dy\leq C_{\phi}.
\end{equation}
\end{lemma}
\begin{proof}
The integral is absolutely convergent for any $t\in S_0,$ and clearly defines an
analytic function of $t.$ For $t\in (0,\infty),$ the integral equals 1, which
proves the first assertion of the lemma. For the second, suppose that $t=\tau
e^{i\theta},$ and change variables, setting $w=y/\tau$ and $\lambda=x/\tau,$ to
obtain:
\begin{equation}
  \int\limits_0^{\infty}|k^b_t(x,y)|dy=
\int\limits_0^{\infty}w^be^{-\cos\theta(w+\lambda)}|\psi_b(w\lambda e^{-2i\theta})|\frac{dw}{w}.
\end{equation}
We split the integral into the part from $0$ to $1/\lambda$ and the rest. In
the compact part we use the estimate
\begin{equation}
  |\psi_b(w\lambda e^{-2i\theta})|\leq \left(\frac{1}{\Gamma(b)}+C_b|w\lambda|\right).
\end{equation}
Inserting this into the integral from $0$ to $1/\lambda,$ it is clear that
this term is uniformly bounded. In the non-compact part we use the asymptotic
expansion for $\psi_b$ to see that this term is bounded by
\begin{equation}
 I_+= C_b\int\limits_{\frac{1}{\lambda}}^{\infty}\left(\frac{w}{\lambda}\right)^{\frac{b}{2}-\frac 14}e^{-\cos\theta(\sqrt{w}-\sqrt{\lambda})^2}\frac{dw}{\sqrt{w}}.
\end{equation}
This integral is $O(e^{-\cos\theta/(2\lambda)})$ as $\lambda\to 0.$ As
$\lambda\to\infty,$ we let $z=\sqrt{w}-\sqrt{\lambda},$ to obtain that
\begin{equation}
  I_+=C_b\int\limits_{\frac{1}{\sqrt{\lambda}}-\sqrt{\lambda}}^{\infty}\left(\frac{z}{\sqrt{\lambda}}
+1\right)^{b-\frac 12}e^{-\cos\theta z^2}dz.
\end{equation}
It is elementary to see that this integral is bounded by a constant depending
only on $\theta.$
\end{proof}

\begin{remark} The proofs of the remaining estimates are in
  Appendix~\ref{prfsoflems}.
\end{remark}

\begin{lemma}\label{lem1new} For $b>0,$ $0<\gamma<1,$ and
  $0<\phi<\frac{\pi}{2},$ there are constants $C_{b, \phi}$
  uniformly bounded with $b,$ so that for $t\in S_{\phi}$
  \begin{equation}
\int\limits_0^{\infty}|k^b_t(x,y)-k^b_t(0,y)|y^{\frac{\gamma}{2}}dy\leq
C_{b,\phi} x^{\frac{\gamma}{2}}.
  \end{equation}
\end{lemma}

\begin{lemma}\label{lem21new} For $b>0$ there is a constant $C_{b,\phi}$ so
  that for $t\in S_{\phi}$
  \begin{equation}\label{lem21newest1}
    \int\limits_{0}^{\infty}|k^b_t(x,y)-k^b_t(0,y)|dy\leq C_{b,\phi}\frac{x/|t|}{1+x/|t|}.
  \end{equation}
  For $0<c<1$ there is a constant $C_{b,c,\phi}$ so that, if $cx_2<x_1<x_2,$
  and $t\in S_{\phi},$ then
\begin{equation}\label{lem21newest2}
  \int\limits_{0}^{\infty}|k^b_t(x_2,y)-k^b_t(x_1,y)|dy\leq
  C_{b,c,\phi}\left(\frac{\frac{\sqrt{x_2}-\sqrt{x_1}}{\sqrt{|t|}}}
{1+\frac{\sqrt{x_2}-\sqrt{x_1}}{\sqrt{|t|}}}\right).
\end{equation}
\end{lemma}

\begin{lemma}\label{lem5new} For $b>0,$ $0<\gamma<1$ and $t\in S_{\phi},$
  $0<\phi<\frac{\pi}{2},$ there is a $C_{b,\phi}$ so that
  \begin{equation}\label{eqn126.00}
    \int\limits_{0}^{\infty}|k^b_t(x,y)||\sqrt{x}-\sqrt{y}|^{\gamma}dy\leq
    C_{b,\phi} |t|^{\frac{\gamma}{2}}. 
  \end{equation}
For fixed $0<\phi,$ and $B,$ these constants are uniformly bounded for $0<b<B.$
\end{lemma}

For several estimates we need to split $\bbR_+$ into a collection of
subintervals. We let $J=[\alpha,\beta],$ where
\begin{equation}\label{eqn8555}
  \sqrt{\alpha}=\max\left\{\frac{3\sqrt{x_1}-\sqrt{x_2}}{2},0\right\}\text{ and }
\sqrt{\beta}=\frac{3\sqrt{x_2}-\sqrt{x_1}}{2}.
\end{equation}

\begin{lemma}\label{lem3new} We assume that $x_1/x_2>1/9$ and
  $J=[\alpha,\beta],$ as defined in~\eqref{eqn8555}.
For $b>0,$ $0<\gamma<1$ and $0<\phi<\frac{\pi}{2},$ there is a $C_{b,\phi}$  so that
if $t\in S_{\phi},$  then 
\begin{equation}
\int\limits_{J^c}|k^b_t(x_2,y)-k^b_t(x_1,y)||\sqrt{y}-\sqrt{x_1}|^{\gamma}dy
\leq C_{b,\phi}|\sqrt{x_2}-\sqrt{x_1}|^{\gamma}
\end{equation}
\end{lemma}

\begin{lemma}\label{lem4new} For $b>0,$ $0<\gamma<1$ and $c<1$ there is a $C_b$ such that if
  $c<s/t<1,$ then
\begin{equation}
  \int\limits_0^{\infty}\left|k_t^b(x,y)-k_s^b(x,y)\right|
|\sqrt{x}-\sqrt{y}|^{\gamma}dy\leq C_b|t-s|^{\frac{\gamma}{2}}.
\end{equation}
\end{lemma}

We also have the simpler result, which holds without restriction on $s<t,$ and
when $\gamma=0.$
\begin{lemma}\label{lem4newp2} For $b>0$ there is a $C_b$ such that if $s<t,$ then
\begin{equation}
  \int\limits_0^{\infty}\left|k_t^b(x,y)-k_s^b(x,y)\right|dy\leq C_b
\left(\frac{t/s-1}{1+[t/s-1]}\right).
\end{equation}
\end{lemma}

\subsection{First Derivative Estimates}
The following lemma is central to many of the results in this paper.
\begin{lemma}\label{lem2new} For $b>0,$ $0\leq\gamma<1,$ and
  $0<\phi<\frac{\pi}{2},$ there is a $C_{b,\phi}$ so that for $t\in S_{\phi}$ we have
  \begin{equation}
    \int\limits_{0}^{\infty}|\pa_xk^b_t(x,y)||\sqrt{y}-\sqrt{x}|^{\gamma}dy\leq
C_{b,\phi}\frac{|t|^{\frac{\gamma}{2}-1}}{1+\lambda^{\frac{1}{2}}},
  \end{equation}
where $\lambda=x/|t|.$
\end{lemma}
\noindent
The case $\gamma=0$ is Lemma 8.1 in~\cite{WF1d}.

\begin{lemma}\label{lem20neww}
For $b>0,$ $0<\gamma<1,$ $0<\phi<\frac{\pi}{2},$
  and $0<c<1,$ there is a constant $C_{b,c,\phi}$ so that for 
  $cx_2<x_1<x_2,$ $t\in S_{\phi},$
  \begin{multline}\label{eqn2400.0}
    \int\limits_0^{\infty}|\sqrt{x_1}\pa_{x}k^b_t(x_1,y)-
    \sqrt{x_2}\pa_{x}k^b_t(x_2,y)||\sqrt{x_1}-\sqrt{y}|^{\gamma}dy\leq\\
    C_{b,c,\phi}|t|^{\frac{\gamma-1}{2}}\frac{\left(\frac{|\sqrt{x_2}-\sqrt{x_1}|}
{\sqrt{|t|}}\right)}
    {1+\left(\frac{|\sqrt{x_2}-\sqrt{x_1}|}{\sqrt{|t|}}\right)}
  \end{multline} 
 \end{lemma}

\begin{lemma}\label{lemA-}
  For $b>0,$ $0<\gamma<1,$ there is a constant $C_b$ so that for $t_1<t_2<2t_1,$ we have:
  \begin{equation}
    \int\limits_{t_2-t_1}^{t_1}\int\limits_{0}^{\infty}
|\pa_xk^b_{t_2-t_1+s}(x,y)-\pa_xk^b_s(x,y)|
|\sqrt{x}-\sqrt{y}|^{\frac{\gamma}{2}}dyds<
C_b|t_2-t_1|^{\frac{\gamma}{2}}.
  \end{equation}
\end{lemma}

This result follows from the more basic:
\begin{lemma}\label{lemAA-}
  For $b>0,$ $0\leq \gamma<1,$ and  $0<t_1<t_2<2t_1,$ we have for $s\in [t_2-t_1,t_1]$
  that there is a constant $C$ so that
  \begin{equation}\label{lemAA-estp}
    \int\limits_{0}^{\infty}
|\pa_xk^b_{t_2-t_1+s}(x,y)-\pa_xk^b_s(x,y)|
|\sqrt{x}-\sqrt{y}|^{{\gamma}}dy<
C\frac{(t_2-t_1)s^{\frac{\gamma}{2}-1}}{(t_2-t_1+s)(1+\sqrt{x/s})}.
  \end{equation}
\end{lemma}

\subsection{Second Derivative Estimates}
\begin{lemma}\label{lemA}
For $b>0,$ $0<\gamma<1,$ and  $0<\phi<\frac{\pi}{2},$ there is a $C_{b,\phi}$ so
  that for $t=|t|e^{i\theta}$ with $|\theta|<\frac{\pi}{2}-\phi,$
  \begin{equation}\label{lemAest}
\begin{split}
    &\int\limits_{0}^{|t|}
\int\limits_0^{\infty}|x\pa_x^2k^b_{se^{i\theta}}(x,y)||\sqrt{y}-\sqrt{x}|^{\gamma}dyds\leq
C_{b,\phi}x^{\frac{\gamma}{2}}
\text{  and }\\
&\int\limits_{0}^{|t|}
\int\limits_0^{\infty}|x\pa_x^2k^b_{se^{i\theta}}(x,y)||\sqrt{y}-\sqrt{x}|^{\gamma}dyds\leq
C_{b,\phi}|t|^{\frac{\gamma}{2}}.
\end{split}
  \end{equation}
 \end{lemma}

This follows from the more basic result:

\begin{lemma}\label{lem25new}
For $b>0,$ $0\leq\gamma<1,$ $0<\phi<\frac{\pi}{2},$ there is a $C_{b,\phi}$ so that
if $t\in S_{\phi},$ then
  \begin{equation}\label{lem25newpest.01}
    \int\limits_0^{\infty}|x\pa_x^2k^b_t(x,y)||\sqrt{x}-\sqrt{y}|^{\gamma}dy\leq
    C_{b,\phi}\frac{\lambda |t|^{\frac{\gamma}{2}-1}}{1+\lambda},
  \end{equation}
where $\lambda=x/|t|.$
 \end{lemma}

\begin{lemma}\label{lemB}
For $b>0,$ $0<\gamma<1,$ $0<\phi<\frac{\pi}{2},$ and $0<x_2/3<x_1<x_2,$ there is a
  constant $C_{b,\phi}$ so that, for $t\in S_{\phi},$ we have
  \begin{equation}\label{lemBest}
    \int\limits_{0}^{|t|}\left|(\pa_yy-b)k^b_{se^{i\theta}}(x_2,\alpha)-
(\pa_yy-b)k^b_{se^{i\theta}}(x_2,\beta)\right|
ds\leq C_{b,\phi},
  \end{equation}
where $\alpha$ and $\beta$ are defined in~\eqref{eqn8555}.
  \end{lemma}

\begin{lemma}\label{lemC}
For $b>0,$ $0<\gamma<1,$  $\phi<\frac{\pi}{2},$ and $0<x_2/3<x_1<x_2,$ if
  $J=[\alpha,\beta],$ with the endpoints given by~\eqref{eqn8555}, there is a
  constant $C_{b,\phi}$ so that if $|\theta|<\frac{\pi}{2}-\phi,$ then
\begin{equation}\label{lemCest}
  \begin{split}
     &I_1=\int\limits_0^{|t|}\int\limits_{\alpha}^{\beta}
|L_bk^b_{se^{i\theta}}(x_2,y)||\sqrt{y}-\sqrt{x_2}|^{\gamma}dyds\leq 
C_{b,\phi}|\sqrt{x_2}-\sqrt{x_1}|^{\gamma}\\
 &I_2=\int\limits_0^{|t|}\int\limits_{\alpha}^{\beta}
|L_bk^b_{se^{i\theta}}(x_1,y)||\sqrt{y}-\sqrt{x_1}|^{\gamma}dyds\leq
C_{b,\phi}|\sqrt{x_2}-\sqrt{x_1}|^{\gamma}.
\end{split}
\end{equation}
  \end{lemma}

\begin{lemma}\label{lemD}
For $b>0,$ $0<\gamma<1,$ $0<\phi<\frac{\pi}{2},$ and $0<x_2/3<x_1<x_2,$ if $J=[\alpha,\beta],$ with the endpoints
  given by~\eqref{eqn8555}, there is a constant $C_{b,\phi}$ so that if
  $|\theta|<\frac{\pi}{2}-\phi,$ then
  \begin{equation}\label{lemDest}
    \int\limits_0^{|t|}\int\limits_{J^c}|L_bk^b_{se^{i\theta}}(x_2,y)-L_bk^b_{se^{i\theta}}(x_1,y)|
|\sqrt{y}-\sqrt{x_1}|^{\gamma}dyds\leq C_{b,\phi}|\sqrt{x_2}-\sqrt{x_1}|^{\gamma}.
  \end{equation}
  
\end{lemma}

\begin{lemma}\label{lemH}  For $b>0,$ $0<\gamma<1,$ and $t_1<t_2<2t_1$ 
there is a constant $C_b$ so that
  \begin{equation}\label{lemHest} 
    \int\limits_{t_2-t_1}^{t_1}\int\limits_0^{\infty}
   |L_bk^{b}_{t_2-t_1+s}(x,y)-L_bk^{b}_{s}(x,y)||\sqrt{x}-\sqrt{y}|^{\gamma}dyds
\leq C_b|t_2-t_1|^{\frac{\gamma}{2}}.
  \end{equation}
\end{lemma}
This lemma follows from the more basic result:
\begin{lemma}\label{lemHp2}  For $b>0,$ $0<\gamma<1,$ and $t_1<t_2<2t_1$ and
  $s>t_2-t_1,$ there is a constant $C_b$ so that
  \begin{equation}\label{lemHestp2} 
    \int\limits_0^{\infty}
   |L_bk^{b}_{t_2-t_1+s}(x,y)-L_bk^{b}_{s}(x,y)||\sqrt{x}-\sqrt{y}|^{\gamma}dy
\leq C_b (t_2-t_1)s^{\frac{\gamma}{2}-2}.
  \end{equation}
\end{lemma}

\subsection{Large $t$ behavior}
To study the resolvent kernel of $L_b,$ which is formally given by
\begin{equation}
  (\mu-L_b)^{-1}=\int\limits_0^{\infty}e^{-\mu t}e^{tL_b}dt,
\end{equation}
and the off-diagonal behavior of the heat kernel in many variables, it is useful
to have estimates for
\begin{equation}
  \int\limits_{0}^{\infty}|\pa_x^jk^b_t(x,y)|dx,\text{ and }
 \int\limits_{0}^{\infty}|x^{\frac{j}{2}}\pa_x^jk^b_t(x,y)|dy
\end{equation}
valid for $0<t.$ In the previous section we gave such results, but these were
intended to study the behavior of these kernels as $t\to 0^+,$ and assumed the
H\"older continuity of the data. To study the resolvent we also need estimates
as $t\to\infty,$ valid for bounded, continuous data.
\begin{lemma}\label{lrgt1db}For $0<b<B,$  $0<\phi<\frac{\pi}{2},$ and
  $j\in\bbN$ there is a constant $C_{j,B,\phi}$ so that if $t\in S_{\phi},$ then
  \begin{equation}\label{eqn12.149.1}
    \int\limits_{0}^{\infty}|\pa_x^jk^b_t(x,y)|dy\leq \frac{C_{j,B,\phi}}{|t|^j},
  \end{equation}
and
\begin{equation}\label{eqn12.149.11}
    \int\limits_{0}^{\infty}|x^{\frac{j}{2}}\pa_x^jk^b_t(x,y)|dy\leq
    \frac{C_{j,B}}{|t|^{{\frac{j}{2}}}}
  \end{equation} 
\end{lemma}

\subsection{The structure of the proofs of the lemmas}
We close this subsection by considering the
structure of the proofs of these estimates. Recall that
\begin{equation}
  k^b_t(x,y)=\frac{1}{y}\left(\frac{y}{t}\right)^be^{-\frac{x+y}{t}}\psi_b\left(\frac{xy}{t^2}\right).
\end{equation}
In most of the estimates that follow we set $w=y/|t|,$ $\lambda=x/|t|,$ and
$\et=e^{-i\theta};$ in these variables
\begin{equation}
  k^b_t(x,y)dy=(\et w)^{b-1}e^{-(w+\lambda)\et}\psi_b(w\lambda\ett)dw.
\end{equation}
Using Taylor's theorem when $w\lambda<1,$ and the asymptotic expansions for the
functions, $\psi_b, \psi_b'$ when $w\lambda\geq 1,$ we repeatedly reduce our
considerations to the estimation of a small collection of types of
integrals. Most of these are integrals that extend from $0$ to $1/\lambda,$ or
from $1/\lambda$ to $\infty.$ We need to consider what happens as $\lambda$
itself varies from $0$ to $\infty.$ The following results are used repeatedly
in the proofs of the foregoing lemmas.
\begin{lemma}\label{lem4}
  For $\gamma>0,$ $0<\phi<\frac{\pi}{2},$ there are constants $C_{b,\phi},C'_{b,\phi}$
  uniformly bounded for $0<b<B,$ so that for $0<\lambda<\infty,$
  $|\theta|\leq\frac{\pi}{2}-\phi,$ we have
  \begin{equation}\label{bscest01}
    \int\limits_{0}^{\frac{1}{\lambda}}w^{b-1}e^{-\cost w}|\sqrt{w}-\sqrt{\lambda}|^{\gamma}dw
\leq \begin{cases}
\frac{C_{b,\phi}}{b}\lambda^{\frac{\gamma}{2}-b}\text{ as }\lambda\to\infty\\
\frac{C_{b,\phi}}{b}\lambda^{\frac{\gamma}{2}+b}+C'_{b,\phi}\text{ as }\lambda\to 0.
\end{cases}
  \end{equation}
\end{lemma}
\begin{proof} The proofs of this estimate follows easily from the change of
  variables $w=\lambda\sigma.$ 
\end{proof}
\begin{lemma}\label{lem5} For  $\gamma\geq 0,$ $0<\phi<\frac{\pi}{2},$ and $\nu\in\bbR,$
There are constants $C_{\nu,\gamma,\phi}$ and $a_{\nu,\gamma},$ uniformly bounded for $|\nu|<B,$ so
  that for $0<\lambda<\infty,$ $|\theta|<\frac{\pi}{2}-\phi,$ we have
  \begin{multline}\label{bscest10}
    \int\limits_{\frac{1}{\lambda}}^{\infty}w^{\frac{\nu}{2}}
e^{-\cost(\sqrt{w}-\sqrt{\lambda})^2}|\sqrt{w}-\sqrt{\lambda}|^{\gamma}
\frac{dw}{\sqrt{w}}\leq\\
\begin{cases} C_{\nu,\gamma,\phi}\lambda^{a_{\nu,\gamma}}e^{-\frac{\cost}{\lambda}}&\text{ as }
\lambda\to 0^+\\
C_{\nu,\gamma,\phi}\lambda^{\frac{\nu}{2}}
&\text{ as }\lambda\to\infty.
\end{cases}
  \end{multline}
\end{lemma}
\begin{proof}
Setting $z=\sqrt{\frac{w}{\lambda}}-1,$ the integral in~\eqref{bscest10}
becomes:
\begin{equation}
  I=\frac{\lambda^{\frac{\nu+\gamma+1}{2}}}{2}
\int\limits_{\frac{1}{\lambda}-1}^{\infty}(1+z)^{\nu}e^{-\cost\lambda z^2}|z|^{\gamma}dz.
\end{equation}
The estimate as $\lambda\to 0$ follows easily from this and Lemma~\ref{lem3},
proved below. To prove the result as $\lambda\to\infty,$ we need to split the
integral into the part from $\frac{1}{\lambda}-1$ to $-\frac{1}{2},$ and the
rest. A simple application of Laplace's method shows that the unbounded part is
estimated by $C_{\nu,\gamma,\phi}\lambda^{\frac{\nu}{2}}.$ We can estimate the
compact part by
\begin{equation}
\begin{split}
\frac{e^{-\cost\frac{\lambda}{4}}\lambda^{\frac{\nu+\gamma+1}{2}}}{2}  
\int\limits_{\frac{1}{\lambda}-1}^{-\frac 12}(1+z)^{\nu}
= \hfill \\ \frac{e^{-\cost\frac{\lambda}{4}}\lambda^{\frac{\nu+\gamma+1}{2}}}{2} 
\begin{cases}  \frac{1}{\nu+1}\left(\left(\frac{1}{2}\right)^{\nu+1}-
\left(\frac{1}{\lambda}\right)^{\nu+1}\right)&\text{ if }
\nu\neq -1 \\
\log\left(\frac{\lambda}{2}\right)&\text{ if }\nu=-1.
\end{cases}
\end{split}
\end{equation}
In all cases this quantity is bounded by $C_{\nu,\gamma,\phi}e^{-\cost\frac{\lambda}{8}},$
completing the proof of the Lemma.
\end{proof}

The following lemma is used to prove these estimates:
\begin{lemma}\label{lem3} Let $\mu\in\bbR$ and $a>0,$ we define
  \begin{equation}
    G_{\mu}(\lambda,a)=\int\limits_{a}^{\infty}e^{-\lambda z^2}z^{\mu}dz.
  \end{equation}
  There are constants $C_{\mu}$ so that
  \begin{equation}
    G_{\mu}(\lambda,a)\leq
    C_{\mu}\frac{e^{-\lambda a^2}}{\lambda a^{1-\mu}}\quad\text{ for
    }a\sqrt{\lambda}>\frac 12.
  \end{equation}
For $\mu> -1,$
\begin{equation}
    G_{\mu}(\lambda,a)\leq
    C_{\mu}\frac{1}{\lambda^{\frac{1+\mu}{2}}}\quad\text{ for
    }a\sqrt{\lambda}\leq\frac 12,
  \end{equation}
if $\mu=-1,$ then
\begin{equation}
    G_{\mu}(\lambda,a)\leq
    C_{-1}|\log(a\sqrt{\lambda})|\quad\text{ for
    }a\sqrt{\lambda}\leq\frac 12,
\end{equation}
if $\mu<-1,$ then
\begin{equation}
    G_{\mu}(\lambda,a)\leq
    C_{\mu}a^{1+\mu}\quad\text{ for
    }a\sqrt{\lambda}\leq\frac 12,
  \end{equation}
\end{lemma}
\begin{proof} The proofs are elementary. A simple change of variables shows
  that
  \begin{equation}
    G_{\mu}(\lambda,a)=\frac{1}{\lambda^{\frac{1+\mu}{2}}}G_{\mu}(1,a\sqrt{\lambda}).
  \end{equation}
The second estimate is immediate from this formula and the fact that
$G_{\mu}(1,0)$ is finite, for $\mu>-1.$ To prove the first relation we
integrate by parts to obtain that: 
\begin{equation}
  \int\limits_{w}^{\infty}e^{-z^2}dz=-\frac{e^{-z^2}}{2z}\Bigg|_{w}^{\infty}+
 \int\limits_{w}^{\infty}\frac{e^{-z^2}}{2z^2}
\end{equation}
This easily implies that
\begin{equation}
  \int\limits_{w}^{\infty}e^{-z^2}dz\leq \frac{e^{-w^2}}{w},
\end{equation}
which implies the first estimate. The final two estimates follow from
the fact that $G_{\mu}(1,a\sqrt{\lambda})$ diverges as $a\sqrt{\lambda}\to 0$
at a rate determined by $\mu\leq -1.$
\end{proof}

\section[H\"older estimates in 1 dimension]{H\"older Estimates for the
1-dimensional Model Problems} 
With these rather extensive preliminaries out of the way, we now give the proofs
for the H\"older estimates on solutions stated above.

\begin{proof}[Proof of Proposition~\ref{prop2}] We first assume that
  $b>0,$ and begin with~\eqref{bscest0} for the case $k=0.$ Using
  Proposition~\ref{prop4.1new}, Lemma~\ref{lem3.1new.0}, the $b=0$
  case follows from the $b>0$ case. To prove the higher order
  estimates we need to assume that the data has support in $[0,R],$
  then these results follow easily from the $k=0$ case by using
  Propositions~\ref{prop4.1new}, and~\ref{wtedest_pr}. From
  the maximum principle it is immediate that the sup-norm of $v$ is
  bounded by $\|f\|_{\WF,0,\gamma}.$ In light of Lemma~\ref{lem1} it
  suffices to separately prove that
  \begin{equation}
\begin{split}
    |v(x,t)-v(y,t)|&\leq C\|f\|_{\WF,0,\gamma}|\sqrt{x}-\sqrt{y}|^{\gamma}\text{
      and }\\
|v(x,t)-v(x,s)|&\leq C\|f\|_{\WF,0,\gamma}|t-s|^{\frac{\gamma}{2}}.
\end{split}
  \end{equation}

We start the spatial estimate, by estimating $|v(x,t)-v(0,t)|.$ Because
for every $x,$ and $t>0,$~\eqref{intkbtis1} holds,
 we use the formula for $k^b_t,$ to deduce that:
\begin{equation}\label{frstdrvCP3}
  v(x,t)-v(0,t)=\int\limits_{0}^{\infty}
\left[k^b_t(x,y)-k^b_t(0,y)\right](f(y)-f(0))dy
\end{equation}
 The basic estimate~\eqref{mstbscest} shows that
\begin{equation}
  |v(x,t)-v(0,t)|=2\|f\|_{\WF,0,\gamma}\int\limits_{0}^{\infty}
\left|k^b_t(x,y)-k^b_t(0,y)\right|y^{\frac{\gamma}{2}}dy;
\end{equation}
we apply Lemma~\ref{lem1new}   to see that:
\begin{equation}\label{estat0}
  |v(x,t)-v(0,t)|\leq C_bx^{\frac{\gamma}{2}}\|f\|_{\WF,0,\gamma},
\end{equation}
for all $t>0,$ and that, for any $B,$ the $\{C_b\}$ are uniformly bounded for
$0<b<B.$

This is a very useful estimate, for observe that if $M>1,$ then
\begin{equation}\label{lrgratioest}
 y^{\frac{\gamma}{2}}\leq \frac{M-1}{M} x^{\frac{\gamma}{2}}\Longrightarrow 
x^{\frac{\gamma}{2}}\leq M(x^{\frac{\gamma}{2}}-y^{\frac{\gamma}{2}}).
\end{equation}
Thus~\eqref{estat0} implies that if $y^{\frac{\gamma}{2}}\leq \frac{M-1}{M}
x^{\frac{\gamma}{2}},$ then there is a constant $C_{\gamma,b}$ so that $v$
satisfies
\begin{equation}\label{estat3}
\begin{split}
  |v(x,t)-v(y,t)|&\leq |v(x,t)-v(0,t)|+|v(0,t)-v(y,t)|\\
&\leq 2C_b\|f\|_{\WF,0,\gamma}x^{\frac{\gamma}{2}}\\
&\leq  C_{\gamma,b}\|f\|_{\WF,0,\gamma}(x^{\frac{\gamma}{2}}-y^{\frac{\gamma}{2}})
\end{split}
\end{equation} 
Applying Lemma~\ref{lem2} we see that~\eqref{estat3} implies that
\begin{equation}\label{estat4}
  |v(x,t)-v(y,t)|\leq C_{\gamma,b}\|f\|_{\WF,0,\gamma}|\sqrt{x}-\sqrt{y}|^{\gamma}.
\end{equation}

To complete the spatial part of the estimate we just need to show that~\eqref{estat4}
holds, as $\lambda\to\infty,$ for pairs $(x,y)$ so that
\begin{equation}\label{bndratio}
  c\leq\frac{y}{x}<1,
\end{equation}
with $c$ a positive number less than $1.$ To that end we introduce a device,
familiar from the Euclidean case that will allow us to obtain the needed
estimate. For points $0\leq x_1<x_2$ we define $J=[\alpha,\beta]$ where
$\alpha$ and $\beta$ are defined by
\begin{equation}
  \sqrt{\alpha}=\max\left\{\frac{3\sqrt{x_1}-\sqrt{x_2}}{2},0\right\}\text{ and }
\sqrt{\beta}=\frac{3\sqrt{x_2}-\sqrt{x_1}}{2}.
\end{equation}
As noted above, this is the WF-ball centered on the
WF-midpoint of $[x_1,x_2],$ with radius equal to $d_{\WF}(x_1,x_2).$

Using the fact that $k_t^b(x,y)$ has $y$-integral
1, we easily deduce that
\begin{multline}\label{dlctest0}
  v(x_2,t)-v(x_1,t)=(f(x_2)-f(x_1))+\\
\int\limits_{J}k^b_t(x_2,y)(f(y)-f(x_2))dy-
\int\limits_{J}k^b_t(x_1,y)(f(y)-f(x_1))dy+\\
\int\limits_{J^c}k^b_t(x_2,y)(f(x_1)-f(x_2))dy+
\int\limits_{J^c}(k^b_t(x_2,y)-k^b_t(x_1,y))(f(y)-f(x_1))dy
\end{multline}
It is a simple matter to see that the first four terms are estimated by 
\begin{equation}\label{eqn9.62.00}
  C\|f\|_{\WF,0\,\gamma}|\sqrt{x_2}-\sqrt{x_1}|^{\gamma},
\end{equation}
leaving just the second integral over $J^c.$ Terms of this type are estimated,
for $c>1/9,$ in Lemma~\ref{lem3new}. Thus for $f\in C_{\WF,0,\gamma}(\bbR_+)$
Lemma~\ref{lem3new} shows that there is a constant $C$ independent of $b\leq B$
so that $v$ satisfies~\eqref{estat4}.

We now turn to the time estimate. We begin by estimating $|v(x,t)-v(x,0)|.$
Arguing as above we see that we have the estimate:
\begin{equation}
\begin{split}
|v(x,t)-v(x,0)|&=\left|\int\limits_{0}^{\infty}k^b_t(x,y)(f(y)-f(x))dy\right|\\
&\leq 2\|f\|_{\WF,0,\gamma}
\int\limits_{0}^{\infty}k^b_t(x,y)|\sqrt{x}-\sqrt{y}|^{\gamma}dy
\end{split}
\end{equation}
Integrals of this type are estimated in Lemma~\ref{lem5new}, which shows that
\begin{equation}\label{t0est1}
  |v(x,t)-v(x,0)|\leq Ct^{\frac{\gamma}{2}}\|f\|_{\WF,0,\gamma}.
\end{equation}
Using the estimate in~\eqref{lrgratioest}, we see that for $M>1,$ 
if $Ms^{\frac{\gamma}{2}}\leq (M-1)t^{\frac{\gamma}{2}},$ then~\eqref{t0est1}
implies that
\begin{equation}\label{t0est2}
  |v(x,t)-v(x,s)|\leq 2MC\|f\|_{\WF,0,\gamma}|t^{\frac{\gamma}{2}}-s^{\frac{\gamma}{2}}|.
\end{equation}
Using Lemma~\ref{lem2} this estimate gives
\begin{equation}\label{t0est3}
  |v(x,t)-v(x,s)|\leq C_{\gamma,b}\|f\|_{\WF,0,\gamma}|t-s|^{\frac{\gamma}{2}},
\end{equation}
for a constant $C_{\gamma,b}$ uniformly bounded for $b\leq B.$ This leaves only
the case of $c<s/t<1,$ for a $c<1.$

To complete the last case, we write
\begin{equation}
\begin{split}
  |v(x,t)-v(x,s)|&\leq\int\limits_0^{\infty}\left|k_t^b(x,y)-k_s^b(x,y)\right||f(y)-f(x)|dy\\
&\leq 2\|f\|_{\WF,0,\gamma}\int\limits_0^{\infty}\left|k_t^b(x,y)-k_s^b(x,y)\right|
|\sqrt{x}-\sqrt{y}|^{\gamma}dy.
\end{split}
\end{equation}
This case follows from Lemma~\ref{lem4new}.
Using  this lemma we easily complete the proof of the Proposition~\ref{prop2}
for the case $k=0.$ The assertion that
$v\in\cC^{0,\gamma}_{\WF}(\bbR_+\times\bbR_+)$ follows easily from these
estimates. Notice that Lemma~\ref{lem25new} applies to show that even if $f$ is
only in $\cC^{0,\gamma}_{\WF}(\bbR_+)$ then
\begin{equation}\label{eqn199new.2}
  \lim_{x\to 0}x\pa_x^2v(x,t)=0\text{ for any }t>0.
\end{equation}
To show that
\begin{equation}\label{eqn200new.2}
  \lim_{x\to\infty}v(x,t)=0\text{ for any }t>0,
\end{equation}
we fix an $R>>0$ and write
\begin{equation}
  v(x,t)=\int\limits_{0}^{R}k^b_t(x,y)f(y)dy+\int\limits_R^{\infty}k^b_t(x,y)f(y)dy.
\end{equation}
Proposition~\ref{prop3.4newdcyinfty} shows that for any fixed $R$ the first
term tends uniformly to zero as $x\to\infty.$ As
$f\in\cC^{0,\gamma}_{\WF}(\bbR_+)$ it follows that $\lim_{x\to\infty}f(x)=0.$
Hence given $\epsilon>0$ we can choose $R_0$ so that $|f(x)|<\epsilon$ for
$x>R_0.$ For this choice of $R_0,$ the second integral is at most $\epsilon,$
for all $x,$ and the first tends to zero as $x\to\infty.$ Thus
\begin{equation}
  \limsup_{x\to\infty}|v(x,t)|\leq\epsilon,
\end{equation}
which proves~\eqref{eqn200new.2}.

The estimates for the $(2+\gamma)$-spaces follow easily
from what we have just proved and Lemma~\ref{lem3.1new.0}. This shows that
if $f\in\dcC^m_b([0,\infty)),$ and $2j+k\leq m,$ then
\begin{equation}\label{comfrm1}
  \pa_t^j \pa_x^kv(x,t)=
  \int\limits_0^{\infty}k_t^{b+k}(x,y)L^j_{b+k}\pa_y^kf(y)dy.
\end{equation}
In particular, the relations
\begin{equation}
\begin{split}
   \pa_xv(x,t)&=
\int\limits_0^{\infty}k_t^{b+1}(x,y)\pa_yf(y)dy,\\
\pa_tv(x,t)&=
\int\limits_0^{\infty}k_t^{b}(x,y)L_bf(y)dy\text{ and }\\
x\pa_x^2v(x,t)&=(\pa_t-b\pa_x)v,
\end{split}
\end{equation}
and the $\gamma$-case, show that $\|v\|_{\WF,0,2+\gamma}$ is bounded by
$\|f\|_{\WF,0,2+\gamma}.$ Using these identities along
with~\eqref{eqn199new.2} and~\eqref{eqn200new.2} allows us  to conclude that
\begin{equation}
  \lim_{x\to\infty}[|v(x,t)|+|\pa_xv(x,t)|+x\pa_x^2v(x,t)|]=0.
\end{equation}
We can therefore apply Lemma~\ref{lem4.2.2} to see that
$v\in\cC^{0,2+\gamma}_{\WF}(\bbR_+\times\bbR_+).$

For the $k>0$ cases, we need to assume that $f$ is supported in
$[0,R].$ Now, using~\eqref{comfrm1} and Proposition~\ref{wtedest_pr}
we easily derive~\eqref{bscest0}, and~\eqref{bscest2} for any
$k\in\bbN,$ and can again conclude that
$v\in\cC^{k,2+\gamma}_{\WF}(\bbR_+\times\bbR_+),$ provided that
$f\in\cC^{k,2+\gamma}_{\WF}(\bbR_+),$ and $\supp f\subset [0,R].$

Finally we consider what happens as $b\to 0^+.$ We begin with the $k=0$ case;
Proposition 7.8 in~\cite{WF1d} shows that the solutions to the Cauchy problem
for $b>0$ converge uniformly to the solution with $b=0$ in sets of the form
$\bbR_+\times [0,T].$ If $v^b(x,t)$ denotes these solutions, then we have
established the existence of constants $C_{\gamma}$ so that for $0<b<1,$ $x\neq
y$ and $t\neq s,$ the following estimates hold:
\begin{equation}
  |v^b(x,t)-v^b(y,s)|\leq 
C_{\gamma}\|f\|_{\WF,0,\gamma}[|\sqrt{x}-\sqrt{y}|^{\gamma}+|t-s|^{\frac{\gamma}{2}}].
\end{equation}
As the constants $C_{\gamma}$ are independent of $b,$ we can let $b$ tend to
zero, and apply Proposition 7.8 of~\cite{WF1d} to conclude that this estimate
continues to hold for $b=0.$ Using~\eqref{comfrm1} as above we can extend all
the remaining estimates for the $\cC^{k,\gamma}_{\WF}$-spaces to the $b=0$ case
as well.

To treat the $\cC^{k,2+\gamma}_{\WF}$-spaces, we use
Proposition~\ref{prop4.1new}. If $f\in\cC^{k,2+\gamma}_{\WF}(\bbR_+),$ then the
solutions $v^b$ to~\eqref{1dmdlb}, with $g=0,$ $v^b(x,0)=f(x)$ and $0<b<1,$ are a
bounded family in $\cC^{k,2+\gamma}_{\WF}(\bbR_+\times\bbR_+).$ Thus for any
$0<\gamma'<\gamma$ there is a subsequence $<v^{b_n}>$ with $b_n\to 0,$ which
converges to $v^{*}\in \cC^{k,2+\gamma'}_{\WF}(\bbR_+\times\bbR_+).$ Evidently
$v^{*}$ satisfies
\begin{equation}
  (\pa_t-L_0)v^{*}=0\text{ with }v^{*}(x,0)=f(x).
\end{equation}
The uniqueness theorem implies that $v^{*}=v^0,$ and therefore the family
$<v^{b}>$ converges in $\cC^{k,2+\gamma'}_{\WF}(\bbR_+\times\bbR_+)$ to
$v^{0}.$ Since each element of $\{v^b:\: 0<b<1\}$ satisfies the estimates
in~\eqref{bscest2}, with uniformly bounded constants, we conclude that
$v^{0}\in\cC^{k,2+\gamma}_{\WF}(\bbR_+\times\bbR_+),$ and $v^{0}$ also satisfies the
estimate in~\eqref{bscest2}.
\end{proof}

 We now turn to the proof of Proposition~\ref{prop1}. Many parts of the
 foregoing argument can be recycled:
\begin{proof}[Proof of Proposition~\ref{prop1}]
We begin by studying the operator:
\begin{equation}
  K^b_tg(x)=\int\limits_{0}^t\int\limits_{0}^{\infty}k^b_{t-s}(x,y)g(y,s)dyds,
\end{equation}
assuming that $g\in\cC_{\WF}^{0,\gamma}(\bbR_+\times [0,T]).$ We want to show
that$$K^b_t:\cC_{\WF}^{0,\gamma}(\bbR_+\times
[0,T])\to\cC_{\WF}^{0,2+\gamma}(\bbR_+\times [0,T])$$ 
is bounded.  This entails differentiating under the integral defining $K^b_tg,$
which is somewhat subtle near to $s=t.$ If $g\in\CI,$ then we can apply
Corollary 7.6 of~\cite{WF1d} to conclude that
\begin{equation}\label{eqn204new.3}
  \pa_t^j\pa_x^kL_b^lu(x,t)=\lim_{\epsilon\to 0^+}
\pa_t^j\pa_x^kL_b^l\int\limits_{0}^{t-\epsilon}\int\limits_{0}^{\infty}k^b_{t-s}(x,y)g(y,s)dyds,
\end{equation}
In the arguments that follow we show that if $g$ is sufficiently smooth, then the
derivatives, $\pa_t^j\pa_x^kL_b^lu,$ exist and can be defined by this
limit. Once cancellations are taken into account, the limits are, in fact,
absolutely convergent. Provided that $g$ is sufficiently smooth, we may use
Lemma~\ref{lem3.2new} to bring derivatives past the kernel onto
$g.$

Of special import is the case $g\in\cC^{0,\gamma}_{\WF}(\bbR_+\times [0,T]).$
For $0<\epsilon<t,$ we let:
\begin{equation}\label{eqn6.94.00}
  u_{\epsilon}(x,t)=\int\limits_{0}^{t-\epsilon}\int\limits_{0}^{\infty}k^b_{t-s}(x,y)g(y,s)dyds
\end{equation}
It follows easily that $u_{\epsilon}$ converges uniformly to $u$ in
$\bbR_+\times [0,T].$ Using the standard estimate on the difference
\begin{equation}
  |g(x,t)-g(y,t)|\leq \|g\|_{\WF,0,\gamma,T}|\sqrt{x}-\sqrt{y}|^{\gamma}
\end{equation}
and the facts that
\begin{equation}
  \begin{split}
\pa_xu_{\epsilon}(x,t)&
=\int\limits_{0}^{t-\epsilon}\int\limits_{0}^{\infty}\pa_xk^b_{t-s}(x,y)[g(y,s)-g(x,s)]dyds\\
x\pa_x^2u_{\epsilon}(x,t)&
=\int\limits_{0}^{t-\epsilon}\int\limits_{0}^{\infty}x\pa_x^2k^b_{t-s}(x,y)[g(y,s)-g(x,s)]dyds,
\end{split}
\end{equation}
we can apply Lemmas~\eqref{lem2new} and~\eqref{lem25new} to establish the
uniform convergence of $\pa_xu_{\epsilon}$ and $x\pa_x^2u_{\epsilon}$ on
$\bbR_+\times [0,T].$ 
This establishes the continuous differentiability of $u$ in $x$ on
$[0,\infty)\times [0,T],$ and the twice continuous differentiability of $u$ in $x$ on
$(0,\infty)\times [0,T].$ 
We can  differentiate $u_{\epsilon}$ in $t$ to obtain that
\begin{equation}
  \pa_tu_{\epsilon}(x,t)=
\int\limits_{0}^{\infty}\pa_xk^b_{\epsilon}(x,y)g(y,t-\epsilon)+
[x\pa_x^2+b\pa_x]u_{\epsilon}(x,t).
\end{equation}
The right hand side converges uniformly to $g(x,t)+
(x\pa_x^2+b\pa_x)u(x,t),$ thereby establishing the continuous differentiability
of $u$ in $t$ and the fact that
\begin{equation}
  [\pa_t-(x\pa_x^2+b\pa_x)]u=g\text{ for }(x,t)\in \bbR_+\times [0,T].
\end{equation}

This argument, or a variant thereof, is used repeatedly  to establish the
differentiability of $u,$ the formul{\ae} for its derivatives:
\begin{equation}\label{eqn6.100.00}
  \begin{split}
\pa_xu(x,t)&
=\int\limits_{0}^{t}\int\limits_{0}^{\infty}\pa_xk^b_{t-s}(x,y)[g(y,s)-g(x,s)]dyds\\
x\pa_x^2u(x,t)&
=\int\limits_{0}^{t}\int\limits_{0}^{\infty}x\pa_x^2k^b_{t-s}(x,y)[g(y,s)-g(x,s)]dyds,
\end{split}
\end{equation}
along with the fact that, for $g\in\cC^{0,\gamma}_{\WF}(\bbR_+\times [0,T]),$
these are absolutely convergent integrals.

We let $u(x,t)=K^b_tg(x).$ From the maximum principle it is evident that
\begin{equation}\label{9.73.1}
  |u(x,t)|\leq t\|g\|_{L^{\infty}}
\end{equation}
The estimate in~\eqref{estat4} can be integrated to prove that
\begin{equation}\label{9.74.1}
  |u(x,t)-u(y,t)|\leq C_b t\|g\|_{\WF,0,\gamma}|\sqrt{x}-\sqrt{y}|^{\gamma}.
\end{equation}
Using~\eqref{2nddrvest1} and~\eqref{eqn108.1}, proved below, and the
equation $\pa_tu=L_bu+g,$ we see that, for $s<t,$
\begin{equation}\label{eqn9.75.1}
\begin{split}
  |u(x,t)-u(x,s)|&\leq C(1+t^{\frac{\gamma}{2}})\|g\|_{\WF,0,\gamma}|t-s|\\
&Ct^{1-\frac{\gamma}{2}}(1+t^{\frac{\gamma}{2}})\|g\|_{\WF,0,\gamma}|t-s|^{\frac{\gamma}{2}}
\end{split}
\end{equation}
Note that~\eqref{9.73.1},~\eqref{9.74.1} and~\eqref{eqn9.75.1} show that
there is a constant $C$ so that
\begin{equation}\label{eqn9.76.1}
  \|u\|_{\WF,0,\gamma,T}\leq CT^{1-\frac{\gamma}{2}}(1+T^{\frac{\gamma}{2}})\|g\|_{\WF,0,\gamma,T}.
\end{equation}
Below we show that there is a constant $C_b$ so that
\begin{equation}\label{2nddrvest1}
  |x\pa_x^2u(x,t)|\leq C_b\min\{x^{\frac{\gamma}{2}},t^{\frac{\gamma}{2}}\}\|g\|_{\WF,0,\gamma};
\end{equation}
Dividing by $x$ and integrating gives the H\"older estimate for the first
spatial derivative:
\begin{equation}
  |\pa_xu(x,t)-\pa_xu(y,t)|\leq C_b\|g\|_{\WF,0,\gamma}|x^{\frac{\gamma}{2}}-y^{\frac{\gamma}{2}}|.
\end{equation}
Lemma~\ref{lem2} then implies that
\begin{equation}
  |\pa_xu(x,t)-\pa_xu(y,t)|\leq C_b\|g\|_{\WF,0,\gamma}|\sqrt{x}-\sqrt{y}|^{\gamma}.
\end{equation}

To complete the analysis of $\pa_xu$ we need to show that there is a constant
$C_b$ so that
\begin{equation}\label{1stderthldr}
  |\pa_xu(x,t)-\pa_xu(x,s)|\leq C_b\|g\|_{\WF,0,\gamma}|t-s|^{\frac{\gamma}{2}}.
\end{equation}
In Lemma~\ref{lem2new} it is shown that there are constants $C_b,$
  uniformly bounded for $b<B,$ so that, with $\lambda=x/t,$ we have
 \begin{equation}\label{frstdrvCP2}
  |\pa_xv(x,t)|\leq C\|f\|_{\WF,0,\gamma}\frac{t^{\frac{\gamma}{2}-1}}
{1+\lambda^{\frac{1}{2}}}=
Cx^{\frac{\gamma}{2}-1}\|f\|_{\WF,0,\gamma}\frac{\lambda^{1-\frac{\gamma}{2}}}
{1+\lambda^{\frac{1}{2}}}.
 \end{equation}
It follows by integrating that
\begin{equation}\label{eqn108.1}
  |\pa_xu(x,t)|\leq C_b\|g\|_{\WF,0,\gamma}t^{\frac{\gamma}{2}},
\end{equation}
and therefore, for any $c<1,$ there is a $C$ so that if $s<ct,$
then~\eqref{1stderthldr} holds with $C_b=C.$ We are left to consider $ct_2<t_1<t_2,$
for any $c<1.$  For $\frac 12 t_2<t_1<t_2$ we have:
\begin{multline}\label{1dertmprts}
  \pa_xu(x,t_2)-\pa_xu(x,t_1)=\int\limits_{0}^{t_2-t_1}\int\limits_{0}^{\infty}
\pa_xk^b_{s}(x,y)[g(y,t_2-s)-g(y,t_1-s)]dyds+\\
\int\limits_{0}^{2t_1-t_2}\int\limits_{0}^{\infty}
[\pa_xk^b_{t_2-s}(x,y)-\pa_xk^b_{t_1-s}(x,y)](g(y,s)-g(x,s))dyds+\\
\int\limits_{2t_1-t_2}^{t_1}\int\limits_{0}^{\infty}
\pa_xk^b_{t_2-s}(x,y)(g(y,s)-g(x,s))dyds.
\end{multline}
To handle the first term, we observe that, for $j=1,2$ we have 
\begin{multline}
  \int\limits_{0}^{t_2-t_1}\int\limits_{0}^{\infty}
\pa_xk^b_{s}(x,y)g(y,t_j-s)dyds=\\
\int\limits_{0}^{t_2-t_1}\int\limits_{0}^{\infty}
\pa_xk^b_{s}(x,y)[g(y,t_j-s)-g(x,t_j-s)]dyds,
\end{multline}
which can be estimated by
\begin{equation}
  \|g\|_{\WF,0,\gamma}\int\limits_{0}^{t_2-t_1}\int\limits_{0}^{\infty}
|\pa_xk^b_{s}(x,y)||\sqrt{x}-\sqrt{y}|^{\frac{\gamma}{2}}dyds.
\end{equation}
Using Lemma~\ref{lem2new}, we see that these terms are bounded by the right hand
side of~\eqref{1stderthldr}. In the third integral in~\eqref{1dertmprts} we
use~\eqref{mstbscest2} to estimate $(g(y,s)-g(x,s)),$ and again apply
Lemma~\ref{lem2new} to see that this term is also bounded by the right hand
side of~\eqref{1stderthldr}. This leaves only the second integral
in~\eqref{1dertmprts}. To estimate this term we use  Lemma~\ref{lemA-}.

We now establish~\eqref{2nddrvest1}, and then the H\"older continuity of
$x\pa_x^2u(x,t).$ Because $k^b_t(x,y)$ integrates to $1$ w.r.t. $y,$ for any
$x,$ and $t>1$ it follows from~\eqref{eqn204new.3} that:
\begin{equation}\label{eqn188.0.4}
  x\pa_x^2u(x,t)=\int\limits_{0}^{t}
\int\limits_0^{\infty}x\pa_x^2k^b_s(x,y)[g(y,t-s)-g(x,t-s)]dyds.
\end{equation}
Using the estimate
\begin{equation}\label{ghldest1}
  |g(y,t-s)-g(x,t-s)|\leq 2\|g\|_{\WF,0,\gamma}|\sqrt{x}-\sqrt{y}|^{\gamma}
\end{equation}
 and  Lemma~\ref{lemA}.gives:
\begin{equation}
\begin{split}
  |x\pa_x^2u(x,t)| &\leq\|g\|_{\WF,0,\gamma}\int\limits_{0}^{t}
\int\limits_0^{\infty}|x\pa_x^2k^b_s(x,y)||\sqrt{y}-\sqrt{x}|^{\gamma}dyds\\
&\leq C_b\|g\|_{\WF,0,\gamma}\min\{t^{\frac{\gamma}{2}},x^{\frac{\gamma}{2}}\}.
\end{split}
\end{equation}
This completes the proof
of~\eqref{2nddrvest1}, and therefore the proof of the spatial H\"older
continuity of $\pa_xu.$ This argument also establishes that
\begin{equation}\label{eqn193.1.4}
  \lim_{x\to 0^+}x\pa_x^2u(x,t)=0.
\end{equation}
To verify the hypotheses of Lemma~\ref{lem4.2.2} need also to show that
\begin{equation}\label{eqn193.1.5}
  \lim_{x\to\infty}[|u(x,t)|+|\pa_xu(x,t)|+|x\pa_x^2u(x,t)|]=0.
\end{equation}
The claim for $|u(x,t)|$ follows as in the proof of~\eqref{eqn200new.2}. To
estimate the derivatives we need to split the integral defining $u$ into a
compact and non-compact part; though more carefully than before. 

Let $\varphi\in\cC^{\infty}$ satisfy
\begin{equation}
  \varphi(x)=1\text{ for }x\leq 0\quad \varphi(x)=0\text{ for }x>1.
\end{equation}
For $R,m\in (0,\infty)$ we let
\begin{equation}
  \varphi_{R,m}(x)=\varphi\left(\frac{x-R}{m}\right).
\end{equation}
Using the mean value theorem we can easily show that there is a constant, $C,$
independent of $\gamma, R,m$ so that
\begin{equation}\label{eqn251new.0.2}
  \bbr{\varphi_{R,m}}_{\WF,0,\gamma}\leq \frac{C\sqrt{R+m}}{m}.
\end{equation}
We let $m=R,$ so that $\lim_{R\to\infty}\bbr{\varphi_{R,m}}_{\WF,0,\gamma}=0.$

Define
\begin{equation}
\begin{split}
  u_{R}^0(x,t)&=\int\limits_{0}^T\int\limits_0^{\infty}k^b_s(x,y)\varphi_{R,R}(y)g(y,t-s)dyds\\
 u_{R}^{\infty}(x,t)&=\int\limits_{0}^T\int\limits_0^{\infty}k^b_s(x,y)(1-\varphi_{R,R}(y))g(y,t-s)dyds.
\end{split}
\end{equation}
For any fixed $R,$ it follows from Proposition~\ref{prop3.4newdcyinfty} that 
\begin{equation}\label{eqn253new0.2}
  \lim_{x\to\infty}[|\pa_xu_R^0(x,t)|+|x\pa_x^2u_R^0(x,t)|]=0.
\end{equation}
Given $\epsilon>0$ we can choose $R$ so that
\begin{equation}\label{eqn254new1.0}
  \|(1-\varphi_{R,R}(y))g\|_{L^{\infty}}<\epsilon.
\end{equation}
Fix a  $0<\gamma'<\gamma.$ Applying~\eqref{leibfrm1}
with~\eqref{eqn251new.0.2} and~\eqref{eqn254new1.0}
along with Lemma~\ref{2gamslem} we see that
\begin{equation}
  \lim_{R\to\infty}\|(1-\varphi_{R,R}(y))g\|_{\WF,0,\gamma'}=0.
\end{equation}
It now follows from~\eqref{2nddrvest1} and~\eqref{eqn108.1}, that for a possibly
larger $R,$ we have the estimate:
\begin{equation}\label{eqn256new.1.1}
  |\pa_xu_{R}^{\infty}(x,t)|+|x\pa_x^2u_{R}^{\infty}(x,t)|\leq \epsilon
\end{equation}
Combining this with~\eqref{eqn253new0.2} we easily complete the proof of~\eqref{eqn193.1.5}.

To finish the spatial estimate, we need only show that $x\pa_x^2u$ is H\"older
continuous. As before, the estimate~\eqref{2nddrvest1} implies that for any
$c<1,$ there is a constant $C$ so that 
\begin{equation}
  x_1<cx_2\Longrightarrow |x_2\pa_x^2u(x_2,t)-x_1\pa_x^2u(x_1,t)|\leq
C\|g\|_{\WF,0,\gamma}|\sqrt{x_2}-\sqrt{x_1}|^{\gamma}.
\end{equation}
We are left to consider pairs $x_1,x_2$ with
\begin{equation}\label{ratbnd0}
  cx_2<x_1<x_2 .
\end{equation}
 Since we have already established the H\"older continuity of the
first derivative, it suffices to show that $L_bu(x,t)$ is Holder
continuous, which technically, is a little easier.

This is a rather delicate estimate; we need to decompose the integral
expression for $L_bu(x_1,t)-L_bu(x_2,t)$ as in~\eqref{dlctest0}. We use the
notation introduced there, with $J=[\alpha,\beta],$ etc.
\begin{multline}\label{dlctest1}
  L_bu(x_2,t)-L_bu(x_1,t)=\\
\int\limits_{0}^t\Bigg[\int\limits_{J}L_bk^b_s(x_2,y)(g(y,t-s)-g(x_2,t-s))dy-\\
\int\limits_{J}L_bk^b_s(x_1,y)(g(y,t-s)-g(x_1,t-s))dy-\\
\int\limits_{J^c}L_bk^b_s(x_2,y)(g(x_2,t-s)-g(x_1,t-s))dy+\\
\int\limits_{J^c}(L_bk^b_s(x_2,y)-L_bk^b_s(x_1,y))(g(y,t-s)-g(x_1,t-s))dy\Bigg]ds\\
=I_1+I_2+I_3+I_4.\phantom{lllllllllllllllllllllllllllllllllllllll}
\end{multline}
In this formula the operator $L_b$ acts in the $x$-variable.  The justification
for this formula is essentially identical to that given for~\eqref{eqn188.0.4}.

We begin by estimating $I_3.$ For this purpose we observe that, for $t>0$ we
have:
\begin{equation}
  L_{b,x}k^b_t(x,y)=\pa_t k^b_t(x,y)=L^t_{b,y}k^b_t(x,y).
\end{equation}
The operator $L^t_{b,y}=\pa_y(\pa_yy-b),$ so we can perform the $y$-integral to
obtain that
\begin{equation}
  I_3=\int\limits_{0}^t\left[(\pa_yy-b)k^b_s(x_2,\alpha)-(\pa_yy-b)k^b_s(x_2,\beta)\right]
(g(x_2,t-s)-g(x_1,t-s))ds.
\end{equation}
As usual we use the estimate
\begin{equation}
  |g(x_2,t-s)-g(x_1,t-s)|\leq 2|\sqrt{x_2}-\sqrt{x_1}|^{\gamma}.
\end{equation}
Lemma~\ref{lemB} therefore completes this step;
it shows that there is a constant $C_b$ so that
\begin{equation}
  |I_3|\leq C_b\|g\|_{\WF,0,\gamma}|\sqrt{x_1}-\sqrt{x_2}|^{\gamma}.
\end{equation}

We now turn to the compactly supported terms $I_1$ and $I_2.$ These terms are
estimated by
\begin{equation}
  \begin{split}
     &\|g\|_{\WF,0,\gamma}\int\limits_0^{t}\int\limits_{\alpha}^{\beta}
|L_bk^b_s(x_2,y)||\sqrt{y}-\sqrt{x_2}|^{\gamma}dyds\\
 &\|g\|_{\WF,0,\gamma}\int\limits_0^{t}\int\limits_{\alpha}^{\beta}
|L_bk^b_s(x_2,y)||\sqrt{y}-\sqrt{x_1}|^{\gamma}dyds.
  \end{split}
\end{equation}
The needed bounds are given in Lemma~\ref{lemC}.
This lemma shows that the terms $I_1$ and $I_2$ are estimated by
\begin{equation}\label{I1I2I4bnd}
  C_b\|g\|_{\WF,0,\gamma}|\sqrt{x_2}-\sqrt{x_1}|^{\gamma}.
\end{equation}
This leaves only the non-compact term, $I_4.$ Recall that
\begin{equation}
  I_4=\int\limits_0^t\int\limits_{J^c}(L_bk^b_s(x_2,y)-L_bk^b_s(x_1,y))(g(y,t-s)-g(x_1,t-s))dyds,
\end{equation}
and that $J^c=[0,\alpha)\cup (\beta,\infty).$ We use~\eqref{mstbscest2} to
estimate $|g(y,t-s)-g(x_1,t-s)|,$ hence Lemma~\ref{lemD} completes this case.
Using Lemma~\ref{lemD} we see that $I_4$ also satisfies the bound
in~\eqref{I1I2I4bnd}, which therefore completes the proof of the spatial part
of the H\"older estimate. To complete the $k=0$ case all that remains is to
estimate $|L_bu(x,t_1)-L_bu(x,t_2)|.$

The time estimate begins very much as the estimate for $|v(x,t_2)-v(x,t_1)|;$ we first
show that 
\begin{equation}\label{lbut0est}
  |L_bu(x,t)|\leq C_b\|g\|_{\WF,0,\gamma}t^{\frac{\gamma}{2}}.
\end{equation}
This implies that for any $M>1,$ there is a $C_{M,b}$ so that if $t_2>Mt_1,$
then
\begin{equation}
  |L_bu(x,t_2)-L_bu(x,t_1)|\leq C_{M,b}\|g\|_{\WF,0,\gamma}|t_2-t_1|^{\frac{\gamma}{2}},
\end{equation}
which leaves only case that $1<t_2/t_1<M.$  

To prove~\eqref{lbut0est} we use Lemma~\ref{lem2new} and Lemma~\ref{lem25new}.
The estimate in~\eqref{eqn108.1} shows that to prove~\eqref{lbut0est} it
suffices to show that
\begin{equation}\label{lbut0est1}
  |x\pa_x^2u(x,t)|\leq C_b\|g\|_{\WF,0,\gamma}t^{\frac{\gamma}{2}}.
\end{equation}
To prove this we write
\begin{equation}
  x\pa_x^2u(x,t)=\int\limits_0^t\int\limits_0^{\infty}
x\pa_x^2k_s(x,y)[g(y,t-s)-g(x,t-s)]dyds,
\end{equation}
which implies that
\begin{equation}
\begin{split}
  |x\pa_x^2u(x,t)|&\leq 2\|g\|_{\WF,0,\gamma}
\int\limits_0^t\int\limits_0^{\infty}|x\pa_x^2k^b_s(x,y)||\sqrt{x}-\sqrt{y}|^{\gamma}dyds\\
&\leq C\|g\|_{\WF,0,\gamma}
\int\limits_0^ts^{\frac{\gamma}{2}-1}\left(\frac{x/s}{1+x/s}\right)ds.
\end{split}
\end{equation}
The second line follows from Lemma~\ref{lem25new}; an elementary argument shows
that the last integral is bounded by a constant times $t^{\frac{\gamma}{2}},$
completing the proof of~\eqref{lbut0est}.

 To complete the time estimate we need to show that, for $M>1,$ there is
a constant $C_{M,b},$ so that $x_1<x_2<Mx_1$ implies that
\begin{equation}\label{lsttest}
  |L_bu(x,t_2)-L_bu(x,t_1)|\leq C_{M,b}\|g\|_{\WF,0,\gamma}|t_2-t_1|^{\frac{\gamma}{2}}.
\end{equation}

The proof of~\eqref{lsttest}, for the remaining cases, is broken into several
parts, where we observe that, for $t_1<t_2<2t_1$ we have:
\begin{multline}\label{1d2dertest}
  L_bu(x,t_2)-L_bu(x,t_1)=\int\limits_{0}^{t_2-t_1}\int\limits_{0}^{\infty}
L_bk^b_{s}(x,y)[g(y,t_2-s)-g(y,t_1-s)]dyds+\\
\int\limits_{0}^{2t_1-t_2}\int\limits_{0}^{\infty}
[L_bk^b_{t_2-s}(x,y)-L_bk^b_{t_1-s}(x,y)](g(y,s)-g(x,s))dyds+\\
\int\limits_{2t_1-t_2}^{t_1}\int\limits_{0}^{\infty}
L_bk^b_{t_2-s}(x,y)(g(y,s)-g(x,s))dyds.
\end{multline}
We denote these terms by $I_1, I_2$ and $I_3.$

We start by estimating $I_1,$ which we split into two parts,
$I_{11}-I_{12};$  each part we rewrite as:
\begin{equation}
  I_{1j}=\int\limits_{0}^{t_2-t_1}\int\limits_{0}^{\infty}
L_bk^b_{s}(x,y)[g(x,t_j-s)-g(y,t_j-s)]dyds\quad j=1,2.
\end{equation}
Indeed this really explains the meaning of this term as a convergent integral. 
These  are estimated, using the same argument, after we employ the estimate:
\begin{equation}\label{gesthol1}
|(g(y,t_j-s)-g(x,t_j-s))|\leq
2\|g\|_{\WF,0,\gamma}||\sqrt{x}-\sqrt{y}|^{\gamma}.
\end{equation}
This shows, using Lemma~\ref{lem25new}, that
\begin{equation}
\begin{split}
  |I_{1j}|&\leq 2\|g\|_{\WF,0,\gamma}|\int\limits_{0}^{t_2-t_1}
\int\limits_{0}^{\infty}|L_bk_s^b(x,y)||\sqrt{x}-\sqrt{y}|^{\gamma}dyds\\
&\leq C \|g\|_{\WF,0,\gamma}|\int\limits_{0}^{t_2-t_1}
s^{\frac{\gamma}{2}-1}\left(\frac{x/s}{1+x/s}\right)ds.
\end{split}
\end{equation}
An elementary argument now applies to show that this is bounded by 
$$C\|g\|_{\WF,0,\gamma}|t_2-t_1|^{\frac{\gamma}{2}}.$$ 
Essentially the same argument works to estimate $I_3,$ which,
using~\eqref{gesthol1}, satisfies:
\begin{equation}
\begin{split}
  |I_3|&\leq 2\|g\|_{\WF,0,\gamma}\int\limits_{2t_1-t_2}^{t_1}\int\limits_{0}^{\infty}
|L_bk^b_{t_2-s}(x,y)||\sqrt{x}-\sqrt{y}|^{\gamma}dyds\\
&=2\|g\|_{\WF,0,\gamma}\int\limits_{t_2-t_1}^{2(t_2-t_1)}\int\limits_{0}^{\infty}
|L_bk^b_{s}(x,y)||\sqrt{x}-\sqrt{y}|^{\gamma}dyds.
\end{split}
\end{equation}
The last line is again estimated using Lemma~\ref{lem25new}.

To complete the proof in the $k=0$ case all that remains is to estimate $I_2.$
This term is bounded by applying Lemma~\ref{lemH}.  This completes the proof of
the H\"older estimates for $L_bu(x,t)$ in the $k=0$ case. As
\begin{equation}
  \pa_t u=L_bu+g,
\end{equation}
the estimates on $\pa_tu$ are now an immediate
consequence. Using equations~\eqref{eqn193.1.4} and~\eqref{eqn193.1.5} we
apply Lemma~\ref{lem4.2.2} to conclude that
$u\in\cC^{0,2+\gamma}_{\WF}(\bbR_+\times [0,T]),$ which completes the proof
of~\eqref{bscest3} in the $k=0$ case.  

For the $k>0$ cases, we need to add the assumption that $\supp
g\subset \{(x,t):\: x\in [0,R]\}.$ To prove the higher order estimates
we use Lemma~\ref{lem3.2new}, which is an extension of Corollaries 7.6
and 7.7 of~\cite{WF1d}, and Proposition~\ref{wtedest_pr} to reduce the
higher order estimates to the $k=0$ case and elementary estimates for
functions in $\cC^{k,\gamma}_{\WF}(\bbR\times [0,T]).$

All that remains is to consider what happens as $b\to 0.$ The estimates proved
above hold uniformly for $b> 0,$ with uniform bounds on the constants $C_b$ for
$b$ in bounded subsets of $[0,\infty).$ If
$g\in\cC^{0,\gamma}_{\WF}(\bbR_+\times [0,T]),$ then the solutions $u^b$ are
uniformly bounded in $\cC^{0,2+\gamma}_{\WF}(\bbR_+\times [0,T]).$
Proposition~\ref{prop4.1new} implies that if $0<\gamma'<\gamma,$ then there is
a subsequence $<u^{b_n}>,$ with $b_n\to 0,$ that converges to some $u^{*}$ in
$\cC^{0,2+\gamma'}_{\WF}(\bbR_+\times [0,T]).$ Since $u^{*}$ solves
\begin{equation}
  (\pa_t-L_0)u^*=g\text{ and }u^*(x,0)=0,
\end{equation}
the uniqueness of the solution implies that in fact $u^{*}=u^0,$ and that
\begin{equation}
  \lim_{b\to 0^+}u^b=u^0\text{ in }\cC^{0,2+\gamma'}_{\WF}(\bbR_+\times [0,T]).
\end{equation}
We can therefore take limits in the estimates satisfied by $u^b$ for $b>0,$ to
conclude that, in fact $u^0\in\cC^{0,2+\gamma}_{\WF}(\bbR_+\times [0,T]),$ and
satisfies~\eqref{bscest3}, with $k=0.$ 
The higher order estimates for the $b=0$ case follow from this argument
and~\eqref{eqn217.0.4}.  

As remarked above the final statements, about convergence to $0,$ as $t\to
0^+,$ with respect to the $\cC^{k,2+\gamma'}$-norms follow easily from these
estimates, Proposition~\ref{htcmpincl}, and Lemma~\ref{lem8.0.6.nu}.  This
completes the proof of Proposition~\ref{prop1}.
\end{proof} 

\section{Properties of the Resolvent Operator}
We conclude this section by proving the estimates for $R(\mu)f$ stated in
Proposition~\ref{prop8.0.4.00}
\begin{proof}[Proof of Proposition~\ref{prop8.0.4.00}]
  As in the proofs of the previous results, we begin by establishing
  these results for the $k=0$ case, and arbitrary $0<b.$ The cases of
  arbitrary $k\in\bbN$ are obtained using Lemma~\ref{lem3.1new.0}.  As
  noted earlier, no assumption about the support of the data is needed
  for the resolvent operator, since we do not have to estimate time
  derivatives. Hence we only require~\eqref{comfrm1} with $q=0.$ We
  fix a $0<\phi \leq\frac{\pi}{2}.$

We begin by showing that if $f\in\cC^{0,\gamma}_{\WF}(\bbR_+),$ then
$R(\mu)f\in\cC^{2}_{\WF}(\bbR_+).$  First we see that Lemma~\ref{lem9.1.3.00}
implies that
\begin{equation}
  \|R(\mu) f\|_{L^{\infty}}\leq C_{b,\phi}\frac{\|f\|_{L^{\infty}}}{|\mu|}.
\end{equation}
The argument in the proof of Proposition~\ref{prop2} between~\eqref{frstdrvCP3}
and~\eqref{eqn9.62.00} applies \emph{mutatis mutandis} to show that there is a
constant $C_{b,\phi},$ so that if $t\in S_{\phi},$ then
\begin{equation}
  |v(x,t)-v(y,t)|\leq C_{b,\phi}\|f\|_{\WF,0,\gamma}|\sqrt{x}-\sqrt{y}|^{\gamma}.
\end{equation}
Integrating this shows that there is a constant $C_{b,\alpha,\gamma}$ so that
\begin{equation}
\int\limits_{0}^{\infty}\bbr{v(\cdot,t)}_{\WF,0,\gamma}|e^{-\mu t}dt|\leq
C_{b,\alpha,\gamma}\frac{\|f\|_{\WF,0,\gamma}}{|\mu|},
\end{equation}
completing the proof of~\eqref{eqn8.13.00}.

Next observe that, for $t\in S_0,$ we have the formul{\ae}:
\begin{equation}
\begin{split}
  \pa_x v(x,t)&=\int\limits_{0}^{\infty}
\pa_xk^b_t(x,y)(f(y)-f(x))dy\\
  x\pa^2_x v(x,t)&=\int\limits_{0}^{\infty}
x\pa_x^2k^b_t(x,y)(f(y)-f(x))dy.
\end{split}
\end{equation}
Using the first formula and the estimate in Lemma~\ref{lem2new} we see that,
for $t\in S_{\phi}$
\begin{equation}
  |\pa_x v(x,t)|\leq
  C_{b,\phi}\|f\|_{\WF,0,\gamma}\frac{|t|^{\frac{\gamma}{2}-1}}
{1+\sqrt{x/|t|}}.
\end{equation}
If $|\arg\mu|<\frac{\pi}{2}+\phi,$ then, by choosing an appropriate ray in the
right half plane we see that there is a positive constant $m_{\phi},$ so that
\begin{equation}
\begin{split}
  |\pa_x R(\mu)f(x)|&\leq  C_{b,\phi}\|f\|_{\WF,0,\gamma}
\int\limits_{0}^{\infty}\frac{e^{-m_{\phi}|\mu|s}s^{\frac{\gamma}{2}-1}ds}
{1+\sqrt{x/s}}\\
&\leq
C'_{b,\phi}\frac{\|f\|_{\WF,0,\gamma}}{(m_{\phi}|\mu|)^{\frac{\gamma}{2}}}.
\end{split}
\end{equation}

Using the estimate in Lemma~\ref{lem25new} we can show that
\begin{equation}
\begin{split}
  |x\pa^2_x R(\mu)f(x)|&\leq  C_{b,\phi}\|f\|_{\WF,0,\gamma}
\int\limits_{0}^{\infty}\frac{e^{-m_{\phi}|\mu|s}xs^{\frac{\gamma}{2}-1}ds}
{x+s}\\
&\leq
C'_{b,\phi}\frac{\|f\|_{\WF,0,\gamma}}{(m_{\phi}|\mu|)^{\frac{\gamma}{2}}}.
\end{split}
\end{equation}
It is useful to note that by splitting this integral into a part from $0$ to
$x$ and the rest, we can also show that
\begin{equation}\label{eqn9.142.00}
  |x\pa^2_x R(\mu)f(x)|\leq C'_{b,\phi}\|f\|_{\WF,0,\gamma}x^{\frac{\gamma}{2}}.
\end{equation}

These estimates show that $R(\mu)f\in\cC^{2}_{\WF}(\bbR_+)$ and, by
integrating~\eqref{eqn9.142.00}, establish the H\"older estimate on the first
derivative:
\begin{equation}
  \bbr{\pa_x R(\mu) f}_{\WF,0,\gamma}\leq C_{b,\phi}\|f\|_{\WF,0,\gamma}.
\end{equation}
With these estimates in hand, we can integrate by parts to establish that
\begin{equation}
  (\mu-L_b)R(\mu) f=f\text{ for }\mu\in\bbC\setminus (-\infty,0]. 
\end{equation}
Below we show that $R(\mu)f\in\cC^{0,2+\gamma}_{\WF}(\bbR_+).$ By the open mapping
theorem,  to show that
$R(\mu)$ is also a left inverse for $(\mu-L_b)$ is suffices to show that the
null-space of $L_b-\mu$ is trivial. For $\mu\in S_0,$ this follows
immediately from the estimate in~\eqref{bscest0}, and the uniqueness of the
solution to the Cauchy problem. If, for some $\mu\in S_0$ there were a
solution $f\in\cC^{0,2+\gamma}_{\WF}(\bbR_+)$ to $L_bf=\mu f,$ then the
solution to the Cauchy problem with this initial data would
be $v(x,t)=e^{\mu t}f.$ This solution grows exponentially,
contradicting~\eqref{bscest0}. Thus for $\mu\in S_0,$ and
$f\in\cC^{0,2+\gamma}_{\WF}(\bbR_+)$ we also have the identity
\begin{equation}
  R(\mu)(\mu-L_b) f=f.
\end{equation}
The permanence of functional relations implies that this holds for
$\mu\in\bbC\setminus (-\infty,0].$

We can also apply the observation in~\eqref{lrgratioest}, along with the
estimate in~\eqref{eqn9.142.00}, to see that if any $0<c<1$ is fixed, then
there is a $C_{b,c,\phi}$ so that if $y<cx,$ then
\begin{equation}\label{eqn9.149.00}
  |x\pa_x^2R(\mu) f(x)-y\pa_x^2R(\mu) f(y)|\leq
  C_{b,c,\phi}\|f\|_{\WF,0,\gamma}
|\sqrt{x}-\sqrt{y}|^{\gamma}.
\end{equation}
To complete the proof that $R(\mu)f\in\cC^{0,2+\gamma}_{\WF}(\bbR_+)$ we only
need to show that there is a $0<c<1,$ so that a similar estimate holds for
$cx<y<x.$  

This is accomplished, exactly as in the proof of Proposition~\ref{prop1}:
it suffices to estimate $\bbr{L_bR(\mu) f}_{\WF,0,\gamma},$ and use
a decomposition like that given in~\eqref{dlctest1}:
\begin{multline}
  L_bR(\mu)f(x_2)-L_bR(\mu)f(x_1)=\\
\int\limits_{0}^{\infty}e^{\mu s \eit}\Bigg[
\int\limits_{J}L_bk^b_{s\eit}(x_2,y)(f(y)-f(x_2))dy-\\
\int\limits_{J}L_bk^b_{s\eit}(x_1,y)(f(y)-f(x_1))dy-\\
\int\limits_{J^c}L_bk^b_{s\eit}(x_2,y)(f(x_2)-f(x_1))dy+\\
\int\limits_{J^c}[L_bk^b_{s\eit}(x_2,y)-L_bk^b_{s\eit}(x_2,y)](f(y)-f(x_1))dy\Bigg]\eit
ds\\
=I_1+I_2+I_3+I_4.
\end{multline}
Here we  select $\theta\in (-\frac{\pi}{2},\frac{\pi}{2}),$ so that
\begin{equation}
  |\arg\mu\eit|<\frac{\pi}{2}.
\end{equation}
Fix a positive constant $1/3<c<1,$ and assume that $cx_2<x_1<x_2,$ so that we
can apply Lemmas~\ref{lemB}--~\ref{lemD} as in the earlier argument.  We use
the fact that $L_{b,x}k^b_t(x,y)= L_{b,y}^tk^b_t(x,y),$ to perform the
$y$-integral in $I_3.$ As the estimate in Lemma~\ref{lemB} holds uniformly for
all $|t|,$ it applies to show that
\begin{equation}
  |I_3|\leq C_{b,\phi}\|f\|_{\WF,0,\gamma}|\sqrt{x_2}-\sqrt{x_1}|^{\gamma}
\end{equation}
The estimates in Lemmas~\ref{lemC} and~\eqref{lemD} also apply uniformly, for
all $|t|,$ and show that $|I_1|,$ $|I_2|,$ and $|I_4|$ each satisfy an estimate
of the same form; thereby completing the proof that
\begin{equation}
  \bbr{x\pa_x^2R(\mu)f}_{\WF,0,\gamma}\leq C_{b,\phi}\|f\|_{\WF,0,\gamma}.
\end{equation}

This completes the $k=0$ case for $b>0.$ The case of $b=0$ is
  obtained by using the fact that the constants in the estimates are uniformly
  bounded for $0<b<1,$ and Proposition~\ref{prop4.1new}.  This shows that if we
  let $f\in\cC^{0,\gamma}_{\WF}(\bbR_+)$ and set $w_b=(\mu-L_b)^{-1}f$ for
  $0<b,$ then $\{w_b:\: 0<b<1\}$ are uniformly bounded in
  $\cC^{0,2+\gamma}_{\WF}(\bbR_+).$ Proposition~\ref{prop4.1new} shows that for
  any $0<\tgamma<\gamma,$ this sequence has a subsequence that converges in 
$\cC^{0,2+\tgamma}_{\WF}(\bbR_+).$ Any such limit
$w_0\in\cC^{0,2+\gamma}_{\WF}(\bbR_+)$ and satisfies
\begin{equation}
  (\mu-L_0)w_0=f.
\end{equation}
The uniqueness result, Proposition~\ref{uniquenesselliptic}, shows that $w_0$
is uniquely determined, which implies that $\{w_b\}$ itself converges in
$\cC^{0,2+\tgamma}_{\WF}(\bbR_+)$ to $w_0,$ and that $w_0$ therefore satisfies
the estimates in the statement of the proposition.  Finally we use
Lemma~\ref{lem3.1new.0} to commute the $x$-derivatives past $k^b_t(x,y)$ and
follow the argument above to establish this theorem for arbitrary $k\in\bbN.$
\end{proof}

\begin{remark} The solution to the Cauchy problem can be expressed as contour
  integral involving $R(\mu):$
  \begin{equation}
    v(x,t)=\frac{1}{2\pi i}\int\limits_{\Gamma_{\alpha,R}}e^{\mu t}R(\mu)fd\mu,
  \end{equation}
where\index{$\Gamma_{\alpha,R}$}
\begin{equation}\label{gamalph1}
  \Gamma_{\alpha,R}:=-b\{\mu:|\arg\mu|<\pi-\alpha\text{ and }|\mu|>R\}.
\end{equation}
Here the $-$ sign indicates that $\Gamma_{\alpha}$ is taken with the opposite
orientation to that it inherits as the boundary of the region on the right hand
side in~\eqref{gamalph1}. Using this formula we easily establish the analytic
continuation  of $v$ to $t\in H_+$ as well as estimates of the form
\begin{equation}
  \|v(\cdot,t)\|_{\WF,k,2+\gamma}\leq C(t)\|f\|_{\WF,k,\gamma}.
\end{equation}
The constant $C(t)$ tends to infinity as $t\to 0$ at a rate that depends on $\gamma.$ As $e^{\mu t}$  is exponentially decreasing along $\Gamma_{\alpha,R},$ we can also estimate the time derivatives $\pa_t^jv(\cdot,t),$ for $t>0.$ 

\end{remark}
\chapter[H\"older Estimates in Corners]{H\"older Estimates for Higher Dimensional  Corner Models}\label{s.cormod}
The estimates proved in the previous chapter form a solid foundation for
proving analogous results in higher dimensions for model operators of the form
\begin{equation}\label{genmod0}
  L_{\bb,m}=\sum_{j=1}^n[x_j\pa_{x_j}^2+b_j\pa_{x_j}]+\sum_{k=1}^m\pa_{y_k}^2,
\end{equation}
here $\bb\in\overline{\bbR}_+^n.$  In this context we exploit the fact the
solution operator for $L_{\bb,m}$ is a product of solution operators for
1-dimensional problems. 

In 2-dimensions we can write
\begin{equation}
  u(x_1,x_2,t)-u(y_1,y_2,t)=[u(x_1,x_2,t)-u(x_1,y_2,t)]+[u(x_1,y_2,y)-u(y_1,y_2,t)],
\end{equation}
and in  $n>2$ dimensions we rewrite
$u(\bx,t)-u(\by,t)$ as
\begin{equation}\label{1varattme}
  u(\bx,t)-u(\by,t)=\sum_{j=0}^{n-1}[u(\bx'_j,x_{j+1},\by''_j,t)-u(\bx'_{j},y_{j+1},\by''_j,t)],
\end{equation}
where:
\begin{equation}\label{1varattme.1}
\begin{split}
  &\bx'_j=(x_1,\dots,x_j)\text{ if }1\leq j\text{ and }\emptyset\text{ if }j\leq
  0,\\
&\bx''_j=(x_{j+2},\dots,x_n)\text{ if }j< n-1\text{ and }\emptyset\text{ if
}j\geq n-1.
  \end{split}
\end{equation}
In this way we are reduced to estimating these differences 1-variable at a time, which,
in light of Lemma~\ref{lem1} suffices.\index{$\bx'_j$} \index{$\bx''_j$} 

In the proofs of the 1-dimensional estimates the only facts about the data we
use are contained in the estimates in~\eqref{mstbscest}
and~\eqref{mstbscest2}. This makes it possible to use these arguments to prove
estimates in higher dimensions ``one variable at a time.'' The only other fact
we use is that if $f(x_1,\dots,x_n)$ is an absolutely integrable function, such
that for some $j$ we know that, for any $\bx'_j,\bx''_j,$ the 1-dimensional
integral
\begin{equation}
  \int\limits_{0}^{\infty} f(\bx'_j,z_{j+1},\bx''_{j})dz_j=0,
\end{equation}
then Fubini's theorem implies that
\begin{equation}
    \int\limits_{0}^{\infty}\cdots \int\limits_{0}^{\infty} f(\bz)dz_1\cdots
    dz_n=0.
\end{equation}
While we cannot simply quote the 1-dimensional estimates, using formul{\ae}
like that in~\eqref{1varattme}, we can reduce the proof of an estimate in
higher dimensions to the estimation of a product of 1-dimensional
integrals. These integrals are in turn estimated in the lemmas stated here in the
previous chapter.

Using the ``one-variable-at-a-time'' approach we prove the higher
dimensional estimates in several stages; we begin by considering the
``pure corner'' case where $m=0,$ and then turn to the Euclidean case,
where $n=0.$ The Euclidean case is of course classical. In the next
chapter we state the results we need for the case of general $(n,m)$
and the estimates on the 1-dimensional solution kernel needed to prove
them. Finally, in Chapter~\ref{s.genmod} we do the general case, where
$n$ and $m$ can assume arbitrary non-negative values.  
\index{one-variable-at-a-time}

We first consider the homogeneous Cauchy problem
\begin{equation}
  L_{\bb,0}v(\bx,t)=0\text{ in }\bbR_+^n\times (0,\infty)\text{ and }
v(\bx,0)=f(\bx).
\end{equation}
Here $\bb$ is a vector in $\bbR_+^n.$ If $f$ is bounded and continuous, then
 the unique bounded solution is given by
\begin{equation}\label{slnhmcpn0}
  v(\bx,t)=\int\limits_{0}^{\infty}\cdots\int\limits_{0}^{\infty}
\prod\limits_{j=1}^nk^{b_j}_t(x_j,z_j)f(\bz)d\bz,
\end{equation}
from which it is clear that
\begin{equation}
  |v(\bx,t)|\leq \|f\|_{\cC^0(\bbR_+^n)}.
\end{equation}
For fixed $\bx,$ $v(\bx,t)$ extends analytically in $t$ to define a
function in $S_0.$

We next turn to estimating the solution, $u,$ of the inhomogeneous problem:
\begin{equation}\label{n0dmdlb0}
\left[\pa_t-  \sum_{j=1}^N[x_j\pa_{x_j}^2+b_j\pa_{x_j}]\right]u=g,
\end{equation}
vanishing at $t=0.$ Proposition~\ref{lem3.4new0.0} shows that the unique
bounded solution is given by the integral:
\begin{equation}\label{n0slninhom00}
  u(\bx,t)=\int\limits_{0}^{t}\int\limits_{0}^{\infty}\cdots\int\limits_{0}^{\infty}
\prod_{j=1}^nk^{b_j}_s(x_j,z_j)g(\bz,t-s)d\bz ds.
\end{equation}

It is quite easy to see that, for any $k\in\bbN_0$ and $0<\gamma<1,$
the operator $\pa_t-L_{\bb,0}$ maps data with compact support in
$C^{k,2+\gamma}_{\WF}(\bbR_+^n\times [0,T])$ to
$C^{k,\gamma}_{\WF}(\bbR_+^n\times [0,T]).$ Our aim, once again, is to
prove that, for data with compact support in
$C^{k,\gamma}_{\WF}(\bbR_+^n\times [0,T]),$ the solution belongs to
$C^{k,2+\gamma}_{\WF}(\bbR_+^n\times [0,T]).$ As in the 1-dimensional
case, when $k=0$ we do not need to assume that the data has compact
support.

\section{The Cauchy Problem}
We begin with the somewhat simpler homogeneous Cauchy problem.
\begin{proposition}\label{prop3} Fix $k\in\bbN_0,$ $0<R,$ and $\bb\in\bbR_+^n.$
  Let $f\in\cC^{k,\gamma}_{\WF}(\bbR_+^n),$ and let $v$ be the unique solution,
  given in~\eqref{slnhmcpn0}, to
  \begin{equation}
    (\pa_t-L_{\bb,0})v=0\quad v(\bx,0)=f(\bx).
  \end{equation}
  If $k>0,$ then assume that $f$ is supported in
  $$B_R^+(\bzero)=\{\bx:\:\bx\in\bbR_+^N\text{ and }\|\bx\|\leq R\}.$$
  For $0<\gamma<1$ there a constant $C_{k,\gamma,b,R}$ so
  that\index{$B_R^+$}
  \begin{equation}\label{bscestn00}
    \|v\|_{\WF,k,\gamma}\leq C_{k,\gamma,b,R}\|f\|_{\WF,k,\gamma}.
  \end{equation}
and, if $f\in\cC^{k,2+\gamma}_{\WF}(\bbR_+^n),$ then
\begin{equation}\label{bscestn02}
    \|v\|_{\WF,k,2+\gamma}\leq C_{k,\gamma,b,R}\|f\|_{\WF,k,2+\gamma}.
  \end{equation}
For fixed $\gamma,$ the constants $C_{k,\gamma,b,R}$ are uniformly bounded for
$0\leq b\leq B.$  If $k=0,$ then the constants are independent of $R.$
\end{proposition}
\begin{proof}
  Suppose that we have proved the estimates above with constants $C$ which, for
  any $B,$ are uniformly bounded for $0<b_j<B.$ As shown in the  proof of
  Proposition~\ref{prop2}, the case where $b_j=0,$ for one or more
  values $j,$ can treated by choosing a sequence $<\bb_n>$ so that
\begin{equation}
  b_{n,j}>0\text{ for all }n\text{ and }\lim_{n\to\infty}b_{n,j}=b_j.
\end{equation}
We let $v_{\bb_n}(\bx,t)$ denote the solutions with the given initial data $f.$
Given that the estimates in the lemma have been proved for each $\bb_n,$
Proposition~\ref{prop4.1new} shows that
the sequence $<v_{\bb_n}>,$ contains subsequences convergent with
respect to the topology on $\cC_{\WF}^{0,\gamma'},$ for any $0<\gamma'<\gamma.$
If $t>\epsilon>0,$ then these solutions also converge uniformly in $\cC^m,$ for any
$m>0.$ Hence the limit satisfies the limiting diffusion equation with the given
initial data; the uniqueness of such solutions shows that any convergent
subsequence has the same limit. Thus $<v_{\bb_n}>$ itself converges in
$\cC_{\WF}^{0,\gamma'}$ to $v_{\bb},$ the solution in the limiting case. This
implies that $v_{\bb}$ also satisfies the estimates in the proposition. This
reasoning applies equally well to all the function spaces under consideration. Thus
it suffices to consider the case where $b_j>0$ for $j=1,\dots,n,$ which we
henceforth assume.  In the sequel we use $C$ to denote positive constants that
may depend on $\gamma$ and $\bb,$ which are uniformly bounded so long as
$0<\gamma<1$ is fixed and, for $j=1,\dots,n,$ $0<b_j<\leq B,$  for any fixed $B.$

The solution is given by formula~\eqref{slnhmcpn0}. We observe that
\begin{multline}
  v(x_1,\dots,x_n,t)-v(y_1,\dots,y_n,t)=
v(x_1,\dots,x_n,t)-v(x_1,\dots,x_{n-1},y_n,t)+\\v(x_1,\dots,x_{n-1},y_n,t)-v(x_1,\dots,x_{n-2},y_{n-1},y_n,t)\\+\dots+
v(x_1,y_2\dots,y_n,t)-v(y_1,\dots,y_n,t).
\end{multline}
Hence it is enough to show that for each $1<k\leq n$  we have:
\begin{multline}
  |v(x_1,\dots,x_{k-1},x_k,y_{k+1},\dots,y_{n},t)-v(x_1,\dots,x_{k-1},y_{k},y_{k+1},\dots,y_{n},t)|
\leq \\
C\|f\|_{\WF,0,\gamma}\rho_s(\bx,\by)^{\gamma}.
\end{multline}
It suffices to assume that $\bx$ and $\by$ differ in exactly one
coordinate, which we can choose to be $n.$ For an $n$-vector $\bx$
we let
\begin{equation}
  \bx'=(x_1,\dots,x_{n-1}).
\end{equation}

The proof is simply a matter of recapitulating the steps in the
1-dimensional case, and showing how the $n$-dimensional case can be
reduced to this case. We first do the $k=0$ case, which does not
require additional hypotheses on $f,$ and then do the $k>0$ case
assuming that $f$ has bounded support.

The first step is to consider the special case
$v(\bx',x_n,t)-v(\bx',0,t).$ Using the fact that
\begin{equation}\label{yintis1}
  \int\limits_0^{\infty}k^b_t(x,y)dy=1,
\end{equation}
 we see that
\begin{multline}
 v(\bx',x_n,t)-v(\bx',0,t)=\int\limits_0^{\infty}\cdots\int\limits_0^{\infty}
\prod_{j=1}^{n-1}k_t^{b_j}(x_j,z_j)\times\\\left[(k_t^{b_n}(x_n,z_n)-k_t^{b_n}(0,z_n))(f(\bz',z_n)-
f(\bz',0))\right]dz_nd\bz'.
\end{multline}
This follows because, for any $\bz'$ we have:
\begin{equation}
  \int\limits_0^{\infty}
(k_t^{b_n}(x_n,z_n)-k_t^{b_n}(0,z_n))f(\bz',0))dz_n=0.
\end{equation}
Using the triangle inequality, the positivity of the kernels, and the obvious estimate:
\begin{equation}\label{ndholdest00}
  |f(\bx)-f(\by)|\leq 2\|f\|_{\WF,0,\gamma}\rho_s(\bx,\by)^{\gamma}.
\end{equation}
and~\eqref{yintis1}, we obtain the estimate
\begin{equation}
  |v(\bx',x_n,t)-v(\bx',0,t)|\leq
2\|f\|_{\WF,0,\gamma}\int\limits_0^{\infty}
|k_t^{b_n}(x_n,z_n)-k_t^{b_n}(0,z_n)|z_n^{\frac{\gamma}{2}}dz_n
\end{equation}
Lemma~\ref{lem1new} shows that integral is bounded by $Cx_n^{\frac{\gamma}{2}},$
showing, as before that for any $c<1,$ there is a $C$ so that if $y_n<cx_n,$
then
\begin{equation}\label{hmcphestn0}
  |v(\bx',x_n,t)-v(\bx',y_n,t)|\leq
C\|f\|_{\WF,0,\gamma}|\sqrt{x_n}-\sqrt{y_n}|^{\gamma}.  
\end{equation}

For the second step we show that
\begin{equation}\label{hmcp1drestn0}
  |\pa_{x_n}v(\bx',x_n,t)|\leq Cx_n^{\frac{\gamma}{2}-1}\|f\|_{\WF,0,\gamma}
\frac{\lambda^{1-\frac{\gamma}{2}}}{1+\lambda^{\frac
    12}}=Ct^{\frac{\gamma}{2}-1}
\frac{\|f\|_{\WF,0,\gamma}}{1+\lambda^{\frac 12}},
\end{equation}
where $\lambda=x_n/t.$ Taking advantage of the fact that, for all $x_n\geq 0$
and $t>0,$ we have 
\begin{equation}
  \int\limits_0^{\infty}\pa_{x_n}k^{b_n}_t(x_n,z_n)dz_n=0,
\end{equation}
it follows that
\begin{multline}
  \pa_{x_n}v(\bx',x_n,t)=\int\limits_0^{\infty}\cdots\int\limits_0^{\infty}
\prod_{j=1}^{n-1}k_t^{b_j}(x_j,z_j)\times\\
\int\limits_0^{\infty}\left[\pa_{x_n}k_t^{b_n}(x_n,z_n)(f(\bz',z_n)-
f(\bz',x_n))\right]dz_nd\bz'.
\end{multline}
As before it follows easily that
\begin{equation}
  |\pa_{x_n}v(\bx',x_n,t)|\leq C\|f\|_{\WF,0,\gamma}\int\limits_0^{\infty}
|\pa_{x_n}k_t^{b_n}(x_n,z_n)||\sqrt{z_n}-\sqrt{x_n}|^{\gamma}dz_n.
\end{equation}
An application of Lemma~\ref{lem2new} suffices to complete the proof
of~\eqref{hmcp1drestn0}. Integrating~\eqref{hmcp1drestn0}, we can now verify
that~\eqref{hmcphestn0} holds so long as $\lambda$ is bounded.

For the last step we fix $0<c<1,$ and consider $cx_n<y_n<x_n,$ and $\lambda\to
\infty.$ For this case we need to find an analogue of the rather complicated
formula in~\eqref{dlctest0}, which is again straightforward:
\begin{multline}\label{vdfn0}
  v(\bx',x_n,t)-v(\bx',y_n,t)=\int\limits_0^{\infty}\cdots\int\limits_0^{\infty}
\prod_{j=1}^{n-1}k_t^{b_j}(x_j,z_j)\times\\
\Bigg[\int\limits_{J}k_t^{b_n}(x_n,z_n)(f(\bz',z_n)-f(\bz',x_n))dz_n+\\
\int\limits_{J}k_t^{b_n}(y_n,z_n)(f(\bz',y_n)-f(\bz',z_n))dz_n+\\
\int\limits_{J^c}k_t^{b_n}(x_n,z_n)(f(\bz',y_n)-f(\bz',x_n))dz_n+\\
\int\limits_0^{\infty}k^{b_n}_t(x_n,z_n)(f(\bz',x_n)-f(\bz',y_n))dz_n+\\
\int\limits_{J^c}[k_t^{b_n}(x_n,z_n)-k_t^{b_n}(y_n,z_n)](f(\bz',z_n)-f(\bz',y_n))dz_n
\Bigg]d\bz'.
\end{multline}
Recall that $J=[\alpha,\beta],$ where
\begin{equation}
 \sqrt{\alpha}=\frac{3\sqrt{y_n}-\sqrt{x_n}}{2}\quad
\sqrt{\beta}=\frac{3\sqrt{x_n}-\sqrt{y_n}}{2}
\end{equation}
Using the triangle inequality repeatedly, and~\eqref{yintis1}, we see that each
term reduces to one appearing in the 1d-argument multiplied by
$\prod_{j=1}^{n-1}k_t^{b_j}(x_j,z_j),$ From the fact that
\begin{equation}
  |f(\bz',x)-f(\bz',y)|\leq 2\|f\|_{\WF,0,\gamma}|\sqrt{x}-\sqrt{y}|^{\gamma}.
\end{equation}
It follows immediately that the first four $z_n$-integrals contribute terms
bounded by a constant times
$\|f\|_{\WF,0,\gamma}|\sqrt{x_n}-\sqrt{y_n}|^{\gamma}.$ This leaves just the
last integral over $J^c.$ This term is estimated by
\begin{equation}
  2\|f\|_{\WF,0,\gamma}
\int\limits_{J^c}|k_t^{b_n}(x_n,z_n)-k_t^{b_n}(y_n,z_n)||\sqrt{z_n}-\sqrt{y_n}|^{\gamma}dz_n.
\end{equation}
For $c>1/9$ we may apply Lemma~\ref{lem3new} to this term, and the
estimate in~\eqref{hmcphestn0} follows once again. This completes the spatial
part of the $k=0$ case.

We now turn to the estimate of $v(\bx,t)-v(\bx,s);$ we begin with the case
$ct<s<t,$ for a $0<c<1.$ By definition we have:
\begin{multline}
  v(\bx,t)-v(\bx,s)=\int\limits_0^{\infty}\cdots\int\limits_0^{\infty}
\left[\prod_{j=1}^{n}k_t^{b_j}(x_j,z_j)-\prod_{j=1}^{n}k_s^{b_j}(x_j,z_j)\right]
f(\bz)d\bz
\end{multline}
The difference of products can be represented as a telescoping sum:
\begin{multline}\label{n0ktksdif}
  \prod_{j=1}^{n}k_t^{b_j}(x_j,z_j)-\prod_{j=1}^{n}k_s^{b_j}(x_j,z_j)=\\
\sum\limits_{l=1}^{n}\Bigg\{
\prod_{j=1}^{n-l}k_t^{b_j}(x_j,z_j)\times\prod_{j=n-l+2}^{n}k_s^{b_j}(x_j,z_j)\times\\
\left[k_t^{b_{n-l+1}}(x_{n-l+1},z_{n-l+1})-k_s^{b_{n-l+1}}(x_{n-l+1},z_{n-l+1})\right]\Bigg\},
\end{multline}
with the convention that, if $q<p,$ then $\prod_p^q=1.$ Recall that
\begin{equation}\begin{cases}
  &\bz'_l=(z_1,\dots,z_{n-l})\text{ for } 0\leq l\leq n-1\\
&\bz''_{l}=(z_{n-l+2},\dots,z_n)\text{ for } 2\leq l\leq n+1,
\end{cases}
\end{equation}
with $\bz'_l$ and $\bz''_l$ equal to the empty set outside the stated ranges.
For $1\leq l\leq n,$ we have:
\begin{multline}
  \int\limits_0^{\infty}\left[k_t^{b_{n-l+1}}(x_{n-l+1},z_{n-l+1})-k_s^{b_{n-l+1}}(x_{n-l+1},z_{n-l+1})\right]\times\\
f(\bz'_l,x_{n-l+1},\bz''_{l})dz_{n-l+1}=0
\end{multline}
Using these observations we can reexpress $v(\bx,t)-v(\bx,s)$ as
\begin{multline}\label{eqn10.36.1}
   v(\bx,t)-v(\bx,s)=\int\limits_0^{\infty}\cdots\int\limits_0^{\infty}\\
\sum\limits_{l=1}^{n}\Bigg\{
\prod_{j=1}^{n-l}k_t^{b_j}(x_j,z_j)\times\prod_{j=n-l+2}^{n}k_s^{b_j}(x_j,z_j)\times\\
\left[k_t^{b_{n-l+1}}(x_{n-l+1},z_{n-l+1})-k_s^{b_{n-l+1}}(x_{n-l+1},z_{n-l+1})\right]\times\\
[f(\bz)-f(\bz'_l,x_{n-l+1},\bz''_{l})]\Bigg\}d\bz
\end{multline}
Inserting absolute values, and using~\eqref{yintis1} repeatedly, we see that
\begin{multline}
   |v(\bx,t)-v(\bx,s)|\leq\\ 2\|f\|_{\WF,0,\gamma}\int\limits_0^{\infty}
\sum\limits_{l=1}^{n}\left|k_t^{b_{n-l+1}}(x_{n-l+1},z_{n-l+1})-k_s^{b_{n-l+1}}(x_{n-l+1},z_{n-l+1})\right|\\
\times
|\sqrt{z_{n-l+1}}-\sqrt{x_{n-l+1}}|^{\gamma}dz_{n-l+1}
\end{multline}
Applying Lemma~\ref{lem4new} it follows that for any $0<c<1,$ there is a $C,$ so
that if $ct<s<t,$ then
\begin{equation}\label{hmcptestn0}
  |v(\bx,t)-v(\bx,s)|\leq C\|f\|_{\WF,0\,\gamma}|t-s|^{\frac{\gamma}{2}}.
\end{equation}
To complete the proof of the proposition we need to consider only
$v(\bx,t)-v(\bx,0),$ using~\eqref{yintis1} this can be expressed as
\begin{equation}
  v(\bx,t)-v(\bx,0)=\int\limits_0^{\infty}\cdots\int\limits_0^{\infty}
\prod_{j=1}^{n}k_t^{b_j}(x_j,z_j)[f(\bz)-f(\bx)]d\bz.
\end{equation}
We rewrite
\begin{equation}
  f(\bz)-f(\bx)=\sum_{l=0}^{n-1}[f(\bz'_{l},\bx''_{l+1})-f(\bz'_{l+1},\bx''_{l+2})]
\end{equation}
Putting this expression into the integral above and repeatedly
using~\eqref{yintis1}, we obtain the estimate:
\begin{equation}
  |v(\bx,t)-v(\bx,0)|\leq 2\|f\|_{\WF,0,\gamma}
\sum_{j=1}^n\int\limits_0^{\infty}
k_t^{b_j}(x_j,z_j)|\sqrt{z_j}-\sqrt{x_j}|^{\gamma}dx_j.
\end{equation}
Applying Lemma~\ref{lem5new} shows that there is a constant $C$ so that
\begin{equation}
    |v(\bx,t)-v(\bx,0)|\leq C\|f\|_{\WF,0,\gamma}t^{\frac{\gamma}{2}}.
\end{equation}
As in the 1-dimensional case, this completes the proof that~\eqref{hmcptestn0}
holds for all $\bx, s, t,$ and thereby the proof of~\eqref{bscestn00} in the
$k=0.$ case. If we assume that $f$ is supported in $B_R^+(\bzero),$ then the
estimates in~\eqref{bscestn00} for $k>0$ follow from the $k=0$ case by
repeatedly applying Propositions~\ref{lem3.3new0.0}
and~\ref{wtedest_pr}. 

Many of the estimates needed to prove~\eqref{bscestn02} with $k=0$
follow from~\eqref{bscestn00}, Lemma~\ref{lem6.0.9} and applications
of Propositions~\ref{lem3.3new0.0} and~\ref{wtedest_pr}. We can use
these results to show that
\begin{equation}
  \|\nabla_{\bx}v\|_{\WF,0,\gamma}+\|\pa_tv\|_{\WF,0,\gamma}+\sum_{i=1}^n\|x_i\pa_{x_i}^2v\|_{\WF,0,\gamma}
\leq
C\|f\|_{\WF,0,2+\gamma}.
\end{equation}
To complete the proof in this case we need to similarly estimate the
derivatives
\begin{equation}
  \sqrt{x_ix_j}\pa_{x_i}\pa_{x_j}v
\end{equation}
in $\cC^{0,\gamma}_{\WF}(\bbR_+^n\times [0,\infty)).$ We can relabel so that
$i=1$ and $j=2.$ Using Proposition~\ref{lem3.3new0.0} we can express these
derivatives are
\begin{multline}\label{eqn10.45.1}
   \sqrt{x_1x_2}\pa_{x_1}\pa_{x_2}v(\bx,t)=
\int\limits_{\bbR_+^{n-2}}
\prod_{j=3}^nk^{b_j}_t(x_j,z_j)\\ \int\limits_{0}^{\infty}\int\limits_{0}^{\infty}
\left(\frac{x_1}{z_1}\right)^{\frac 12}k^{b_1+1}_t(x_1,z_1)
\left(\frac{x_2}{z_2}\right)^{\frac 12}k^{b_2+1}_t(x_2,z_2)
\sqrt{z_1z_2}\pa_{z_1}\pa_{z_2}f(\bz)d\bz.
\end{multline}
Since $f\in\cC^{0,2+\gamma}_{\WF}$ it is not immediately obvious that this is
true, but can be obtained by a simple limiting argument. We let
\begin{equation}
  f_{\epsilon}=f(x_1+\epsilon,x_2+\epsilon,x_3,\dots,x_n),
\end{equation}
and $v_{\epsilon}$ the solution of the Cauchy problem with this initial data.
For $\epsilon>0$ it follows easily from Proposition~\ref{lem3.3new0.0} that
\begin{multline}
   \sqrt{x_1x_2}\pa_{x_1}\pa_{x_2}v_{\epsilon}(\bx,t)=
\int\limits_{\bbR_+^{n-2}}
\prod_{j=3}^nk^{b_j}_t(x_j,z_j)
\\\int\limits_{0}^{\infty}\int\limits_{0}^{\infty}
\left(\frac{x_1}{z_1}\right)^{\frac 12}k^{b_1+1}_t(x_1,z_1)
\left(\frac{x_2}{z_2}\right)^{\frac 12}k^{b_2+1}_t(x_2,z_2)
\sqrt{z_1z_2}\pa_{z_1}\pa_{z_2}f_{\epsilon}(\bz)d\bz.
\end{multline}
For $t>0,$ the left hand side converges uniformly to 
$\sqrt{x_1x_2}\pa_{x_1}\pa_{x_2}v.$ Since the scaled derivative
$\sqrt{(z_1+\epsilon)(z_2+\epsilon)}\pa_{z_1}\pa_{z_2}f_{\epsilon}(\bz)$ is
uniformly bounded and converges to
$\sqrt{z_1z_2}\pa_{z_1}\pa_{z_2}f(\bz),$ and the kernel
\begin{equation}
  \left(\frac{x_1}{z_1}\right)^{\frac 12}k^{b_1+1}_t(x_1,z_1)
\left(\frac{x_2}{z_2}\right)^{\frac 12}k^{b_2+1}_t(x_2,z_2)
\end{equation}
is absolutely integrable, we see that the limit can be taken inside the
integral to give~\eqref{eqn10.45.1}.

The following lemma is used to bound these integrals
\begin{lemma}\label{lem10.0.1} If   $0\leq\gamma\leq 1,$  
and $b>\nu-\frac{\gamma}{2}>0,$ then there is a constant
  $C_{b,\phi},$ bounded for $b\leq B,$ and $B^{-1}<b+\frac{\gamma}{2}-\nu,$ so
  that, for $t\in S_{\phi},$ where $0<\phi<\frac{\pi}{2},$ we have the estimate
  \begin{equation}
    \int\limits_0^{\infty}
\left(\frac{x}{y}\right)^{\nu}|k_t^{b}(x,y)|y^{\frac{\gamma}{2}}dy\leq C_{b,\phi}x^{\frac{\gamma}{2}}.
  \end{equation} 
\end{lemma}
\noindent
The lemma is proved in Appendix~\ref{prfsoflems}.

Since $f\in\cC^{0,2+\gamma}_{\WF}(\bbR_+^n)$ it follows that
\begin{equation}
 | \sqrt{x_1x_2}\pa_{x_1}\pa_{x_2}f(\bx)|\leq
 \|f\|_{\WF,0,2+\gamma}\min\{x_1^{\frac{\gamma}{2}},
x_2^{\frac{\gamma}{2}},1\},
\end{equation}
applying this lemma shows that
\begin{equation}
  |\sqrt{x_ix_j}\pa_{x_i}\pa_{x_j}v(\bx,t)|\leq
C\|f\|_{\WF,0,2+\gamma}.
\end{equation}
We need to now establish the H\"older continuity of these derivatives. The
argument used above for $v$ applies directly to show the H\"older continuity in
the $(x_3,\dots,x_n)$ variables, leaving only $x_1$ and $x_2.$ It clearly suffices to
do the $x_1$-case. We have the estimate:
\begin{multline}
  |\sqrt{x_1x_2}\pa_{x_1}\pa_{x_2}v(\bx,t)|\leq\\
\|f\|_{\WF,0,2+\gamma}\int\limits_{0}^{\infty}\int\limits_{0}^{\infty}
\left(\frac{x_1}{z_1}\right)^{\frac 12}k^{b_1+1}_t(x_1,z_1)
\left(\frac{x_2}{z_2}\right)^{\frac
  12}k^{b_2+1}_t(x_2,z_2)z_1^{\frac{\gamma}{2}}dz_1dz_2.
\end{multline} 
Lemma~\ref{lem10.0.1} bounds both $z_1$- the $z_2$-integrals and therefore
\begin{equation}
  |\sqrt{x_1x_2}\pa_{x_1}\pa_{x_2}v(\bx,t)|\leq
  C\min\{x_1^{\frac{\gamma}{2}},
  x_2^{\frac{\gamma}{2}}\}\|f\|_{\WF,0,2+\gamma}.
\end{equation}
Note that this implies that
\begin{equation}
  \lim_{x_i\vee x_j\to 0^+}|\sqrt{x_ix_j}\pa_{x_i}\pa_{x_j}v(\bx,t)|=0.
\end{equation}
If $c<1,$ then this implies the estimate
\begin{equation}
   |\sqrt{x_1x_2}\pa_{x_1}\pa_{x_2}v(\bx,t)-\sqrt{x_1'x_2}\pa_{x_1}\pa_{x_2}v(\bx',t)|\leq
C|\sqrt{x_1}-\sqrt{x_1'}|^{\gamma}\|f\|_{\WF,0,2+\gamma}.
\end{equation}
for $x_1'<cx_1.$ 

Thus we are left to consider $cx_1<x_1'<x_1.$ To simplify the notation we let
$\tbz=(z_3,\dots,z_n),$ $\tbb=(b_3,\dots,b_n),$
\begin{equation}
  f_{12}(\bz)=\pa_{z_1}\pa_{z_2}f(\bz),
\end{equation}
and
\begin{equation}
  k^{\tbb}_t(\tbx,\tbz)=\prod\limits_{j=3}^{n}k^{b_j}_t(x_j,z_j).
\end{equation}
We have the formula
\begin{multline}
  \sqrt{x_1x_2}\pa_{x_1}\pa_{x_2}v(\bx,t)-\sqrt{x_1'x_2}\pa_{x_1}\pa_{x_2}v(\bx',t)=\\
\int\limits_{\bbR_+^{n-2}}\int\limits_0^{\infty}\int\limits_0^{\infty}
k^{\tbb}_t(\tbx,\tbz)k_t^{b_2+1}(x_2,z_2)\left(\frac{x_2}{z_2}\right)^{\frac{1}{2}}\times\\
\left[\sqrt{x_1}k_t^{b_1+1}(x_1,z_1)-
\sqrt{x_1'}k_t^{b_1+1}(x_1',z_1)\right]
\sqrt{z_2}f_{12}(z_1,z_2,\tbz)dz_1dz_2d\tbz.
\end{multline}
We assume that 
\begin{equation}\label{eqn10.58.4}
  \frac{1}{4}\leq \frac{x_1'}{x_1}\leq 1,
\end{equation}
and let
\begin{equation}\label{eqn10.59.5}
  \sqrt{\alpha}=\max\left\{\frac{3\sqrt{x_1'}-\sqrt{x_1}}{2},0\right\}\text{ and }
\sqrt{\beta}=\frac{3\sqrt{x_1}-\sqrt{x_1'}}{2}.
\end{equation}
Note that~\eqref{eqn10.58.4} implies that $\alpha>x_1'/4.$ We let
$J=[\alpha,\beta],$ and observe that
\begin{equation}\label{eqn10.60.3}
  |\sqrt{x_1z_2}f_{12}(x_1,z_2,\tbz)- \sqrt{x_1'z_2}f_{12}(x_1',z_2,\tbz)|\leq
2\|f\|_{\WF,0,2+\gamma}| \sqrt{x_1}- \sqrt{x_1'}|^{\gamma}.
\end{equation}
This estimate implies that
\begin{equation}
  |\sqrt{z_2}f_{12}(x_1,z_2,\tbz)|\leq 2C\|f\|_{\WF,0,2+\gamma}|x_1^{\frac{\gamma-1}{2}}.
\end{equation}

To estimate this difference we dissect the $z_1$-integral in a manner similar
to that used in~\eqref{dlctest0}:
\begin{multline}\label{eqn10.63.2}
   \sqrt{x_1x_2}\pa_{x_1}\pa_{x_2}v(\bx,t)-\sqrt{x_1'x_2}\pa_{x_1}\pa_{x_2}v(\bx',t)=\\
\int\limits_{\bbR_+^{n-2}}\int\limits_0^{\infty}
k^{\tbb}_t(\tbx,\tbz)k_t^{b_2+1}(x_2,z_2)\left(\frac{x_2}{z_2}\right)^{\frac{1}{2}}\times\\
\Bigg[\int\limits_{J}k_t^{b_1+1}(x_1,z_1)\sqrt{x_1z_2}[f_{12}(z_1,z_2,\tbz)-f_{12}(x_1,z_2,\tbz)]dz_1-\\
\int\limits_{J}k_t^{b_1+1}(x_1',z_1)\sqrt{x_1'z_2}[f_{12}(z_1,z_2,\tbz)-f_{12}(x_1',z_2,\tbz)]dz_1+\\
\int\limits_{J^c}k_t^{b_1+1}(x_1,z_1)\sqrt{x_1z_2}[f_{12}(x_1',z_2,\tbz)-f_{12}(x_1,z_2,\tbz)]dz_1+\\
\int\limits_{J^c}[k_t^{b_1+1}(x_1,z_1)\sqrt{x_1z_2}-
k_t^{b_1+1}(x_1',z_1)\sqrt{x_1'z_2}][f_{12}(z_1,z_2,\tbz)-f_{12}(x_1',z_2,\tbz)]dz_1+\\
[\sqrt{x_1z_2}f_{12}(x_1,z_2,\tbz)-\sqrt{x_1'z_2}f_{12}(x_1',z_2,\tbz)]
\Bigg]
dz_2d\tbz.
\end{multline}
We denote the terms on the right hand side by $I,II,III,IV,$  and $V.$ The
terms $I,II,III,$ and $V$ can be estimated fairly easily. For $V$ we apply the
estimate in~\eqref{eqn10.60.3} and Lemma~\ref{lem10.0.1} to conclude that
\begin{equation}
  |V|\leq 2\|f\|_{\WF0,2+\gamma}| \sqrt{x_1}- \sqrt{x_1'}|^{\gamma}.
\end{equation}
To handle $I$ we observe that
\begin{equation}
\begin{split}
  |&\sqrt{x_1z_2}[f_{12}(z_1,z_2,\tbz)-f_{12}(x_1,z_2,\tbz)]|\\
&\leq
|\sqrt{z_1z_2}f_{12}(z_1,z_2,\tbz)-\sqrt{x_1z_2}f_{12}(x_1,z_2,\tbz)|+
|\sqrt{z_1z_2}-\sqrt{x_1z_2}||f_{12}(z_1,z_2,\tbz)|\\
&\leq  2\|f\|_{\WF0,2+\gamma}\left[| \sqrt{x_1}- \sqrt{z_1}|^{\gamma}+
z_1^{\frac{\gamma-1}{2}}|\sqrt{z_1}-\sqrt{x_1}|\right].
\end{split}
\end{equation}
We note that
\begin{equation}\label{eqn10.65.6} 
  z_1^{\frac{\gamma-1}{2}}|\sqrt{z_1}-\sqrt{x_1}|=
|\sqrt{z_1}-\sqrt{x_1}|^{\gamma}\left|1-\sqrt{\frac{x_1}{z_1}}\right|^{1-\gamma}
\end{equation}
For $z_1\in J$ the ratio
\begin{equation}
  \frac{x_1}{z_1}\leq 16,
\end{equation}
and the differences $|\sqrt{z_1}-\sqrt{x_1}|$ and $|\sqrt{z_1}-\sqrt{x_1'}|$
are bounded above by a multiple of $|\sqrt{x_1}-\sqrt{x_1'}|.$ Once again we
can use Lemma~\ref{lem10.0.1} to see that
\begin{equation}
  |I|\leq 2C\|f\|_{\WF0,2+\gamma}| \sqrt{x_1}- \sqrt{x_1'}|^{\gamma}
\end{equation}
The same argument applies with minor modifications to show that
\begin{equation}
   |II|\leq 2C\|f\|_{\WF0,2+\gamma}| \sqrt{x_1}- \sqrt{x_1'}|^{\gamma} 
\end{equation}

This argument also shows that
\begin{equation}
  |\sqrt{x_1z_2}[f_{12}(x_1',z_2,\tbz)-f_{12}(x_1,z_2,\tbz)|\leq
2\|f\|_{\WF,0,2+\gamma}| \sqrt{x_1}- \sqrt{x_1'}|^{\gamma},
\end{equation}
which implies that
\begin{equation}
   |III|\leq 2C\|f\|_{\WF0,2+\gamma}| \sqrt{x_1}- \sqrt{x_1'}|^{\gamma} 
\end{equation}
This leaves only the term of type $IV.$ We rewrite this term as
\begin{multline}\label{eqn10.71.3}
IV= \\\int\limits_{J^c}\left|k_t^{b_1+1}(x_1,z_1)\sqrt{\frac{x_1}{z_1}}-
k_t^{b_1+1}(x_1',z_1)\sqrt{\frac{x_1'}{z_1}}\right|\sqrt{z_1z_2}
|f_{12}(z_1,z_2,\tbz)-f_{12}(x_1',z_2,\tbz)|dz_1.
\end{multline}
Arguing as in~\eqref{eqn10.65.6} we see that
\begin{multline}
  |\sqrt{z_1z_2}
[f_{12}(z_1,z_2,\tbz)-f_{12}(x_1',z_2,\tbz)]|\leq\\
2\|f\|_{\WF,0,2+\gamma}\left[| \sqrt{z_1}- \sqrt{x_1'}|^{\gamma}+
(x_1')^{\frac{\gamma-1}{2}}| \sqrt{z_1}- \sqrt{x_1'}|\right].
\end{multline}
We complete the estimate of $IV$ with the following lemma:
\begin{lemma}\label{lem10.0.3.1}
If $J=[\alpha,\beta],$ with $\alpha,\beta$ are given by~\eqref{eqn10.59.5},
  assuming that $x_1',x_1$ satisfy~\eqref{eqn10.58.4}, and $b>0,$ $0<\gamma\leq 1,$
  and $0<\phi<\frac{\pi}{2},$ there is a $C_{b,\phi}$ so that if $t\in
  S_{\phi},$ then,
\begin{equation}
  \int\limits_{J^c}
\left|k_t^{b+1}(x_1,z_1)\sqrt{\frac{x_1}{z_1}}-
k_t^{b+1}(x_1',z_1)\sqrt{\frac{x_1'}{z_1}}\right|
| \sqrt{z_1}- \sqrt{x_1'}|^{\gamma}dz_1\leq C_{b,\phi}| \sqrt{x_1}- \sqrt{x_1'}|^{\gamma}.
\end{equation}
 \end{lemma}

Applying the lemma to the expression in~\eqref{eqn10.71.3} completes the proof
that
\begin{equation}
  |\sqrt{x_ix_j}\pa_{x_i}\pa_{x_j}v(\bx,t)-
\sqrt{x_i'x_j'}\pa_{x_i}\pa_{x_j}v(\bx',t)|\leq
C\|f\|_{\WF,0,2+\gamma}\rho_s(\bx,\bx')^{\gamma}.
\end{equation}

We now consider the H\"older continuity in time for these derivatives. We
re-write
this difference as
\begin{multline}\label{eqn10.75.6}
   \sqrt{x_1x_2}\pa_{x_1}\pa_{x_2}[v(\bx,t)-v(\bx,0)]=\\
\sqrt{x_1x_2}
\int\limits_{\bbR_+^{n-2}}\int\limits_0^{\infty}\int\limits_0^{\infty}
k_t^{\tbb}(\tbx,\tbz)k^{b_1+1}_t(x_1,z_1)k^{b_2+1}_t(x_2,z_2)
[f_{12}(\bz)-f_{12}(\bx)]dz_1dz_2d\tbz,
\end{multline}
where $f_{12}=\pa_{z_1}\pa_{z_2}f.$ We re-write the difference
$f_{12}(\bz)-f_{12}(\bx)$ as
\begin{multline}\label{eqn10.76.6}
  f_{12}(\bz)-f_{12}(\bx)=\sqrt{\frac{z_1z_2}{z_1z_2}}\sum_{j=3}^n[f_{12}(\bz'_l,z_{l},\bx''_{l})-
f_{12}(\bz'_l,x_{l},\bx''_{l})]+\\
[f_{12}(z_1,z_2,\bx''_n)-f_{12}(z_1,x_2,\bx''_n)]+
[f_{12}(z_1,x_2,\bx''_n)-f_{12}(x_1,x_2,\bx''_n)].
\end{multline}
Substituting from the sum in~\eqref{eqn10.76.6} into~\eqref{eqn10.75.6}, we see
that the terms for $l=3,\dots,n$ are each bounded by:
\begin{equation}
 \|f\|_{\WF,0,2+\gamma} \int\limits_0^{\infty}\int\limits_0^{\infty}\int\limits_0^{\infty}
k^{b_1+1}_t(x_1,z_1)k^{b_2+1}_t(x_2,z_2)\sqrt{\frac{x_1x_2}{z_1z_2}}
k^{b_l}_t(x_l,z_l)|\sqrt{x_l}-\sqrt{z_l}|^{\gamma}dz_ldz_1dz_2.
\end{equation}
From Lemma~\ref{lem10.0.1} and Lemma~\ref{lem5new} we see that these terms are
bounded by $C\|f\|_{\WF,0,2+\gamma}  t^{\frac{\gamma}{2}}.$ 

We re-write $[f_{12}(z_1,z_2,\bx''_n)-f_{12}(z_1,x_2,\bx''_n)]$  as
\begin{multline}
  [f_{12}(z_1,z_2,\bx''_n)-f_{12}(z_1,x_2,\bx''_n)]=\\
\frac{1}{\sqrt{z_1x_2}} [(\sqrt{z_1x_2}-\sqrt{z_1z_2})f_{12}(z_1,z_2,\bx''_n)
+(\sqrt{z_1z_2}f_{12}(z_1,z_2,\bx''_n)-\sqrt{z_1x_2}f_{12}(z_1,x_2,\bx''_n)].
\end{multline}
The right hand side is estimated by
\begin{equation}
  \frac{\|f\|_{\WF,0,2+\gamma}}{\sqrt{z_1x_2}}[|\sqrt{x_2}-\sqrt{z_2}|z_2^{\frac{\gamma-1}{2}}+
|\sqrt{x_2}-\sqrt{z_2}|^{\gamma}].
\end{equation}
The contribution of this term is therefore bounded by
\begin{equation}
\|f\|_{\WF,0,2+\gamma}\int\limits_0^{\infty}\int\limits_0^{\infty}
\sqrt{\frac{x_1}{z_1}}k^{b_1+1}_t(x_1,z_1)k^{b_2+1}_t(x_2,z_2)
[|\sqrt{x_2}-\sqrt{z_2}|^{\gamma}+|\sqrt{x_2}-\sqrt{z_2}|z_2^{\frac{\gamma-1}{2}}]dz_1dz_2.
\end{equation}
Lemmas~\ref{lem10.0.1} and~\ref{lem5new} show that the
$|\sqrt{x_2}-\sqrt{z_2}|^{\gamma}$-term is bounded by 
$$C\|f\|_{\WF,0,2+\gamma}t^{\frac{\gamma}{2}}.$$
To bound the contribution of the other term we apply
\begin{lemma}\label{lem10.0.3} For  $0\leq b,$ $0\leq \gamma<1,$ there is a constant $C_{b,\phi},$
  bounded for $b\leq B,$ so that for $t\in S_{\phi},$
  \begin{equation}
    \int\limits_0^{\infty}
|k^{b+1}_t(x,y)||\sqrt{y}-\sqrt{x}|y^{\frac{\gamma-1}{2}}dy\leq C_{b,\phi}|t|^{\frac{\gamma}{2}}.
  \end{equation}
\end{lemma}

The last term is re-written as
\begin{multline}\label{eqn10.82.3}
  \sqrt{x_1x_2}[f_{12}(z_1,x_2,\bx''_n)-f_{12}(x_1,x_2,\bx''_n)]=\\
[ \sqrt{z_1x_2}f_{12}(z_1,x_2,\bx''_n)- \sqrt{x_1x_2}f_{12}(x_1,x_2,\bx''_n)]+
( \sqrt{x_1x_2}- \sqrt{z_1x_2})f_{12}(z_1,x_2,\bx''_n),
\end{multline}
which is estimated by
\begin{equation}\label{eqn10.83.3}
  \|f\|_{\WF0,2+\gamma}[|\sqrt{z_1}-\sqrt{x_1}|z_1^{\frac{\gamma-1}{2}}+
|\sqrt{z_1}-\sqrt{x_1}|^{\gamma}].
\end{equation}
These terms are estimated  as in the previous case, showing that
altogether there is a $C$ so that:
\begin{equation}
  |\sqrt{x_ix_j}\pa_{x_i}\pa_{x_j}v(\bx,t)-\sqrt{x_ix_j}\pa_{x_i}\pa_{x_j}v(\bx,0)|\leq
C\|f\|_{\WF,0,2+\gamma}t^{\frac{\gamma}{2}},
\end{equation}
which implies that for a $c<1,$ there is a $C$ so that, if $s<ct,$ then
\begin{equation}
|\sqrt{x_ix_j}\pa_{x_i}\pa_{x_j}v(\bx,t)-\sqrt{x_ix_j}\pa_{x_i}\pa_{x_j}v(\bx,s)|
\leq C\|f\|_{\WF,0,2+\gamma}|t-s|^{\frac{\gamma}{2}}.
\end{equation}

This leaves only the case  $ct<s<t,$ for a $c<1.$ We begin with the analogue
of~\eqref{eqn10.36.1}
\begin{multline}
   \sqrt{x_1x_2}\pa_{x_1}\pa_{x_2}[v(\bx,t)-v(\bx,s)]=
\int\limits_0^{\infty}\cdots\int\limits_0^{\infty}
\sqrt{x_1x_2}\sum\limits_{l=1}^{n}\Bigg\{
\prod_{j=1}^{n-l}k_t^{\tb_j}(x_j,z_j)\times \\\prod_{j=n-l+2}^{n}k_s^{\tb_j}(x_j,z_j)
\left[k_t^{\tb_{n-l+1}}(x_{n-l+1},z_{n-l+1})-k_s^{\tb_{n-l+1}}(x_{n-l+1},z_{n-l+1})\right]\times\\
[f_{12}(\bz)-f_{12}(\bz'_l,x_{n-l+1},\bz''_{l})]\Bigg\}d\bz,
\end{multline}
where
\begin{equation}
 \tb_1=b_1+1,\quad \tb_2=b_2+1,\text{ and } \tb_j=b_j\text{ for }j>2.
\end{equation}
Each term in this sum with $1\leq l\leq n-2$ is estimated by
\begin{multline}
 \|f\|_{\WF,0,2+\gamma} \int\limits_{0}^{\infty}\int\limits_{0}^{\infty}\int\limits_{0}^{\infty}
\left(\frac{x_1}{z_1}\right)^{\frac 12}k^{b_1+1}_t(x_1,z_1)
\left(\frac{x_2}{z_2}\right)^{\frac 12}k^{b_2+1}_t(x_2,z_2) \\
\times |k_t^{\tb_{n-l+1}}(x_{n-l+1},z_{n-l+1})-k_s^{\tb_{n-l+1}}(x_{n-l+1},z_{n-l+1})| \\
\times |\sqrt{z_{n-l+1}}-\sqrt{x_{n-l+1}}|^{\gamma}dz_{n-l+1}dz_2dz_1.
\end{multline}
Lemmas~\ref{lem10.0.1} and~\ref{lem4new} show that these terms are bounded by 
\begin{equation}
C \|f\|_{\WF,0,2+\gamma}|t-s|^{\frac{\gamma}{2}}.  
\end{equation}

We now turn to $l=n-1,$ and $n.$ These cases are essentially identical; we give
the details for $l=n-1.$ The contribution of this term is bounded by
\begin{multline}
  \|f\|_{\WF,0,2+\gamma}
  \int\limits_{\bbR_+^{n-2}}\prod\limits_{j=3}^nk_s^{b_j}(x_j,z_j)
\int\limits_{0}^{\infty}\int\limits_{0}^{\infty}
\sqrt{x_1x_2}k^{b_1+1}_{t}(x_1,z_1)\times\\
\left|k^{b_2+1}_{t}(x_2,z_2)-k^{b_2+1}_{s}(x_2,z_2)\right|
|f_{12}(\bz)-f_{12}(z_1,x_2,\bz''_{n-1})|dz_2dz_1d\bz''_{n-1}.
\end{multline}
Proceeding as in~\eqref{eqn10.82.3} and~\eqref{eqn10.83.3}, we see that
\begin{multline}
\sqrt{x_1x_2}|f_{12}(\bz)-f_{12}(z_1,x_2,\bz''_{n-1})|\leq  \\ 
\|f\|_{\WF,0,2+\gamma} \sqrt{\frac{x_1}{z_1}} 
[|\sqrt{x_2}-\sqrt{z_2}|z_2^{\frac{\gamma-1}{2}}+|\sqrt{x_2}-\sqrt{z_2}|^{\gamma}].
\end{multline}
Applying Lemma~\ref{lem10.0.1} shows that we are left to estimate
\begin{multline}
  \|f\|_{\WF,0,2+\gamma}   \int\limits_{0}^{\infty} 
\left|k^{b_2+1}_{t}(x_2,z_2)-k^{b_2+1}_{s}(x_2,z_2)\right| \\
\times [|\sqrt{x_2}-\sqrt{z_2}|z_2^{\frac{\gamma-1}{2}}+|\sqrt{x_2}-\sqrt{z_2}|^{\gamma}]dz_2.
\end{multline}
Lemma~\ref{lem4new} shows that
\begin{equation}
   \int\limits_{0}^{\infty}
\left|k^{b_2+1}_{t}(x_2,z_2)-k^{b_2+1}_{s}(x_2,z_2)\right|
|\sqrt{x_2}-\sqrt{z_2}|^{\gamma}dz_2\leq C|t-s|^{\frac{\gamma}{2}},
\end{equation}
leaving only
\begin{equation}
   \int\limits_{0}^{\infty}
\left|k^{b_2+1}_{t}(x_2,z_2)-k^{b_2+1}_{s}(x_2,z_2)\right|
|\sqrt{x_2}-\sqrt{z_2}|z_2^{\frac{\gamma-1}{2}}dz_2.
\end{equation}
This term is bounded in the following lemma:
\begin{lemma}\label{lem10.1.4} For $0\leq \gamma<1,$ $1\leq b,$ and $0<c<1,$
  there is a constant $C_b,$ so that if $ct<s<t,$ then
  \begin{equation}
     \int\limits_{0}^{\infty}
\left|k^{b}_{t}(x,z)-k^{b}_{s}(x,z)\right|
|\sqrt{x}-\sqrt{z}|z^{\frac{\gamma-1}{2}}dz\leq C_b|t-s|^{\frac{\gamma}{2}}.
  \end{equation}
\end{lemma}
\noindent
The proof of the lemma is in Appendix~\ref{prfsoflems}. Applying this result
completes the proof that
\begin{equation}\label{eqn10.97.2}
  | \sqrt{x_1x_2}\pa_{x_1}\pa_{x_2}[v(\bx,t)-v(\bx,s)]|
\leq C\|f\|_{\WF,0,2+\gamma}|t-s|^{\frac{\gamma}{2}}.
\end{equation}
The fact that $\pa_tv=L_{\bb,0}v$ allows us to deduce that
$\pa_tv\in\cC^{0,\gamma}_{\WF}(\bbR_+^n\times [0,\infty))$ and satisfies the
same estimates as the spatial derivative. This finishes the proof
of~\eqref{bscestn02} in the $k=0$ case.

We can now proceed as we did in the proof of~\eqref{bscestn00} for
$k>0,$ applying Proposition~\ref{lem3.3new0.0} to commute derivatives
past the kernel functions. We now assume that $f$ has support in
$B_R^+(\bzero),$ which allows the use of Proposition~\ref{wtedest_pr}
to estimate the resultant data. This reduces the proof
of~\eqref{bscestn02} for $k>0,$ to the $k=0$ case, which thereby
completes the proof of the proposition.
\end{proof}

\section{The Inhomogeneous Case}
We now turn to estimating the solution of the inhomogeneous problem in a
$n$-dimensional corner. Let $g\in\cC_{\WF,0,\gamma}(\bbR_+^n\times [0,T]),$ and
let $u$ denote the solution to
\begin{equation}\label{n0dmdlb}
\left[\pa_t-  \sum_{j=1}^N[x_j\pa_{x_j}^2+b_j\pa_{x_j}]\right]u=g,
\end{equation}
which vanishes at $t=0.$ According to Proposition~\ref{lem3.4new0.0}, it is
given by the integral:
\begin{equation}
  u(\bx,t)=\int\limits_{0}^{t}\int\limits_{0}^{\infty}\cdots\int\limits_{0}^{\infty}
\prod_{j=1}^nk^{b_j}_s(x_j,z_j)g(\bz,t-s)d\bz ds.
\end{equation}

  \begin{proposition}\label{prop1n0} Fix  $0<\gamma<1,$ $0<R,$ $k\in\bbN_0,$ and
    $(b_1,\dots,b_n)\in\bbR_+^n.$ Let
    $g\in\cC^{k,\gamma}_{\WF}(\bbR_+^n\times [0,T]),$ and let $u$ be
    the unique solution, given in~\eqref{n0slninhom00}
    to~\eqref{n0dmdlb}, with $u(\bx,0)=0.$ If $k>0,$ then assume that
    $g$ is supported in $B_R^+(\bzero)\times [0,T].$ The solution
    \begin{equation}
      u\in\cC^{k,2+\gamma}_{\WF}(\bbR_+^n\times [0,T]);
    \end{equation}
 there is a constant $C_{k,\gamma,b,R}$ so that
  \begin{equation}\label{bscest3n0}
    \|u\|_{\WF,k,2+\gamma,T}\leq C_{k,\gamma,b,R}(1+T)\|g\|_{\WF,k,\gamma,T}.
  \end{equation}
  For fixed $\gamma,$ the constants $C_{k,\gamma,b,R}$ are uniformly bounded for
  $0\leq b<B.$ If $k=0,$ then the constants are independent of $R.$
\end{proposition}
\begin{proof}
As before it suffices to  assume that $b_j>0$ for $j=1,\dots,n.$ The estimates we prove below
have constants $C,$ which, for any $B,$  are uniformly bounded if  $0<b_j<B.$ The
case where $b_j=0,$ for one or more values $j,$ is again treated by choosing a
sequence $<\bb_n>$ so that
\begin{equation}
  b_{n,j}>0\text{ for all }n\text{ and }\lim_{n\to\infty}b_{n,j}=b_j.
\end{equation}
We let $<u_{\bb_n}>$ denote the solutions with the given data $g.$ Given that
the estimates in the lemma have been proved for each $\bb_n$ we see that,
Proposition~\ref{prop4.1new} and uniqueness imply that for
$0<\gamma'<\gamma,$ the sequence $<u_{\bb_n}>,$ converges, in
$\cC_{\WF}^{0,2+\gamma'},$ to $u_{\bb},$ the solution in the limiting case. This
implies that $u_{\bb}$ also satisfies the estimates in the proposition. This
reasoning applies equally well to all the function spaces under
consideration. It therefore suffices to consider the case where $b_j>0$ for
$j=1,\dots,n,$ which we henceforth assume.

As in the proof of Proposition~\ref{prop1} we note that with
\begin{equation}
  u_{\epsilon}(\bx,t)=\int\limits_{0}^{t-\epsilon}\int\limits_{0}^{\infty}\cdots\int\limits_{0}^{\infty}
\prod_{j=1}^nk^{b_j}_s(x_j,z_j)g(\bz,t-s)d\bz ds,
\end{equation}
the solution $u$ is the uniform limit of $u_{\epsilon}.$ The functions
$u_{\epsilon}$ are smooth where $t>0,$ and we can show as before that, for
$0<\epsilon<t,$ we have
\begin{equation}\label{eqn10.49.1}
  \begin{split}
\pa_{x_k} u_{\epsilon}( & \bx,t)=\\
\int\limits_{0}^{t-\epsilon}\int\limits_{0}^{\infty}\cdots\int\limits_{0}^{\infty} & 
\pa_{x_k}\prod_{j=1}^nk^{b_j}_s(x_j,z_j)[g(\bz,t-s)-g(\bz_k',x_k,\bz_k'')]d\bz ds\\
\sqrt{x_kx_l}\pa_{x_l}\pa_{x_k}u_{\epsilon}(& \bx,t)=\\
\int\limits_{0}^{t-\epsilon}\int\limits_{0}^{\infty}\cdots\int\limits_{0}^{\infty} & 
\sqrt{x_kx_l}\pa_{x_l}\pa_{x_k}\prod_{j=1}^nk^{b_j}_s(x_j,z_j) 
[g(\bz,t-s)-g(\bz_k',x_k,\bz_k'')]d\bz ds.
\end{split}
\end{equation}
Assume that $g\in\cC^{0,\gamma}_{\WF}(\bbR_+^n\times [0,T]),$ for a
$0<\gamma<1.$ Using Lemma~\ref{lem2new} for the first derivatives, and the
mixed derivatives where $k\neq l,$ and Lemma~\ref{lem25new}, when $k=l,$ we can
again show that these derivatives converge, as $\epsilon\to 0^+,$ uniformly on
$\bbR_+^n\times [0,T].$ This shows that $u$ has continuous first partial
$\bx$-derivatives on $\bbR_+^n\times [0,T],$ and continuous second
$\bx$-derivatives on $(0,\infty)^n\times [0,T],$ with
\begin{equation}
  \lim_{x_l\vee x_k\to 0^+}\sqrt{x_kx_l}\pa_{x_l}\pa_{x_k}u(\bx,t)=0.
\end{equation}
This also shows that we can allow $\epsilon\to 0^+,$ in the expressions for
these derivatives in~\eqref{eqn10.49.1} to obtain absolutely convergent
expressions for the corresponding derivatives of $u.$ Finally we argue as
before to show that
\begin{equation}
  \pa_tu=L_{\bb,0}u+g,
\end{equation}
and therefore the $t$-derivative of $u$ is continuous and $u$ satisfies the desired equation.

Note that
\begin{equation}
  u(\bx,t)=\int\limits_{0}^tv^s(\bx,t-s)ds,
\end{equation}
where $v^s(\bx,t)$ is the solution to
\begin{equation}
  [\pa_t-L_{\bb,0}]v^s=0\text{ with }v^s(\bx,0)=g(\bx,s).
\end{equation}
This relation allows us to use estimates on the solution to the Cauchy problem
to derive bounds on $u.$

From the positivity of the heat kernel and~\eqref{yintis1} it is immediate that
\begin{equation}\label{inhocpn00est}
  |u(\bx,t)|\leq \|g\|_{\WF,0,\gamma}t.
\end{equation}
To establish the Lipschitz continuity of $u$ we integrate~\eqref{hmcp1drestn0} to
conclude that, for $1\leq j\leq n,$ we have the estimate:
\begin{equation}\label{n0t1stderest}
  |\pa_{x_j}u(\bx,t)|\leq C\|g\|_{\WF,0,\gamma}t^{\frac{\gamma}{2}}
\end{equation}
and therefore
\begin{equation}
  |u(\bx_j',x_j,\bx_j'',t)-u(\bx_j',y_j,\bx_j'',t)|
\leq C\|g\|_{\WF,0,\gamma}t^{\frac{\gamma}{2}}|x_j-y_j|.
\end{equation}
Thus we can also integrate the estimate in~\eqref{hmcphestn0} with respect to
$t$ to see that
  \begin{equation}
  |u(\bx_j',x_j,\bx_j'',t)-u(\bx_j',y_j,\bx_j'',t)|
\leq C\|g\|_{\WF,0,\gamma}t|\sqrt{x_j}-\sqrt{y_j}|^{\gamma}.
\end{equation}

The estimates, proved below on the first and second derivatives, show that
\begin{equation}
|L_{\bb,0}u(\bx,t)|\leq C\|g\|_{\WF,0,\gamma} t^{\frac{\gamma}{2}},
\end{equation}
thus the equation $[\pa_t-L_{\bb,0}]u(\bx,t)=g(\bx,t)$ implies that, for
$t_1<t_2$ we have
\begin{equation}
 |u(\bx,t_2)-u(\bx,t_1)|\leq  C(1+t_2^{\frac{\gamma}{2}})\|g\|_{\WF,0,\gamma}|t_2-t_1|.
\end{equation}

Our next task is to establish the H\"older estimate for the first
spatial-derivatives. There is a small twist in the higher dimensional case: we
use one argument to estimate
$|\pa_{x_j}u(\bx_j',x_j,\bx_j'',t)-\pa_{x_j}u(\bx_j',y_j,\bx_j'',t)|,$ and a
rather different argument to estimate 
$|\pa_{x_l}u(\bx_j',x_j,\bx_j'',t)-\pa_{x_l}u(\bx_j',y_j,\bx_j'',t)|,$ for
$l\neq j.$ The former follows exactly as in the 1-dimensional case, we show
that  there is a constant  so that for $1\leq l\leq n,$
\begin{equation}\label{2ndderestn0}
  |x_l\pa_{x_l}^2u(\bx,t)|\leq C\|g\|_{\WF,0,\gamma}\min\{t^{\frac{\gamma}{2}},x_l^{\frac{\gamma}{2}}\}.
\end{equation}
This estimate implies that
\begin{equation}\label{eqn302new.0}
  \lim_{x_l\to 0^+}x_l\pa_{x_l}^2u(\bx,t)=0.
\end{equation}
The proof of~\eqref{2ndderestn0} follows simply from:
\begin{multline}
  x_l\pa_{x_l}^2u(\bx,t)=
\int\limits_{0}^{t}\int\limits_{0}^{\infty}\cdots\int\limits_{0}^{\infty}
\prod_{j\neq l}^nk^{b_j}_s(x_j,z_j)\times\\
x_l\pa_{x_l}^2k^{b_l}_s(x_l,z_l)[g(\bz,t-s)-g(\bz'_l,x_l,\bz''_l,t-s)]d\bz ds.
\end{multline}
Putting absolute values inside the integral, and using the
estimates
\begin{equation}\label{nhtdholdest00}
  |g(\bx,t)-g(\by,s)|\leq 2\|g\|_{\WF,0,\gamma}[\rho_s(\bx,\by)+\sqrt{|t-s|}]^{\gamma},
\end{equation}
and~\eqref{yintis1}, we see
that~\eqref{2ndderestn0} follows from Lemma~\ref{lemA}.  

Integrating $\pa_{x_j}^2u$ and applying~\eqref{2ndderestn0}, we see that
\begin{equation}\label{frstdrdiaghldestn0}
  |\pa_{x_j}u(\bx_j',x_j,\bx_j'',t)-\pa_{x_j}u(\bx_j',y_j,\bx_j'',t)|
\leq C\|g\|_{\WF,0,\gamma}|\sqrt{x_j}-\sqrt{y_j}|^{\gamma}.
\end{equation}
To do the ``off-diagonal'' case we use Lemma~\ref{lem21new}.  To estimate
$$|\pa_{x_l}u(\bx_m',x_m,\bx_m'',t)-\pa_{x_l}u(\bx_m',y_m,\bx_m'',t)|,\text{
  for }l\neq m,$$ 
we observe that
\begin{multline}
  \pa_{x_l}u(\bx_m',x_m,\bx_m'',t)-\pa_{x_l}u(\bx_m',y_m,\bx_m'')=
\int\limits_0^t\int\limits_{0}^{\infty}\cdots\int\limits_{0}^{\infty}
\prod_{j\neq l,m}^nk^{b_j}_s(x_j,z_j)\times\\
[k^{b_m}_s(x_m,z_m)-k^{b_m}_s(y_m,z_m)]\pa_{x_l}k^{b_l}_s(x_l,z_l)[g(\bz,t-s)-g(\bz'_l,x_l,\bz''_l,t-s)]
d\bz ds.
\end{multline}
Putting absolute values into the integral and using~\eqref{nhtdholdest00},
and~\eqref{yintis1},  gives:
\begin{multline}
  |\pa_{x_l}u(\bx_m',x_m,\bx_m'',t)-\pa_{x_l}u(\bx_m',y_m,\bx_m'')|\leq 
\|g\|_{\WF,0,\gamma}\times\\
\int\limits_0^t\int\limits_{0}^{\infty}\int\limits_{0}^{\infty}
|k_s^{b_m}(x_m,z_m)-k_s^{b_m}(y_m,z_m)|
|\pa_{x_l}k^{b_l}_s(x_l,z_l)||\sqrt{z_l}-\sqrt{x_l}|^{\gamma}
dz_mdz_l ds.
\end{multline}

If $y_m=0,$ then applying~\eqref{lem21newest1} we see that this is estimated by
\begin{equation}
  C\|g\|_{\WF,0,\gamma}\int\limits_{0}^{t}s^{\frac{\gamma}{2}-1}\frac{x_m/s}{1+x_m/s}ds,
\end{equation}
which is easily seen to be bounded by
$C\|g\|_{\WF,0,\gamma}x_m^{\frac{\gamma}{2}}.$ Applying~\eqref{lrgratioest} we
see that, if $0<c<1,$ then there is a constant $C$ so that for $y_m<cx_m,$ we have
\begin{equation}\label{frstdroffdiaghldestn0}
  |\pa_{x_l}u(\bx_m',x_m,\bx_m'',t)-\pa_{x_l}u(\bx_m',y_m,\bx_m'')|\leq 
C\|g\|_{\WF,0,\gamma}|\sqrt{x_m}-\sqrt{y_m}|^{\gamma}.
\end{equation}
We are therefore reduced to considering $cx_m<y_m<x_m,$ for a $c<1.$  If we
use~\eqref{lem21newest2} it follows that
\begin{multline}
  |\pa_{x_l}u(\bx_m',x_m,\bx_m'',t)-\pa_{x_l}u(\bx_m',y_m,\bx_m'')|\leq 
C\|g\|_{\WF,0,\gamma}\times\\
\int\limits_0^ts^{\frac{\gamma}{2}-1}\left(\frac{\frac{\sqrt{x_m}-\sqrt{y_m}}{\sqrt{s}}}
{1+\frac{\sqrt{x_m}-\sqrt{y_m}}{\sqrt{s}}}\right)ds.
\end{multline}
We split this into an integral from $0$ to $(\sqrt{x_m}-\sqrt{y_m})^2$ and the
rest, to obtain:
\begin{multline}
  |\pa_{x_l}u(\bx_m',x_m,\bx_m'',t)-\pa_{x_l}u(\bx_m',y_m,\bx_m'',t)|\leq 
C\|g\|_{\WF,0,\gamma}\times\\
\left[\int\limits_0^{(\sqrt{x_m}-\sqrt{y_m})^2}s^{\frac{\gamma}{2}-1}+
\int\limits_{(\sqrt{x_m}-\sqrt{y_m})^2}^ts^{\frac{\gamma}{2}-1}\left(\frac{\sqrt{x_m}-\sqrt{y_m}}{\sqrt{s}}
\right)ds\right].
\end{multline}
Performing these integrals shows that~\eqref{frstdroffdiaghldestn0} holds
in this case as well.

To complete the analysis of $\pa_{x_j}u(\bx,t)$ we need to show that there is a
constant so that
\begin{equation}\label{n01stdertest}
  |\pa_{x_j}u(\bx,t_2)-\pa_{x_j}u(\bx,t_1)|\leq C\|g\|_{\WF,0,\gamma}|t_2-t_1|^{\frac{\gamma}{2}}.
\end{equation}
This follows immediately from the 1-dimensional
argument. Using~\eqref{n0t1stderest}, we see that for any $c<1,$ there is a $C$
so that this estimate holds for $t_1<ct_2.$ As in the 1-dimensional case, we now
assume that $t_1<t_2<2t_1.$ Without loss of
generality we can take $j=n;$
use~\eqref{1dertmprts} to re-express $\pa_{x_n}u(\bx,t_2)-\pa_{x_n}u(\bx,t_1)$ as
\begin{multline}\label{eqn218.00}
  \pa_{x_n}u(\bx,t_2)-\pa_{x_n}u(\bx,t_1)=\\
\int\limits_{0}^{t_2-t_1}\int\limits_0^{\infty}\cdots\int\limits_0^{\infty}
\prod_{j=1}^{n-1}k^{b_j}_s(x_j,z_j)\pa_{x_n}k^{b_n}_s(x_n,z_n)[g(\bz_n',z_n,t_2-s)-
g(\bz_n',z_n,t_1-s)]dz_nd\bz'_nds+\\
\int\limits_{0}^{2t_1-t_2}\int\limits_0^{\infty}\cdots\int\limits_0^{\infty}
\pa_{x_n}\left[\prod_{j=1}^{n}k^{b_j}_{t_2-s}(x_j,z_j)-
\prod_{j=1}^{n}k^{b_j}_{t_1-s}(x_j,z_j)\right]\times\\
[g(\bz_n',z_n,s)-
g(\bz_n',x_n,s)]dz_nd\bz'_nds+\\
\int\limits_{2t_1-t_2}^{t_1}\int\limits_0^{\infty}\cdots\int\limits_0^{\infty}
\prod_{j=1}^{n-1}k^{b_j}_{t_2-s}(x_j,z_j)\pa_{x_n}k^{b_n}_{t_2-s}(x_n,z_n)[g(\bz_n',z_n,s)-
g(\bz_n',x_n,s)]dz_nd\bz'_nds.
\end{multline}
In the first integral we replace $g(\bz_n',z_n,t_j-s)$ with
$g(\bz_n',z_n,t_j-s)-
g(\bz_n',x_n,t_j-s),$ for $j=1,2,$ and then apply Lemma~\ref{lem2new}, as in
the 1-dimensional case, to show that this term is bounded by the right hand
side of~\eqref{n01stdertest}. A similar argument is applied to estimate the
third integral.

To handle the second term we use formula~\eqref{n0ktksdif} to conclude that
\begin{multline}\label{1stdern0ktksdif}
 \pa_{x_n}\left[ \prod_{l=1}^{n}k_{t_2}^{b_l}(x_l,z_l)-\prod_{l=1}^{n}k_{t_1}^{b_l}(x_l,z_l)\right]=\\
\sum\limits_{m=1}^{n} \pa_{x_n}\Bigg\{
\prod_{l=1}^{n-m}k_{t_2}^{b_l}(x_l,z_l)\times\prod_{l=n-m+2}^{n}k_{t_1}^{b_l}(x_l,z_l)\times\\
\left[k_{t_2}^{b_{n-m+1}}(x_{n-m+1},z_{n-m+1})-k_{t_1}^{b_{n-m+1}}(x_{n-m+1},z_{n-m+1})\right]\Bigg\}.
\end{multline}
To estimate the contribution to the second integral coming from the term
in~\eqref{1stdern0ktksdif}with $m=1,$ we observe that
\begin{equation}
  |g(\bz_n',z_n,s)-
g(\bz_n',x_n,s)|\leq 2\|g\|_{\WF,0,\gamma}|\sqrt{x_n}-\sqrt{z_n}|^{\gamma},
\end{equation}
and apply Lemma~\ref{lemA-}. The contributions of the other terms are
bounded by
\begin{multline}
 C  \|g\|_{\WF,0,\gamma}\int\limits_{0}^{2t_1-t_2}\int\limits_0^{\infty}\int\limits_0^{\infty}
|k^{b_j}_{t_2-s}(x_j,z_j)-k^{b_j}_{t_1-s}(x_j,z_j)|\times\\
|\pa_{x_n}k^{b_n}_{t_1-s}(x_n,z_n)|
|\sqrt{x_n}-\sqrt{z_n}|^{\gamma}dz_j dz_nds.
\end{multline}
We use Lemma~\ref{lem2new} and Lemma~\ref{lem4newp2} to see that, upon setting
$\sigma=t_1-s,$  this integral is bounded by
\begin{equation}
     C  \|g\|_{\WF,0,\gamma}\int\limits_{t_2-t_1}^{t_1}
\frac{\sigma^{\frac{\gamma-1}{2}}(t_2-t_1)}{(\sqrt{\sigma}+\sqrt{x_n})(t_2-t_1+\sigma)}
\leq C  \|g\|_{\WF,0,\gamma}|t_2-t_1|^{\frac{\gamma}{2}}.
  \end{equation}
This completes the proof that there is a constant $C$ so that
\begin{equation}\label{1stdertmestn0.0}
  |\pa_{x_j}u(\bx,t_1)-\pa_{x_j}u(\bx,t_2)|\leq C\|g\|_{\WF,0,\gamma}|t_2-t_1|^{\frac{\gamma}{2}}.
\end{equation}
The fact that there is a constant $C$ so that
\begin{equation}
  |\nabla_{\bx} u(\bx_1,t_1)-\nabla_{\bx} u(\bx_2,t_2)|\leq
C\|g\|_{\WF,0,\gamma}[\rho_s(\bx_1,\bx_2)+|t_2-t_1|^{\frac{1}{2}}]^{\gamma}
\end{equation}
now follows from the foregoing estimates and Lemma~\ref{lem1}.

An estimate showing the boundedness of
$|x_j\pa_{x_j}^2u(\bx,t)|$ is given in~\eqref{2ndderestn0}. We can use
Lemma~\ref{lem2new} to prove an analogous estimate for the mixed partial
derivatives. Arguing as above, we easily establish that, for $j\neq k,$ $x_j$ and $x_k$
both positive, we have
\begin{multline}
  |\pa_{x_j}\pa_{x_k}u(\bx,t)|\leq
\|g\|_{\WF,0,\gamma}\int\limits_{0}^{t}\int\limits_0^{\infty}
\int\limits_0^{\infty}|\pa_{x_k}k^{b_j}_{s}(x_k,z_k)|\times\\
|\pa_{x_j}k^{b_j}_{s}(x_j,z_j)||g(\bz'_j,z_j,\bz''_j,t-s)-
g(\bz'_j,x_j,\bz''_j,t-s)|dz_jdz_kds
\end{multline}
Lemma~\ref{lem2new} applies to show that this quantity is bounded by
\begin{equation}
  C\|g\|_{\WF,0,\gamma}\int\limits_{0}^{t}\frac{s^{\frac{\gamma}{2}-1}ds}
{\sqrt{x_jx_k}},
\end{equation}
which implies that
\begin{equation}\label{eqn211.0}
  |\sqrt{x_jx_k}\pa_{x_j}\pa_{x_k}u(\bx,t)|\leq
C\|g\|_{\WF,0,\gamma}t^{\frac{\gamma}{2}},
\end{equation}
which is our first indication that these mixed derivatives are ``small''
perturbations of the principal terms in $L_{\bb,0}.$ All that remains to
complete the estimate of spatial derivatives is the proof of 
the H\"older continuity of the second derivatives of $u.$

We begin by proving the H\"older continuity of $x_j\pa_{x_j}^2u.$ The estimates
in~\eqref{2ndderestn0} and~\eqref{lrgratioest} show that for any $c<1,$ there
is a $C$ so that if $y_j<cx_j,$ then
\begin{equation}\label{2ndhldrdiagestn0}
  |x_j\pa_{x_j}^2u(\bx_j',x_j,\bx_j'',t)-y_j\pa_{x_j}^2u(\bx_j',y_j,\bx_j'',t)|
\leq C\|g\|_{\WF,0,\gamma}|\sqrt{x_j}-\sqrt{y_j}|^{\gamma}.
\end{equation}
Thus in the ``diagonal'' case we only need to consider $cx_j<y_j<x_j.$ The
proof in this case follows exactly as in the one dimensional case; we establish the
H\"older continuity of $L_{b_j,x_j}u,$ which is sufficient, as we have already
done so for the first derivatives. To do this we express the difference:
\begin{equation}
L_{b_j,x_j}u(\bx_j',x_j,\bx_j'',t)-L_{b_j,x_j}u(\bx_j',y_j,\bx_j'',t),
\end{equation}
using~\eqref{dlctest1} in the $j$-variable, much like the formula
in~\eqref{vdfn0}. The estimate for each term in~\eqref{dlctest1} carries over
to the present situation to immediately establish that~\eqref{2ndhldrdiagestn0}
holds for a suitable $C,$ for all pairs $(x_j,y_j).$  

To finish the spatial estimate in this case we need to consider the
``non-diagonal'' situation. With $j\neq m,$ we express this difference as
\begin{multline}\label{2ndhldrnondiag}
  x_j\pa_{x_j}^2u(\bx_m',x_m,\bx_m'',t)-x_j\pa_{x_j}^2u(\bx_m',y_m,\bx_m'',t)=\\
\int\limits_0^{t}\int\limits_0^{\infty}\cdots\int\limits_0^{\infty}
\prod\limits_{l\neq j,m}k^{b_l}_s(x_l,z_l)
x_j\pa_{x_j}^2k^{b_j}_s(x_j,z_j)
[k^{b_m}_s(x_m,z_m)-k^{b_m}_s(y_m,z_m)]\times\\
[g(\bz_j',z_j,\bz_j'',t-s)-g(\bz_j',x_j,\bz_j'',t-s)]d\bz ds.
\end{multline}
From this formula it follows that
\begin{multline}\label{2ndhldrnondiag1}
  |x_j\pa_{x_j}^2u(\bx_m',x_m,\bx_m'',t)-x_j\pa_{x_j}^2u(\bx_m',y_m,\bx_m'',t)|\leq\\
2\|g\|_{\WF,0,\gamma}\int\limits_0^{t}\int\limits_0^{\infty}\int\limits_0^{\infty}
|x_j\pa_{x_j}^2k^{b_j}_s(x_j,z_j)||\sqrt{x_j}-\sqrt{z_j}|^{\gamma}\times\\
|k^{b_m}_s(x_m,z_m)-k^{b_m}_s(y_m,z_m)]|dz_jdz_m ds.
\end{multline}
This case is completed by employing Lemmas~\ref{lem21new} and~\ref{lem25new}.

We begin with case that $y_m=0;$   Lemmas~\ref{lem21new}
and~\ref{lem25new} in~\eqref{2ndhldrnondiag1} show that
\begin{multline}\label{2ndhldrnondiag2}
  |x_j\pa_{x_j}^2u(\bx_m',x_m,\bx_m'',t)-x_j\pa_{x_j}^2u(\bx_m',0,\bx_m'',t)|\leq\\
2\|g\|_{\WF,0,\gamma}\int\limits_0^{t}s^{\frac{\gamma}{2}-1}\frac{x_m/s}{1+x_m/s}ds.
\end{multline}
This is easily seen to be bounded by
$C\|g\|_{\WF,0,\gamma}x_m^{\frac{\gamma}{2}}.$ We are therefore left to
consider the case $cx_m<y_m<x_m,$ for any $c<1.$ We now use the second estimate
in Lemma~\ref{lem21new} to see that
\begin{multline}\label{2ndhldrnondiag3}
  |x_j\pa_{x_j}^2u(\bx_m',x_m,\bx_m'',t)-x_j\pa_{x_j}^2u(\bx_m',y_m,\bx_m'',t)|\leq\\
\|g\|_{\WF,0,\gamma}\int\limits_0^{t}s^{\frac{\gamma}{2}-1}\frac{\frac{\sqrt{x_m}-\sqrt{y_m}}{\sqrt{s}}}
{1+\frac{\sqrt{x_m}-\sqrt{y_m}}{\sqrt{s}}}ds.
\end{multline}
An elementary argument shows that the right hand side is bounded by
\begin{equation}
  C\|g\|_{\WF,0,\gamma}|\sqrt{x_m}-\sqrt{y_m}|^{\gamma}.
\end{equation}
This completes the proof of the spatial part of the H\"older estimates for
$x_j\pa_{x_j}^2u(\bx,t).$ We next turn to the time estimate.

From~\eqref{2ndderestn0} it follows that
\begin{equation}\label{eqn361new1.0}
 |x_j\pa_{x_j}^2u(\bx,t)|\leq
 C\|g\|_{\WF,0,\gamma}t^{\frac{\gamma}{2}}.
\end{equation}
This shows that $x_j\pa_{x_j}^2u(\bx,t)$ tends uniformly to zero like $t^{\frac{\gamma}{2}}.$
Applying~\eqref{lrgratioest} we see that for $c<1,$ there is a $C$ so that, if
$s<ct,$ then
\begin{equation}
  |x_j\pa_{x_j}^2u(\bx,t)-x_j\pa_{x_j}^2u(\bx,s)|\leq C\|g\|_{\WF,0,\gamma}|t-s|^{\frac{\gamma}{2}}
\end{equation}
We are therefore left to consider the case $t_1<t_2<2t_1.$ To handle this case
we begin with the formula from~\eqref{1d2dertest}
\begin{multline}\label{n0d2dertest}
  L_{b_j}u(\bx,t_2)-L_{b_j}u(\bx,t_1)=\\
\int\limits_{0}^{t_2-t_1}
\int\limits_{0}^{\infty}\cdots\int\limits_{0}^{\infty}\prod\limits_{l\neq j}k^{b_l}_s(x_l,z_l)
L_{b_j}k^{b_j}_{s}(x_j,z_j)[g(\bz,t_2-s)-g(\bz,t_1-s)]d\bz ds+\\
\int\limits_{0}^{2t_1-t_2}\int\limits_{0}^{\infty}\cdots\int\limits_{0}^{\infty}
L_{b_j,x_j}\left[
\prod\limits_{l=1}^{n}k^{b_l}_{t_2-s}(x_l,z_l)-
\prod\limits_{l=1}^{n}k^{b_l}_{t_1-s}(x_l,z_l)\right] \\ \times [g(\bz,s)-g(\bz'_j,x_j,\bz''_j,s)]d\bz
ds+\\
\int\limits_{2t_1-t_2}^{t_1}\int\limits_{0}^{\infty}\cdots\int\limits_{0}^{\infty}
\prod\limits_{l\neq
  j}k^{b_l}_{t_2-s}(x_l,z_l)L_{b_j}k^{b_j}_{t_2-s}(x_j,z_j)[g(\bz,s)-g(\bz'_j,x_j,\bz''_j,s)]d\bz
ds.
\end{multline}
In the first integral, as in the 1-dimensional case, we use the estimates
\begin{equation}
  |g(\bz,t_q-s)-g(\bz'_j,x_j,\bz''_j,t_q-s)|\leq 2\|g\|_{\WF,0,\gamma}|\sqrt{x_j}-\sqrt{z_j}|^{\gamma};
\end{equation}
here $q=1,2.$ In the last integral in~\eqref{n0d2dertest} we use the estimate
\begin{equation}\label{eqn22500}
  |g(\bz,s)-g(\bz'_j,x_j,\bz''_j,s)|\leq 2\|g\|_{\WF,0,\gamma}|\sqrt{x_j}-\sqrt{z_j}|^{\gamma}.
\end{equation}
This immediately reduces these cases to 1-dimensional case, and these terms are
therefore bounded by $C\|g\|_{\WF,0,\gamma}|t_2-t_1|^{\frac{\gamma}{2}}.$

To handle the second term we use formula~\eqref{n0ktksdif} to conclude that
\begin{multline}\label{2nddern0ktksdif}
 L_{b_j,x_j}\left[ \prod_{l=1}^{n}k_{t_2}^{b_l}(x_l,z_l)-\prod_{l=1}^{n}k_{t_1}^{b_l}(x_l,z_l)\right]=\\
\sum\limits_{m=1}^{n} L_{b_j,x_j}\Bigg\{
\prod_{l=1}^{n-m}k_{t_2}^{b_l}(x_l,z_l)\times\prod_{l=n-m+2}^{n}k_{t_1}^{b_l}(x_l,z_l)\times\\
\left[k_{t_2}^{b_{n-m+1}}(x_{n-m+1},z_{n-m+1})-k_{t_1}^{b_{n-m+1}}(x_{n-m+1},z_{n-m+1})\right]\Bigg\}.
\end{multline}
There are now two types of terms: those with $n-m+1\neq j,$ and the term with
$n-m+1=j.$ In all cases we use the estimate in~\eqref{eqn22500}. With this
understood, the term with $n-m+1=j$ immediately reduces to the 1-dimensional
case. Terms where $n-m+1\neq j$ are bounded by:
\begin{multline}
I=\|g\|_{\WF,0,\gamma} \int\limits_{0}^{2t_1-t_2}\int\limits_0^{\infty}\int\limits_0^{\infty}
|k^{b_{l}}_{t_2-s}(x_l,z_l)-k^{b_{l}}_{t_1-s}(x_l,z_l)|\times\\
|L_{b_j}k^{b_j}_{t_q-s}(x_j,z_j)|
|\sqrt{x_j}-\sqrt{z_j}|^{\gamma} dz_ldz_jds;
\end{multline}
here $q=1$ or $2.$ Using Lemmas~\ref{lem2new},~\ref{lem4newp2}
and~\ref{lem25new} we see that
\begin{equation}
  I\leq C\|g\|_{\WF,0,\gamma}
  \int\limits_{t_2-t_1}^{t_1}(s+(q-1)\tau)^{\frac{\gamma}{2}-1}\frac{\tau ds}{s+\tau},
\end{equation}
with $\tau=t_2-t_1.$ The case $q=1$ clearly produces a larger value. In this
case we set $w=s/\tau,$ obtaining 
  \begin{equation}
  I\leq C\|g\|_{\WF,0,\gamma}
 \tau^{\frac{\gamma}{2}} \int\limits_{1}^{\frac{t_1}{\tau}}w^{\frac{\gamma}{2}-1}\frac{dw}{1+w},
\end{equation}
which completes the proof that, for $j=1,\dots,n,$ we have:
\begin{equation}
  |L_{b_j}u(\bx,t_2)-L_{b_j}u(\bx,t_1)|\leq C\|g\|_{\WF,0,\gamma}|t_2-t_1|^{\frac{\gamma}{2}}
\end{equation}

To finish the proof of the Proposition we need to show that the mixed
derivatives $\sqrt{x_jx_l}\pa_{x_j}\pa_{x_l}u$ are H\"older continuous. Here
there are two cases depending upon whether  the variable that is allowed
to vary is one of $x_j, x_l$ or not. The latter case is immediate from Lemmas
we have already proved. Let $m\neq j$ or $l,$ then we easily see that
\begin{multline}\label{eqn10.153.01}
  |\sqrt{x_jx_l}\pa_{x_j}\pa_{x_l}u(\bx'_m,x_m,\bx''_m,t)-
\sqrt{x_jx_l}\pa_{x_j}\pa_{x_l}u(\bx'_m,y_m,\bx''_m,t)|\leq\\
\|g\|_{\WF,0,\gamma}
\int\limits_0^t\int\limits_0^{\infty}\int\limits_0^{\infty}\int\limits_0^{\infty}
|k^{b_m}_s(x_m,z_m)-k^{b_m}_s(y_m,z_m)|\times\\
|\sqrt{x_j}\pa_{x_j}k^{b_j}_s(x_j,z_j)||\sqrt{x_l}\pa_{x_l}k^{b_l}_s(x_l,z_l)|
|\sqrt{z_l}-\sqrt{y_l}|^{\gamma}dz_mdz_ldz_jds.
\end{multline}
We first let $y_m=0$ and use the first estimate in Lemma~\ref{lem21new} to bound the
$z_m$-integral, and Lemma~\ref{lem2new} to estimate the other two. This shows
that this expression is bounded by
\begin{equation}
  \|g\|_{\WF,0,\gamma} \int\limits_0^t\left( \frac{x_m}{s+x_m}\right)s^{\frac{\gamma}{2}-1}ds.
\end{equation}
This is bounded by $C \|g\|_{\WF,0,\gamma} x_m^{\frac{\gamma}{2}},$ which
allows us to restrict to the case that $cx_m<y_m<x_m,$ for a $c<1.$ Applying the
other estimate in Lemma~\ref{lem21new} we easily deduce that
\begin{multline}\label{eqn10.156.00}
  |\sqrt{x_jx_l}\pa_{x_j}\pa_{x_l}u(\bx'_m,x_m,\bx''_m,t)-
\sqrt{x_jx_l}\pa_{x_j}\pa_{x_l}u(\bx'_m,y_m,\bx''_m,t)|\leq\\ C
\|g\|_{\WF,0,\gamma}|\sqrt{x_m}-\sqrt{y_m}|^{\gamma}. 
\end{multline}

Now suppose that $m=l$ and $y_l=0.$ In this case we see that
\begin{multline}\label{eqn10.156.7}
  |\sqrt{x_jx_l}\pa_{x_j}\pa_{x_l}u(\bx'_l,x_l,\bx''_l,t)|\leq\\
\|g\|_{\WF,0,\gamma}
\int\limits_0^t\int\limits_0^{\infty}\int\limits_0^{\infty}
|\sqrt{x_j}\pa_{x_j}k^{b_j}_s(x_j,z_j)||\sqrt{x_l}\pa_{x_l}k^{b_l}_s(x_l,z_l)|
|\sqrt{z_l}-\sqrt{x_l}|^{\gamma}dz_ldz_jds.
\end{multline}
We apply Lemma~\ref{lem2new} to see that this is bounded by
\begin{equation}
\int\limits_0^t\left(\frac{\sqrt{x_l}s^{\frac{\gamma-1}{2}}}{\sqrt{s}+\sqrt{x_l}}\right)
\left(\frac{\sqrt{x_j}s^{-\frac{1}{2}}}{\sqrt{s}+\sqrt{x_j}}\right)ds.
 \end{equation}
An elementary argument shows that 
\begin{equation}\label{378new1.0}
 |\sqrt{x_jx_l}\pa_{x_j}\pa_{x_l}u|\leq C\|g\|_{\WF,0,\gamma}
\min\{x_j^{\frac{\gamma}{2}},x_l^{\frac{\gamma}{2}},t^{\frac{\gamma}{2}}\}.
\end{equation}
This estimate implies that
\begin{equation}\label{eqn346new.0}
  \lim_{x_j\vee x_l\to 0^+}\sqrt{x_jx_l}\pa_{x_j}\pa_{x_l}u(\bx,t)=0.
\end{equation}

In light of~\eqref{lrgratioest} all that remains is to consider $cx_l<y_l<x_l,$
for a $0<c<1,$ for which we require an estimate of the quantity:
\begin{equation}
    \int\limits_0^{\infty}|\sqrt{x_1}\pa_{x}k^b_t(x_1,y)-
\sqrt{x_2}\pa_{x}k^b_t(x_2,y)||\sqrt{x_1}-\sqrt{y}|^{\gamma}dy.
  \end{equation}
We now show how to
use~\eqref{eqn2400.0}, and the estimate in Lemma~\ref{lem20neww}, to prove the spatial H\"older
estimate for $\sqrt{x_jx_l}\pa_{x_j}\pa_{x_l}u,$  with respect to $x_j$ and
$x_l.$ 
\begin{multline}
  |\sqrt{x_jx_l}\pa_{x_j}\pa_{x_l}u(\bx'_l,x_l,\bx''_l,t)-
\sqrt{x_jy_l}\pa_{x_j}\pa_{x_l}u(\bx'_l,y_l,\bx''_l,t)|\leq\\
2\|g\|_{\WF,0,\gamma}
\int\limits_0^t\int\limits_0^{\infty}\int\limits_0^{\infty}
|\sqrt{x_j}\pa_{x_j}k^{b_j}_s(x_j,z_j)|\times\\
|\sqrt{x_l}\pa_{x_l}k^{b_l}_s(x_l,z_l)-\sqrt{y_l}\pa_{x_l}k^{b_l}_s(y_l,z_l)|
|\sqrt{z_l}-\sqrt{y_l}|^{\gamma}dz_ldz_jds.
\end{multline}
We apply Lemmas~\ref{lem2new} and~\ref{lem20neww} to see that this integral is
bounded by
\begin{multline}\label{eqn255.00}
C  \int\limits_0^ts^{\frac{\gamma}{2}-1}\frac{\left(\frac{|\sqrt{x_l}-\sqrt{y_l}|}{s}\right)}
{1+\left(\frac{|\sqrt{x_l}-\sqrt{y_l}|}{s}\right)}\leq
\\
C\left[\int\limits_0^{|\sqrt{x_l}-\sqrt{y_l}|^2}s^{\frac{\gamma}{2}-1}+
\int\limits_{|\sqrt{x_l}-\sqrt{y_l}|^2}^ts^{\frac{\gamma-3}{2}}|\sqrt{x_l}-\sqrt{y_l}|ds\right].
\end{multline}
Here we  implicitly assume that $t>|\sqrt{x_l}-\sqrt{y_l}|^2;$ if this is not the
case then only the first integral on the right side of~\eqref{eqn255.00} is
needed. In either case we easily see that the right hand side is bounded by
$C|\sqrt{x_l}-\sqrt{y_l}|^{\gamma}.$ 

To finish the proof of the proposition all that is remains is to show that
these derivatives are H\"older continuous with respect to time. The estimate
in~\eqref{eqn211.0} and~\eqref{lrgratioest} show that if $0<c<1,$ then there is
a $C$ so that for $0<t_1<ct_2,$ we have the estimate
\begin{equation}
  |\sqrt{x_jx_l}\pa_{x_j}\pa_{x_l}u(\bx,t_2)-\sqrt{x_jx_l}\pa_{x_j}\pa_{x_l}u(\bx,t_1)|\leq
C\|g\|_{\WF,0,\gamma}|t_2-t_1|^{\frac{\gamma}{2}},
\end{equation}
thus we are left to consider on the case $t_1<t_2<2t_1.$ To that end we express
\begin{multline}\label{n0d2dertimest}
\sqrt{x_jx_l}\pa_{x_j}\pa_{x_l}[ u(\bx,t_2)-u(\bx,t_1)]=\\
\int\limits_{0}^{t_2-t_1}
\int\limits_{0}^{\infty}\cdots\int\limits_{0}^{\infty}\prod\limits_{m\neq j,l}k^{b_m}_s(x_m,z_m)
\sqrt{x_j}\pa_{x_j}k^{b_j}_{s}(x_j,z_j)
\sqrt{x_l}\pa_{x_l}k^{b_l}_{s}(x_l,z_l)\times\\
[(g(\bz,t_2-s)-g(\bz'_j,x_j,\bz''_j,t_2-s))+(g(\bz'_j,x_j,\bz''_j,t_1-s)-g(\bz,t_1-s))]d\bz ds+\\
\int\limits_{0}^{2t_1-t_2}\int\limits_{0}^{\infty}\cdots\int\limits_{0}^{\infty}
\sqrt{x_jx_l}\pa_{x_j}\pa_{x_l}\left[
\prod\limits_{m=1}^{n}k^{b_m}_{t_2-s}(x_m,z_m)-
\prod\limits_{m=1}^{n}k^{b_m}_{t_1-s}(x_m,z_m)\right]\times\\
g(\bz,s)d\bz
ds+\\
\int\limits_{2t_1-t_2}^{t_1}\int\limits_{0}^{\infty}\cdots\int\limits_{0}^{\infty}
\prod\limits_{m\neq
  j,l}k^{b_m}_{t_2-s}(x_m,z_m)\sqrt{x_j}\pa_{x_j}k^{b_j}_{t_2-s}(x_j,z_j)
\sqrt{x_l}\pa_{x_l}k^{b_l}_{t_2-s}(x_l,z_l)\times \\
[g(\bz,s)-g(\bz'_j,x_j,\bz''_j,s)]d\bz
ds.
\end{multline}
In the first integral, as in the 1-dimensional case, we use the estimates
\begin{equation}
  |g(\bz,t_q-s)-g(\bz'_j,x_j,\bz''_j,t_q-s)|\leq 2\|g\|_{\WF,0,\gamma}|\sqrt{x_j}-\sqrt{z_j}|^{\gamma};
\end{equation}
here $q=1,2.$ In the last integral in~\eqref{n0d2dertimest} we use the estimate
\begin{equation}\label{eqn225000}
  |g(\bz,s)-g(\bz'_j,x_j,\bz''_j,s)|\leq 2\|g\|_{\WF,0,\gamma}|\sqrt{x_j}-\sqrt{z_j}|^{\gamma}.
\end{equation}
Applying Lemma~\ref{lem2new} we see that these terms are bounded by
\begin{equation}
  C\int\limits_{0}^{2(t_2-t_1)}\frac{\sqrt{x_jx_l}s^{\frac{\gamma}{2}-1}}
{(\sqrt{x_j}+\sqrt{s})(\sqrt{x_l}+\sqrt{s})}\leq C|t_2-t_1|^{\frac{\gamma}{2}}.
\end{equation}

All the remains is to estimate the second integral in~\eqref{n0d2dertimest},
where once again we employ formula~\eqref{n0ktksdif}.  In each of the terms
which arise, we can replace $g(\bz,s)$ with either
$[g(\bz,s)-g(\bz'_j,x_j,\bz''_j,s)],$ or $[g(\bz,s)-g(\bz'_l,x_l,\bz''_l,s)],$
without changing the values of these integral. With this understood, there are
only five essentially different cases to consider, depending upon which terms
in the products on the right hand side of~\eqref{n0ktksdif} are
differentiated. We let $\tau=t_2-t_1;$ the cases requiring consideration are
integrands with terms of the form
\begin{enumerate}
\renewcommand{\labelenumi}{\Roman{enumi}.}
\item 
  \begin{equation}\label{eqn261.01}
    |\sqrt{x_j}\pa_{x_j}k^{b_j}_{s}(x_j,z_j)
\sqrt{x_l}\pa_{x_l}[k^{b_l}_{\tau+s}(x_l,z_l)-k^{b_l}_{s}(x_l,z_l)]||\sqrt{x_j}-\sqrt{z_j}|^{\gamma},
  \end{equation}
\item 
  \begin{equation}
    |\sqrt{x_j}\pa_{x_j}k^{b_j}_{\tau+s}(x_j,z_j)
\sqrt{x_l}\pa_{x_l}[k^{b_l}_{\tau+s}(x_l,z_l)-k^{b_l}_{s}(x_l,z_l)]||\sqrt{x_j}-\sqrt{z_j}|^{\gamma},
  \end{equation}
\item With $m\neq j$ or $l$
  \begin{multline}
    |\sqrt{x_j}\pa_{x_j}k^{b_j}_{s}(x_j,z_j)
\sqrt{x_l}\pa_{x_l}k^{b_l}_{s}(x_l,z_l)\times\\
[k^{b_m}_{\tau+s}(x_m,z_m)-k^{b_m}_{s}(x_m,z_m)]||\sqrt{x_j}-\sqrt{z_j}|^{\gamma},
  \end{multline}
\item With $m\neq j$ or $l$
  \begin{multline}
    |\sqrt{x_j}\pa_{x_j}k^{b_j}_{\tau+s}(x_j,z_j)
\sqrt{x_l}\pa_{x_l}k^{b_l}_{s}(x_l,z_l)\times\\
[k^{b_m}_{\tau+s}(x_m,z_m)-k^{b_m}_{s}(x_m,z_m)]||\sqrt{x_j}-\sqrt{z_j}|^{\gamma},
  \end{multline}
\item With $m\neq j$ or $l$
\begin{multline}\label{eqn265.01}
    |\sqrt{x_j}\pa_{x_j}k^{b_j}_{\tau+s}(x_j,z_j)
\sqrt{x_l}\pa_{x_l}k^{b_l}_{\tau+s}(x_l,z_l)\times\\
[k^{b_m}_{\tau+s}(x_m,z_m)-k^{b_m}_{s}(x_m,z_m)]||\sqrt{x_j}-\sqrt{z_j}|^{\gamma}.
  \end{multline}
\end{enumerate}
Applying Lemmas~\ref{lem2new} and~\ref{lem4newp2} we see that the integrals
of types III, IV and V are all bounded by
\begin{equation}
  C\int\limits_{\tau}^{\infty}\frac{s^{\frac{\gamma}{2}-1}\tau
    ds}{\tau+s}=
C\tau^{\frac{\gamma}{2}}\int\limits_1^{\infty}\frac{\sigma^{\frac{\gamma}{2}-1}d\sigma}{1+\sigma}\leq
C\tau^{\frac{\gamma}{2}},
\end{equation}
leaving just the terms of types I and II. These are estimated using
Lemma \ref{lem2new} and Lemma~\ref{lemAA-}. Both of these terms are bounded by
\begin{equation}
  C\int\limits_{\tau}^{\infty}\frac{\sqrt{x_jx_l}s^{\frac{\gamma}{2}-1}\tau ds}
{(\sqrt{x_j}+\sqrt{s})(\sqrt{x_l}+\sqrt{s})(\tau+s)}\leq
C\int\limits_{\tau}^{\infty}s^{\frac{\gamma}{2}-2}\tau ds\leq C\tau^{\frac{\gamma}{2}}.
\end{equation}

This completes the proof that
\begin{equation}
  |\sqrt{x_jx_l}\pa_{x_j}\pa_{x_l}u(\bx,t_2)-\sqrt{x_jx_l}\pa_{x_j}\pa_{x_l}u(\bx,t_1)|\leq
C\|g\|_{\WF,0,\gamma}|t_2-t_1|^{\frac{\gamma}{2}}.
\end{equation}
Using the estimates on the spatial derivatives, and the differential
equation~\eqref{n0dmdlb}, we easily establish that
$\pa_tu\in\cC^{0,\gamma}_{\WF}(\bbR_+^n\times\bbR_+),$ satisfies the desired
estimates. The estimates in~\eqref{eqn302new.0} and~\eqref{eqn346new.0} show
that the appropriate scaled second derivatives tend to zero along portions of
$b\bbR_+^n.$ The argument applied in the 1-dimensional case  to show that
(see equations~\eqref{eqn193.1.5} to~\eqref{eqn256new.1.1})
\begin{equation}
  \lim_{x\to\infty}[|u(x,t)|+|\pa_xu(x,t)|+x\pa_x^2u(x,t)|]=0,
\end{equation}
applies \emph{mutatis mutandis} to show that 
\begin{equation}
  \lim_{\|\bx\|\to\infty}[|u(\bx,t)|+|\nabla_{\bx}u(\bx,t)|
+\sum_{k,l}|\sqrt{x_kx_l}\pa_{x_k}\pa_{x_l}u(\bx,t)|]=0.
\end{equation}
One merely needs to observe that, if
\begin{equation}
  \Phi_{R}(\bx)=\prod\limits_{j=1}^n\varphi_{R,R}(x_j),
\end{equation}
then 
\begin{enumerate}
\item $[1-\Phi_{R}(\bx)]g$ tends to zero in
  $\cC^{0,\gamma}_{\WF}(\bbR_+^n\times [0,T]),$ as $R\to\infty.$
\item For any fixed $R$ the solution
  \begin{equation}
    u^0_R(\bx,t)=\int\limits_{0}^t\int\limits_0^{\infty}
\dots\int\limits_0^{\infty}\prod\limits_{j=1}^nk^{b_j}_s(x_j,y_j)\Phi_{R}(\by)g(\by,t-s)
d\by ds
  \end{equation}
along with all derivatives, tends rapidly to zero as $\|\bx\|\to\infty.$
\end{enumerate}
These observations and the various H\"older estimates established above imply
that we can apply Lemma~\ref{lem4.9new0.0} to conclude that
\begin{equation}
  u\in\cC^{0,2+\gamma}_{\WF}(\bbR_+^n\times [0,T]).
\end{equation}

This completes the proof of the proposition in the $k=0$ case. As in
the 1-dimensional case, we can use Proposition~\ref{lem3.4new0.0} to
commute derivatives $\pa_t^j\pa_{\bx}^{\balpha}$ past the kernel
function in the integral representation. Assuming that $g$ has support
in $B_R^+(\bzero)\times [0,T]$ allows us to apply
Proposition~\ref{wtedest_pr} to bound the resultant data in terms of
$\|g\|_{\WF,k,\gamma}.$ Hence we can apply the estimates in the $k=0$
case to establish the estimates in~\eqref{bscest3n0} for all
$k\in\bbN.$
\end{proof}

\section{The Resolvent Operator}
As in the 1-dimensional case we can define the resolvent operator as the
Laplace transform of the heat kernel. For $\mu\in S_0,$ we have the formula
\begin{equation}
  R(\mu)f(\bx)=\lim_{\epsilon\to
    0^+}\int\limits_{\epsilon}^{\frac{1}{\epsilon}}
\int\limits_{0}^{\infty}\cdots\int\limits_{0}^{\infty}
\prod\limits_{j=1}^nk^{b_j}_t(x_j,z_j)f(\bz)d\bz e^{-t\mu} dt.
\end{equation}
Using the asymptotic expansion for the 1-dimensional factors it follows easily
that for each fixed $\bx\in\bbR_+^n,$ $R(\mu)f(\bx)$ is an analytic function of
$\mu.$ Applying Cauchy's theorem we can easily show that, so long as
$\Re[\mu\eit>0],$ we can rewrite this as:
\begin{equation}\label{eqn10.182.00}
  R(\mu)f(\bx)=\lim_{\epsilon\to
    0^+}\int\limits_{\epsilon}^{\frac{1}{\epsilon}}
\int\limits_{0}^{\infty}\cdots\int\limits_{0}^{\infty}
\prod\limits_{j=1}^nk^{b_j}_{s\eit}(x_j,z_j)f(\bz)d\bz e^{-\eit\mu s}\eit ds.
\end{equation}
This shows that $R(\mu)f(\bx)$ extends analytically to $\bbC\setminus
(-\infty,0].$ We close this section by stating a proposition summarizing the
properties of $R(\mu)$ as an operator on the H\"older spaces
$\cC^{k,\gamma}_{\WF}(\bbR_+^n).$ The proof is deferred to the end of
Chapter~\ref{s.genmod} where the analogous result covering all model operators
is proved.
\begin{proposition}\label{prop8.0.4.n0}  The resolvent operator $R(\mu)$ is analytic
  in the complement of $(-\infty,0],$ and is given by the integral
  in~\eqref{eqn10.182.00} provided that $\Re(\mu e^{i\theta})>0.$ For $\alpha\in
  (0,\pi],$ there are constants $C_{b,\alpha}$
 so that if 
 \begin{equation}\label{argest001n}
   \alpha-\pi<\arg{\mu}<\pi-\alpha, \end{equation}
then, for $f\in\cC^0_b(\bbR_+^n)$ we have
  \begin{equation}
    \|R(\mu) f\|_{L^{\infty}}\leq \frac{C_{b,\alpha}}{|\mu|}\|f\|_{L^{\infty}};
  \end{equation}
with $C_{b,\pi}=1.$ Moreover, for $0<\gamma<1,$ there is a constant
  $C_{b,\alpha,\gamma}$ so that if $f\in\cC^{0,\gamma}_{\WF}(\bbR_+^n),$ then
  \begin{equation}\label{eqn10.13.000}
      \|R(\mu) f\|_{\WF,0,\gamma}\leq \frac{C_{b,\alpha,\gamma}}{|\mu|}\|f\|_{\WF,0,\gamma}.
  \end{equation}

If for a $k\in\bbN_0,$ and $0<\gamma<1,$ $f\in\cC^{k,\gamma}_{\WF}(\bbR^n_+),$ then
$R(\mu)f\in\cC^{k,2+\gamma}_{\WF}(\bbR^n_+),$ and, we have
\begin{equation}
  (\mu-L_b)R(\mu) f=f.
\end{equation}
If $f\in \cC^{0,2+\gamma}_{\WF}(\bbR^n_+),$ then
\begin{equation}
  R(\mu)(\mu-L_b) f=f.
\end{equation}
There are constants $C_{b,k,\alpha}$ so that, for $\mu$
satisfying~\eqref{argest001n}, we have
\begin{equation}
  \|R(\mu)f\|_{\WF,k,2+\gamma}\leq 
C_{b,k,\alpha}\left[1+\frac{1}{|\mu|}\right]\|f\|_{\WF,k,\gamma}.
\end{equation}
For any $B>0,$ these constants are uniformly bounded for $0\leq b<B.$
\end{proposition}

\chapter{H\"older Estimates for Euclidean Models}\label{s.eucmod}
The Euclidean model problems are given by
\begin{equation}\label{euchteqn}
  \left[\pa_t-\sum\limits_{j=1}^m\pa^2_{y_j}\right]u(\by,t)=g(\by,t)\text{ with
    } u(\by,0)=f(\by).
\end{equation}
The one dimensional solution kernel is
\begin{equation}
  k^e_t(x,y)=\frac{e^{\frac{|x-y|^2}{4t}}}{\sqrt{4\pi t}},
\end{equation}
and the solution to the equation in~\eqref{euchteqn}, vanishing at $t=0,$  is given by
\begin{equation}\label{eucslninhom}
  u(\by,t)=\int\limits_0^t\int\limits_{-\infty}^{\infty}\cdots
\int\limits_{-\infty}^{\infty}\prod\limits_{j=1}^mk^e_{t-s}(y_j,z_j)g(\bz,s)d\bz ds.
\end{equation}
The solution to the homogeneous initial value problem with $v(\by,0)=f(\by)$ is
given by
\begin{equation}\label{eucslncp0m}
  v(\by,t)=\int\limits_{-\infty}^{\infty}\cdots
\int\limits_{-\infty}^{\infty}\prod\limits_{j=1}^mk^e_{t}(y_j,z_j)f(\bz)d\bz.
\end{equation}
For fixed $\by,$ $v(\by,t)$ extends analytically in $t$ to define a function in
$S_0.$ The H\"older estimates for the solutions of this problem are, of course,
classical. In this chapter we state the estimates and the 1-dimensional kernel
estimates needed to prove them.

\section{H\"older estimates for Solutions in the Euclidean Case}
The solutions of the  problems
\begin{equation}\label{euccp0m}
  [\pa_t-\sum_{j=1}^m\pa_{y_j}^2]v(\by,t)=0\text{ with }v(\by,t)=
f(\by)\in\cC^{k,\gamma}(\bbR^m)
\end{equation}
and
\begin{equation}\label{eucinhom0m}
  [\pa_t-\sum_{j=1}^m\pa_{y_j}^2]u(\by,t)=g(\by,t)\in\cC^{k,\gamma}(\bbR^m\times\bbR_+)\text{
    with }u(\by,t)=0, 
\end{equation}
are well known to satisfy H\"older estimates. These can easily be derived from
the 1-dimensional kernel estimates, which are stated in the following
subsection, much as in the degenerate case, though with considerably less
effort. 

For the homogeneous Cauchy problem we have:
\begin{proposition}\label{prop5.1} Let $k\in\bbN_0$ and $0<\gamma<1.$
  The solution $v$  to~\eqref{euccp0m} with initial data
  $f\in\cC^{k,\gamma}(\bbR^m),$  given in~\eqref{eucslncp0m},
  belongs to $\cC^{k,\gamma}(\bbR^m\times\bbR_+).$ There are constants $C$ so
  that
  \begin{equation}
    \|v\|_{k,\gamma}\leq C\|f\|_{k,\gamma}.
  \end{equation}
\end{proposition}

For the inhomogeneous problem, with zero initial data, we have:
\begin{proposition}\label{prop5.2} Let $k\in\bbN_0$ and $0<\gamma<1.$
  The solution, $u,$ to~\eqref{eucinhom0m} with 
  $g\in\cC^{k,\gamma}(\bbR^m\times\bbR_+),$ is given in~\eqref{eucslninhom};
  it belongs to $\cC^{k+2,\gamma}(\bbR^m\times\bbR_+).$ There are constants $C$ so
  that
  \begin{equation}
    \|u\|_{k+2,\gamma,T}\leq C(1+T)\|g\|_{k,\gamma,T}.
  \end{equation}
\end{proposition}

The proofs of these propositions are in all essential ways identical to the
proofs of Propositions~\ref{prop3} and Proposition~\ref{prop1n0} respectively,
where the 1-dimensional kernel estimates from Chapter~\ref{1ddegmods} are
replaced by those given below in Chapter~\ref{1deucests}. The Euclidean
arguments are a bit simpler, as there is no spatial boundary, and hence the
special arguments needed, in the degenerate case, as $x_j\to 0$ are not
necessary. The $k>0$ estimates follow easily from the $k=0$ estimates using
Proposition~\ref{lem3.3new0.0}, in the $n=0$ case. The details of these
arguments are left to the interested reader. As noted above, these results are
classical, and complete proofs can be found in~\cite{KrylovGSM12}.

We can also define the resolvent operator $R(\mu)f,$ as the
Laplace transform of the heat kernel. For $\mu\in S_0,$ we have the formula
\begin{equation}
  R(\mu)f(\by)=\lim_{\epsilon\to
    0^+}\int\limits_{\epsilon}^{\frac{1}{\epsilon}}
\int\limits_{0}^{\infty}\cdots\int\limits_{0}^{\infty}
\prod\limits_{j=1}^nk^{e}_t(y_j,w_j)f(\bw)d\bw e^{-t\mu} dt.
\end{equation}
Using the asymptotic expansion for the 1-dimensional factors it follows easily
that for each fixed $\by\in\bbR^m,$ $R(\mu)f(\by)$ is an analytic function of
$\mu.$ Applying Cauchy's theorem we can easily show that, so long as
$\Re[\mu\eit>0],$ we can rewrite this as:
\begin{equation}\label{eqn10.182.00m}
  R(\mu)f(\by)=\lim_{\epsilon\to
    0^+}\int\limits_{\epsilon}^{\frac{1}{\epsilon}}
\int\limits_{0}^{\infty}\cdots\int\limits_{0}^{\infty}
\prod\limits_{j=1}^nk^{e}_{s\eit}(y_j,w_j)f(\bw)d\bw e^{-\eit\mu s}\eit ds.
\end{equation}
This shows that $R(\mu)f(\by)$ extends analytically to $\bbC\setminus
(-\infty,0].$ We close this section by stating a proposition summarizing the
properties of $R(\mu)$ as an operator on the H\"older spaces
$\cC^{k,\gamma}(\bbR^m).$ The proof is deferred to the end of
Chapter~\ref{s.genmod} where the analogous result covering all model operators
is proved.
\begin{proposition}\label{prop8.0.4.0m}  The resolvent operator $R(\mu)$ is analytic
  in the complement of $(-\infty,0],$ and is given by the integral
  in~\eqref{eqn10.182.00m} provided that $\Re(\mu e^{i\theta})>0.$ For $\alpha\in
  (0,\pi],$ there are constants $C_{\alpha}$
 so that if 
 \begin{equation}\label{argest001m}
   \alpha-\pi\leq\arg{\mu}\leq\pi-\alpha, \end{equation}
then for $f\in\cC^0_b(\bbR^m)$ we have
  \begin{equation}
    \|R(\mu) f\|_{L^{\infty}}\leq \frac{C_{\alpha}}{|\mu|}\|f\|_{L^{\infty}};
  \end{equation}
with $C_{\pi}=1.$ Moreover, for $0<\gamma<1,$ there is a constant
  $C_{\alpha,\gamma}$ so that if $f\in\cC^{0,\gamma}(\bbR^m),$ then
  \begin{equation}\label{eqn8.13.01}
      \|R(\mu) f\|_{0,\gamma}\leq \frac{C_{\alpha,\gamma}}{|\mu|}\|f\|_{0,\gamma}.
  \end{equation}
For $k\in\bbN_0,$ and $0<\gamma<1,$ if $f\in\cC^{k,\gamma}(\bbR^m),$ then
$R(\mu)f\in\cC^{2+k,\gamma}(\bbR^m),$ and, we have
\begin{equation}
  (\mu-L_{\bzero,m})R(\mu) f=f.
\end{equation}
If $f\in \cC^{2,\gamma}
(\bbR^m),$ then
\begin{equation}
  R(\mu)(\mu-L_{\bzero,m}) f=f.
\end{equation}
There are constants $C_{k,\alpha}$ so that, for $\mu$
satisfying~\eqref{argest001m}, we have
\begin{equation}
  \|R(\mu)f\|_{k+2,\gamma}\leq 
C_{k,\alpha}\left[1+\frac{1}{|\mu|}\right]\|f\|_{k,\gamma}.
\end{equation}
\end{proposition}

\begin{remark} As before the solution $v(\by,t)$ to the Cauchy problem can be
  expressed as a contour integral:
  \begin{equation}
    v(\by,t)=\frac{1}{2\pi
      i}\int\limits_{\Gamma_{\alpha,R}}[R(\mu)f](\by)e^{\mu t}d\mu.
  \end{equation}
From this representation it follows that $v$ extends analytically in $t$ to
$S_0.$ Moreover, for $t\in S_0$ we see that $v(\cdot,t)$ belongs to
$\cC^{2,\gamma}(\bbR^m).$ Hence by the semi-group property
$v(\cdot,t)\in\cC^{\infty}(\bbR^m).$
\end{remark}

\section{1-dimensional Kernel Estimates}\label{1deucests}
The 1-dimensional kernel estimates can easily be used to prove the H\"older
estimates stated in the previous subsection. They form essential components of
the proofs of the H\"older estimates for the general model problems, considered
in the next chapter. The proofs of these estimates are elementary, largely following
from the facts that the kernel, $k^e_t(x,y),$  is a function of $(x-y)^2/t,$ which extends
analytically to $\bbR^2\times S_0.$ As in the degenerate case, we have
\begin{equation}
  \int\limits_{-\infty}^{\infty}k^e_t(x,y)dy=1\text{ for }x\in\bbR, t\in S_0.
\end{equation}
The proofs of the following classical results are left to the reader.

\subsection{Basic Kernel Estimates}

\begin{lemma}\label{lem5newe} For $0\leq \gamma<1,$ $0<\phi<\frac{\pi}{2},$ there
  is a $C_{\phi}$ so that, for $t\in S_{\phi},$ 
  \begin{equation}\label{eqn126.00e}
    \int\limits_{-\infty}^{\infty}|k^e_t(x,y)||x-y|^{\gamma}dy\leq C_{\phi} |t|^{\frac{\gamma}{2}}.
  \end{equation}
\end{lemma}

\begin{lemma}\label{lem21newe} For $0<\phi<\frac{\pi}{2},$ there is a constant
  $C_{\phi}$ so that
for $t\in S_{\phi}$
 \begin{equation}\label{lem21newest2e}
  \int\limits_{-\infty}^{\infty}|k^e_t(x_2,z)-k^e_t(x_1,z)|dz\leq
  C_{\phi}\left(\frac{\frac{|x_2-x_1|}{\sqrt{|t|}}}
{1+\frac{|x_2-x_1|}{\sqrt{|t|}}}\right).
\end{equation}
\end{lemma}

We set
\begin{equation}\label{eqn8555e}
  \alpha_e=\frac{3x_1-x_2}{2}\text{ and }\beta_e=\frac{3x_2-x_1}{2}
\end{equation}
\begin{lemma}\label{lem3newe} Let
  $J=[\alpha_e,\beta_e],$ as defined in~\eqref{eqn8555e}.
For $0<\gamma<1,$ $0<\phi<\frac{\pi}{2},$ there is a $C_{\phi}$ so that for $t=|t|\eit\in S_{\phi}$
\begin{equation}
\int\limits_{J^c}|k^e_t(x_2,y)-k^e_t(x_1,y)||y-x_1|^{\gamma}dy
\leq C_{\phi}|x_2-x_1|^{\gamma}e^{-\cost\frac{(x_2-x_1)^2}{2|t|}}
\end{equation}
\end{lemma}

\begin{lemma}\label{lem4newe} For $0<\gamma<1$ and $c<1$ there is a $C$ such that if
$c<s/t<1,$ then
\begin{equation}
  \int\limits_{-\infty}^{\infty}\left|k_t^e(x,y)-k_s^e(x,y)\right|
|x-y|^{\gamma}dy\leq C|t-s|^{\frac{\gamma}{2}}.
\end{equation}
Without an upper bound on $0<t/s,$ we have the estimate
\begin{equation}
  \int\limits_{-\infty}^{\infty}\left|k_t^e(x,y)-k_s^e(x,y)\right|dy\leq C
\left(\frac{t/s-1}{1+[t/s-1]}\right).
\end{equation}
\end{lemma}
\subsection{First Derivative Estimates}

\begin{lemma}\label{lem2newe} For $0\leq\gamma<1,$ and
  $0<\phi<\frac{\pi}{2},$ there is a $C_{\phi}$ so that for $t\in S_{\phi}$ we have
  \begin{equation}
    \int\limits_{0}^{\infty}|\pa_xk^e_t(x,y)||\sqrt{y}-\sqrt{x}|^{\gamma}dy\leq
C_{\phi}|t|^{\frac{\gamma-1}{2}}.
  \end{equation}
\end{lemma}

\begin{lemma}\label{lem20newwe} For $0<\gamma<1,$ $0<\phi<\frac{\pi}{2},$
  there is a constant $C_{\phi}$ so that for  $t\in S_{\phi},$
  \begin{multline}\label{eqn2400.0e}
    \int\limits_0^{\infty}|\pa_{x}k^e_t(x_1,y)-
    \pa_{x}k^e_t(x_2,y)||\sqrt{x_1}-\sqrt{y}|^{\gamma}dy\leq\\
    C_{\phi}|t|^{\frac{\gamma-1}{2}}\frac{\left(\frac{|\sqrt{x_2}-\sqrt{x_1}|}
{\sqrt{|t|}}\right)}
    {1+\left(\frac{|\sqrt{x_2}-\sqrt{x_1}|}{\sqrt{|t|}}\right)}
  \end{multline} 
\end{lemma}

\begin{lemma}\label{lemAA-e}
  For $0\leq \gamma<1,$ and  $0<\tau,$ we have for $s\in [\tau,\infty)$
  that there is a constant $C$ so that
  \begin{equation}\label{lemAA-este}
    \int\limits_{-\infty}^{\infty}
|\pa_xk^e_{\tau+s}(x,y)-\pa_xk^e_s(x,y)|
|x-y|^{{\gamma}}dy<
C\frac{\tau s^{\frac{\gamma-1}{2}}}{(\tau+s)}.
  \end{equation}
\end{lemma}

\subsection{Second derivative estimates}
\begin{lemma}\label{lem25newe}
  For $0\leq \gamma<1,$  $0<\phi<\frac{\pi}{2},$ there is a $C_{\phi}$ so that
  for $t\in S_{\phi}$ we have the estimate
  \begin{equation}\label{lem25newpeste}
    \int\limits_{-\infty}^{\infty}|\pa_x^2k^e_t(x,y)||x-y|^{\gamma}dy\leq
    C_{\phi} |t|^{\frac{\gamma}{2}-1}.
  \end{equation}
This implies that, if $\gamma>0,$ then
\begin{equation}\label{lem25newpeste2}
    \int\limits_{0}^{|t|}\int\limits_{-\infty}^{\infty}|\pa_x^2k^e_{s\eit}(x,y)||x-y|^{\gamma}dyds\leq
    \left(\frac{2C_{\phi}}{\gamma}\right) |t|^{\frac{\gamma}{2}}.
\end{equation}
\end{lemma}

\begin{lemma}\label{lemBe}
 For $0<\phi<\frac{\pi}{2},$ there is a $C_{\phi}$ so that for $t\in S_{\phi}$
  \begin{equation}\label{lemBeste}
    \int\limits_{0}^{|t|}\left|\pa_yk^e_{s\eit}(x_2,\alpha_e)-\pa_yk^e_{s\eit}(x_2,\beta_e)\right|
ds\leq C_{\phi}e^{-\cost\frac{(x_2-x_1)^2}{|t|}},
  \end{equation}
where $\alpha_e$ and $\beta_e$ are defined in~\eqref{eqn8555e}.
\end{lemma}

\begin{lemma}\label{lemCe}
  For $0<\gamma<1,$ $0<\phi<\frac{\pi}{2},$ there is a $C_{\phi}$ so that for
  $|\theta|<\frac{\pi}{2}-\phi$ and $J=[\alpha_e,\beta_e],$ with the endpoints given
  by~\eqref{eqn8555e}, we have
\begin{equation}\label{lemCeste}
  \begin{split}
     &\int\limits_0^{|t|}\int\limits_{\alpha_e}^{\beta_e}
|\pa_x^2k^e_{s\eit}(x_2,y)||y-x_2|^{\gamma}dyds\leq C_{\phi}|x_2-x_1|^{\gamma}\\
 &\int\limits_0^{|t|}\int\limits_{\alpha_e}^{\beta_e}
|\pa_x^2k^e_{s\eit}(x_1,y)||y-x_1|^{\gamma}dyds\leq C_{\phi}|x_2-x_1|^{\gamma}.
  \end{split}
\end{equation}
\end{lemma}
These estimates follow from the more basic
\begin{lemma}\label{lemCe0}
  For $0<\gamma<1,$ $0<\phi<\frac{\pi}{2},$ there is a $C_{\phi}$ so that for
  $|\theta|<\frac{\pi}{2}-\phi$  and $J=[\alpha_e,\beta_e],$ with the endpoints
  given by~\eqref{eqn8555e}, we have
\begin{equation}\label{lemCeste0}
\int\limits_{\alpha_e}^{\beta_e}
|\pa_x^2k^e_{s\eit}(x_2,y)||y-x_2|^{\gamma}dy\leq 
C_{\phi}\begin{cases} s^{\frac{\gamma}{2}-1}&\text{ if }s<(x_2-x_1)^2\\
\frac{|x_2-x_1|^{\gamma+1}}{s^{\frac 32}}&\text{ if }s\geq (x_2-x_1)^2.
\end{cases}
\end{equation}
\end{lemma}

\begin{lemma}\label{lemDe}
  For $0<\gamma<1,$ $0<\phi<\frac{\pi}{2},$ there is a $C_{\phi}$ so that for
  $|\theta|<\frac{\pi}{2}-\phi$ and $J=[\alpha_e,\beta_e],$ with the endpoints
  given by~\eqref{eqn8555e}, we have
  \begin{equation}\label{lemDeste}
    \int\limits_0^{|t|}\int\limits_{J^c}|\pa_x^2k^e_{s\eit}(x_2,y)-\pa_x^2k^e_{s\eit}(x_1,y)|
|y-x_1|^{\gamma}dyds\leq C_{\phi}|x_2-x_1|^{\gamma}.
  \end{equation}
\end{lemma}
This follows from the more basic:
\begin{lemma}\label{lemDe0}
  For $0<\gamma<1,$ $0<\phi<\frac{\pi}{2},$ there is a $C_{\phi}$ so that for
  $|\theta|<\frac{\pi}{2}-\phi$ and $J=[\alpha_e,\beta_e],$ with the endpoints
  given by~\eqref{eqn8555e}, we have 
  \begin{equation}\label{lemDeste0}
   \int\limits_{J^c}|\pa_x^2k^e_{s\eit}(x_2,y)-\pa_x^2k^e_{s\eit}(x_1,y)|
|y-x_1|^{\gamma}dy\leq
C_{\phi}s^{\frac{\gamma}{2}-1}\left|\frac{x_2-x_1}{\sqrt{s}}\right|
e^{-\cost\frac{(x_2-x_1)^2}{8s}}.
  \end{equation}
\end{lemma}

Finally we have:
\begin{lemma}\label{lemHp2e}  For $0\leq\gamma<1,$  $0<\tau$ and
  $\tau<s,$ there is a constant $C$ so that
  \begin{equation}\label{lemHestp2e} 
    \int\limits_{-\infty}^{\infty}
   |\pa_x^2k^{e}_{\tau+s}(x,y)-\pa_x^2k^{e}_{s}(x,y)||x-y|^{\gamma}dy
\leq C |\tau|s^{\frac{\gamma}{2}-2}.
  \end{equation}
\end{lemma}

\subsection{Large $t$ behavior}

To prove estimates on the resolvent, and to study the off-diagonal behavior of
the heat kernel in many variables,  it is useful to have estimates on the
derivatives of $k^e_t(x,y)$ valid for $t$ bounded away from zero.
\begin{lemma}\label{lrgt1de}For $j\in\bbN$ 
  and $0<\phi<\frac{\pi}{2}$ there is a constant $C_{j,\phi}$ so that if $t\in
  S_{\phi},$ then
  \begin{equation}\label{eqn12.152.1}
    \int\limits_{-\infty}^{\infty}|\pa_x^jk^e_t(x,y)|dy\leq \frac{C_{j,\phi}}{|t|^{\frac{j}{2}}}.
  \end{equation} 
\end{lemma}
\noindent
The proof of this lemma is in the Appendix.
\chapter{H\"older Estimates  for General Models}\label{s.genmod}
We now turn to the task of estimating solutions to heat equations defined by
the operators of the form:
\begin{equation}\label{genmodnm}
  L_{\bb,m}=\sum_{j=1}^n[x_j\pa_{x_j}^2+b_j\pa_{x_j}]+\sum\limits_{k=1}^m\pa_{y_k}^2.
\end{equation}
The general model operator on
$\bbR_+^n\times\bbR^m,$ denoted $L_{\bb,m},$ is labeled by a non-negative
$n$-vector $\bb=(b_1,\dots,b_n),$ and $m,$ the dimension of the corner. We
use $x$-variables to denote points in $\bbR_+^n$ and $y$-variables to denote
points in $\bbR^m.$ If we have a function of these variables $f(\bx,\by)$ then,
as before we estimate differences $f(\bx^2,\by^2)-f(\bx^1,\by^1)$
1-variable-at-a-time. We first observe that 
\begin{equation}
  f(\bx^2,\by^2)-f(\bx^1,\by^1)=
[f(\bx^2,\by^2)-f(\bx^1,\by^2)]+[f(\bx^1,\by^2)-f(\bx^1,\by^1)];
\end{equation}
each term in brackets can then be written as a telescoping sum:
\begin{multline}\label{eqn12.3.1}
f(\bx^2,\by^2)-f(\bx^1,\by^1)= \\ \left\{
\sum\limits_{j=0}^{n-1}[f(\bx^{2\prime}_j,x^2_{j+1},\bx^{1\prime\prime}_j,\by^2)-
f(\bx^{2\prime}_j,x^1_{j+1},\bx^{1\prime\prime}_j,\by^2)]\right\}+\\
\left\{\sum_{l=0}^{m-1}[f(\bx^1,\by^{2\prime}_l,y^2_{l+1},\by^{1\prime\prime}_l)-
f(\bx^1,\by^{2\prime}_l,y^2_{l+1},\by^{1\prime\prime}_l)]\right\},
\end{multline}
where $\bx_j^{\prime}$ and $\bx_j^{\prime\prime}$ are defined in~\eqref{1varattme.1}.
We say that terms in the first sum have a ``variation in an $x$-variable,'' and
terms in the second have a ``variation in a $y$-variable.''  We only
need to deal with terms that have a variation in one or the other type of
variable, and this simplifies the proofs for the general case considerably.

In this chapter we prove H\"older estimates for the solutions on
$\bbR_+^n\times\bbR^m\times\bbR_+$ to homogeneous Cauchy problem
\begin{equation}\label{genmodhcp}
  [\pa_t-L_{\bb,m}]v(\bx,\by,t)=0\text{ with }v(\bx,\by,0)=f(\bx,\by),
\end{equation}
and the inhomogeneous problem
\begin{equation}\label{genmodinhom}
  [\pa_t-L_{\bb,m}]u(\bx,\by,t)=g(\bx,\by,t)\text{ with }u(\bx,\by,0)=0.
\end{equation} 
The solution to the homogeneous initial value problem with $v(\bx,\by,0)=f(\bx,\by)$ is
given by\index{$\kappa^{\bb,m}_t$}
\begin{equation}\label{genslncpnm}
\begin{split}
  v(\bx,\by,t)&=\int\limits_{-\infty}^{\infty}\cdots
\int\limits_{-\infty}^{\infty}
\prod\limits_{l=1}^nk^{b_l}_{t}(x_l,w_l)
\prod\limits_{j=1}^mk^e_{t}(y_j,z_j)f(\bw,\bz) d\bz d\bw \\
&=\kappa^{\bb,m}_tf.
\end{split}
\end{equation}
For fixed $(\bx,\by),$ $v(\bx,\by,t)$ extends analytically in $t$ to define a
function in $S_0.$

The solution to the inhomogeneous  problem is  given by the operator
$K^{\bb,m}_t$ defined by\index{$K^{\bb,m}_t$}
\begin{equation}\label{genslninhomnm}
\begin{split}
  u(\bx,\by,t)&=\int\limits_0^t
\int\limits_{-\infty}^{\infty}\cdots
\int\limits_{-\infty}^{\infty}
\prod\limits_{l=1}^nk^{b_l}_{t-s}(x_l,w_l)
\prod\limits_{j=1}^mk^e_{t-s}(y_j,z_j)g(\bw,\bz,s) d\bz d\bw ds\\
&=K^{\bb,m}_tg.
\end{split}
\end{equation}
As before, for $t>0,$ this expression should be understood as
\begin{equation}
  u(\bx,\by,t)=\lim_{\epsilon\to 0^+}u_{\epsilon}(\bx,\by,t),
\end{equation}
where
\begin{equation}
  u_{\epsilon}(\bx,\by,t)=
\int\limits_0^{t-\epsilon}
\int\limits_{-\infty}^{\infty}\cdots
\int\limits_{-\infty}^{\infty}
\prod\limits_{l=1}^nk^{b_l}_{t-s}(x_l,w_l)
\prod\limits_{j=1}^mk^e_{t-s}(y_j,z_j)g(\bw,\bz,s) d\bz d\bw ds.
\end{equation}
We also note that if we let $v^s(\bx,\by,t)$ denote the solution to the Cauchy problem:
\begin{equation}
[\pa_t-L_{\bb,m}]  v^s=0\text{ with }v^s(\bx,\by,0)=g(\bx,\by,s),
\end{equation}
then
\begin{equation}\label{eqn12.14.4}
  u(\bx,\by,t)=\int\limits_{0}^tv^s(\bx,\by,t-s)ds
\end{equation}

The resolvent operator $R(\mu)f$ is defined, for $f\in
\dcC^0(\bbR_+^n\times\bbR^m),$ and $\mu\in S_0,$ by
\begin{equation}
  R(\mu)f(\bx,\by)=\lim_{\epsilon\to
    0^+}\int\limits_{\epsilon}^{\frac{1}{\epsilon}}
e^{-\mu t}v(\bx,\by,t)dt.
\end{equation}
As in the earlier cases, this is an analytic function of $\mu\in S_0.$ By
deforming the contour we can replace this representation with
\begin{equation}\label{eqn10.182nm}
  R(\mu)f(\bx,\by)=\lim_{\epsilon\to
    0^+}\int\limits_{\epsilon}^{\frac{1}{\epsilon}}
e^{-\mu s\eit}v(\bx,\by,s\eit)\eit ds,
\end{equation}
which converges if $\Re[\mu\eit]>0.$ This analytically extends
$R(\mu)f(\bx,\by)$ to $\mu\in\bbC\setminus (-\infty,0].$ As in the
previous two chapters, the estimates herein are all proved by
reduction to 1-variable kernel estimates.

\section{The Cauchy Problem}
We begin with estimates for the homogeneous Cauchy Problem.
\begin{proposition}\label{prop6.1} Let $k\in\bbN_0,$ $0<R,$
  $(b_1,\dots,b_n)\in\bbR_+^n,$ and $0<\gamma<1.$ The initial data
  $f\in\cC^{k,\gamma}_{\WF}(\bbR_+^n\times\bbR^m);$ if $k>0,$ then
  assume that $f$ is supported in $B_R^+(\bzero)\times\bbR^m.$ The
  solution $v$ to~\eqref{genmodhcp}, with initial data $f,$ given
  in~\eqref{genslncpnm}, belongs to
  $\cC^{k,\gamma}_{\WF}(\bbR_+^n\times\bbR^m\times\bbR_+).$ There are
  constants $C_{\bb,\gamma,R}$ so that
  \begin{equation}\label{eqn12.12.1}
    \|v\|_{\WF,k,\gamma}\leq C_{\bb,\gamma,R}\|f\|_{\WF,k,\gamma}.
  \end{equation}
If $f\in\cC^{k,2+\gamma}_{\WF}(\bbR_+^n\times\bbR^m),$ then $v$ belongs to
  $\cC^{k,2+\gamma}_{\WF}(\bbR_+^n\times\bbR^m\times\bbR_+).$ 
There are constants $C_{\bb,\gamma,R}$ so that
  \begin{equation}\label{eqn12.13.1}
    \|v\|_{\WF,k,2+\gamma}\leq C_{\bb,\gamma,R}\|f\|_{\WF,k,2+\gamma}.
  \end{equation}
For $B>0,$ these constants are uniformly bounded for $\bb\leq B\bone,$ and
if $k=0,$ then the constants are independent of $R.$ 
\end{proposition}
\begin{proof}[Proof of Proposition~\ref{prop6.1}]
    For the $k=0, \bb>\bzero,$ the estimates in~\eqref{eqn12.12.1}
  follow as in the proof of Proposition~\ref{prop3},
  via the 1-variable-at-a-time method.  The cases where two ``$x$'' (or
  parabolic) variables differ follow, essentially verbatim, as in the proof of
  Proposition~\ref{prop3}, from the lemmas in Chapter~\ref{1ddegmods}. The new
  cases in the proof of this proposition are those involving the ``$y$''- (or
  Euclidean) -variables.  As noted after the statement of
  Proposition~\ref{prop5.1}, these cases follow, \emph{mutatis mutandis}, via
  the arguments used in the proof of Proposition~\ref{prop3}. The estimates
  for the kernel $k^b_t$ must be replaced with estimates for the 1-dimensional,
  \emph{Euclidean} heat kernel. These are stated  in
  Section~\ref{1deucests}. As these cases also arise in the proof of
  Proposition~\ref{prop6.2}, to avoid excessive repetition, we forego giving
  the details now, and leave them for the proof of the next proposition.

  As before the constants in these estimates are uniformly bounded for
  bounded $\bzero<\bb\leq \BB.$ Applying the compactness result,
  Proposition~\ref{prop4.1new} we can allow entries of $\bb$ to tend
  to zero, obtaining the unique limiting solution with all the desired
  estimates for these cases as well. If we assume that $f$ is
  supported in $B_R^+(\bzero)\times\bbR^m,$ then the estimates
  in~\eqref{eqn12.12.1} for the H\"older spaces with $k>0,$ follow from the
  $k=0$ results, and Propositions~\ref{lem3.3new0.0}
  and~\ref{wtedest_pr}.

  As in the proof of Proposition~\ref{prop3}, some additional estimates are
  needed to establish~\eqref{eqn12.13.1}. We begin with the $k=0$
  case. Applying Proposition~\ref{lem3.3new0.0} to commute derivatives through
  the integral kernel, we see that estimates for
  $\cC^{0,\gamma}_{\WF}$-norm of $\nabla_{\bx} v,$ $\nabla_{\by} v,$
  $x_j\pa_{x_j}^2 v,$ $j=1,\dots,n$ and $\pa_{y_k}\pa_{y_l}v,$ $1\leq k,l\leq
  m$ follow from~\eqref{eqn12.12.1}. To establish~\eqref{eqn12.13.1}, we need
  only estimate the $\cC^{0,\gamma}_{\WF}$-norm of
  \begin{equation}
    \sqrt{x_i}\pa_{x_i}\pa_{y_k}v\text{ and } \sqrt{x_ix_j}\pa_{x_i}\pa_{x_j}v.
  \end{equation}

To estimate $\sqrt{x_i}\pa_{x_i}\pa_{y_k}v,$ we can relabel so that $i=n,$ and
$k=m;$ this derivative is given by
\begin{multline}\label{eqn12.15.2}
  \sqrt{x_n}\pa_{x_n}\pa_{y_m}v=
\int\limits_{\bbR_+^{n-1}}\int\limits_{\bbR^{m-1}}\kappa^{\bb',m-1}_t(\bx'_{n-1},\bw'_{n-1};\by'_{m-1},\bz'_{m-1})
\times\\
\int\limits_{0}^{\infty}\int\limits_{-\infty}^{\infty}\sqrt{\frac{x_n}{w_n}}k^{b_n+1}_{t}(x_n,w_n)k^e_{t}(y_m,z_m)
\sqrt{w_n}\pa_{w_n}\pa_{z_m}f(\bw,\bz)dw_ndz_md\bw'_{n-1}d\bz'_{m-1},
\end{multline}
where
\begin{equation}
  \kappa^{\bb',m-1}_t(\bx'_{n-1},\bw'_{n-1};\by'_{m-1},\bz'_{m-1})=
\prod_{j=1}^{n-1}k^{b_j}_t(x_j,w_j)\prod_{k=1}^{m-1}k^e_t(y_k,z_k).
\end{equation}
Applying Lemma~\ref{lem10.0.1} we see that
\begin{equation}
  | \sqrt{x_n}\pa_{x_n}\pa_{y_m}v|\leq C\|f\|_{\WF,0,2+\gamma}.
\end{equation}
Since $f\in\cC^{0,2+\gamma}_{\WF},$ we know that
\begin{equation}
  |\sqrt{x_n}\pa_{x_n}\pa_{y_m}f(\bx,\by)|\leq \|f\|_{\WF,0,2+\gamma}x_n^{\frac{\gamma}{2}}.
\end{equation}
Using this estimate in~\eqref{eqn12.15.2} and applying Lemma~\ref{lem10.0.1}
shows that
\begin{equation}\label{eqn12.18.2}
    | \sqrt{x_n}\pa_{x_n}\pa_{y_m}v|\leq C\|f\|_{\WF,0,2+\gamma}x_n^{\frac{\gamma}{2}},
\end{equation}
establishing that
\begin{equation}
  \lim_{x_n\to 0^+}\sqrt{x_n}\pa_{x_n}\pa_{y_m}v=0.
\end{equation}

The H\"older continuity of this derivative in the $\by$-variables follows by
re-expressing the difference,
\begin{equation}
  \sqrt{x_n}\pa_{x_n}\pa_{y_m}v(\bx,\by,t)-\sqrt{x_n}\pa_{x_n}\pa_{y_m}v(\bx,\tby,t)
\end{equation}
as a sum of terms like those appearing in~\eqref{vdfn0}. We then apply
estimates from Lemmas~\ref{lem21newe} and~\ref{lem3newe} to the terms in this
sum, along with Lemma~\ref{lem10.0.1} to conclude that
\begin{equation}
  |\sqrt{x_n}\pa_{x_n}\pa_{y_m}v(\bx,\by,t)-\sqrt{x_n}\pa_{x_n}\pa_{y_m}v(\bx,\tby,t)|\leq
C\|f\|_{\WF,0,2+\gamma}\rho_e(\by,\tby)^{\gamma}.
\end{equation}
The argument to establish the estimates
\begin{multline}\label{eqn12.23.2}
 |\sqrt{x_n}\pa_{x_n}\pa_{y_m}v(\bx'_j,x^1_{j+1},\bx''_{j},\by,t)-
\sqrt{x_n}\pa_{x_n}\pa_{y_m}v(\bx'_j,x^2_{j+1},\bx''_{j},\by,t)|\leq\\
C\|f\|_{\WF,0,2+\gamma}\left|\sqrt{x^1_{j+1}}-\sqrt{x^2_{j+1}}\right|^{\gamma},\text{ with }j\leq n-2,
\end{multline}
is essentially identical, with Lemmas~\ref{lem21newe} and~\ref{lem3newe}
replaced by Lemmas~\ref{lem21new} and~\ref{lem3new}. 

The only remaining spatial estimate is~\eqref{eqn12.23.2} with $j=n-1.$ The
estimate in~\eqref{eqn12.18.2} implies that for any $c<1,$ there is a $C$ so
that, if $x_n^1<cx_n^2,$ then
\begin{multline}\label{eqn12.24.2}
 |\sqrt{x_n}\pa_{x_n}\pa_{y_m}v(\bx'_{n-1},x^1_{n},\by,t)-
\sqrt{x_n}\pa_{x_n}\pa_{y_m}v(\bx'_{n-1},x^2_{n},\by,t)|\leq\\
C\|f\|_{\WF,0,2+\gamma}\left|\sqrt{x^1_{n}}-\sqrt{x^2_{n}}\right|^{\gamma},
\end{multline}
leaving only the case $cx_n^2<x_n^1<x_n^2.$ Using an obvious modification
of~\eqref{eqn10.63.2}, and essentially the same argument  as appears
after~\eqref{eqn10.63.2}, we can  prove this
estimate as well. To complete the estimates of this derivative we need to show
that it is H\"older continuous in the time variable. The proof of this estimate
is a small modification of that used to prove~\eqref{eqn10.97.2}. We begin with 
\begin{multline}\label{eqn12.24.6}
   \sqrt{x_n}\pa_{x_n}\pa_{y_m}[v(\bx,\by,t)-v(\bx,\by,0)]=
\int\limits_{\bbR_+^{n-1}}\int\limits_{\bbR^{m-1}}\kappa^{\bb',m-1}_t(\bx'_{n-1},\bw'_{n-1};\by'_{m-1},\bz'_{m-1})
\times\\
\int\limits_0^{\infty}\int\limits_0^{\infty}
\sqrt{x_n}
k^{b_n+1}_t(x_n,w_n)k^{e}_t(y_m,z_m)
[f_{nm}(\bw,\bz)-f_{nm}(\bx,\by)]dw_ndz_md\bw'_{n-1}d\bz'_{m-1},
\end{multline}
where $f_{nm}=\pa_{w_n}\pa_{z_m}f.$ We rewrite the difference, $[f_{nm}(\bw,\bz)-f_{nm}(\bx,\by)]$
as a telescoping sum like that in~\eqref{eqn12.3.1}:
\begin{multline}
  f_{nm}(\bw,\bz)-f_{nm}(\bx,\by)=
\left\{
\sum\limits_{j=0}^{n-1}[f_{nm}(\bw^{\prime}_j,w_{j+1},\bx^{\prime\prime}_j,\bz)-
f_{nm}(\bw^{\prime}_j,x_{j+1},\bx^{\prime\prime}_j,\bz)]\right\}+\\
\left\{\sum_{l=0}^{m-1}[f_{nm}(\bx,\bz^{\prime}_l,z_{l+1},\by^{\prime\prime}_l)-
f_{nm}(\bx,\bz^{\prime}_l,y_{l+1},\by^{\prime\prime}_l)]\right\}
\end{multline}
Each term in the second sum is estimated by
\begin{equation}
\sqrt{x_n} | f_{nm}(\bx,\bz^{\prime}_l,z_{l+1},\by^{\prime\prime}_l)-
f_{nm}(\bx,\bz^{\prime}_l,y_{l+1},\by^{\prime\prime}_l)|\leq
\|f\|_{\WF,0,2+\gamma}|z_{l+1}-y_{l+1}|^{\gamma}.
\end{equation}
Each term in the first sum, except for $j=n-1$ is estimated by
\begin{equation}
 \sqrt{x_n}|f_{nm}(\bw^{\prime}_j,w_{j+1},\bx^{\prime\prime}_j,\bz)-
f_{nm}(\bw^{\prime}_j,x_{j+1},\bx^{\prime\prime}_j,\bz)|\leq
\|f\|_{\WF,0,2+\gamma}|\sqrt{x_{j+1}}-\sqrt{w_{j+1}}|^{\gamma}.
\end{equation}
The remaining case is estimated by
\begin{multline}
\sqrt{x_n}|f_{nm}(\bw'_{n-1},w_n,\bz)-f_{nm}(\bw'_{n-1},x_n,\bz)|\leq
|\sqrt{x_n}-\sqrt{w_n}||f_{nm}(\bw'_{n-1},w_n,\bz)|+\\
|\sqrt{w_n}f_{nm}(\bw'_{n-1},w_n,\bz)-\sqrt{x_n}f_{nm}(\bw'_{n-1},x_n,\bz)|\\
\leq \|f\|_{\WF,0,2+\gamma}|[|\sqrt{x_n}-\sqrt{w_n}|w_{n}^{\frac{\gamma-1}{2}}+
|\sqrt{x_n}-\sqrt{w_n}|^{\gamma}].
\end{multline}

Using these estimates in~\eqref{eqn12.24.6} we apply
Lemmas~\ref{lem5new},~\ref{lem10.0.3} and~\ref{lem5newe} to
deduce that
\begin{equation}
  | \sqrt{x_n}\pa_{x_n}\pa_{y_m}[v(\bx,\by,t)-v(\bx,\by,0)]|\leq
  C\|f\|_{\WF,0,2+\gamma}
t^{\frac{\gamma}{2}}.
\end{equation}
As usual, this shows that for $0<c<1,$ there is a $C$ so that if $s<ct,$ then
\begin{equation}
  | \sqrt{x_n}\pa_{x_n}\pa_{y_m}[v(\bx,\by,t)-v(\bx,\by,s)]|\leq
  C\|f\|_{\WF,0,2+\gamma}
|t-s|^{\frac{\gamma}{2}}.
\end{equation}
We are left with the case $ct<s<t,$ which again closely follows the pattern of
the proof of~\eqref{eqn10.97.2}. As before we use an analogue
of~\eqref{eqn10.36.1}:
\begin{multline}\label{eqn10.36.11}
\sqrt{x_n}\pa_{x_n}\pa_{y_m}[   v(\bx,\by,t)-v(\bx,\by,s)]\\
=\sqrt{x_n}\int\limits_{\bbR_+^n}\int\limits_{\bbR^m}\Bigg[
\kappa^{e,m}_t(\by,\bz)\sum\limits_{l=1}^{n}\Bigg\{
\prod_{j=1}^{n-l}k_t^{\tb_j}(x_j,w_j)\times\prod_{j=n-l+2}^{n}k_s^{\tb_j}(x_j,w_j)\times\\
\left[k_t^{\tb_{n-l+1}}(x_{n-l+1},w_{n-l+1})-k_s^{\tb_{n-l+1}}(x_{n-l+1},w_{n-l+1})\right]\times\\
[f_{nm}(\bw,\bz)-f_{nm}(\bw'_l,x_{n-l+1},\bw''_{l},\bz)]\Bigg\}+\\
\kappa^{\tbb,0}_s(\bx,\bw)\sum\limits_{q=1}^{m}\Bigg\{
\prod_{j=1}^{m-q}k_t^{e}(y_j,z_j)\times\prod_{j=m-q+2}^{m}k_s^{e}(y_j,z_j)\times\\
\left[k_t^{e}(y_{m-q+1},z_{m-q+1})-k_s^{e}(y_{m-q+1},z_{m-q+1})\right]\times\\
[f_{nm}(\bw,\bz)-f_{nm}(\bw,\bz'_q,y_{m-q+1},\bz''_{q})]\Bigg\}\Bigg]d\bw d\bz,
\end{multline}
where
\begin{equation}
  \tb_j=b_j\text{ for }j=1,\dots,n\text{ and }\tb_n=b_n+1.
\end{equation}
Each term in the second sum is estimated by an integral of the form
\begin{multline}
  \|f\|_{\WF,0,2+\gamma}\int\limits_0^{\infty}
\int\limits_{0}^{\infty}\sqrt{\frac{x_n}{w_n}}k^{b_n+1}_s(x_n,w_n)\times\\
|k^e_t(y,z)-k^e_s(y,z)||y-z|^{\gamma}dz dw_n.
\end{multline}
Lemmas~\ref{lem10.0.1} and~\ref{lem4newe} show that these terms are bounded by
\begin{equation}
  C\|f\|_{\WF,0,2+\gamma} |t-s|^{\frac{\gamma}{2}}.
\end{equation}

Every term in the first sum, with $l\neq 1,$ is bounded by an integral of the
form
\begin{multline}
  \|f\|_{\WF,0,2+\gamma}\int\limits_0^{\infty}
\int\limits_{0}^{\infty}\sqrt{\frac{x_n}{w_n}}k^{b_n+1}_s(x_n,w_n)\times\\
|k^{b_{j}}_t(x_j,w_j))-k^{b_j}_s(x_j,w_j)||\sqrt{x_j}-\sqrt{w_j}|^{\gamma}dw_j dw_n.
\end{multline}
Lemmas~\ref{lem10.0.1} and~\ref{lem4new} show that these terms are bounded by
\begin{equation}
  C\|f\|_{\WF,0,2+\gamma} |t-s|^{\frac{\gamma}{2}}.
\end{equation} 
This leaves just the $l=1$ case, which is bounded by
\begin{multline}
  \|f\|_{\WF,0,2+\gamma}\int\limits_{0}^{\infty}
|k^{b_{n}+1}_t(x_n,w_n))-k^{b_n+1}_s(x_n,w_n)|
\times\\
[|\sqrt{x_n}-\sqrt{w_n}|^{\gamma}+|\sqrt{x_n}-\sqrt{w_n}|w_n^{\frac{1-\gamma}{2}}]dw_n.
\end{multline}
Lemmas~\ref{lem4new} and~\ref{lem10.1.4} show that this is also bounded by $
C\|f\|_{\WF,0,2+\gamma} |t-s|^{\frac{\gamma}{2}},$ thereby completing the proof
that
\begin{equation}
  |\sqrt{x_n}\pa_{x_n}\pa_{y_m}(v(\bx,\by,t)-v(\bx,\by,s)]|\leq
C\|f\|_{\WF,0,2+\gamma} |t-s|^{\frac{\gamma}{2}}.
\end{equation}

This brings us to the H\"older estimates for
$\sqrt{x_ix_j}\pa_{x_i}\pa_{x_j}v.$ The proofs here are quite similar to the
analogous result in Proposition~\ref{prop3}. The proof that
\begin{equation}
  |\sqrt{x_ix_j}\pa_{x_i}\pa_{x_j}v(\bx,\by,t)-
\sqrt{\tx_i\tx_j}\pa_{x_i}\pa_{x_j}v(\tbx,\by,t)|\leq C\rho_s(\bx,\tbx)^{\gamma},
\end{equation}
follows exactly as before.  Using Lemma~\ref{lem10.0.1}, working
one-variable-at-a-time, we also easily establish
\begin{equation}
   |\sqrt{x_ix_j}\pa_{x_i}\pa_{x_j}v(\bx,\by,t)-
\sqrt{x_ix_j}\pa_{x_i}\pa_{x_j}v(\bx,\tby,t)|\leq C\rho_e(\by,\tby)^{\gamma}.
\end{equation}
The H\"older continuity in time follows as in Proposition~\ref{prop3},
while incorporating the Euclidean variables as in the previous case,
i.e.~$\sqrt{x_j}\pa_{x_j}\pa_{y_l}v.$ Finally we observe that, as
$\pa_tv=L_{\bb,m}v,$ and we have established that
$L_{\bb,m}v\in\cC^{0,\gamma}_{\WF}(\bbR_+^n\times\bbR^m\times
[0,\infty)),$ the same is true of $\pa_t v.$ This completes the proof
of~\eqref{eqn12.13.1} in the $k=0$ case. Assuming that $f$ is
supported in $B_R^+(\bzero)\times \bbR^m,$ applying
Propositions~\ref{lem3.3new0.0} and~\ref{wtedest_pr}, we can easily
deduce~\eqref{eqn12.13.1} when $k>0$ from the $k=0$ case.
\end{proof}

\section{The Inhomogeneous Problem}
We now turn to the inhomogeneous problem.
\begin{proposition}\label{prop6.2}
  Let $k\in\bbN_0,$ $(b_1,\dots,b_n)\in\bbR_+^n,$ $0<R,$ and
  $0<\gamma<1.$ Let
  $g\in\cC^{k,\gamma}(\bbR_+^n\times\bbR^m\times\bbR_+).$ If $0<k,$
  then assume that $g$ is supported in
  $B_R^+(\bzero)\times\bbR^m\times [0,T].$ The solution $u$
  to~\eqref{genmodinhom}, with right hand side $g,$ given
  in~\eqref{genslninhomnm}, belongs to
  $\cC^{k,2+\gamma}(\bbR_+^n\times\bbR^m\times\bbR_+).$ There are
  constants $C_{k,\bb,\gamma,R}$ so that
  \begin{equation}
    \|u\|_{\WF,k+2,\gamma,T}\leq C_{k,\bb,\gamma,R}(1+T)\|g\|_{\WF,k,\gamma,T}.
  \end{equation}
The tangential first derivatives satisfy a stronger estimate, there is a
constant $C$ so that  if $T\leq 1,$ then
\begin{equation}\label{tngtestnm}
  \|\nabla_yu\|_{\WF,0,\gamma,T}\leq C T^{\frac{\gamma}{2}}\|g\|_{\WF,0,\gamma,T}.
\end{equation}
The constants are uniformly bounded for $\bb\leq B\bone,$ and
independent of $R$ if $k=0.$
\end{proposition}
\begin{proof}[Proof of Proposition~\ref{prop6.2}] As before we begin
  by assuming that $\bzero<\bb,$ and $k=0.$ Using the
  1-variable-at-a-time method, any estimate of the variation in an
  $x$-variable of a derivative in the $x$-variables alone, or the
  variation in a $y$-variable of a derivative in the $y$-variables
  alone, follows easily from the lemmas in Chapters~\ref{1ddegmods}
  and~\ref{1deucests}.

The maximum principle and~\eqref{eqn12.14.4} show that
\begin{equation}
  |u(\bx,\by,t)|\leq t\|g\|_{L^{\infty}(\bbR_+^n\times\bbR^m\times[0,t])}.
\end{equation}
We use the representation in~\eqref{eqn12.14.4} and Proposition~\ref{prop6.1}
to deduce that, for $t\leq T,$
  \begin{equation}
    |u(\bx^1,\by^1,t)-u(\bx^2,\by^2,t)|\leq
    Ct\rho((\bx^1,\by^1),(\bx^2,\by^2))^{\gamma}
\|g\|_{\WF,0,\gamma,T}.
  \end{equation}
As before, estimates of the second derivatives (see equations~\eqref{2nddrvest1} and
~\eqref{eqn108.1} and Lemma~\ref{lem25newe}) show that
\begin{equation}
  |L_{\bb,m}u(\bx,\by,t)|\leq C\|g\|_{\WF,0,\gamma}t^{\frac{\gamma}{2}}.
\end{equation}
Integrating the equation, $\pa_tu=L_{\bb,m}u+g,$ in $t$ we can therefore show
that there is a constant $C$ so that, if $t_1,t_2<\leq T,$ then
\begin{equation}\label{eqn12.20}
  |u(\bx,\by,t_2)-u(\bx,\by,t_1)|\leq
  C\|g\|_{\WF,0,\gamma,T}\left[|t_2^{\frac{\gamma}{2}+1}-t_1^{\frac{\gamma}{2}+1}|
+|t_2-t_1|\right].
\end{equation}
These results show that there is a constant $C$ so that
\begin{equation}\label{eqn12.15.1}
  \|u\|_{\WF,0,\gamma,T}\leq CT^{1-\frac{\gamma}{2}}(1+T^{\frac{\gamma}{2}})\|g\|_{\WF,0,\gamma,T}.
\end{equation}

\subsection{First derivative estimates}
Using the estimates proved above, we can easily show that
  \begin{equation}\label{frstderestnmt}
    |\pa_{x_j}u(\bx,\by,t)|\leq C\|g\|_{\WF,0,\gamma}t^{\frac{\gamma}{2}}\text{
      and }
|\pa_{y_l}u(\bx,\by,t)|\leq C\|g\|_{\WF,0,\gamma}t^{\frac{\gamma+1}{2}}
  \end{equation}
The first estimate follows by the argument used to prove~\eqref{n0t1stderest}.
We indicate how the second estimate is proved. The standard limiting argument
shows that:
\begin{multline}
\pa_{y_m}u(\bx,\by,t)=
\int\limits_0^t\int\limits_0^{\infty}\cdots\int\limits_0^{\infty} 
\int\limits_{-\infty}^{\infty}\cdots
\int\limits_{-\infty}^{\infty}
\prod\limits_{j=1}^nk^{b_j}_{s}(x_j,w_j)\prod\limits_{l\neq m}k^e_{s}(y_l,z_l)\times\\
\pa_{y_m}k^e_{s}(y_m,z_m)[g(\bw,\bz,t-s)-g(\bw,\bz'_{m-1}, y_m,\bz''_{m-1},t-s)]d\bw d\bz ds.
\end{multline}
Putting in absolute values we see that
\begin{equation}
  |\pa_{y_m}u(\bx,\by,t)|\leq 2\|g\|_{\WF,0,\gamma}
\int\limits_0^t\int\limits_0^{\infty}|\pa_{y_m}k^e_{s}(y_m,z_m)||y_m-z_m|^{\gamma}dz_mds.
\end{equation}
Lemma~\ref{lem2newe} shows that
\begin{equation}
  |\pa_{y_m}u(\bx,\by,t)|\leq C\|g\|_{\WF,0,\gamma}
\int\limits_0^ts^{\frac{\gamma-1}{2}}ds=C\|g\|_{\WF,0,\gamma}\frac{2t^{\frac{\gamma+1}{2}}}{\gamma+1}.
\end{equation}

We note that by integrating these estimates for $\nabla_{\bx,\by}u$ we obtains
a Lipschitz estimate for $u$ itself:
\begin{equation}
  |u(\bx^2,\by^2,t)-u(\bx^1,\by^1,t)|\leq 
C\|g\|_{\WF,0,\gamma}t^{\frac{\gamma}{2}}(1+\sqrt{t})\|(\bx^2,\by^2)-(\bx^1,\by^1)\|,
\end{equation}
though these estimates are not directly relevant to estimating
$\|u\|_{\WF,0,2+\gamma,T}.$

The arguments used to prove~\eqref{frstdrdiaghldestn0}
and~\eqref{frstdroffdiaghldestn0} apply, essentially verbatim to show that
  \begin{equation}
          |\pa_{x_j}u(\bx^2,\by,t)-\pa_{x_j}u(\bx^1,\by,t)|\leq 
C\|g\|_{\WF,0,\gamma}\rho_s(\bx^1,\bx^2)^{\gamma}
  \end{equation}
provided $\rho_s(\bx^1,\bx^2)$ is bounded by $1.$
For  $\rho_e(\by^1,\by^2)<1$ we  have
\begin{equation}
  |\pa_{y_l}u(\bx,\by^2,t)-\pa_{y_l}u(\bx,\by^1,t)|\leq 
Ct^{\frac{\gamma}{2}}\|g\|_{\WF,0,\gamma}\rho_e(\by^1,\by^2)^{\gamma}.
\end{equation}
To prove this we can assume that $\by^2-\by^1$ has exactly one non-zero entry.
If $\by^1$ and $\by^2$ differ in the $l$th entry,
then
\begin{multline}
   |\pa_{y_l}u(\bx,\by^2,t)-\pa_{y_l}u(\bx,\by^1,t)|\leq \\
\|g\|_{\WF,0,\gamma}\int\limits_0^t\int\limits_{-\infty}^{\infty}
|\pa_{y_l}[k^e_s(y^2_l,z_l)-k^e_s(y^1_l,z_l)]|
|y^1_l-z_l|^{\gamma}dz_lds.
\end{multline}
Applying Lemma~\ref{lem20newwe} we see that this integral is bounded by
\begin{equation}\label{eqn332.01}
  C\int\limits_0^ts^{\frac{\gamma-1}{2}}\frac{\left(\frac{|y^2_l-y^1_l|}{\sqrt{s}}\right)}
{1+\left(\frac{|y^2_l-y^1_l|}{\sqrt{s}}\right)}ds.
\end{equation}
An elementary calculation shows that if  $\rho_e(\by^1,\by^2)<1,$
then this integral is bounded by a constant times
$t^{\frac{\gamma}{2}}\rho_e(\by^1,\by^2)^{\gamma}.$  If $\by^1$ and $\by^2$
differ in a coordinate other than the $l$th, then Lemmas~\ref{lem21newe}
and~\ref{lem2newe} show that the estimate for
$|\pa_{y_l}u(\bx,\by^2,t)-\pa_{y_l}u(\bx,\by^1,t)|$ reduces again to the
integral in~\eqref{eqn332.01}.
We are therefore left to consider the off-diagonal cases:
$|\pa_{x_j}u(\bx,\by^2,t)-\pa_{x_j}u(\bx,\by^1,t)$ and 
$|\pa_{y_l}u(\bx^2,\by,t)-\pa_{y_l}u(\bx^1,\by,t)|.$ 

We can again assume that $\bx^2-\bx^1$ and $\by^2-\by^1$ each have exactly one
non-zero entry, which we can assume is the first. We first consider
\begin{multline}
  |\pa_{x_j}u(\bx,\by^2,t)-\pa_{x_j}u(\bx,\by^1,t)|\leq
2\|g\|_{\WF,0,\gamma}
\int\limits_0^t\int\limits_0^{\infty}\int\limits_{-\infty}^{\infty}
|\pa_{x_j}k^{b_j}_{s}(x_j,w_j)|\times\\
|\sqrt{x_j}-\sqrt{w_j}|^{\gamma}
|k^e_s(y^2_1,z_1)-k^e_s(y^1_1,z_1)|dw_jdz_1ds.
\end{multline}
Applying Lemmas~\ref{lem21newe} and~\ref{lem2new} shows that this is bounded by
\begin{equation}\label{eqn12.61.00}
  |\pa_{x_j}u(\bx,\by^2,t)-\pa_{x_j}u(\bx,\by^1,t)|\leq
C\|g\|_{\WF,0,\gamma}\int\limits_0^ts^{\frac{\gamma}{2}-1}
\left(\frac{\frac{|y^2_1-y^1_1|}{\sqrt{s}}}
{1+\frac{|y^2_1-y^1_1|}{\sqrt{s}}}\right)ds.
\end{equation}
An elementary estimate and Lemma~\ref{lem1} shows that therefore
\begin{equation}
  |\pa_{x_j}u(\bx,\by^2,t)-\pa_{x_j}u(\bx,\by^1,t)|\leq
C\|g\|_{\WF,0,\gamma}\rho_e(\by^1,\by^2)^{\gamma}.
\end{equation}

To estimate $|\pa_{y_m}u(\bx^2,\by,t)-\pa_{y_m}u(\bx^1,\by,t)|,$ first bound
$|\pa_{x_q}\pa_{y_l}u(\bx,\by,t)|.$ By relabeling it suffices to consider
$q=1,$ for which we use the expression
\begin{multline}
  \pa_{x_1}\pa_{y_m}u(\bx,\by,t)=
\int\limits_0^t\int\limits_0^{\infty}\cdots\int\limits_0^{\infty}
\int\limits_{-\infty}^{\infty}\cdots\int\limits_{-\infty}^{\infty}  
\prod\limits_{j\neq 1}k^{b_j}_{s}(x_j,w_j)\prod\limits_{l\neq
  m}k^e_{s}(y_l,z_l)\times\\
\pa_{x_1}k^{b_1}_s(x_1,w_1)\pa_{y_m}k^e_s(y_m,z_m)[g(\bw,\bz,t-s)-
g(x_1,\bw''_0,\bz,t-s)]d\bw d\bz ds.
\end{multline}
Putting in absolute values and using the standard estimate for the difference $g(\bw,\bz,t-s)-
g(x_1,\bw''_0,\bz,t-s),$ gives the bound:
\begin{multline}\label{eqn12.64.00}
  |\pa_{x_1}\pa_{y_m}u(\bx,\by,t)|\leq\\
2\|g\|_{\WF,0,\gamma}\int\limits_{0}^t
\int\limits_{-\infty}^{\infty}\int\limits_0^{\infty}
|\pa_{x_1}k^{b_1}_s(x_1,w_1)\pa_{y_m}k^e_s(y_m,z_m)||\sqrt{x_1}-\sqrt{w_1}|^{\gamma}
dw_1dz_mds.
\end{multline}
Applying Lemmas~\ref{lem2new} and~\ref{lem2newe} we see that
\begin{equation}\label{eqn12.65.00}
\begin{split}
   |\pa_{x_1}\pa_{y_m}u(\bx,\by,t)|&\leq
C\|g\|_{\WF,0,\gamma}\int\limits_{0}^t\frac{s^{\frac{\gamma}{2}-1}ds}{\sqrt{x_1}+\sqrt{s}}\\
&\leq
C\|g\|_{\WF,0,\gamma}\frac{t^{\frac{\gamma}{2}}}{\sqrt{x_1}}.
\end{split}
\end{equation}
By integrating the last expression we see that
\begin{equation}
  |\pa_{y_m}u(x^2_1,\bx^{\prime\prime}_0,\by,t)-\pa_{y_m}u(x^1_1,\bx^{\prime\prime}_0,\by,t)|
\leq C\|g\|_{\WF,0,\gamma}t^{\frac{\gamma}{2}}\left|\sqrt{x^2_1}-\sqrt{x^1_x}\right|.
\end{equation}
This estimate implies that, for $\bx^1,\bx^2$ with $\rho_s(\bx^1,\bx^2)\leq 1,$
we have
\begin{equation}
  |\pa_{y_m}u(x^2_1,\bx^{\prime\prime}_0,\by,t)-\pa_{y_m}u(x^1_1,\bx^{\prime\prime}_0,\by,t)|\leq
C\|g\|_{\WF,0,\gamma}t^{\frac{\gamma}{2}}\rho_s(\bx^1,\bx^2)^{\gamma},
\end{equation}
completing the proof of the spatial part of~\eqref{tngtestnm}.

To complete the estimates of the first derivatives we need to bound $|\nabla
u(\bx,\by,t_2)-\nabla u(\bx,\by,t_1)|.$ From the estimates
in~\eqref{frstderestnmt}, we see that for $t_1,t_2\leq T,$ and any $0<c<1,$
there is a $C_T$ so that
if $t_1<ct_2,$ then
\begin{equation}
\begin{split}
  &|\nabla_x u(\bx,\by,t_2)-\nabla_x u(\bx,\by,t_1)|\leq
C_T\|g\|_{\WF,0,\gamma}|t_2-t_1|^{\frac{\gamma}{2}}\\
&|\nabla_yu(\bx,\by,t_2)-\nabla_y u(\bx,\by,t_1)|\leq
C_T\|g\|_{\WF,0,\gamma}|t_2-t_1|^{\frac{\gamma+1}{2}}.
\end{split}
\end{equation}
As usual, this reduces us to consideration of the case that $ct_2<t_1<t_2.$ For
this argument we fix a $\frac 12<c<1,$ and use a slightly different argument
depending upon whether we are estimating an $x$-derivative or a
$y$-derivative. The $x$-derivatives are done very much like the estimates in
Chapter~\ref{s.cormod} beginning with~\eqref{eqn218.00}. For example, to
estimate the $x_n$-derivative we use the representation
\begin{multline}\label{eqn218.01}
  \pa_{x_n}u(\bx,\by,t_2)-\pa_{x_n}u(\bx,\by,t_1)=\\
\int\limits_{0}^{t_2-t_1}\int\limits_{-\infty}^{\infty}\cdots
\int\limits_{-\infty}^{\infty}
\int\limits_0^{\infty}\cdots\int\limits_0^{\infty}\prod\limits_{j=1}^mk^e_{s}(y_j,z_j)
\prod_{j=1}^{n-1}k^{b_j}_s(x_j,w_j)\pa_{x_n}k^{b_n}_s(x_n,w_n)\times\\
[g(\bw_n',w_n,\bz,t_2-s)-
g(\bw_n',w_n,\bz,t_1-s)]dw_nd\bw'_nd\bz ds+\\
\int\limits_{0}^{2t_1-t_2}
\int\limits_{-\infty}^{\infty}\cdots
\int\limits_{-\infty}^{\infty}
\int\limits_0^{\infty}\cdots\int\limits_0^{\infty}
\Bigg[\prod\limits_{j=1}^mk^e_{t_2-s}(y_j,z_j)\prod_{j=1}^{n-1}k^{b_j}_{t_2-s}(x_j,w_j)
\pa_{x_n}k^{b_n}_{t_2-s}(x_n,w_n)-\\
\prod\limits_{j=1}^mk^e_{t_1-s}(y_j,z_j)\prod_{j=1}^{n-1}k^{b_j}_{t_1-s}(x_j,w_j)
\pa_{x_n}k^{b_n}_{t_1-s}(x_n,w_n)\Bigg]\times\\
[g(\bw_n',w_n,\bz,s)-
g(\bw_n',x_n,\bz,s)]dw_nd\bw'_nd\bz ds+\\
\int\limits_{2t_1-t_2}^{t_1}
\int\limits_{-\infty}^{\infty}\cdots
\int\limits_{-\infty}^{\infty}
\int\limits_0^{\infty}\cdots\int\limits_0^{\infty}\prod\limits_{j=1}^mk^e_{t_2-s}(y_j,z_j)
\prod_{j=1}^{n-1}k^{b_j}_{t_2-s}(x_j,w_j)\pa_{x_n}k^{b_n}_{t_2-s}(x_n,w_n)\times\\
[g(\bw_n',w_n,\bz,s)-
g(\bw_n',x_n,\bz,s)]dw_nd\bw'_n d\bz ds.
\end{multline}
The first and last terms are estimated exactly as before. To estimate the
second integral we use the analogue of the expression
in~\eqref{1stdern0ktksdif}, first observing that
\begin{multline}\label{1stdern0ktksdifnm}
 \pa_{x_n}\left\{\prod\limits_{j=1}^mk^e_{t_2-s}(y_j,z_j)
   \prod_{l=1}^{n}k_{t_2-s}^{b_l}(x_l,w_l)-
\prod\limits_{j=1}^mk^e_{t_1-s}(y_j,z_j)\prod_{l=1}^{n}k_{t_1-s}^{b_l}(x_l,w_l)\right\}=\\
\pa_{x_n}\Bigg\{\prod\limits_{j=1}^mk^e_{t_2-s}(y_j,z_j)
   \left[\prod_{l=1}^{n}k_{t_2-s}^{b_l}(x_l,w_l)-\prod_{l=1}^{n}k_{t_1-s}^{b_l}(x_l,w_l)\right]+\\
\left[\prod\limits_{j=1}^mk^e_{t_2-s}(y_j,z_j)-\prod\limits_{j=1}^mk^e_{t_1-s}(y_j,z_j)\right]
  \prod_{l=1}^{n}k_{t_1-s}^{b_l}(x_l,w_l)\Bigg\}.
\end{multline}
We use the expansion in~\eqref{n0ktksdif} to replace the differences of
products on the right hand side of~\eqref{1stdern0ktksdifnm} with terms
containing a single term of the form
\begin{equation}
  \begin{split}
    &[k^{b_l}_{t_2-s}(x_l,w_l)-k^{b_l}_{t_1-s}(x_l,w_l)]\text{ or }\\
&[k^{e}_{t_2-s}(y_j,z_j)-k^{e}_{t_1-s}(y_j,z_j)].
  \end{split}
\end{equation}
If we always use the estimate
\begin{equation}
  |g(\bw_n',w_n,\bz,s)-g(\bw_n',x_n,\bz,s)|\leq 2\|g\|_{\WF,0,\gamma}|\sqrt{w_n}-\sqrt{x_n}|^{\gamma},
\end{equation}
then we see that there are three types of terms that must be bounded:
\begin{enumerate}
\renewcommand{\labelenumi}{\Roman{enumi}.}
\item 
\begin{equation}
\begin{aligned}
\int\limits_0^{2t_1-t_2}\int\limits_0^{\infty}\int\limits_0^{\infty}
|k^{b_l}_{t_2-s}(x_l,w_l)-& k^{b_l}_{t_1-s}(x_l,w_l)| |\pa_{x_n}k^{b_n}_{t_2-s}(x_n,w_n)| \\ 
& |\sqrt{x_n}-\sqrt{w_n}|^{\gamma}dw_l dw_nds, 
\end{aligned}
\end{equation}
\item 
\begin{equation}
\int\limits_0^{2t_1-t_2}\int\limits_0^{\infty}
|\pa_{x_n}[k^{b_n}_{t_2-s}(x_n,w_n)-k^{b_n}_{t_1-s}(x_n,w_n)]|
|\sqrt{x_n}-\sqrt{w_n}|^{\gamma}dw_nds,
\end{equation}
\item 
\begin{equation}
\begin{aligned}
\int\limits_0^{2t_1-t_2}\int\limits_{-\infty}^{\infty}\int\limits_0^{\infty}
|k^{e}_{t_2-s}(y_j,z_j)- & k^{e}_{t_1-s}(y_j,z_j)||\pa_{x_n}k^{b_n}_{t_1-s}(x_n,w_n)| \\
& |\sqrt{x_n}-\sqrt{w_n}|^{\gamma}dw_l dz_jds.
\end{aligned}
\end{equation}
\end{enumerate}
Terms of types I, and II were shown, in the proof of~\eqref{1stdertmestn0.0},
to be bounded by $C|t_2-t_1|^{\frac{\gamma}{2}},$
leaving just the term of type III. Using Lemma~\ref{lem2new}
and Lemma~\ref{lem4newe} we see that these terms are bounded by
\begin{equation}
  \int\limits_{t_2-t_1}^{t_1}(t_2-t_1)\sigma^{\frac{\gamma}{2}-2}d\sigma\leq C|t_2-t_1|^{\frac{\gamma}{2}}.
\end{equation}

The argument for estimating the differences
\begin{equation}
  |\pa_{y_j}u(\bx,\by,t_2)-\pa_{y_j}u(\bx,\by,t_1)|
\end{equation}
is essentially identical, though the results are a bit different. We can assume
that $j=m$ and use the analogue of~\eqref{eqn218.01} with $\pa_{x_n}$ replaced
with $\pa_{y_m}$ and $g(\bw_n',w_n,\bz,s)-g(\bw_n',x_n,\bz,s)$ replaced by
\begin{equation}
  g(\bw,\bz'_m,z_m,s)-g(\bw,\bz'_m,y_m,s).
\end{equation}
The contributions of the first and third integrals are then bounded by
\begin{equation}
  2\|g\|_{\WF,0,\gamma}\int\limits_0^{2(t_2-t_1)}
\int\limits_{-\infty}^{\infty}|\pa_{y_m}k^e_s(y_m,z_m)||y_m-z_m|^{\gamma}dz_mds.
\end{equation}
Applying Lemma~\eqref{lem2newe} we see that this integral is bounded by
\begin{equation}
  C\int\limits_{0}^{2(t_2-t_1)}s^{\frac{\gamma-1}{2}}ds=\frac{2}{1+\gamma}[2|t_2-t_1|]^{\frac{\gamma+1}{2}},
\end{equation}
which suffices to prove the desired estimate. 

This leaves the analogue of the second integral in~\eqref{eqn218.01}, which we
expand using the analogue of~\eqref{1stdern0ktksdifnm} and~\eqref{n0ktksdif},
replacing $\pa_{x_n}$ with $\pa_{y_m}.$ We need to estimate three types of
terms:
\begin{enumerate}
\renewcommand{\labelenumi}{\Roman{enumi}.}
\item 
\begin{equation}
\int\limits_0^{2t_1-t_2}\int\limits_0^{\infty}\int\limits_{-\infty}^{\infty}
|k^{b_l}_{t_2-s}(x_l,w_l)-k^{b_l}_{t_1-s}(x_l,w_l)||\pa_{y_m}k^{e}_{t_2-s}(y_m,z_m)|
|y_m-z_m|^{\gamma}dz_mdw_l ds,
\end{equation}
\item 
\begin{equation}
\int\limits_0^{2t_1-t_2}\int\limits_{-\infty}^{\infty}
|\pa_{y_m}[k^{e}_{t_2-s}(y_m,z_m)-k^{e}_{t_1-s}(y_m,z_m)]|
|y_m-z_m|^{\gamma}dz_mds,
\end{equation}
\item 
\begin{equation}
\int\limits_0^{2t_1-t_2}\int\limits_0^{\infty}\int\limits_0^{\infty}
|k^{e}_{t_2-s}(y_j,z_j)-k^{e}_{t_1-s}(y_j,z_j)||\pa_{y_m}k^{e}_{t_1-s}(y_m,z_m)|
|y_m-z_m|^{\gamma}dz_m dz_jds.
\end{equation}
\end{enumerate}
Lemma~\ref{lem4newp2} and Lemma~\ref{lem2newe} show that terms of type
I are  bounded by bounded by
\begin{equation}
  \int\limits_{0}^{2t_1-t_2}\frac{(t_2-t_1)(t_1-s)^{\frac{\gamma-1}{2}}ds}
{(t_2-s)}\leq 
\int\limits_{0}^{2t_1-t_2}(t_2-t_1)(t_1-s)^{\frac{\gamma-3}{2}}ds
\leq C|t_2-t_1|^{\frac{\gamma+1}{2}}.
\end{equation}
Using Lemma~\ref{lemAA-e} we see that the terms of type II are bounded by
\begin{equation}
  C|t_2-t_1|^{\frac{\gamma+1}{2}}.
\end{equation}
The second estimate in Lemma~\ref{lem4newe} and Lemma~\ref{lem2newe} show that
terms of type III are also bounded by
\begin{equation}\label{eqn12.57.4}
  C\int\limits_{t_2-t_1}^{t_1}\frac{s^{\frac{\gamma-1}{2}}(t_2-t_1)ds}{t_2-t_1+s}\leq
  C|t_2-t_1|^{\frac{\gamma+1}{2}}. 
\end{equation}
Thus we see that there is a constant $C$ so that we have:
\begin{equation}
  |\nabla_{\by}u(\bx,\by,t_2)-\nabla_{\by}u(\bx,\by,t_1)|\leq
C\|g\|_{\WF,0,\gamma}|t_2-t_1|^{\frac{\gamma+1}{2}}.
\end{equation}

If $t_1<t_2,$ then~\eqref{eqn12.57.4}  also gives the estimate
\begin{equation}
  |\nabla_{\by}u(\bx,\by,t_2)-\nabla_{\by}u(\bx,\by,t_1)|\leq
Ct_2^{\frac{1}{2}}\|g\|_{\WF,0,\gamma}|t_2-t_1|^{\frac{\gamma}{2}},
\end{equation}
completing the proof of~\eqref{tngtestnm} as well as the proof that, if
$t_1,t_2<T,$ then there is a constant $C$ so that the first derivatives
satisfy
\begin{multline}
   |\nabla_{\bx,\by} u(\bx^2,\by^2,t_2)-\nabla_{\bx,\by} u(\bx^1,\by^1,t_1)|\leq\\
C(1+T^{\frac{\gamma}{2}})\|g\|_{\WF,0,\gamma}[\rho_s(\bx^1,\bx^2)+\rho_e(\by^1,\by^2)+\sqrt{|t_2-t_1|}]^{\gamma}.
\end{multline}

\subsection{Second derivative estimates}
This brings us to the second derivatives. As it is essentially the same as the
1-dimensional case, Lemma~\ref{lemA}
suffices to prove the bounds
\begin{equation}\label{eqn476new0.0}
  |x_l\pa_{x_l}^2u(\bx,\by,t)|\leq C\|g\|_{\WF,0,\gamma}\min\{x_l^{\frac{\gamma}{2}},t^{\frac{\gamma}{2}}\},
\end{equation}
for $l=1,\dots,n.$ The calculations between~\eqref{eqn10.156.7}
and~\eqref{378new1.0} suffice to prove that for $1\leq l,k\leq n,$ we have the estimates
\begin{equation}\label{eqn340.01}
  |\sqrt{x_lx_k}\pa_{x_l}\pa_{x_k}u(\bx,\by,t)|\leq C\|g\|_{\WF,0,\gamma}
\min\{x_l^{\frac{\gamma}{2}},x_k^{\frac{\gamma}{2}},t^{\frac{\gamma}{2}}\}
\end{equation}
Using Lemmas~~\ref{lem2newe} and~\ref{lem25newe}, we
easily derive the estimates
\begin{equation}\label{eqn341.01}
  |\pa_{y_j}\pa_{y_k}u(\bx,\by,t)|\leq C\|g\|_{\WF,0,\gamma}t^{\frac{\gamma}{2}},
\end{equation}
where $1\leq j,k\leq m.$ Using Lemmas~\ref{lem2new} and~\ref{lem2newe}  we can also show that
\begin{equation}\label{eqn353.00}
   |\sqrt{x_l}\pa_{x_l}\pa_{y_j}u(\bx,\by,t)|\leq C
   \|g\|_{\WF,0,\gamma}\min\{x_l^{\frac{\gamma}{2}},t^{\frac{\gamma}{2}}\},
\end{equation}
for $1\leq l\leq n$ and $1\leq j\leq m.$ 

To complete the spatial part of the estimate, we need to show that the second
derivatives are H\"older continuous. As before, the earlier arguments suffice
to show that
\begin{multline}
|\sqrt{x_l^2x_k^2}\pa_{x_l}\pa_{x_k}u(\bx^2,\by,t)-\sqrt{x_l^1x_k^1}\pa_{x_l}\pa_{x_k}u(\bx^1,\by,t)|\leq \\
C\|g\|_{\WF,0,\gamma}\rho_s(\bx^1,\bx^2)^{\gamma}\text{ and }\\ 
|\pa_{y_j}\pa_{y_k}u(\bx,\by^2,t)-\pa_{y_j}\pa_{y_k}u(\bx,\by^1,t)|\leq 
C\|g\|_{\WF,0,\gamma}\rho_e(\by^1,\by^2)^{\gamma}.
\end{multline}
Thus we are left to estimate
\begin{equation}\label{mxvar2ders0}
\begin{split}
  &|\sqrt{x_lx_k}\pa_{x_l}\pa_{x_k}u(\bx,\by^2,t)-\sqrt{x_lx_k}\pa_{x_l}\pa_{x_k}u(\bx,\by^1,t)|
\text{ and }\\
&|\pa_{y_j}\pa_{y_k}u(\bx^2,\by,t)-\pa_{y_j}\pa_{y_k}u(\bx^1,\by,t)|,
\end{split}
\end{equation}
and the mixed derivatives $\sqrt{x_l}\pa_{x_l}\pa_{y_j}u.$

We begin with  the quantities in~\eqref{mxvar2ders0}, by considering
$$|\sqrt{x_lx_k}\pa_{x_l}\pa_{x_k}u(\bx,\by^2,t)-\sqrt{x_lx_k}\pa_{x_l}\pa_{x_k}u(\bx,\by^1,t)|,$$
with $k\neq l.$ Without loss of generality we can assume  $k=1, l=2$ and
$\by^2-\by^1=(y^2_1-y^1_1,0,\dots,0).$ With these assumptions, using the observation that
\begin{equation}
  \int\limits_0^{\infty}\pa_{x_1}k^{b_1}_s(x_1,w_1)g(x_1,\bw_0'',\bz,t-s)dw_1=0,
\end{equation}
for all values of $(\bw_0'',\bz,t-s),$ we get  the estimate
\begin{multline}\label{eqn12.97.01}
  |\sqrt{x_1x_2}\pa_{x_1}\pa_{x_2}u(\bx,\by^2,t)-\sqrt{x_lx_k}\pa_{x_l}\pa_{x_k}u(\bx,\by^1,t)|
\leq\\
2\|g\|_{\WF,0,\gamma}\int\limits_0^t\int\limits_0^{\infty}\int\limits_0^{\infty}\int\limits_{-\infty}^{\infty}
|\sqrt{x_1}\pa_{x_1}k^{b_1}_{s}(x_1,w_1)\sqrt{x_2}\pa_{x_2}k^{b_2}_{s}(x_2,w_2)|
|\sqrt{x_1}-\sqrt{w_1}|^{\gamma}\times\\
|k^e_s(y_1^2,z_1)-k_s^e(y_1^1,z_1)|dw_1dw_2dz_1ds.
\end{multline}
Applying Lemmas~\ref{lem2new} and~\ref{lem21newe}, we see that the integral is
estimated by
\begin{equation}\label{eqn12.98.01}
    C\int\limits_0^t\frac{\sqrt{x_1}s^{\frac{\gamma}{2}-1}}{1+\sqrt{x_1/s}}
\frac{\sqrt{x_2}s^{-1}}{1+\sqrt{x_2/s}}
\left(\frac{\frac{|y^2_1-y^1_1|}{\sqrt{s}}}
{1+\frac{|y^2_1-y^1_1|}{\sqrt{s}}}\right)ds\leq
C\int\limits_0^ts^{\frac{\gamma}{2}-1}
\left(\frac{\frac{|y^2_1-y^1_1|}{\sqrt{s}}}
{1+\frac{|y^2_1-y^1_1|}{\sqrt{s}}}\right)ds,
\end{equation}
and that
\begin{equation}\label{eqn359.00}
  \begin{split}
    \int\limits_0^ts^{\frac{\gamma}{2}-1}
\left(\frac{\frac{|y^2_1-y^1_1|}{\sqrt{s}}}
{1+\frac{|y^2_1-y^1_1|}{\sqrt{s}}}\right)ds&\leq
\int\limits_0^{|y^2_1-y^1_1|^2}s^{\frac{\gamma}{2}-1}ds+
\int\limits_{|y^2_1-y^1_1|^2}^t|y^2_1-y^1_1|s^{\frac{\gamma-3}{2}}ds\\
&\leq \frac{2}{\gamma(1-\gamma)}|y^2_1-y^1_1|^{\gamma},
  \end{split}
\end{equation}
which completes this case.

We next consider this situation with $k=l;$ we can take $k=l=1,$ with
$\by^2-\by^1$ as before. Using the fact that
\begin{equation}
  \int\limits_0^{\infty}\pa_{x_1}^2k^{b_1}_s(x_1,w_1)g(x_1,\bw_0'',\bz,t-s)dw_1=0,
\end{equation}
for all values of $(\bw_0'',\bz,t-s),$ we get  the estimate
\begin{multline}
  |x_1\pa_{x_1}^2u(\bx,\by^2,t)-x_1\pa_{x_1}^2u(\bx,\by^1,t)|
\leq\\
2\|g\|_{\WF,0,\gamma}\int\limits_0^t\int\limits_0^{\infty}\int\limits_{-\infty}^{\infty}
|x_1\pa_{x_1}^2k^{b_1}_{s}(x_1,w_1)|
|\sqrt{x_1}-\sqrt{w_1}|^{\gamma}\times\\
|k^e_s(y_1^2,z_1)-k_s^e(y_1^1,z_1)|dw_1dz_1ds.
\end{multline}
We now apply Lemma~\ref{lem25new} and Lemma~\ref{lem21newe} to see that this
integral is bounded by
\begin{multline}\label{eqn362.00}
C\int\limits_0^t\frac{\sqrt{x_1}s^{\frac{\gamma}{2}-1}}{\sqrt{x_1}+\sqrt{s}}
\left(\frac{\frac{|y^2_1-y^1_1|}{\sqrt{s}}}
{1+\frac{|y^2_1-y^1_1|}{\sqrt{s}}}\right)ds \\
\leq  C\left[\int\limits_0^{|y^2_1-y^1_1|^2}s^{\frac{\gamma}{2}-1}ds+ 
\int\limits_{|y^2_1-y^1_1|^2}^t|y^2_1-y^1_1|s^{\frac{\gamma-3}{2}}ds\right]\\ 
\leq \frac{2C}{\gamma(1-\gamma)}|y^2_1-y^1_1|^{\gamma}.
\end{multline}
We have implicitly assumed that $t>|y^2_1-y^1_1|^2;$ if this is not the case,
then one gets a single term in~\eqref{eqn359.00} and~\eqref{eqn362.00}.
Otherwise the argument is identical. This completes the spatial-part of the
H\"older estimate for the second $x$-derivatives.

We now turn to
$|\pa_{y_j}\pa_{y_k}u(\bx^2,\by,t)-\pa_{y_j}\pa_{y_k}u(\bx^1,\by,t)|;$ by considering
$j\neq k.$ We can assume that $j=1,$ $k=2,$ and
$\bx^2-\bx^1=(x^2_1-x^1_1,0,\dots,0).$ We first need to consider the case where
$x_1^1=0.$ We use the fact that
\begin{equation}
  \int\limits_{-\infty}^{\infty}\pa_{y_1}k^e_s(y_1,z_1)[g(\bw,z_1,\bz''_0,t-s)-
g(\bw,y_1,\bz''_0,t-s)]dz_1=0
\end{equation}
to see that
\begin{multline}\label{eqn12.104.01}
  |\pa_{y_j}\pa_{y_k}u(\bx^2,\by,t)-\pa_{y_j}\pa_{y_k}u(\bx^1,\by,t)|\leq \\
2\|g\|_{\WF,0,\gamma}\int\limits_0^t\int\limits_{-\infty}^{\infty}\int\limits_{-\infty}^{\infty}
\int\limits_{0}^{\infty}|\pa_{y_1}k^e_s(y_1,z_1)\pa_{y_2}k^e_s(y_2,z_2)||y_1-z_1|^{\gamma}\times\\
|k^{b_1}_s(x^2_1,w_1)-k^{b_1}_s(0,w_1)|dw_1dz_1dz_2ds.
\end{multline}
Applying Lemma~\ref{lem2newe} and the first estimate in Lemma~\ref{lem21new} we
see that the integral is estimated by
\begin{equation}
  \int\limits_0^t\frac{s^{\frac{\gamma}{2}-1}x^2_1ds}{x^2_1+s}\leq
  \frac{4}{\gamma(2-\gamma)}
(x^2_1)^{\frac{\gamma}{2}}.
\end{equation}
Applying~\eqref{lrgratioest}, this estimate implies that for any $0<c<1,$ there is a constant
$C$ so that if $x^1_1<cx^2_1,$ then
\begin{equation}
   |\pa_{y_j}\pa_{y_k}u(\bx^2,\by,t)-\pa_{y_j}\pa_{y_k}u(\bx^1,\by,t)|\leq C
\|g\|_{\WF,0,\gamma}\left|\sqrt{x^2_1}-\sqrt{x^1_1}\right|^{\gamma}.
\end{equation}
We are therefore reduced to the case $cx^2_1<x^1_1<x^2_1.$ In this case we have
\begin{multline}
  |\pa_{y_j}\pa_{y_k}u(\bx^2,\by,t)-\pa_{y_j}\pa_{y_k}u(\bx^1,\by,t)|\leq \\
2\|g\|_{\WF,0,\gamma}\int\limits_0^t\int\limits_{-\infty}^{\infty}\int\limits_{-\infty}^{\infty}
\int\limits_{0}^{\infty}|\pa_{y_1}k^e_s(y_1,z_1)\pa_{y_2}k^e_s(y_2,z_2)||y_1-z_1|^{\gamma}\times\\
|k^{b_1}_s(x^2_1,w_1)-k^{b_1}_s(x^1_1,w_1)|dw_1dz_1dz_2ds,
\end{multline}
which can be estimated using Lemma~\ref{lem2newe} and the second estimate in
Lemma~\ref{lem21new}.  These lemmas show that the integral is bounded by
\begin{multline}\label{eqn368.01}
 C \int\limits_0^ts^{\frac{\gamma}{2}-1}\left(\frac{\frac{\sqrt{x^2_1}-\sqrt{x^1_1}}{\sqrt{s}}}
{1+\frac{\sqrt{x^2_1}-\sqrt{x^1_1}}{\sqrt{s}}}\right)ds\leq \\
C\left[\int\limits_0^{(\sqrt{x^2_1}-\sqrt{x^1_1})^2}s^{\frac{\gamma}{2}-1}ds+
\int\limits_{(\sqrt{x^2_1}-\sqrt{x^1_1})^2}^ts^{\frac{\gamma-3}{2}}
\left|\sqrt{x^2_1}-\sqrt{x^1_1}\right|ds\right]\\
\leq \frac{2}{\gamma(1-\gamma)}\left|\sqrt{x^2_1}-\sqrt{x^1_1}\right|^{\gamma}.
\end{multline}

The case $j=k=1$ follows exactly the same pattern. We use the fact that
\begin{equation}
   \int\limits_{-\infty}^{\infty}\pa_{y_1}^2k^e_s(y_1,z_1)
g(\bw,y_1,\bz''_0,t-s)dz_1=0
\end{equation}
for any values of $(\bw,\bz''_0,t-s),$ and the estimate~\eqref{lem25newpeste} to
see that
\begin{equation}\label{eqn12.110.01}
  \int\limits_{-\infty}^{\infty}|\pa_{y_1}^2k^e_s(y_1,z_1)||y_1-z_1|^{\gamma}\leq Cs^{\frac{\gamma}{2}-1}.
\end{equation}
From this point the argument used for the case $j\neq k$ can be followed
verbatim. We have again implicitly assumed that
$t>\left(\sqrt{x^2_1}-\sqrt{x^1_1}\right)^2.$ If this is not the case, then we
get only the first
term in the second line of~\eqref{eqn368.01}; otherwise the argument is unchanged.

To complete the spatial estimates we need only show that the mixed partial
derivatives $\sqrt{x_l}\pa_{x_l}\pa_{y_j}u(\bx,\by,t)$ are H\"older
continuous. Without loss of generality we can take $j=l=1.$  As usual we can
assume that the points of evaluation differ in a single coordinate. We start by
considering variations in the $x$-variables. There are two cases to consider:
1. The $x$-variable differs in the first slot, 2. The $x$-variable differs in
another slot.

For case 1, we first need to take $x_1^1=0.$ In this case we see that the
second estimate in~\eqref{eqn353.00} and~\eqref{lrgratioest} imply that, if
$0<c<1,$ then there is a $C$ so that, for $x^1_1<cx^2_1$ we have the estimate
\begin{equation}\label{eqn12.111.01}
 \left|
    \sqrt{x^2_1}\pa_{x_1}\pa_{y_1}u(\bx^2,\by,t)-\sqrt{x^1_1}\pa_{x_1}\pa_{y_1}u(\bx^1,\by,t)\right| 
\leq C\|g\|_{\WF,0,\gamma}\left|\sqrt{x^2_1}-\sqrt{x^1_1}\right|^{\gamma}.
\end{equation}
We are therefore left to consider $cx^2_1<x^1_1<x^2_1.$  For this case we see
that
\begin{multline}
  \left|
    \sqrt{x^2_1}\pa_{x_1}\pa_{y_1}u(\bx^2,\by,t)-\sqrt{x^1_1}\pa_{x_1}\pa_{y_1}u(\bx^1,\by,t)\right| 
\leq\\ C\|g\|_{\WF,0,\gamma}
\int\limits_0^t\int\limits_0^{\infty}\int\limits_{-\infty}^{\infty}
\left|\sqrt{x^2_1}\pa_{x_1}k^{b_1}_s(x^2_1,w_1)-\sqrt{x^1_1}\pa_{x_1}k^{b_1}_s(x^1_1,w_1)\right|\times\\
\left|\sqrt{x^2_1}-\sqrt{w_1}\right|^{\gamma}|\pa_{y_1}k^e_s(y_1,z_1)|dw_1dz_1ds.
\end{multline}
Applying Lemma~\ref{lem2newe} and~\ref{lem20neww} we see that this integral is
bounded by
\begin{equation}
  C\int\limits_0^ts^{\frac{\gamma}{2}-1}
\frac{\left(\frac{|\sqrt{x^2_1}-\sqrt{x^1_1}|}{\sqrt{s}}\right)}
{1+\left(\frac{|\sqrt{x^2_1}-\sqrt{x^1_1}|}{\sqrt{s}}\right)}ds
\end{equation}
As before we easily establish that, when $|\sqrt{x^2_1}-\sqrt{x^1_1}|<1,$ this is bounded by
$C\left|\sqrt{x^2_1}-\sqrt{x^1_1}\right|^{\gamma}.$

We now turn to the case that $\bx^2-\bx^1$ is non-zero in the $j$th entry where
$j>1.$ As in the previous case, we need to first consider $x^1_j=0.$ In this
case the difference is estimated by
\begin{multline}\label{eqn364.00}
\left|
    \sqrt{x^1_1}\pa_{x_1}\pa_{y_1}u(\bx^2,\by,t)-\sqrt{x^1_1}\pa_{x_1}\pa_{y_1}u(\bx^1,\by,t)\right|
  \leq\\
2\|g\|_{\WF,0,\gamma}\int\limits_0^t\int\limits_0^{\infty}
\int\limits_0^{\infty}\int\limits_{-\infty}^{\infty}
\left|\sqrt{x^1_1}\pa_{x_1}k^{b_1}_s(x^1_1,w_1)\right|
\left|\sqrt{x^1_1}-\sqrt{w_1}\right|^{\gamma}\times\\
|k^{b_j}_s(x^2_j,w_j)-k^{b_j}_s(0,w_j)||\pa_{y_1}k^e_s(y_1,z_1)|dz_1dw_jdw_1dt
\end{multline}
Applying Lemmas~\ref{lem2new},~\ref{lem2newe} and the first estimate in
Lemma~\ref{lem21new} we see that this integral is estimated by
\begin{equation}\label{eqn365.00}
  \int\limits_0^ts^{\frac{\gamma}{2}-1}\frac{x^2_jds}{s+x^2_j}\leq C(x^2_j)^{\frac{\gamma}{2}}.
\end{equation}
Applying~\eqref{lrgratioest} we are reduced to consideration of the case
$cx^2_j<x^1_j<x^2_j,$ for a $0<c<1.$ In this case we use the second estimate in
Lemma~\ref{lem21new} to see that the replacement for~\eqref{eqn365.00} is
\begin{equation}
  C\int\limits_0^ts^{\frac{\gamma}{2}-1}
\frac{\left(\frac{|\sqrt{x^2_j}-\sqrt{x^1_j}|}{\sqrt{s}}\right)}
{1+\left(\frac{|\sqrt{x^2_j}-\sqrt{x^1_j}|}{\sqrt{s}}\right)}ds\leq 
C\left|\sqrt{x^2_j}-\sqrt{x^1_j}\right|^{\gamma}. 
\end{equation}
This completes the proof that
\begin{equation}
 \left|
    \sqrt{x^2_l}\pa_{x_l}\pa_{y_j}u(\bx^2,\by,t)-\sqrt{x^1_l}\pa_{x_l}\pa_{y_j}u(\bx^1,\by,t)\right|\leq
C \|g\|_{\WF,0,\gamma}\rho_s(\bx^1,\bx^2)^{\gamma}.
\end{equation}

We are left to consider
$\left|
    \sqrt{x_1}\pa_{x_1}\pa_{y_1}u(\bx,\by^2,t)-\sqrt{x_1}\pa_{x_1}\pa_{y_1}u(\bx,\by^2,t)\right|,$
  where, as before, we need to distinguish between the case that the
  $y$-variables differ in the first coordinate and in other coordinates. If
  $\by^2-\by^1=(y^2_1-y^1_1,0,\dots,0)$ then this difference estimated by
  \begin{multline}
    2\|g\|_{\WF,0,\gamma}\int\limits_0^t\int\limits_0^{\infty}\int\limits_{-\infty}^{\infty}
|\sqrt{x_1}k_s^{b_1}(x_1,w_1)|\times\\
|\pa_{y_1}k^e_s(y^2_1,z_1)-\pa_{y_1}k^e_s(y^1_1,z_1)||y^1_1-z_1|^{\gamma}dz_1dw_1ds.
  \end{multline}
Lemmas~\ref{lem2new} and~\ref{lem20newwe} show that this integral is estimated
by
\begin{equation}
  \int\limits_0^ts^{\frac{\gamma}{2}-1}
\frac{\left(\frac{|y^2_1-y^1_1|}{\sqrt{s}}\right)}
{1+\left(\frac{|y^2_1-y^1_1|}{\sqrt{s}}\right)}ds\leq C\left|y^2_1-y^1_1\right|^{\gamma}.
\end{equation}
If the $y$-variables differ in another coordinate, then the difference of
second derivatives is estimated by
\begin{multline}
  2\|g\|_{\WF,0,\gamma}\int\limits_0^t\int\limits_0^{\infty}
\int\limits_{-\infty}^{\infty}\int\limits_{-\infty}^{\infty}|\sqrt{x_1}k_s^{b_1}(x_1,w_1)|\times\\
|\pa_{y_1}k^e_s(y^1_1,z_1)||y^1_1-z_1|^{\gamma}|k^e_s(y^2_j,z_j)-k^e_s(y^1_j,z_j)|dz_jdz_1dw_1ds.
\end{multline}
We now apply Lemma~\ref{lem2new},~\ref{lem2newe}, and~\ref{lem21newe} to see
that the integral is bounded by
\begin{equation}
  \int\limits_0^t s^{\frac{\gamma}{2}-1}\frac{\left(\frac{|y^2_j-y^1_j|}{\sqrt{s}}\right)}
{1+\left(\frac{|y^2_j-y^1_j|}{\sqrt{s}}\right)}ds
\leq C\left|y^2_j-y^1_j\right|^{\gamma}.
\end{equation}
This completes the proof that
\begin{equation}\label{eqn12.122.01}
  \left|
    \sqrt{x^1_l}\pa_{x_l}\pa_{y_j}u(\bx,\by^2,t)-\sqrt{x^1_l}\pa_{x_l}\pa_{y_j}u(\bx,\by^1,t)\right|\leq
C \|g\|_{\WF,0,\gamma}\rho_e(\by^1,\by^2)^{\gamma}
\end{equation}

To finish the proof of the proposition we need to establish the H\"older
continuity in time of the second spatial-derivatives of $u.$
Using~\eqref{lrgratioest} along with the estimates
in~\eqref{eqn340.01},~\eqref{eqn341.01}, and~\eqref{eqn353.00}, we see that
for any $0<c<1,$ there is a constant $C$ so that, if $t_1<ct_2,$ then
\begin{equation}
  \begin{split}
    &|\sqrt{x_lx_k}\pa_{x_l}\pa_{x_k}u(\bx,\by,t_2)-
\sqrt{x_lx_k}\pa_{x_l}\pa_{x_k}u(\bx,\by,t_1)|\leq
C\|g\|_{\WF,0,\gamma}|t_2-t_1|^{\frac{\gamma}{2}}\\
&|\sqrt{x_l}\pa_{x_l}\pa_{y_j}u(\bx,\by,t_2)-
\sqrt{x_l}\pa_{x_l}\pa_{y_j}u(\bx,\by,t_1)|\leq
C\|g\|_{\WF,0,\gamma}|t_2-t_1|^{\frac{\gamma}{2}}\\
&|\pa_{y_l}\pa_{y_j}u(\bx,\by,t_2)-
\pa_{y_l}\pa_{y_j}u(\bx,\by,t_1)|\leq C\|g\|_{\WF,0,\gamma}|t_2-t_1|^{\frac{\gamma}{2}}.
  \end{split}
\end{equation}
We are left to consider these differences for $ct_2<t_1<t_2,$ where we assume
that $\frac 12<c<1.$ For all these cases we use an expansion like that
in~\eqref{eqn218.01}, with the operator $\pa_{x_n}$ replaced by the appropriate
second order operator.

We first treat the pure $x$-derivatives, $\sqrt{x_lx_k}\pa_{x_l}\pa_{x_k}u,$
where $l\neq k.$ Without loss of generality we can assume that
$k=n.$  By replacing $g(\bw_n',w_n,\bz,t_2-s)- g(\bw_n',w_n,\bz,t_1-s)$ with
\begin{equation}
  g(\bw_n',w_n,\bz,t_2-s)-g(\bw_n',x_n,\bz,t_2-s)+
g(\bw_n',x_n,\bz,t_1-s)-
g(\bw_n',w_n,\bz,t_1-s)
\end{equation}
 in the first integral in the analogue of~\eqref{eqn218.00}, we see that it is estimated by
\begin{multline}
  C\|g\|_{\WF,0,\gamma}\int\limits_{0}^{t_2-t_1}
\int\limits_0^{\infty}\int\limits_0^{\infty}\sqrt{x_nx_{l}}|\pa_{x_n}k^{b_n}_s(x_n,w_n)
\pa_{x_l}k^{b_l}_s(x_l,w_l)|\times\\
|\sqrt{w_n}-\sqrt{x_n}|^{\gamma}dw_ldw_nds
\end{multline}
Using Lemma~\ref{lem2new} the integral in this term is estimated by
\begin{equation}
 C \int\limits_0^{t_2-t_1}\frac{\sqrt{x_lx_n}s^{\frac{\gamma}{2}-1}ds}
{(\sqrt{x_l}+\sqrt{s})(\sqrt{x_n}+\sqrt{s})}\leq C|t_2-t_1|^{\gamma}.
\end{equation}
The last integral in the analogue of~\eqref{eqn218.00} is easily seen to be
bounded by
\begin{equation}
 C \int\limits_{t_2-t_1}^{2(t_2-t_1)}\frac{\sqrt{x_lx_n}s^{\frac{\gamma}{2}-1}ds}
{(\sqrt{x_l}+\sqrt{s})(\sqrt{x_n}+\sqrt{s})}\leq C|t_2-t_1|^{\gamma}.
\end{equation}
This leaves only the second integral in the analogue of~\eqref{eqn218.00},
which we replace by a sum of terms using the analogue
of~\eqref{1stdern0ktksdifnm} and~\eqref{n0ktksdif}. All the possible terms that
arise from the analogue of the first term on the right hand side
of~\eqref{1stdern0ktksdifnm}  are enumerated
in~\eqref{eqn261.01}--~\eqref{eqn265.01}, and shown to be bounded by
$C\|g\|_{\WF,0,\gamma}|t_2-t_1|^{\frac{\gamma}{2}}.$ The second term on the
right hand side of~\eqref{1stdern0ktksdifnm}  produces an additional type of
term:
\begin{multline}
  \int\limits_{0}^{2t_1-t_2}\int\limits_{-\infty}^{\infty}
\int\limits_{0}^{\infty}\int\limits_{0}^{\infty}|k^e_{t_2-s}(y_j,z_j)-k^e_{t_1-s}(y_j,z_j)|\times\\
\sqrt{x_nx_{l}}|\pa_{x_n}k^{b_n}_{t_1-s}(x_n,w_n)
\pa_{x_l}k^{b_l}_{t_1-s}(x_l,w_l)|
|\sqrt{w_n}-\sqrt{x_n}|^{\gamma}dw_ldw_ndz_jds
\end{multline}
Lemma~\ref{lem2new} and the second estimate in Lemma~\ref{lem4newe}
show that this integral is bounded by
\begin{equation}
  \int\limits_{t_2-t_1}^{t_1}\frac{\sqrt{x_lx_n}(t_2-t_1)s^{\frac{\gamma}{2}-1}ds}
{(\sqrt{x_l}+\sqrt{s})(\sqrt{x_s}+\sqrt{s})(t_2-t_1+s)}\leq 
 \int\limits_{t_2-t_1}^{t_1}|t_2-t_1|s^{\frac{\gamma}{2}-2}ds\leq C|t_2-t_1|^{\frac{\gamma}{2}}.
\end{equation}

Now we need to consider the case $k=l=n.$ For these cases we are free to
replace $x_n\pa_{x_n}^2$ with $L_{b_n,x_n}.$ Most of the terms that arise in
this case have been treated in the proof of Proposition~\ref{prop1n0}. The only
new type of term  arises from expanding the second term on right hand
side of the analogue of~\eqref{1stdern0ktksdifnm} in the second integral. These
are of the form
\begin{multline}
  \int\limits_0^{2t_1-t_2}\int\limits_{-\infty}^{\infty}\int\limits_0^{\infty}
|k^e_{t_2-s}(y_j,z_j)-k^e_{t_1-s}(y_j,z_j)|\times\\
|x_n\pa_{x_n}^2k^{b_n}_{t_1-s}(x_n,w_n)
|\sqrt{w_n}-\sqrt{x_n}|^{\gamma}dw_ndz_jds
\end{multline}
Lemma~\ref{lem4newe} and Lemma~\ref{lem25new} show that this term is bounded by
\begin{equation}
  C\int\limits_{t_2-t_1}^{t_1}\frac{x_n
    (t_2-t_1)s^{\frac{\gamma}{2}-1}ds}{(s+x_n)(t_2-t_1+s)}\leq C|t_2-t_1|^{\frac{\gamma}{2}}.
\end{equation}
This completes the proof that 
\begin{equation}
  |\sqrt{x_lx_k}\pa_{x_l}\pa_{x_k}[u(\bx,\by,t_2)-
u(\bx,\by,t_1)]|\leq C\|g\|_{\WF,0,\gamma}|t_2-t_1|^{\frac{\gamma}{2}}.
\end{equation}
The verification that $\pa_{y_j}\pa_{y_k}u(\bx,\by,t)$ satisfies the same
estimate is essentially identical, simply interchanging estimates for $k^b_s$
with estimates for $k^e_s$ and vice versa. We leave the details to the
interested reader. 

To conclude the proof of Proposition~\ref{prop6.2} in the $k=0$ case, we verify
that
\begin{equation}
  | \sqrt{x_l}\pa_{x_l}\pa_{y_j}[u(\bx,\by,t_2)-
u(\bx,\by,t_1)]|\leq C\|g\|_{\WF,0,\gamma}|t_2-t_1|^{\frac{\gamma}{2}}.
\end{equation}
To prove this estimate we use the expression  in~\eqref{eqn218.00} with
$\pa_{x_n}$ replaced by $\sqrt{x_n}\pa_{x_n}\pa_{y_m}.$ The first and third
integrals are estimated by
\begin{multline}
  \int\limits_{0}^{2(t_2-t_1)}\int\limits_{-\infty}^{\infty}\int\limits_0^{\infty}
|\pa_{y_m}k^e_{s}(y_m,z_m)\sqrt{x_n}\pa_{x_n}k^{b_n}_s(x_n,w_n)|\times\\
|\sqrt{x_n}-\sqrt{w_n}|^{\gamma}dz_mdw_nds
\end{multline}
We use Lemma~\ref{lem2new} and~\ref{lem2newe} to see that this integral is
bounded by
\begin{equation}
  C\int\limits_0^{2(t_2-t_1)}\frac{\sqrt{x_n}s^{\frac{\gamma}{2}-1}ds}{\sqrt{x_n}+\sqrt{s}}\leq
  C |t_2-t_1|^{\frac{\gamma}{2}}.
\end{equation}

Two cases arise in the estimation of the contribution of first term on the
right hand side of~\eqref{1stdern0ktksdifnm}. In the first case we get terms of
the form:
\begin{multline}
  \int\limits_{0}^{2t_1-t_2}\int\limits_{-\infty}^{\infty}\int\limits_0^{\infty}
|\pa_{y_m}k^e_{t_2-s}(y_m,z_m)|\times\\
|\sqrt{x_n}\pa_{x_n}[k^{b_n}_{t_2-s}(x_n,w_n)-k^{b_n}_{t_1-s}(x_n,w_n)]||\sqrt{x_n}-\sqrt{w_n}|^{\gamma}
dz_mdw_nds
\end{multline}
Applying Lemma~\ref{lem2newe} and~\ref{lemAA-} we see that this term is bounded
by
\begin{equation}\label{eqn395.00}
  \int\limits_{t_2-t_1}^{t_1}\frac{\sqrt{x_n}(t_2-t_1)s^{\frac{\gamma}{2}-1}ds}
{(\sqrt{x_n}+\sqrt{s})(t_2-t_1+s)}\leq 
 \int\limits_{t_2-t_1}^{t_1}(t_2-t_1)s^{\frac{\gamma}{2}-2}ds\leq C|t_2-t_1|^{\frac{\gamma}{2}}.
\end{equation}
For the second case we have terms of the form
\begin{multline}
  \int\limits_{0}^{2t_1-t_2}\int\limits_{-\infty}^{\infty}\int\limits_0^{\infty}\int\limits_0^{\infty}
|\pa_{y_m}k^e_{t_2-s}(y_m,z_m)||\sqrt{x_n}\pa_{x_n}k^{b_n}_{t_1-s}(x_n,w_n)|\times\\
|k^{b_l}_{t_2-s}(x_l,w_l)-k^{b_l}_{t_1-s}(x_l,w_l)||\sqrt{x_n}-\sqrt{w_n}|^{\gamma}
dz_mdw_ldw_nds
\end{multline}
and
\begin{multline}
  \int\limits_{0}^{2t_1-t_2}\int\limits_{-\infty}^{\infty}\int\limits_0^{\infty}\int\limits_0^{\infty}
|\pa_{y_m}k^e_{t_2-s}(y_m,z_m)||\sqrt{x_n}\pa_{x_n}k^{b_n}_{t_2-s}(x_n,w_n)|\times\\
|k^{b_l}_{t_2-s}(x_l,w_l)-k^{b_l}_{t_1-s}(x_l,w_l)||\sqrt{x_n}-\sqrt{w_n}|^{\gamma}
dz_mdw_ldw_nds
\end{multline}
Applying Lemmas~\ref{lem2new},~\ref{lem2newe} and~\ref{lem4newp2} shows that
these terms are estimated by
\begin{multline}
 C\int\limits_{t_2-t_1}^{t_1}\frac{\sqrt{x_n}s^{\frac{\gamma-1}{2}}(t_2-t_1+s)^{-\frac{1}{2}}(t_2-t_1)ds}
{(\sqrt{x_n}+\sqrt{s})(t_2-t_1+s)}\leq  \\ C
\int\limits_{t_2-t_1}^{t_1}(t_2-t_1)s^{\frac{\gamma}{2}-2}ds\leq C|t_2-t_1|^{\frac{\gamma}{2}}.
\end{multline}

Two cases also arise in the estimation of the contribution of second term on the
right hand side of~\eqref{1stdern0ktksdifnm}. For the first case we get
\begin{multline}
    \int\limits_{0}^{2t_1-t_2}\int\limits_{-\infty}^{\infty}\int\limits_0^{\infty}
|\pa_{y_m}k^e_{t_2-s}(y_m,z_m)-\pa_{y_m}k^e_{t_1-s}(y_m,z_m)|\times\\
|\sqrt{x_n}\pa_{x_n}k^{b_n}_{t_1-s}(x_n,w_n)||\sqrt{x_n}-\sqrt{w_n}|^{\gamma}dz_mdw_nds.
\end{multline}
Lemma~\ref{lem2new} and Lemma~\ref{lemAA-e} show that this integral is
estimated by the expression in~\eqref{eqn395.00}. For the second case we get
\begin{multline}
  \int\limits_{0}^{2t_1-t_2}\int\limits_{-\infty}^{\infty}\int\limits_0^{\infty}\int\limits_0^{\infty}
|\pa_{y_m}k^e_{t_2-s}(y_m,z_m)||\sqrt{x_n}\pa_{x_n}k^{b_n}_{t_1-s}(x_n,w_n)|\times\\
|k^{e}_{t_2-s}(y_k,z_k)-k^{e}_{t_1-s}(y_k,z_k)||\sqrt{x_n}-\sqrt{w_n}|^{\gamma}
dz_mdz_kdw_nds
\end{multline}
and
\begin{multline}
  \int\limits_{0}^{2t_1-t_2}\int\limits_{-\infty}^{\infty}\int\limits_0^{\infty}\int\limits_0^{\infty}
|\pa_{y_m}k^e_{t_1-s}(y_m,z_m)||\sqrt{x_n}\pa_{x_n}k^{b_n}_{t_1-s}(x_n,w_n)|\times\\
|k^{e}_{t_2-s}(y_k,z_k)-k^{e}_{t_1-s}(y_k,z_k)||\sqrt{x_n}-\sqrt{w_n}|^{\gamma}
dz_mdz_kdw_nds.
\end{multline}
Using Lemmas~\ref{lem2new},~\ref{lem2newe} and~\ref{lem4newe} we see that these
terms are estimated by
\begin{equation}
  C\int\limits_{t_2-t_1}^{t_1}\frac{\sqrt{x_n}s^{\frac{\gamma-1}{2}}s^{-\frac{1}{2}}(t_2-t_1)ds}
{(\sqrt{x_n}+\sqrt{s})(t_2-t_1+s)}\leq C
\int\limits_{t_2-t_1}^{t_1}(t_2-t_1)s^{\frac{\gamma}{2}-2}ds\leq C|t_2-t_1|^{\frac{\gamma}{2}}.
\end{equation}

This completes the proof that the second derivatives satisfy the
appropriate H\"older estimates: there is a constant $C$ uniformly bounded for
$\bzero<\bb<\BB$ so that:
\begin{multline}
    \max_{1\leq k,l\leq n;\, 1\leq i,j\leq m}\Bigg\{
\left|\sqrt{x^2_lx^2_k}\pa_{x_l}\pa_{x_k}u(\bx^2,\by^2,t_2)-
\sqrt{x^1_lx^1_k}\pa_{x_l}\pa_{x_k}u(\bx^1,\by^1,t_1)\right|,\\
\left|\sqrt{x^2_l}\pa_{x_l}\pa_{y_j}u(\bx^2,\by^2,t_2)-
\sqrt{x^1_l}\pa_{x_l}\pa_{y_j}u(\bx^1,\by^1,t_1)\right|,\\
|\pa_{y_j}\pa_{y_i}u(\bx^2,\by^2,t_2)-\pa_{y_j}\pa_{y_i}u(\bx^1,\by^1,t_1)|\Bigg\}\\
\leq
C\|g\|_{\WF,0,\gamma}\left[\rho_s(\bx^1,\bx^2)+\rho_e(\by^1,\by^2)+\sqrt{|t_2-t_1|}
\right]^{\gamma}.
 \end{multline}

To prove the H\"older continuity of $\pa_tu(\bx,\by,t)$ we simply use the
equation
\begin{equation}
  \pa_tu=\sum\limits_{l=1}^nL_{b_j,x_j}u+\sum_{j=1}^m\pa_{y_j}^2u+g,
\end{equation}
and the H\"older continuity of the expression appearing on the right hand side
of this relation.  Arguing as before we can use Proposition~\ref{prop4.1new} to
allow components of $\bb$ to tend to zero, and thereby extend these estimates
to the case that $\bzero\leq\bb.$ Using the estimates for scaled second
derivatives~\eqref{eqn476new0.0},~\eqref{eqn340.01}, and~\eqref{eqn353.00},
along with Proposition~\ref{prop3.4newdcyinfty} we argue as before to apply
Lemma~\ref{lem4.9new0.0} and show that
$u\in\cC^{0,2+\gamma}_{\WF}(\bbR_+^n\times\bbR^m\times[0,T]).$ This completes
the proof of the proposition in the $k=0$ case.  

For the $k>0$ cases we assume that $g$ is supported in
$B_R^+(\bzero)\times\bbR^m\times [0,T].$ Using
Propositions~\ref{lem3.4new0.0} and~\ref{wtedest_pr}, as before, we
easily obtain the desired conclusions for all $k\in\bbN,$ and thereby
complete the proof of the Proposition~\ref{prop6.2}.
\end{proof}

\section{Off-diagonal and Long-time Behavior}
We next consider a general result describing the off-diagonal behavior
of the solution kernel for~\eqref{genmodinhom}. This result is
important in the perturbation theory that follows in the next chapter.

\begin{proposition}\label{offdiagnm} Let $\varphi,\psi\in\cC_c^{\infty}(\bbR^n_+\times\bbR^m),$
  and assume that
  \begin{equation}
    \dist_{\WF}(\supp\varphi,\supp\psi)=\eta>0.
  \end{equation}
Let $0<B$ and $0\leq b_j\leq B.$
For any $k\in\bbN_0,$  the map $\cK_{\varphi,\psi,t}^{\bb}:g\mapsto
\psi(w)K^{\bb}_t[\varphi(z)g(z,\cdot)]$ defines a bounded operator
\begin{equation}
  \cK_{\varphi,\psi,t}^{\bb}:\dcC^0(\bbR^n_+\times\bbR^m\times[0,T])\to \dcC^k(\bbR^n_+\times\bbR^m\times[0,T]).
\end{equation}
There are positive constants $c_{\eta}, C,$ where $C$ depends on $k, \eta$ and
$B$ so that the operator norm of this map is bounded, as $T\to 0^+,$ by
$Ce^{-\frac{c_{\eta}}{T}}.$
\end{proposition}
This proposition is a consequence of estimates on the 1-dimensional
kernels. For the degenerate models we have:
\begin{lemma}\label{lem12.2.1} Let $\eta>0$ and for $x\in\bbR_+$ define the set
  \begin{equation}
    J_{x,\eta}=\{y\in\bbR_+:|\sqrt{x}-\sqrt{y}|\geq\eta\}.
  \end{equation}
  For $0< b\leq B,$ $0<\phi<\frac{\pi}{2},$ and $j\in\bbN_0$ there is a constant
  $C_{\eta, j,B,\phi}$ so that if $t =|t|e^{i\theta},$ with $|\theta|\leq
  \frac{\pi}{2}-\phi$, then
\begin{equation}
  \int\limits_{J_{x,\eta}}|\pa_x^jk^b_t(x,y)|dy\leq
  C_{\eta,j,B,\phi}\frac{e^{-\cos\theta\frac{\eta^2}{2|t|}}}{|t|^{j}}.
\end{equation} 
\end{lemma}
For the Euclidean models we have
\begin{lemma}\label{lem12.2.2}
Let  $\eta>0$ and for $x\in\bbR$ define the set
  \begin{equation}
    J_{x,\eta}=\{y\in\bbR:|x-y|\geq\eta\}.
  \end{equation}
For $j\in\bbN_0,$ $0<\phi<\frac{\pi}{2},$  there is a constant
  $C_{\eta, j,\phi}$ so that if $t =|t|e^{i\theta},$ with $|\theta|\leq
  \frac{\pi}{2}-\phi$, then 
\begin{equation}
  \int\limits_{J_{x,\eta}}|\pa_x^jk^e_t(x,y)|dy\leq
  C_{\eta,j,\phi}\frac{e^{-\cos\theta\frac{\eta^2}{8t}}}{|t|^{\frac j2}}.
\end{equation} 
\end{lemma}
\noindent
The Lemmas are proved in the Appendix.

\begin{proof}[Proof of Proposition~\ref{offdiagnm}] 
We need to consider integrals of the form
\begin{equation}
 I((\bx,\by),t)=
 \int\limits_{\bbR_+^n\times\bbR^m}|\pa_{\by}^{\bbeta}\pa_{\bx}^{\balpha}\prod_{i=1}^nk^{b_i}_t(x_i,w_i)
\prod_{l=1}^mk^{e}_{t}(y_l,z_l)\varphi(\bw,\bz)
f(\bw,\bz)|d\bw d\bz,
\end{equation}
for $(\bx,\by)\in\supp\psi,$ with $|\balpha|=q, |\bbeta|=p.$ For such
$(\bx,\bz)$ we let
\begin{equation}
\begin{split}
  U_{(\bx,\by),j}=\{(\bw,\bz)\in\supp\varphi:\:
  |\sqrt{x_j}-\sqrt{w_j}|\geq\frac{\eta}{n+m}\}
\text{ for }1\leq j\leq n,\\
 U_{(\bx,\by),j}=\{(\bw,\bz)\in\supp\varphi:\:
  |y_{j-n}-z_{j-n}|\geq\frac{\eta}{n+m}\}
\text{ for }n+1\leq j\leq n+m
\end{split}
\end{equation}
Since 
\begin{equation}
  \dist_{\WF}((\bx,\by),(\bw,\bz))=\sum_{j=1}^n|\sqrt{x_j}-\sqrt{w_j}|+
\sum_{l=1}^m|y_l-z_l|,
\end{equation}
and $\dist_{\WF}(\supp\psi,\supp\varphi)\geq \eta,$ it  follows that
\begin{equation}
  \supp\varphi\subset\bigcup\limits_{j=1}^{m+n} U_{(\bx,\by),j}.
\end{equation}
and that these sets are measurable. Thus we have the estimate
\begin{equation}
  I((\bx,\by),t)\leq\|\varphi f\|_{L^{\infty}}\sum_{j=1}^{m+n}
\int\limits_{U_{(\bx,\by),j}}
|\pa_{\by}^{\bbeta}\pa_{\bx}^{\balpha}\prod_{i=1}^nk^{b_i}_t(x_i,w_i)
\prod_{l=1}^mk^{e}_{t}(y_l,z_l)|d\bw d\bz.
\end{equation}
Let $I_j((\bx,\by),t)$ denote the integral in this sum over $U_{(\bx,\by),j}.$
We observe that
\begin{equation}
  \begin{split}
U_{(\bx,\by),j}&\subset \bbR_+^{j-1}\times
J_{x_j,\frac{\eta}{n+m}}\times\bbR_+^{n-j}\times \bbR^m\text{ for }1\leq j\leq
n\\
U_{(\bx,\by),j}&\subset 
\bbR_+^{n}\times \bbR^{j-n-1}\times J_{{y_{j-n}},\frac{\eta}{n+m}}\times\bbR^{m+n-j}\text{ for }n+1\leq j\leq
n+m.
  \end{split}
\end{equation}
Applying the 1-dimensional estimates we see that, if $1\leq j\leq n,$ then
\begin{equation}
\begin{split}
  &I_j((\bx,\by),t)\leq\\
 & \|\varphi
  f\|_{L^{\infty}}\int\limits_{0}^{\infty}\cdots\int\limits_{0}^{\infty}\int\limits_{\bbR^m}
\int \limits_{J_{x_j,\frac{\eta}{n+m}}}
|\pa_{\bx}^{\balpha}\prod_{i=1}^nk^{b_i}_t(x_i,w_i)\pa_{\by}^{\bbeta}
\prod_{l=1}^mk^{e}_{t}(y_l,z_l)|dw_j\widehat{dw_j}d\bz.
\end{split}
\end{equation}
Where, as usual, $\widehat{dw_j}$ is the volume form in $\bbR_+^n$ with $dw_j$
omitted. Lemmas~\ref{lem12.2.1}, \ref{lem12.2.2}, \ref{lrgt1db},
and~\ref{lrgt1de} show that
\begin{equation}
  I_j((\bx,\by),t)\leq C\|\varphi
  f\|_{L^{\infty}}\frac{e^{-\cost\frac{\eta^2}{2(n+m)^2t}}}{t^{q+\frac{p}{2}}}.
\end{equation}
A similar estimate applies for $n+1\leq j\leq n+m,$ which, upon summing shows that:
\begin{equation}
  I((\bx,\by),t)\leq C \|\varphi
  f\|_{L^{\infty}}\frac{e^{-\cost\frac{\eta^2}{8(n+m)^2t}}}{t^{q+\frac p2}}.
\end{equation}
The estimate on the right hand side is independent of $(\bx,\by)\in\supp\psi,$ so we
can integrate it to obtain
\begin{equation}
  \int\limits_{0}^{T}I((\bx,\by),t)dt\leq C\|\varphi f\|_{L^{\infty}}e^{-\cost\frac{\eta^2}{16(n+m)^2T}}.
\end{equation}
Coupling this with the Leibniz formula, the proposition follows easily from
these estimates.
\end{proof}

For each $t>0,$ we have defined the map $f\mapsto K^{\bb,m}_tf,$ where
\begin{equation}
  K^{\bb,m}_tf(\bx,\by)= \int\limits_{\bbR_+^n\times\bbR^m}\prod_{i=1}^nk^{b_i}_t(x_i,w_i)
\prod_{l=1}^mk^{e}_{t}(y_l,z_l)
f(\bw,\bz)d\bz d\bw.
\end{equation}
For any $t>0$ and $f\in L^{\infty}(\bbR_+^{n}\times\bbR^m),$ it is clear that
$K^{\bb,m}_tf\in \cC^{\infty}(\bbR_+^{n}\times\bbR^m).$ The 1-dimensional
estimates~\eqref{eqn12.149.1} and~\eqref{eqn12.152.1} imply the following
result:
\begin{proposition} For multi-indices $\balpha\in\bbN_0^n$ and
  $\bbeta\in\bbN_0^m,$ there are constants $C_{\balpha,\bbeta}$ so that
  \begin{equation}
    |\pa_{\bx}^{\balpha}\pa_{\by}^{\bbeta}K^{\bb,m}_t f(\bx,\by)|\leq
C_{\balpha,\bbeta}\frac{\|f\|_{L^{\infty}}}{t^{|\balpha|+\frac{|\bbeta|}{2}}}.
  \end{equation}
If we let $(\sqrt{\bx})=(\sqrt{x_1},\dots,\sqrt{x_n}),$ then we also have
\begin{equation}
    |(\sqrt{\bx})^{\balpha}\pa_{\bx}^{\balpha}\pa_{\by}^{\bbeta}K^{\bb,m}_t f(\bx,\by)|\leq
C_{\balpha,\bbeta}\frac{\|f\|_{L^{\infty}}}{t^{\frac{|\balpha|+|\bbeta|}{2}}}.
  \end{equation}
\end{proposition}
\noindent
This proposition follows easily from~\eqref{eqn12.149.1},~\eqref{eqn12.149.11}
and~\eqref{eqn12.152.1}. We leave the details to the interested reader.
\section{The Resolvent Operator}
We close this section by stating a proposition summarizing the
properties of the resolvent operator, $R(\mu),$ as an operator on the
H\"older spaces $\cC^{k,\gamma}_{\WF}(\bbR_+^n\times\bbR^m).$ As
contrasted with the case of the heat equation, we do not need to
assume that the data has compact support in the $x$-variables to prove
estimates when $k>0.$ As before we use Proposition~\ref{lem3.4new0.0}
to commute differential operators of the form
$\pa_{\bx}^{\balpha}\pa_{\by}^{\bbeta}$ past the heat kernel. Since we
are only proving spatial estimates we do not need to commute $\pa_t^j$
past the kernel, and hence do not encounter the needed for weighted
estimates on the data.

\begin{proposition}\label{prop8.0.4.nm}  The resolvent operator $R(\mu)$ is analytic
  in the complement of $(-\infty,0],$ and is given by the integral
  in~\eqref{eqn10.182nm} provided that $\Re(\mu e^{i\theta})>0.$ For $\alpha\in
  (0,\pi],$ there are constants $C_{\bb ,\alpha}$
 so that if 
 \begin{equation}\label{argest001nm}
   \alpha-\pi\leq \arg{\mu}\leq\pi-\alpha, \end{equation}
then for $f\in\dcC^0(\bbR_+^n\times\bbR^m)$ we have
  \begin{equation}\label{eqn12.168nm00}
    \|R(\mu) f\|_{L^{\infty}}\leq \frac{C_{\bb ,\alpha}}{|\mu|}\|f\|_{L^{\infty}};
  \end{equation}
with $C_{\bb ,\pi}=1.$ Moreover, for $0<\gamma<1,$ and $k\in\bbN_0,$ there are constants
  $C_{k,\bb ,\alpha,\gamma}$ so that if $f\in\cC^{k,\gamma}_{\WF}(\bbR_+^n\times\bbR^m),$ then
  \begin{equation}\label{eqn8.13.02}
      \|R(\mu) f\|_{\WF,k,\gamma}\leq \frac{C_{k,\bb ,\alpha,\gamma}}{|\mu|}\|f\|_{\WF,k,\gamma}.
  \end{equation}
We also have the estimates
\begin{equation}\label{eqn11.169.04}
\|\nabla_{\by}R(\mu)f\|_{\WF,k,\gamma}\leq
C_{k,\alpha}\left[\frac{1}{|\mu|^{\frac{\gamma}{2}}}+
\frac{1}{|\mu|^{\frac{\gamma+1}{2}}}\right]\|f\|_{\WF,k,\gamma},
\end{equation}
\begin{equation}\label{eqn11.170.04}
\begin{split}
\|\sqrt{\bx}\cdot\nabla_{\bx}R(\mu) f\|_{\WF,0,\gamma}&\leq
C_{k,\alpha}\left[\frac{1}{|\mu|^{\frac{\gamma}{2}}}+
\frac{1}{|\mu|^{\frac{\gamma+1}{2}}}\right]\|f\|_{\WF,0,\gamma},\\
\|\psi(\bx)\bx\cdot\nabla_{\bx}R(\mu) f\|_{\WF,k,\gamma}&\leq 
\sqrt{X}C_{k,\alpha}\left[\frac{1}{|\mu|^{\frac{\gamma}{2}}}+
\frac{1}{|\mu|}\right]\|f\|_{\WF,k,\gamma}.
\end{split}
\end{equation}
Here $\psi(\bx)$ is a smooth function with $|\nabla\psi(\bx)|\leq 1$ and
$\supp\psi\subset [0,X]^n.$

If for a $k\in\bbN_0,$ and $0<\gamma<1,$
$f\in\cC^{k,\gamma}_{\WF}(\bbR^n_+\times\bbR^m),$ then
$R(\mu)f\in\cC^{k,2+\gamma}_{\WF}(\bbR^n_+\times\bbR^m),$ and, we have
\begin{equation}
  (\mu-L_{\bb,m})R(\mu) f=f.
\end{equation}
If $f\in \cC^{0,2+\gamma}_{\WF}(\bbR^n_+\times\bbR^m),$ then
\begin{equation}
  R(\mu)(\mu-L_{\bb,m}) f=f.
\end{equation}
There are constants $C_{\bb ,k,\alpha}$ so that, for $\mu$
satisfying~\eqref{argest001n}, we have
\begin{equation}\label{eqn11.173.006}
  \|R(\mu)f\|_{\WF,k,2+\gamma}\leq 
C_{\bb ,k,\alpha}\left[1+\frac{1}{|\mu|}\right]\|f\|_{\WF,k,\gamma}.
\end{equation}
For any $B>0,$ these constants are uniformly bounded for $0\leq \bb<B\bone.$
\end{proposition}
\begin{proof} This proof of this proposition, with $k=0,$ is almost
  immediate from the proof of Proposition~\ref{prop6.2}. If
  $\kappa^{\bb,m}_t(\bx,\bz;\by,\bw)$ denotes the heat kernel for
  $L_{\bb,m},$ then this proof estimated the integrals
  \begin{equation}
    \int\limits_{0}^{t}\int\limits_{\bbR_+^n}\int\limits_{\bbR^m}
\kappa^{\bb,m}_s(\bx,\bz;\by,\bw)g(\bz;\bw,t-s)d\bz d\bw ds.
  \end{equation}
The only estimate on $g$ that is used in these arguments is
  \begin{equation}
    |g(\bx^1,\by^1,t)-g(\bx^2,\by^2,t)|\leq
\|g\|_{\WF,0,\gamma}\rho((\bx^1,\by^1),(\bx^2,\by^2))^{\gamma}.
  \end{equation}
To prove the present theorem we consider  integrals  of the form
 \begin{equation}
    \int\limits_{0}^{\infty}\int\limits_{\bbR_+^n}\int\limits_{\bbR^m}
\kappa_{s\eit}^{\bb,m}(\bx,\bz;\by,\bw)f(\bz;\bw)e^{-se^{i\phi}\mu}d\bz d\bw e^{i\phi} ds,
  \end{equation}
where $\mu=|\mu|e^{-i\psi},$ with $|\psi|\leq \pi-\alpha,$ for an
$0<\alpha\leq\pi.$ We can choose $|\phi|<\frac{\pi-\alpha}{2}$ so that 
\begin{equation}
  -\frac{\pi}{2}+\frac{\alpha}{2}<\phi-\psi< \frac{\pi}{2}-\frac{\alpha}{2},
\end{equation}
leading to an absolutely convergent integral. All the arguments used in the
proof of Proposition~\ref{prop6.2} apply with the modification that the time
integrals now extend from $0$ to $\infty$ and include a factor of $e^{-\cost
  s},$ where $\theta=\phi-\alpha.$ In light of this we only give a detailed
outline for the proof of the current proposition, with references to
formul{\ae} in the previous argument.

 As in the proofs of the previous results it suffices to
  establish these results for the $k=0$ case, and arbitrary $\bzero<\bb.$ The
  case where some components of $\bb$ vanish and arbitrary $k\in\bbN$ are then
  obtained using Proposition~\ref{prop4.1new} and Lemma~\ref{lem3.1new.0}
  respectively.  We fix a $0<\alpha\leq\pi.$

We begin by showing that if $f\in\cC^{0,\gamma}_{\WF}(\bbR_+^n\times\bbR^m),$ then
$R(\mu)f\in\cC^{2}_{\WF}(\bbR_+^n\times\bbR^m).$  
First we see that Lemma~\ref{lem9.1.3.00}
implies that
\begin{equation}
  \|R(\mu) f\|_{L^{\infty}}\leq C_{\bb,\alpha}\frac{\|f\|_{L^{\infty}}}{|\mu|}.
\end{equation}
To prove the estimate on $\|R(\mu) f\|_{\WF,k,\gamma},$ for $k=0,$ we observe
that the argument in the proof of Proposition~\ref{prop6.1} showing that
$v(\bx,\by,t),$ with initial data
$f\in\cC^{0,\gamma}_{\WF}(\bbR_+^n\times\bbR^m),$ satisfies H\"older
estimates applies equally well to complex times $t\in S_{\phi}.$ Thus we know
that there is a constant $C_{\bb ,\phi}$ so that, for $t\in S_{\phi},$ we have
\begin{equation}
  |v(\bx^1,\by^1,t)-v(\bx^2,\by^2,t)|\leq C_{\bb ,\phi}\|f\|_{\WF,0,\gamma}
[\rho((\bx^1,\by^1),(\bx^2,\by^2))]^{\gamma}.
\end{equation}
Integrating the estimate that this implies for $\bbr{v(\cdot,\cdot,t)}_{\WF,0,\gamma},$
shows that there is a constant $C_{\bb ,\alpha,\gamma}$ for which
\begin{equation}
  \|R(\mu)f\|_{\WF,0,\gamma}\leq C_{\bb ,\alpha,\gamma}\frac{\|f\|_{\WF,0,\gamma}}{|\mu|}.
\end{equation}
We obtain this estimate for $k>0$ by using the formul{\ae} in
Proposition~\ref{lem3.3new0.0} to commute the derivatives
$\pa_{\bx}^{\balpha}\pa_{\by}^{\bBeta}$ through the heat kernel and
onto the data, $f.$ As noted above, in this context there is no need
for time derivatives, hence we do not need to assume that $f$ has
compact support in the $x$-variables.

Next observe that we  can use the single variable estimates in
formul{\ae} analogous to those in~\eqref{eqn10.49.1} to show that
\begin{equation}
  |\pa_{x_j} R(\mu)f(\bx,\by)|\leq C_{\bb,\alpha}\frac{\|f\|_{\WF,0,\gamma}}
{|\mu|^{\frac{\gamma}{2}}}
\text{ and }
|\sqrt{x_j}\pa_{x_j} R(\mu)f(\bx,\by)|\leq C_{\bb,\alpha}\frac{\|f\|_{\WF,0,\gamma}}
{|\mu|^{\frac{\gamma+1}{2}}},
\end{equation}
and
\begin{equation}
  |\pa_{y_l} R(\mu)f(\bx,\by)|\leq C_{\bb,\alpha}\frac{\|f\|_{\WF,0,\gamma}}
{|\mu|^{\frac{\gamma+1}{2}}}.
\end{equation}

The simple 1-dimensional estimates also suffice to prove that:
\begin{equation}\label{eqn12.178.00}
  \begin{split}
     |x_j\pa^2_{x_j} R(\mu)f(\bx,\by)|&\leq C_{\bb,\alpha}\frac{\|f\|_{\WF,0,\gamma}}
{|\mu|^{\frac{\gamma}{2}}}\\
|x_j\pa^2_{x_j} R(\mu)f(\bx,\by)|&\leq C_{\bb,\alpha}\|f\|_{\WF,0,\gamma}x_j^{\frac{\gamma}{2}},
  \end{split}
\end{equation}
and
\begin{equation}\label{eqn12.179.00}
  |\pa_{y_l}^2R(\mu)f(\bx,\by)|\leq C_{\bb,\alpha}\frac{\|f\|_{\WF,0,\gamma}}
{|\mu|^{\frac{\gamma}{2}}}.
\end{equation}
Using a formula like that in~\eqref{eqn10.156.7} we can show that
\begin{equation}\label{eqn11.185.05}
\begin{split}
   |\sqrt{x_ix_j}\pa_{x_i}\pa_{x_j}
R(\mu)f(\bx,\by)|&\leq C_{\bb,\alpha}\frac{\|f\|_{\WF,0,\gamma}}
{|\mu|^{\frac{\gamma}{2}}}\\
 |\pa_{y_l}\pa_{y_m}
R(\mu)f(\bx,\by)|&\leq C_{\bb,\alpha}\frac{\|f\|_{\WF,0,\gamma}}
{|\mu|^{\frac{\gamma}{2}}}.
\end{split}
\end{equation}
Finally, using an expression like that in~\eqref{eqn12.64.00}, we can show that
\begin{equation}\label{eqn11.186.05}
  |\sqrt{x_i}\pa_{x_i}\pa_{y_l}
R(\mu)f(\bx,\by)|\leq C_{\bb,\alpha}\frac{\|f\|_{\WF,0,\gamma}}
{|\mu|^{\frac{\gamma}{2}}}.
\end{equation}
This establishes that $R(\mu)f\in\cC^{2}_{\WF}(\bbR_+^n\times\bbR^m).$

We can now use a standard integration by parts argument,
see~\eqref{resolinprts0}--\eqref{resolinprts1}, to show that
\begin{equation}
  (\mu-L_{\bb,m})R(\mu)f=f.
\end{equation}
As in the 1-d case, we demonstrate below that, if $f\in
\cC^{0,\gamma}_{\WF}(\bbR_+^n\times\bbR^m),$ then
$R(\mu)f\in\cC^{0,2+\gamma}_{\WF}(\bbR_+^n\times\bbR^m),$ and therefore, by the
open mapping theorem, to show that $R(\mu)$ is also a left inverse it suffices
to show that the null-space $(\mu-L_{\bb,m})$ is trivial for $\mu\in S_0.$ If
there were a non-trivial eigenfunction $f,$ for such a $\mu,$ then $e^{\mu t}f$
would solve the Cauchy problem, and grow exponentially with $t.$ As this
contradicts~\eqref{eqn12.12.1}, it follows that this null-space is empty. We
can therefore conclude that if
$f\in\cC^{0,2+\gamma}_{\WF}(\bbR_+^n\times\bbR^m),$ and $\mu\in S_0,$ then
\begin{equation}
  R(\mu)(\mu-L_{\bb,m})f=f.
\end{equation}
As the left hand side is analytic in $\mu\in\bbC\setminus (-\infty,0],$ it
follows that this relation also holds in the complement of the negative real
axis.

It remains to establish the H\"older continuity of the first and second
derivatives of $R(\mu)f.$ 
Equation~\eqref{eqn12.179.00} implies that if $\by$ and $\by'$ differ only in the
$l$th coordinate then
\begin{equation}
 |\pa_{y_l}R(\mu)f(\bx,\by)-
\pa_{y_l}R(\mu)f(\bx,\by')|\leq C_{\bb,\alpha}\frac{\|f\|_{\WF,0,\gamma}}
{|\mu|^{\frac{\gamma}{2}}}|y_l-y_l'|.
\end{equation}
As observed earlier, if $\by-\by'$ is supported in the $m$th place and $m\neq
l,$ then Lemmas~\ref{lem21newe}
and~\ref{lem2newe} show that we have the bound:
\begin{equation}
\begin{split}
|\pa_{y_l}R(\mu)& f(\bx,\by)- \pa_{y_l}R(\mu)f(\bx,\by')| \\
& \leq C_{\bb,\alpha}\|f\|_{\WF,0,\gamma} \int\limits_0^{\infty}s^{\frac{\gamma-1}{2}}\frac{\left(\frac{|y_m-y'_m|}{\sqrt{s}}\right)}
{1+\left(\frac{|y_m-y'_m|}{\sqrt{s}}\right)}e^{-\cost|\mu| s}ds\\
&\leq C_{\bb,\alpha}\|f\|_{\WF,0,\gamma}\int\limits_0^{\infty}s^{\frac{\gamma}{2}-1}|y_m-y'_m|
e^{-\cost|\mu| s}ds,
\end{split}
\end{equation}
from which it follows easily that, for $\rho_e(\by,\by')<1$ we have:
\begin{equation}
  |\pa_{y_l}R(\mu)f(\bx,\by)-
\pa_{y_l}R(\mu)f(\bx,\by')|\leq
C_{\bb,\alpha}\frac{\|f\|_{\WF,0,\gamma}[\rho_e(\by,\by')]^{\gamma}}
{|\mu|^{\frac{\gamma}{2}}}
\end{equation}

To complete the estimate of the first $\by$-derivatives, we need to bound the
difference $|\pa_{y_l}R(\mu)f(\bx,\by)-\pa_{y_l}R(\mu)f(\bx',\by)|.$ We can
assume that $\bx-\bx'$ is supported in the first slot. The derivation
of~\eqref{eqn12.65.00} implies that
\begin{equation}
  |\pa_{x_1}\pa_{y_l}R(\mu)f(\bx,\by)|\leq
C_{\bb ,\alpha}\|f\|_{\WF,0,\gamma}\int\limits_{0}^{\infty}
\frac{s^{\frac{\gamma}{2}-1}e^{-\cost|\mu| s}ds}{\sqrt{x_1}+\sqrt{s}}.
\end{equation}
Splitting the integral into the part from $0$ to $x_1$ and the rest we see that
\begin{equation}
  |\pa_{x_1}\pa_{y_l}R(\mu)f(\bx,\by)|\leq
C_{\bb ,\alpha}\frac{\|f\|_{\WF,0,\gamma}}{\sqrt{x_1}|\mu|^{\frac{\gamma}{2}}},
\end{equation}
which upon integration implies that
\begin{equation}
  |\pa_{y_l}R(\mu)f(\bx,\by)-
\pa_{y_l}R(\mu)f(\bx',\by)|\leq
C_{\bb,\alpha}\frac{\|f\|_{\WF,0,\gamma}\rho_s(\bx,\bx')}{
|\mu|^{\frac{\gamma}{2}}}.
\end{equation}
As we only require an estimate when $\rho_s(\bx,\bx')<1,$ this shows that
\begin{equation}
  |\pa_{y_l}R(\mu)f(\bx,\by)-
\pa_{y_l}R(\mu)f(\bx',\by)|\leq
C_{\bb,\alpha}\frac{\|f\|_{\WF,0,\gamma}\rho_s(\bx,\bx')^{\gamma}}
{|\mu|^{\frac{\gamma}{2}}}.
\end{equation}
By commuting \emph{spatial} derivatives past the kernel using
Proposition~\ref{lem3.3new0.0}, we obtain~\eqref{eqn11.169.04} for all
$k\in\bbN_0.$

To obtain an estimate for $\|\sqrt{x_j}\pa_{x_j}R(\mu) f\|_{\WF,0,\gamma},$ we integrate  the
estimate of $|\sqrt{x_j}\pa_{x_j}v(\cdot,t)|$ afforded by Lemma~\ref{lem2new}
to conclude that
\begin{equation}
  |\sqrt{x_j}\pa_{x_j}R(\mu)f(\bx,\by)|\leq
  C_{\alpha}\frac{x_j^{\frac{\gamma}{2}}}
{|\mu|^{\frac{\gamma}{2}}}.
\end{equation}
As usual this implies that for $0<c<1,$ there is a $C$ so that if $x'_j<cx_j$
and $\bx-\bx'$ is supported in the $j$th place, then
\begin{equation}
  |\sqrt{x_j}\pa_{x_j}R(\mu)f(\bx,\by)-\sqrt{x'_j}\pa_{x_j}R(\mu)f(\bx',\by)|\leq
C\frac{|\sqrt{x_j}-\sqrt{x_j'}|^{\gamma}}
{|\mu|^{\frac{\gamma}{2}}}.
\end{equation}
To obtain a  similar estimate when $cx_j<x'_j<x_j,$ we use
Lemma~\ref{eqn2400.0p}. Integrating the estimates in~\eqref{eqn11.185.05}
and~\eqref{eqn11.186.05} we easily complete the proof that, for
$\rho((\bx,\by),(\bx',\by'))<1,$ we have that
\begin{equation}
  |\sqrt{x_j}\pa_{x_j}R(\mu)f(\bx,\by)-\sqrt{x'_j}\pa_{x_j}R(\mu)f(\bx',\by')|\leq
C\frac{|\rho((\bx,\by),(\bx',\by'))|^{\gamma}}
{|\mu|^{\frac{\gamma}{2}}},
\end{equation}
finishing the proof of the first estimate in~\eqref{eqn11.170.04}.  The
second estimate is proved by using Proposition~\ref{lem3.3new0.0} to commute
derivatives past the heat kernel, and applying the first estimate
in~\eqref{eqn11.170.04} and the Leibniz formula,~\eqref{leibfrmnm}, to terms of
the form:
$$\psi(\bx)\sqrt{x_j}\sqrt{x_j}\pa_{x_j}R(\mu)\pa_{\bx}^{\balpha}\pa_{\by}^{\bBeta}f.$$
This explains the appearance of the $\sqrt{X}.$ All other terms are of lower
order and easily estimated. This completes the proof of the estimates
in~\eqref{eqn11.170.04}.

We still need to establish the H\"older continuity of the unscaled first
derivatives in the $\bx$-variables.  By integrating the second estimate
in~\eqref{eqn12.178.00} we can show that if $\bx$ and $\bx'$ differ only in the
$j$th coordinate, then
\begin{equation}
  |\pa_{x_j}R(\mu)f(\bx,\by)-
\pa_{x_j}R(\mu)f(\bx',\by)|\leq C_{\bb ,\alpha}\|f\|_{\WF,0,\gamma}|\sqrt{x_j}-\sqrt{x'_j}|^{\gamma}.
\end{equation}
To do the off-diagonal cases, we assume that $\bx-\bx'$ is supported in the
$m$th slot, with $j\neq m$. If $x_m'=0,$ then by arguing as
in~\eqref{frstdroffdiaghldestn0} we see that
\begin{multline}
|\pa_{x_j}R(\mu)f(\bx,\by)-
\pa_{x_j}R(\mu)f(\bx',\by)|
 C_{\bb ,\alpha}\|f\|_{\WF,0,\gamma} \\
\times \int\limits_{0}^{\infty}s^{\frac{\gamma}{2}-1}\frac{x_m/s}{1+x_m/s}e^{-\cost|\mu| s}ds,
\end{multline}
which is easily seen to be bounded by
$C_{\bb ,\alpha}\|f\|_{\WF,0,\gamma}x_m^{\frac{\gamma}{2}}.$ Applying~\eqref{lrgratioest} we
see that, if $0<c<1,$ then there is a constant $C_{\bb ,\alpha}$ so that for $cx_m>x'_m,$ we have
\begin{equation}\label{frstdroffdiaghldestnmR}
  |\pa_{x_j}R(\mu)f(\bx,\by)-\pa_{x_j}R(\mu)f(\bx',\by)|\leq 
C_{\bb ,\alpha}\|f\|_{\WF,0,\gamma}|\sqrt{x_m}-\sqrt{x'_m}|^{\gamma}.
\end{equation}
We are therefore reduced to considering $cx_m<x'_m<x_m,$ for a $c<1.$  If we
use~\eqref{lem21newest2} it follows that
\begin{multline}\label{eqn12.190.01}
  |\pa_{x_j}R(\mu)f(\bx,\by)-\pa_{x_j}R(\mu) f(\bx',\by)|\leq 
C_{\bb ,\alpha}\|f\|_{\WF,0,\gamma}\times\\
\int\limits_0^{\infty}s^{\frac{\gamma}{2}-1}\left(\frac{\frac{\sqrt{x_m}-\sqrt{x'_m}}{\sqrt{s}}}
{1+\frac{\sqrt{x_m}-\sqrt{x'_m}}{\sqrt{s}}}\right)e^{-\cost|\mu| s}ds.
\end{multline}
We split this into an integral from $0$ to $(\sqrt{x_m}-\sqrt{x'_m})^2$ and the
rest, to obtain:
\begin{multline}
  |\pa_{x_j}R(\mu)f(\bx,\by)-\pa_{x_j}R(\mu)f(\bx',\by)|\leq 
C_{\bb ,\alpha}\|f\|_{\WF,0,\gamma}\times\\
\left[\int\limits_0^{(\sqrt{x_m}-\sqrt{x_m'})^2}s^{\frac{\gamma}{2}-1}+
\int\limits_{(\sqrt{x_m}-\sqrt{x_m'})^2}^{\infty}s^{\frac{\gamma}{2}-1}
\left(\frac{\sqrt{x_m}-\sqrt{x_m'}}{\sqrt{s}}
\right)e^{-\cost|\mu| s}ds\right].
\end{multline}
Performing these integrals shows that this term is also estimated by
\begin{equation}\label{eqn12.189.00}
  |\pa_{x_j}R(\mu)f(\bx,\by)-\pa_{x_j}R(\mu)f(\bx',\by)|\leq
C_{\bb ,\alpha}\|f\|_{\WF,0,\gamma} |\sqrt{x_m}-\sqrt{x_m'}|^{\gamma}.
\end{equation}

We now estimate $|\pa_{x_j}R(\mu)f(\bx,\by)-\pa_{x_j}R(\mu)f(\bx,\by')|,$ with
$\by-\by'$ supported at a single index, which we label 1. Arguing exactly as in
the derivation of~\eqref{eqn12.61.00}, we see that
\begin{multline}
|\pa_{x_j}R(\mu)f(\bx,\by)-\pa_{x_j}R(\mu)f(\bx,\by')| \\ \leq
C_{\bb ,\alpha}\|f\|_{\WF,0,\gamma}\int\limits_0^{\infty}s^{\frac{\gamma}{2}-1}
\left(\frac{\frac{|y^2_1-y^1_1|}{\sqrt{s}}}
{1+\frac{|y^2_1-y^1_1|}{\sqrt{s}}}\right)e^{-\cost|\mu| s}ds.
\end{multline}
The same argument used to prove~\eqref{eqn12.189.00} can be
employed to show that
\begin{equation}
  |\pa_{x_j}R(\mu)f(\bx,\by)-\pa_{x_j}R(\mu)f(\bx,\by')|\leq
C_{\bb ,\alpha}\|f\|_{\WF,0,\gamma}\rho_e(\by,\by')^{\gamma}.
\end{equation}

All that remains is to prove the H\"older continuity of the second derivatives.
Using the second estimate in~\eqref{eqn12.178.00} and the argument used in the
proof of~\eqref{eqn9.149.00} we can show that
\begin{multline}
   |x_j\pa_{x_j}^2R(\mu)f(\bx_j',x_j,\bx_j'',\by)-y_j\pa_{x_j}^2R(\mu)f(\bx_j',y_j,\bx_j'',\by)|
\leq\\ C_{b,\alpha}\|f\|_{\WF,0,\gamma}|\sqrt{x_j}-\sqrt{y_j}|^{\gamma}.
\end{multline}
Arguing as in the derivation of~\eqref{2ndhldrnondiag2}, we see that
\begin{multline}\label{2ndhldrnondiag2R}
  |x_j\pa_{x_j}^2R(\mu)f(\bx_m',x_m,\bx_m'',\by)-x_j\pa_{x_j}^2R(\mu)(\bx_m',0,\bx_m'',\by)|\leq\\
C_{b,\alpha}\|f\|_{\WF,0,\gamma}\int\limits_0^{\infty}s^{\frac{\gamma}{2}-1}e^{-\cost|\mu|
  s}\frac{x_m/s}{1+x_m/s}ds.
\end{multline}
As before the integral is bounded by a constant times $x_m^{\frac{\gamma}{2}},$
which suffices to prove the H\"older estimate for $x'_m<cx_m,$ for a fixed
$0<c<1.$ If we fix such a $c,$ then for $cx_m<x'_m<x_m,$ the argument leading
to~\eqref{2ndhldrnondiag3} gives
\begin{multline}\label{2ndhldrnondiag3R}
  |x_j\pa_{x_j}^2R(\mu)f(\bx_m',x_m,\bx_m'',\by)-x_j\pa_{x_j}^2R(\mu)f(\bx_m',x'_m,\bx_m'',\by)|\leq\\
C_{b,\alpha}\|f\|_{\WF,0,\gamma}\int\limits_0^{\infty}
s^{\frac{\gamma}{2}-1}\frac{\frac{\sqrt{x_m}-\sqrt{x'_m}}{\sqrt{s}}}
{1+\frac{\sqrt{x_m}-\sqrt{x'_m}}{\sqrt{s}}}e^{-\cost|\mu| s}ds.
\end{multline}
The argument used to estimate the integral in~\eqref{eqn12.190.01} applies to
show that
\begin{multline}\label{2ndhldrnondiag4R}
  |x_j\pa_{x_j}^2R(\mu)f(\bx_m',x_m,\bx_m'',\by)-x_j\pa_{x_j}^2R(\mu)f(\bx_m',x'_m,\bx_m'',\by)|\leq\\
C_{b,\alpha}\|f\|_{\WF,0,\gamma}|\sqrt{x_m}-\sqrt{x'_m}|^{\gamma}.
\end{multline}
To estimate $|x_j\pa_{x_j}^2R(\mu)f(\bx,\by)-x_j\pa_{x_j}^2R(\mu)f(\bx,\by')|,$
we argue as in the derivation of~\eqref{eqn362.00} to see that, if $\by-\by'$
is supported in the first argument, then
\begin{equation}\label{eqn362.00R}
  \begin{split}
&|x_j\pa_{x_j}^2R(\mu)f(\bx,\by)-x_j\pa_{x_j}^2R(\mu)f(\bx,\by')|\\&\leq
    C_{b,\alpha}\|f\|_{\WF,0,\gamma}
\int\limits_0^{\infty}\frac{\sqrt{x_1}s^{\frac{\gamma}{2}-1}}{\sqrt{x_1}+\sqrt{s}}
\left(\frac{\frac{|y^2_1-y^1_1|}{\sqrt{s}}}
{1+\frac{|y^2_1-y^1_1|}{\sqrt{s}}}\right)e^{-\cost|\mu| s}ds.
  \end{split}
\end{equation}
As before, this integral is estimated by a constant times
$|y^2_1-y^1_1|^{\gamma},$ completing the proof that
\begin{multline}
  |x_j\pa_{x_j}^2R(\mu)f(\bx,\by)-x_j\pa_{x_j}^2R(\mu)f(\bx',\by')|\\\leq
    C_{b,\alpha}\|f\|_{\WF,0,\gamma}[\rho((\bx,\by),(\bx',\by'))]^{\gamma}.
\end{multline}

The argument between~\eqref{eqn10.153.01} and~\eqref{eqn255.00} applies with
small modifications (largely replacing the upper limit in the $s$-integrations
with $\infty$ and the measure $ds$ with $e^{-\cost|\mu| s}ds$), to show that,
with $j\neq k,$ we have:
  \begin{multline}
  |\sqrt{x_jx_k}\pa_{x_j}\pa_{x_k}R(\mu)f(\bx,\by)-\sqrt{x'_jx'_k}\pa_{x_j}\pa_{x_k}R(\mu)f(\bx',\by)|\\
\leq
    C_{b,\alpha}\|f\|_{\WF,0,\gamma}[\rho_s(\bx,\bx')]^{\gamma}.
\end{multline}
Similarly, the derivation of the estimate
in~\eqref{eqn12.97.01}--\eqref{eqn12.98.01} shows that, if $\by-\by'$ is
supported in the first slot, then
\begin{multline}\label{eqn12.98.01R}
|\sqrt{x_jx_k}\pa_{x_j}\pa_{x_k}R(\mu)f(\bx,\by)-\sqrt{x_jx_k}\pa_{x_j}\pa_{x_k}R(\mu)f(\bx,\by')|\leq\\
C_{b,\alpha}\|f\|_{\WF,0,\gamma}\int\limits_0^{\infty}s^{\frac{\gamma}{2}-1}
\left(\frac{\frac{|y^2_1-y^1_1|}{\sqrt{s}}}
{1+\frac{|y^2_1-y^1_1|}{\sqrt{s}}}\right)e^{-\cost|\mu| s}ds,
\end{multline}
which completes the proof that
\begin{multline}
  |\sqrt{x_jx_k}\pa_{x_j}\pa_{x_k}R(\mu)f(\bx,\by)-\sqrt{x'_jx'_k}
\pa_{x_j}\pa_{x_k}R(\mu)f(\bx',\by')|\\\leq
    C_{b,\alpha}\|f\|_{\WF,0,\gamma}[\rho((\bx,\by),(\bx',\by'))]^{\gamma}.
\end{multline}

As before we can use the analogous estimates for the Euclidean kernel to show
that
\begin{multline}
  |\pa_{y_j}\pa_{y_k}R(\mu)f(\bx,\by)-
\pa_{y_j}\pa_{y_k}R(\mu)f(\bx,\by')|\\\leq
    C_{b,\alpha}\|f\|_{\WF,0,\gamma}[\rho_e(\by,\by')]^{\gamma}.
\end{multline}
The argument between~\eqref{eqn12.104.01} and~\eqref{eqn12.110.01} applies with
the usual small changes to show that
\begin{multline}
  |\pa_{y_j}\pa_{y_k}R(\mu)f(\bx,\by)-
\pa_{y_j}\pa_{y_k}R(\mu)f(\bx',\by)|\\\leq
    C_{b,\alpha}\|f\|_{\WF,0,\gamma}[\rho_s(\bx,\bx')]^{\gamma}.
\end{multline}
To estimate the mixed derivatives
$\sqrt{x_j}\pa_{x_j}\pa_{y_k}R(\mu)f(\bx,\by),$ we slightly modify the argument
between~\eqref{eqn12.111.01} and~\eqref{eqn12.122.01}. In each case we are
reduced to estimating an integral of one of the following two forms:
\begin{equation}
  \int\limits_{0}^{\infty}s^{\frac{\gamma}{2}-1}\frac{D e^{-\cost|\mu|
      s}ds}{\sqrt{s}+D}\text{ or }
\int\limits_{0}^{\infty}s^{\frac{\gamma}{2}-1}\frac{D e^{-\cost |\mu|
      s}ds}{s+D}.
\end{equation}
The first integral is estimated by $C_{\gamma}D^{\gamma}$ and the second by
$C_{\gamma}D^{\frac{\gamma}{2}}.$ Using these estimates we complete the proof that
\begin{multline}
 |\sqrt{x_j}\pa_{x_j}\pa_{y_k}R(\mu)f(\bx,\by)-
\sqrt{x'_j}\pa_{x_j}\pa_{y_k}R(\mu)f(\bx',\by')|\leq\\
 C_{b,\alpha}\|f\|_{\WF,0,\gamma}[\rho((\bx,\by),(\bx',\by'))]^{\gamma}
\end{multline}

This completes the $k=0$ case for $\bb>\bzero.$ The constants are again
uniformly bounded for $\bzero<\bb\leq B\bone,$ and so we can apply
Proposition~\ref{prop4.1new} to extend this results to $\bzero\leq\bb.$ 

Finally, to treat $k>0,$ we use Proposition~\ref{lem3.3new0.0} to
commute the spatial derivatives past the heat kernel and follow the
argument above to establish this theorem for arbitrary $k\in\bbN.$ The
only terms that require additional consideration are contributions to
the left hand side of~\eqref{eqn11.170.04} from terms of the form:
\begin{equation}
  \|\pa_{\bx}^{\balpha}\pa_{\by}^{\bBeta}x_j\pa_{x_j}R(\mu) f\|_{\WF,0,\gamma},
\text{ for }f\in\cC^{k,\gamma}_{\WF},
\end{equation}
where $\alpha_j\neq 0.$  In all other cases
$$[\pa_{\bx}^{\balpha}\pa_{\by}^{\bBeta},x_j\pa_{x_j}]=0$$
and the estimate follows easily using Proposition~\ref{lem3.3new0.0}. If $\alpha_j\neq 0,$ then
\begin{equation}
  \pa_{\bx}^{\balpha}\pa_{\by}^{\bBeta}x_j\pa_{x_j}=
x_j\pa_{x_j}\pa_{\bx}^{\balpha}\pa_{\by}^{\bBeta}+\alpha_j\pa_{\bx}^{\balpha}\pa_{\by}^{\bBeta},
\end{equation}
which shows that
\begin{equation}
\begin{split}
  \|\pa_{\bx}^{\balpha}\pa_{\by}^{\bBeta}x_j\pa_{x_j}R(\mu)
  f\|_{\WF,0,\gamma}&\leq
  \|x_j\pa_{x_j}\pa_{\bx}^{\balpha}\pa_{\by}^{\bBeta}R(\mu) f\|_{\WF,0,\gamma}+
\|R(\mu) f\|_{\WF,k,\gamma}
\end{split}
\end{equation}
It now follows from Proposition~\ref{lem3.3new0.0} and the $k=0$ case that this
term is bounded by
\begin{equation}
\leq C_{\alpha}\left[
\frac{1}{|\mu|^{\frac{\gamma}{2}}}+
\frac{1}{|\mu|}\right]\|f\|_{\WF,k,\gamma}.
\end{equation}
This completes the proof of the proposition.
\end{proof}

\part{Analysis of Generalized Kimura Diffusions}

\chapter{Existence of Solutions}\label{exstsoln0}
We now return to the  principal goal of this monograph, the
analysis of a generalized Kimura diffusion operator, $L,$ defined on a manifold
with corners, $P.$ The estimates proved in the previous chapters for the
solutions to model problems, along with the adapted local coordinates introduced in
Chapter~\ref{s.nrmfrms}, allow the use  of the Schauder method to prove existence
of solutions to the inhomogeneous problem
\begin{equation}\label{inhmprbP00}
  (\pa_t-L)w=g\text{ in }P\times(0,T)\text{ with }w(x,0)=f.
\end{equation}
Ultimately we will show that if
\begin{equation}\label{inhmprbPdata}
f\in\cC^{k,2+\gamma}_{\WF}(P)\quad g\in\cC^{k,\gamma}_{\WF}(P\times [0,T]))
\end{equation}
then the unique solution $w\in\cC^{k,2+\gamma}_{\WF}(P\times [0,T]).$ In this
chapter we prove the basic existence result:
\begin{theorem}\label{thm13.1} For $k\in\bbN_0$ and $0<\gamma<1,$ if the data $f,g$
  satisfy~\eqref{inhmprbPdata}, then equation~\eqref{inhmprbP00} has a
  unique solution $w\in \cC^{k,2+\gamma}_{\WF}(P\times [0,T]).$ There is a
  constant $C_{k,\gamma}$ so that
  \begin{equation}
    \|w\|_{\WF,k,2+\gamma,T}\leq C_{k,\gamma}(1+T)[\|g\|_{\WF,k,\gamma,T}+\|f\|_{\WF,k,2+\gamma}].
  \end{equation}
\end{theorem}
\begin{remark}
The hypothesis that $f\in\cC^{k,2+\gamma}_{\WF}(P)$ is not what one should expect: 
the result suggested by the non-degenerate case would be that for $f\in\cC^{k,\gamma}_{\WF}(P),$ 
there is a solution in $\cC^{k,2+\gamma}_{\WF}(P\times (0,\infty))\cap \cC^{k,\gamma}_{\WF}(P\times
[0,\infty))$. This is true for the model problems. For basic applications to probability theory, 
Theorem~\ref{thm13.1} suffices. We return   to this question in Chapter~\ref{c.resolv}.
\end{remark}

As we have done before, we write the solution $w=v+u,$ where $v$
solves the homogeneous Cauchy problem with $v(x,0)=f(x)$ and $u$
solves the inhomogeneous problem with $u(x,0)=0.$ Each part is
estimated separately. In the early sections of this chapter we treat
the $k=0$ case, returning to the problem of higher regularity at the
end. The issues with the support of the data that arose in the
analysis of higher regularity for the model problems does not arise in
the present context. This is because whenever a model solution
operator appears  as part of a parametrix it is always multiplied on the
right by a smooth compactly supported function. Hence it can be
regarded as acting on data with fixed compact support.

With $k=0,$ we begin by proving the existence of $u$ for $t\in [0,T_0],$ where
$T_0>0$ is independent of $u.$ A similar argument establishes the existence of
$v.$ Using these arguments together, we obtain existence up to time $T,$ and
the estimate given in the theorem in the $k=0$ case.  Before delving into the
details of the argument, we first give definitions for the WF-H\"older spaces
on a general compact manifold with corners, and then a brief account of the
steps involved in the existence proof. 
\section[WF H\"older spaces]{WF-H\"older spaces on a manifold with corners}
We now give precise definitions for various function spaces,
$\cC^{k,\gamma}_{\WF}(P), \cC^{k,\gamma}_{\WF}(P\times[0,T]),$ etc. which we need to
use. For the $(0,\gamma)$-case we could use an intrinsic definition, using
the singular, incomplete metric, $g_{\WF},$ determined by the principal symbol
of $L,$ to define a distance function, $d_{\WF}(\bx,\by).$ We could then define
the global $(0,\gamma)$-WF-semi-norm by setting
\begin{equation}
  \bbr{f}_{\WF,0,\gamma}=\supone_{\bx\neq \by}\frac{|f(p_1)-f(p_2)|}{d_{\WF}(p_1,p_2)^{\gamma}},
\end{equation}
and a norm on $\cC^{0,\gamma}_{\WF}(P)$ by letting
\begin{equation}\label{eqn13.5.1}
  \|f\|_{\WF,0,\gamma}=\|f\|_{L^{\infty}(P)}+ \bbr{f}_{\WF,0,\gamma}.
\end{equation}
For computations it is easier to build the global norms out of locally defined
norms. \index{$\cC^{0,\gamma}_{\WF}(P)$}

By Proposition~\ref{p.nrmfrm}, there are coordinate charts covering a neighborhood of $bP$ 
in which the operator $L$ assumes a simple normal form.  At a 
point $q\in bP$ of codimension $n$ this coordinate system
$(x_1,\dots,x_n;y_1,\dots,y_m)$ is parametrized by a subset of the form
\begin{equation}
 C^{n,m}(l)=[0,l^2)^n\times (-\frac{l}{2},\frac{l}{2})^m,
\end{equation}
where $q\leftrightarrow (\bzero_n;\bzero_m).$ Let $U$ denote the open
set centered at $q$ covered by this coordinate patch and $\phi:C^{n,m}(l)\to
U,$ the coordinate map. We call this a \emph{normal cubic coordinate} or NCC
\index{normal cubic coordinate}\index{NCC}
patch centered at $q.$ The parameter domain, $C^{n,m}(l)$ is called a
``positive cube'' of side length $l$ in $\bbR_+^n\times\bbR^m.$ In these
coordinates the operator $L$ takes the form
\begin{multline}\label{Lnrmfrm20}
  L=\sum_{i=1}^nx_i\pa_{x_i}^2+\sum_{1\leq k,l\leq
    m}c'_{kl}(\bx,\by)\pa_{y_k}\pa_{y_l}+\sum_{i=1}^nb_i(\bx,\by)\pa_{x_i}+\\
\sum_{1\leq i\neq j\leq
  n}x_ix_ja'_{ij}(\bx,\by)\pa_{x_i}\pa_{x_j}+\sum_{i=1}^n\sum_{l=1}^mx_ib'_{il}(\bx,\by)\pa_{x_i}\pa_{y_l}+
\sum_{l=1}^md_l(\bx,\by)\pa_{y_l}.
\end{multline}
The principal part of $L$ at $q$ is given by
\begin{equation}
  L^p_q=\sum_{i=1}^nx_i\pa_{x_i}^2+\sum_{1\leq k,l\leq
    m}c'_{kl}(\bzero_n,\bzero_m)\pa_{y_k}\pa_{y_l}+\sum_{i=1}^n
b_i(\bzero_n,\bzero_m)\pa_{x_i}.
\end{equation}
The matrix $c'_{kl}(\bzero_n,\bzero_m)$ is positive definite and the
coefficients $\{b_i(\bzero_n,\bzero_m)\}$ are non-negative. The estimates in
the previous chapter show that $L-L^p_q$ is, in a precise sense, a residual
term.

If $\psi\in\cC^{\infty}_c(U),$ and $f$ is defined in $U,$ then we can use the
local definitions of the various $\WF$-norms to define the local $\WF$-norms:
\begin{equation}
\begin{split}
  \|\psi f\|^U_{\WF,k,\gamma}&=\|(\psi f)\circ\phi\|_{\WF,k,\gamma}\\
\|\psi f\|^U_{\WF,k,2+\gamma}&=\|(\psi f)\circ\phi\|_{\WF,k,2+\gamma}.
\end{split}
\end{equation}
If $g$ is defined in $U\times [0,T]$ then we similarly define the local (in
space and time) norm:
\begin{equation}
\begin{split}
  \|\psi g\|^U_{\WF,k,\gamma,T}&=\|(\psi g)(\phi,\cdot)\|_{\WF,k,\gamma,T}\\
 \|\psi g\|^U_{\WF,k,2+\gamma,T}&=\|(\psi g)(\phi,\cdot)\|_{\WF,k,2+\gamma,T}
\end{split}
\end{equation}

\begin{definition}
  Let $\fW=\{(W_j,\phi_j):\: j=1,\dots, K\}$ be a cover of $bP$ by NCC
  charts, $W_0\subsubset \Int P,$ covering $P\setminus \cup_{j=1}^K
  W_j,$ and let $\{\varphi_j:\: j=0,\dots, K\}$ be a partition of
  unity subordinate to this cover. A function $f\in
  \cC^{k,\gamma}_{\WF}(P)$ provided
  $(\varphi_jf)\circ\phi_j\in\cC^{k,\gamma}_{\WF}(W_j),$ for each $j.$
  We define a global norm on $\cC^{k,\gamma}_{\WF}(P)$ by setting
\begin{equation}\label{eqn13.10.1}
  \|f\|_{\WF,k,\gamma}=\sum\limits_{j=0}^K\|(\varphi_jf)\circ\phi_j\|^{W_j}_{\WF,k,\gamma},
\end{equation}
There are analogous definitions for $\cC^{k,2+\gamma}_{\WF}(P),$
$\cC^{k,\gamma}_{\WF}(P\times[0,T]),$ and
$\cC^{k,2+\gamma}_{\WF}(P\times[0,T]).$ The corresponding norms are denoted by
$$\|f\|_{\WF,k,2+\gamma}, \|g\|_{\WF,k,\gamma,T}, \|g\|_{\WF,k,2+\gamma,T}.$$
\end{definition}
\index{$\cC^{k,\gamma}_{\WF}(P)$}\index{$\cC^{k,2+\gamma}_{\WF}(P)$}

It is straightforward to show that different NCC covers define
equivalent norms and therefore, in all cases, the topological vector
spaces do not depend on the choice of NCC cover.  Once we have fixed
such a cover, then the definitions of the norm on
$\cC^{0,\gamma}_{\WF}(P)$ in~\eqref{eqn13.5.1} and~\eqref{eqn13.10.1}
are also equivalent. In fact, if $U$ is an NCC coordinate patch of
codimension $n,$ with local coordinates
$(x_1,\dots,x_n;y_1,\dots,y_m),$ then there is a constant $C$ so that, 
for $(\bx^1,\by^1),(\bx^2,\by^2)\in U,$ we have
\begin{multline}\label{dstcmp0}
Cd_{\WF}((\bx^1,\by^1),(\bx^2,\by^2))\leq [\rho_s(\bx^1,\bx^2)+\rho_e(\by^1,\by^2)]\leq\\
C^{-1}d_{\WF}((\bx^1,\by^1),(\bx^2,\by^2)).
\end{multline}
In the remainder of this chapter we fix the cover $\fW$.

\subsection{Properties of WF-H\"older spaces}
The details of the construction of the parametrix rely on some general
results about the local function spaces
$\cC^{0,\gamma}_{\WF}(\bbR_+^n\times\bbR^m),$
$\cC^{0,\gamma}_{\WF}(\bbR_+^n\times\bbR^m\times [0,T]),$ for which it is
useful to recall the local semi-norms
 \begin{equation}
  \bbr{f}_{\WF,0,\gamma}=\supone_{(\bx^1,\by^1)\neq (\bx^2,\by^2)}
\frac{|f(\bx^1,\by^1)-f(\bx^2,\by^2)|}{[\rho_s(\bx^1,\bx^2)+\rho_e(\by^1,\by^2)]^{\gamma}},
\end{equation}
 \begin{equation}
  \bbr{g}_{\WF,0,\gamma}=\supone_{(\bx^1,\by^1,t^1)\neq (\bx^2,\by^2,t^2)}
\frac{|g(\bx^1,\by^1,t^1)-g(\bx^2,\by^2,t^2)|}{[\rho_s(\bx^1,\bx^2)+\rho_e(\by^1,\by^2)
+\sqrt{|t^2-t^1|}]^{\gamma}},
\end{equation}
and the Leibniz formula:
\begin{lemma}\label{lem7.1} Suppose that $f,g\in\cC^{0,\gamma}_{\WF}(\bbR_+^n\times\bbR^m)$ 
  or $\cC^{0,\gamma}_{\WF}(\bbR_+^n\times\bbR^m\times [0,T]).$ The semi-norm of the
  product $fg$ satisfies the estimate:
  \begin{equation}
    \bbr{fg}_{\WF,0,\gamma}\leq \|f\|_{L^{\infty}} \bbr{g}_{\WF,0,\gamma}+
\|g\|_{L^{\infty}} \bbr{f}_{\WF,0,\gamma}
  \end{equation}
\end{lemma}
\begin{proof} These estimates follow easily from the observation that, with
  $\bw_j=(\bx_j,\by_j),$ or $\bw_j=(\bx_j,\by_j,t_j),$ $j=1,2,$  we have
  \begin{multline}
    \frac{|f(\bw_1)g(\bw_1)-f(\bw_2)g(\bw_2)|}{[\rho(\bw_1,\bw_2)]^{\gamma}}\leq\\
\frac{|[f(\bw_1)-f(\bw_2)]g(\bw_1)|+|f(\bw_2)[g(\bw_1)-g(\bw_2)]|}{[\rho(\bw_1,\bw_2)]^{\gamma}},
  \end{multline}
from which the assertions of the lemma are immediate.
\end{proof}

We also have a result about the behavior of $\WF$-norms under the scaling of
cutoff functions.

\begin{lemma}\label{lem7.2} Suppose that $f\in\cC^{1}_c(\bbR_+^n\times\bbR^m)$
  has support in the positive cube $[0,l^2]^n\times [-l,l]^m.$ If $\epsilon>0$
  and we define
  \begin{equation}
    f_{\epsilon}(\bx,\by)=f\left(\frac{\bx}{\epsilon^2},\frac{\by}{\epsilon}\right),
  \end{equation}
then there is a constant $C_l$ depending on the support of $f$ so that
  \begin{equation}\label{eqn581}
    \|f_{\epsilon}\|_{\WF,0,\gamma}\leq C_l[\epsilon^{-\gamma}+1]\|f\|_{\cC^1}.
  \end{equation}
\end{lemma}
\begin{proof}  First observe that
  $\|f_{\epsilon}\|_{L^{\infty}}=\|f\|_{L^{\infty}};$ so we only need to
  estimate $\bbr{f_{\epsilon}}_{\WF,0,\gamma}.$ This estimate follows from the
  observation that
  \begin{equation}\label{eqn583.1}
    \epsilon\rho\left(\left(\frac{\bx_1}{\epsilon^2},\frac{\by_1}{\epsilon}\right),
\left(\frac{\bx_2}{\epsilon^2},\frac{\by_2}{\epsilon}\right)\right)=
\rho((\bx_1,\by_1),(\bx_2,\by_2)),
  \end{equation}
and therefore
  \begin{equation}\label{eqn586.1}
    \frac{|f_{\epsilon}(\bx_2,\by_2)-f_{\epsilon}(\bx_1,\by_1)|}
{[\rho((\bx_1,\by_1),(\bx_2,\by_2))]^{\gamma}}=
\epsilon^{-\gamma}\frac{\left|f\left(\frac{\bx_2}{\epsilon^2},\frac{\by_2}{\epsilon}\right)-
f\left(\frac{\bx_1}{\epsilon^2},\frac{\by_1}{\epsilon}\right)\right|}
{\left[\rho\left(\left(\frac{\bx_1}{\epsilon^2},\frac{\by_1}{\epsilon}\right),
\left(\frac{\bx_2}{\epsilon^2},\frac{\by_2}{\epsilon}\right)\right)\right]^{\gamma}}.
  \end{equation}
Letting $\bw_j=\bx_j/\epsilon^2,\,\bz_j=\by_j/\epsilon,$ for $j=1,2,$ this
becomes:
\begin{equation}\label{eqn586.1.1}
\begin{split}
    \frac{|f_{\epsilon}(\bx_2,\by_2)-f_{\epsilon}(\bx_1,\by_1)|}
{[\rho((\bx_1,\by_1),(\bx_2,\by_2))]^{\gamma}}&=
\epsilon^{-\gamma}\frac{\left|f(\bw_2,\bz_2)-
f(\bw_1,\bz_1)\right|}
{\left[\rho((\bw_2,\bz_2),(\bw_1,\bz_1))\right]^{\gamma}}\\
&\leq\epsilon^{-\gamma}\frac{\|\nabla f\|_{L^{\infty}}[|\bw_2-\bw_1|+|\bz_2-\bz_1|]}
{\left[\rho((\bw_2,\bz_2),(\bw_1,\bz_1))\right]^{\gamma}}
\end{split}
  \end{equation}
  where we used the mean value theorem on the right hand side
  of~\eqref{eqn586.1.1}. The second line in~\eqref{eqn586.1.1} is estimated by
  \begin{equation}\label{eqn13.22.3}
    \epsilon^{-\gamma}\|\nabla f\|_{L^{\infty}}\left[|\bz_2-\bz_1|^{1-\gamma}+
\sum_{l=1}^n|\sqrt{w_{2l}}-\sqrt{w_{1l}}|^{1-\gamma}|\sqrt{w_{2l}}+\sqrt{w_{1l}}|
\right].
  \end{equation}

  Taking the supremum of the quantity in the brackets in~\eqref{eqn13.22.3} for
  pairs $(\bw_1,\bz_1),$ $(\bw_2,\bz_2)$ lying in $[0,4l^2]^n\times
  [-2l,2l]^{m},$ shows that there is a constant $C_l,$ so that for such pairs:
\begin{equation}
  \frac{|f_{\epsilon}(\bx_2,\by_2)-f_{\epsilon}(\bx_1,\by_1)|}
{[\rho((\bx_1,\by_1),(\bx_2,\by_2))]^{\gamma}}\leq
C_l\epsilon^{-\gamma}\|\nabla f\|_{L^{\infty}}.
\end{equation}
This covers the case where both $(\bx_1,\by_1)$ and $(\bx_2,\by_2)$ lie in
certain neighborhood of the $\supp f_{\epsilon}.$ If neither point lies in
$\supp f,$ then the numerator is zero. Hence the only case remaining is when
$(\bw_1,\bz_1)\in [0,l^2]^n\times [-l,l]^{m},$ and $(\bw_2,\bz_2)\notin
[0,4l^2]^n\times [-2l,2l]^m.$ In this case the denominator in the first line
of~\eqref{eqn586.1.1} is bounded below by $l^{\gamma},$ and the numerator is
bounded above by $2\|f\|_{L^{\infty}},$ which completes the proof of the lemma.
\end{proof}

\begin{lemma}\label{lem13.3}
  Suppose that $f\in\cC^{1}(\bbR_+^n\times\bbR^m)$ and
  $a\in\cC^{1}(\bbR_+^n\times\bbR^m)$ with support in a positive cube of side
  length $l,$ and $a(0,0)=0.$ There is a constant $C,$ depending on $l$ and the
  dimension, so that, if $m=0,$ then we have
  \begin{equation}\label{eqn46200}
    \bbr{af_{\epsilon}}_{\WF,0,\gamma}\leq C
\|f\|_{\cC^1}\|a\|_{\cC^1}\epsilon^{2-\gamma}.
  \end{equation}
If $m\geq 1,$ then
 \begin{equation}\label{eqn46300}
    \bbr{af_{\epsilon}}_{\WF,0,\gamma}\leq C\|f\|_{\cC^1}\|a\|_{\cC^1}\epsilon^{1-\gamma}.
  \end{equation}
  If $a$ is a $\cC^1$-function of the variables
  $(\sqrt{\bx},\by)=(\sqrt{x_1},\dots,\sqrt{x_n},y_1,\dots,y_m),$ that is
  \begin{equation}
    a(\bx,\by)=A(\sqrt{\bx},\by),\text{ where }A\in\cC^1(\bbR_+^n\times\bbR^m),
  \end{equation}
  then the estimate in~\eqref{eqn46300} holds for $af_{\epsilon},$  with
  $\|a\|_{\cC^1}$ replaced by $\|A\|_{\cC^1}.$
\end{lemma}
\begin{proof}
We begin with  the case $m=0.$ The triangle inequality shows that
\begin{multline}\label{13.27.3}
\frac{|a(\bx^1)f_{\epsilon}(\bx^1)-a(\bx^2)f_{\epsilon}(\bx^2)|}{[\rho(\bx^1,\bx^2)]^{\gamma}}\leq\\
\frac{|a(\bx^1)-a(\bx^2)||f_{\epsilon}(\bx^1)|}{[\rho(\bx^1,\bx^2)]^{\gamma}}+
\frac{|a(\bx^2)||f_{\epsilon}(\bx^1)-f_{\epsilon}(\bx^2)|}{[\rho(\bx^1,\bx^2)]^{\gamma}}.
\end{multline}
We first assume that $\bx^1,\bx^2\in\supp f_{\epsilon}.$ In this case the
second term on the right hand side of~\eqref{13.27.3} is bounded by
\begin{equation}
  l\epsilon^2\|\nabla a\|_{L^{\infty}}\bbr{f_{\epsilon}}_{\WF,0,\gamma}\leq
C_l\epsilon^{2-\gamma}\|a\|_{\cC^1}\|f\|_{\cC^1},
\end{equation}
where we use Lemma~\ref{lem7.1} ~\ref{lem7.1} to bound $\bbr{f_{\epsilon}}_{\WF,0,\gamma}.$

The first term is bounded by
\begin{equation}
  \frac{\|f\|_{L^{\infty}}\|a\|_{\cC^1}|\bx^1-\bx^2|}
{[\rho(\bx^1,\bx^2)]^{\gamma}}\leq C \epsilon^{2-\gamma}\|f\|_{L^{\infty}}\|a\|_{\cC^1}.
\end{equation}
This proves~\eqref{eqn46200} when both $\bx^1,\bx^2\in\supp f_{\epsilon}.$
Essentially the same argument applies if $\bx^2\in\supp f_{\epsilon},$ and
$\bx^1/\epsilon^2\in [0,4l^2]^n\setminus\supp f_{\epsilon},$
  though only the second term on the right hand side of~\eqref{13.27.3} is
  non-zero. The final case we need to consider is $\bx^2\in\supp f_{\epsilon},$ and
$\bx^1/\epsilon^2\notin [0,4l^2]^n.$ For this case, the
denominator in 
\begin{equation}
  \frac{|a(\bx^2)||f_{\epsilon}(\bx^1)-f_{\epsilon}(\bx^2)|}{[\rho(\bx^1,\bx^2)]^{\gamma}}
\end{equation}
is bounded below by $(l\epsilon)^{-\gamma},$ the numerator is bounded above by
\begin{equation}
  \epsilon^2\|a\|_{\cC^1}\|f\|_{L^{\infty}},
\end{equation}
thus completing the proof of~\eqref{eqn46200} in case $m=0.$

For the case $m\neq 0,$ observe that 
\begin{multline}\label{13.32.1}
\frac{|a(\bx^1,\by^1)f_{\epsilon}(\bx^1,\by^1)-a(\bx^2,\by^2)f_{\epsilon}(\bx^2,\by^2)|}
{[\rho((\bx^1,\by^1),(\bx^2,\by^2))]^{\gamma}}\leq\\
\frac{|a(\bx^1,\by^1)-a(\bx^2,\by^2)||f_{\epsilon}(\bx^1,\by^1)|+
|f_{\epsilon}(\bx^1,\by^1)-f_{\epsilon}(\bx^2,\by^2)||a(\bx^2,\by^2)|}
{[\rho((\bx^1,\by^1),(\bx^2,\by^2))]^{\gamma}}
\end{multline}
Note that $\supp f_{\epsilon}\subset [0,\epsilon^2l^2]^n\times [-\epsilon l,\epsilon
l]^m.$ If both points again belong to the $\supp f_{\epsilon},$  then the
quantity on the right hand side of~\eqref{13.32.1} is bounded by
\begin{equation}
  (l\epsilon)^{1-\gamma}\|a\|_{\cC^1}\|f\|_{L^{\infty}}+\epsilon\bbr{f_{\epsilon}}_{\WF,0,\gamma}
\leq C_l\epsilon^{1-\gamma}\|a\|_{\cC^1}\|f\|_{\cC^1}.
\end{equation}
If now $(\bx^2,\by^2)\in\supp f_{\epsilon},$ but $(\bx^1,\by^1)\in
[0,4\epsilon^2l^2]^n\times [-2\epsilon l,2\epsilon l]^m\setminus\supp
f_{\epsilon},$ then only the second term on the right hand side
of~\eqref{13.32.1} is non-zero; it is estimated by
\begin{equation}
 (l\epsilon) \|a\|_{\cC^1}\bbr{f_{\epsilon}}_{\WF,0,\gamma}\leq
C_l\epsilon^{1-\gamma}\|a\|_{\cC^1}\|f\|_{\cC^1}.
\end{equation}
Finally, if $(\bx^2,\by^2)\in\supp f_{\epsilon},$ but $(\bx^1,\by^1)\notin
[0,4\epsilon^2l^2]^n\times [-2\epsilon l,2\epsilon l]^m,$ then the denominator
is bounded below by $(l\epsilon)^{\gamma},$ and the numerator is bounded above
by $\epsilon\|a\|_{\cC^1}\|f\|_{L^{\infty}},$ which completes the proof in this case.

The argument if $a$ is a $\cC^1$-function of $(\sqrt{\bx},\by)$ is again
quite similar.
\begin{multline}\label{eqn13.35.1}
  \frac{|A(\sqrt{\bx^1},\by^1)f_{\epsilon}(\bx^1,\by^1)-A(\sqrt{\bx^2},\by^2)f_{\epsilon}(\bx^2,\by^2)|}{
|[\rho((\bx^1,\by^1),(\bx^2,\by^2))]^{\gamma}}\leq\\
 \frac{|A(\sqrt{\bx^1},\by^1)-A(\sqrt{\bx^2},\by^2)||f_{\epsilon}(\bx^1,\by^1)|+
|A(\sqrt{\bx^2},\by^2)||f_{\epsilon}(\bx^1,\by^1)-f_{\epsilon}(\bx^2,\by^2)|}{
|[\rho((\bx^1,\by^1),(\bx^2,\by^2))]^{\gamma}}
\end{multline}
We observe that
\begin{equation}
\begin{split}
  |A(\sqrt{\bx^1},\by^1)-A(\sqrt{\bx^2},\by^2)|&\leq
\|\nabla A\|_{L^{\infty}}\left[\sum_{j=1}^n|\sqrt{x^1_j}-\sqrt{x^2_j}|+|\by^1-\by^2|\right]\\
&=\|\nabla A\|_{L^{\infty}}\rho((\bx^1,\by^1),(\bx^2,\by^2)).
\end{split}
\end{equation}
If both points are in $\supp f_{\epsilon},$ then the right hand side
of~\eqref{eqn13.35.1} is bounded by
\begin{equation}
  \|\nabla A\|_{L^{\infty}}\left\{[\rho((\bx^1,\by^1),(\bx^2,\by^2))]^{1-\gamma}\|f\|_{L^{\infty}}+
\epsilon\bbr{f_{\epsilon}}_{\WF,0,\gamma}\right\}
\end{equation}
Since both points are in $\supp f_{\epsilon},$ Lemma~\ref{lem7.1} shows that
this is bounded by
\begin{equation}
  C\epsilon^{1-\gamma}\|A\|_{\cC^1}\|f\|_{\cC^2}
\end{equation}
The other cases follow similarly.
\end{proof}

\section{Overview of the Proof}\label{s.ovrprfHK}
The domain  $P$ is assumed to be a manifold with corners of dimension $N>1.$
The boundary of $P$ is a stratified space with
\begin{equation}
  bP=\bigcup_{j=1}^M\Sigma_j,
\end{equation}
where $\Sigma_j$ is the  (open) stratum of co-dimension $j$ boundary points. From
the definition of manifold with corners it follows that
\begin{equation}
  \overline{\Sigma_k}=\bigcup_{j=k}^M\Sigma_j.
\end{equation}
To prove the existence of a solution to the equation,~\eqref{inhmprbP00} we use
an induction over $M$ the maximal codimension of a stratum of $bP.$

The argument begins by assuming that $P$ is a manifold with boundary,
i.e. $M=1.$ Using the estimates proved in the previous chapter we can easily
show that there is a function $\varphi\in\cC^{\infty}_c(P),$ equal to $1$ in
a neighborhood of $bP$ and  an operator
\begin{equation}\label{eqn423.02}
  \hQ_b^t:\cC^{k,\gamma}_{\WF}(P\times [0,T])\to
\cC^{k,2+\gamma}_{\WF}(P\times [0,T]),
\end{equation}
so that 
\begin{equation}
  (\pa_t-L)\hQ_b^tg=\varphi g+(E_b^{0,t}+E_b^{1,t})g,
\end{equation}
where
\begin{equation}
  E_b^{0,t}, E_b^{1,t}:\cC^{k,\gamma}_{\WF}(P\times [0,T])\to :\cC^{k,\gamma}_{\WF}(P\times [0,T]),
\end{equation}
are bounded and $E_b^{1,t},$ is a compact operator on this space, which tends to
zero in norm as $T$ tends to zero. If $k=0,$ then we can arrange for $E_b^{0,t}$ to
have norm as small as we please.

Let $U$ be a neighborhood of $bP$ so that $bU\cap \Int P$ is a smooth
hypersurface in $P$, and $\overline{U}\subsubset \varphi^{-1}(1).$ The subset
$P_U=P\cap U^c$ is a smooth compact manifold with boundary, and
$L\restrictedto_{P_U}$ is a non-degenerate elliptic operator.  We can double
$P_{U}$ across its boundary to obtain $\tP_{U},$ which is a manifold without
boundary. The operator $L$ can be extended to a classically elliptic operator
$\tL$ defined on all of $\tP_{U}.$ The classical theory of non-degenerate
parabolic equations on compact manifolds, without boundary, applies to
construct an \emph{exact} solution operator $u_i=\tQ^t[(1-\varphi) g]$ to the
inhomogeneous equation:
\begin{equation}
\begin{split}
  &(\pa_t-\tL)u_i=(1-\varphi)\tg\text{ in }\tP_{U}\times [0,T]\\
  &\text{with }u_i(p,0)=0, p\in \tP_U.
\end{split}
\end{equation}
This operator defines bounded maps from $\cC^{k,\gamma}_{\WF}(\tP_U\times
[0,T])\to \cC^{k,2+\gamma}_{\WF}(\tP_U\times [0,T]),$ for any $0<\gamma<1$ and
$k\in\bbN_0.$ Of course, in $\tP_U$ these spaces are equivalent to the classical
heat H\"older spaces $\cC^{k,\gamma}(\tP_U\times [0,T])$ and
$\cC^{k+2,\gamma}(\tP_U\times [0,T]),$ respectively.  

To complete the construction when $M = 1$, choose $\psi\in\cC^{\infty}_c(P_U)$
so that $\psi\equiv 1$ on a neighborhood of the support of $(1-\varphi),$ and set
\begin{equation}
  \hQ^tg=\hQ_b^tg+\hQ_i^tg
\end{equation}
where
\begin{equation}
  \hQ_i^t=\psi\tQ^t[(1-\varphi)g].
\end{equation}
Here it is understood that $(1-\varphi)g$ and $\psi\tQ^t[(1-\varphi)g]$ are extended by zero to all of 
$\tP_{U}$ and $P$, respectively. Applying the operator gives 
\begin{equation}
  (\pa_t-L)\hQ^tg=g+(E_b^{0,t}+E_b^{1,t})g+E_i^{\infty,t}g,
\end{equation}
where
\begin{equation}
  E_i^{\infty,t}g=[\psi,L]\hQ_i^t[(1-\varphi)g].
\end{equation}

Since $\psi\equiv 1$ on a neighborhood of the support of $(1-\varphi)$ it follows
from classical results that $E_i^{\infty,t}$ is a smoothing operator
which tends to zero exponentially as $T\to 0^+.$ More generally, assume by
induction that $E_i^{\infty,t}$ is a compact operator tending to zero, as
$T_0\to 0,$ in the operator norm defined by $\cC^{k,\gamma}_{\WF}(P\times
[0,T_0]).$ If $T_0$ is sufficiently small, then the operator
\begin{equation}
  E_b^{0,t}+E_b^{1,t}+E_i^{\infty,t}=E^t:
\cC^{0,\gamma}_{\WF}(P\times [0,T_0])\longrightarrow \cC^{0,\gamma}_{\WF}(P\times [0,T_0])
\end{equation}
has norm strictly less than $1,$ and therefore $(\Id+E^t)$ is invertible. Thus
the operator
\begin{equation}
 \cQ^t=\hQ^t(\Id+E^t)^{-1}.
\end{equation}
is a right inverse to $(\pa_t-L)$ up to time $T_0$ and is a bounded map
\begin{equation}
  \cQ^t:\cC^{0,\gamma}_{\WF}(P\times [0,T_0])\to
\cC^{0,2+\gamma}_{\WF}(P\times [0,T_0]).
\end{equation}
At the end of this chapter we use a result from~\cite{EpNeumSeries} to show that the Neumann
series for $(\Id+E^t)^{-1}$ converges in the operator norm topology of
$\cC^{k,\gamma}_{\WF}(P\times [0,T_0]),$ for any $k\in\bbN.$

To handle the case of higher codimension boundaries we use the following induction
hypotheses:
 \smallskip 
\newline
{\bf[Inhomogeneous Case:]}  Let $P$ be any
    manifold with corners  such that the maximal codimension of $bP$  is less
    than or equal to $M,$ and let $L$ be a generalized Kimura diffusion on $P.$ We
    assume that the solution operator $\cQ^t$ of the initial value problem
    \begin{equation}
      (\pa_t-L)u=g\text{ with } u(x,0)=0,
    \end{equation}
exists and has the following properties: 
\begin{enumerate}
\item For $k\in \bbN_0,$ $0<T,$ and $0<\gamma<1,$ the maps
  \begin{equation}
    \cQ^t:\cC^{k,\gamma}_{\WF}(P\times [0,T])\longrightarrow
\cC^{k,2+\gamma}_{\WF}(P\times [0,T])
  \end{equation}
are bounded.  The maps
\begin{equation}
    \cQ^t:\cC^{k,\gamma}_{\WF}(P\times [0,T])\longrightarrow
\cC^{k,\gamma}_{\WF}(P\times [0,T])
  \end{equation}
tend to zero in norm as $T\to 0^+.$
\item Let $\psi_1,\psi_2\in\cC^{\infty}(P)$ be such that $\dist(\supp\psi_1,\supp\psi_2)>0$. 
Then the operator
\begin{equation}
\psi_1\cQ^t\psi_2:\cC^{k,\gamma}_{\WF}(P\times [0,T])\longrightarrow
\cC^{k,2+\gamma}_{\WF}(P\times [0,T])
\end{equation}
is compact, and its norm tends to zero as $T\to 0^+.$ We call this the \emph{small time localization property}.\index{small time localization property}
\end{enumerate}
{\bf[Homogeneous Case:]} 
We also assume the existence of a solution operator $\cQ^t_0$ for the homogeneous
Cauchy problem:
\begin{equation}
      (\pa_t-L)v=v\text{ with } v(x,0)=f,
    \end{equation}
with the following properties: 
\begin{enumerate}
\item For $k\in \bbN_0,$ $0<T,$ and $0<\gamma<1,$ the maps
  \begin{equation}
    \cQ^t_0:\cC^{k,2+\gamma}_{\WF}(P)\longrightarrow
\cC^{k,2+\gamma}_{\WF}(P\times [0,T])
  \end{equation}
  are bounded. 
\item As $t\to 0^+,$ for $f\in\cC^{k,2+\gamma}_{\WF}(P),$ and
  $0<\tgamma<\gamma,$ we have that
\begin{equation}
  \lim_{t\to 0^+}\|\cQ^t_0f-f\|_{\WF,k,\tgamma}=0.
\end{equation}
\item If $\psi_1,\psi_2\in\cC^{\infty}(P)$ have
    $\dist(\supp\psi_1,\supp\psi_2)>0,$ then the  operator
\begin{equation}
\psi_1\cQ^t_0\psi_2:\cC^{k,2+\gamma}_{\WF}(P\times [0,T])\longrightarrow
\cC^{k,2+\gamma}_{\WF}(P\times [0,T]),
\end{equation}
is compact and tends to zero in norm as $T\to 0^+.$
\end{enumerate}

\noindent
To carry out the induction step we require the following basic geometric result:
\begin{theorem}\label{dblthm} Let $P$ be a compact manifold with corners with
maximal codimension of $bP$ equal to $M\geq 1,$ and  $L$  a generalized
Kimura diffusion operator on $P.$ Suppose that 
\begin{equation}
bP=\Sigma_1\cup\dots\cup\Sigma_M,
\end{equation}
where each $\Sigma_j$ is the boundary component of $bP$ of codimension $j$, 
and let $U\subset P$ be a neighborhood of $\Sigma_M.$ There exists a compact
manifold with corners $\tP$ so that the maximal codimension of $b\tP$ is $M-1,$
with a generalized Kimura diffusion operator $\tL$ defined on $\tP.$
The subset $P_U=P\cap U^c$ is diffeomorphic to a subset of $\tP$ under a map $\Psi$
which carries $L_U=L\restrictedto_{P\cap U^c}$ to $\tL\restrictedto_{\Psi(P\cap
  U^c)}.$ 
\end{theorem}
\begin{remark} Informally we say that $(P_U,L_U)$ is embedded into $(\tP,\tL).$
\end{remark}
\noindent 
The proof of Theorem~\ref{dblthm} is given later in this chapter. To
carry out the induction step, we use Theorem~\ref{dblthm} to embed
$(P_U,L_U)$ into $(\tP,\tL),$ where $\tP$ is a manifold with corners,
of codimension at most $M-1.$ The induction hypothesis shows that
there is an exact solution operator $\tQ^t_i$ for the equation
$(\pa_t-\tL)\tu=\tg$ on $\tP.$ In the sequel we refer to this as the
\emph{interior term}, which explain the $i$ subscript. In the context
of inductive arguments over the maximal codimension of the $bP,$ we
use the adjective ``interior'' to refer to the things coming from
parts of $P$ disjoint from the maximal codimensional part of
$bP.$\index{interior}

We use the codimension $M$ model operators to build a boundary
parametrix, $\tQ^t_b,$ in a neighborhood of $\Sigma_M.$ Arguing much
as in the codimension 1 case, we can glue $\tQ^t_i$ to $\tQ^t_b$ to
obtain an operator
$$\tQ^t:\cC^{k,\gamma}_{\WF}(P\times [0,T])\to \cC^{k,2+\gamma}_{\WF}(P\times
[0,T]),$$
so that
\begin{equation}
  (\pa_t-L)\tQ^t=\Id+E^t, \text{ with }E^t:\cC^{k,\gamma}_{\WF}(P\times
  [0,T])\to \cC^{k,\gamma}_{\WF}(P\times [0,T]).
\end{equation}
As before, if $k=0,$ then we can arrange to have the norm of the error term
$E^t$ bounded by any fixed $\delta<1,$ as $T\to 0.$ Thus, for some $T_0>0,$ we
obtain the exact solution operator for $(P,L)$ by setting
$\cQ^t=\tQ^t(\Id+E^t)^{-1};$ this operator defines a bounded map
\begin{equation}\label{eqn561.1}
\cQ^t:
\cC^{0,\gamma}_{\WF}(P\times [0,T_0])\to \cC^{0,2+\gamma}_{\WF}(P\times
[0,T_0]).
\end{equation}
In Section~\ref{ss.bndprmtrx} we give the detailed construction of a boundary
parametrix for the maximal codimension stratum of the boundary. Combining this
with the estimates in Section~\ref{ss13.4} we verify the induction hypothesis
in the base case that $M=1$ and also the inductive step itself, which completes
the proof for the $k=0$ case. The
estimates with $k>0$ are left for the end of this chapter.

\section{The induction argument}\label{ss13.4}

To complete the proof of the theorem we need only verify the induction
hypothesis.   Assume that $P$ is a manifold with corners so
that the maximal codimension of $bP$ is $M+1,$ and that $L$ is a generalized
Kimura diffusion operator on $P.$ Using the estimates in the previous chapters we show in
Section~\ref{ss.bndprmtrx} that there is a function
$\varphi\in\cC^{\infty}_c(P),$ that equals $1$ on a small neighborhood of
$\Sigma_{M+1},$ and vanishes outside a slightly larger neighborhood, and an
operator $\hQ_b^t,$ with the mapping properties in~\eqref{eqn423.02}, so that,
for $g\in \cC^{0,\gamma}_{\WF}(P\times [0,T]),$ we have:
\begin{equation}
  (\pa_t-L)\hQ_b^tg=\varphi g+(E_{b}^{0,t}+E_{b}^{1,t})g.
\end{equation}
Here $E_{b}^{0,t}$ and $E_{b}^{1,t}$ are bounded maps of
$\cC^{k,\gamma}_{\WF}(P\times [0,T]),$ for any $k\in\bbN_0$ and
$0<\gamma<1.$ Below we show that for any $\delta > 0$ we can construct
$E_{b}^{0,t}$ so that its norm, acting on
$\cC^{0,\gamma}_{\WF}(P\times [0,T]),$ is less than $\delta$ and
$E_{b}^{1,t}$ is a compact map of this space to itself, which tends to
zero in norm as $T\to 0.$ At the end of the chapter this is verified
for $\cC^{k,\gamma}_{\WF}(P\times [0,T]),$ with $k\in\bbN.$

Let $U$ be a neighborhood of $\Sigma_{M+1}$ so that $\overline{U}\subsubset
\Int\varphi^{-1}(1);$ set
\begin{equation}
  P_U=P\cap U^c\text{ and }L_U=L\restrictedto_{P_U}.
\end{equation}
We now apply Theorem~\ref{dblthm} to find a manifold with corners $\tP,$ of
maximal codimension $M,$ and a generalized Kimura diffusion operator $\tL$ so
that $(P_U,L_U)$ is embedded into $(\tP,\tL).$ The induction hypothesis implies
that there is a solution operator $\tQ^t$ to the equation $(\pa_t-\tL)\tu=\tg$
on $\tP$ with the desired mapping properties with respect to the $\WF$-H\"older
spaces on $\tP.$ As before, we choose $\psi\in\cC^{\infty}_c(P_U)$ so that
$\psi\equiv 1$ on a neighborhood of the support of $(1-\varphi)$ and define
\begin{equation}
  \hQ_i^tg=\psi\tQ^t[(1-\varphi)g],
\end{equation}
where it  is understood that we extend $(1-\varphi)g$ by zero, to  $\tP,$
and $\psi\tQ^t[(1-\varphi)g]$ by zero to $P.$

If we let $\hQ^t=\hQ_b^t+\hQ_i^t,$ then
\begin{equation}
  (\pa_t-L)\hQ^tg=g+(E_{b}^{0,t}+E_{b}^{1,t}+E_{i}^{t})g,
\end{equation}
where, as before, $E_{i}^tg=[\psi,L]\tQ^t[(1-\varphi)g].$ The support of the
kernel of $E_{i}^t$ is a positive distance from the diagonal and therefore the
induction hypothesis implies that this is again a compact operator, tending to
zero, as $T\to 0^+,$ in the operator norms defined by
$\cC^{k,\gamma}_{\WF}(P\times [0,T_0]).$ If we choose $T_0$ sufficiently small,
then, with $E^t=(E_{b}^{0,t}+E_{b}^{1,t}+E_{i}^t),$ the operator $\Id+E^t$ is
invertible as map from $\cC^{k,\gamma}_{\WF}(P\times[0,T_0])$ to itself. We set
  \begin{equation}
    \cQ^t=\hQ^t(\Id+E^t)^{-1},
  \end{equation}
to get a right inverse to $(\pa_t-L)$ on the time interval $[0,T_0],$ which
clearly has the correct mapping properties with respect to the $\WF$-H\"older spaces on $P.$ 

But for the construction of the boundary parametrix, which is done in
Section~\ref{ss.bndprmtrx}, we can complete the proof of the induction step in this case by
showing that $ \cQ^t$ has the small time localization property. That is, if
$\varphi',\psi'$ are smooth functions on $P$ with
\begin{equation}
  \dist(\supp\varphi',\supp\psi')>0,
\end{equation}
then the operator
\begin{equation}
\psi'\cQ^t\varphi':\cC^{k,\gamma}_{\WF}(P\times [0,T])\longrightarrow
\cC^{k,2+\gamma}_{\WF}(P\times [0,T]),
\end{equation}
is a compact operator that tends to zero in norm, as $T\to 0^+.$

The operator $(\Id+E^t)^{-1}$ as a map from $\cC^{k,\gamma}_{\WF}(P\times
[0,T])$ to itself is defined as a convergent Neumann series
\begin{equation}
  (\Id+E^t)^{-1}=\sum_{j=0}^{\infty}(-E^t)^j.
\end{equation}
Given $\eta>0,$ there is a $N$ so that for any $T<T_0,$ we have that  
\begin{equation}
  \left\|\sum_{j=N+1}^{\infty}(-E^t)^j\right\|_{\WF,k,\gamma,T}\leq \eta.
\end{equation}
The induction hypothesis and the properties of the solution operators to the model
problems shows that the operator $\psi' \hQ^t\varphi'$ has the small time
localization property. Therefore the essential point is to see that this is
true of a composition $A^tB^t.$ 
\begin{lemma}\label{lemsmtm} Suppose that for $t\in [0,T],$ the maps
  \begin{equation}\label{eqn13.73.02}
    \begin{split}
A^t:&\cC^{k,\gamma}_{\WF}(P\times [0,T])\longrightarrow
\cC^{k,\gamma}_{\WF}(P\times [0,T])\\ 
B^t:&\cC^{k,\gamma}_{\WF}(P\times [0,T])\longrightarrow \cC^{k,\gamma}_{\WF}(P\times [0,T])    
    \end{split}
  \end{equation}
are bounded, so that if $\varphi$ and $\psi$ are smooth functions with disjoint supports, 
then $\varphi A^t\psi$ and $\varphi B^t\psi$ have the small time localization property, i.e., 
are compact and tend to zero in norm as $T\to 0^+.$  Moreover, the composition
\begin{equation}
A^t B^t:\cC^{k,\gamma}_{\WF}(P\times [0,T])\longrightarrow \cC^{k,\gamma}_{\WF}(P\times [0,T]),
\end{equation}
has the same property.
\end{lemma}
\begin{proof}
Let $\varphi,\psi$ be as above. Choose $\theta\in\cC^{\infty}(P)$ with the properties:
\begin{equation}
  \dist(\supp\psi,\supp\theta)>0,\text{ and }\dist(\supp\varphi,\supp(1-\theta))>0,
\end{equation}
so that $\theta\equiv 1$ on a neighborhood of $\supp\varphi.$ We observe that
\begin{equation}
  \psi A^tB^t\varphi= [\psi A^t\theta]B^t\varphi+ \psi A^t[(1-\theta)B^t\varphi].
\end{equation}
The operators $[\psi A^t\theta]$ and $[(1-\theta)B^t\varphi]$ have the small time
localization property. Hence $ \psi A^tB^t\varphi$ is compact and
converges in norm to zero as $T\to 0^+.$  
\end{proof}

If we let
\begin{equation}
 \Id+F_N^t= \sum_{j=0}^{N}(-E^t)^j\text{ and }\cQ^t_N=\hQ^t(\Id+F_N^t),
\end{equation}
then this lemma shows that the operator $\psi'(\Id+F_N^t)\varphi'$ is compact as a map from
$\cC^{k,\gamma}_{\WF}(P\times [0,T])$ to itself and tends to zero in norm, as $T\to 0^+.$ 
Furthermore, the difference 
\begin{equation}
\cQ^t-\cQ^t_N:\cC^{k,\gamma}_{\WF}(P\times [0,T])\longrightarrow
\cC^{k,2+\gamma}_{\WF}(P\times [0,T])
\end{equation}
tends to zero in the norm topology. With $\theta$ as above:
\begin{equation}
 \psi' \cQ^t_N\varphi'=
[\psi' \hQ^t\theta](\Id+F_N^t)\varphi'+\psi' \hQ^t[(1-\theta)(\Id+F_N^t)\varphi'],
\end{equation}
which shows, as above, that $\psi' \cQ^t_N\varphi'$ has the small time
localization property, and therefore $\psi' \cQ^t\varphi'$
is also compact. Finally for any $\eta>0,$ there is an $N$ so that
as a map from  $\cC^{k,\gamma}_{\WF}(P\times [0,T])$ to
$\cC^{k,2+\gamma}_{\WF}(P\times [0,T])$ for $T<T_0$ we have
\begin{equation}
\|\psi'(\cQ^t-\cQ^t_N)\varphi'\|\leq \eta.
\end{equation}
This shows that the norm of
\begin{equation}
 \psi' \cQ^t\varphi':\cC^{k,\gamma}_{\WF}(P\times [0,T])\longrightarrow
\cC^{k,2+\gamma}_{\WF}(P\times [0,T])
\end{equation}
tends to zero as $T\to 0^+.$ This establishes that as an operator from
$\cC^{k,\gamma}_{\WF}(P\times [0,T])$ to $\cC^{k,2+\gamma}_{\WF}(P\times
[0,T]),$ the solution operator $\cQ^t$ has the small time localization property. 

To complete this part of the argument, we need only show that for any $k\geq 0$
the Neumann series for $(\Id+E^t)^{-1}$ converges in operator norm topology
defined by $\cC^{k,\gamma}_{\WF}(P\times [0,T]),$ and that
$E^t:\cC^{k,\gamma}_{\WF}(P\times [0,T])\to \cC^{k,\gamma}_{\WF}(P\times
[0,T])$ has the small time localization property. The induction hypothesis
shows that the interior error term $E^t_i$ has this property, so it only needs
to be verified for the boundary contribution to $E^t.$ The detailed
construction of the boundary parametrix is done in the following section, for
$k=0.$ The argument for $k>0$ is presented at the end of the chapter.

\section{The Boundary Parametrix Construction}\label{ss.bndprmtrx}

In this section we give the details of the argument that if $P$ is a manifold with
corners so that the maximal codimension of $bP$ is $M+1,$ and $L$ is a
generalized Kimura diffusion defined on $P,$ then given $\delta >0,$ there is
an operator $\hQ^t_b$ and a function $\varphi\in\cC^{\infty}_c(P)$ so that
\begin{enumerate}
\item $\varphi$ equals 1 in a neighborhood of $\Sigma_{M+1}.$
\item For any $0<\gamma<1$ and some $T_0>0,$ we have
  $\hQ^t_b:\cC^{k,\gamma}_{\WF}(P\times [0,T_0])\to\cC^{k,2+\gamma}_{\WF}(P\times
  [0,T_0])$ is a bounded operator. As a map from $\cC^{k,\gamma}_{\WF}(P\times
  [0,T_0])$ to itself, this operator tends, as $T_0\to 0^+,$ to zero in norm.
\item For $g\in \cC^{k,\gamma}_{\WF}(P\times [0,T_0])$ 
  \begin{equation}
    (\pa_t-L)\hQ^t_bg=\varphi g+(E_b^{0,t}g+E_b^{1,t}g),
  \end{equation}
where $E_b^{0,t}$ has norm at most $\delta$ as an operator on
$\cC^{k,\gamma}_{\WF}(P\times [0,T_0])$ and $E_b^{1,t}$ is a compact operator on
this space with norm tending to zero as $T_0\to 0^+.$
\item The family of operators $E_b^{0,t}$ has the small time localization property.
\end{enumerate}
In this section we verify claims 1--4 in the case that $k=0.$

\subsection{The Codimension $N$ case }
The argument is a little simpler if $M+1=N=\dim P,$ so that the stratum
$\Sigma_{N}$ consists of a finite number of isolated points.  We begin the
construction by choosing an $0<\eta<\!<1.$ The set $\Sigma_N$ is finite and
consists of $p$ points, which we generically denote by $q.$ For each $1\leq
j\leq p$ we let $\fU_{N}=\{(U_{ j},\varphi_{ j}):\: j=1,\dots,p\}$ be an
NCC covering of a neighborhood of $\Sigma_N.$ By shrinking these neighborhoods,
if needed, we can assume that these sets are disjoint, each containing a single
element of $\Sigma_N.$ For consistency with later cases we let $F_{ j}=\Sigma_N\cap U_{ j}.$ 
We use the sets in $\fU_N$ to define local norms, $\|\cdot\|^{ j}_{\WF,k,\gamma,T}$ on $\cC^{k,\gamma}_{\WF}.$ 

Let $(x_1,\dots,x_N)$ denote normal cubic coordinates in one of these
neighborhoods, $U_{ j},$ centered at the point $q.$ In these coordinates the
operator $L$ takes the form
\begin{equation}
  L=\sum_{j=1}^Nx_j\pa_{x_j}^2+\sum_{j=1}^N(b_j+\tb_j(\bx))\pa_{x_j}+
\sum_{1\leq j\neq k\leq
  N}\sqrt{x_jx_k}a'_{jk}(\bx)\sqrt{x_jx_k}\pa_{x_j}\pa_{x_k},
\end{equation}
here $\tb_j$ are smooth functions vanishing at $\bx=0$ and $a'_{jk}$ are smooth
functions, and we let $\bb=(b_1,\dots,b_N).$ We let
$\chi\in\cC^{\infty}_c(\bbR^N_+)$ be a non-negative function which equals $1$
in the positive cube of side length $2,$ centered at $\bx=0$ and vanishes
outside the positive cube of side length $3,$ and
$\psi\in\cC^{\infty}_c(\bbR^N_+)$ be a non-negative function, which equals $1$
in the positive cube of side length $4$ centered at $\bx=0$ and vanishes
outside the positive cube of side length $5.$ We define
\begin{equation}
  \chi_{\epsilon}(\bx)=\chi\left(\frac{\bx}{\epsilon^2}\right)\text{ and }
 \psi_{\epsilon}(\bx)=\psi\left(\frac{\bx}{\epsilon^2}\right).
\end{equation}
Let $K^{\bb,t}_{ i,q}$ denote the solution operator for the model problem
\begin{equation}
  \left(\pa_t-L_{ i,q}\right)u=g,\quad u(x,0)=0,
\end{equation}
where
\begin{equation}
  L_{ i,q}=\sum_{j=1}^N[x_j\pa_{x_j}^2+b_j\pa_{x_j}].
\end{equation}
In the calculations that follow we suppress the explicit changes of variable,
but understand that they introduce bounded constants into the estimates that are
independent of $\epsilon.$ We have that
\begin{multline}
  (\pa_t-L)\psi_{\epsilon}K^{\bb,t}_{ i,q}[\chi_{\epsilon} g]=\chi_{\epsilon} g+
[\psi_{\epsilon},L]K^{\bb,t}_{ i,q}[\chi_{\epsilon} g]+\\
\psi_{\epsilon}( L_{ i,q}-L)K^{\bb,t}_{ i,q}[\chi_{\epsilon} g].
\end{multline}
As $\psi_{\epsilon}=1$ on the $\epsilon$-neighborhood of the
$\supp\chi_{\epsilon},$ the support of the kernel function of the commutator
term is contained in the complement of the $\epsilon$-neighborhood of the
diagonal. Hence, for any $\epsilon>0,$ Proposition~\ref{offdiagnm} shows that
this term converges exponentially to zero in the
$\cC^{k,\gamma}_{\WF}$-operator norm for any $k\in\bbN_0$ and $\gamma\in (0,1).$
That is, for any $k,\gamma,T$ there are positive constants $C(k,\gamma,T)$ and
$\mu(k,\gamma),$ so that, with
\begin{equation}
  E^{\infty,t}_{ i,q}g=[\psi_{\epsilon},L]K^{\bb,t}_{i,q}[\chi_{\epsilon} g],
\end{equation}
we have
\begin{equation}\label{0Ninflocest}
  \|E^{\infty,t}_{ i,q}g\|_{\WF,k,\gamma,T}\leq
  C(T,k,\gamma)\epsilon^{-\mu(k,\gamma)}\|g\|_{\WF,0,\gamma,T}. 
\end{equation}
The constants $C(T,k,\gamma)$ tends to zero as $T\to 0^+.$

This leaves only the last term:
\begin{multline}\label{eqn478.00}
 E^{0,t}_{ i,q}g=  \psi_{\epsilon}( L_{ i,q}-L)K^{\bb,t}_{ i,q}[\chi_{\epsilon} g]=\\
-\psi_{\epsilon}\left[\sum_{i=1}^N\tb_i(\bx)\pa_{x_i}+
\sum_{1\leq i\neq j\leq
  N}\sqrt{x_ix_j}a'_{ij}(\bx)\sqrt{x_ix_j}\pa_{x_i}\pa_{x_j}\right]
K^{\bb,t}_{ i,q}[\chi_{\epsilon} g]
\end{multline}
We need to estimate the $\cC^{0,\gamma}_{\WF}$-norm of this term, which involves
two parts, the sup-norm part
\begin{equation}
 I= \left\|\psi_{\epsilon}\left[\sum_{i=1}^N\tb_i(\bx)\pa_{x_i}+
\sum_{1\leq i\neq j\leq
  N}\sqrt{x_ix_j}a'_{ij}(\bx)\sqrt{x_ix_j}\pa_{x_i}\pa_{x_j}\right]
K^{\bb,t}_{ i,q}[\chi_{\epsilon} g]\right\|_{L^{\infty}},
\end{equation}
and the $(0,\gamma)$-semi-norm part:
\begin{equation}
II=  \bbr{\psi_{\epsilon}\left[\sum_{i=1}^N\tb_i(\bx)\pa_{x_i}+
\sum_{1\leq i\neq j\leq
  N}\sqrt{x_ix_j}a'_{ij}(\bx)\sqrt{x_ix_j}\pa_{x_i}\pa_{x_j}\right]
K^{\bb,t}_{ i,q}[\chi_{\epsilon} g]}_{\WF,0,\gamma,T},
\end{equation}
which we estimate using Lemma~\eqref{lem7.1}. Since the function
$\psi_{\epsilon}$ is supported in the set where $x_i\leq 5\epsilon^2,$
Proposition~\ref{prop1n0} implies that the
first term is estimated by
\begin{equation}
  C\epsilon^2\|\chi_{\epsilon} g\|_{\WF,0,\gamma,T}.
\end{equation}
Applying Lemmas~\ref{lem7.1} and~\ref{lem7.2} and we see that
\begin{equation}
  I\leq C\epsilon^{2-\gamma}\|g\|_{\WF,0,\gamma,T},
\end{equation}
where the constant $C$ is independent of $\epsilon.$ 

To estimate II, Lemma~\ref{lem7.1} shows that we need to consider terms of the
forms
\begin{equation}\label{eqn596}
  \|\psi_{\epsilon}\tb_i(\bx)\|_{L^{\infty}}\bbr{\pa_{x_i}K^{\bb,t}_{ i,q}[\chi_{\epsilon}
    g]}_{\WF,0,\gamma,T},
\bbr{ \psi_{\epsilon}\tb_i(\bx) }_{\WF,0,\gamma,T}\|\pa_{x_i}K^{\bb,t}_{ i,q}[\chi_{\epsilon}  g]\|_{L^{\infty}}
\end{equation}
and
\begin{multline}\label{eqn597}
  \|\psi_{\epsilon}\sqrt{x_ix_j}a'_{ij}(\bx)\|_{L^{\infty}}\bbr{\sqrt{x_ix_j}\pa_{x_i}\pa_{x_j}
K^{\bb,t}_{ i,q}[\chi_{\epsilon} g]}_{\WF,0,\gamma,T},\\
\bbr{\psi_{\epsilon}\sqrt{x_ix_j}a'_{ij}(\bx)}_{\WF,0,\gamma,T}\|\sqrt{x_ix_j}\pa_{x_i}\pa_{x_j}
K^{\bb,t}_{ i,q}[\chi_{\epsilon} g]\|_{L^{\infty}}.
\end{multline}
Lemma~\ref{lem7.2} and Proposition~\ref{prop1n0} show that the terms where the
sup-norm is on the coefficients are estimated by $C\epsilon^{2-\gamma}.$
Applying Lemma~\ref{lem13.3} we see that there is a $C$ independent of
$\epsilon$ so that:
\begin{equation}
  \bbr{ \psi_{\epsilon}\tb_i(\bx) }_{\WF,0,\gamma,T}+
\bbr{\psi_{\epsilon}\sqrt{x_ix_j}a'_{ij}(\bx)}_{\WF,0,\gamma,T}\leq
C\epsilon^{2-\gamma}
\end{equation}
We get an additional order of vanishing in the second term because the
coefficients vanish to second order in the variables $\{\sqrt{x_i}\}.$ We again
use the estimate from Proposition~\ref{prop1n0} to see that
\begin{equation}
  \|\pa_{x_i}K^{\bb,t}_{ i,q}[\chi_{\epsilon}  g]\|_{L^{\infty}}+
\|\sqrt{x_ix_j}\pa_{x_i}\pa_{x_j}
K^{\bb,t}_{ i,q}[\chi_{\epsilon} g]\|_{L^{\infty}}\leq C\epsilon^{-\gamma}\|g\|_{\WF,0,\gamma,T},
\end{equation}
showing that these products in~\eqref{eqn596} and~\eqref{eqn597} are
bounded by a constant times $\epsilon^{2(1-\gamma)}\|g\|_{\WF,0,\gamma,T}.$

Altogether the right hand side of~\eqref{eqn478.00} contributes $M_N$ terms of
these types, which allows us to conclude that there is a $C$
independent of $\epsilon,$ and $T\leq T_0,$ so that
\begin{equation}\label{0N0locest}
  I+II\leq 
CM_N\epsilon^{2(1-\gamma)}\|g\|_{\WF,0,\gamma,T},
\end{equation}
whence
\begin{equation}
  \|E^{0,t}_{ i,q}g\|_{\WF,0,\gamma,T}\leq
  CM_N\epsilon^{2(1-\gamma)} \|g\|_{\WF,0,\gamma,T}.
\end{equation}
These  calculations apply at each of the $p$ points in $\Sigma_N.$ 

For each $\epsilon>0$ we let $\chi_{ i,q}$ ( $\psi_{ i,q}$ resp.) denote the
function $\chi_{\epsilon}$ ($\psi_{\epsilon}$ resp.) in the $i$th-coordinate
patch, with this choice of $\epsilon.$ The contribution of $\Sigma_N$ to the
boundary parametrix is given by
\begin{equation}
  \hQ^t_b=\sum_{j=1}^{p}\sum_{q\in F_{ j}}\psi_{ i,q}K^{\bb,t}_{ i,q}\chi_{ i,q}.
\end{equation}
We therefore have
\begin{equation}
\begin{split}
(\pa_t-L)  \hQ^t_bg&=\sum_{j=1}^{p}\sum_{q\in F_{ j}}\left[\chi_{ i,q}g+
 E^{0,t}_{ i,q}g+ E^{\infty,t}_{ i,q}g\right]\\
&=\varphi_{\epsilon} g+E^{0,t}_{\epsilon}g+E^{\infty,t}_{\epsilon}g,
\end{split}
\end{equation}
where
\begin{equation}
  \varphi_{\epsilon}=\sum_{i=1}^p\chi_{i,q},\quad
E^{0,t}_{\epsilon}=\sum_{j=1}^{p}\sum_{q\in F_{ j}} E^{0,t}_{ i,q},\quad\text{ and }
\quad
E^{\infty,t}_{\epsilon}=\sum_{j=1}^{p}\sum_{q\in F_{ j}} E^{\infty,t}_{ i,q}.
\end{equation}

The local estimate~\eqref{0N0locest} shows that there is a constant $C$ so that
for any $\epsilon>0$ we have
\begin{equation}\label{0N0est}
 \|E^{0,t}_{\epsilon}g\|_{\WF,0,\gamma}\leq 
C\epsilon^{2(1-\gamma)}\|g\|_{\WF,0,\gamma,T}.
\end{equation}
We can therefore choose $\epsilon>0$ so that
\begin{equation}
  C\epsilon^{2(1-\gamma)}=\delta.
\end{equation}
With this choice of $\epsilon$ we let
\begin{equation}
  \varphi=\sum_{i=1}^p\chi_{i,q}.
\end{equation}
For this fixed $\epsilon>0,$ the estimate
in~\eqref{0Ninflocest} shows that
\begin{equation}\label{0Ninfest}
  \|E^{\infty,t}_{\epsilon}g\|_{\WF,k,\gamma,T}\leq
  C(T,k,\gamma)\epsilon^{-\mu(k,\gamma)}\|g\|_{\WF,0,\gamma,T}, 
\end{equation} 
where $C(T,k,\gamma)\to 0$ as $T\to 0^+.$ Thus with
\begin{equation}\label{eqn642.3}
  E^t_{\epsilon}=E^{0,t}_{\epsilon}+E^{\infty,t}_{\epsilon},
\end{equation}
we have the norm estimate
\begin{equation}
  \|E^t_{\epsilon}\|_{0,\gamma}\leq [\delta+C(T,0,\gamma)\epsilon^{-\mu(0,\gamma)}]
\end{equation}
The function $\varphi$ equals $1$ in a neighborhood of $\Sigma_N,$ and we have
estimate
\begin{equation}
  \|(\pa_t-L)  \hQ^t_bg-\varphi g\|_{\WF,0,\gamma,T}\leq [\delta
+C(T,0,\gamma)\epsilon^{-\mu(0,\gamma)}]\|g\|_{\WF,0,\gamma,T}.
\end{equation}

It only remains to verify the small time localization property for the error
term.  The operator $E^t_{\epsilon}$ is built from a finite combination of
terms of the form $ GK^t\theta,$ where $G$ is a differential operator, $\theta$
is a smooth function, and $K^t$ is the heat kernel of a model operator. If
$\varphi$ and $\psi$ are smooth functions with disjoint supports, then we can
choose another smooth function $\varphi'$ so that
\begin{equation}\label{eqn13.111.00}
  \supp\varphi'\cap\supp\psi=\emptyset\text{ and }\varphi'=1\text{ on }\supp\varphi.
\end{equation}
Since $G$ is a differential operator, it is immediate that
\begin{equation}
  \varphi G K^t\theta\psi=\varphi G\varphi' K^t\theta\psi.
\end{equation}
As the supports of $\varphi'$ and $\psi$ are disjoint, it follows that
$\varphi' K^t\theta\psi$ is a family of smoothing operators tending to zero as
$T\to 0^+$ as a map from $\cC^0(P\times[0,T])$ to $\cC^k(P\times[0,T]),$ for
any $k\in\bbN.$ This completes the construction of the boundary parametrix in this case.

\subsection{Intermediate Codimension case}\label{s.intcodimHK}
Now assume that $n=M+1<\dim P,$ and that $\Sigma_{M+1}$ is the maximal
co-dimensional stratum of $bP.$ This includes the case that $n=1,$ which is the
base case needed to start the induction. 

We let $N^+\Sigma_{M+1}$ denote the inward pointing normal bundle of
$\Sigma_{M+1}.$ Since $\Sigma_{M+1}$ is the maximal codimensional stratum, the
tubular neighborhood theorem for manifolds with corners implies that there is a
neighborhood $W$ of $\Sigma_{M+1}$ in $P$ that is diffeomorphic to a
neighborhood $W_{0},$ of the zero section in $N^+\Sigma_{M+1}.$ Let
$\Psi:W_0\to W$ be such a diffeomorphism, which reduces to the inclusion map
along the zero section. We let $\Phi$ denote a $\CI$-function defined on
$N^+\Sigma_{M+1}$ so that $\Phi=1$ in a neighborhood $W_1$ of zero section and
$\Phi=0$ outside a somewhat larger neighborhood $W_2.$ We define a family of
functions $\{\Phi_{\epsilon}:\:\epsilon\in (0,1]\}$ in $\CI(W)$ by setting
\begin{equation}\label{13.114}
  \Phi_{\epsilon}(r)=\Phi\left(\epsilon^{-2}\cdot\Psi^{-1}(r)\right),
\end{equation} 
here $\epsilon^{-2}\cdot $ denotes the usual action of $\bbR_+$ on the fiber of
$N^+\Sigma_{M+1}.$

Let $\fU=\{U_j:\:j=1,\dots,p\}$ denote a
covering of a neighborhood of $\Sigma_{M+1}$ by NCC charts. The fact that
$\Sigma_{M+1}$ is the maximum codimensional stratum implies that all of these
charts have coordinates lying in $\bbR_+^n\times\bbR^m.$ Let
$$(\bx;\by)=(x_1,\dots,x_n;y_1,\dots,y_m)$$ 
be the normal cubic coordinates in a subset $U_{j},$ so that in these
coordinates $L$ is given by
\begin{multline}\label{Lnrmfrm2}
  L=\sum_{i=1}^nx_i\pa_{x_i}^2+\sum_{1\leq k,l\leq
    m}c_{kl}(\bx,\by)\pa_{y_k}\pa_{y_l}+\sum_{i=1}^nb_i(\bx,\by)\pa_{x_i}+\\
\sum_{1\leq i\neq j\leq
  n}x_ix_ja'_{ij}(\bx,\by)\pa_{x_i}\pa_{x_j}+
\sum_{i=1}^n\sum_{l=1}^mx_ib'_{il}(\bx,\by)\pa_{x_i}\pa_{y_l}+
\sum_{l=1}^md_l(\bx,\by)\pa_{y_l}.
\end{multline}
We let $L^p$ denote the sum on the first line; this is the principal part of
$L.$ 

There is a positive constant $K$ so that within the coordinate chart the
coefficient matrix $c_{kl}$ satisfies
\begin{equation}\label{2ndtgnlb00}
  K\Id_m\leq c_{kl}(\bx,\by)\leq K^{-1}\Id_m.
\end{equation}
For each point in $q\in \Sigma_{M+1,j}=\Sigma_{M+1}\cap U_j$ we could choose an affine
change of coordinates in the $\by$-variables, which we denote by $(\bx,\tby),$
so that in these variables $q\leftrightarrow (\bzero;\bzero)$ and:
\begin{equation}
  L^p\restrictedto_{q}=
  \sum_{i=1}^nx_i\pa_{x_i}^2+\sum_{1\leq k,l\leq
    m}(\delta_{kl}+\tc_{kl}(\bx,\tby))\pa_{\ty_k}\pa_{\ty_l}+\sum_{i=1}^n(b_i+\tb_i(\bx,\tby))\pa_{x_i}, 
\end{equation}
where
\begin{equation}
  \tc_{kl}(\bzero_n,\bzero_m)=\tb_i(\bzero_n,\bzero_m)=0.
\end{equation}
In light of the bounds~\eqref{2ndtgnlb00} these affine changes of variable come
from a compact subset of $GL_m\ltimes\bbR^m$ and therefore, under all these changes of
variable, the coefficients $\tc_{kl},b_i, \tb_i$ and $a'_{ij},b'_{il},d_l$
remain uniformly bounded in the $\cC^{\infty}$-topology.

In fact we do not use these changes of variables  in our construction, but simply note
that the constants in the estimates for the model operators at points $q=(\bzero;\by_q)\in
U_i,$ which we can take to be
\begin{equation}\label{modopiq}
    L_{i,q}=\sum_{i=1}^n[x_i\pa_{x_i}^2+b_i(\bzero,\by_q)\pa_{x_i}]+
\sum_{k,l=1}^mc_{kl}(\bzero,\by_q)\pa_{y_l}\pa_{y_k},
\end{equation}
are uniformly bounded.  We have
\begin{equation}\label{eqn616.0}
  L=L_{i,q}+L^r_{i,q}+\sum_{l=1}^md_l(\bx,\by)\pa_{y_l},
\end{equation}
where the residual ``second order'' part at $q$ is:
\begin{multline}\label{eqn616}
  L^r_{ i,q}=\sum_{1\leq k,l\leq
    m}c'_{kl,q}(\bx,\by)\pa_{y_k}\pa_{y_l}+\sum_{i=1}^nb'_{i,q}(\bx,\by)\pa_{x_i}+\\
\sum_{1\leq i\neq j\leq
  n}\sqrt{x_ix_j}a'_{ij}(\bx,\by)\sqrt{x_ix_j}\pa_{x_i}\pa_{x_j}
+
\sum_{i=1}^n\sum_{l=1}^m\sqrt{x_i}b'_{il}(\bx,\by)\sqrt{x_i}\pa_{x_i}\pa_{y_l}.
\end{multline}\
The coefficients of $L^r_{i,q}$ are smooth functions of $(\bx,\by),$ and
\begin{equation}
     c'_{kl,q}(\bx,\by)=c_{kl}(\bx,\by)- c_{kl}(\bzero,\by_q)\text{ and }
b'_{i,q}(\bx,\by)=b_{i}(\bx,\by)-b_{i}(0,\by_q),
\end{equation}
so that
\begin{equation}
c'_{kl,q}(\bzero,\by_q)=b'_{i,q}(0,\by_q)=0.
\end{equation}

We let $\chi(\bx,\by)$ and $\psi(\bx,\by)$ be functions in
$\cC^{\infty}_{c}(\bbR_+^{n}\times\bbR^m)$ so that
\begin{equation}
\begin{split}
&\chi\equiv 1\text{ in }
[0,4]^{m}\times (-2,2)^n\\
&\supp\chi\subset [0,9]^{m}\times (-3,3)^n,
\end{split}
\end{equation}
and
\begin{equation}
\begin{split}
&\psi\equiv 1\text{ in }
[0,16]^{m}\times (-4,4)^n\\
&\supp\psi\subset [0,25]^{m}\times (-5,5)^n,
\end{split}
\end{equation}
 With $(\bzero;\by_q)$  the coordinates of $q,$  we define
\begin{equation}
  \tchi_{ i,q}=\chi\left(\frac{\bx}{\epsilon^2},\frac{\by-\by_q}{\epsilon}\right),
\end{equation}
and
\begin{equation}
  \psi_{ i,q}=\psi\left(\frac{\bx}{\epsilon^2},\frac{\by-\by_q}{\epsilon}\right).
\end{equation}
Of course these functions depend on the choice of $\epsilon,$ but to simplify
the notation, we leave  this dependence implicit.
We let $F_{i,\epsilon}$ be the points in $U_i\cap\Sigma_{M+1}$ with coordinates
$\{(0,\epsilon\bj):\: \bj\in \bbZ^m\}.$ 

It is immediate from these definitions that 
\begin{lemma}
Every point lies in the support of at most a fixed finite number of the functions 
$\{\psi_{i,q}:\: q\in F_{i,\epsilon};\, i=1,\dots,p\}$, independently of $\epsilon$.
\label{covering}
\end{lemma}
From the definition of the sets $F_{i,\epsilon}$ it is clear
that there is a constant $S,$ independent of $\epsilon,$ so that for $r\in P$
we have the estimate
\begin{equation}\label{eqn622}
  X_{\epsilon}(r)=  \sum_{i=1}^p\sum_{q\in F_{i,\epsilon}}\tchi_{i,q}(r)\leq S.
\end{equation}
It is also clear that $X_{\epsilon}(r)\geq 1$ for $r\in\Sigma_{M+1}.$ By
choosing the neighborhoods $W_1, W_2$ (independently of $\epsilon$) used in the
definition of $\Phi_{\epsilon},$ (see~\eqref{13.114}) we can arrange to have $\Phi_{\epsilon}=1$ on
the set where $X_{\epsilon}\geq \frac 12,$ and
\begin{equation}\label{eqn623}
 \supp\Phi_{\epsilon}\subset X_{\epsilon}^{-1}([\frac{1}{16},S]).
\end{equation}
To get a partition of unity of a neighborhood of $\Sigma_{M+1},$ we replace the
functions $\{\tchi_{i,q}\}$ with
  \begin{equation}
    \chi_{i,q}=\Phi_{\epsilon}
\left[\frac{\tchi_{i,q}}{\sum_{i=1}^p\sum_{q\in F_{i,\epsilon}}\tchi_{i,q}}\right].
  \end{equation}
For any choice of $\epsilon>0,$ these functions are smooth and define a
partition of unity in a neighborhood of $\Sigma_{M+1}.$ By repeated application
of Lemmas~\ref{lem7.1} and~\ref{lem7.2}  and~\eqref{eqn623}, it follows that there is a constant
$C$ independent of $\epsilon>0$ so that
\begin{equation}\label{eqn625}
 \|\psi_{i,q}\|_{\WF,0,\gamma}+ \| \chi_{i,q}\|_{\WF,0,\gamma}\leq C\epsilon^{-\gamma}.
\end{equation}

For each $\epsilon>0$ we define a boundary parametrix by setting
\begin{equation}\label{eqn12.131.06}
  \hQ^t_b=\sum_{i=1}^p\sum_{q\in F_{i,\epsilon}}\psi_{i,q}K^{\bb,t}_{i,q}\chi_{i,q},
\end{equation}
where $K^{\bb,t}_{i,q}$ denotes the solution operator constructed above for the model
problem
$$(\pa_t-L_{i,q})u=g\quad u(r,0)=0,$$ 
with $L_{i,q}$ defined in~\eqref{modopiq}, and with
$\bb=(b_1(\bzero,\by_q),\dots,b_n(\bzero,\by_q)).$ We now consider the typical
term appearing in the parametrix.  If $g$ is a H\"older continuous function
defined in a neighborhood of $\supp\chi_{ i,q},$ then
\begin{equation}
  u_{ i,q}=\psi_{ i,q} K^{\bb,t}_{ i,q}[\chi_{ i,q} g]
\end{equation}
is well defined throughout $U_{ i}$ and can be extended, by zero, to all of
$P.$ We apply the operator to $ u_{ i,q}$ obtaining:
\begin{equation}\label{eqn440.00}
  \begin{split}
    (\pa_t-L)u_{ i}^q&=\psi_{ i,q} (\pa_t-L_{i,q}+L_{i,q}-L)K^{\bb,t}_{ i,q}[\chi_{ i,q}g]+
[\psi_{ i,q},L]K^{\bb,t}_{ i,q}[\chi_{ i,q} g]\\
&=\chi_{ i,q}g+\psi_{ i,q} (L_{ i,q}-L)K^{\bb,t}_{ i,q}[\chi_{ i,q} g]+[\psi_{ i,q},L]K^{\bb,t}_{ i,q}[\chi_{ i}^q g].
  \end{split}
\end{equation}

The estimates for the sizes of these errors will be in terms of $\|\chi_{ i,q}
g\|^{ i}_{\WF,0,\gamma,T}.$ It follows from Lemmas~\ref{lem7.1} and~\ref{lem7.2},
and~\eqref{eqn625} that there is a constant $C,$ independent of $\epsilon, T$ so
that
\begin{equation}
  \|\chi_{ i,q} g\|^{ i}_{\WF,0,\gamma,T}\leq C\epsilon^{-\gamma} \|g\|_{\WF,0,\gamma,T}
\end{equation}
There are three types of
error terms:
\begin{equation}\label{eqn13.132}
\begin{split}
  E^{\infty,t}_{ i,q}g=[\psi_{ i,q},L]K^{\bb,t}_{ i,q}[\chi_{ i,q}g],\quad
&E^{1,t}_{ i,q}g=\psi_{ i,q}\left[\sum_{l=1}^md_l(\bx,\by)\pa_{y_l}\right]K^{\bb,t}_{ i,q}[\chi_{
  i,q}g],\\
&E^{0,t}_{i,q}=\psi_{ i,q}L^r_{i,q}K^{\bb,t}_{ i,q}[\chi_{ i,q}g],
\end{split}
\end{equation}
where $L^r_{i,q}$ is given by~\eqref{eqn616}. For each $\epsilon>0$ and
$d\in\{0,1,\infty\}$ we define:
\begin{equation}
  E^{d,t}_{\epsilon}=\sum_{i=1}^p\sum_{q\in F_{i,\epsilon}}E^{d,t}_{i,q}
\end{equation}

Observe that the support of the coefficients of $[\psi_{ i,q},L]$ is disjoint
from that of $\chi_{ i,q}$ and therefore Proposition~\ref{offdiagnm} shows that
$E^{\infty,t}_{i,q}$ is a smoothing operator tending exponentially to zero as
$T\to 0^+.$ As before there are positive constants $C(T,k,\gamma)$ and
$\mu(k,\gamma)$ so that, for any $\epsilon>0,$ we have
\begin{equation}
  \|E^{\infty,t}_{ i,q}g\|_{\WF,k,\gamma,T}\leq
  C(T,k,\gamma)\epsilon^{-\mu(k,\gamma)}\|g\|_{\WF,k,\gamma,T},
\end{equation}
where $C(T,k,\gamma)=O(T)$ as $T\to 0^+.$

The error term $E^{1,t}_{ i,q}g$ produced by the tangential first derivatives is
of lower order, but more importantly, equation~\eqref{tngtestnm} shows that the
norm of this term also tends to zero as $T\to 0^+.$ Hence there is a positive
constant $C$ independent of $\epsilon$ so that
\begin{equation}
  \|E^{1,t}_{ i,q}g\|_{\WF,0,\gamma,T}\leq
  C\epsilon^{-\gamma}T^{\frac{\gamma}{2}}\|g\|_{\WF,0,\gamma,T}.
\end{equation}
Recalling Lemma~\ref{covering}, each point will lie in the support of at most $S$ 
of the functions $\psi_{i,q}$ and therefore there is a constant $C$ independent 
of $\epsilon$ so that the sum of these terms satisfies an estimate of the form
\begin{equation}\label{eqn13.136}
    \|[E^{1,t}_{\epsilon}+E^{\infty,t}_{\epsilon}]g\|_{\WF,0,\gamma,T}\leq
  SC\epsilon^{-(\mu(0,\gamma)+\gamma)}T^{\frac{\gamma}{2}}\|g\|_{\WF,0,\gamma,T}.
\end{equation}

The remaining error term is
\begin{equation}
  E^{0,t}_{ i,q}g=\psi_{ i,q} L^r_{ i,q}K^{\bb,t}_{ i,q}[\chi_{ i,q} g],
\end{equation}
which is a bounded map of $\cC^{k,\gamma}_{\WF}.$ to itself, for any
$k\in\bbN_0.$ We need to estimate both $\| E^{0,t}_{ i,q}g\|_{L^{\infty}}$ and
$\bbr{E^{0,t}_{i,q}g}_{\WF,0,\gamma}.$ The vanishing properties of the coefficients
of $L^r_{i,q},$ Proposition~\ref{prop6.2} and Lemmas~\ref{lem7.1},~\ref{lem7.2}
imply that the $L^{\infty}$-term satisfies
\begin{equation}\label{eqn649.2}
\| E^{0,t}_{ i,q}g\|_{L^{\infty}}\leq C\epsilon
\|\chi_{i,q}g\|_{\WF,0,\gamma}\leq C\epsilon^{1-\gamma}\|g\|_{\WF,0,\gamma}.
\end{equation}
The second inequality follows from Lemmas~\ref{lem7.1} and~\ref{lem7.2}. 

To estimate the H\"older semi-norm we need to consider a variety of terms, much
like those in~\eqref{eqn596} and~\eqref{eqn597}. For the case at hand we have
the terms
\begin{multline}\label{eqn635}
  \bbr{\psi_{i,q}b'_{i,q}(\bx,\by)\pa_{x_i}K^{\bb,t}_{ i,q}[\chi_{ i,q} g]}_{\WF,0,\gamma},\quad
\bbr{\psi_{i,q}c'_{kl,q}(\bx,\by)\pa_{y_k}\pa_{y_l}K^{\bb,t}_{ i,q}[\chi_{ i,q}
  g]}_{\WF,0,\gamma},\\
\bbr{\psi_{i,q}\sqrt{x_ix_j}a'_{ij}(\bx,\by)\sqrt{x_ix_j}\pa_{x_i}\pa_{x_j}K^{\bb,t}_{
    i,q}[\chi_{ i,q} g]}_{\WF,0,\gamma},\\
\bbr{\psi_{i,q}
\sqrt{x_i}b'_{il}(\bx,\by)\sqrt{x_i}\pa_{x_i}\pa_{y_l}K^{\bb,t}_{i,q}[\chi_{ i,q} g]}_{\WF,0,\gamma},
\end{multline}
each of which is estimated by using the Leibniz formula in
Lemma~\ref{lem7.1}. Lemma~\ref{lem7.2} shows that there is a constant $C$ so that the terms 
\begin{multline}
  \|\psi_{i,q}b'_{i,q}(\bx,\by)\|_{L^{\infty}}\bbr{\pa_{x_i}K^{\bb,t}_{ i,q}[\chi_{ i,q}
    g]}_{\WF,0,\gamma},\\
\|\psi_{i,q}c'_{kl,q}(\bx,\by)\|_{L^{\infty}}\bbr{\pa_{y_k}\pa_{y_l}K^{\bb,t}_{ i,q}[\chi_{ i,q}
  g]}_{\WF,0,\gamma},\\
\|\psi_{i,q}\sqrt{x_i}b'_{il}(\bx,\by)\|_{L^{\infty}}\bbr{\sqrt{x_i}\pa_{x_i}\pa_{y_l}
K^{\bb,t}_{i,q}[\chi_{ i,q} g]}_{\WF,0,\gamma}
\end{multline}
are all bounded by
\begin{equation}
  C\epsilon\bbr{\chi_{ i,q} g}_{\WF,0,\gamma}\leq C\epsilon^{1-\gamma}\|g\|_{\WF,0,\gamma}.
\end{equation}
Similarly, we see that
\begin{equation}
 \|\psi_{i,q}\sqrt{x_ix_j}a'_{ij}(\bx,\by)\|_{L^{\infty}} 
\bbr{\sqrt{x_ix_j}\pa_{x_i}\pa_{x_j}K^{\bb,t}_{i,q}[\chi_{ i,q} g]}_{\WF,0,\gamma}
\leq C\epsilon^{2-\gamma}\|g\|_{\WF,0,\gamma}.
\end{equation}

To complete the estimates for the terms in~\eqref{eqn635} we need to bound:
\begin{multline}\label{eqn654.2}
  \bbr{\psi_{i,q}b'_{i,q}(\bx,\by)}_{\WF,0,\gamma}\|\pa_{x_i}K^{\bb,t}_{ i,q}[\chi_{ i,q} g]\|_{L^{\infty}},\\
\bbr{\psi_{i,q}c'_{kl,q}(\bx,\by)}_{\WF,0,\gamma}\|\pa_{y_k}\pa_{y_l}K^{\bb,t}_{ i,q}[\chi_{ i,q}
  g]\|_{L^{\infty}},\\
\bbr{\psi_{i,q}\sqrt{x_ix_j}a'_{ij}(\bx,\by)}_{\WF,0,\gamma}\|\sqrt{x_ix_j}\pa_{x_i}\pa_{x_j}K^{\bb,t}_{
    i,q}[\chi_{ i,q} g]\|_{L^{\infty}},\\
\bbr{\psi_{i,q}
\sqrt{x_i}b'_{il}(\bx,\by)}_{\WF,0,\gamma}\|\sqrt{x_i}\pa_{x_i}\pa_{y_l}K^{\bb,t}_{i,q}[\chi_{ i,q} g]\|_{L^{\infty}}.
\end{multline}
Proposition~\ref{prop6.2} shows that for any $0<\gamma'\leq\gamma<1$ the
sup-norms appearing in~\eqref{eqn654.2} are bounded by
\begin{equation}
  C_{\gamma'}\|\chi_{ i,q} g\|_{\WF,0,\gamma'}\leq C_{\gamma'}\epsilon^{-\gamma'}\| g\|_{\WF,0,\gamma'}.
\end{equation}
We therefore fix a $0<\gamma'\leq \gamma$ so that
\begin{equation}
  \gamma'+\gamma<1.
\end{equation}
To complete this estimate we only need to bound the H\"older semi-norms of
the coefficients. Lemma~\ref{lem13.3} shows that all of these terms are bounded
by $C\epsilon^{1-\gamma},$ for a constant independent of $\epsilon>0.$ Together
these estimates show that there is a constant $C$ independent of $\epsilon, i$
and $q$ so that
\begin{equation}
  \|E^{0,t}_{i,q}g\|_{\WF,0,\gamma,T}\leq C\epsilon^{1-\gamma-\gamma'}\|g\|_{\WF,0,\gamma,T}.
\end{equation}
Once again we use the fact that for any point in $P$ at most a fixed finite
number of terms in the sum defining $E^0_{\epsilon}$ is non-zero to conclude
that there an constant $S$ so that
\begin{equation}
  \|E^{0,t}_{\epsilon}g\|_{\WF,0,\gamma,T}\leq SC\epsilon^{1-\gamma-\gamma'}\|g\|_{\WF,0,\gamma,T}.
\end{equation}

We can therefore choose $\epsilon>0$ so that
\begin{equation}
  SC\epsilon^{1-\gamma-\gamma'}\leq\delta.
\end{equation}
With this fixed choice of $\epsilon$ we let
\begin{equation}\label{12.152.06}
  \varphi=\sum_{i=1}^p\sum_{q\in F_{i,\epsilon}}\chi_{i,q};
\end{equation}
this function equals $1$ in a neighborhood of $\Sigma_{M+1}.$ Using the
definition for $\hQ^t_b$ with this choice of $\epsilon,$ we see that, with
\begin{equation}\label{eqn684.3}
  E^t_{\epsilon}=E^{0,t}_{\epsilon}+E^{1,t}_{\epsilon}+E^{\infty,t}_{\epsilon}
\end{equation}
we have that
\begin{equation}
  (\pa_t-L)\hQ^t_b-\varphi g=E^t_{\epsilon}g,
\end{equation}
and therefore
\begin{equation}
  \|(\pa_t-L)\hQ^t_b-\varphi g\|_{\WF,0,\gamma,T}\leq [\delta
  +C\epsilon^{-\mu(0,\gamma)}T^{\frac{\gamma}{2}}]\|g\|_{\WF,0,\gamma,T}.
\end{equation}
This estimate completes the construction of the boundary parametrix for the
case of arbitrary maximal co-dimension between $1$ and $\dim P.$ 

It only remains to verify the small time localization property for the error
term.  As before, the operator $E^t_{\epsilon}$ is built from a finite
combination of terms of the form $ GK^t\theta,$ where $G$ is a differential
operator, $\theta$ is a smooth function, and $K^t$ is the heat kernel of a
model operator. Precisely the same argument as given in maximal codimension
case shows that if $\varphi$ and $\psi$ are smooth functions with disjoint
supports, then $\varphi G K^t\theta\psi$ is a family of smoothing operators
tending to zero as $T\to 0^+$ as a map from $\cC^0(P\times[0,T])$ to
$\cC^j(P\times[0,T]),$ for any $j\in\bbN.$ This in turn completes the proof, in
case $k=0,$ of the existence of a solution to the inhomogeneous problem up to a
time $T_0>0.$ In the next section we show how to use this result to demonstrate
the existence of solutions to the Cauchy problem, which in turn allows us to
prove a global in time existence result for the inhomogeneous problem.

\section{Solution of the homogeneous problem}
Assuming the existence of a solution to the inhomogeneous problem for data in
$\cC^{k,\gamma}_{\WF}(P\times [0,T_0])$ for a fixed $T_0>0,$ a very similar
parametrix construction is used to show the existence of $v,$ the solution for
all time, to the homogeneous Cauchy problem, with initial data
$f\in\cC^{k,2+\gamma}_{\WF}(P).$ Assume that $\cQ^t,$ the solution operator for
the inhomogeneous problem, is defined for $t\in [0,T_0].$ As above, we use
Proposition~\ref{prop6.1} to build a boundary parametrix for the homogeneous
Cauchy problem, which we then glue to the exact solution operator for $P_U.$
This gives an operator 
\begin{equation}
\begin{split}
\hQ^t_0:\cC^{k,2+\gamma}_{\WF}(P)&\longrightarrow \cC^{k,2+\gamma}_{\WF}(P\times
[0,\infty))\\
(\pa_t-L)\hQ^t_0f&=E^t_0f\text{ and }\hQ^t_0f\restrictedto_{t=0}=f,
\end{split}
\end{equation}
where
\begin{equation}
E^t_0 :\cC^{k,2+\gamma}_{\WF}(P)\longrightarrow \cC^{k,\gamma}_{\WF}(P\times
[0,\infty))
\end{equation}
is a bounded map. A slightly stronger statement is true.
\begin{proposition}
Given $\delta>0,$ we can make 
\begin{equation}
  \lim_{T\to 0+}\|E^t_0 \|_{\cC^{k,2+\gamma}_{\WF}(P)\to \cC^{k,\gamma}_{\WF}(P\times
[0,\infty))}\leq \delta.
\end{equation}
\end{proposition}

The existence of the operator $\hQ^t_0$ is a simple consequence of the
induction hypothesis and the properties of the solution operators for the model
homogeneous Cauchy problems established in Proposition~\ref{prop6.1}. Suppose
that the maximal codimension of $bP$ is $M.$ Let
$\{W_j:\: j=1,\dots,J\}$ be an NCC cover of $\Sigma_M,$ and $W_0$ a relatively
compact subset of $\Int P,$ which covers $P\setminus \cup_{j=1}^J W_j,$ and has
a smooth boundary.  Let $\{\varphi_j\}$ be a partition of unity subordinate to
this cover of $P,$ and $\{\psi_j\}$ smooth functions of compact support in
$W_j,$ with $\psi_j\equiv 1$ on $\supp\varphi_j.$ For each $j\in\{1,\dots, J\}$
let $\hQ^t_{j0}$ be the solution operator for the homogeneous Cauchy problem
defined by the model operator in $W_j.$ As above, we let $\hQ^t_{00}$ be the
exact solution operator for the Cauchy problem $(\pa_t-L)u=0$ on $W_0$ with
Dirichlet data on $bW_0\times [0,\infty).$ We then define
\begin{equation}
  \hQ^t_0=\sum_{j=0}^{J}\psi_j\hQ^t_{j0}\varphi_j.
\end{equation}
From the mapping properties of the component operators it follows that, for any
$0<\gamma<1,$ and $k\in\bbN_0,$ this operator defines bounded maps:
\begin{equation}
\begin{split}
\hQ^t_{0}:\cC^{k,\gamma}_{\WF}(P)&\longrightarrow \cC^{k,\gamma}_{\WF}(P\times
[0,\infty))\\
  \hQ^t_{0}:\cC^{k,2+\gamma}_{\WF}(P)&\longrightarrow \cC^{k,2+\gamma}_{\WF}(P\times
[0,\infty)).
\end{split}
\end{equation}
As $t\to 0^+,$ the operator $\hQ^t_0$ tends strongly to the identity, with
respect to the topologies $\cC^{k,\tgamma}_{\WF}(P),$
$\cC^{k,2+\tgamma}_{\WF}(P)$ respectively, for any $\tgamma<\gamma.$ 
 
If we set $\tQ^t_{0}f=\cQ^tE^t_0f,$ and
$\cQ_0^tf=(\hQ^t_0-\tQ^t_{0})f,$ then
\begin{equation}
  \begin{split}
\cQ^t_0:\cC^{k,2+\gamma}_{\WF}(P)&\longrightarrow \cC^{k,2+\gamma}_{\WF}(P\times
[0,T_0])\text{ is bounded}\\
(\pa_t-L)\cQ^t_0f&=0.
\end{split}
\end{equation}
For any $0<\tgamma<\gamma,$ the solution $\cQ^t_0f$ tends to $f$ in
$\cC^{k,2+\tgamma}_{\WF}(P).$ From the induction hypothesis and the
properties of the boundary terms this is certainly true of $\hQ^t_0f.$
To treat the correction term we observe that $\cQ^t$ defines a bounded
map from $\cC^{k,\gamma}_{\WF}(P\times [0,T])$ to
$\cC^{k,2+\gamma}_{\WF}(P\times [0,T]).$ For a fixed $\delta>0,$ by
constructing the partition of unity $\{\varphi_j\}$ as in
Section~\ref{ss.bndprmtrx}, and choosing $\epsilon>0$ sufficiently
small, we can arrange to have
$$\lim_{T\to 0^+}\|E^t_0f\|_{\WF,k,\gamma,T}\leq\delta\|f\|_{\WF,k,2+\gamma}.$$
Hence, for any $\delta>0,$
\begin{equation}
  \lim_{t\to 0^+}\|\cQ^tf-f\|_{\WF,k,2+\tgamma}\leq C\delta\|f\|_{\WF,k,2+\gamma}.
\end{equation}
 
To show that the solution to the homogeneous problem exists for all $t>0,$ we
observe that the time of existence $T_0$ already obtained is independent of the
initial data, and there is a constant $C$ so that, with $v=\cQ^t_0f,$
\begin{equation}
  \|v(\cdot,T_0)\|_{\WF,k,2+\gamma}\leq C\|f\|_{\WF,k,2+\gamma}.
\end{equation}
We can therefore apply this argument again, with data $v(\cdot,T_0)$ specified
at $t=T_0,$ to obtain a solution on $[0,2T_0].$ We have the same estimate on
$[0,2T_0]$ with $C$ replaced by $2C.$ This can be repeated \emph{ad libitum} to
show that there is a solution $v$ to the homogeneous Cauchy problem, belonging
to $\cC^{k,2+\gamma}_{\WF}(P\times [0,T]),$ for any $T>0,$ which satisfies the
estimate
\begin{equation}
   \|v\|_{\WF,k,2+\gamma,T}\leq C(1+T)\|f\|_{\WF,k,2+\gamma}. 
\end{equation}

To verify that $\cQ^t_0$ satisfies the small time localization property
(condition (3) in the induction hypothesis) we recall that
$\cQ^t_0=\hQ^t_0-\cQ^tE^t_0.$ The induction hypothesis and the properties of
the model heat kernels show that $\hQ^t_0$ has this property. We have
established this for the operator $\cQ^t.$ The error term is again of the form
$GK^t_0\theta,$ where $G$ is a differential operator and $K^t_0$ is either a model
heat kernel, or the heat kernel from the interior. As before, if $\varphi$ and
$\psi$ have disjoint support, then we can choose $\varphi'$
satisfying~\eqref{eqn13.111.00}. From this it is immediate that, as maps from
$\cC^{k,2+\gamma}_{\WF}(P)$ to $\cC^{k,\gamma}_{\WF}(P),$ the operators 
\begin{equation}
  \varphi G K^t_0\theta\psi= \varphi G\varphi' K^t_0\theta\psi
\end{equation}
have the small time localization property. Using the arguments in the proof of
Lemma~\ref{lemsmtm} it follows easily that, as maps from
$\cC^{k,2+\gamma}_{\WF}(P)$ to itself the operator $\cQ^tE^t_0$ also has the small
time localization property.

This completes the proof of the following theorem,
which is part of Theorem~\ref{thm13.1}, in the $k=0$ case.
\begin{theorem}\label{thm13.3} Let $P$ be a manifold with corners and $L$ a
  generalized Kimura diffusion operator defined on $P.$ There is an operator
  \begin{equation}
    \cQ^t_0:\cC^{k,2+\gamma}_{\WF}(P)\longrightarrow
    \cC^{k,2+\gamma}_{\WF}(P\times [0,\infty)),
  \end{equation}
so that
\begin{equation}
  (\pa_t-L)\cQ^t_0f=0\text{ for all }t>0,
\end{equation}
moreover, for any $\tgamma<\gamma,$ $\cQ^t_0f$ converges to $f$ in
$\cC^{k,2+\tgamma}_{\WF}(P).$ There are constants $C_{k,\gamma}$ so that
\begin{equation}\label{eqn13.68.1}
  \|\cQ^t_0f\|_{\WF,k,2+\gamma,T}\leq C_{k,\gamma}(1+T) \|f\|_{\WF,k,2+\gamma}.
\end{equation}
\end{theorem}
Contingent upon verification of the convergence of the Neumann series
for $k\in\bbN,$ and the proof of Theorem~\ref{dblthm}, this completes
the proof of Theorem~\ref{thm13.3}

This theorem has a corollary about the point spectrum of $L$ on the spaces
$\cC^{k,2+\gamma}_{\WF}(P).$
\begin{corollary}\label{cor13.spec} If there is a non-trivial solution
  $f\in\cC^{0,2+\gamma}_{\WF}(P)$ to the 
  equation $(L-\mu)f=0,$ then $\Re\mu\leq 0.$  
\end{corollary}
\begin{remark} This extends the consequence of the maximum principle in
  Proposition~\ref{uniquenesselliptic.1} from $\mu\in (0,\infty)$ to $\mu$ in
  the right half plane.
\end{remark}
\begin{proof} Suppose there were a solution $f_{\mu}\neq 0,$ for a complex number $\mu$
  with $\Re\mu>0.$ The unique solution to the initial value problem
  $(\pa_t-L)v=0,$ with $v(x,0)=f_{\mu}(x),$ would be $v(x,t)=e^{\mu t}f_{\mu}(x).$
  The norm of this solution grows exponentially, which
  contradicts~\eqref{eqn13.68.1}.
\end{proof} 

We also observe that the solution of the homogeneous problem can be used to
extend the time of existence for the inhomogeneous problem. Contingent upon
proving the convergence of the Neumann series for $(\Id+E^t)^{-1},$ we have
proved the existence of a solution, $u\in \cC^{k,2+\gamma}_{\WF}(P\times [0,T_0])$ to
  \begin{equation}\label{13.156.1}
(\pa_t-L)u=g\in\cC^{k,\gamma}_{\WF}(P\times [0,T])\text{ with }u(w,0)=0,
  \end{equation}
where we assume that $T>T_0.$  We now let $v_1$ denote the solution to the
Cauchy Problem with initial data $u(w,T_0)\in\cC^{k,2+\gamma}_{\WF}(P),$ which
exists on the interval $[0,T_0],$ and let $u_1$ denote the solution
to~\eqref{13.156.1}, with $g$ replaced by $g(w,t+T_0).$ We see that setting
\begin{equation}
  u(w,t)=v_1(w,t-T_0)+u_1(w,t-T_0)\text{ for }t\in [T_0,2T_0],
\end{equation}
extends $u$ as a solution of~\eqref{13.156.1} to the interval $[0,2T_0].$ This
process is repeated $n$ times until $nT_0\geq T,$ or infinitely often if
$T=\infty.$ It is clear that $u\in\cC^{k,2+\gamma}_{\WF}(P\times [0,T]),$ with
norm growing at most linearly in $T.$ 

To complete the proof of Theorem~\ref{thm13.1} we need to prove~\ref{dblthm},
the convergence of the Neumann series for $(\Id+E^t)^{-1}$ in the topologies
defined by $\cC^{k,\gamma}_{\WF}(P\times [0,T_0]),$ for $k>0.$
Theorem~\ref{dblthm} is proved in the next section, and the higher regularity
is established at the end of this chapter.

\section{Proof of the Doubling Theorem}

Let $P$ be a manifold with corners up to codimension $M$ and $L$ a
generalized Kimura diffusion operator on $P$. Let $\Sigma = \Sigma_M$
denote the corner of maximal codimension $M$. This is a closed
manifold without boundary; for simplicity we assume here that it is
connected, although this is not important.  We first examine the
geometry of $P$ near $\Sigma$ and use this to indicate how to perform
the doubling construction for $P$ itself. Once we have accomplished
this, we show how to extend $L$ to an operator of the same type on the
doubled space.

A key property of manifolds with corners is that $\Sigma$ possesses a
neighborhood $\tilde{U}$ which is diffeomorphic in the category of
manifolds with corners to a bundle over $\Sigma$, where each fiber is
the `positive' unit ball $B_+^M = \{x \in \RR^M: x_j \geq 0\ \forall\,
j,\ ||x|| < 1\}$ in the positive orthant in $\RR^M$.  Indeed, the
existence of this fibration is just the correct global version of the
fact that near any point $q \in \Sigma$ there is an adapted coordinate
chart $(x_1, \ldots, x_M, y_1, \ldots, y_\ell)$ for $P$ with each $x_j
\in [0,1)$ and $y_i \in (-1,1)$. The point we do not belabor is that
one can choose a coherent set of coordinate charts of this type so
that in the overlaps of these charts, the fibers $\{y =
\mbox{const.}\}$ are the same and the transition maps induce
diffeomorphisms of the positive orthant fibers.  In fact, we need a
slightly more refined version of this. Use polar coordinates $0 \leq r
< 1$ and $\omega \in S^M_+ = \{x \in B^M_+: ||x|| = 1\}$ to identify
each fiber with a truncated cone $C_1(S^{M-1}_+)$. Then it is possible
to choose the atlas of coordinate charts so that the transition maps
preserve the radial coordinate $r$. In other words, each hypersurface
$\{r = \mbox{const.}\}$ is globally defined, and is itself a manifold
with corners up to codimension $M-1$.  In particular, set $\Sigma^o =
\{r = 1\}$. Note that $\Sigma^o$ is the total space of a fibration
over $\Sigma$ with fiber $S^{M-1}_+$.

We next define the doubled space $\widetilde{P}$. Let $P^o$ denote the
open manifold with corners $P \setminus \Sigma$. As a set, define
\[
\widetilde{P} = \left((-P^o) \sqcup P^o \sqcup (-1,1) \times \Sigma^o\right) / \sim,
\]
where $-P^o$ denotes $P^o$ with the opposite orientation. The
identification is the obvious one between $P^o \cap \calU \cong (0,1)
\times \Sigma^o$ and the corresponding portion of the cylinder, with
the analogous identification between $-P^o$ and the other side of the
cylinder.  This space has the structure of a smooth manifold with
corners only up to codimension $M-1$.

For the second step of the proof, we must define an extension of the
operator $L$ to $\widetilde{P}$.  It is most convenient now to express
the restriction of $L$ to the neighborhood $\calU$ of $\Sigma$ in
polar coordinate form. For this we recall that in these coordinates,
\[
\del_{x_j} = \omega_j \del_r + \frac{1}{r} V_j,
\]
where $V_j$ is tangent to each hypersurfaces $r = \mbox{const.}$ and
transversal to $\{\omega_i = 0\}$.  On the other hand, each
$\del_{y_i}$ lifts to a vector field of precisely the same
form. Therefore,
\begin{multline*}
L =  r \del_r^2 + \sum_{i=1}^M  \left(\frac{1}{r} \omega_i V_i^2 -  \frac{1}{r} \omega_i^2 V_i + \omega_i V_i(\omega_i) \del_r\right)+ 
\sum c_{kl}' \del_{y_k} \del_{y_l}  \\ + \sum_{i \neq j} a_{ij}' \left( \omega_i^2 \omega_j^2 (r\del_r)^2 + \omega_i \omega_j V_i V_j
+ \omega_i V_i (\omega_j^2) r\del_r + \omega_i V_i(\omega_j) V_j \right) \\ + 
\sum b_{il}' \left(r \omega_i^2 \del_r \del_{y_l} + \omega_i V_i \del_{y_l} \right) + \sum d_l \del_{y_l}.
\end{multline*}
The coefficients $a_{ij}'$, $b_{il}'$, $c_{kl}'$, $d_l$ are smooth in $(y,r,\omega)$.  Notice that the first term ($r\del_r^2$ 
and the operator in the first parenthetic expression are both homogeneous of degree $-1$ and odd in $r$ and are 
independent of $y$.  All of the other operators are homogeneous of degree $0$ and even in $r$ provided we neglect the smooth
dependence of their coefficients in $r$.  We can obviously regard this as an operator on the cylinder, at least away from $r=0$,
so we must simply define a modification of the coefficients which extends smoothly and in the same class
of Kimura-type operators across $r=0$.  Recall that we wish to make this modification in any fixed but arbitrarily
small region $|r| < \eta$.  To this end, choose a smooth nonnegative cutoff function $\chi(r)$ which equals $1$ in
$r \geq \eta$ and vanishes when $r \leq \eta/2$. Now replace $a_{ij}'$, for example, by 
\[
a_{ij}'' := \chi(r) a_{ij}'(y,r,\omega) +(1-\chi(r))a_{ij}'(y,0,\omega),
\]
and similarly for all the other coefficients. These modified terms are now exactly homogeneous of degree $0$ in $r \leq \eta/2$
and extend by even reflection across $r=0$.  It remains only to define the extensions of the first two terms.  For this,
let $\rho(r)$ be a smooth function defined when $|r| < 1$ with the following properties:  $\rho(r) = \rho(-r)$, 
$\rho(r) \geq \eta/4$ for all $|r| < 1$, $\rho(r) = r$ when $|r| \geq \eta$ and $\rho(r) \leq \eta$ for $|r| < \eta$. 
We then replace these first two terms in $L$ by 
\[
\rho(r) \del_r^2 + \sum_{i=1}^M  \left(\frac{1}{\rho(r)} \omega_i V_i^2 -  \frac{1}{\rho(r)} \omega_i^2 V_i + \omega_i V_i(\omega_i) 
\del_r\right).
\]

We have now defined the full extension of $L$ to an operator
$\widetilde{L}$ of Kimura type on the doubled space
$\widetilde{P}$. This completes the proof of Theorem~\ref{dblthm}.

\section{The Weak Resolvent and $\cC^0$-semi-group}
The existence of a solution to the Cauchy problem, with initial data in
$\cC^{0,2+\gamma}_{\WF}(P)$ suffices to establish the existence of a
contraction semi-group on $\cC^0(P),$ generated by the $\cC^0$-graph closure of
$L$ acting $\cC^2_{\WF}(P).$ Though these results suffice to establish the uniqueness
of the solution to the SDE associated to $L$ and therefore the existence of a
strong Markov process with support in $P$, they are not optimal as regards the
smoothing properties of the resolvent $(\mu-L)^{-1}.$ We revisit this question
in the following section.

If $f\in\cC^{0,2+\gamma}(P)$ then Theorem~\ref{thm13.3} shows that there is a
unique solution $v\in\cC^{0,2+\gamma}(P\times [0,\infty))$ to the initial value
problem
\begin{equation}
  (\pa_t-L)v=0\text{ with }v(\cdot,0)=f.
\end{equation}
The maximum principle shows that
\begin{equation}
  \|v\|_{L^{\infty}(P\times [0,\infty))}\leq \|f\|_{L^{\infty}(P)},
\end{equation}
and the theorem gives the estimate
\begin{equation}
  \|v\|_{\WF,0,2+\gamma,T}\leq C(1+T)\|f\|_{\WF,0,2+\gamma}.
\end{equation}

These estimates easily imply that, so long as $\Re\mu >0,$ the limit
\begin{equation}
  \lim_{\epsilon\to 0^+}\int\limits_{\epsilon}^{\frac{1}{\epsilon}}v(\cdot,t)e^{-\mu t}dt
\end{equation}
exists as both a $\cC^0(P)$- and a $\cC^{0,2+\gamma}_{\WF}(P)$-valued
integral. We denote this limit by $R(\mu)f.$ The estimates on $v$ given above
imply that
\begin{equation}
\begin{split}
  \|R(\mu) f\|_{L^{\infty}}&\leq \frac{1}{\Re\mu}\|f\|_{L^{\infty}}\\
\|R(\mu) f\|_{\WF,0,2+\gamma}&\leq
C\frac{1+\Re\mu}{[\Re\mu]^2}\|f\|_{\WF,0,2+\gamma}.
\end{split}
\end{equation}

Using the same integration by parts argument as was used in
section~\ref{s.1dholoext} we establish that
  \begin{equation}
    (\mu-L)R(\mu) f=f.
  \end{equation}
  The maximum principle shows that the operator $L$ with domain
  $\cC^2_{\WF}(P),$ considered as an unbounded operator on $\cC^0(P),$ is
  dissipative, see Lemma \ref{lem3.0.7.05}.  As $\cC^{0,2+\gamma}_{\WF}(P)$ is
  a dense subset of $\cC^0(P),$ we can apply a theorem of Lumer and Phillips,
  see~\cite{LumerPhillips}, to conclude the existence of a $\cC_0$-semi-group of
  operators $e^{tL}:\cC^0(P)\to \cC^0(P),$ with domain given by the
  $\cC^0$-graph closure of $(L,\cC^{0,2+\gamma}_{\WF}(P)).$ The maximum
  principle implies that this semi-group is actually actually contractive.

  This establishes, for example, the uniqueness of the solution to the
  martingale problem, supported on $\cC^0([0,\infty);P)$ and the
  uniqueness-in-law for the solution to the SDE formally defined by
  this second order operator.  The fact that the paths of this process
  are confined almost surely, to $P,$ follows using an argument like
  that  in~\cite{CerraiClement1,CerraiClement2,CerraiClement3}. We
  will return to these questions in a later publication.

\section{Higher Order Regularity}\label{s.highereg}
In the earlier sections of this chapter we constructed a boundary parametrix with an error term
$E^t_{\epsilon}$ defined in~\eqref{eqn642.3} or~\eqref{eqn684.3}. These operators
define bounded maps from $\cC^{k,\gamma}_{\WF}(P\times [0,T])$ to itself for
any $k\in\bbN_0$ and $0<\gamma<1.$ To complete the proof of
Theorems~\ref{thm13.1} and~\ref{thm13.3} we need only establish the convergence
of the Neumann series for $(\Id+E^t_{\epsilon})^{-1}$ in the operator norm topology
defined by $\cC^{k,\gamma}_{\WF}(P\times [0,T_0]),$ for some $\epsilon>0$ and
$T_0>0.$ We accomplish this by using a general result about the convergence of
Neumann series in higher norms proved in~\cite{EpNeumSeries}. We begin by
recalling the main result of that paper.

Suppose that we have a ladder of Banach spaces $X_0\supset X_1\supset
X_2\supset\cdots,$ with norms $\|\cdot\|_k,$ satisfying
\begin{equation}
  \|x\|_{k-1}\leq \|x\|_{k}\text{ for all }x\in X_k.
\end{equation}
\begin{theorem}\label{thm14.0.1} Fix any $K \in \bbN$. 
Assume that $A$ is a linear map so that $AX_k\subset X_k$ for every
$k\in\bbN_0,$ and that there are non-negative constants
$\{\alpha_j:\:j=0,1,\dots,K\}$ and $\{\beta_j:\:j=1,\dots,K\},$ with
\begin{equation}
\alpha_j<1\quad\text{ for }0\leq j\leq K,
\end{equation}
for which we have the estimates:
\begin{equation}
\begin{split}
  \|Ax\|_{0}&\leq \alpha_0\|x\|_0\text{ for }x\in X_0,\text{ and }\\
\|Ax\|_{k}&\leq \alpha_k\|x\|_k+\beta_k\|x\|_{k-1}\text{ for }x\in X_k.
\end{split}
\end{equation}
In this case the Neumann series
\begin{equation}
  (\Id-A)^{-1}=\sum_{j=0}^{\infty}A^j
\end{equation}
converges in the operator norm topology defined by $(X_k,\|\cdot\|_k)$ for all
$k\in\{0,\dots,K\}.$
\end{theorem}

 To apply this theorem we need to show that for any $K\in\bbN$ and
 $1<\gamma<0,$ we can choose $0<\epsilon,\, 0<T_0$ so that there are constants 
$\{\alpha_0,\dots,\alpha_K\}$ and $\{\beta_0,\dots,\beta_K\}$ with
$\beta_0=0,$ 
$\alpha_j<1,$ for $0\leq j\leq K,$ and we have the estimates
\begin{equation}\label{hgordestind}
  \|E^t_{\epsilon} g\|_{\WF,k,\gamma,T_0}\leq
\alpha_k\|g\|_{\WF,k,\gamma,T_0}+\beta_k\|g\|_{\WF,k-1,\gamma,T_0}
\text{ for }0\leq k\leq K.
\end{equation}
Recalling the definition of the norms on the spaces
$\cC^{k,\gamma}_{\WF}(P\times [0,T])$ and $\cC^{k,2+\gamma}_{\WF}(P\times [0,T]),$
we see that the proofs of such estimates follow quite easily from what is done
in Chapter~\ref{exstsoln0}. Equivalent norms can be defined inductively by starting at
$k=0$ with the definitions in~\eqref{eqn137.3} and~\eqref{eqn146.3} and then
setting
\begin{equation}
\begin{split}
  \|g\|_{\WF,k,\gamma,T}&=\|g\|_{\WF,k-1,\gamma,T}+
\sup_{|\balpha|+|\bbeta|+2l=k}\|\pa_t^l\pa_{\bx}^{\balpha}\pa_{\by}^{\bbeta}g\|_{\WF,0,\gamma,T}\\
\|g\|_{\WF,k,2+\gamma,T}&=\|g\|_{\WF,k-1,2+\gamma,T}+
\sup_{|\balpha|+|\bbeta|+2l=k}\|\pa_t^l\pa_{\bx}^{\balpha}\pa_{\by}^{\bbeta}g\|_{\WF,0,2+\gamma,T}.
\end{split}
\end{equation}

The operators appearing in the sum that defines the boundary
contributions to $E^t_{\epsilon}$ are of the form
\begin{equation}
  GK^{\bb,t}_{i,q}\chi_{\epsilon},
\end{equation}
where $G$ is a differential operator. From the form of this operator
it is clear that we can regard it as acting on functions with support
in a compact subset of the coordinate chart, independent of
$\epsilon.$ This allows the application of the higher order estimates
proved in Chapters~\ref{chap.1ddegen_ests}--~\ref{s.genmod} with
constants that are independent of $\epsilon.$ The higher order
estimates for the contributions from the interior are covered by the
induction hypothesis.

The part of the estimate for $\|E^t_{\epsilon}g\|_{\WF,k,\gamma,T}$
which cannot be subsumed into a large multiple of
$\|E^t_{\epsilon}g\|_{\WF,k-1,\gamma,T}$ will be called
\begin{equation}
\|E^t_{\epsilon}g\|_{\WF,k,\gamma,T} \ \mbox{rel}\ \|E^t_{\epsilon}g\|_{\WF,k-1,\gamma,T}.
\end{equation}
This arises only from terms of the form
\begin{equation}
  \|\pa_t^l\pa_{\bx}^{\balpha}\pa_{\by}^{\bbeta}E^t_{\epsilon}g\|_{\WF,0,\gamma,T},]\text{
    where }2l+|\balpha|+|\bbeta|=k.
\end{equation}
The structure of the operators that make up $E^t_{\epsilon}$ shows that the parts of
these terms that cannot be estimated by a multiple of
$\|E^t_{\epsilon}g\|_{\WF,k-1,\gamma,T}$ arise from one of two sources. The
simpler terms to estimate are of the form:
\begin{equation}\label{eqn14.8}
   \|\tE^t_{\epsilon}\pa_t^l\pa_{\bx}^{\balpha}\pa_{\by}^{\bbeta}g\|_{\WF,0,\gamma,T},
\end{equation}
where $\tE^t_{\epsilon}$ is the error term in the parametrix construction for a
generalized Kimura diffusion operator $\tL_{\balpha}$ derived in a
straightforward manner from $L.$  The other ``new'' terms arise from
$\bx$-derivatives being applied to the coefficients of terms in
$E^t_{\epsilon}$ involving $x_j\pa_{x_j},$ $x_{i}\pa_{x_i}\pa_{y_l}$ and
$x_{i}x_j\pa_{x_i}\pa_{x_j}.$ These terms are not of lower order, but applying
a derivative to the coefficients of one of these terms leaves one less
derivative to apply to $g.$ Terms of the type appearing in~\eqref{eqn14.8} are
controlled by choosing a small $\epsilon>0,$ whereas this latter type of term
is controlled by taking $T_0$ sufficiently small.

\subsection{The 1-Dimensional Case}
We explain this first in the 1-dimensional case, where $P$ is the
interval $[0,1].$ The operator takes the form:
$L=x(1-x)\pa_x^2+b(x)\pa_x,$ where with $b(x)\pa_x$ inward pointing at
each boundary component. We can introduce coordinates $x_0,$ $x_1,$
respectively, so that $j\leftrightarrow x_j=0,$ $j=0,1$ and, in these
coordinates:
\begin{equation}
  L=x_j\pa_{x_j}^2+(b_j+\tb_j(x))\pa_{x_j},\text{ where }b_j\geq 0\text{ and }\tb(0)=0.
\end{equation}
We let $L^b=x\pa_{x}^2+b\pa_{x}$ denote the model operators, and
$K^{b}_t$ the solution operators for $(\pa_t-L^b)u=g,$ $u(x,0)=0.$ The boundary
parametrix has the form
\begin{equation}
  \hQ^t_{b\epsilon}=\sum_{j=0}^1\psi_{\epsilon}(x_j)K^{b_j}_t\varphi_{\epsilon}(y_j).
\end{equation}
Here $\varphi(x)$ is a smooth function equal to $1$ in $[0,\frac 18],$ and
supported in $[0,\frac 14],$ and $\psi$ is a smooth function equal to $1$ in $[0,\frac 12],$ and
supported in $[0,\frac 34].$ As usual $f_{\epsilon}(x)=f(x/\epsilon^2).$ We 
observe that for any smooth function $\theta$
\begin{equation}
  [L,\theta]u=2x(1-x)\pa_x\theta\pa_xu+[b(x)\pa_x\theta+x(1-x)\pa_x^2\theta]u,
\end{equation}
which consists entirely of lower order terms, and
\begin{equation}
  (L-\pa_t)\hQ^t_{b\epsilon}=\sum_{j=0}^1\left[\varphi_{j,\epsilon}+([\psi_{j,\epsilon},L]+
\psi_{j,\epsilon}(\tb_j(x_j)\pa_{x_j}))K_t^{b_j}\varphi_{j,\epsilon}\right].
\end{equation}
The error terms are
\begin{equation}
\begin{split}
  E^{\infty,t}_{\epsilon}&=[\psi_{0,\epsilon},L]K_t^{b_0}\varphi_{0,\epsilon}+[\psi_{1,\epsilon},L]K_t^{b_1}\varphi_{1,\epsilon}\\
E^{0,t}_{\epsilon}&=[\psi_{0,\epsilon}(\tb_0(x_0)\pa_{x_0})]K_t^{b_0}\varphi_{0,\epsilon}+
[\psi_{1,\epsilon}(\tb_1(x_1)\pa_{x_1})]K_t^{b_1}\varphi_{1,\epsilon}.
\end{split}
\end{equation}
Together $E^{t}_{b\epsilon}=E^{0,t}_{\epsilon}+ E^{\infty,t}_{\epsilon}.$ 

We want to give an estimate $\|E_{b\epsilon}^{t}g\|_{\WF,k,\gamma,T}$ of the form
\begin{equation}\label{eqn14.14}
  \|E_{b\epsilon}^{t}g\|_{\WF,k,\gamma,T}\leq
\alpha_k\|g\|_{\WF,k,\gamma,T}+\beta_k\|g\|_{\WF,k-1,\gamma,T},
\end{equation}
where $\alpha_k<1.$ The new terms in going from $k-1$ to $k$ are of the form
\begin{equation}
  \|\pa_t^m\pa_x^lE_{b\epsilon}^{t}g\|_{\WF,0,\gamma,T},
\end{equation}
where $2m+l=k.$ Any derivatives that fall onto the coefficients of
$K^{b_j}_t\varphi_{j,\epsilon}g,$ other than $\tb_j(x_j),$ will lead to terms
that can be estimated by multiples (possibly depending on $\epsilon$) of
$\|g\|_{\WF,k-1,\gamma,T};$ which are of no consequence. 
From Lemma~\ref{lem3.2new} it follows that:
\begin{equation}
  \pa_t^m\pa_x^lK_t^{b}g\equiv K_t^{b+l}[L^{m}_{b+l}\pa_y^lg] +
\sum_{q=0}^{m-1}L_{b+l}^q\pa_t^{m-q-1}g.
\end{equation}
We write that
\begin{equation}
  \pa_t^m\pa_x^lK_t^{b}g\equiv K_t^{b+l}[L^{m}_{b+l}\pa_y^lg] +\cO(k-2),
\end{equation}
Here $\cO(k-2)$ denotes terms for which $(\WF,0,\gamma,T)$-norms are estimated by
multiples of $\|g\|_{\WF,k-2,\gamma,T},$ which are also of no consequence.

The new contributions to $\|E_{\epsilon}^tg\|_{\WF,k,\gamma,T}$ come from terms like:
\begin{equation}\label{eqn14.16.1}
\begin{split}
\|[\psi_{j,\epsilon},L]K_t^{b_j+l}L^{m}_{b+l}\pa_y^l(\varphi_{j,\epsilon}g)\|_{\WF,0,\gamma,T}, \\
\|\psi_{j,\epsilon}(\tb_j(x_0)\pa_{x_j})K_t^{b_j+l}L^{m}_{b+l}\pa_y^l(\varphi_{j,\epsilon}g)
\|_{\WF,0,\gamma,T},
\end{split}
\end{equation}
and
\begin{equation}\label{eqn14.17.1}
  \psi_{j,\epsilon}[\pa_{x_j}\tb_j]\pa_{x_j}
K^{b_j+l-1}_tL^{m}_{b+l-1}\pa_y^{l-1}(\varphi_{0,\epsilon}g),
\end{equation}
for $j=0,1.$ 
 The terms in~\eqref{eqn14.16.1} are precisely the sorts of terms estimated earlier in the
chapter, with exactly the same coefficients. All that has changed is that we
have replaced $K^{b_j}_t$ with $K^{b_j+l}_t$ and $\varphi_{j,\epsilon}g$ with
$L^{m}_{b+l}\pa_y^l(\varphi_{j,\epsilon}g).$ From the Leibniz formula, it is
again clear that the only terms that cannot be subsumed into
$\|g\|_{\WF,k-1,\gamma,T}$ are those of the form:
\begin{multline*}
\|[\psi_{j,\epsilon},L]K_t^{b_j+l}\varphi_{j,\epsilon}L^{m}_{b+l}\pa_y^l(g)\|_{\WF,0,\gamma,T}, \\ 
\|\psi_{j,\epsilon}(\tb_j(x_0)\pa_{x_j})K_t^{b_j+l}\varphi_{j,\epsilon}L^{m}_{b+l}\pa_y^l(g)
\|_{\WF,0,\gamma,T},
\end{multline*}
These terms can all be estimated by
\begin{equation}
  C\epsilon^{2-\gamma}\|g\|_{\WF,k,\gamma,T},
\end{equation}
where the constant $C$ is uniformly bounded for $k\leq K.$

To complete this case we need to consider the terms in~\eqref{eqn14.17.1}; 
these are not \emph{a priori} of lower order because $\pa_{x_j}\tb_j$ may not
vanish at $x_j=0.$ On the other hand, the estimate given in~\eqref{eqn9.76.1} shows
that there are constants $C_k,\mu(k,\gamma)$ so that
\begin{equation}
\begin{split}
  \|\psi_{j,\epsilon}[\pa_{x_j}\tb_j]\pa_{x_j}
K^{b_j+l-1}_t&L^{m}_{b+l-1}\pa_y^{l-1}(\varphi_{0,\epsilon}g)\|_{\WF,0,\gamma,T}=\\
 &\|\psi_{j,\epsilon}[\pa_{x_j}\tb_j]
K^{b_j+l}_t\pa_{y}L^{m}_{b+l-1}\pa_y^{l-1}(\varphi_{0,\epsilon}g)\|_{\WF,0,\gamma,T}\\
&\leq
C_k\epsilon^{-\mu(k,\gamma)}
T^{1-\frac{\gamma}{2}}\|g\|_{\WF,k,\gamma,T}.
\end{split}
\end{equation}

If we fix any $\delta>0,$ then we can choose an $\epsilon>0$ so that,
$C\epsilon^{2-\gamma}<\delta$ and therefore for some constants $\{\beta'_k\},$
the estimates
\begin{equation}\label{eqn14.14.2}
  \|E_{b\epsilon}^{t}g\|_{\WF,k,\gamma,T}\leq
  (\delta+C_k\epsilon^{-\mu(k,\gamma)}
  T^{1-\frac{\gamma}{2}})\|g\|_{\WF,k,\gamma,T}+\beta'_k\|g\|_{\WF,k-1,\gamma,T},
\end{equation}
hold for $k\leq K.$ 
Fix a $\delta<1/2,$ which thereby fixes an $\epsilon>0.$ Let
$\varphi_b=\varphi_{0,\epsilon}+\varphi_{1,\epsilon},$ and choose $\psi_i$ with
compact support $[a,b]\subset (0,1)$ and equal to $1$ on a neighborhood of
$\supp(1-\varphi_b).$  Finally we let $\hQ^t_i$ be the exact solution operator
to the Dirichlet problem
\begin{equation}
  (\pa_t-L)u=g\text{ on }[a,b]\text{ with }u(x,0)=u(a,t)=u(b,t)=0.
\end{equation}
With the global parametrix given by
\begin{equation}
  \tQ^t=\hQ^t_{b\epsilon}+\psi_i\hQ^t_i(1-\varphi_b),
\end{equation}
we see that
\begin{equation}
  (\pa_t-L)\tQ^tg=g+E^t_{b\epsilon}g+[\psi_i,L]\hQ^t_i(1-\varphi_b)g.
\end{equation}
Since the support of $1-\psi_i$ and $1-\varphi_b$ do not overlap, the
induction hypothesis shows that there is a constant $C(T,k,\gamma),$ which tends to $0$ as $T\to 0^+,$ so that
\begin{equation}
  \|[\psi_i,L]\hQ^t_i(1-\varphi_b)g\|_{\WF,k,\gamma,T}\leq
  C(T,k,\gamma)\|g\|_{\WF,k,\gamma,T}.
\end{equation}
Note that $\epsilon>0$ has already been fixed.

If we let $E^tg=E_{b\epsilon}^{t}g+[\psi_i,L]\hQ^t_i(1-\varphi_b)g,$ then
for some $T_0>0,$ there are constants $\{\beta_k:\: k=0,\dots, K\}$ so
that we have the estimates
\begin{multline}
  \|E^{t}g\|_{\WF,k,\gamma,T_0}\leq\\
2\delta\|g\|_{\WF,k,\gamma,T_0}+\beta_k
\|g\|_{\WF,k-1,\gamma,T_0},\text{ for }0\leq k\leq K.
\end{multline}
Theorem~\ref{thm14.0.1} applies to show that the Neumann series for
$(\Id+E^t)^{-1}$ converges in the operator norm topologies defined by
$\cC^{k,\gamma}_{\WF}([0,1]\times [0,T_0])$ for $0\leq k\leq K.$ The argument
at the end of Chapter~\ref{ss13.4} applies to show that this operator has the
small time localization property as a map from
$\cC^{k,\gamma}_{\WF}([0,1]\times [0,T_0])$ to
$\cC^{k,2+\gamma}_{\WF}([0,1]\times [0,T_0]).$ As $K$ is arbitrary we see that
this completes the proof, in dimension 1, of induction step for the
inhomogeneous case and any $k.$ 

\subsection{The Higher Dimensional Case}
The argument in the general case is quite similar to the 1-dimensional case,
though there are more terms analogous to those appearing
in~\eqref{eqn14.17.1}. We now briefly describe it. As above the key point is to
show that estimates like those in~\eqref{eqn14.14} and~\eqref{eqn14.14.2} hold
for the error terms coming from the boundary parametrix. This fixes a choice of
$\epsilon>0,$ and then we can apply the induction hypothesis to obtain similar
estimates for the contribution of the interior parametrix to the error term,
which, along with the contributions of terms like those in~\eqref{eqn14.17.1},
is made as small as we like by taking $0<T_0$ small enough. The boundary
contributions to the error term are enumerated in~\eqref{eqn13.132}. 

It is immediate that the only contributions to
$\|[E^{\infty,t}_{\epsilon}+E^{1,t}_{\epsilon}]g\|_{\WF,k,\gamma,T}$ rel
$\|[E^{\infty,t}_{\epsilon}+E^{1,t}_{\epsilon}]g\|_{\WF,k-1,\gamma,T}$ are of
the terms of the types:
\begin{equation}
  \|\psi_{i,q}d_l(\bx,\by)\pa_{y_l}K^{\bb+\balpha,t}_{i,q}[L_{\bb+\balpha,m}\pa_{\bw}^{\balpha}
\pa_{\bz}^{\bbeta}]\chi_{i,q}g\|_{\WF,0,\gamma,T}
\end{equation}
and
\begin{equation}
  \|[\psi_{i,q},L]K^{\bb+\balpha,t}_{i,q}[L_{\bb+\balpha,m}\pa_{\bw}^{\balpha}
\pa_{\bz}^{\bbeta}]\chi_{i,q}g\|_{\WF,0,\gamma,T}.
\end{equation}
These are types of terms that we have estimated earlier
(see~\eqref{eqn13.136}); hence for $0\leq k\leq K,$ there are constants
$C(T,k,\gamma),$ $\mu(k,\gamma)$ and $\{\beta'_k\}$ so that
\begin{equation}\label{eqn14.31}
 \|[E^{\infty,t}_{\epsilon}+E^{1,t}_{\epsilon}]g\|_{\WF,k,\gamma,T}\leq
C(T,k,\gamma)\epsilon^{-\mu(k,\gamma)}\|g\|_{\WF,k,\gamma,T}+\beta'_k
\|g\|_{\WF,k-1,\gamma,T}.
\end{equation}
Moreover $C(T,k,\gamma)$ tend to zero as $T\to 0^+.$

This leaves only  terms of the form
\begin{equation}
  \|\pa_{t}^m\pa_{\by}^{\bbeta}\pa_{\bx}^{\balpha}\left(\psi_{i,q}L^r_{i,q}
K^{\bb,t}_{i,q}[\chi_{i,q}g]\right)\|_{\WF,0,\gamma,T},
\end{equation}
with $2m+|\balpha|+|\bbeta|=k.$ As in the 1-dimensional case, there are two
types of terms that now need to be estimated. The first type arises by passing all
derivatives through to $\chi_{i,q}g,$ which are of the form:
  \begin{equation}
  \|\psi_{i,q}L^r_{i,q}
K^{\bb+\balpha,t}_{i,q}\left(L_{\bb+\balpha,m}\pa_{\by}^{\bbeta}
\pa_{\bx}^{\balpha}[\chi_{i,q}g]\right)\|_{\WF,0,\gamma,T}.
\end{equation}
These are precisely the sorts of terms estimated in the $k=0$ case. As before
we choose $0<\gamma'$ so that $\gamma+\gamma'<1.$ For
$0\leq k\leq K,$ there are constants $C(k,\gamma)$ for which
\begin{equation}\label{eqn14.34}
  \|\psi_{i,q}L^r_{i,q}
K^{\bb+\balpha,t}_{i,q}\left(L_{\bb+\balpha,m}\pa_{\by}^{\bbeta}
\pa_{\bx}^{\balpha}[\chi_{i,q}g]\right)\|_{\WF,0,\gamma,T}\leq
C(k,\gamma)\epsilon^{1-\gamma-\gamma'}
\|g\|_{\WF,k,\gamma,T}.
\end{equation}

The only terms that remain result from differentiation of the coefficients of
terms appearing in $L^r_{i,q}$ of the forms $x_j\pa_{x_j},$
$x_j\pa_{x_j}\pa_{y_l},$ or $x_ix_j\pa_{x_i}\pa_{x_j}.$ The parts of terms of
these types that cannot be subsumed by a large multiple of
$\|g\|_{\WF,k-1,\gamma,T}$ are
\begin{multline}\label{eqn14.35.2}
  \|\psi_{i,q}\pa_{x_j}
K^{\bb+\balpha',t}_{i,q}\left(L_{\bb+\balpha',m}\pa_{\by}^{\bbeta}
\pa_{\bx}^{\balpha'}[\chi_{i,q}g]\right)\|_{\WF,0,\gamma,T}=\\
 \|\psi_{i,q}
K^{\bb+\balpha,t}_{i,q}\pa_{x_j}\left(L_{\bb+\balpha',m}\pa_{\by}^{\bbeta}
\pa_{\bx}^{\balpha'}[\chi_{i,q}g]\right)\|_{\WF,0,\gamma,T},
\end{multline}

\begin{multline}
  \|\psi_{i,q}\pa_{x_j}\pa_{y_l}
K^{\bb+\balpha',t}_{i,q}\left(L_{\bb+\balpha',m}\pa_{\by}^{\bbeta}
\pa_{\bx}^{\balpha'}[\chi_{i,q}g]\right)\|_{\WF,0,\gamma,T}=\\
 \|\psi_{i,q}
K^{\bb+\balpha,t}_{i,q}\pa_{x_j}\pa_{y_l}\left(L_{\bb+\balpha',m}\pa_{\by}^{\bbeta}
\pa_{\bx}^{\balpha'}[\chi_{i,q}g]\right)\|_{\WF,0,\gamma,T},
\end{multline}
where $\balpha'=\balpha-e_j;$ and
\begin{multline}
  \|\psi_{i,q}x_i\pa_{x_j}\pa_{x_i}
K^{\bb+\balpha'',t}_{i,q}\left(L{\bb+\balpha'',m}\pa_{\by}^{\bbeta}
\pa_{\bx}^{\balpha''}[\chi_{i,q}g]\right)\|_{\WF,0,\gamma,T}=\\
\|\psi_{i,q}x_i
K^{\bb+\balpha,t}_{i,q}\pa_{x_i}\pa_{x_j}\left(L_{\bb+\balpha'',m}\pa_{\by}^{\bbeta}
\pa_{\bx}^{\balpha''}[\chi_{i,q}g]\right)\|_{\WF,0,\gamma,T},
\end{multline}
where $\balpha''=\balpha-e_i-e_j;$ and
\begin{multline}\label{eqn14.38.2}
  \|\psi_{i,q}\pa_{x_j}\pa_{x_i}
K^{\bb+\balpha'',t}_{i,q}\left(L_{\bb+\balpha'',m}\pa_{\by}^{\bbeta}
\pa_{\bx}^{\balpha''}[\chi_{i,q}g]\right)\|_{\WF,0,\gamma,T}=\\
 \|\psi_{i,q}
K^{\bb+\balpha,t}_{i,q}\pa_{x_j}\pa_{x_i}\left(L_{\bb+\balpha'',m}\pa_{\by}^{\bbeta}
\pa_{\bx}^{\balpha''}[\chi_{i,q}g]\right)\|_{\WF,0,\gamma,T}.
\end{multline}
It follows from~\eqref{eqn12.15.1} and the foregoing argument that for $0\leq
k\leq K,$ there are constants $C(k,\gamma),\mu(k,\gamma)$ so that, for $T<1,$
each of the terms in~\eqref{eqn14.35.2}--\eqref{eqn14.38.2} is bounded by
\begin{equation}\label{eqn14.39}
  C(k,\gamma)\epsilon^{-\mu(k,\gamma)}T^{\frac{\gamma}{2}}\|g\|_{\WF,0,\gamma,T}.
\end{equation}

Combining~\eqref{eqn14.31} with~\eqref{eqn14.34}, and~\eqref{eqn14.39} we see
that there are constants $\{\beta_k\}$ so that
\begin{multline*}
\|E^t_{b\epsilon}g\|_{\WF,k,\gamma,T}\leq  \\
[C(T,k,\gamma)\epsilon^{-\mu(k,\gamma)}+C(k,\gamma)\epsilon^{1-\gamma-\gamma'}]\|g\|_{\WF,k,\gamma,T}+
\beta_k \|g\|_{\WF,k-1,\gamma,T},
\end{multline*}
where $C(T,k,\gamma)\to 0$ as $T\to 0^+.$ If we fix $0<\delta<1/2,$ then by first
choosing $\epsilon>0$ and then $T_0>0$ we can arrange to have:
\begin{equation}
\|E^t_{b\epsilon}g\|_{\WF,k,\gamma,T}\leq 
\delta\|g\|_{\WF,k,\gamma,T}+
\beta_k \|g\|_{\WF,k-1,\gamma,T},
\end{equation}
for $0\leq k\leq K.$ As in the 1-dimensional case, the argument is finished by
augmenting the boundary parametrix with an interior term, obtaining
\begin{equation}
  \cQ^t=\hQ^t_{b\epsilon}+\psi_i\hQ^t_i(1-\varphi_b).
\end{equation}
Possibly decreasing $T_0,$ we obtain an error term $E^t$ that satisfies:
\begin{equation}
\|E^tg\|_{\WF,k,\gamma,T_0}\leq 
2\delta\|g\|_{\WF,k,\gamma,T_0}+
\beta_k \|g\|_{\WF,k-1,\gamma,T_0},
\end{equation}
for $0\leq k\leq K.$  This completes the proof of~\eqref{hgordestind} for an
arbitrary $K\in \bbN$ and $0<\gamma<1.$

\chapter{The Resolvent Operator}\label{c.resolv}
We have shown that $e^{tL},$ the formal solution operator for the Cauchy
problem $(\pa_t-L)v=0,$ $v(p,0)=f(p),$ makes sense for initial data
$f\in\cC^{0,2+\gamma}_{\WF}(P),$ and that the solution belongs to
$\cC^{0,2+\gamma}_{\WF}(P\times [0,\infty)).$ Of course, much more is true, but
the extension to less regular data, seems to entail rather different techniques
from those employed thus far. 

The Laplace transform of $e^{tL}$ is formally the resolvent operator:
\begin{equation}
  (\mu-L)^{-1}=\int\limits_{0}^{\infty}e^{-\mu t}e^{tL}fdt.
\end{equation}
Using the Laplace transform of a parametrix for the heat kernel and a
perturbative argument, we construct below an operator $R(\mu),$ which
depends analytically on $\mu$ lying in the complement of a set $E\subset\bbC$ 
which lies in a conic neighborhood of $(-\infty,0]$. This means  that 
for any $\alpha>0$ there exists an $0<R_{\alpha}$ so that 
\begin{equation}\label{eqn_musectest0}
E \subset \{|\arg\mu|> \pi-\alpha\text{ or }|\mu|<R_{\alpha}\}. 
\end{equation}

If $f\in\cC^{0,\gamma}(P),$ then $R(\mu)$ satisfies 
\begin{equation}\label{eqn.rgtinv}
  (\mu-L)R(\mu)f=f.
\end{equation}
Hence $R(\mu)$ is a right inverse for $(\mu-L).$ As a map from
$\cC^{0,\gamma}_{\WF}(P)$ to itself the operator $R(\mu)$ is compact.
In fact for any $0<\gamma<1$ and $k\in\bbN,$ $R(\mu)$ defines a bounded map
from $\cC^{k,\gamma}_{\WF}(P)$ to $\cC^{k,2+\gamma}_{\WF}(P).$ In H\"older
spaces, these are the natural elliptic estimates for generalized Kimura
diffusions. Coupling this with Corollary~\ref{cor13.spec} shows that 
$\mu-L,$ acting on the spaces $\cC^{k,2+\gamma}_{\WF}(P),$ is injective for $\mu$
in the right half plane.  Since we already have shown that for such $\mu$, 
$$(\mu-L):\cC^{0,2+\gamma}_{\WF}(P)\to\cC^{0,\gamma}_{\WF}(P)$$ 
is surjective, the open mapping theorem implies that $R(\mu)$ is also a left
inverse, and hence equals the resolvent operator $(\mu-L)^{-1}.$ 

 Since the domain $\cC^{0,2+\gamma}_{\WF}(P)$ is \emph{not} dense  in
$\cC^{0,\gamma}_{\WF}(P)$ a few more remarks are in order. Suppose
that $\mu$ is a value for which $(\mu-L)$ is invertible. We can rewrite
\begin{equation}
  (\mu+\nu-L)=[\Id+\nu(\mu-L)^{-1}](\mu-L)
\end{equation}
The map $(\mu-L):\cC^{0,2+\gamma}_{\WF}(P)\to \cC^{0,\gamma}_{\WF}(P)$ is an
isomorphism. The maps $[\Id+\nu(\mu-L)^{-1}]: \cC^{0,\gamma}_{\WF}(P)\to
\cC^{0,\gamma}_{\WF}(P)$ depend analytically on $\nu$ and are Fredholm of index
zero. From this we conclude that the set of $\nu$ for which $(\mu+\nu-L)$ fails
to be invertible is discrete and coincides with the set 
\begin{equation}
  \{\nu:\Ker(\mu+\nu-L)\neq 0\}.
\end{equation}
Thus $L:\cC^{0,2+\gamma}_{\WF}(P)\to \cC^{0,\gamma}_{\WF}(P)$ has a
compact resolvent, with discrete  spectrum lying in a conic neighborhood of the
negative real axis.  Moreover, the elliptic estimates show that all
eigenfunctions belong to $\cC^{\infty}(P),$ so the spectrum of $L$ acting on
the spaces $\cC^{k,2+\gamma}_{\WF}(P)$ does not depend on $k$ or $\gamma.$

Using standard functional analytic techniques this allows us to show that the
solution to the Cauchy problem
\begin{equation}
  (\pa_t-L)v=0\text{ with }v(p,0)=f(p)
\end{equation}
is defined for $f\in\cC^{0,\gamma}_{\WF}(P)$ and, in fact, extends analytically
in $t$ to the right half plane. The solution belongs to
$\cC^{0,2+\gamma}_{\WF}(P)$ for any time with positive real part. Indeed we
also show that, for any $k\in\bbN,$  if $f\in\cC^{k,\gamma}_{\WF}(P),$ then the
solution belongs to $\cC^{k,2+\gamma}_{\WF}(P),$ for $t$ in the right half plane.

The solution operator, $\cQ_0^t,$ defines a semi-group; thus, for any $N\in\bbN$
\begin{equation}
\cQ_0^tf=[\cQ_0^{\frac{t}{N}}]^Nf
\end{equation}
We have the obvious inclusions
$\cC^{1,\gamma}_{\WF}(P)\subset\cC^{0,2+\gamma}_{\WF}(P),$ and in fact for any
$k\in\bbN,$ we have $\cC^{k+1,\gamma}_{\WF}(P)\subset
\cC^{k,2+\gamma}_{\WF}(P).$ These inclusions, the semi-group property,
and these regularity results show that the solution $v$ to Cauchy problem, with
H\"older initial data, belongs to $\cC^{\infty}(P\times (0,\infty)).$

In the next section we construct the resolvent kernel, using an induction over the
maximal codimension of $bP,$ similar to that employed in the previous chapter
to construct the heat kernel. We also prove various estimates on it and
corresponding estimates for the solution operator for the Cauchy problem. 

\section{Construction of the resolvent}
To construct the resolvent operator we proceed very much as for the 
construction of the heat kernel. We use an induction over the maximal
codimension of $bP,$ which allows us to construct an approximate solution
operator for the Cauchy problem of the form
\begin{equation}
  \hQ^t=\hQ_i^t+\hQ_b^t,
\end{equation}
with $\hQ_i^t$ the ``interior'' and $\hQ_b^t$ the ``boundary'' contributions,
respectively. We then analyze the operator:
\begin{equation}
\begin{split}
  \hR(\mu)&=\int\limits_{0}^{\infty}e^{-t\mu}\hQ^tdt\\
&=\hR_b(\mu)+\hR_i(\mu).
\end{split}
\end{equation}
The operator $\hQ_b^t$ extend analytically to $\Re t>0,$ and from its
form we see that $\hR_b(\mu)$ extends analytically to the complement
of $(-\infty,0].$ From the induction hypothesis it follows that
$\hR_i(\mu)$ is analytic in the complement of a discrete set lying in
a conic neighborhood of $(-\infty,0].$

We show that
\begin{equation}
  (\mu-L)\hR(\mu)=(\Id-E_{\mu}),
\end{equation}
where the operator $E_{\mu}:\cC^{k,\gamma}_{\WF}(P)\to \cC^{k,\gamma}_{\WF}(P)$
is bounded for arbitrary $k\in\bbN_0$ and $0<\gamma<1.$ We show that for a
given $k,$ $0<\gamma,$ and $0<\alpha$ there is an $R_{\alpha}$ so that for
$\mu$ satisfying
\begin{equation}\label{eqn_musectest}
    |\arg\mu|<\pi-\alpha\text{ and }|\mu|>R_{\alpha},
  \end{equation}
  the norm of this operator is less than $1$ and therefore, for $\mu$ in this
  domain, we can define the analytic family of operators:
\begin{equation}
  R(\mu)=\hR(\mu)(\Id-E_{\mu})^{-1}
\end{equation}
This operator is a right inverse
\begin{equation}
  (\mu-L)R(\mu)f=f\text{ for all }f\in\cC^{k,\gamma}_{\WF}(P).
\end{equation}
We then verify the estimates in the induction hypothesis. 

As noted above this allows us to construct the solution operator for the Cauchy
problem for the heat equation via the contour integral:
\begin{equation}
  \cQ^t_0=\frac{1}{2\pi i}\int\limits_{\Gamma_{\alpha}}R(\mu)e^{\mu t}d\mu.
\end{equation}
The contour $\Gamma_{\alpha}$ is the boundary of the complement of the region
described in~\eqref{eqn_musectest}. This defines a semi-group, analytic in $\Re
t>0,$ acting on the spaces $\cC^{k,\gamma}_{\WF}(P).$

The theorem we prove is the following:
\begin{theorem}\label{thm.resolP} Let $P$ be a manifold with corners of codimension $n$ and $L$ a
generalized Kimura diffusion operator. Fix $k\in\bbN$ and $0<\gamma<1$. There is a discrete 
subset $E,$ independent of $(k,\gamma),$ contained in $\Re\mu\leq 0$ and lying in 
a conic neighborhood of $(-\infty,0],$ such that the spectrum of $L$
  acting on $\cC^{k,2+\gamma}_{\WF}(P)$ is contained in the set $E.$ The
  resolvent operator $R(\mu)$ is analytic in $\bbC\setminus E.$ For $0<\alpha$
  there is an $R_{\alpha}$ so that for $\mu$ satisfying~\eqref{eqn_musectest}
  there are constants $C_{\alpha}, C_{k,\alpha}$ so that $R(\mu)$ satisfies the following
  estimates:
  \begin{equation}\label{eqn13.13.04}
    \begin{split}
\|R(\mu)f\|_{L^{\infty}}&\leq \frac{C_{\alpha}}{\mu}\|f\|_{L^{\infty}}
\text{ for }\mu\in (0,\infty)\\
\|R(\mu)f\|_{\WF,k,\gamma}&\leq \frac{C_{k,\alpha}}{|\mu|}\|f\|_{\WF,k,\gamma}\\
\|R(\mu)f\|_{\WF,k,2+\gamma}&\leq C_{k,\alpha}\|f\|_{\WF,k,\gamma}.
\end{split}
  \end{equation}
Let $V$ be a vector field defined in $P$ so that, in the neighborhood of a
boundary point of codimension $l,$ $V$ takes the form
\begin{equation}\label{eqn13.14.066}
  V(\bx,\by)=\sum_{j=1}^lb_j(\bx,\by) x_j\pa_{x_j}+\sum_{l=1}^{n-l}d_l(\bx,\by)\pa_{y_l}.
\end{equation}
For $\mu$ satisfying~\eqref{eqn_musectest} there are constants $C_{k,\alpha}$ so
that, if $|\mu|>1,$ then
\begin{equation}\label{eqn13.15.04}
  \|V R(\mu)f\|_{\WF,k,\gamma}\leq \frac{C_{k,\alpha}}{|\mu|^{\frac{\gamma}{2}}}\|f\|_{\WF,k,\gamma}.
\end{equation}
\end{theorem}
\begin{proof}
  The proof is very similar to that of Theorem~\ref{thm13.1}, and so many
  details are left to the reader. The construction of the resolvent is done by
  induction over the maximal codimension of $bP.$ The verification of the
  induction hypothesis in this proof is actually somewhat simpler, as we do not
  use the weak localization property. We begin with the case that $P$ is
  compact manifold without boundary, i.e. the maximum codimension of $bP$ is
  zero, and $L$ is a non-degenerate elliptic operator without constant
  term. The H\"older spaces are simply the classical H\"older spaces, and the
  statement of the theorem is more or less contained in~\cite{KrylovGSM12},
  though this text does not address the compact manifold case explicitly. As
  the detailed estimates for $R(\mu)$ stated in~\eqref{eqn13.13.04}
  and~\eqref{eqn13.15.04} also do not seem to be available in the literature,
  we start by briefly outlining this case.

  As before we begin with the $k=0$ case. The case of $k>0$ follows by
  applying~\ref{thm14.0.1}, very much like in the proof of
  Theorem~\ref{thm13.1}. The details of this argument are also left to the
reader. We continue to use the notation and constructions from 
sections~\ref{s.ovrprfHK}--~\ref{ss.bndprmtrx}.
\subsection{The compact manifold case}
For $\epsilon>0$ we cover $P$ by open balls of radius $\epsilon,$
$\{B_{\epsilon}(\bx_q):q=1,\dots,N_{\epsilon}\},$ so that any point lies in at
most $S$ of the balls $\{B_{3\epsilon}(\bx_q):q=1,\dots,N_{\epsilon}\}.$ As
noted earlier, $S$ can be taken to be independent of $\epsilon>0.$ Let
$\{\varphi_q\}$ denote a partition of unity subordinate to this cover and
$\{\psi_q\}$ smooth functions, such that:
  \begin{equation}
    \psi_q(x)=1\text{ in }B_{2\epsilon}(\bx_q), \text{ and }
 \supp\psi_q\subset B_{3\epsilon}(\bx_q).
  \end{equation}

We let $L^q$ denote the constant coefficient operator obtained by freezing the
coefficients of the second order part of $L$ at the point $\bx_q.$ If
$(y_1,\dots, y_n)$ are local coordinates near to $\bx_q,$ then:
\begin{equation}
  L^q=\sum_{l,m=1}^nc_{lm}(\bx_q)\pa_{y_l}\pa_{y_m}.
\end{equation}
We let $Q_q^t$ be the heat kernel defined by $L^q.$ This is obtained from the
Euclidean heat kernel by a linear change of variables. 

The parametrix for the heat kernel is defined, for $t$ in the right half plane,
by
\begin{equation}
  \hQ^t=\sum_{q=1}^{N_{\epsilon}}\psi_qQ_q^t\varphi_q.
\end{equation}
and the resolvent, for $\mu\in\bbC\setminus (-\infty,0]$ by
\begin{equation}
  \hR(\mu)=\int\limits_{0}^{\infty}e^{-s\eit\mu}\hQ^{s\eit}\eit ds,
\end{equation}
where $\Re[\eit\mu]>0.$ In the sequel we let
\begin{equation}
  \Gamma_{\theta}=\{s\eit: s\in [0,\infty)\}.
\end{equation}

We now compute the ``error term,'' $E_{\mu}$
\begin{equation}
  \begin{split}
L\hR(\mu)f&=\int\limits_{\Gamma_{\theta}}e^{-t\mu}
\sum_{q=1}^{N_{\epsilon}}L\psi_qQ_q^{t}\varphi_qf dt\\
&=\int\limits_{\Gamma_{\theta}}e^{-t\mu}
\sum_{q=1}^{N_{\epsilon}}\left\{[L,\psi_q]+\psi_q(L-L^q)+\psi_qL^qr\right\}Q_q^{t}\varphi_qf
dt.
\end{split}
\end{equation}
Using that $L^qQ_q^t=\pa_tQ_q^t$, we integrate by parts in the last term and obtain:
\begin{equation}
\begin{split}
(\mu-L)\hR(\mu)f &=f- \int\limits_{\Gamma_{\theta}}e^{-t\mu}
\sum_{q=1}^{N_{\epsilon}}\left\{[L,\psi_q]+\psi_q(L-L^q)\right\}Q_q^{t}\varphi_qf
dt\\
&=f-E_{\mu}f.
\end{split}
\end{equation}

There are two kinds of error terms: those arising from the commutators
$[L,\psi_q],$ which are lower order, and those arising from freezing
coefficients $\psi_q(L-L^q).$ The differences $L-L^q$ are of the form
\begin{equation}
\begin{split}
  L-L^q&=\sum_{l,m=1}^n(c(\by)-c_{lm}(\bx_q))\pa_{y_l}\pa_{y_m}
+\sum_{l=1}^nd_l(\by)\pa_{y_l}\\
&=\sum_{l,m=1}^n\Delta c^q(\by)\pa_{y_l}\pa_{y_m}
+V^qf
\end{split}
\end{equation}
As in the previous case, the second order terms of this type 
are controlled by taking $\epsilon$ sufficiently small.  The contribution of
each such term is of the form
\begin{equation}
  \|\psi_q\Delta c^q\pa_{y_l}\pa_{y_m}R^q(\mu)\varphi_qf\|_{\WF,0,\gamma},
\end{equation}
where
\begin{equation}
  R^q(\mu)f=\int\limits_{\Gamma_{\theta}}e^{-\mu t}Q_q^tfdt.
\end{equation}
Arguing as in the proof of Theorem~\ref{thm13.1}, and using the estimates for
$R^q(\mu)$ given in Proposition~\ref{prop8.0.4.nm} we see that there is a
$\tgamma<1,$ so that
\begin{equation}
  \|\psi_q\Delta c^q\pa_{y_l}\pa_{y_m}R^q(\mu)\varphi_qf\|_{\WF,0,\gamma}\leq
C_{\alpha}\epsilon^{1-\tgamma} \|f\|_{\WF,0,\gamma}.
\end{equation}
As before, for each point in $P,$ only a fixed finite number, $S$ (independent
of $\epsilon$)  of terms contributes to this error term, so we get the estimate
\begin{equation}
  \|\sum_{q=1}^{N_{\epsilon}}\psi_q\sum_{l,m=1}^n(c(\by)-c_{lm}(\bx_q))\pa_{y_l}\pa_{y_m}R^q(\mu)\varphi_q
  f\|_{\WF,0,\gamma}\leq SC_{\alpha}\epsilon^{1-\tgamma} \|f\|_{\WF,0,\gamma}.
\end{equation}
We can now fix $\epsilon>0$ so that the coefficient 
$$SC_{\alpha}\epsilon^{1-\tgamma} \leq \frac{1}{4}.$$

The commutators are
first order operators:
\begin{equation}
  [L,\psi_q]f=2\sum_{l,m=1}^nc_{lm}(\by)\pa_{y_l}\psi_q\pa_{y_m}f+
\sum_{l=1}^nd_l(\by)(\pa_{y_l}\psi_q) f.
\end{equation}
These terms along with that defined by the vector fields $\{V^q\},$ are
controlled using the estimates in~\eqref{eqn8.13.02}
and~\eqref{eqn11.169.04}. These estimates show that, for some positive $\nu,$
there is a constant $C_{\alpha}$ so that:
\begin{equation}
  \|[L,\psi_q]\varphi_q
  f\|_{\WF,0,\gamma}+\|V^q\varphi_q
  f\|_{\WF,0,\gamma}\leq
\frac{\epsilon^{-\nu}C_{\alpha}}{|\mu|^{\frac{\gamma}{2}}}\|f\|_{\WF,0,\gamma}.
\end{equation}
Combining these estimates gives
\begin{equation}
  \|E_{\mu} f\|_{\WF,0,\gamma}\leq
  SC_{\alpha}\left[\frac{\epsilon^{-\nu}}{|\mu|^{\frac{\gamma}{2}}}+
\epsilon^{1-\tgamma}\right]\|f\|_{\WF,0,\gamma}.
\end{equation}
This shows that there is an $R_0$ so that if $|\mu|>R_0,$ then the norm of
$E_{\mu}$ is less than $\frac 12$, and therefore $(\Id-E_\mu)$ is well defined as
an operator from $\cC^{0,\gamma}_{\WF}(P)$ to itself. The analytic dependence
on $\mu$ follows from the analyticity of $\Id-E_{\mu}$ and the uniform norm
convergence of the Neumann series. If we define
\begin{equation}
  R(\mu)f=\hR(\mu)(\Id-E_{\mu})^{-1}f,
\end{equation}
then we see that, for any $f\in\cC^{0,\gamma}_{\WF}(P)$ we have that
\begin{equation}
  R(\mu)f\in\cC^{0,2+\gamma}_{\WF}(P)\text{ and } (\mu-L)R(\mu)f=f.
\end{equation}
Finally, it is a classical result that for $\Re\mu>0,$ on a compact manifold,
the only solution of the equation
  \begin{equation}
    (\mu-L)f=0
  \end{equation}
is $f\equiv 0.$ Hence $(\mu-L):\cC^{0,2+\gamma}_{\WF}(P)\to
\cC^{0,\gamma}_{\WF}(P),$ is a one-to-one and onto mapping. The open mapping
theorem implies that $R(\mu)$ is also a left inverse. Hence the identity
\begin{equation}
  R(\mu)(\mu-L)=\Id=(\mu-L)R(\mu),
\end{equation}
holds in the connected component, containing the right half plane, where
$R(\mu)$ is analytic. We have shown that this set contains the complement of a
conic neighborhood of $(-\infty,0].$

The first estimate in~\eqref{eqn13.13.04} follows from the maximum principle.
As a map from $\cC^{0,\gamma}_{\WF}(P)$ to itself $R(\mu)$ is compact, and
therefore the spectrum of $L$ acting on $\cC^{0,2+\gamma}_{\WF}(P)$ is a
discrete set $E.$ We have shown that $E$ is contained in a conic neighborhood
of $(-\infty,0].$

Arguing as in section~\ref{s.highereg} we can show that, for any $0<\phi,$
there are constants $\{\alpha_k,\beta_k,R_k\},$ with $\alpha_k<1,$ so that, if
$|\arg\mu|<\pi-\phi,$ and $|\mu|>R_{k},$ then
\begin{equation}
  \|E_{\mu}f\|_{\WF,k,\gamma}\leq \alpha_k\|f\|_{\WF,k,\gamma}
+\beta_k\|f\|_{\WF,k-1,\gamma}.
\end{equation}
Applying Theorem~\ref{thm14.0.1} we see that the Neumann series for
$(\Id-E_{\mu})^{-1}$ converges in the operator norm defined by
$\cC^{k,\gamma}_{\WF}(P).$ Thus establishing that these results extend to show that,
for any $k\in\bbN,$ the maps
\begin{equation}\label{reslhoreg}
  R(\mu):\cC^{k,\gamma}_{\WF}(P)\longrightarrow \cC^{k,2+\gamma}_{\WF}(P)
\end{equation}
are also bounded. The estimates in the statement of the
theorem,~\eqref{eqn13.13.04} and~\eqref{eqn13.15.04} for $k>0$ follow easily
since it is simply a matter of establishing these estimates for $\hR(\mu).$ For
example, using that the second estimate in~\eqref{eqn13.13.04} holds for
$\hR(\mu),$ we see that
\begin{equation}
\begin{split}
  \|R(\mu) f\|_{\WF,k,\gamma}&=\|\hR(\mu)(\Id-E_{\mu}^{-1}) f\|_{\WF,k,\gamma}\\
&\leq \frac{C_{\alpha}}{|\mu|}\|(\Id-E_{\mu}^{-1}) f\|_{\WF,k,\gamma}\\
&\leq \frac{C'_{\alpha}}{|\mu|}\|f\|_{\WF,k,\gamma}.
\end{split}
\end{equation}
As the other estimates hold for $\hR(\mu)$ it follows by the same sort of
argument that they also hold for $R(\mu).$

Suppose that $\mu\in E,$ and
$f_{\mu}\in\cC^{0,2+\gamma}_{\WF}(P)$ is a non-trivial eigenfunction, with
\begin{equation}
  (\mu-L)f_{\mu}=0.
\end{equation}
If we select $\nu$ so that $\Re(\nu+\mu)$ is sufficiently large, then the
eigenvalue equation implies that
\begin{equation}
  f=\nu R(\mu+\nu) f.
\end{equation}
Since $\cC^{k,2+\gamma}_{\WF}(P)\subset \cC^{k+1,\gamma}_{\WF}(P),$ we can
use~\eqref{reslhoreg} in a boot-strap argument to conclude that
\begin{equation}
  f\in \cC^{k,\gamma}_{\WF}(P)\text{ for all } k.
\end{equation}
From which we conclude that $f\in\cC^{\infty}(P),$ and the
spectrum of $L$ acting on $\cC^{k,2+\gamma}_{\WF}(P)$ does not depend on $k.$

\subsection{The induction argument}\label{s.indargRK}

The proof now proceeds by induction on the maximal codimension of the components
of $bP.$ Suppose that the theorem has been proved for all pairs $(\tP,\tL)$ where $\tP$
is a manifold with corners, with the maximal codimension of $b\tP$ at most $M,$ and
$\tL$ is a generalized Kimura diffusion on $\tP.$ We let $P$ be a manifold with
corners where the maximal codimension of $bP$ is $M+1$ and $L$ be a generalized
Kimura diffusion on $P.$ The parametrix $\hR(\mu)$ for $R(\mu)$ is constructed
as in section~\eqref{s.intcodimHK},with $\hR(\mu)=\hR_b(\mu)+\hR_i(\mu).$ As
\begin{equation}
  \hR_i(\mu)=\int\limits_{0}^{\infty}e^{-t\mu}\psi e^{t\tL}(1-\varphi)dt,
\end{equation}
the induction hypothesis implies that $\hR_i(\mu)$ has an analytic
extension to $\bbC\setminus F,$ where $F$ is a discrete set lying in a
conic neighborhood of $(-\infty,0].$

The only change is that, instead of~\eqref{eqn12.131.06} we let
\begin{equation}
  \hR_b(\mu)=\sum_{i=1}^p\sum_{q\in F_{i,\epsilon}}\psi_{i,q}R^{\bb}_{i,q}(\mu)\chi_{i,q},
\end{equation}
where
\begin{equation}
  R^{\bb}_{i,q}(\mu)=\int\limits_{0}^{\infty}e^{-\mu t}k_{i,q}^{\bb,t}dt,
\end{equation}
with $k_{i,q}^{\bb,t}$  the solution operator for the model problem
\begin{equation}
  (\pa_t-L_{i,q})v=0\text{ with }v(p,0)=f(p).
\end{equation}

The error terms are quite similar to those arising in the previous case. If
\begin{equation}
  w_{i,q}=\psi_{i,q}R^{\bb}_{i,q}(\mu)\chi_{i,q}f,
\end{equation}
then
\begin{equation}
  (\mu-L_{i,q})w_{i,q}=\chi_{i,q}f+
\psi_{ i,q} (L_{ i,q}-L)R^{\bb}_{ i,q}(\mu)[\chi_{ i,q} f]+
[\psi_{ i,q},L]R^{\bb}_{ i,q}(\mu)[\chi_{ i}^q f].
\end{equation}
We again use the decomposition of $L_{i,q}-L$ given in~\eqref{eqn616.0} to write
the error terms as
\begin{equation}
\begin{split}
  E^2_{i,q,\mu}f&=\psi_{ i,q} L^r_{ i,q}R^{\bb}_{ i,q}(\mu)[\chi_{ i,q} f],\\
 E^1_{i,q,\mu}f&=\left\{[\psi_{ i,q},L]+
\sum_{l=1}^md_l(\bx,\by)\pa_{y_l}\right\}R^{\bb}_{ i,q}(\mu)[\chi_{ i,q} f]
\end{split}
\end{equation}

Using the estimates in~\eqref{eqn11.173.006} and the argument from
section~\ref{s.intcodimHK} we  conclude that there is a constant $C_{\alpha}$
so that if $|\arg\mu|<\pi-\alpha,$ then
\begin{equation}
  \| E^2_{i,q,\mu}f\|_{\WF,0,\gamma}\leq C_{\alpha}\epsilon^{1-\gamma-\gamma'}
 \| f\|_{\WF,0,\gamma},
\end{equation}
where $\gamma+\gamma'<1.$ Once again there is an $S$ independent of
$\epsilon>0,$ so that
\begin{equation}
  \| E^2_{\epsilon,\mu}f\|_{\WF,0,\gamma}\leq SC_{\alpha}\epsilon^{1-\gamma-\gamma'}
 \| f\|_{\WF,0,\gamma},
\end{equation}
with
\begin{equation}
  E^2_{\epsilon,\mu}=\sum_{i=1}^{p}\sum_{q\in F_{i,\epsilon}} E^2_{i,q,\mu}.
\end{equation}
We can now fix $\epsilon>0$ so that
$SC_{\alpha}\epsilon^{1-\gamma-\gamma'}<\frac 14.$

The commutators are of the form
\begin{equation}
  [\psi_{ i,q},L] = \sum_{i=1}^{M+1}b_{i}(\bx,\by)x_i\pa_{x_i}+
\sum_{l=1}^{n-(M+1)}d'_{l}(\bx,\by)\pa_{y_l}.
\end{equation}
The estimates in~\eqref{eqn11.169.04} and~\eqref{eqn11.170.04}, along with the
argument in section~\ref{s.intcodimHK} show that there are constants
$C_{\alpha}$ and $\nu$ so that
\begin{equation}
  \| E^1_{\epsilon,\mu}f\|_{\WF,0,\gamma}\leq SC_{\alpha}\frac{\epsilon^{-\nu}}
{|\mu|^{\frac{\gamma}{2}}} \| f\|_{\WF,0,\gamma}.
\end{equation}

With the given choice of $\epsilon>0,$ we again define $\varphi$ as
in~\eqref{12.152.06}. With this choice we have the estimate
\begin{equation}
\begin{split}
\|(\mu-L)\hR_b(\mu)f-\varphi f\|_{\WF,0,\gamma}=\| E_{b,\mu}f\|_{\WF,0\gamma}\leq  \hfill \\
\hfill  SC_{\alpha}\left[\frac{\epsilon^{-\nu}}
{|\mu|^{\frac{\gamma}{2}}}+ \epsilon^{1-\gamma-\gamma'}\right]\| f\|_{\WF,0,\gamma}.
\end{split}
\end{equation}

We now proceed exactly as in section~\ref{ss13.4}: let $U$ be a
neighborhood of $\Sigma_{M+1}$ with $\overline{U}\subsubset
\Int\varphi^{-1}(1).$ As before we apply Theorem~\ref{dblthm} to find
$(\tP,\tL)$ so that the maximal codimension of $b\tP$ is $M$ and $(P_U,L_U)$
is embedded into $(\tP,\tL).$ We let $\tR(\mu)$ be the resolvent operator for
$L_U,$ whose existence and properties follow from the induction
hypothesis. Finally we choose $\psi,$ a smooth function equal to $1$ on
$\supp(1-\varphi),$  compactly supported in $U^c,$ and let
\begin{equation}
  \hR_i(\mu)f=\psi\tR(\mu)[(1-\varphi)f],
\end{equation}
and
\begin{equation}
  \hR(\mu)f=\hR_i(\mu)f+\hR_b(\mu)f.
\end{equation}
We see that
\begin{equation}
  (\mu-L)\hR(\mu) f=f-[L,\psi]\tR(\mu)[(1-\varphi)f]+ E_{b,\mu}f.
\end{equation}
The commutator $[L,\psi]$ is a vector field of the form~\eqref{eqn13.14.066} 
in each adapted coordinate frame. Hence the induction
hypothesis implies that there is a constant $C_{\alpha}$ so that
\begin{equation}
  \|[L,\psi]\tR(\mu)[(1-\varphi)f]\|_{\WF,0,\gamma}\leq
\frac{C_{\alpha}\epsilon^{-\nu}}{|\mu|^{\frac{\gamma}{2}}}\|f\|_{\WF,0,\gamma}.
\end{equation}

Altogether this shows that, with
\begin{equation}
  -E_{\mu}f=(\mu-L)\hR(\mu) f-f
\end{equation}
we have
\begin{equation}
 \|E_{\mu}f\|_{\WF,0,\gamma}\leq  
SC_{\alpha}\left[\frac{\epsilon^{-\nu}}
{|\mu|^{\frac{\gamma}{2}}}+ \epsilon^{1-\gamma-\gamma'}\right]\| f\|_{\WF,0,\gamma}.
\end{equation}
Thus, we can choose $R_{\alpha}$ so that if $|\mu|>R_{\alpha}$ then 
\begin{equation}
  SC_{\alpha}\left[\frac{\epsilon^{-\nu}}
{|\mu|^{\frac{\gamma}{2}}}+ \epsilon^{1-\gamma-\gamma'}\right]<\frac{1}{2},
\end{equation}
so that the Neumann series for $(\Id-E_{\mu})^{-1}$ converges in the operator
norm topology defined by $\cC^{0,\gamma}_{\WF}(P)$ in the set
\begin{equation}
  |\arg\mu|<\pi-\alpha\text{ and }|\mu|>R_{\alpha}
\end{equation}
It is clear that the family
of operators $(\Id-E_{\mu})^{-1}$ is analytic in a conic neighborhood of
$(-\infty,0].$

If we let 
\begin{equation}
  R(\mu)=\hR(\mu)(\Id-E_{\mu})^{-1},
\end{equation}
then this is an analytic family of operators, mapping
$\cC^{0,\gamma}_{\WF}(P)$ to $\cC^{0,2+\gamma}_{\WF}(P),$ which
satisfies
\begin{equation}
  (\mu-L)R(\mu)f=f\text{ for }f\in\cC^{0,\gamma}_{\WF}(P).
\end{equation}
As before, Corollary~\ref{cor13.spec} shows that $(\mu-L)$ is injective for
$\mu$ in the right half plane. The open mapping theorem then implies that
$(\mu-L)$ is actually invertible for $\Re\mu>0,$ and therefore
\begin{equation}
  R(\mu)(\mu-L)f=f\text{ for }f\in\cC^{0,2+\gamma}_{\WF}(P),
\end{equation}
as well. As noted in the compact manifold case, the fact that $R(\mu)$ satisfies
all the estimates in~\eqref{eqn13.13.04} and~\eqref{eqn13.15.04}, follows
immediately from the boundedness of 
\begin{equation}
  (\Id-E_{\mu})^{-1}:\cC^{0,\gamma}_{\WF}(P)\longrightarrow \cC^{0,\gamma}_{\WF}(P),
\end{equation}
and the fact that $\hR(\mu)$ satisfies these estimates. This latter claim
follows from the fact that the model operators satisfy these estimates, and, by
the induction hypothesis, so does $\tR(\mu).$ This completes the induction step
in the $k=0$ case. 

The cases where $k>0$ are quite similar to that treated in
section~\ref{s.highereg}. This case is somewhat simpler, as we do not
need to estimate time derivatives. This means that we only need to use
the formula{\ae} in~\eqref{eqn47new.03} with $j=0.$ In this case
powers of $L{\bb,m}$ do ot appear on the right hand side, and no
hypothesis is required on the support of the data. The only other
significant difference concerns the higher order estimates
in~\eqref{eqn13.15.04}.  The contributions of the interior terms are
estimated, for all $k,$ by using the induction hypothesis and the fact
that the commutator $[\psi,L]$ is of the form given
in~\eqref{eqn13.14.066}. The estimates in
Proposition~\ref{eqn11.170.04} gives
\begin{equation}
  \|\bx\cdot\nabla_{\bx}\hR_b(\mu)
  f\|_{\WF,k,\gamma}+\|\nabla_{\by}\hR_b(\mu)
  f\|_{\WF,k,\gamma}\leq
   C_{\alpha}\epsilon^{-\nu}\left[
\frac{1}{|\mu|^{\frac{\gamma}{2}}}+
\frac{1}{|\mu|}\right]\|f\|_{\WF,k,\gamma}.
\end{equation}
For $|\mu|>1,$ this therefore gives the desired estimate, and completes the
verification of the induction hypothesis for $M+1.$ This completes the
proof of the theorem.
\end{proof}

\section{Holomorphic semi-groups}

Now that we have constructed the resolvent operator for $L$ and demonstrated
that it is analytic in the complement of a conic neighborhood of $(-\infty,0],$
we can use contour integration to construct the solution to the heat
equation. Our second pass through this problem represents a distinct
improvement over our previous result for several reasons:
\begin{enumerate}
\item This time we can work with data belonging
to $\cC^{0,\gamma}_{\WF}(P),$ rather than $\cC^{0,2+\gamma}_{\WF}(P).$ 
\item For such data the solution is shown to belong to
  $\cC^{0,2+\gamma}_{\WF}(P)$ for positive times. If the data is in
  $\cC^{k,\gamma}_{\WF}(P),$  then the solution belongs to $\cC^{k,2+\gamma}_{\WF}(P).$ 
\item A bootstrapping argument, using the inclusion
$$\cC^{k,2+\gamma}_{\WF}(P)\subset \cC^{k+1,\gamma}_{\WF}(P)$$
and the semi-group property, gives that the solution belongs to
$\cC^{\infty}(P),$ for positive times.
\item The solution extends analytically to $t$ in the right half plane, $H_+.$
\end{enumerate}

For any $\alpha>0,$ $0<\gamma<1,$ and $k\in\bbN,$ there is an $0<R_{\alpha}$ so
that as a map from $\cC^{k,\gamma}_{\WF}(P)$ to $\cC^{k,2+\gamma}_{\WF}(P)$ the
operator $R(\mu),$ constructed in Theorem~\ref{thm.resolP}, is analytic in a
domain containing the set $|\arg\mu|\leq \pi -\alpha,$ and $|\mu|\geq
R_{\alpha},$ and satisfies the estimates in the theorem. We let
$\Gamma_{\alpha}$ denote the boundary of the complement of this region. From
these observations, the following theorem follows from standard results in semi-group theory.
See, for example, the proof of Theorem 8.2.1 in~\cite{KrylovGSM12}, or that of
Theorem 2.34 in~\cite{Davies1PS}.
\begin{theorem}\label{thm12.2.1.00} For $f\in\cC^{k,\gamma}_{\WF}(P)$ and $t$ with $|\arg
  t|<\frac{\pi}{2}-\alpha,$ define
  \begin{equation}\label{eqn11.68.006}
   v(t,p)= T^tf(p)=\frac{1}{2\pi i}\int\limits_{\Gamma_{\alpha}}e^{t\mu}R(\mu)f(p) d\mu.
  \end{equation}
Then:
\begin{enumerate}
\item For any $p\in P$ the function $t\mapsto v(t,p)$ is analytic in the right
  half plane, and, for $s,t$ in the right half plane: 
  \begin{equation}
    T^tT^sf=T^{t+s}f
  \end{equation}
\item For any $t\in H_+,$ $v(t,\cdot)\in\cC^{k,2+\gamma}_{\WF}(P).$
\item For any $0<\alpha,$ there is a $C_{\alpha}$ so that for $t$ with $|\arg
  t|<\frac{\pi}{2}-\alpha,$ we have the estimates
\begin{equation}
\|T^tf\|_{\WF,k,2+\gamma}\leq
C_{\alpha}\left[e^{R_{\alpha}|t|}+\frac{1}{|t|}\right]\|f\|_{\WF,k,\gamma}
\end{equation}
\item $v$ satisfies the heat equation in $\{t:\Re t>0\}\times P:$
  \begin{equation}
    \pa_t v=Lv.
  \end{equation}
\item For $t$ real we have
  \begin{equation}
    \|T^tf\|_{L^{\infty}}\leq  \|f\|_{L^{\infty}}\text{ and }
 \lim_{t\to 0^+}\|T^tf-f\|_{L^{\infty}}=0
  \end{equation}
\item For $\tgamma<\gamma,$ we have
\begin{equation}
  \lim_{t\to 0^+}\|T^tf-f\|_{\WF,k,\tgamma}=0.
\end{equation}
\end{enumerate}
\end{theorem}
\begin{remark} From the higher order regularity results and a simple
  integration by parts argument, it follows that if
  $f\in\cC^{2(k+1),\gamma}_{\WF}(P),$ for a $k\in\bbN$ and
  $0<\gamma<1,$ then $v(p,t)=e^{tL}f(p),$ is given by the Taylor
  series, with remainder, for the exponential:
\begin{equation}
  v(p,t)=\sum_{j=0}^k\frac{t^jL^jf(p)}{j!}+\int\limits_{0}^t\frac{(t-s)^{j+1}}{(j+1)!}\pa_t^{j+1}v(p,s)ds.
\end{equation}
\end{remark}

As noted earlier, the regularity statement in this theorem and the fact that
$\cC^{k+1,\gamma}_{\WF}(P)\subset \cC^{k,2+\gamma}_{\WF}(P)$ have an important corollary:
\begin{corollary}
  If, for some $0<\gamma,$  $f\in\cC^{0,\gamma}_{\WF}(P),$ then $v=T^tf$
  belongs to $\cC^{\infty}(H_+\times P).$
\end{corollary}

\chapter{The Semi-Group on $\cC^0(P)$}\label{c.adjsmgrp}
In the previous chapters we have dealt almost exclusively with
solutions to \eqref{inhmprbP00} with inhomogeneous terms $f$ and $g$
in the WF H\"older spaces.  As explained early in this monograph, the
reason for working in H\"older spaces in the first place is to handle
the perturbation theory in passing from the model operator to the
actual one. The original problem, suggested by applications to
population genetics, is to study \eqref{inhmprbP00} with $g = 0$ and
$f \in \calC^0(P)$.  As noted earlier, the existence theory we have
developed suffices to prove that the $\cC^0(P)$-graph closure of $L$
with domain $\cC^{0,2+\gamma}_{\WF}(P),$ for any $0<\gamma$ is the
generator of a $\cC^0$-semi-group on $\cC^0(P).$ We let $\oL$ denote
this operator.  As noted earlier this suffices to establish the
uniqueness of the solution to the martingale problem, and the weak
uniqueness of the solution to associated SDE, which leads to the
existence of a strong Markov process, whose paths are confined to $P,$
almost surely.

Perhaps surprisingly, the refined regularity of solutions with initial
data in $\cC^0(P)$ does not seem to follow easily from all that we
have accomplished thus far. In fact, if the Cauchy data $f$ is
continuous but has no better regularity, then it is not clear that the
solution $u$ gains any smoothness at points of $bP$ at times $t > 0$.
Of course, we do know that $u$ becomes smooth if $f \in
\calC^{0,\gamma}_{\WF}(P)$; we also know that solutions to the model
problem $(\del_t - L_{b,m})$ on $\RR_+^n \times \RR^m$ with continuous
initial data also become smooth. While it seems quite likely that this
also holds for continuous initial data, it does not seem easy to prove
this for general Kimura diffusions using the present methods.  There
are related difficulties concerning the graph closure $\oL$ of $L$ on
$\cC^0$. For example, it is not clear that the resolvent of $\oL$
is compact. We will return to these questions in a later
publication. In this chapter we establish several properties of the
elements of $\Dom(\oL)$ and features of the adjoint operator that can be
deduced from the analysis above.

To be explicit, the graph closure $\oL$ on $\cC^0(P)$ is defined as the unique linear
operator defined on the dense subspace $\Dom(\oL) \subset \cC^0(P)$ characterized by
the condition that $u \in \Dom(\oL)$ and $\oL u  = f \in \cC^0(P)$ if there exists a
sequence $<u_j> \subset \cC^{0,2+\gamma}_{\WF}(P)$ such that $L u_j = f_j$ and
\begin{equation}
u_j \to u\ \mbox{and}\ f_j \to f\ \mbox{in}\ \cC^0(P).
\end{equation}
Since $L$ is a nondegenerate elliptic operator away from $bP$, it is standard that
any $u \in \Dom(\oL)$ is ``almost'' twice differentiable in $\Int P$ in the sense that
\begin{equation}
\Dom(\oL)\subset \bigcap_{1<p<\infty}W^{2,p}_{\loc}(\Int P),
\end{equation}
see~\cite{lunardi}. As is well-known, there is no completely explicit way 
to characterize the regularity of elements of this domain in the interior, but
this is not particularly important here. The more interesting difficulties are 
connected with describing the boundary behavior of elements of $\Dom(\oL)$, and
we turn to this now.

We recall from Chapter 3 that under the assumption that $L$ meets $bP$ cleanly,
$L$ is either tangent to a hypersurface boundary of $P$ or uniformly
transverse. We also recall the notion of the minimal and terminal boundary
components, $bP_{\min}(L)$ and $bP_{\ter}(L):$ $bP_{\min}(L)$ consists of
boundary components that are themselves manifolds without boundary, and
$bP_{\min}^T(L),$ elements of $bP_{\min}(L)$ to which $L$ is tangent. The
terminal boundary, $bP_{\ter}(L)$, consists of $bP_{\min}^{T}(L),$ and boundary
components, $\Sigma,$ to which $L$ is tangent, such that $L_{\Sigma}$ is
transverse to $b\Sigma.$

Even without the cleanness assumption, if $\Sigma$ is a component of a
stratum of $bP$ to which $L$ is tangent, then $L_{\Sigma},$ the
restriction of $L$ to $\Sigma,$ defines a Kimura diffusion operator on
$\cC^2(\Sigma)$.  We can then say something about the behavior of
elements of $\Dom(\oL)$ near to $\Sigma.$
\begin{proposition}\label{restrict1} Suppose that $L$ is tangent to $\Sigma,$ a component of a
stratum of $bP.$ If $w\in \Dom(\oL),$ then $w\!\! \restrictedto_{\Sigma}$ lies in $\Dom(\oL_{\Sigma}),$ and
\begin{equation}
L_{\Sigma}[w\restrictedto_{\Sigma}]=[Lw]\restrictedto_{\Sigma}.
\end{equation}

\end{proposition}
\begin{proof} This is immediate from the fact that $\Dom(\oL)$ is the $\cC^0(P)$-graph closure of $L$ 
acting on $\cC^{0,2+\gamma}_{\WF}(P).$ For if $w\in \Dom(\oL),$ then there exists a sequence $<w_n>
\, \subset \cC^{0,2+\gamma}_{\WF}(P)$ such that 
\begin{equation}
\|w-w_n\|_{\cC^0(P)}+ \|Lw-Lw_n\|_{\cC^0(P)}\to 0.
\end{equation}
Clearly, for each $n$, 
\begin{equation}
L_{\Sigma}[w_n\restrictedto_{\Sigma}]=[Lw_n]\restrictedto_{\Sigma}.
\end{equation}
By assumption, the sequence on the right converges to $[Lw] \! \restrictedto_{\Sigma}$, and 
hence the sequence on the left also converges. This shows that 
$w \! \restrictedto_{\Sigma}\in \Dom(\oL_{\Sigma}),$ and $  L_{\Sigma}[w \! \restrictedto_{\Sigma}]=[Lw] \! \restrictedto_{\Sigma}.$
\end{proof}
\noindent
As already observed by Shimakura, this result implies that
\begin{equation}
  [e^{t\oL}f]\restrictedto_{\Sigma}= e^{t\oL_{\Sigma}}[f\restrictedto_{\Sigma}],
\end{equation}
so that these boundary components are  effectively ``decoupled'' from the rest of $P.$

This result gives some information about the behavior of $w \in
\Dom(\oL)$ in directions transverse to hypersurfaces to which $L$ is
tangent. Let $\Sigma$ be such a hypersurface again, then, in adapted
coordinates near an interior point of $\Sigma$, $L$ takes the form
\begin{equation}
L=x\pa_x^2+\tb(x,\by)x\pa_x+\sum_{l=1}^{N-1}xa_{1l}(x,\by)\pa_x\pa_{y_l}+(1+O(x))L_{\Sigma}.
\end{equation}
Proposition~\ref{restrict1} implies that if $w\in \Dom(\oL),$ then
\begin{equation}
\lim_{x\to  0^+}\left[x\pa_x^2+\tb(x,\by)x\pa_x+\sum_{l=1}^{N-1}xa_{1l}(x,\by)\pa_x\pa_{y_l} \right]w(x,\by)=0.
\end{equation}
A similar result holds at strata of codimension greater than $1$ to which $L$ is tangent.

The space $\cC^0(P)$ is non-reflexive, which means that the semi-group defined
by $\oL^*$ on $\cM(P),$ (the Borel measures of finite total variation) may not be
strongly continuous at $t=0.$ This is a reflection of the fact that
$\Dom(\oL^*)$ may fail to be dense in $\cM(P).$ A solution to this problem was
introduced by Lumer and Phillips, whereby we consider $e^{t\oL^*}$ acting on a
smaller space:
\begin{equation}
  \cM^{\odot}(P)=\overline{\Dom(\oL^*)}\cap(\oL^*-1)\Dom(\oL^*).
\end{equation}
The semi-group $e^{t\oL^*}$ is strongly continuous at $t=0$ when acting on this space.
\index{$\cM^{\odot}$}

Because $\oL^*$ is formally a non-degenerate elliptic operator in the interior
of $P$ it is clear that at positive times $e^{t\oL^*}\nu$ is represented at
interior points of $P$ by a measure with a smooth density. From this it is
apparent that elements of $\Dom(\oL^*)$ are represented by absolutely
continuous measures in the interior of $P.$

In adapted coordinates, a generalized Kimura diffusion takes the form:
 \begin{equation}
L=x\pa_x^2+b(x,\by)\pa_x+\sum_{l=1}^{N-1}xa_{1l}(x,\by)\pa_x\pa_{y_l}+(1+O(x))L_{\Sigma},
\end{equation}
near to a hypersurface boundary component. Integrating by parts, we
see that an element  $gdxd\by\in\Dom(\oL^*)$ that is smooth in $\Int P,$ and
has no support, as a measure, on $bP,$ must satisfy the boundary condition:
\begin{equation}\label{eqn13.10.001}
  \lim_{x\to 0^+}\left[b(x,\by)g(x,\by)-\left(1+\sum_{l=1}^{N-1}a_{1l}(x,\by)\right)
\pa_x(xg(x,\by))\right]=0.
\end{equation}
Generally, this condition forces $g$ to have a complicated singularity along
$\{x=0\}.$ 

This begs the question of whether or not an element of $\Dom(\oL^*)$ can have
atomic support along $bP.$ We do not answer the general question, but show that
a non-negative measure solution $\nu$ to $\oL^*\nu=0$ cannot have such a
component along a face of $bP$ to which $L$ is transverse. Let $\{P_n\}$ be an
exhaustion of $P$ by a nested sequence of compact subsets with $P_n\subsubset
\Int P.$ We define the measure $\nu_i$ via the equation
\begin{equation}
  \langle f,\nu_i\rangle=\lim_{n\to\infty}\langle\chi_{P_n}f,\nu\rangle.
\end{equation}
If $\nu$ is a non-negative measure, then evidently $\nu_b=\nu-\nu_i$ is as
well, and the support of $\nu_b$ is contained in $bP.$ 

As noted above, in $\Int P,$ $\oL^*\nu_i=0,$ in the classical sense.  
For $f$ with compact support in $\Int P$ we have that 
\begin{equation}
  \langle Lf,\nu_i\rangle=\langle f,L^*\nu_i\rangle=0
\end{equation}
A simple limiting argument then implies that $\langle Lf,\nu_i\rangle=0$ for
$f\in\Dom(\oL)$ with $f\restrictedto_{bP}=0.$ If $f\in\CI(P),$ with
$f\restrictedto_{bP}=0,$ and $\pa_xf=0$ outside a small neighborhood,
$U\subsubset \Int\Sigma$ of $p,$ it follows from these observations that
\begin{equation}\label{eqn13.13.001}
\begin{split}
  \langle Lf,\nu\rangle&=\langle Lf\restrictedto_{U},\nu_b\rangle\\
&= \langle b(0,\cdot)\pa_xf,\nu_b\rangle.
\end{split}
\end{equation}
In order for $\nu$ to belong to $\Dom(\oL^*),$ there must be a constant $C$ so
that
\begin{equation}
  |\langle Lf,\nu\rangle|\leq C\|f\|_{L^{\infty}}.
\end{equation}
If $L$ meets $bP$ cleanly, then this along with~\eqref{eqn13.13.001} implies
that $\supp\nu_b$ is disjoint from the interior of any face of $bP$ to which
$L$ is transverse. This reasoning can be applied recursively to the
stratification of $bP$ to prove the following result:
\begin{proposition} If $L$ is transverse to $bP,$ then any non-negative measure
  $\nu\in\Dom(\oL^*)$ solving $\oL^*\nu=0$ is represented by a smooth density
  supported in $\Int P,$ which satisfies the boundary condition
  in~\eqref{eqn13.10.001}.
\end{proposition}

Below we show that if $L$ is everywhere transverse to $bP,$
then there is a unique solution $\nu\in\Dom(\oL^*)$ to $L^*\nu=0,$
which is a probability measure. As explained in~\cite{KarlinTaylor2},
section 15.2, there are circumstances where there may be multiple
solutions to this equation, which are non-negative and
normalizable. Evidently our method picks out the solution that
satisfies the boundary condition in~\eqref{eqn13.10.001}.  A more
detailed analysis of this and related questions will need to wait for
a later publication.

\section{The nullspace of $\oL^*$}
As noted above, we are, at present, missing the compactness of the resolvent of
$\oL.$ We can nonetheless give a precise description of the null-space of the
adjoint, $\oL^*,$ under the hypothesis that $L$ meets $bP$ cleanly. For the
following result it suffices to consider the operator acting on
$\cC^{0,2+\gamma}_{\WF}(P),$ for a $0<\gamma<1.$
\begin{proposition}\label{prop13.1.1} Suppose that $L$ meets $bP$ cleanly. To each
  element of $bP_{\ter}(L)$ there is an element of the nullspace of
  $\oL^*.$ These are represented by non-negative measures supported on
  $\Sigma\in bP_{\ter}(L) .$
\end{proposition}
\begin{proof}  
For any $0<\gamma<1$, denote by $L_{\gamma}$ the operator 
\begin{equation}
L_{\gamma}:\cC^{0,2+\gamma}_{\WF}(P)\longrightarrow \cC^{0,\gamma}_{\WF}(P).
\end{equation}
We have established that this map is Fredholm; in fact, this map has index zero 
since it can be deformed amongst Fredholm operators to $L_\gamma - 1$, which is
invertible. Thus 
\begin{equation}
\dim \Ker L_{\gamma}=\dim \Ker L_{\gamma}^*.
\end{equation}
For the remainder of the argument we fix a $0<\gamma<1.$

Consider first the extreme case that $L$ is transverse to every boundary
hypersurface. It then follows from Lemma~\ref{terface} that $\ker(L_{\gamma})$ consists
of constant functions. Moreover, using this same lemma, if $f \not \equiv 0$ is
continuous and nonnegative, then the equation
\begin{equation}
L_{\gamma}w=f
\end{equation}
is not solvable, since any solution would be a subsolution of $L_{\gamma}$.  

The adjoint operator, $L_{\gamma}^*$ acts canonically as a map from
$[\cC^{0,\gamma}_{\WF}(P)]^*$ to $[\cC^{0,2+\gamma}_{\WF}(P)]^*.$
Since we are still assuming that $bP_{\ter}(L) = P$, $\ker(L_\gamma)$
contains only the constant functions, so there is precisely one
non-trivial element $\ell\in [\cC^{0,\gamma}_{\WF}(P)]^*,$ unique up
to scaling, which satisfies $L_{\gamma}^*\ell=0.$ By the Fredholm
alternative, the equation $L_{\gamma}w=f$ is solvable for $f\in
\cC^{0,\gamma}_{\WF}(P)$ if and only if
$$
\ell(f)=0.
$$ 
This means that if $f \in \cC^{0,\gamma}_{\WF}(P)$ is nonnegative (and nonzero), then 
$L_\gamma w = f$ is not solvable, so that $\ell(f)\neq 0.$  We may as well assume that 
\begin{equation}\label{eqn12.11.055}
  \ell(f)>0 
\end{equation}
on the set of nonnegative functions; we further normalize so that $\ell(1) = 1$. 

A priori, we only know  that $\ell$ lies in the dual of a H\"older
space, and thus could be a distribution of negative order.
If $f\in \cC^{0,\gamma}_{\WF}(P),$ then
\begin{equation}
f^+(p)=\max\{f(p),0\}\text{ and } f^-(p)=\min\{f(p),0\}
\end{equation}
both lie in this same function space, and therefore~\eqref{eqn12.11.055}
and our normalization imply that
\begin{equation}
  \min f\leq \ell(f)\leq \max f,
\end{equation}
and therefore,  for
$f\in \cC^{0,\gamma}_{\WF}(P),$ we have
\begin{equation}
|\ell(f)|\leq \|f\|_{L^{\infty}}.
\end{equation}
The WF H\"older spaces are dense in $\cC^0(P)$, so $\ell$ has a unique
extension as an element of $[\cC^0(P)]'.$ By the Riesz-Markov theorem, there is
a non-negative Borel measure, $d\nu$ so that
\begin{equation}
\ell(f)=\int\limits_{P}f(p)d\nu(p).
\end{equation}

The adjoint $\oL^*$ is elliptic in $\Int P$, so by standard elliptic regularity, 
\begin{equation}
 d\nu\restrictedto_{\Int P}=v_0dV,
\end{equation}
for some smooth, non-negative function $v_0$ on $\Int P$; here $dV$ is a smooth non-degenerate 
density on $P.$  Using \eqref{eqn12.11.055} again, we see that the support of $v_0$ is all of $P.$ 
Note, however, that since $\oL^*$ can have a zero order part, there is no obvious reason 
that $v_0$ is strictly positive. 

Let us now turn to components $\Sigma\in bP_{\min}^{T}(L).$  If $\dim\Sigma=0,$ 
so $\Sigma=p_{\Sigma}$ is a single point, then the fact that $L$ is tangent to $\Sigma$
simply means that the restriction $L_{\Sigma}\equiv 0.$ Hence if $\delta_{\Sigma}$ denotes
the functional
\begin{equation}
\langle w,\delta_{\Sigma}\rangle=w(p_\Sigma),
\end{equation}
then clearly, for $w\in\cC^{0,2+\gamma}_{\WF}(P)$, we have
\begin{equation}
\langle Lw,\delta_{\Sigma}\rangle=0,
\end{equation}
and this equation remains true for $w \in \Dom(\oL)$. Hence $\delta_{\Sigma}\in \Dom(\oL^{\, *}),$
and  
\begin{equation}
\oL^{\, *}\delta_{\Sigma}=0.
\end{equation}

Suppose, on the other hand, that $\dim\Sigma>0,$ i.e., $\Sigma$ is a compact
manifold without boundary, and $L_{\Sigma}$ is a non-degenerate elliptic
operator, without constant term, acting on $\cC^2(\Sigma).$ Clearly 
$L_{\Sigma} 1=0$. On the other hand, the strong maximum principle shows that 
all solutions to $L_{\Sigma}w=0$ are constant. Also from the strong maximum principle, 
the equation $L_{\Sigma}w=f$ is not solvable whenever $f\in\cC^0(\Sigma)$ is non-negative
and not identically zero. Arguing as above for the case that $bP_{\ter}(L)=P$ we conclude that 
there is a non-negative measure with smooth density, $d\nu=v_0dV_{\Sigma}$ that spans the
nullspace of $L_{\Sigma}^*.$ The functional
\begin{equation}
\ell(w)=\int\limits_{\Sigma}wd\nu,
\end{equation}
defines an element of $\Ker \oL^*.$  As before, the support of $v_0$ is all of $\Sigma$. 

To complete the construction of $\Ker \oL^*$ we need only consider elements $\Sigma\in
bP_{\ter}(L)\setminus bP^T_{\min}(L).$ In this case $L_{\Sigma}$ is a generalized Kimura diffusion 
on $\Sigma,$  and $b\Sigma_{\ter}(L_{\Sigma})=\Sigma.$ The argument above produces a measure
$d\nu$ with support equal to $\Sigma$ and such that $L_{\Sigma}^*d\nu=0.$ If we define
\begin{equation}
\ell(w)=\int\limits_{\Sigma}wd\nu,
\end{equation}
then $\oL^*\ell=0.$  This completes the proof of the proposition.
\end{proof}

\begin{definition} We denote by $\{\delta_{\Sigma}:\: \Sigma\in
  bP_{\ter}(L)\},$ the measures, belonging to $\Ker\oL^*,$ constructed
  in the proof of Proposition~\ref{prop13.1.1}.
\end{definition}
These measures define non-trivial functionals on $\cC^{0}(P)\supset
\cC^{0,\gamma}_{\WF}(P),$ and are certainly linearly independent.
This argument shows that $\dim\Ker L_{\gamma}^*\geq |bP_{\ter}(L)|.$
On the other hand, by Corollary~\ref{cor3.1.000}, $\dim\Ker
L_{\gamma}\leq |bP_{\ter}(L)|.$

We summarize all of this in a proposition:
\begin{proposition} If $L$ meets $bP$ cleanly, then, for any $0<\gamma<1,$
\begin{equation}
\dim \Ker L_{\gamma}^*=\dim \Ker L_{\gamma}= |bP_{\ter}(L)|. 
\end{equation}
The $\Ker L_{\gamma}$ is contained in $\cC^{\infty}(P)$; on the other
hand, $\Ker L_{\gamma}^*$ is spanned by a finite collection of
non-negative Borel measures, each of which has a smooth nonnegative
density supported on one of the terminal boundary components of
$P$. The operator $L$ has no generalized eigenvectors at $0,$ i.e.,
functions $w\in\cC^{0,2+\gamma}_{\WF}(P)$ with $Lw\in\Ker L.$
\end{proposition}

\begin{remark}
If $bP_{\min}^T(L)=\emptyset,$  and  $bP_{\ter}(L)=P,$ then 
\begin{equation}
\dim \Ker L_{\gamma}^*=\dim \Ker L_{\gamma}= 1.
\end{equation}
The $\Ker L_{\gamma}^*$ is spanned by a non-negative measure with
support all of $P$. This is the equilibrium measure.  If
$|bP_{\ter}(L)|>1,$ then instead of a single equilibrium measure,
there is a collection of such measures, each supported on one of the
terminal components of $bP.$ Zero-dimensional components of
$bP_{\ter}(L)$ are classical absorbing states of the underlying Markov
process. Higher dimensional components correspond to generalized
absorbing states; these are again characterized by an equilibrium
measure.
\end{remark}
\begin{proof} Only the last statement still requires proof. If
  $bP_{\ter}(L)=P,$ then $\Ker L_{\gamma}$ consists of constant
  functions. We observe that $L_{\gamma} w = 1$ is not solvable, for
  otherwise $w$ would be a non-trivial subsolution.

  Suppose that $bP_{\ter}(L)\neq P$ and that
  $w\in\cC^{0,2+\gamma}_{\WF}(P)$ satisfies $Lw=f,$ where $Lf=0.$ Then
  necessarily
\begin{equation}
\langle f,\delta_{\Sigma}\rangle=0\text{ for all }\Sigma\in bP_{\ter}(L).
\end{equation}
However, any $f \in \Ker L\cap \cC^{0,2+\gamma}_{\WF}(P)$ is constant on
each component of $bP_{\ter}(L).$ Since each of the measures
$\{\delta_{\Sigma}\}$ is non-negative and non-trivial,
Proposition~\ref{prop.kdbvs} shows that $f\equiv 0.$
\end{proof}

We have not proved that all elements of $\Ker\oL$
  belong to $\cD^2_{\WF}(P)$, nor have we established the Hopf maximum
  principle for elements of $\Dom(\oL)$, hence we cannot presently conclude
  that $\dim\Ker\oL=\dim\Ker\oL^*$. On the other hand, elements of
  $\Ker \oL^*$ are represented by Borel measures, and furthermore
  $\Ker\oL^*\subset \Ker L_{\gamma}^*.$ Since we have shown that $\Ker
  L_{\gamma}^*$ is also spanned by Borel measures, we obtain:
\begin{proposition} If $L$ meets $bP$ cleanly, then
\begin{equation}
\dim\Ker \oL^*=|bP_{\ter}(L)|.
\end{equation}
The nullspace is spanned by Borel measures with support on the components of $bP_{\ter}(L).$
\end{proposition}

\section{Long Time Asymptotics}
These observations have several interesting consequences. 
\begin{proposition}\label{prop13.2.1}
If $f$ in $\cC^0(P),$ and $e^{t\oL}f$ denotes 
the action of the semi-group, then the functions
\begin{equation}
t \mapsto \langle e^{t\oL}f,\delta_{\Sigma}\rangle
\end{equation}
are constant for every $\Sigma\in bP_{\ter}(L)$. 
\end{proposition}
\begin{proof} Indeed, this is clear when $f\in \Dom(\oL)$ since
\begin{equation}
\begin{split}
\pa_t\langle e^{t\oL}f,\delta_{\Sigma}\rangle &=\langle \oL e^{t\oL}f,\delta_{\Sigma}\rangle\\
&=\langle e^{t\oL}f,\oL^*\delta_{\Sigma}\rangle\\
&=0
\end{split}
\end{equation}
However, the domain $\Dom(\oL)$ is dense in $\cC^0(P)$, so for any $f \in \cC^0(P)$ we can choose
a sequence $<f_n>$ in $\Dom(\oL)$ which converges to $f$ in  $\cC^0(P)$. Then
\begin{equation}
\langle e^{t\oL}f,\delta_{\Sigma}\rangle=\lim_{n\to\infty}\langle e^{t\oL}f_n,\delta_{\Sigma}\rangle.
\end{equation}
The right hand side is independent of $t$ for each $n$, hence so is the limit. If $bP_{\ter}(L)=P,$ 
then we can also conclude that
\begin{equation}
\langle e^{t\oL}f,\delta_{P}\rangle
\end{equation}
is constant.
\end{proof}
\begin{remark} A similar observation, for a special case, appears in~\cite{ChalubSouza}. 
\end{remark}

We now show that $0$ is the only element in the spectrum of $L_{\gamma}$ on the imaginary axis, $\Re\mu=0.$
\begin{lemma}\label{lem12.0.1} Let $P$ be a compact manifold with corners, and $L$ a generalized Kimura 
diffusion on $P.$  If $\varphi\in\cC^{0,2+\gamma}_{\WF}(P)$ is a non-trivial solution to $L\varphi=i\alpha\varphi,$ 
for $\alpha\in\bbR,$ then $\alpha=0.$
\end{lemma}
\begin{proof}
Let $q_t(x,dy)$ denote the Schwartz kernel for $e^{tL}$.  Then for each $x \in P$, 
\begin{equation}\label{eqn12.56.101}
  \int\limits_{P}q_t(x,dy)=1.
\end{equation}
Also, $q_t(x,dy)$ is a non-negative measure. By the strong maximum
principle, $e^{tL}$ is strictly positivity improving within $\Int P.$ Hence
if $U\subset \Int P$ is any open subset, then
\begin{equation}\label{eqn12.57.001}
\int\limits_{U}q_t(x,dy)>0
\end{equation}
for each $x \in \Int P$. 

Now, 
\begin{equation}\label{eqn12.55.001}
e^{it\alpha}\varphi(x)=e^{tL}\varphi=\int\limits_{P}q_t(x,dy)\varphi(y),
\end{equation}
so by the non-negativity of $q_t(x,dy)$, 
\begin{equation}\label{eqn12.56.001}
|\varphi(x)|\leq \int\limits_{P}q_t(x,dy)|\varphi(y)|=[e^{tL}|\varphi|](x).
\end{equation}
Note also that $|\varphi|$ lies in $\cC^{0,\gamma}_{\WF}(P)$, so $e^{tL}|\varphi|\in\cC^{\infty}(P)$
for $t > 0$.  The estimate in~\eqref{eqn12.56.001} implies that for any $s > 0$, 
\begin{equation}
0\leq \frac{[e^{(s+t)L}|\varphi|](x)-[e^{sL}|\varphi|](x)}{t}.
\end{equation}
Since $e^{sL}|\varphi|\in \Dom(L_{\gamma}),$ we can let $t\to 0^+$ to conclude that 
\begin{equation}
Le^{sL}|\varphi|(x)\geq 0\text{ for all }x\in P.
\end{equation}

Choose any non-negative $\psi\in\cC^{\infty}_c(\Int P)$. Then integrating by parts with respect to
some smooth non-degenerate density on $P$ gives
\begin{equation}
0\leq \langle e^{sL}|\varphi|,L^*\psi\rangle,
\end{equation}
so letting $s\to 0^+$, we obtain that
\begin{equation}
0\leq \langle |\varphi|,L^*\psi\rangle.
\end{equation}
If the support of $\psi$ is further constrained to lie in a set where $|\varphi|$ is smooth, 
then we can integrate by parts again to conclude that
\begin{equation}
 0\leq \langle L|\varphi|,\psi\rangle.
\end{equation}
In particular, $L |\varphi| \geq 0$ in the open subset of $\Int P$ where $|\varphi| > 0$. 

There are now several cases to consider. If $bP^T(L)=\emptyset,$ then Lemma~\ref{hopfmaxpriple} 
shows that $|\varphi(x)|\equiv 1.$   We have
\begin{equation}
1 = |\varphi(x)| = \left| \int q_t(x,dy) \varphi(y) \right| \leq \int q_t(x,dy) |\varphi(y)| = \int q_t(x,dy) = 1,
\end{equation}
so the inequality in the middle is an equality, and since $\int
q_t(x,dy) = 1$, this can only happen if $\varphi$ has constant
phase. This shows that $\alpha = 0$ in this case.

If $bP^T(L)\neq\emptyset,$ then we use an induction on the dimension
of $P$.  The result has been proved when $\dim P=1$ in~\cite{Feller1}
and~\cite{WF1d}, so we now assume that it is true whenever $\dim P\leq
N-1.$ Let $P$ have dimension $N$ and assume that $bP^T(L)
\neq\emptyset.$ Suppose that $\varphi$ is a non-trivial solution, as
above, and that $\|\varphi\|_{\cC^0(P)}=1.$ For each $\Sigma\in
bP^T(L)$ we know that
\begin{equation}
L_{\Sigma}\, \varphi \!\! \restrictedto_{\Sigma}=i\alpha \varphi \! \! \restrictedto_{\Sigma}.
\end{equation}
By induction, either $\alpha=0$ or $\varphi\!\!
\restrictedto_{\Sigma}=0.$ In the former case we are done, so we can
reduce to the case that $\varphi\!\! \restrictedto_{\Sigma}=0$ for
every $\Sigma\in bP^T(L).$

If $L$ is tangent to every face of $bP,$ then $|\varphi|$ attains its
maximal value $1$ at some point $x_0\in\Int P$. It follows directly
from~\eqref{eqn12.56.101} and~\eqref{eqn12.57.001} that
$|\varphi(x)|\equiv 1$ in $\Int P$. For if this were false, then the
fact that $e^{tL}$ is strictly positivity improving in $\Int P$ would
show that
\begin{equation}
\int_P q_t(x_0,dy) |\varphi(y)| < 1,
\end{equation}
which contradicts \eqref{eqn12.56.001}. By induction $\varphi$ vanishes on $bP,$ which is clearly impossible, as $\varphi\in\cC^0(P).$

We are left to consider the case where $bP^T(L)\neq bP.$ If $|\varphi(x)|$ assumes its maximum 
in the interior of $P$ then we conclude as above that $|\varphi(x)|\equiv 1$ in $P$, which
leads to the same contradiction as before. Thus $|\varphi(x)|$ must assume the value $1$ 
at $x_0\in bP\setminus bP^T(L).$ Indeed $x_0\in bP^{\pitchfork}(L),$ for otherwise $x_0$ would belong
to the closure of $bP^T(L)$ and hence would vanish. Applying Lemma~\ref{hopfmaxpriple} gives 
that $|\varphi|$ is identically equal to $1$ in a neighborhood $U$ of $x_0.$  Using the
previous argument at $x_1\in U\cap \Int P$ gives $|\varphi(x)|\equiv 1$ in $P$, which
contradicts that $\varphi\restrictedto_{bP^T(L)}=0.$ 

Thus the only tenable case is that $\alpha=0$, as claimed. 
\end{proof}

Combining this Lemma with Theorem~\ref{thm.resolP} and Lemma~\ref{lem12.0.1} gives the following corollary:
\begin{corollary}\label{cor12.0.2} For any $0<\gamma<1,$ $\Spec(L_{\gamma})\setminus \{0\}$
lies in a half plane $\Re\mu<\eta<0.$
\end{corollary}

Write $N_0=|bP_{\ter}(L)|$ and let $\{\Sigma_j:\: j=1,\dots,N_0\}$
enumerate the components of $bP_{\ter}(L)$.  Also, denote by
$\ell_{j}=\delta_{\Sigma_j}$ the probability measures which span $\Ker
\oL^*$, constructed above.  From the support properties of these
measures, we can choose smooth functions $\{f_j:\: j=1,\dots, N_0\}$
so that
\begin{equation}\label{eqn12.30.0001}
\ell_k(f_j)=\delta_{jk}.
\end{equation}
Choose a smooth basis $\{w_{j}:\:j=1,\dots,N_0\}$ for $\Ker
L_{\gamma},$ for any $0<\gamma<1.$ Corollary~\ref{cor12.0.2} and the
fact that $L_{\gamma}$ has no generalized eigenvectors at $0$, shows
that $\Spec L_{\gamma}\setminus \{0\}$ lies in $\Re\mu< \eta,$ for an
$\eta<0.$ Thus, as the spectrum lies in a conic neighborhood of the
negative real axis, we can deform the contour in~\eqref{eqn11.68.006},
to show there exist continuous linear functionals $\{a_j\}$ such that,
for any $f\in\cC^{0,\gamma}_{\WF}(P),$ we have
\begin{equation}
e^{tL}f=\sum_{j=1}^{N_0}a_j(f)w_j+O(e^{\eta t}).
\end{equation}
Proposition~\ref{prop13.2.1} shows that the quantities $\{
\ell_k(e^{tL}f)\}$ are independent of $t$, so letting $t \to \infty$,
we conclude that
\begin{equation}
\ell_k(f)=\sum_{j=1}^{N_0}a_j(f)\ell_k(w_j).
\end{equation}
In light of~\eqref{eqn12.30.0001}, $\ell_k(w_j)$ is an invertible matrix, so we can find a new basis $\{\tw_j\}$ for $\Ker
L_{\gamma}$ so that $\ell_j(\tw_k)=\delta_{jk}$ and therefore 
\begin{equation}\label{eqn12.31.001}
e^{tL}f=\sum_{j=1}^{N_0}\ell_j(f)\tw_j+O(e^{\eta t}).
\end{equation}
By duality we can conclude that if $\nu$ is a Borel measure belonging to
$\cM^{\odot}(P),$  then
\begin{equation}\label{13.47.001}
  e^{tL^*}\nu=\sum_{j=1}^{N_0}\langle \tw_j,\nu\rangle\ell_j+O(e^{\eta t}).
\end{equation}

The $\{\ell_j\}$ are non-negative measures with disjoint supports. Since the
forward Kolmogorov equation maps non-negative measures to non-negative
measures, it follows that the eigenfunctions $\{\tw_j\}$ must be
non-negative. As a special case of~\eqref{eqn12.31.001} note that
\begin{equation}\label{eqn12.33.001}
  1=\sum_{j=1}^{N_0}\ell_j(1)\tw_j.
\end{equation}
We summarize these results in a proposition.
\begin{proposition} For $0<\gamma<1,$ there is a basis for $\Ker L_{\gamma}$ consisting of non-negative smooth functions $\{\tw_j\}.$ There is an  $\eta<0,$ so that for initial data $f\in\cC^{0,\gamma}(P),$ the asymptotic formula~\eqref{eqn12.31.001} holds. For initial data $\nu\in\cM^{\odot}(P),$ the asymptotic formula~\eqref{13.47.001} holds.
\end{proposition}
\begin{remark} In the classical case, with $L=L_{\Kimura}$ and
  $P=\cS_{N},$ a basis for $\Ker L_{\Kimura}$ is given by the functions
  $\{1,x_1,\dots,x_N\}.$ It is very likely that~\eqref{eqn12.31.001}
  also holds for data in $\cC^0(P).$
\end{remark}

\appendix
\chapter{Proofs of Estimates for the Degenerate 1d Model}\label{prfsoflems}
This appendix contains proofs of estimates, used throughout the paper, of the
1-dimensional solution operators $k^b_t(x,y),$ which we recall is given by
\begin{equation}
  k^b_t(x,y)=\frac{y^{b-1}}{t}e^{-\frac{(x+y)}{t}}\psi_b\left(\frac{xy}{t^2}\right),
\end{equation}
where
\begin{equation}
  \psi_b(z)=\sum_{j=0}^{\infty}\frac{z^j}{j!\Gamma(j+b)}.
\end{equation}
This function has the following asymptotic development, as $z\to\infty:$
\begin{equation}
  \psi_b(z)\sim \frac{z^{\frac 14-\frac b2}e^{2\sqrt{z}}}{\sqrt{4\pi}}\left[
1+\sum_{j=1}^{\infty}\frac{c_{b,j}}{z^{\frac j2}}\right].
\end{equation}

In several of the arguments below we need a suitable replacement for the Mean
Value Theorem, that is valid for complex valued functions.
\begin{lemma}\label{MVTineq} Let $f$ be a continuously differentiable, complex
  valued function defined on the interval $[a,b].$ There is a point $c\in
  (a,b)$ such that
  \begin{equation}\label{MVTineqest}
    |f(b)-f(a)|\leq (b-a)|f'(c)|.
  \end{equation}
\end{lemma}
\begin{proof} As an immediate consequence of the fundamental theorem of
  calculus and the triangle inequality we see that
  \begin{equation}
    |f(b)-f(a)|\leq \int\limits_{a}^{b}|f'(y)|dy.
  \end{equation}
The estimate in the lemma now follows from the standard mean value theorem
applied to the differentiable function
\begin{equation}
  F(x)= \int\limits_{a}^{x}|f'(y)|dy.
\end{equation}
\end{proof}

The kernel functions $k^b_t(x,y)$ extend to be analytic for $\Re t>0,$ and we
prove estimates for the spatial derivatives of this analytic
continuation. These are needed to study the resolvent kernel, which for the
1-dimensional model problem is defined in the right half plane by
\begin{equation}
  R(\mu)=\int\limits_{0}^{\infty}e^{-\mu t}e^{t L_b}dt.
\end{equation}
The contour of integration can be deformed to lie along any ray $\arg t=\theta$
with $|\theta|<\frac{\pi}{2}.$ This provides an analytic continuation of
$R(\mu)$ to $\bbC\setminus (-\infty,0].$ In these arguments we let $t=\tau
e^{i\theta},$ where $\tau=|t|.$

\begin{remark}[Notational Convention] To simplify the notation in the ensuing
  arguments we let
  \begin{equation}
    e_{\phi}\overset{d}{=}e^{-i\phi}.
  \end{equation}
\end{remark}
\section{Basic kernel estimates}
\begin{lemmabis}\labelbis{lem1new} For  $b>0,$ $0<\gamma<1,$ and
  $0<\phi<\frac{\pi}{2},$ there are constants $C_{b,\phi}$ uniformly bounded with
  $b,$ so that for $t\in S_{\phi}$
  \begin{equation}
\int\limits_0^{\infty}|k^b_t(x,y)-k^b_t(0,y)|y^{\frac{\gamma}{2}}dy\leq 
C_{b,\phi}x^{\frac{\gamma}{2}}.
  \end{equation}
\end{lemmabis}
\begin{proof} We let $t=\tau e^{i\theta},$ where $|\theta|<\frac{\pi}{2}-\phi.$
Using the formula for $k^b_t$ we see that
\begin{equation}\label{estx01}
\begin{split}
\int\limits_0^{\infty}|k^b_t(x,y)-k^b_t(0,y)|y^{\frac{\gamma}{2}}dy&=  
\int\limits_0^{\infty}\left(\frac {y}{\tau}\right)^be^{-\cos\theta\frac{y}{\tau}}
\left|e^{-e_{\theta}\frac {x}{\tau}}
\psi_b\left(\frac{xye_{2\theta}}{\tau^2}\right)-\psi_b(0)\right|
y^{\frac{\gamma}{2}}\frac{dy}{y}\\
&\leq \int\limits_0^{\infty}w^be^{- \cos\theta w}
\left|e^{-\lambda e_{\theta}}\psi_b\left(\lambda w e_{2\theta}\right)-\psi_b(0)\right|
(w\tau)^{\frac{\gamma}{2}}\frac{dw}{w}.
\end{split}
\end{equation}
On the second line we let $w=y/\tau$ and $\lambda=x/\tau.$ 
We split the integral into a part, $I_1(t,\lambda),$ from $0$ to $1/\lambda,$
and the rest, $I_2(t,\lambda).$ We estimate the compact part first, using the
the FTC we see that
\begin{equation}
\begin{split}
  \psi_b(\lambda w e_{2\theta})-\psi_b(0)=&\lambda w
  e_{2\theta}\int\limits_0^1
\psi_b'(s\lambda w e_{2\theta})ds\\
=&M\lambda w e_{2\theta},
\end{split}
\end{equation}
where $M$ is a complex number satisfying:
\begin{equation}
  |M|\leq \sup_{z:\:|z|\leq 1}|\psi_b'(z)|.
\end{equation}

This gives
\begin{equation}
 I_1(t,\lambda)\leq 
C_b \tau^{\frac{\gamma}{2}}\int\limits_0^{\frac{1}{\lambda}}
w^{b+\frac{\gamma}{2}-1}e^{-\cos\theta w}
\left[ e^{-\cos\theta\lambda}\lambda w+|1-e^{-e_{\theta}\lambda}|\right]dw.
\end{equation}
The constant $C_b$ is uniformly bounded for $0<b<B.$  For $\lambda$ bounded,
and $|\theta|<\frac{\pi}{2}$ we
can estimate $|1-e^{-e_{\theta}\lambda}|$ by $\lambda$ and therefore
the integral can also be estimated by a constant times $\lambda.$ Altogether we get
\begin{equation}\label{estI11}
  I_1(t,\lambda)\leq 
C_{b,\theta}  \tau^{\frac{\gamma}{2}}\lambda=
C_{b,\theta}  x^{\frac{\gamma}{2}}\lambda^{1-\frac{\gamma}{2}}.
\end{equation}
As $\lambda$ is bounded, this is the desired estimate. Now we turn to
$\lambda\to\infty.$ In this case it is easy to see that the integral tends to
zero. As $\lambda>1,$ this implies that
\begin{equation}\label{estI12}
  I_1(t,\lambda)\leq C_{b,\theta}x^{\frac{\gamma}{2}}.
\end{equation}

We are left to estimate $I_2.$ In this case there is no cancellation between
the terms on the right hand side of~\eqref{estx01}. As $\psi_b(0)=1/\Gamma(b),$
it is elementary to see that, in all cases,
\begin{equation}\label{estI21}
  \tau^{\frac{\gamma}{2}}\int\limits_{\frac{1}{\lambda}}^{\infty}w^be^{-
    \cos\theta w}
|\psi_b(0)|
w^{\frac{\gamma}{2}}\frac{dw}{w}\leq C_{b,\theta}x^{\frac{\gamma}{2}}.
\end{equation}
To complete the proof, for this case, we need to estimate the other term, which
we denote $I'_{2}(t,\lambda).$ To that end we use the asymptotic expansion to
estimate $\psi_b(w\lambda):$
\begin{equation}
  |\psi_b(w\lambda e_{2\theta})|\leq C_b(w\lambda)^{\frac 14-\frac b2}e^{2\cos\theta\sqrt{w\lambda}}.
\end{equation}
Inserting this estimate gives:
\begin{equation}
  I'_{2}(t,\lambda)\leq 
C_b\tau^{\frac{\gamma}{2}}
\int\limits_{\frac{1}{\lambda}}^{\infty}
\left(\frac{w}{\lambda}\right)^{\frac{b}{2}-\frac{1}{4}}
e^{-\cos\theta(\sqrt{w}-\sqrt{\lambda})^2} w^{\frac{\gamma}{2}}\frac{dw}{\sqrt{w}}.
\end{equation}
Applying Lemma~\ref{lem5}
it is a simple matter to see that this is uniformly bounded by
$C_{b,\theta}x^{\frac{\gamma}{2}}\|f\|_{\WF,0,\gamma},$ for a constant bounded
when $b$ is bounded.
and therefore
\begin{equation}\label{estI22}
  I_2(t,\lambda)\leq C_{b,\theta}x^{\frac{\gamma}{2}}.
\end{equation}
Combining~\eqref{estI11},~\eqref{estI12},~\eqref{estI21} and~\eqref{estI22}
completes the proof of the lemma
\end{proof}

\begin{lemmabis}\labelbis{lem21new} For  $b>0,$ there is a constant $C_{b,\phi}$ so
  that, for $t\in S_{\phi}$
  \begin{equation}\label{lem21newest1p}
    \int\limits_{0}^{\infty}|k^b_t(x,z)-k^b_t(0,z)|dz\leq C_{b,\phi}\frac{x/|t|}{1+x/|t|}.
  \end{equation}
For  $0<c<1,$ there is a constant $C_{b,c,\phi}$ so that, if $cx_2<x_1<x_2,$
and $t\in S_{\phi}$ then
\begin{equation}\label{lem21newest2p}
  \int\limits_{0}^{\infty}|k^b_t(x_2,z)-k^b_t(x_1,z)|dz\leq
  C_{b,c,\phi}\left(\frac{\frac{|\sqrt{x_2}-\sqrt{x_1}|}{\sqrt{|t|}}}
{1+\frac{|\sqrt{x_2}-\sqrt{x_1}|}{\sqrt{|t|}}}\right).
\end{equation}
\end{lemmabis}
\begin{proof} First observe that Lemma~\ref{lem9.1.3.00} implies that, for
  $t\in S_{\phi}$ the integrals in~\eqref{lem21newest1p}
  and~\eqref{lem21newest2p} are always bounded by a constant $C_{\phi}.$ We
  start with the proof of~\eqref{lem21newest1p}. We let $t=\tau e^{i\theta},$
  with $|\theta|<\frac{\pi}{2},$ and set $w=z/\tau,$
  and $\lambda=x/\tau,$ then we see that the expression on the right hand side
  of~\eqref{lem21newest1p} equals
  \begin{equation}
    \int\limits_{0}^{\infty}w^{b-1}e^{-\cos\theta w}|e^{-e_{\theta}\lambda}\psi_b(\lambda we_{2\theta})-\psi_b(0)|dw.
  \end{equation}
  We split this into an integral over $[0,\frac{1}{\lambda}]$ and the rest. As
  $\lambda\to\infty$ it is clear that the compact part remains bounded, and as
  $\lambda\to 0,$ it is $O(\lambda).$ In the non-compact part we use the
  trivial bound when $\lambda\to\infty,$ and the asymptotic expansion when
  $\lambda\to 0.$ This latter term is easily seen to be bounded by
  $O(e^{-\frac{\cos\theta}{2\lambda}}),$ completing the proof in this case.

  To prove~\eqref{lem21newest2p}, we assume that $cx_2<x_1<x_2,$ and use the
  formula for $k^b_t;$ setting $w=z/\tau,$ $\lambda=x_1/\tau,$ and $\mu=x_2/x_1,$ we
  obtain:
\begin{multline}
  \int\limits_{0}^{\infty}|k^b_t(x_2,z)-k^b_t(x_1,z)|dz=\\
\int\limits_0^{\infty}w^{b-1}e^{-\cos\theta w}|e^{-e_{\theta}\lambda}\psi_b(\lambda
we_{2\theta})-e^{-e_{\theta}\mu\lambda}\psi_b(\mu\lambda we_{2\theta})|dw.
\end{multline}
We let $F(\mu)=e^{-e_{\theta}\mu\lambda}\psi_b(\mu\lambda we_{2\theta});$ 
from Lemma~\ref{MVTineq} it follows that
\begin{equation}
 |F(\mu)-F(1)|\leq(\mu-1)\lambda
e^{-\cos\theta\xi\lambda}|\psi_b(\xi\lambda we_{2\theta})-e_{\theta}w\psi_b'(\xi\lambda we_{2\theta})|,
\end{equation}
for a $\xi\in(1,\mu)\subset (1,\frac{1}{c}).$ We split the integral into the
part from $0$ to $1/\lambda,$  $I_-,$ and the rest, $I_+.$ Using the Taylor
expansion we see that
\begin{equation}\label{eqn232}
  I_-\leq C(\mu-1)\lambda\int\limits_0^{\frac{1}{\lambda}}w^{b-1}
e^{-\cos\theta(w+\lambda)}
\left[\frac{1}{\Gamma(b)}+w(1+\lambda)\right]dw 
\end{equation}
As $\lambda\to\infty$ this is bounded by a constant times
$\lambda^{1-b}(\mu-1)e^{-\lambda}.$ This in turn satisfies
\begin{equation}
  I_-\leq e^{-\frac{\lambda}{2}}\left(\frac{\sqrt{x_2}-\sqrt{x_1}}{\sqrt{|t|}}\right).
\end{equation}
As $\lambda\to 0$ the integral in~\eqref{eqn232} remains bounded and therefore
\begin{equation}
  I_-\leq C_{b,\theta}\left(\frac{x_2-x_1}{|t|}\right)=
C_{b,\theta}\left(\frac{\sqrt{x_2}-\sqrt{x_1}}{\sqrt{|t|}}\right)
\left(\frac{\sqrt{x_2}+\sqrt{x_1}}{\sqrt{|t|}}\right),
\end{equation}
which shows that
\begin{equation}
  I_-\leq C_{b,\theta}\sqrt{\lambda}(\sqrt{\mu}+1)\left(\frac{\sqrt{x_2}-\sqrt{x_1}}{\sqrt{|t|}}\right),
\end{equation}
thus completing this case.

Using the asymptotic expansions for $\psi_b$
and $\psi_b'$  we see that
\begin{equation}\label{eqn236000}
  I_+\leq C_b(\mu-1)\lambda\int\limits_{\frac{1}{\lambda}}^{\infty}
  \left(\frac{w}{\xi\lambda}\right)^{\frac{b}{2}-\frac{1}{4}}
  e^{-\lambda\cos\theta(\sqrt{\frac{w}{\lambda}}-\sqrt{\xi})^2}
\left|1-\sqrt{\frac{w}{\xi\lambda}}+O(\frac{1}{\sqrt{w\lambda}}+\frac{1}{\lambda})\right|
\frac{dw}{\sqrt{w}}.
\end{equation}
As $\lambda\to 0$ this satisfies an estimate of the form
\begin{equation}
  I_+\leq C_{b,\theta}\left(\frac{x_2-x_1}{|t|}\right)e^{-\frac{1}{2\lambda}}.
\end{equation}
To analyze the non-compact part as $\lambda\to\infty,$ we first note that we
are only interested in the case that
\begin{equation}
  \frac{\sqrt{x_2}-\sqrt{x_1}}{\sqrt{|t|}}\leq 1,
\end{equation}
for otherwise we use the trivial estimate. Dividing by $\sqrt{\lambda}$ we see
that this constraint is equivalent to
\begin{equation}\label{muest1}
  \sqrt{\mu}-1\leq\frac{1}{\sqrt{\lambda}},
\end{equation}
which clearly implies that $\mu\to 1$ as $\lambda\to\infty.$ 

We change variables in~\eqref{eqn236000} letting $z=\sqrt{w/\lambda},$ to
obtain:
\begin{equation}\label{eqn236001}
  I_+\leq C_b(\mu-1)\lambda\int\limits_{\frac{1}{\lambda}}^{\infty}
  z^{b-\frac 12}
  e^{-\lambda\cos\theta(z-\sqrt{\xi})^2}
\left|1-\frac{z}{\sqrt{\xi}}+O(\frac{1}{z\lambda}+\frac{1}{\lambda})\right|
\sqrt{\lambda}dz.
\end{equation}
To estimate this integral, we split the domain into three pieces
$[\frac{1}{\lambda},1],$ $[1,\sqrt{\mu}],$ and $[\sqrt{\mu},\infty].$ Recall
that $\xi\in[1,\sqrt{\mu}],$ and therefore, in the
first segment we see that
\begin{equation}
  |z-\sqrt{\xi}|>|z-1|,
\end{equation}
and in the third segment,
\begin{equation}
  |z-\sqrt{\xi}|>|z-\sqrt{\mu}|.
\end{equation}
With these observations we see that
\begin{multline}
 \label{eqn236002}
  I_+\leq C_b(\mu-1)\lambda\Bigg[\int\limits_{\frac{1}{\lambda}}^{1}
  z^{b-\frac 12}
  e^{-\lambda\cos\theta(z-1)^2}
\left[|1-z|+|\sqrt{\mu}-1|+O(\frac{1}{z\lambda}+\frac{1}{\lambda})\right]
\sqrt{\lambda}dz+\\
\int\limits_{1}^{\sqrt{\mu}}
  z^{b-\frac 12}
  e^{-\lambda\cos\theta(z-\sqrt{\xi})^2}
\left|1-\frac{z}{\sqrt{\xi}}+O(\frac{1}{z\lambda}+\frac{1}{\lambda})\right|
\sqrt{\lambda}dz+\\
\int\limits_{\sqrt{\mu}}^{\infty}
  z^{b-\frac 12}
  e^{-\lambda\cos\theta(z-\sqrt{\mu})^2}
\left[|\sqrt{\mu}-z|+|\sqrt{\mu}-1|+O(\frac{1}{z\lambda}+\frac{1}{\lambda})\right]
\sqrt{\lambda}dz\Bigg]
\end{multline}
Using Laplace's method to estimate the first and third terms, as well
as~\eqref{muest1}, we easily show that the sum of the three integrals is
bounded by $C_{b,c,\theta}/\sqrt{\lambda},$ which implies that 
\begin{equation}
 I_+\leq C_{b,c,\theta} \left(\frac{\sqrt{x_2}-\sqrt{x_1}}{\sqrt{|t|}}\right).
\end{equation}
This completes the proof of the lemma.
\end{proof}

\begin{lemmabis}\labelbis{lem5new} For  $b>0,$ $0<\gamma<1$ and $t\in S_{\phi},$
  $0<\phi<\frac{\pi}{2},$ there is a $C_{b,\phi}$ so that
  \begin{equation}\label{eqn126.00p}
    \int\limits_{0}^{\infty}|k^b_t(x,y)||\sqrt{x}-\sqrt{y}|^{\gamma}dy\leq
    C_{b,\phi} |t|^{\frac{\gamma}{2}}. 
  \end{equation}
For fixed $0<\phi,$  and $B,$ these constants are uniformly bounded for $0<b<B.$
\end{lemmabis}
\begin{proof}
  We let $t=\tau e^{i\theta},$ with $|\theta|<\frac{\pi}{2}-\phi,$ and set
  $w=y/\tau,\, \lambda=x/\tau,$ obtaining:
\begin{multline}
   \int\limits_{0}^{\infty}|k^b_t(x,y)||\sqrt{x}-\sqrt{y}|^{\gamma}dy=\\
 |t|^{\frac{\gamma}{2}}
\int\limits_0^{\infty}w^be^{-\cos\theta(w+\lambda)}|\psi_b(w\lambda e_{2\theta})|
|\sqrt{w}-\sqrt{\lambda}|^{\gamma}\frac{dw}{w}.
\end{multline}

We split this integral into the part from $0$ to $1/\lambda,$ $I_1,$ and the
rest, $I_2.$  Using the estimate
\begin{equation}\label{psiestnr0}
  |\psi_b(z)|\leq C_b\left(\frac{1}{\Gamma(b)}+z\right),
\end{equation}
we easily show that the compact part is uniformly bounded by
$C_{b,\theta}|t|^{\frac{\gamma}{2}},$ for a constant $C_{b,\theta}$ uniformly
bounded for $b<B$ and $|\theta|\leq \frac{\pi}{2}-\phi.$

In the non-compact part use the asymptotic expansion to
obtain that
\begin{equation}
  I_2\leq
  C_b|t|^{\frac{\gamma}{2}}\lambda^{\frac{1}{4}-\frac{b}{2}}
\int\limits_{\frac{1}{\lambda}}^{\infty}w^{\frac{1}{2}\left(b-\frac{1}{2}\right)}
e^{-\cos\theta(\sqrt{w}-\sqrt{\lambda})^2}
|\sqrt{w}-\sqrt{\lambda}|^{\gamma}\frac{dw}{\sqrt{w}}
\end{equation}
Lemma~\ref{lem5} shows that as $\lambda\to 0$ this is bounded by
$C_{b,\theta}|t|^{\frac{\gamma}{2}}\lambda^{-(b+\gamma-\frac{1}{2})}e^{-\frac{1}{\lambda}},$
showing that, in this regime~\eqref{eqn126.00} holds.

As $\lambda\to\infty,$ Lemma~\ref{lem5} shows that this is bounded by
$C_{b,\theta}|t|^{\frac{\gamma}{2}},$ thus completing
the argument to show that~\eqref{eqn126.00} holds for all $x,$ and $t\in S_{\phi}.$
\end{proof}

\begin{lemmabis}\labelbis{lem3new} We assume that $x_1/x_2>1/9$ and
  $J=[\alpha,\beta],$ as defined in~\eqref{eqn8555}.  For $b>0,$ $0<\gamma<1$
  and $0<\phi<\frac{\pi}{2},$ there is a $C_{b,\phi}$ so that, for $t\in
  S_{\phi}$
\begin{equation}
\int\limits_{J^c}|k^b_t(x_2,y)-k^b_t(x_1,y)||\sqrt{y}-\sqrt{x_1}|^{\gamma}dy
\leq C_{b,\phi}|\sqrt{x_2}-\sqrt{x_1}|^{\gamma}
\end{equation}
\end{lemmabis}
\begin{proof}
We let $t=\tau e^{i\theta},$ with $|\theta|<\frac{\pi}{2}-\phi,$ and set
\begin{equation}
  \begin{split}
    I_-&=
\int\limits_{0}^{\alpha}|k^b_t(x_2,y)-k^b_t(x_1,y)||\sqrt{y}-\sqrt{x_1}|^{\gamma}dy\\
I_+&=\int\limits_{\beta}^{\infty}|k^b_t(x_2,y)-k^b_t(x_1,y)||\sqrt{y}-\sqrt{x_1}|^{\gamma}dy
  \end{split}
\end{equation}
Since $x_1/x_2>1/9,$ we know that that $\alpha>0.$
\begin{equation}
I_-\leq \int\limits_{0}^{\alpha}
\left(\frac{y}{\tau}\right)^be^{-\cos\theta\frac y\tau}\left|
e^{-\frac{x_2}{t}}\psi_b\left(\frac{x_2y}{t^2}\right)-
e^{-\frac{x_1}{t}}\psi_b\left(\frac{x_1y}{t^2}\right)
\right||\sqrt{x_1}-\sqrt{y}|^{\gamma}\frac{dy}{y}.
 \end{equation}
As usual we let $y/\tau=w$ and $x_1/\tau=\lambda,$ obtaining
\begin{equation}\label{eqn89}
  I_-\leq 
|t|^{\frac{\gamma}{2}}\int\limits_{0}^{\frac{\alpha}{|t|}}
w^be^{-\cos\theta w}\left|
e^{-\frac{x_2}{x_1}e_{\theta}\lambda}\psi_b\left(\frac{x_2}{x_1}w\lambda e_{2\theta}\right)-
e^{-e_{\theta}\lambda }\psi_b\left(w\lambda e_{2\theta}\right)
\right||\sqrt{\lambda}-\sqrt{w}|^{\gamma}\frac{dw}{w}.
\end{equation}
The upper limit of integration can be re-expressed as
\begin{equation}
  \frac{\alpha}{|t|}=\lambda\frac{\left(3-\sqrt{\frac{x_2}{x_1}}\right)^2}{4}.
\end{equation}
As before we use Lemma~\ref{MVTineq} to obtain:
\begin{multline}\label{eqn92.0}
  \left|e^{-\frac{x_2}{x_1}e_{\theta}\lambda}\psi_b\left(\frac{x_2}{x_1}w\lambda
    e_{2\theta}\right)-
e^{-e_{\theta}\lambda}\psi_b\left(w\lambda e_{2\theta}\right)\right|\leq\\
\lambda e^{-\lambda\xi\cos\theta}\left|we_{\theta}\psi_b'(\xi\lambda w
e_{2\theta})-\psi_b(\xi\lambda w e_{2\theta})\right|
\left(\frac{x_2}{x_1}-1\right),
\end{multline}
where $\xi\in(1,\frac{x_2}{x_1})\subset (1,9).$

As in an earlier estimate we need to split this integral into the part from
$0$ to $1/\lambda$ and the rest.  In the first part, $I_{-1},$ we estimate
\begin{equation}\label{psibest}
|\psi_b(\xi\lambda w e_{2\theta})|\leq 
\frac{1}{\Gamma(b)}+C_b(\xi\lambda w)
\end{equation}
and $|\psi_b'(\xi\lambda w e_{2\theta})|$ by a
constant; in the second part, $I_{-2},$ we will use the asymptotic
expansions. The term in~\eqref{eqn89} coming from $1/\Gamma(b)$ is estimated by
\begin{equation}
  I_{-1}'\leq C_{\theta}|t|^{\frac{\gamma}{2}}
\lambda^{1+b+\gamma/2}\frac{e^{-\cos\theta\lambda}}{b\Gamma(b)}\left(\frac{x_2}{x_1}-1\right).
\end{equation}
We observe that this is bounded by
$C_{\theta}\|f\|_{\WF,0,\gamma}|\sqrt{x_2}-\sqrt{x_1}|^{\gamma}$ 
provided that
\begin{equation}
  \lambda^{b+\gamma/2}e^{-\cos\theta\lambda}\left(\sqrt{\frac{x_2}{|t|}}-
\sqrt{\frac{x_1}{|t|}}\right)^{1-\gamma}\left(\sqrt{\frac{x_2}{|t|}}+
\sqrt{\frac{x_1}{|t|}}\right)\leq C.
\end{equation}
As $c<x_1/x_2,$ we see that the quantity on the left is bounded by a multiple
of $\lambda^{b+1}e^{-\cos\theta\lambda},$ which remains bounded as $\lambda\to \infty.$
Thus, there is a constant $C_{\theta},$ independent of $b,$ so that
\begin{equation}\label{est-111}
  I_{-1}'\leq C_{\theta}|\sqrt{x_2}-\sqrt{x_1}|^{\gamma}.
\end{equation}
The other part of $I_{-1}$ (coming from the $C_b[\xi\lambda w+w]$-terms) is
easily seen to satisfy an estimate of the form 
\begin{equation}
  I_{-1}''\leq
C_{\theta}|t|^{\frac{\gamma}{2}}
\lambda^{2+b+\gamma/2}e^{-\cos\theta\lambda}\left(\frac{x_2}{x_1}-1\right),
\end{equation}
for a constant independent of $b.$ Arguing as before shows that this also
satisfies~\eqref{est-111}, so that $I_{-1}$ satisfies the desired estimate.

Using the asymptotic expansions we see that the other part, $I_{-2},$ satisfies
\begin{multline}
  I_{-2}\leq C|t|^{\frac{\gamma}{2}}
\lambda^{\frac 34-\frac b2}\left(\frac{x_2}{x_1}-1\right)\\
\int\limits_{\frac{1}{\lambda}}^{\frac{\alpha}{|t|}}e^{-\cos\theta(\sqrt{w}-\sqrt{\lambda\xi})^2}
w^{\frac b2-\frac 14}|\sqrt{w}-\sqrt{\xi\lambda}+O((w\lambda)^{-\frac 12})|
|\sqrt{w}-\sqrt{\lambda}|^{\gamma}
\frac{dw}{\sqrt{w}}.
\end{multline}
As $\xi>1,$ for $w\in [0,\frac{\alpha}{|t|}]$ the exponential is only increased
if we replace $\lambda\xi$ with $\lambda.$  For a large enough $C$ it is also
the case that, for $w$ in the domain of integration:
\begin{equation}
  \sqrt{\xi\lambda}-\sqrt{w}\leq C(\sqrt{\lambda}-\sqrt{w})
\end{equation}
Letting $z=\sqrt{\frac{w}{\lambda}}-1,$ we obtain
\begin{multline}
  I_{-2}\leq  C_b|t|^{\frac{\gamma}{2}}\lambda^{\frac 32+\frac{\gamma}{2}}
\left(\frac{x_2}{x_1}-1\right)\times\\
\int\limits_{\frac{1}{\lambda}-1}^{\frac{1}{2}\left(1-\sqrt{\frac{x_2}{x_1}}\right)}
e^{-\cos\theta\lambda z^2}(1+z)^{b-\frac 12}|z|^{1+\gamma}dz.
\end{multline}
We are interested in the case $x_2/x_1$ approaches 1, and $\lambda\to\infty.$
Even if  $b<\frac 12,$ then we see that the part of the integral near
to $z=-1$ contributes a term much like $I_{-1}.$ Lemma~\ref{lem3}
shows that
\begin{equation}
  I_{-2}\leq C_{b,\theta}I_{-1}+
C_{b,\theta}|t|^{\frac{\gamma}{2}}\lambda^{\frac 32+\frac{\gamma}{2}}
\left(\frac{x_2}{x_1}-1\right)
\frac{e^{-\frac{\cos\theta\left(\sqrt{x_1}-\sqrt{{x_2}}\right)^2}{4|t|}}
\left(\frac{\sqrt{x_2}-\sqrt{x_1}}{2|t|}\right)^{\gamma}}
{\lambda^{1+\frac{\gamma}{2}}}
\end{equation}
The complicated expression on the right hand side can be rewritten as
\begin{equation}
 C_{b,\theta}(\sqrt{x_2}-\sqrt{x_1})^{\gamma}
\left(\frac{\sqrt{x_2}-\sqrt{x_1}}{\sqrt{|t|}}\right)
\left(\frac{\sqrt{x_2}+\sqrt{x_1}}{\sqrt{x_1}}\right)
e^{-\frac{{\cos\theta}(\sqrt{x_2}-\sqrt{x_1})^2}{4|t|}},
\end{equation}
showing that 
\begin{equation}
  I_{-2}\leq C_{b,\theta}I_{-1}+
  C_{b,\theta}|\sqrt{x_2}-\sqrt{x_1}|^{\gamma}.
\end{equation}
which is precisely the bound that we need. The error term contributes a term of
this size times $\lambda^{-1},$ completing the analysis of this term

We now turn to $I_+;$ in this part the lower limit of integration is
  $$w=\frac{\beta}{|t|}=\frac{\lambda}{4}\left(3\sqrt{\frac{x_2}{x_1}}-1\right)^2\geq
  \lambda.$$ 
  If $\lambda<1,$ we need to split the integral into the part from $\beta/t$ to
  $1/\lambda,$ and use the Taylor expansion at zero the estimate the $\psi_b$-
  and $\psi_b'$-terms. If $\lambda>1,$ then we only need to use the asymptotic
  expansions of $\psi_b$ and $\psi_b'.$ For $\lambda<1,$ we have to estimate
  \begin{equation}
    \int\limits_{\frac{\beta}{t}}^{\frac{1}{\lambda}}w^be^{-\cos\theta w}
\left|
e^{-\frac{x_2}{x_1}\lambda e_{\theta}}\psi_b\left(\frac{x_2}{x_1}w\lambda e_{2\theta}\right)-
e^{-\lambda e_{\theta}}\psi_b\left(w\lambda e_{2\theta}\right)
\right||\sqrt{\lambda}-\sqrt{w}|^{\gamma}\frac{dw}{w}.
  \end{equation}
  Using~\eqref{eqn92.0} and arguing as before we can show that the contribution
  of this term is estimated by
  $C_{b,\theta}|t|^{\frac{\gamma}{2}}(x_2-x_1)/|t|.$ This, in turn, is
  estimated by 
$$C_{b,\theta}|\sqrt{x_2}-\sqrt{x_1}|^{\gamma}.$$ 

If $\lambda<1,$ then the contribution of the integral from $1/\lambda$ to
infinity is of the form
$e^{-\frac{\cos\theta}{2\lambda}}C_{b,\theta}|\sqrt{x_2}-\sqrt{x_1}|^{\gamma}.$

  Assuming now that $\lambda>1,$ we change variables as before, to see that
\begin{multline}
  I_+\leq 
C_b|t|^{\frac{\gamma}{2}}\lambda^{\frac{3}{4}-\frac{b}{2}}
\left(\frac{x_2}{x_1}-1\right)\times\\
\int\limits_{\frac{\lambda}{4}\left(3\sqrt{\frac{x_2}{x_1}}-1\right)^2}^{\infty}w^{\frac{b}{2}-\frac{1}{4}}
e^{-\cos\theta\left(\sqrt{w}-\sqrt{\lambda\xi}\right)^2}[\sqrt{\xi\lambda}-\sqrt{w}+ 
O((w\lambda)^{-\frac 12})|
|\sqrt{\lambda}-\sqrt{w}|^{\gamma}\frac{dw}{\sqrt{w}}.
\end{multline}
In this case $\xi\leq\sqrt{\frac{x_2}{x_1}},$ and so changing variables again,
as above, we see that the leading term is bounded by:
\begin{multline}
  I_+\leq 
C_b|t|^{\frac{\gamma}{2}}\lambda^{\frac 32+\frac{\gamma}{2}}
\left(\frac{x_2}{x_1}-1\right)\times\\
\int\limits_{\frac{3}{2}\sqrt{\frac{x_2}{x_1}}-\frac{1}{2}}^{\infty}
e^{-\lambda\cos\theta\left(z-\sqrt{\frac{x_2}{x_1}}\right)^2}z^{b-\frac 12}|z-1|^{\gamma+1}dz.
\end{multline}
This term, as well as the error term, satisfy the same estimates as those
satisfied by $I_{2-},$ thereby completing the proof of Lemma~\ref{lem3new}.
\end{proof}

\begin{lemmabis}\labelbis{lem4new} For $b>0,$ $0<\gamma<1$ and $c<1$ there is a
  $C_b$ such that if $c<s/t<1,$ then
\begin{equation}
  \int\limits_0^{\infty}\left|k_t^b(x,y)-k_s^b(x,y)\right|
|\sqrt{x}-\sqrt{y}|^{\gamma}dy\leq C_b|t-s|^{\frac{\gamma}{2}}.
\end{equation}
\end{lemmabis}
\begin{proof}
If we let $w=y/t,$ $\lambda=x/t$ and $\mu=t/s,$ then this becomes:
\begin{multline}
 \int\limits_0^{\infty}\left|k_t^b(x,y)-k_s^b(x,y)\right|
|\sqrt{x}-\sqrt{y}|^{\gamma}dy\leq\\
t^{\frac{\gamma}{2}}\int\limits_0^{\infty}\left|w^be^{-(w+\lambda)}\psi_b(w\lambda)-
(\mu w)^be^{-\mu(w+\lambda)}\psi_b(\mu^2w\lambda)\right|
|\sqrt{w}-\sqrt{\lambda}|^{\gamma}\frac{dw}{w}.
\end{multline}
We denote the quantity on the left by $K(x,s,t).$
If $F(\mu)=(\mu w)^be^{-\mu(w+\lambda)}\psi_b(\mu^2w\lambda),$ then the
difference  in the integral can be written:
\begin{equation}
  F(\mu)-F(1)=F'(\xi)(\mu-1)\text{ for a }\xi\in (1,\mu).
\end{equation}
Computing the derivative, we see that
\begin{multline}
 K(x,s,t)
\leq
t^{\frac{\gamma}{2}}(\mu-1)
\int\limits_0^{\infty}(\xi w)^be^{-\xi(w+\lambda)}\times\\
\left|\left(\frac{b}{\xi}-(w+\lambda)\right)\psi_b(\xi^2 w\lambda)+
2\xi w\lambda\psi_b'(\xi^2 w\lambda)\right|
|\sqrt{w}-\sqrt{\lambda}|^{\gamma}\frac{dw}{w}.
\end{multline}
As usual, we split this into a part, $I_1$ from $0$ to $1/\lambda,$ and the
rest, which we denote by $I_2.$

To bound $I_1$ we estimate $\psi_b'$ by a constant and use the
estimate of $\psi_b$ in~\eqref{psibest}. Arguing exactly as before we see
that
\begin{equation}
  I_1\leq C_bt^{\frac{\gamma}{2}}\left(\frac{t-s}{s}\right).
\end{equation}
The fact that $t/s<1/c,$ implies that there is a constant $C$ so that
\begin{equation}
  t^{\frac{\gamma}{2}}\left(\frac{t-s}{s}\right)\leq C(t-s)^{\frac{\gamma}{2}},
\end{equation}
showing that
\begin{equation}
   I_1\leq C_b(t-s)^{\frac{\gamma}{2}},
\end{equation}
for constants $C_b$ that are uniformly bounded for $0<\lambda<B.$

To estimate $I_2$ we use the asymptotic formul{\ae} for $\psi_b$ and $\psi_b'.$
It is straightforward to estimate the $\frac{b}{\xi}$-term. To estimate the
other two terms we need to take advantage of cancellations that occur, to
leading order, and then use the error terms in the asymptotic expansions to
estimate the remainder. The $\frac{b}{\xi}$-term is estimated by
\begin{equation}
  Cbt^{\frac{\gamma}{2}}(\mu-1)
\int\limits_{\frac{1}{\lambda}}^{\infty}w^be^{-\xi(\sqrt{w}-\sqrt{\lambda})^2}
(w\lambda)^{\frac{1}{4}-\frac{b}{2}}
|\sqrt{w}-\sqrt{\lambda}|^{\gamma}\frac{dw}{w}.
\end{equation}
As before we apply Lemma~\ref{lem5} to show that this integral is uniformly bounded for
$\lambda\in (0,\infty).$ This term is again bounded by
\begin{equation}
   C_bt^{\frac{\gamma}{2}}\left(\frac{t-s}{s}\right),
\end{equation}
which is handled exactly like $I_1.$ 

This leaves only
\begin{multline}
  I_2'=t^{\frac{\gamma}{2}}(\mu-1)\times\\
\int\limits_{\frac{1}{\lambda}}^{\infty}w^be^{-(w+\lambda)}
\left|2\xi w\lambda\psi_b'(\xi^2w\lambda)-(w+\lambda)\psi_b(\xi^2w\lambda)\right|
|\sqrt{w}-\sqrt{\lambda}|^{\gamma}\frac{dw}{w}
\end{multline}
Using the asymptotic expansions for $\psi_b$ and $\psi_b'$ this is bounded by
\begin{multline}
   I_2'\leq C_bt^{\frac{\gamma}{2}}(\mu-1)
\int\limits_{\frac{1}{\lambda}}^{\infty}w^be^{-\xi(\sqrt{w}-\sqrt{\lambda})^2}
(\xi^2 w\lambda)^{\frac 14-\frac{b}{2}}\times\\
\left|(\sqrt{w}-\sqrt{\lambda})^2+O((w\lambda)^{-\frac 12})\right|
|\sqrt{w}-\sqrt{\lambda}|^{\gamma}\frac{dw}{w}
\end{multline}
This is negligible as $\lambda\to 0;$
lemma~\ref{lem5} implies that the leading term is bounded by
$C_bt^{\frac{\gamma}{2}}\left(\frac{t-s}{s}\right),$ as
before, and that the error term is bounded by
$$\lambda^{-\frac 12}C_bt^{\frac{\gamma}{2}}\left(\frac{t-s}{s}\right).$$
This completes the proof of the lemma.
\end{proof}

The proof of Lemma~\ref{lem4new} also establishes the following simpler result:
\begin{lemmabis}\labelbis{lem4newp2} For  $b>0,$ there is a $C_b$ such that if $s<t,$ then
\begin{equation}
  \int\limits_0^{\infty}\left|k_t^b(x,y)-k_s^b(x,y)\right|dy\leq C_b
\left(\frac{t/s-1}{1+[t/s-1]}\right).
\end{equation}
\end{lemmabis}

We now consider the effects of scaling these kernels by powers of $x/y.$
\begin{lemmabis}\labelbis{lem10.0.1} If $0\leq\gamma\leq 1,$  and $b>\nu-\frac{\gamma}{2}>0,$ then there is a constant
  $C_{b,\phi},$ bounded for $b\leq B,$ and $B^{-1}<b+\frac{\gamma}{2}-\nu,$ so
  that, for $t\in S_{\phi},$ where $0<\phi<\frac{\pi}{2},$ we have the estimate
  \begin{equation}
    \int\limits_0^{\infty}
\left(\frac{x}{y}\right)^{\nu}|k_t^{b}(x,y)|y^{\frac{\gamma}{2}}dy\leq C_{b,\phi}x^{\frac{\gamma}{2}}.
  \end{equation}
\end{lemmabis}
\begin{proof} We let $t=\tau e^{i\theta},$ with $|\theta|<\frac{\pi}{2}-\phi.$
  Using $w=y/|t|$ and $\lambda=x/|t|,$ shows that we need to bound
  \begin{equation}
    |t|^{\frac{\gamma}{2}}
\int\limits_0^{\infty}\left(\frac{\lambda}{w}\right)^{\nu}w^b
e^{-\cos\theta(w+\lambda)}|\psi_b(w\lambda e_{2\theta})|
w^{\frac{\gamma}{2}}\frac{dw}{w}.
  \end{equation}
As usual we split this into an integral from $0$ to $1/\lambda$ and the
rest. The compact part we can estimate by
\begin{equation}
 C_b |t|^{\frac{\gamma}{2}}\lambda^{\nu}e^{-\cos\theta\lambda} 
\int\limits_0^{\frac{1}{\lambda}}w^{b+\frac{\gamma}{2}-\nu-1}e^{-\cos\theta w}dw.
\end{equation}
As $b+\nu-\frac{\gamma}{2}>0,$ the integral is clearly bounded uniformly in
$\lambda.$ Because $\nu-\frac{\gamma}{2}\geq 0,$ the contribution of this term is bounded by
\begin{equation}
  C_{b,\theta}x^{\frac{\gamma}{2}}\lambda^{\nu-\frac{\gamma}{2}}e^{-\cos\theta\lambda}
\leq C_{b,\theta}' x^{\frac{\gamma}{2}}.
\end{equation}

We use
the asymptotic expansion of $\psi_b$ to see that the non-compact part is
bounded by
\begin{equation}
  C_b|t|^{\frac{\gamma}{2}}
\int\limits_{\frac{1}{\lambda}}^{\infty}\left(\frac{w}{\lambda}\right)^{\frac{b}{2}-\frac{1}{4}-\nu}
e^{-\cos\theta(\sqrt{w}-\sqrt{\lambda})^2}w^{\frac{\gamma}{2}}\frac{dw}{\sqrt{w}}.
\end{equation}
The integral tends to zero like $e^{-\frac{\cos\theta}{2\lambda}}$ as $\lambda\to 0,$
showing that again this term is bounded by $Cx^{\frac{\gamma}{2}}.$ We let
$\sqrt{w}-\sqrt{\lambda}=z,$ to obtain that, as $\lambda\to \infty,$ this is bounded by
\begin{equation}
  C_b|t|^{\frac{\gamma}{2}}\lambda^{\nu+\frac{1}{4}-\frac{b}{2}}
\int\limits_{\frac{1}{\sqrt{\lambda}}-\sqrt{\lambda}}^{\infty}
(z+\sqrt{\lambda})^{b-\frac{1}{2}-2\nu+\gamma}e^{-\cos\theta z^2}dz
\end{equation}
It is again not difficult to see that, as $\lambda\to\infty,$ this is  bounded by
$C_{b,\theta}x^{\frac{\gamma}{2}}.$ 
\end{proof}

\begin{lemmabis}\labelbis{lem10.0.3.1}
  If $J=[\alpha,\beta],$ with $\alpha,\beta$ are given by~\eqref{eqn10.59.5},
  assuming that $x_1',x_1$ satisfy~\eqref{eqn10.58.4},  $b>0,$ and $0<\gamma\leq 1,$
  and $0<\phi<\frac{\pi}{2},$ there is a $C_{b,\phi}$ so that if $t\in
  S_{\phi},$ then,
\begin{equation}
  \int\limits_{J^c}
\left|k_t^{b+1}(x_1,z_1)\sqrt{\frac{x_1}{z_1}}-
k_t^{b+1}(x_1',z_1)\sqrt{\frac{x_1'}{z_1}}\right|
| \sqrt{z_1}- \sqrt{x_1'}|^{\gamma}dz_1\leq C_{b,\phi}| \sqrt{x_1}- \sqrt{x_1'}|^{\gamma}.
\end{equation}
 \end{lemmabis}
 \begin{proof} We let $t=\tau e^{i\theta},$ with $|\theta|<\frac{\pi}{2}-\phi.$
   Observe that it suffices to show that
  \begin{equation}
 I=    \int\limits_{J^c}
|k_t^{b+1}(x_1,z_1)|\left|\frac{\sqrt{x_1}-\sqrt{x_1'}}{\sqrt{z_1}}\right|
|\sqrt{z_1}- \sqrt{x_1'}|^{\gamma}dz_1\leq C_{b,\phi}| \sqrt{x_1}- \sqrt{x_1'}|^{\gamma},
  \end{equation}
and
\begin{equation}
II= \int\limits_{J^c}
\sqrt{\frac{x_1'}{z_1}}\left|k_t^{b+1}(x_1,z_1)-
k_t^{b+1}(x_1',z_1)\right|
| \sqrt{z_1}- \sqrt{x_1'}|^{\gamma}dz_1\leq C_{b,\phi}| \sqrt{x_1}-
\sqrt{x_1'}|^{\gamma}.
\end{equation}
The integral in $I$ is relatively simple to bound, and we can extend the
integral over $[0,\infty),$ rather than just over $J^c.$ Before switching the
domain of integration we observe that there is a constant $C$ so that if
$z_1\in J^c,$ then
\begin{equation}
  C^{-1}| \sqrt{z_1}- \sqrt{x_1}|\leq | \sqrt{z_1}- \sqrt{x_1'}|\leq C| \sqrt{z_1}- \sqrt{x_1}|
\end{equation}
It therefore suffices to show that
\begin{equation}
 I'=  \int\limits_{0}^{\infty}
|k_t^{b+1}(x_1,z_1)|\left|\frac{\sqrt{x_1}-\sqrt{x_1'}}{\sqrt{z_1}}\right|
| \sqrt{z_1}- \sqrt{x_1}|^{\gamma}dz_1\leq C_{b,\phi}| \sqrt{x_1}- \sqrt{x_1'}|^{\gamma},
\end{equation}
To estimate this integral we let $w=z_1/|t|$ and $\lambda=x_1/|t|,$ to obtain that
\begin{equation}\label{eqnA.79.1}
  I'=|t|^{\frac{\gamma-1}{2}}
|\sqrt{x_1}-\sqrt{x_1'}|\int\limits_0^{\infty}w^{b}e^{-\cos\theta(w+\lambda)}
|\psi_{b+1}(w\lambda e_{2\theta})|
|\sqrt{w}-\sqrt{\lambda}|^{\gamma}\frac{dw}{\sqrt{w}}.
\end{equation}

We observe that
\begin{equation}
  |t|^{\frac{\gamma-1}{2}}|\sqrt{x_1}-\sqrt{x_1'}|=
|\sqrt{x_1}-\sqrt{x_1'}|^{\gamma}\left[\left|1-\sqrt{\frac{x'_1}{x_1}}\right|\sqrt{\lambda}
\right]^{\frac{1-\gamma}{2}},
\end{equation}
 For bounded $\lambda$ it is easy to see that the
integral in~\eqref{eqnA.79.1} is bounded, so  we  need
to consider what happens as $\lambda\to\infty.$ As usual we split the integral
into the part over $[0,1/\lambda],$ and the rest. The compact part is estimated by
\begin{equation}
  I'_-\leq 
C_b|\sqrt{x_1}-\sqrt{x_1'}|^{\gamma}\lambda^{\frac{1-\gamma}{2}}e^{-\cos\theta\lambda}
\int\limits_0^{\frac{1}{\lambda}}w^be^{-\cos\theta w}|\sqrt{w}-\sqrt{\lambda}|^{\gamma}\frac{dw}
{\sqrt{w}}.
\end{equation}
Whether $\lambda$ is going to zero or infinity, we see that the contribution
of this term is bounded by $C_{b,\theta}|\sqrt{x_1}-\sqrt{x_1'}|^{\gamma}$

The non-compact is estimated using the asymptotic expansion for $\psi_{b+1}$  as
\begin{equation}
  I'_+\leq C_b|\sqrt{x_1}-\sqrt{x_1'}|^{\gamma}\lambda^{\frac{1-\gamma}{2}}
\int\limits_{\frac{1}{\lambda}}^{\infty}\left(\frac{w}{\lambda}\right)^{\frac{b}{2}-\frac{1}{4}}
e^{-\cos\theta(\sqrt{w}-\sqrt{\lambda})^2}|\sqrt{w}-\sqrt{\lambda}|^{\gamma}
\frac{dw}{\sqrt{w\lambda}}.
\end{equation}
As $\lambda\to 0,$ the integral is $O(e^{-\frac{1}{2\lambda}}).$ We let
$z=\sqrt{w}-\sqrt{\lambda},$ to obtain that
\begin{equation}
    I'_+\leq C|\sqrt{x_1}-\sqrt{x_1'}|^{\gamma}\lambda^{-\frac{\gamma}{2}}
\int\limits_{\frac{1}{\sqrt{\lambda}}-\sqrt{\lambda}}^{\infty}
\left(1+\frac{z}{\sqrt{\lambda}}\right)^{b-\frac{1}{2}}
e^{-\cos\theta z^2}|z|^{\gamma}dz,
\end{equation}
from which it follows easily that
\begin{equation}
  I\leq I'\leq C_{b,\theta}|\sqrt{x_1}-\sqrt{x_1'}|^{\gamma}
\end{equation}

We now turn to $II.$ With $y/|t|=w,$ $\lambda=x_1'/|t|,$ and $\mu=x_1/x_1',$ we
see that
\begin{multline}\label{eqnA.85.1}
  II\leq \sqrt{x_1'}|t|^{\frac{\gamma-1}{2}}\left[\int\limits_{0}^{\frac{\alpha}{|t|}}+
\int\limits_{\frac{\beta}{|t|}}^{\infty}\right]w^{b}e^{-\cos\theta w}
|e^{-\mu\lambda  e_{\theta}}\psi_{b+1}(\mu\lambda
w e_{2\theta})-e^{-\lambda  e_{\theta}}\psi_{b+1}(\lambda w
  e_{2\theta})|\times\\
|\sqrt{w}-\sqrt{\lambda}|^{\gamma}
\frac{dw}{\sqrt{w}}.
\end{multline}
Note that $1\leq\mu\leq 4.$ To estimate this term we use the
Lemma~\ref{MVTineq} to obtain:
\begin{multline}\label{eqnA.86.1}
 | e^{-\mu\lambda e_{\theta}}\psi_{b+1}(\mu\lambda
we_{2\theta})-e^{-\lambda e_{\theta}}\psi_{b+1}(\lambda we_{2\theta})
|\leq (\mu-1)\lambda e^{-\cos\theta\xi\lambda}\times\\
|w\psi_{b+2}(\xi\lambda
we_{2\theta})-\psi_{b+1}(\xi\lambda w e_{2\theta})|,\text{ where }\xi\in (1,\mu).
\end{multline}

The limits of integration in~\eqref{eqnA.85.1}  can be re-expressed as
\begin{equation}
  \frac{\alpha}{|t|}=\lambda\left(\frac{3-\sqrt{\mu}}{2}\right)^2\quad
 \frac{\beta}{|t|}=\lambda\left(\frac{3\sqrt{\mu}-1}{2}\right)^2.
\end{equation}
When we use the expression in~\eqref{eqnA.86.1} in~\eqref{eqnA.85.1}, we see
that the integral is multiplied by
\begin{equation}
  \sqrt{x_1'}|t|^{\frac{\gamma-1}{2}}(\mu-1)=|\sqrt{x_1}-\sqrt{x_1'}|^{\gamma}
[\sqrt{\lambda}(\sqrt{\mu}-1)]^{1-\gamma}(\sqrt{\mu}+1).
\end{equation}
We therefore need to show that 
\begin{multline}\label{eqnA.89.1}
 \lambda [\sqrt{\lambda}(\sqrt{\mu}-1)]^{1-\gamma}
\left[\int\limits_{0}^{\frac{\alpha}{|t|}}+
\int\limits_{\frac{\beta}{|t|}}^{\infty}\right]w^{b}e^{-\cos\theta (w+\xi\lambda)}\times\\
|we_{\theta}\psi_{b+2}(\xi\lambda
we_{2\theta})-\psi_{b+1}(\xi\lambda we_{2\theta})|
|\sqrt{w}-\sqrt{\lambda}|^{\gamma}\frac{dw}{\sqrt{w}}
\end{multline}
is uniformly bounded.

It is clear that the contribution of the integral from $0$ to
$\frac{\alpha}{|t|}$ is bounded for $\lambda$ bounded, so we only need to
evaluate the behavior of this term as $\lambda\to\infty.$ For this purpose we
need to split the integral into a part from $0$ to $1/\lambda,$ and the rest.
The part from $0$ to $1/\lambda$ is bounded by a constant times
$\lambda^{2+b}e^{-\cos\theta\lambda},$ and is therefore controlled.  The remaining
contribution is bounded by
\begin{multline}\label{eqnA.90.1}
 \lambda [\sqrt{\lambda}(\sqrt{\mu}-1)]^{1-\gamma}
\int\limits_{\frac{1}{\lambda}}^{\frac{\alpha}{|t|}}w^{b}e^{-\cos\theta(w+\xi\lambda)}
|we_{\theta}\psi_{b+2}(\xi\lambda
we_{2\theta})-\psi_{b+1}(\xi\lambda we_{2\theta})|\times \\
|\sqrt{w}-\sqrt{\lambda}|^{\gamma}\frac{dw}{\sqrt{w}} 
\end{multline}

Using the asymptotic expansions for $\psi_{b+1}$ and $\psi_{b+2}$ we see that
\begin{multline}\label{eqnA.91.1}
  |we_{\theta}\psi_{b+2}(\xi\lambda we_{2\theta})-\psi_{b+1}(\xi\lambda we_{2\theta})|\leq\\
C\sqrt{w}e^{2\cos\theta\sqrt{\xi\lambda w}}(\xi \lambda w)^{-\frac{b}{2}-\frac{3}{4}}\left[
|\sqrt{w}-\sqrt{\lambda\xi}|+O\left(\frac{1}{\sqrt{w}}+\frac{1}{\sqrt{\lambda}}\right)\right]
\end{multline}
The integral in~\eqref{eqnA.90.1} is bounded by
\begin{equation}\label{eqnA.93.1}
 \frac{C}{\lambda}\int\limits_{\frac{1}{\lambda}}^{\frac{\alpha}{|t|}}
\left(\frac{w}{\xi\lambda}\right)^{\frac{b}{2}-\frac{1}{4}}e^{-\cos\theta(\sqrt{w}-\sqrt{\xi\lambda})^2}
|\sqrt{w}-\sqrt{\lambda\xi}|^{\gamma}\left[|\sqrt{w}-\sqrt{\lambda}|+
O\left(\frac{1}{\sqrt{w}}+\frac{1}{\sqrt{\lambda}}\right)\right]\frac{dw}{\sqrt{w}} 
\end{equation}
In the interval of integration
$\sqrt{\lambda\xi}-\sqrt{w}>\sqrt{\lambda}-\sqrt{w},$ and
\begin{equation}
  \sqrt{\lambda\xi}-\sqrt{w}\leq C(\sqrt{\lambda}-\sqrt{w}),
\end{equation}
provided that $C>3.$ We can therefore estimate the leading term
in~\eqref{eqnA.93.1} by
\begin{equation}\label{eqnA.95.1}
 \frac{C}{\lambda}\int\limits_{\frac{1}{\lambda}}^{\frac{\alpha}{|t|}}
\left(\frac{w}{\lambda}\right)^{\frac{b}{2}-\frac{1}{4}}e^{-\cos\theta(\sqrt{w}-\sqrt{\lambda})^2}
|\sqrt{w}-\sqrt{\lambda}|^{1+\gamma}\frac{dw}{\sqrt{w}}.
\end{equation}
We let $z=\sqrt{w}-\sqrt{\lambda},$ to obtain that this is bounded by
\begin{equation}
  \frac{C}{\lambda}
\int\limits_{\frac{1}{\sqrt{\lambda}}-\sqrt{\lambda}}^{\sqrt{\frac{\alpha}{|t|}}-\sqrt{\lambda}}
\left(1+\frac{z}{\sqrt{\lambda}}\right)^{b-\frac{1}{2}}e^{-\cos\theta z^2}
|z|^{1+\gamma}dz,
\end{equation}
with the upper of limit of integration given by
\begin{equation}
  \sqrt{\frac{\alpha}{|t|}}-\sqrt{\lambda}=
-\sqrt{\lambda}\frac{\sqrt{\mu}-1}{2}.
\end{equation}
When the upper limit of integration is bounded, then the integral is bounded,
and the contribution of this term is again bounded by
$C|\sqrt{x_1}-\sqrt{x_1'}|^{\gamma}.$ If the upper limit tends to
$-\infty,$ then we easily show that this term is bounded by
\begin{equation}
  \frac{C_{b,\theta}}{\lambda}[\sqrt{\lambda}(\sqrt{\mu}-1)]^{\gamma}
e^{-\cos\theta\frac{\lambda(\sqrt{\mu}-1)^2}{4}},
\end{equation}
and therefore the contribution of this term is again bounded by
$C|\sqrt{x_1}-\sqrt{x_1'}|^{\gamma}.$ 

To complete the analysis of~\eqref{eqnA.93.1} we need to estimate the
contribution of the error terms. Using the same change of variables we see that
these terms are bounded by
\begin{equation}
  \frac{C_b}{\lambda}
\int\limits_{\frac{1}{\sqrt{\lambda}}-\sqrt{\lambda}}^{\sqrt{\frac{\alpha}{|t|}}-\sqrt{\lambda}}
\left(1+\frac{z}{\sqrt{\lambda}}\right)^{b-\frac{1}{2}}e^{-\cos\theta z^2}
|z|^{\gamma}O\left(\frac{1}{z+\sqrt{\lambda}}+\frac{1}{\sqrt{\lambda}}\right)dz.
\end{equation}
These contributions are bounded as before if
$\sqrt{\lambda}(\sqrt{\mu}-1)$ remains bounded. If the upper limit tends
to $-\infty,$ then this expression is bounded by
\begin{equation}
    \frac{C_{b,\theta}}{\lambda^{\frac 32}}
[\sqrt{\lambda}(\sqrt{\mu}-1)]^{\gamma-1}e^{-\cos\theta\frac{\lambda(\sqrt{\mu}-1)^2}{4}},
\end{equation}
completing the proof that the contribution from $0$ to $\alpha/|t|$ is altogether
bounded by 
$C_{b,\theta}|\sqrt{x_1}-\sqrt{x_1'}|^{\gamma}.$

We turn now to the part of~\eqref{eqnA.89.1} from $\beta/|t|$ to $\infty.$ We
need to split this integral into two parts only for $\lambda<1.$ In this case
we get a term bounded by
\begin{equation}\label{eqnA.100.1}
 \lambda [\sqrt{\lambda}(\sqrt{\mu}-1)]^{1-\gamma}
\int\limits_{\frac{\beta}{|t|}}^{\frac{1}{\lambda}}w^{b}e^{-\cos\theta w}
|\sqrt{w}-\sqrt{\lambda}|^{\gamma}\frac{dw}{\sqrt{w}},
\end{equation}
which is clearly negligible as $\lambda\to 0.$ The other term takes the form
\begin{multline}\label{eqnA.101.1}
 \lambda [\sqrt{\lambda}(\sqrt{\mu}-1)]^{1-\gamma}
\int\limits_{\max{\left\{\frac{1}{\lambda},\frac{\beta}{|t|}\right\}}}^{\infty}w^{b}
e^{-\cos\theta(w+\xi\lambda)}\times \\
|we_{\theta}\psi_{b+2}(\xi\lambda
w e_{2\theta})-\psi_{b+1}(\xi\lambda we_{2\theta})|
|\sqrt{w}-\sqrt{\lambda}|^{\gamma}\frac{dw}{\sqrt{w}} 
\end{multline}
The integral is bounded by
\begin{equation}\label{eqnA.102.1}
 \frac{C_b}{\lambda}\int\limits_{\max{\left\{\frac{1}{\lambda},\frac{\beta}{|t|}\right\}}}^{\infty}
\left(\frac{w}{\xi\lambda}\right)^{\frac{b}{2}-\frac{1}{4}}e^{-\cos\theta(\sqrt{w}-\sqrt{\xi\lambda})^2}
|\sqrt{w}-\sqrt{\lambda\xi}|^{\gamma}\left[|\sqrt{w}-\sqrt{\lambda}|+
O\left(\frac{1}{\sqrt{w}}+\frac{1}{\sqrt{\lambda}}\right)\right]\frac{dw}{\sqrt{w}} 
\end{equation}

For $w\in [\frac{\beta}{|t|},\infty)$ we have the inequalities
\begin{equation}
  \frac{1}{3}(\sqrt{w}-\sqrt{\lambda})\leq (\sqrt{w}-\sqrt{\lambda\xi})\leq
(\sqrt{w}-\sqrt{\lambda}),
\end{equation}
and therefore the expression in~\eqref{eqnA.102.1} is bounded by
\begin{equation}\label{eqnA.103.1}
 \frac{C_b}{\lambda}\int\limits_{\max{\left\{\frac{1}{\lambda},\frac{\beta}{|t|}\right\}}}^{\infty}
\left(\frac{w}{\lambda}\right)^{\frac{b}{2}-\frac{1}{4}}
e^{-\cos\theta\frac{(\sqrt{w}-\sqrt{\lambda})^2}{9}}
|\sqrt{w}-\sqrt{\lambda}|^{\gamma}\left[|\sqrt{w}-\sqrt{\lambda}|+
O\left(\frac{1}{\sqrt{w}}+\frac{1}{\sqrt{\lambda}}\right)\right]\frac{dw}{\sqrt{w}} 
\end{equation}
As before we let $z=\sqrt{w}-\sqrt{\lambda}.$ If $1/\lambda>\beta/|t|,$ then we obtain:
\begin{equation}\label{eqnA.104.1}
 \frac{C_b}{\lambda}\int\limits_{\frac{1}{\sqrt{\lambda}}-\sqrt{\lambda}}^{\infty}
\left(1+\frac{z}{\sqrt{\lambda}}\right)^{b-\frac{1}{2}}e^{-\cos\theta\frac{z^2}{9}}
|z|^{\gamma}\left[|z|+
O\left(\frac{1}{z+\sqrt{\lambda}}+\frac{1}{\sqrt{\lambda}}\right)\right]dz.
\end{equation}
This is bounded by $C_{b,\theta}e^{-\frac{\cos\theta}{20\lambda}}/\lambda$ as $\lambda\to 0,$ so in
this case the contribution of the integral from $\beta/|t|$ to infinity is
bounded by $C_{b,\theta}|\sqrt{x_1}-\sqrt{x_1'}|^{\gamma}.$

The final case to consider is when $1/\lambda<\beta/|t|,$ so that
the lower limit of integration in~\eqref{eqnA.104.1} would be:
\begin{equation}
  \sqrt{\frac{\beta}{|t|}}-\sqrt{\lambda}=\sqrt{\lambda}\frac{3(\sqrt{\mu}-1)}{2}.
\end{equation}
An analysis, essentially identical to that above, shows that this term is bounded
by
\begin{equation}
  \frac{C_b}{\lambda}[\sqrt{\lambda}(\sqrt{\mu}-1)]^{\gamma}
e^{-\cos\theta\frac{\lambda(\sqrt{\mu}-1)^2}{4}}.
\end{equation}
As before we conclude that the contribution of this term is bounded by
$C|\sqrt{x_1}-\sqrt{x_1'}|^{\gamma},$ which completes the proof of the lemma.
\end{proof}

\begin{lemmabis}\labelbis{lem10.0.3}For $b>0,$ $\gamma\geq 0,$ there is a
  constant $C_{b,\phi},$ bounded for $b\leq B,$ so that for $t\in S_{\phi},$
  \begin{equation}
    \int\limits_0^{\infty}
|k^{b+1}_t(x,y)||\sqrt{y}-\sqrt{x}|y^{\frac{\gamma-1}{2}}dy\leq C_{b,\phi}|t|^{\frac{\gamma}{2}}.
  \end{equation}
\end{lemmabis}
\begin{proof}
  We let $t=\tau e^{i\theta},$ with $|\theta|<\frac{\pi}{2}-\phi,$ and change
  variables with $w=y/|t|,$ and $\lambda=x/|t|$ to obtain:
\begin{equation}
  |t|^{\frac{\gamma}{2}}\int\limits_{0}^{\infty}w^{b+\frac{\gamma}{2}}e^{-\cos\theta(w+\lambda)}
|\psi_{b+1}(w\lambda e_{2\theta}))|
|\sqrt{w}-\sqrt{\lambda}|\frac{dw}{\sqrt{w}}.
\end{equation}
In the part of the integral from $0$ to $1/\lambda,$ we estimate $\psi_{b+1}$ by
a constant, obtaining
\begin{equation}
   C_b|t|^{\frac{\gamma}{2}}\int\limits_{0}^{\frac{1}{\lambda}}w^{b+\frac{\gamma}{2}}e^{-\cos\theta(w+\lambda)}
|\sqrt{w}-\sqrt{\lambda}|\frac{dw}{\sqrt{w}},
\end{equation}
which is easily seen to be bounded by $C_{b,\theta}|t|^{\frac{\gamma}{2}}.$

In the non-compact part we use the asymptotic expansion of $\psi_{b+1}$ to see
that this contribution is bounded by
\begin{equation}
  C_b|t|^{\frac{\gamma}{2}}\int\limits_{\frac{1}{\lambda}}^{\infty}
\left(\frac{w}{\lambda}\right)^{\frac{b}{2}-\frac{1}{4}}e^{-\cos\theta(\sqrt{w}-\sqrt{\lambda})^2}
w^{\frac{\gamma}{2}}\frac{|\sqrt{w}-\sqrt{\lambda}|}{\sqrt{\lambda}}
\frac{dw}{\sqrt{w}}.
\end{equation}
As $\lambda\to 0$ this is bounded by
$C_{b,\theta}|t|^{\frac{\gamma}{2}}e^{-\frac{\cos\theta}{2\lambda}}.$ To estimate this term as
$\lambda\to\infty,$ we let $z=\sqrt{w}-\sqrt{\lambda}$ to obtain:
\begin{equation}
  C_b|t|^{\frac{\gamma}{2}}\int\limits_{\frac{1}{\sqrt{\lambda}}-\sqrt{\lambda}}^{\infty}
\left(1+\frac{z}{\sqrt{\lambda}}\right)^{b-\frac{1}{2}}e^{-\cos\theta z^2}
(\sqrt{\lambda}+z)^{\gamma}\frac{|z|}{\sqrt{\lambda}}
dz.
\end{equation}
This term is bounded by $C_{b,\theta}|t|^{\frac{\gamma}{2}} \lambda^{\frac{\gamma-1}{2}},$
thereby completing the proof of the lemma.
\end{proof}

\begin{lemma}\label{lem10.0.3.3} If $b>\nu>0,$ and $0<\phi<\frac{\pi}{2},$
  then there is a constant $C_{b,\phi},$ bounded for $b\leq B,$ so that if
  $t\in S_{\phi},$ we have
  \begin{equation}
    \int\limits_0^{\infty}
\left(\frac{x}{y}\right)^{\nu}|k^{b}_t(x,y)||\sqrt{y}-\sqrt{x}|^{\gamma}dy
\leq C_{b,\phi}|t|^{\frac{\gamma}{2}}.
  \end{equation}
\end{lemma}
\begin{proof}
  We let $t=\tau e^{i\theta},$ with $|\theta|<\frac{\pi}{2}-\phi,$ using
  $w=y/|t|,$ and $\lambda=x/|t|,$ we see that the integral in the lemma equals
\begin{equation}
  |t|^{\frac{\gamma}{2}}\int\limits_{0}^{\infty}
\left(\frac{\lambda}{w}\right)^{\nu}w^b
e^{-\cos\theta(w+\lambda)}\psi_b(w\lambda e_{2\theta})|\sqrt{w}-\sqrt{\lambda}|^{\gamma}
\frac{dw}{w}
\end{equation}
It therefore suffices to show that this integral is uniformly bounded for
$\lambda\in [0,\infty).$

The contribution from
$[0,1/\lambda]$ is bounded by
\begin{equation}
  C_b\lambda^{\nu}e^{-\cos\theta\lambda}\int\limits_{0}^{\frac{1}{\lambda}}
w^{b-\nu-1}
e^{-\cos\theta w}|\sqrt{w}-\sqrt{\lambda}|^{\gamma}dw.
\end{equation}
Since $b-\nu>0$ this is uniformly bounded for all $\lambda.$ The remaining
contribution comes from
\begin{multline}
   \int\limits_{\frac{1}{\lambda}}^{\infty}
\left(\frac{\lambda}{w}\right)^{\nu}w^b
e^{-\cos\theta(w+\lambda)}\psi_b(w\lambda e_{2\theta})|\sqrt{w}-\sqrt{\lambda}|^{\gamma}\frac{dw}{w}
\leq\\
C_b \int\limits_{\frac{1}{\lambda}}^{\infty}
\left(\frac{w}{\lambda}\right)^{\frac{b}{2}-\nu-\frac{1}{4}}
e^{-\cos\theta(\sqrt{w}-\sqrt{\lambda})^2}|\sqrt{w}-\sqrt{\lambda}|^{\gamma}\frac{dw}{\sqrt{w}}.
\end{multline}
As $\lambda\to 0,$ this is bounded by $C_be^{-\frac{\cos\theta}{2\lambda}}.$ Letting
$z=\sqrt{w}-\sqrt{\lambda},$ we obtain
\begin{equation}
  C_b \int\limits_{\frac{1}{\sqrt{\lambda}}-\sqrt{\lambda}}^{\infty}
\left(1+\frac{z}{\sqrt{\lambda}}\right)^{b-2\nu-\frac{1}{2}}
e^{-\cos\theta z^2}|z|^{\gamma}dz.
\end{equation}
It is again straightforward to see that this remains bounded as
$\lambda\to\infty,$ thus completing the proof of the lemma.
\end{proof}

\begin{lemmabis}\labelbis{lem10.1.4} For $0\leq \gamma<1,$ $1\leq b,$  and $0<c<1,$ there is a
  constant $C,$ so that if $ct<s<t,$ then
  \begin{equation}
     \int\limits_{0}^{\infty}
\left|k^{b}_{t}(x,z)-k^{b}_{s}(x,z)\right|
|\sqrt{x}-\sqrt{z}|z^{\frac{\gamma-1}{2}}dz\leq C|t-s|^{\frac{\gamma}{2}}.
  \end{equation}
\end{lemmabis}
\begin{proof} The proof of this lemma is very similar to that of
  Lemma~\ref{lem4new}. If we let $w=z/t,$ $\lambda=x/t,$ and $\mu=t/s,$ then
  the integral we need to estimate becomes
  \begin{equation}
    t^{\frac{\gamma}{2}}\int\limits_{0}^{\infty}w^{b+\frac{\gamma}{2}-1}
|e^{-(w+\lambda)}\psi_{b}(w\lambda)-\mu^{b}e^{-\mu(w+\lambda)}\psi_{b}(\mu^2w\lambda)|
|\sqrt{w}-\sqrt{\lambda}|\frac{dw}{\sqrt{w}}.
  \end{equation}
Proceeding as in the proof of Lemma~\ref{lem4new} we see see that
\begin{multline}
  |e^{-(w+\lambda)}\psi_{b}(w\lambda)-\mu^{b}e^{-\mu(w+\lambda)}\psi_{b}(\mu^2w\lambda)|=\\
(\mu-1)\xi^be^{-\xi(w+\lambda)}\left|\left(\frac{b}{\xi}-(w+\lambda)\right)
\psi_b(\xi^2w\lambda)+2\xi w\lambda\psi_{b+1}(\xi^2w\lambda)\right|,
\end{multline}
where $\xi\in (1,\mu)\subset (1,\frac{1}{c}).$ Since 
\begin{equation}
  t^{\frac{\gamma}{2}}\left(\frac{t-s}{s}\right)\leq \frac{|t-s|^{\frac{\gamma}{2}}}{c},
\end{equation}
it again suffices to show that the integral
\begin{multline}\label{eqnA.122.1}
  \int\limits_{0}^{\infty}w^{b+\frac{\gamma}{2}-1}
e^{-\xi(w+\lambda)}\bigg|\left(\frac{b}{\xi}-(w+\lambda)\right)
\psi_b(\xi^2w\lambda)+\\
2\xi w\lambda\psi_{b+1}(\xi^2w\lambda)\bigg|
|\sqrt{w}-\sqrt{\lambda}|\frac{dw}{\sqrt{w}}
\end{multline}
is uniformly bounded for $\lambda\in (0,\infty).$

We split the integral into a part from $0$ to
$1/\lambda$ and the rest. The compact part is bounded by
\begin{equation}
  C\int\limits_0^{\frac{1}{\lambda}}
w^{b+\frac{\gamma}{2}-1}e^{-(w+\lambda)}|b+w+\lambda+2w\lambda||\sqrt{w}-\sqrt{\lambda}|
\frac{dw}{\sqrt{w}}.
\end{equation}
As $b\geq 1,$ it is not difficult to see that this integral is uniformly
bounded for $\lambda\in (0,\infty).$ For the non-compact part, we first
estimate the contribution from the $b/\xi$-term. Using the asymptotic expansion
we see that this part is bounded by
\begin{equation}
  C\int\limits_{\frac{1}{\lambda}}^{\infty}
\left(\frac{w}{\lambda}\right)^{\frac{b}{2}-\frac{1}{4}}w^{\frac{\gamma-1}{2}}
e^{-\xi(\sqrt{w}-\sqrt{\lambda})^2}|\sqrt{w}-\sqrt{\lambda}|
\frac{dw}{\sqrt{w}}.
\end{equation}
As $\lambda\to 0$ this is $O(e^{-\frac{1}{2\lambda}}).$ To estimate this
expression as $\lambda\to\infty,$ we let $\sqrt{w}-\sqrt{\lambda}=z;$ this
integral is then bounded by:
\begin{equation}
  \int\limits_{\frac{1}{\sqrt{\lambda}}-\sqrt{\lambda}}^{\infty}
\left(1+\frac{z}{\sqrt{\lambda}}\right)^{b-\frac{1}{2}}(\sqrt{\lambda}+z)^{\gamma-1}
e^{-z^2}|z|dz,
\end{equation}
which, as $\lambda\to\infty,$ is bounded by $C\lambda^{\frac{\gamma-1}{2}}.$

The other term is estimated by
\begin{multline}\label{eqnA.126.1}
  \left|(w+\lambda)
\psi_b(\xi^2w\lambda)-2\xi w\lambda\psi_{b+1}(\xi^2w\lambda)\right|=\\
(\xi^2w\lambda)^{\frac{1}{4}-\frac{b}{2}}\frac{e^{2\sqrt{\xi^2
      w\lambda}}}{\sqrt{4\pi}}
\left[(\sqrt{w}-\sqrt{\lambda})^2+O\left(1+\sqrt{\frac{w}{\lambda}}+
\sqrt{\frac{\lambda}{w}}\right)\right].
\end{multline}
It is again easy to see that, as $\lambda\to 0,$ the contribution of this term
is $O(e^{-\frac{1}{2\lambda}}).$ To bound this term as $\lambda\to\infty,$ we
use the estimate from~\eqref{eqnA.126.1} in~\eqref{eqnA.122.1}, and let
$z=\sqrt{w}-\sqrt{\lambda}$ to obtain
\begin{multline}
  \int\limits_{\frac{1}{\sqrt{\lambda}}-\sqrt{\lambda}}^{\infty}
\left(1+\frac{z}{\sqrt{\lambda}}\right)^{b-\frac{1}{2}}(\sqrt{\lambda}+z)^{\gamma-1}
e^{-z^2}|z|\times\\
\left[|z|^2+O\left(1+\frac{z}{\sqrt{\lambda}}+\frac{\sqrt{\lambda}}{z+\sqrt{\lambda}}\right)
\right]dz.
\end{multline}
It is again straightforward to see that all contributions in this integral are 
$O(\lambda^{\frac{\gamma-1}{2}}),$ which completes the proof of the lemma.
\end{proof}
\section{First derivative estimates}
\begin{lemmabis}\labelbis{lem2new} For  $b>0,$ $0\leq\gamma<1,$ and
  $0<\phi<\frac{\pi}{2},$ there is a $C_{b,\phi}$ so that for $t\in S_{\phi}$ we have
  \begin{equation}
    \int\limits_{0}^{\infty}|\pa_xk^b_t(x,y)||\sqrt{y}-\sqrt{x}|^{\gamma}dy\leq
C_{b,\phi}\frac{|t|^{\frac{\gamma}{2}-1}}{1+\lambda^{\frac{1}{2}}},
  \end{equation}
where $\lambda=x/|t|.$
\end{lemmabis}
\begin{proof}[Proof of Lemma~\ref{lem2new}]
  We let $t=\tau e^{i\theta},$ where $|\theta|<\frac{\pi}{2}-\phi.$ Arguing as
  in the proof of Lemma 8.1 in~\cite{WF1d} we see that
\begin{multline}
  \int\limits_{0}^{\infty}|\pa_xk^b_t(x,y)||\sqrt{y}-\sqrt{x}|^{\gamma}dy\leq\\
|t|^{\frac{\gamma}{2}-1}
\int\limits_{0}^{\infty}w^{b-1}e^{-\cos\theta(w+\lambda)}\left|we_{\theta}\psi_b'(w\lambda
  e_{2\theta})-\psi_b(w\lambda e_{2\theta})\right|
|\sqrt{w}-\sqrt{\lambda}|^{\gamma}dw.
\end{multline}
Here $\lambda=x/|t|$ and $w=y/|t|.$
We now estimate the quantity:
\begin{equation}
  I(\lambda)=
\int\limits_{0}^{\infty}w^{b-1}e^{-\cos\theta(w+\lambda)}\left|we_{\theta}\psi_b'(w\lambda
  e_{2\theta})-\psi_b(w\lambda e_{2\theta})\right|
|\sqrt{w}-\sqrt{\lambda}|^{\gamma}dw.
\end{equation}
We divide this integral into a part from $0$ to $1/\lambda$ and the rest. We
write $I(\lambda)=I_0(\lambda)+I_1(\lambda).$ In the first part we can estimate
the integral using the Taylor expansions for $\psi_b$ and $\psi_b'$ as
\begin{equation}
 I_0(\lambda)\leq C_b
\int\limits_{0}^{1/\lambda}w^{b-1}e^{-\cos\theta(w+\lambda)}\left(w+\frac{1}{\Gamma(b)}\right)
|\sqrt{w}-\sqrt{\lambda}|^{\gamma}dw.
\end{equation}
When $\lambda$ remains bounded, the only difficulty that arises is that as
$b\to 0,$ the $w^{b-1}$ term introduces a $1/b,$ but this is compensated for by
the $1/\Gamma(b),$ showing that this expression remains bounded as $b\to 0.$
This term is bounded by a constant times 
\begin{equation}
e^{-\cos\theta\lambda}\left(\frac{1+\lambda^{\frac{\gamma}{2}}}
{1+\lambda^b}\right).
\end{equation}

To estimate the other part of the integral, we use the asymptotic expansions
for $\psi_b$ and $\psi_b',$ giving
\begin{equation}
I_1(\lambda)\leq C\int\limits_{\frac{1}{\lambda}}^{\infty}
\left(\frac{w}{\lambda}\right)^{\frac{b}{2}-\frac 14}
\left|\sqrt{\lambda}-\sqrt{w}\right|^{1+\gamma}
e^{-\cos\theta(\sqrt{\lambda}-\sqrt{w})^2}\frac{dw}{\sqrt{w\lambda}}.
\end{equation}
Applying Lemma~\ref{lem5}, this is $O(e^{-\frac{1}{\lambda}}),$ as $\lambda\to
  0,$ and, as $\lambda\to\infty,$
\begin{equation}
  I_1(\lambda)\leq C_{b,\phi}\lambda^{-\frac{1}{2}}.
\end{equation}
Combining this with the estimates above, we can show that
\begin{equation}
  I(\lambda)\leq \frac{C}{1+\sqrt{\lambda}}.
\end{equation}
This proves Lemma~\ref{lem2new}
\end{proof}

\begin{lemmabis}\labelbis{lem20neww}  For  $b>0,$ $0<\gamma<1,$ $0<\phi<\frac{\pi}{2},$
  and $0<c<1,$ there is a constant $C_{b,c,\phi}$ so that for $cx_2<x_1<x_2,$
  $t\in S_{\phi},$
  \begin{multline}\label{eqn2400.0p}
    \int\limits_0^{\infty}|\sqrt{x_1}\pa_{x}k^b_t(x_1,y)-
    \sqrt{x_2}\pa_{x}k^b_t(x_2,y)||\sqrt{x_1}-\sqrt{y}|^{\gamma}dy\leq\\
    C_{b,c,\phi}|t|^{\frac{\gamma-1}{2}}\frac{\left(\frac{|\sqrt{x_2}-\sqrt{x_1}|}
{\sqrt{|t|}}\right)}
    {1+\left(\frac{|\sqrt{x_2}-\sqrt{x_1}|}{\sqrt{|t|}}\right)}
  \end{multline}
\end{lemmabis}
\begin{proof} We let $t=\tau e^{i\theta},$ where $|\theta|<\frac{\pi}{2}-\phi,$ and
note that Lemma~\ref{lem2new} provides a ``trivial''
  estimate,
  \begin{equation}\label{trvest100}
\int\limits_0^{\infty}|\sqrt{x_1}\pa_{x}k^b_t(x_1,y)-
\sqrt{x_2}\pa_{x}k^b_t(x_2,y)||\sqrt{x_1}-\sqrt{y}|^{\gamma}dy\leq
C_{b,\phi}
\frac{\sqrt{x_1}|t|^{\frac{\gamma-1}{2}}}{\sqrt{x_1}+\sqrt{|t|}}\leq C_{b,\phi}
|t|^{\frac{\gamma-1}{2}},
  \end{equation}
which is the desired estimate when $|\sqrt{x_2}-\sqrt{x_1}|/\sqrt{|t|}>\!>1.$
Recalling that
  \begin{equation}
    \pa_xk^b_t(x,y)=\frac{1}{yt}\left(\frac{y}{t}\right)^be^{-\frac{x+y}{t}}
\left[\left(\frac{y}{t}\right)\psi'_b\left(\frac{xy}{t^2}\right)-
\psi_b\left(\frac{xy}{t^2}\right)\right].
  \end{equation}
and setting $w=y/|t|,$ $\lambda=x_1/|t|,$ and $\mu=x_2/x_1,$ we see that
\begin{multline}
  \int\limits_0^{\infty}|\sqrt{x_1}\pa_{x}k^b_t(x_1,y)-
\sqrt{x_2}\pa_{x}k^b_t(x_2,y)||\sqrt{x_1}-\sqrt{y}|^{\gamma}dy=\\
\sqrt{\lambda}|t|^{\frac{\gamma-1}{2}}\int\limits_0^{\infty}w^{b-1}
e^{-\cos\theta w}
|F(\mu,\lambda,w)-F(1,\lambda,w)|
|\sqrt{\lambda}-\sqrt{w}|^{\gamma}dw,
\end{multline}
where we let
\begin{equation}
  F(\mu,\lambda,w)=\sqrt{\mu}e^{-\mu\lambda e_{\theta}}
\left[w e_{\theta}\psi_b'(\mu\lambda w e_{2\theta} )-\psi_b(\mu\lambda w e_{2\theta})\right].
\end{equation}
Using Lemma~\ref{MVTineq} we see that 
\begin{equation}
  |F(\mu,\lambda,w)-F(1,\lambda,w)|\leq(\mu-1)|\pa_{\mu}F(\xi,\lambda,w)|\quad\text{ for a
  }\quad\xi\in (1,\mu), 
\end{equation}
where $\xi\in (1,\frac{1}{c}).$ We therefore need to estimate
\begin{equation}\label{eqn413.0}
  \sqrt{\lambda}|t|^{\frac{\gamma-1}{2}}(\mu-1)\int\limits_0^{\infty}w^{b-1}
e^{-\cos\theta  w}
|\pa_{\mu}F(\xi,\lambda,w)|
|\sqrt{\lambda}-\sqrt{w}|^{\gamma}dw
\end{equation}

Using the equation satisfied by $\psi_b,$ we see that
\begin{multline}
  \pa_{\mu}F(\xi,\lambda,w)=\\
e_{\theta}\frac{e^{-\xi\lambda e_{\theta}}}{\sqrt{\xi}}\left[
\left(\xi\lambda +w-\frac{e_{-\theta}}{2}\right)\psi_b(\xi\lambda we_{2\theta})-
\left(2\xi\lambda w e_{\theta}+(b-\frac 12)w\right)
\psi_b'(\xi\lambda  we_{2\theta})\right]
\end{multline}
As usual we split the integral in~\eqref{eqn413.0} into the part from $0$ to
$1/\lambda$ and the rest. In the compact part we use the estimate
\begin{equation}
  |\pa_{\mu}F(\xi,\lambda,w)|\leq
  C_be^{-\cos\theta\xi\lambda}\left[\frac{1+\lambda+w}{\Gamma(b)}+\frac{w}{\Gamma(b+1)}+
\lambda w O(1+\lambda+w)\right].
\end{equation}
Using this estimate we see that
\begin{multline}
    \sqrt{\lambda}|t|^{\frac{\gamma-1}{2}}(\mu-1)\int\limits_0^{\frac{1}{\lambda}}w^{b-1}
e^{-\cos\theta w}
|\pa_{\mu}F(\xi,\lambda,w)||\sqrt{\lambda}-\sqrt{w}|^{\gamma}dw\leq \\
C_{b,\theta}\sqrt{\lambda}|t|^{\frac{\gamma-1}{2}}e^{-\frac{\cos\theta\lambda}{2}}|\mu-1|
\end{multline}
where the constant is uniformly bounded for $b\in (0,B],$ and $|\theta|\leq
\frac{\pi}{2}-\phi.$  

Now we turn to the non-compact part where we use the asymptotic expansions of
$\psi_b$ and $\psi_b'$ to obtain:
\begin{multline}
|\pa_{\mu}F(\xi,\lambda,w)|=\\
(\xi\lambda w)^{\frac{1}{4}-\frac{b}{2}} e^{2\cos\theta\sqrt{\xi\lambda w}}e^{-\cos\theta\xi\lambda}
\left[\frac{\lambda}{\xi}\left(1-\sqrt{\frac{w}{\xi\lambda}}\right)^2+O\left(1+\sqrt{\frac{w}{\lambda}}+
\sqrt{\frac{\lambda}{w}}\right)\right].
\end{multline}
Using this expansion in the integral and setting $z=\sqrt{w/\lambda},$ we see
that this term is bounded by
\begin{equation}
  C_b\lambda^{\frac{\gamma}{2}+1}|t|^{\frac{\gamma-1}{2}}(\mu-1)
\int\limits_{\frac{1}{\lambda}}^{\infty}z^{b-\frac{1}{2}}e^{-\cos\theta\lambda(z-\sqrt{\xi})^2}
|z-1|^{\gamma}\left(\lambda(z-\sqrt{\xi})^2+O(1+z+1/z)\right)dz.
\end{equation}
As $\lambda\to 0$ this term is easily seen to be bounded by
$C_{b,\theta}|t|^{\frac{\gamma-1}{2}}(\mu-1)e^{-\frac{\cos\theta}{2\lambda}}.$ 

In this case, when $|\sqrt{x_2}-\sqrt{x_1}|/\sqrt{|t|}>1/4,$ we use the
``trivial'' estimate in~\eqref{trvest100}. We henceforth assume that
$|\sqrt{x_2}-\sqrt{x_1}|/\sqrt{|t|}\leq 1/4,$ which implies that
\begin{equation}\label{eqn419.0}
  \sqrt{\mu}-1\leq \frac{1}{4\sqrt{\lambda}}.
\end{equation}
In order to estimate the integral, we need to split it into three parts, with
$z$ lying in $[1/\lambda,1],$ $[1,\sqrt{\mu}],$ and $[\sqrt{\mu},\infty),$
respectively. Using the assumption in~\eqref{eqn419.0} we easily show that the
integral over $[1,\sqrt{\mu}]$ is bounded by
$C_{b,\theta}\sqrt{\lambda}|t|^{\frac{\gamma-1}{2}}(\mu-1),$ as desired. To
treat the other two terms we use Laplace's method. The integral over
$[1/\lambda,1]$ is bounded by
\begin{equation}
  C_b\lambda^{\frac{\gamma}{2}+1}|t|^{\frac{\gamma-1}{2}}(\mu-1)
\int\limits_{\frac{1}{\lambda}}^{1}z^{b-\frac{1}{2}}e^{-\cos\theta\lambda(z-1)^2}
|z-1|^{\gamma}\left(\lambda(z-\sqrt{\mu})^2+O(1+z+1/z)\right)dz.
\end{equation}
Laplace's method, using~\eqref{eqn419.0}, shows that this term is
also bounded by 
$$C_{b,\theta}\sqrt{\lambda}|t|^{\frac{\gamma-1}{2}}(\mu-1).$$ 

Finally the integral over $[\sqrt{\mu},\infty)$ is bounded by
\begin{equation}
  C_b\lambda^{\frac{\gamma}{2}+1}|t|^{\frac{\gamma-1}{2}}(\mu-1)
\int\limits_{\sqrt{\mu}}^{\infty}z^{b-\frac{1}{2}}e^{-\cos\theta\lambda(z-\sqrt{\mu})^2}
|z-1|^{\gamma}\left(\lambda(z-1)^2+O(1+z+1/z)\right)dz.
\end{equation}
Applying Laplace's method to this integral shows that it is bounded by
$$C_{b,\theta}\sqrt{\lambda}|t|^{\frac{\gamma-1}{2}}(\mu-1),$$
thereby completing the estimate of the non-compact term. The proof of the lemma
is completed by noting that, 
\begin{equation}
  \sqrt{\lambda}(\mu-1)=\frac{\sqrt{x_2}-\sqrt{x_1}}{\sqrt{|t|}}
\frac{\sqrt{x_2}+\sqrt{x_1}}{\sqrt{x_1}},
\end{equation}
and therefore the lemma follows from the assumption that $0<c< x_1/x_2<1.$
\end{proof}

\begin{lemmabis}\labelbis{lemA-}
  For  $b>0,$ $0<\gamma<1,$ there is a constant $C_b$ so that for $t_1<t_2<2t_1,$ we have:
  \begin{equation}\label{lemA-estp}
    \int\limits_{t_2-t_1}^{t_1}\int\limits_{0}^{\infty}
|\pa_xk^b_{t_2-t_1+s}(x,y)-\pa_xk^b_s(x,y)|
|\sqrt{x}-\sqrt{y}|^{\gamma}dyds<
C_b|t_2-t_1|^{\frac{\gamma}{2}}.
  \end{equation}
\end{lemmabis}

This result follows from the more basic:
\begin{lemmabis}\labelbis{lemAA-}
  For $b>0,$ $0\geq \gamma<1,$ and $0<t_1<t_2<2t_1,$ we have for $s\in
  [t_2-t_1,t_1]$ that there is a constant $C$ so that
  \begin{equation}\label{lemAA-estp0}
    \int\limits_{0}^{\infty}
|\pa_xk^b_{t_2-t_1+s}(x,y)-\pa_xk^b_s(x,y)|
|\sqrt{x}-\sqrt{y}|^{{\gamma}}dy<
C\frac{(t_2-t_1)s^{\frac{\gamma}{2}-1}}{(t_2-t_1+s)(1+\sqrt{x/s})}.
  \end{equation}
\end{lemmabis}

\begin{proof}[Proof of Lemma~\ref{lemA-}] The estimate in~\eqref{lemA-estp}
  follows by integrating the estimate in~\eqref{lemAA-estp}:
  \begin{multline}
     \int\limits_{t_2-t_1}^{t_1}\int\limits_{0}^{\infty}
|\pa_xk^b_{t_2-t_1+s}(x,y)-\pa_xk^b_s(x,y)|
|\sqrt{x}-\sqrt{y}|^{\gamma}dyds\leq C\int\limits_{t_2-t_1}^{t_1}
\frac{(t_2-t_1)s^{\frac{\gamma}{2}-1}ds}{(t_2-t_1)+s}\\
\leq C\int\limits_{t_2-t_1}^{\infty}
(t_2-t_1)s^{\frac{\gamma}{2}-2}ds=\frac{2C}{2-\gamma}|t_2-t_1|^{\frac{\gamma}{2}}.
  \end{multline}
\end{proof}

\noindent
We now give the proof of Lemma~\ref{lemAA-}.
\begin{proof}
This argument is very similar to the proof of Lemma~\ref{lem20neww}. Set $\tau=t_2-t_1,$
and define
\begin{equation}
  G(\mu,\lambda,w)=\mu^{b+1}e^{-\mu(w+\lambda)}\left[\mu
      w\psi_b'(\mu^2w\lambda)-\psi_b(\mu^2w\lambda)\right].
\end{equation}
Setting $w=y/s,$ $\lambda=x/s,$ and $\mu=s/(\tau+s),$ we see that
\begin{multline}
  \int\limits_{0}^{\infty}
|\pa_xk^b_{t_2-t_1+s}(x,y)-\pa_xk^b_s(x,y)|
|\sqrt{x}-\sqrt{y}|^{\gamma}dy=\\
s^{\frac{\gamma}{2}-1}\int\limits_0^{\infty}w^{b-1}|G(\mu,\lambda,w)-G(1,\lambda,w)||\sqrt{\lambda}-
\sqrt{w}|^{\gamma}dw=\\
s^{\frac{\gamma}{2}-1}(\mu-1)\int\limits_0^{\infty}w^{b-1}|\pa_{\mu}G(\xi,\lambda,w)||\sqrt{\lambda}-
\sqrt{w}|^{\gamma}dw,
\end{multline}
here $\xi\in[\mu,1].$ In the last line we use the mean value theorem.  The
assumption $t_1<t_2<2t_1$ shows that $\mu\in [\frac{1}{2},1).$ 

A calculation, using the equation satisfied by $\psi_b$ shows that
\begin{multline}
  \pa_{\mu}G(\xi,\lambda,w)=\xi^{b+1}e^{-\xi(w+\lambda)}\times\\
\left[\psi_b(\xi^2 w\lambda)(3w+\lambda-\frac{1+b}{\xi})-
w\psi_b'(\xi^2 w\lambda)(b-2+\xi(w+3\lambda))\right].
\end{multline}
As usual, we split the integral into a part from $0$ to $1/\lambda$ and the
rest. In the compact part we observe that
\begin{equation}
  |\pa_{\mu}G(\xi,\lambda,w)|\leq
Ce^{-\frac{w+\lambda}{2}}\left|\frac{w+\lambda+1}{\Gamma(b)}+
O\left(w(1+w+\lambda)^2)\right)\right|.
\end{equation}
The compact part is therefore bounded by 
\begin{equation}
  s^{\frac{\gamma}{2}-1}(\mu-1)\int\limits_0^{\frac{1}{\lambda}}w^{b-1}
e^{-\frac{w+\lambda}{2}}\left|\frac{w+\lambda+1}{\Gamma(b)}+
O\left(w(1+w+\lambda)^2)\right)\right||\sqrt{\lambda}-
\sqrt{w}|^{\gamma}dw
\end{equation}
As $\lambda\to\infty$ this is bounded by $C
s^{\frac{\gamma}{2}-1}(\mu-1)e^{-\frac{\lambda}{4}},$ and as $\lambda\to 0,$ by 
$Cs^{\frac{\gamma}{2}-1}(\mu-1).$

For the non-compact part we use the asymptotic expansions of $\psi_b$ and
$\psi_b'$ of order 2, given in~\eqref{2ndordasymp00}, to obtain:
\begin{multline}
  G(\xi,\lambda,w)\leq C\xi^{b+1}e^{-\xi(\sqrt{w}-\sqrt{\lambda})^2}
(\xi^2w\lambda)^{\frac{1}{4}-\frac{b}{2}}\times\\
\Bigg|\lambda\left(1-\sqrt{\frac{w}{\lambda}}\right)^3-
a_1(b)\sqrt{\frac{w}{\lambda}}\left(1-\sqrt{\frac{w}{\lambda}}\right)+\\
a_2(b)\left(1-\sqrt{\frac{w}{\lambda}}\right)+a_3(b)\left(1-\sqrt{\frac{\lambda}{w}}\right)
+O\left(1+\frac{1}{\lambda}+\frac{1}{w}+\frac{1}{\sqrt{w\lambda}}\right)\Bigg|,
\end{multline}
here $a_j(b),$ $j=1,2,3$ are polynomials in $b.$
Using this expression in the integral and letting
$z=\sqrt{\lambda}(\sqrt{w/\lambda}-1),$ we see this is bounded by
\begin{equation}
  Cs^{\frac{\gamma}{2}-1}(\mu-1)\int\limits_{\frac{1}{\sqrt{\lambda}}-\sqrt{\lambda}}^{\infty}
\left(\frac{z}{\sqrt{\lambda}}+1\right)^{b-\frac 12}e^{-\frac{z^2}{2}}z^{\gamma}
\left[\frac{|z|^3+|z|}{\sqrt{\lambda}}+O\left(\frac{1}{\lambda}\right)\right]dz
\end{equation}
As $\lambda\to 0$ this is bounded by $C
s^{\frac{\gamma}{2}-1}(\mu-1)e^{-\frac{1}{4\lambda}}.$ When $\lambda\to\infty,$
Laplace's method applies to show that it is bounded by 
$Cs^{\frac{\gamma}{2}-1}(\mu-1)/\sqrt{\lambda}.$ This completes the proof of
the lemma.
\end{proof}

\section{Second derivative estimates}
\begin{lemmabis}\labelbis{lemA}
  For $b>0,$ $0<\gamma<1,$ and $0<\phi<\frac{\pi}{2},$ there is a $C_{b,\phi}$
  so that for $t=|t|e^{i\theta}$ with $|\theta|<\frac{\pi}{2}-\phi,$
  \begin{equation}\label{lemAestp}
\begin{split}
    &\int\limits_{0}^{|t|}
\int\limits_0^{\infty}|x\pa_x^2k^b_{se^{i\theta}}(x,y)||\sqrt{y}-\sqrt{x}|^{\gamma}dyds\leq
C_{b,\phi}x^{\frac{\gamma}{2}}
\text{  and }\\
&\int\limits_{0}^{t}
\int\limits_0^{\infty}|x\pa_x^2k^b_{se^{i\theta}}(x,y)||\sqrt{y}-\sqrt{x}|^{\gamma}dyds\leq
C_{b,\phi}|t|^{\frac{\gamma}{2}}.
\end{split}
  \end{equation}
\end{lemmabis}

We deduce this lemma from the following result, of interest in its own right:
\begin{lemmabis}\labelbis{lem25new}
  For $b>0,$ $0\leq\gamma<1,$ $0<\phi<\frac{\pi}{2},$ there is a $C_{b,\phi}$
  so that if $t\in S_{\phi},$ then
  \begin{equation}\label{lem25newpest}
    \int\limits_0^{\infty}|x\pa_x^2k^b_t(x,y)||\sqrt{x}-\sqrt{y}|^{\gamma}dy\leq
    C_{b,\phi}\frac{\lambda |t|^{\frac{\gamma}{2}-1}}{1+\lambda},
  \end{equation}
where $\lambda=x/|t|.$
\end{lemmabis}
We let $t=\tau e^{i\theta},$ where $|\theta|<\frac{\pi}{2}-\phi,$ and
first show how to deduce~\ref{lemA} from~\eqref{lem25newpest}.
\begin{proof}[Proof of Lemma~\ref{lemA}]
  To prove the first estimate in~\eqref{lemAestp}, using~\eqref{lem25newpest},
  we see that
\begin{equation}\label{lemAestp1}
    \int\limits_{0}^{|t|}
\int\limits_0^{\infty}|x\pa_x^2k^b_{se^{i\theta}}(x,y)||\sqrt{y}-\sqrt{x}|^{\gamma}dyds
\leq C_{b,\phi}\int\limits_0^{|t|}s^{\frac{\gamma}{2}-1}\frac{x/s}{1+x/s}ds.
  \end{equation}
Splitting this into an integral from $0$ to $x$ and the rest (if needed), we
easily see that the first estimate in~\eqref{lemAestp} holds. The second estimate follows
from~\eqref{lemAestp1} and the observation that $x/(x+s)\leq 1.$
\end{proof}

Now we prove Lemma~\ref{lem25new}.
\begin{proof}[Proof of Lemma~\ref{lem25new}]
We denote the left hand side of~\eqref{lem25newpest} by $I.$
The formula~\eqref{fndslnfrm1} for $k^b_t,$
and the second order equation satisfied by $\psi_b(z):$
\begin{equation}
  z\psi_b''+b\psi_b'-\psi_b=0.
\end{equation}
imply that
\begin{multline}
  x\pa_x^2k^b_t(x,y)=\\
\frac{1}{yt}\left(\frac{y}{t}\right)^b
e^{-\frac{(x+y)}{t}}\left[\left(\frac{x+y}{t}\right)\psi_b\left(\frac{xy}{t^2}\right)-
\left(\frac{2xy}{t^2}\right)\psi_b'\left(\frac{xy}{t^2}\right)-
\left(\frac{by}{t}\right)\psi_b'\left(\frac{xy}{t^2}\right)\right].
\end{multline}
We let $w=y/|t|$ $\lambda=x/|t|,$ to obtain
\begin{multline}
  I=\frac{1}{|t|^{1-\frac{\gamma}{2}}}
\int\limits_0^{\infty}w^{b}
e^{-\cos\theta(w+\lambda)}\times\\
\left|(w+\lambda)\psi_b(w\lambda e_{2\theta})-
(2w\lambda e_{\theta}+bw)\psi_b'(w\lambda e_{2\theta})\right|
|\sqrt{w}-\sqrt{\lambda}|^{\gamma}\frac{dw}{w},
\end{multline}
which we split into a part from $[0,\frac{1}{\lambda}]$ and the
rest. In the compact part, we use the estimate 
\begin{equation}\label{eqn156}
 \left|(w+\lambda)\psi_b(w\lambda e_{2\theta})-
(2w\lambda e_{\theta}+bw)\psi_b'(w\lambda e_{2\theta})
\right|\leq C_b\lambda\left[\frac{1}{\Gamma(b)}+w(1+w+\lambda)\right].
\end{equation}
Applying Lemma~\ref{lem4} shows that these parts of the $w$-integral are bounded by
\begin{equation}
  C_{b,\theta}e^{-\cos\theta\frac{\lambda}{2}}\text{ as }\lambda\to\infty\text{ and }
C_{b,\theta}\lambda\text{ as }\lambda\to 0.
\end{equation}

In the non-compact part of the $w$-integral, we use the asymptotic expansion to
obtain
\begin{multline}\label{aprxfrm2}
  \left|(w+\lambda)\psi_b(w\lambda e_{2\theta})-
(2w\lambda e_{\theta}+bw)\psi_b'(w\lambda e_{2\theta})\right|= \\
(w\lambda)^{\frac{1}{4}-\frac{b}{2}}e^{2\cos\theta\sqrt{w\lambda}}
\left[(\sqrt{w}-\sqrt{\lambda})^2+O\left(\frac{w+\lambda}{\sqrt{w\lambda}}+1\right)\right].
\end{multline}
Applying Lemma~\ref{lem5} shows that the principal terms of the non-compact
part of the $w$-integral are bounded by
\begin{equation}\label{noncpctintest1}
  C_{b,\theta}e^{-\frac{1}{2\lambda}}.
\end{equation}
 This leaves only the error term
in~\eqref{aprxfrm2}. Again applying Lemma~\ref{lem5} shows that these terms are
also bounded by the expression in~\eqref{noncpctintest1}.
\end{proof}

\begin{lemmabis}\labelbis{lemB}
  For $b>0,$ $0<\gamma<1,$ $0<\phi<\frac{\pi}{2},$ and $0<x_2/3<x_1<x_2,$ there
  is a constant $C_{b,\phi}$ so that, for $t\in S_{\phi},$ we have
  \begin{equation}\label{lemBestp}
    \int\limits_{0}^{|t|}\left|(\pa_yy-b)k^b_{se^{i\theta}}(x_2,\alpha)-
(\pa_yy-b)k^b_{se^{i\theta}}(x_2,\beta)\right|
ds\leq C_{b,\phi},
  \end{equation}
where $\alpha$ and $\beta$ are defined in~\eqref{eqn8555}.
\end{lemmabis}
\begin{proof}
We let $t=|t|e^{i\theta},$ where $|\theta|<\frac{\pi}{2}-\phi,$
and use $I$ to denote the quantity on the left in~\eqref{lemBestp}.
Using~\eqref{fndslnfrm1}, we see that, for $t$ in the right half plane, 
\begin{equation}
  (\pa_y
  y-b)k^b_t(x,y)=\frac{1}{t}\left(\frac{y}{t}\right)^be^{-\frac{(x+y)}{t}}
\left[\left(\frac
    xt\right)\psi_b'\left(\frac{xy}{t^2}\right)-\psi_b\left(\frac{xy}{t^2}\right),
\right].
\end{equation}
and therefore:
\begin{multline}
  |I|\leq
\int\limits_{0}^t\frac{e^{-\cos\theta\frac{x_2}{s}}}{s}\\
\Bigg|\left(\frac{\alpha}{s}\right)^be^{-\frac{\alpha e_{\theta}}{s}}
\left[\left(\frac{\alpha  e_{\theta}}{s}\right)
\psi_b'\left(\frac{x_2\alpha
    e_{2\theta}}{s^2}\right)-\psi_b\left(\frac{x_2\alpha  e_{2\theta}}{s^2}\right)\right]-\\
\left(\frac{\beta}{s}\right)^be^{-\frac{\beta
    e_{\theta}}{s}}\left[\left(\frac{\beta  e_{\theta}}{s}\right)
\psi_b'\left(\frac{x_2\beta
    e_{2\theta}}{s^2}\right)-\psi_b\left(\frac{x_2\beta  e_{2\theta}}{s^2}\right)\right]\Bigg|ds.
\end{multline}

As
\begin{equation}\label{ratbnd0.1}
  \frac{1}{3}\leq\frac{x_1}{x_2}<1,
\end{equation}
the numbers
$$\frac{\alpha}{s}<\frac{x_1}{s}<
\frac{x_2}{s} <\frac{\beta}{s}$$ 
are all comparable.  If $s<x_1,$ then we can use the asymptotic expansion to
estimate the integrand by
\begin{equation}
  \frac{C_b}{s}e^{-\cos\theta\frac{(\sqrt{x_2}-\sqrt{x_1})^2}{4s}}\left[\frac{(\sqrt{x_2}-\sqrt{x_1})}{2\sqrt{s}}+
O\left(\sqrt{\frac{s}{\alpha}}\right)\right].
\end{equation}
Changing variables with
\begin{equation}
  \sigma=\frac{(\sqrt{x_2}-\sqrt{x_1})^2}{s},
\end{equation}
the principal term in the integral from $0$ to $x_1,$ becomes:
\begin{equation}
  C_b
  \int\limits_{\frac{(\sqrt{x_2}-\sqrt{x_1})^2}{x_1}}^{\infty}
  e^{-\cos\theta\frac{x}{4}}\frac{dx}{\sqrt{x}}
\end{equation}
This is uniformly bounded. The
integral of the error term is bounded by
\begin{equation}
  \int\limits_0^{x_1}\frac{ds}{\sqrt{\alpha s}}=\frac{1}{2}\sqrt{\frac{x_1}{\alpha}}.
\end{equation}
As $x_1/x_2>1/3$ this is bounded by 1, completing the estimate of this part of
the $s$-integral.

If $|t|>x_1,$ then we also need to estimate the $s$-integral over $[x_1,t].$ If
we let
\begin{equation}
  F_{\mu}(z)=z^be^{-ze_{\theta}}[ze_{\theta}\psi_b'(\mu z
  e_{2\theta})-\psi_b(\mu z e_{2\theta})],
\end{equation}
then the remaining part of the $s$-integral can be written:
\begin{equation}
  \int\limits_{x_1}^{t}\frac{e^{-\cos\theta\frac{x_2}{s}}}{s}\left|F_{\frac{x_2}{s}}\left(\frac{\alpha}{s}\right)-
F_{\frac{x_2}{s}}\left(\frac{\beta}{s}\right)\right|ds.
\end{equation}
Lemma~\ref{MVTineq} shows that this is estimated by
\begin{equation}
  C \int\limits_{x_1}^{t}\frac{e^{-\cos\theta\frac{x_2}{s}}}{s}|F'_{\frac{x_2}{s}}(\xi)|
\left(\frac{\beta-\alpha}{s}\right)ds.
\end{equation}
Here $\xi\in[\frac{\alpha}{s},\frac{\beta}{s}]\subset (0,\frac{\beta}{x_1}]$
and $\mu\in [\frac{x_2}{t},\frac{x_2}{x_1}].$ The $z^b$-term in $F_{\mu}(z)$ is
the only term which may contribute something unbounded to $F'_{\mu}(z),$ and
this occurs only if $b<1.$ The remaining terms are easily seen to contribute a
term bounded by $C_b(1-x_1/x_2).$ The $z^{b-1}$-term is bounded by
\begin{equation}
   K=bC_b \int\limits_{x_1}^{t}
\frac{e^{-\cos\theta\frac{x_2}{s}}}{s}\left(\frac{\alpha}{s}\right)^{b-1}
\left(\frac{\beta-\alpha}{s}\right)ds.
\end{equation}
We let $w=x_2/s$ to obtain
\begin{equation}
  K=bC_b\left(\frac{\beta-\alpha}{x_2}\right)\left(\frac{\alpha}{x_2}\right)^{b-1}
\int\limits_{\frac{x_2}{|t|}}^{\frac{x_2}{x_1}}w^{b-1}e^{-\cos\theta w}dw\leq
C_{b,\theta}b\Gamma(b)\left(1-\frac{x_1}{x_2}\right).
\end{equation}
This completes the proof that there is a constant $C_{b,\phi}$ uniformly bounded with $b,$
so that
\begin{equation}\label{2drvI3est}
  I\leq C_{b,\phi}.
\end{equation}
\end{proof}

\begin{lemmabis}\labelbis{lemC}
  For  $b>0,$ $0<\gamma<1,$  $\phi<\frac{\pi}{2},$ and $0<x_2/3<x_1<x_2,$ if
  $J=[\alpha,\beta],$ with the endpoints given by~\eqref{eqn8555}, there is a
  constant $C_{b,\phi}$ so that if $|\theta|<\frac{\pi}{2}-\phi,$ then
\begin{equation}\label{lemCestp}
  \begin{split}
     &I_1=\int\limits_0^{|t|}\int\limits_{\alpha}^{\beta}
|L_bk^b_{se^{i\theta}}(x_2,y)||\sqrt{y}-\sqrt{x_2}|^{\gamma}dyds\leq 
C_{b,\phi}|\sqrt{x_2}-\sqrt{x_1}|^{\gamma}\\
 &I_2=\int\limits_0^{|t|}\int\limits_{\alpha}^{\beta}
|L_bk^b_{se^{i\theta}}(x_1,y)||\sqrt{y}-\sqrt{x_1}|^{\gamma}dyds\leq C_{b,\phi}|\sqrt{x_2}-\sqrt{x_1}|^{\gamma}.
  \end{split}
\end{equation}
\end{lemmabis}
\begin{proof}
Throughout these
calculations we use the formula, valid for $t$ in the right half plane:
\begin{multline}\label{Lbkb0}
  L_bk^b_t(x,y)=\pa_t k^b_t(x,y)=\\
\frac{1}{yt}\left(\frac{y}{t}\right)^be^{-\frac{(x+y)}{t}}
\left[\left(\frac{x+y}{t}-b\right)\psi_b\left(\frac{xy}{t^2}\right)
-\left(\frac{2xy}{t^2}\right)\psi_b'\left(\frac{xy}{t^2}\right)\right].
\end{multline}
We give the argument for $I_1,$ the argument for $I_2$ is essentially
identical.

If we let 
\begin{equation}
  w=\frac{y}{s}\quad \lambda=\frac{x_2}{s},
\end{equation}
and
\begin{equation}
  R_{\alpha,\beta,t}=\{(w,\lambda):\:\frac{\alpha}{x_2}\lambda\leq w\leq 
\frac{\beta}{x_2}\lambda\text{ and }\frac{x_2}{|t|}\leq\lambda\},
\end{equation}
then $I_1$ becomes:
\begin{multline}\label{I1strpt0}
  |I_1|\leq x_2^{\frac{\gamma}{2}}\iint\limits_{R_{\alpha,\beta,t}}
w^{b-1}e^{-\cost(w+\lambda)}|[(w+\lambda)\et-b]\psi_b\left(w\lambda\ett\right)-2w\lambda\ett
\psi_b'\left(w\lambda\ett\right)|\\\times
\left|1-\sqrt{\frac{w}{\lambda}}\right|^{\gamma}\frac{dwd\lambda}{\lambda}.
\end{multline}
As in the previous cases we estimate $\psi_b$ and $\psi_b'$ using the Taylor
expansion where $w\lambda<1$ and using the asymptotic expansion where
$w\lambda\geq 1.$ In the present instance this divides the argument into two
cases: 1. $\frac{x_1}{|t|}\geq 1$ and 2. $\frac{x_1}{|t|}< 1.$ In case 1 we only
need to use the asymptotic expansions, whereas in case 2 we also have to
consider another term, where we estimate $\psi_b$ and $\psi_b'$ using the
Taylor expansion. We begin with case 1.

The asymptotic expansion gives the estimate
\begin{multline}
  |I_1|\leq
  C_bx_2^{\frac{\gamma}{2}}\iint\limits_{R_{\alpha,\beta,t}}
w^{b-1}(w\lambda)^{\frac{1}{4}-\frac{b}{2}}e^{-\cost(\sqrt{w}-\sqrt{\lambda})^2}\times\\
\left|(\sqrt{w}-\sqrt{\lambda})^2+b\left[1+O\left(\frac{1}{\sqrt{w\lambda}}\right)
\right]\right|\left|1-\sqrt{\frac{w}{\lambda}}\right|^{\gamma}\frac{dwd\lambda}{\lambda}.
\end{multline}
As $w/\lambda$ is bounded above and below, this satisfies
\begin{multline}
  |I_1|\leq
  C_bx_2^{\frac{\gamma}{2}}\iint\limits_{R_{\alpha,\beta,t}}
\left(\frac{w}{\lambda}\right)^{\frac{b}{2}-\frac{1}{4}}
e^{-\lambda\cost\left(1-\sqrt{\frac{w}{\lambda}}\right)^2}
\times\\
\left[\lambda\left(1-\sqrt{\frac{w}{\lambda}}\right)^2+1\right]
\left|1-\sqrt{\frac{w}{\lambda}}\right|^{\gamma}\frac{dwd\lambda}{\sqrt{w}\lambda}.
\end{multline}
We let $z=\sqrt{w/\lambda}-1,$ taking account that $z$ is bounded we obtain:
\begin{equation}\label{I1lrgprt0}
  |I_1|\leq
  C_bx_2^{\frac{\gamma}{2}}\int\limits_{\frac{x_2}{|t|}}^{\infty}
\int\limits_{\sqrt{\frac{\alpha}{x_2}}-1}^{\sqrt{\frac{\beta}{x_2}}-1}
e^{-\lambda\cost z^2}
\left[\lambda z^2+1\right]
|z|^{\gamma}\frac{dzd\lambda}{\sqrt{\lambda}}.
\end{equation}
We interchange the order of the integrations and set $x=\lambda z^2,$ in the
$\lambda$-integral, to see that:
\begin{equation}
   |I_1|\leq
  C_{b}x_2^{\frac{\gamma}{2}}
\int\limits_{\sqrt{\frac{\alpha}{x_2}}-1}^{\sqrt{\frac{\beta}{x_2}}-1}
\int\limits_{\frac{x_2z^2}{|t|}}^{\infty}
e^{-\cost x}\left(\frac{1}{\sqrt{x}}+\sqrt{x}\right)dx|z|^{\gamma-1}dz.
\end{equation}
The $x$-integral is bounded by a constant depending only on $\theta,$ and
this shows that there is a constant $C_{b,\theta}$ bounded for $0<b$ bounded, so that
\begin{equation}\label{I1finest}
  |I_1|\leq
  C_{b,\theta}\|g\|_{\WF,0,\gamma}|\sqrt{x_2}-\sqrt{x_1}|^{\gamma}.
\end{equation}

Now we turn to case 2. The foregoing analysis is used to estimate the part of
the integral where $w\lambda>1,$ by using $1$ as the lower limit of integration
in~\eqref{I1lrgprt0} instead of $x_2/|t|.$ This leaves the part of the integral
in~\eqref{I1strpt0} over the set
\begin{equation}
  R_{\alpha,\beta,t}\cap\{(w,\lambda):\:w\lambda<1\}.
\end{equation}
We replace this set, with the slightly larger set
\begin{equation}
  R'_{\alpha,\beta,t}=\{(w,\lambda):\:
\frac{\alpha}{x_2}\lambda<w<\frac{\beta}{x_2}\lambda\text{ and
}\frac{x_2}{|t|}\leq \lambda\leq \frac{x_2}{\alpha}\}.
\end{equation}
Using the Taylor series, we see that this term is bounded by
\begin{multline}
   Cx_2^{\frac{\gamma}{2}}\int\limits_{\frac{x_2}{|t|}}^{\frac{x_2}{\alpha}}
\int\limits_{\lambda\frac{\alpha}{x_2}}^{\lambda\frac{\beta}{x_2}}
w^{b-1}\left[(w+\lambda+b)\left(\frac{1}{\Gamma(b)}+w\lambda\right)+w\lambda\right]
\times
\\
\left|1-\sqrt{\frac{w}{\lambda}}\right|^{\gamma}\frac{dwd\lambda}{\lambda}
\end{multline}
In the $w$-integral we let $\sigma=w/\lambda$ to see that this is bounded by
\begin{multline}
  Cx_2^{\frac{\gamma}{2}}\int\limits_{0}^{\frac{x_2}{\alpha}}
\int\limits_{\frac{\alpha}{x_2}}^{\frac{\beta}{x_2}}
(\sigma\lambda)^{b-1}\left[(\sigma\lambda+\lambda+b)\left(\frac{1}{\Gamma(b)}+
\sigma\lambda^2\right)+\sigma\lambda^2\right]\times
\\
\left|1-\sqrt{\sigma}\right|^{\gamma}d\sigma d\lambda
\end{multline}
As $c$ in~\eqref{ratbnd0.1} is at least $1/3,$ we know that range of the
$\sigma$-integral satisfies
\begin{equation}
  \frac{\sqrt{3}-1}{2}\leq \sqrt{\sigma}\leq 1
\end{equation}
In the domain of the $\sigma$-integral, the quantity
$x_2^{\frac{\gamma}{2}}\left|1-\sqrt{\sigma}\right|^{\gamma}$ is bounded by a
constant multiple of $|\sqrt{x_2}-\sqrt{x_1}|^{\gamma}.$ As $\sigma$ is bounded
above and below, all that remains is the $\lambda$-integral. An elementary
calculation shows that it remains bounded, even as $b\to 0.$ This completes the
proof, in all cases, that there is a constant $C_b,$ bounded with $b,$ so
that~\eqref{lemCestp} holds for $I_1.$ The estimate for $I_2$ is essentially the
same. 
\end{proof}

\begin{lemmabis}\labelbis{lemD}
  For $b>0,$ $0<\gamma<1,$ $0<\phi<\frac{\pi}{2},$ and $0<x_2/3<x_1<x_2,$ if
  $J=[\alpha,\beta],$ with the endpoints given by~\eqref{eqn8555}, there is a
  constant $C_{b,\phi}$ so that if $|\theta|<\frac{\pi}{2}-\phi,$ then
  \begin{equation}\label{lemDestp}
    \int\limits_0^t\int\limits_{J^c}|L_bk^b_{se^{i\theta}}(x_2,y)-L_bk^b_{se^{i\theta}}(x_1,y)|
|\sqrt{y}-\sqrt{x_1}|^{\gamma}dyds\leq C_{b,\phi}|\sqrt{x_2}-\sqrt{x_1}|^{\gamma}.
  \end{equation}
\end{lemmabis}
\begin{proof}
We use the formula for
$L_bk^b_t,$ given in~\eqref{Lbkb0}, hence:
  \begin{multline}
    I^{+}=\int\limits_{0}^{|t|}\int\limits_{\beta}^{\infty}\left(\frac{y}{s}\right)^{b-1}
e^{-\cost\frac{y}{s}}
\Bigg|e^{-\frac{x_2\et}{s}}\left[\left(\frac{(x_2+y)\et}{s}-b\right)\psi_b\left(\frac{x_2y\ett}{s^2}\right)-
2\left(\frac{x_2y\ett}{s^2}\right)\psi_b'\left(\frac{x_2y\ett}{s^2}\right)\right]\\
-e^{-\frac{x_1\et}{s}}\left[\left(\frac{(x_1+y)\et}{s}-b\right)\psi_b\left(\frac{x_1y\ett}{s^2}\right)-
2\left(\frac{x_1y\ett}{s^2}\right)\psi_b'\left(\frac{x_1y\ett}{s^2}\right)\right]\Bigg|
|\sqrt{y}-\sqrt{x_1}|^{\gamma}\frac{dwds}{s^2},
\end{multline}
and
 \begin{multline}
    I^{-}=\int\limits_{0}^{|t|}\int\limits_{0}^{\alpha}\left(\frac{y}{s}\right)^{b-1}
e^{-\cost\frac{y}{s}}
\Bigg|e^{-\frac{x_2\et}{s}}\left[\left(\frac{(x_2+y)\et}{s}-b\right)
\psi_b\left(\frac{x_2y\ett}{s^2}\right)-
2\left(\frac{x_2y\ett}{s^2}\right)\psi_b'\left(\frac{x_2y\ett}{s^2}\right)\right]\\
-e^{-\frac{x_1\et}{s}}\left[\left(\frac{(x_1+y)\et}{s}-b\right)\psi_b\left(\frac{x_1y\ett}{s^2}\right)-
2\left(\frac{x_1y\ett}{s^2}\right)\psi_b'\left(\frac{x_1y\ett}{s^2}\right)\right]\Bigg
|\sqrt{y}-\sqrt{x_1}|^{\gamma}\frac{dwds}{s^2}.
\end{multline}
For this case we give the details for $I^{-},$ and leave $I^{+}$ to
the interested reader.

We change variables, setting
\begin{equation}
  w=\frac{y}{s}\quad \lambda=\frac{x_1}{s}\text{ so that
  }\frac{dyds}{s^2}=\frac{dw d\lambda}{\lambda};
\end{equation}
we also let $\mu= x_2/x_1.$ The integral now satisfies:
\begin{multline}
    |I^{-}|\leq 
\int\limits_{\frac{x_1}{|t|}}^{\infty}\int\limits_{0}^{\frac{\alpha\lambda}{x_1}}
w^{b-1}
e^{-\cost w}
\Bigg|e^{-\mu\lambda\et}\left[[(\mu\lambda+w)\et-b]\psi_b\left(\mu\lambda w\ett\right)-
2\left(\mu\lambda w\ett\right)\psi_b'\left(\mu\lambda w\ett\right)\right]-\\
e^{-\lambda\et}\left[\left[(\lambda+w)\et-b\right]\psi_b\left(\lambda w\ett\right)-
2\left(\lambda w\ett\right)\psi_b'\left(\lambda w\ett\right)\right]\Bigg|
|\sqrt{\lambda}-\sqrt{w}|^{\gamma}\frac{x_1^{\frac{\gamma}{2}}dwd\lambda}{\lambda^{1+\frac{\gamma}{2}}}.
\end{multline}
We split the $w$-integral into the part, $I^{--}(\lambda)$ with $w\in
[0,\frac{1}{\lambda}],$ and the rest, $I^{-+}(\lambda),$ which only arises
when $\lambda^2>x_1/\alpha.$

To estimate $I^{--}(\lambda)$ we let
\begin{equation}
  F(\mu,\lambda,w)=e^{-\mu\lambda\et}\left[[(\mu\lambda+w)\et-b]\psi_b\left(\mu\lambda w\ett\right)-
2\left(\mu\lambda w\ett\right)\psi_b'\left(\mu\lambda w\ett\right)\right].
\end{equation}
It follows from Lemma~\ref{MVTineq} that for some $\xi\in(1,\mu)$ we have that
\begin{equation}\label{taylorfrm3}
  \Bigg|F(\mu,\lambda,w)-F(1,\lambda,w)\Bigg|\leq(\mu-1)|\pa_{\mu}F(\xi,\lambda,w)|.
\end{equation}
In the set $\{w\lambda<1\},$ we have the bound (see~\eqref{dFmu0}):
\begin{equation}
  |\pa_{\mu}F(\xi,\lambda,w)|\leq
C_be^{-\cost\xi\lambda}\lambda(1+\lambda+w)\left[\frac{1}{\Gamma(b)}+(1+\lambda) w\right].
\end{equation}
In this case the $w$-integral is bounded by
\begin{multline}
 C_b x_1^{\frac{\gamma}{2}}(\mu-1)\int\limits_{0}^{\frac{1}{\lambda}}
w^{b-1}e^{-\cost(w+\xi\lambda)}\left[\frac{(1+\lambda)}{\Gamma(b)}+w(1+\lambda)^2+w^2(1+\lambda)\right]\times\\
|\sqrt{\lambda}-\sqrt{w}|^{\gamma}\frac{dwd\lambda}{\lambda^{\frac{\gamma}{2}}},
\end{multline}
and therefore
\begin{equation}
 I^{--}(\lambda)\leq   C_{b,\theta}\|g\|_{\WF,0,\gamma}
 x_1^{\frac{\gamma}{2}}(\mu-1)
\frac{e^{-\cost\lambda}(1+\lambda^{1+\frac{\gamma}{2}})}
{\lambda^{\frac{\gamma}{2}}(1+\lambda^b)}. 
\end{equation}
As this is integrable from $0$ to $\infty,$ we see that 
\begin{equation}\label{eqn218}
  I^{--}\leq   C_{b,\theta} \frac{x_2-x_1}{x_1^{1-\frac{\gamma}{2}}}.
\end{equation}
As $x_2/x_1$ is bounded from above, it follows immediately that
\begin{equation}\label{eqn219}
  I^{--}\leq   C_{b,\theta}|\sqrt{x_2}-\sqrt{x_1}|^{\gamma}.
\end{equation}

This leaves only $I^{-+},$ which is estimated by
\begin{multline}
  |I^{-+}|\leq 
 x_1^{\frac{\gamma}{2}}(\mu-1)\int\limits_{\max\{\frac{x_1}{|t|},\sqrt{\frac{x_1}{\alpha}}\}}^{\infty}
\int\limits_{\frac{1}{\lambda}}^{\frac{\alpha\lambda}{x_1}}
w^{b-1}e^{-\cost w}|\sqrt{\lambda}-\sqrt{w}|^{\gamma}\times\\
|F_{\mu}(\xi,\lambda,w)|\frac{dwd\lambda}
{\lambda^{1+\frac{\gamma}{2}}}
\end{multline}
As before we apply Lemma~\ref{MVTineq} as in~\eqref{taylorfrm3} to see that we
need to estimate:
\begin{multline}\label{dFmu0}
  |\pa_{\mu}F(\xi,\lambda,w)|=
e^{-\cost\xi\lambda}\big|\lambda\et\psi_b(\xi\lambda w\ett)(1-(\xi\lambda+w)\et+b)+\\
\lambda w\ett\psi_b'(\xi\lambda w\ett)[(3\xi\lambda +w)\et+b -2]-2\xi(\lambda
w\ett)^2\psi_b''(\xi\lambda w\ett)\big|,
\quad\xi\in[1,\frac{x_2}{x_1}].
\end{multline}
To get a controllable error  term, we must use the asymptotic expansions for
$\psi_b, \psi_b'=\psi_{b+1}$ through second order:
\begin{equation}\label{2ndordasymp00}
  \begin{split}
    \psi_b(z)&=\frac{z^{\frac 14-\frac b2}e^{2\sqrt{z}}}{\sqrt{4\pi}}
\left[ 1-\frac{(2b-1)(2b-3)}{16\sqrt{z}}+O(\frac{1}{z})\right]\\
 \psi_b'(z)&=\frac{z^{-\frac 14-\frac b2}e^{2\sqrt{z}}}{\sqrt{4\pi}}
\left[ 1-\frac{(2b+1)(2b-1)}{16\sqrt{z}}+O(\frac{1}{z})\right].
  \end{split}
\end{equation}
Using the equation
$z\psi_b''=\psi_b-b\psi_b',$ and inserting these relations into~\eqref{dFmu0},  gives 
\begin{multline}\label{eqn222}
  |\pa_{\mu}F(\xi,\lambda,w)|\leq\lambda e^{\cost(2\sqrt{\xi\lambda w}-\xi\lambda)}\frac{(\xi\lambda
  w)^{\frac 14-\frac b2}}{\sqrt{4\pi}}\Bigg[
\xi\lambda\left(\sqrt{\frac{w}{\lambda\xi}}-1\right)^3+
a_1(b)\left(\sqrt{\frac{w}{\lambda\xi}}-1\right)+\\
a_2(b)\sqrt{\frac{w}{\lambda\xi}}\left(\sqrt{\frac{w}{\lambda\xi}}-1\right)+
a_3(b)\left(\sqrt{\frac{w}{\lambda\xi}}-\sqrt{\frac{\lambda\xi}{w}}\right)+
O\left(\frac{1}{\lambda}+\frac{1}{w}+\frac{1}{\sqrt{\lambda w}}\right)\Bigg].
\end{multline}
Here $a_1(b), a_2(b)$ and $a_3(b)$ are polynomials in $b.$ We denote the
contributions of these
terms by $M_0, M_1,M_2, M_3, M_e.$

We first consider $M_0:$
\begin{multline}
 M_0\leq C_b
 x_1^{\frac{\gamma}{2}}(\mu-1)\int\limits_{\max\{\frac{x_1}{|t|},\sqrt{\frac{x_1}{\alpha}}\}}^{\infty}
\int\limits_{\frac{1}{\lambda}}^{\frac{\alpha\lambda}{x_1}}
\left(\frac{w}{\lambda}\right)^{\frac{b}{2}-\frac{1}{4}}e^{-\cost(\sqrt{w}-\sqrt{\xi\lambda})^2}\times\\
|\sqrt{\lambda}-\sqrt{w}|^{\gamma}
|\sqrt{\xi\lambda}-\sqrt{w}|^{3}\frac{dwd\lambda}
{\sqrt{w}\lambda^{\frac{1+\gamma}{2}}}
\end{multline}
In this integral $w<\lambda$ and $1\leq \xi\leq x_2/x_1,$ and therefore this is
bounded by
  \begin{multline}
  M_0\leq C_b
 x_1^{\frac{\gamma}{2}}(\mu-1)\int\limits_{\max\{\frac{x_1}{|t|},\sqrt{\frac{x_1}{\alpha}}\}}^{\infty}
\int\limits_{\frac{1}{\lambda}}^{\frac{\alpha\lambda}{x_1}}
\left(\frac{w}{\lambda}\right)^{\frac{b}{2}-\frac{1}{4}}e^{-\cost(\sqrt{w}-\sqrt{\xi\lambda})^2}\times\\
|\sqrt{\xi\lambda}-\sqrt{w}|^{3+\gamma}
\frac{dwd\lambda}{\sqrt{w}\lambda^{\frac{1+\gamma}{2}}}
\end{multline}
We now let $z=\sqrt{w/\lambda}-\sqrt{\xi}$ to obtain that
\begin{multline}
  M_0\leq C_b
 x_1^{\frac{\gamma}{2}}(\mu-1)\int\limits_{\max\{\frac{x_1}{|t|},\sqrt{\frac{x_1}{\alpha}}\}}^{\infty}
\int\limits_{\frac{1}{\lambda}-\sqrt{\xi}}^{\sqrt{\frac{\alpha}{x_1}}-\sqrt{\xi}}
(\sqrt{\xi}+z)^{b-\frac 12}e^{-\cost\lambda z^2}\times\\
|z|^{3+\gamma}
\lambda^{\frac{3}{2}}dzd\lambda
\end{multline}

As $\lambda$ is bounded from below by $\sqrt{x_1/\alpha},$ we need to estimate
the $z$-integral as $\lambda\to\infty.$ We apply Lemma~\ref{lem3} to see that
\begin{multline}\label{eqnA.226.00}
 \int\limits_{\frac{1}{\lambda}-\sqrt{\xi}}^{\sqrt{\frac{\alpha}{x_1}}-\sqrt{\xi}}
(\sqrt{\xi}+z)^{b-\frac 12}e^{-\cost\lambda z^2}|z|^{3+\gamma}dz\leq\\
\begin{cases}
  &\frac{C_{b,\theta}}{\lambda}e^{-\cost\frac{\lambda(\sqrt{x_2}-\sqrt{x_1})^2}{4x_1}}
\left(\frac{\sqrt{x_2}-\sqrt{x_1}}{\sqrt{x_1}}\right)^{2+\gamma}\text{ if
  }\sqrt{\lambda}\left(\sqrt{\frac{x_2}{x_1}}-1\right)>1\\
  &\frac{C_{b,\theta}}{\lambda^{\frac{4+\gamma}{2}}}
  \text{ if
  }\sqrt{\lambda}\left(\sqrt{\frac{x_2}{x_1}}-1\right)\leq 1.
\end{cases}
\end{multline}
The large $\lambda$ contribution (the first estimate in~\eqref{eqnA.226.00})
leads to terms of the form
\begin{multline}
  C_{b,\theta}|\sqrt{x_2}-\sqrt{x_1}|^{\gamma} \frac{x_2-x_1}{x_1}
\frac{\sqrt{x_1}}{\sqrt{x_2}-\sqrt{x_1}}\\
\leq C_{b,\theta}
\left(\frac{\sqrt{x_2}+\sqrt{x_1}}{\sqrt{x_1}}\right)
 |\sqrt{x_2}-\sqrt{x_1}|^{\gamma},
\end{multline}
as above (see~\eqref{eqn218}--~\eqref{eqn219}).  Integrating the second
estimate in~\eqref{eqnA.226.00} over
$$\lambda\in \left[\max\left\{\frac{x_1}{|t|},\sqrt{\frac{x_1}{\alpha}} \right\},
\frac{x_1}{(\sqrt{x_2}-\sqrt{x_1})^2}\right],$$
  gives a term bounded by
  \begin{equation}\label{eqn214}
     C_{b,\theta}
     \left(\frac{\sqrt{x_2}+\sqrt{x_1}}{\sqrt{x_1}}\right)
|\sqrt{x_2}-\sqrt{x_1}|^{\gamma}.
  \end{equation}
This completes the proof that $M_0$ satisfies the desired
bound. 

Using the same change of variables we see that $M_1$ and $M_2$ are also
bounded by the quantity in~\eqref{eqn214}. To treat $M_3$ we let
$z=\sqrt{w/\lambda};$ this gives the bound:
\begin{multline}
  |M_3|\leq C_b
 x_1^{\frac{\gamma}{2}}(\mu-1)\int\limits_{\max\{\frac{x_1}{|t|},\sqrt{\frac{x_1}{\alpha}}\}}^{\infty}
\int\limits_{\frac{1}{\lambda}}^{\sqrt{\frac{\alpha}{x_1}}}
z^{b-\frac 12}e^{-\cost\lambda (\sqrt{\xi}-z)^2}\times\\
|\sqrt{\xi}-z|^{\gamma}\frac{|z^2-\xi|}{z\sqrt{\xi}}
\sqrt{\lambda}dzd\lambda.
\end{multline}
As $\lambda\to\infty,$ the part of the $z$-integral from $1/\lambda$ to $1/2$
(e.g.) is bounded by a constant multiple of
$\lambda^{1-b}e^{-\cost\frac{\lambda}{4}},$ and so contributes term to $M_3$ that
satisfies the desired estimate. 

We are left to estimate the contribution from
near the diagonal, i.e. for $\lambda\in [1/2,\sqrt{\alpha/x_1}].$ If
$\sqrt{\lambda}(\sqrt{x_2/x_1}-1)>1/2,$ then the $z$-integral is bounded by
\begin{equation}
  C_{b,\theta}\left|\frac{\sqrt{x_2}-\sqrt{x_1}}{2\sqrt{x_1}}\right|^{\gamma}
\frac{e^{-\cost\lambda\left(\frac{\sqrt{x_2}-\sqrt{x_1}}{2\sqrt{x_1}}\right)^2}}{\lambda};
\end{equation}
the contribution of this term satisfies the desired bound. If
$\sqrt{\lambda}(\sqrt{x_2/x_1}-1)<1/2,$ then the $z$-integral is bounded by
$\frac{C_{b,\theta}}{\lambda^{\frac{2+\gamma}{2}}}.$
Integrating in $\lambda$ completes the proof that
\begin{equation}
  |M_3|\leq C_{b,\theta}
     \left(\frac{\sqrt{x_2}+\sqrt{x_1}}{\sqrt{x_1}}\right)
|\sqrt{x_2}-\sqrt{x_1}|^{\gamma}.
\end{equation}

To complete the estimate of $I^{-+},$ and thereby of $I^{-},$ we only need
to show that the error terms satisfy the desired bound. To that end we let
$z=\sqrt{w/\lambda};$ the contribution of the error terms is bounded by
\begin{multline}
  |M_e|\leq C_b
 x_1^{\frac{\gamma}{2}}(\mu-1)\int\limits_{\max\{\frac{x_1}{|t|},\sqrt{\frac{x_1}{\alpha}}\}}^{\infty}
\int\limits_{\frac{1}{\lambda}}^{\sqrt{\frac{\alpha}{x_1}}}
z^{b-\frac 12}e^{-\cost\lambda (\sqrt{\xi}-z)^2}\times\\
|\sqrt{\xi}-z|^{\gamma}\left(\frac{1}{\sqrt{\lambda}}+\frac{1}{z}+\frac{1}{\sqrt{\lambda}
    z}
\right)dzd\lambda.
\end{multline}
Arguing as above, we see that these terms all satisfy the desired bound. As
noted, the estimate of $I^+$ is quite similar and is left to the reader. 
\end{proof}

\begin{lemmabis}\labelbis{lemH} 
  For $ $b>0,$ 0<\gamma<1,$ and $t_1<t_2<2t_1$ there is a constant $C_b$ so
  that
  \begin{equation}\label{lemHestp} 
    \int\limits_{t_2-t_1}^{t_1}\int\limits_0^{\infty}
   |L_bk^{b}_{t_2-t_1+s}(x,y)-L_bk^{b}_{s}(x,y)||\sqrt{x}-\sqrt{y}|^{\gamma}dyds
\leq C_b|t_2-t_1|^{\frac{\gamma}{2}}.
  \end{equation}
\end{lemmabis}

This lemma follows from the more basic:
\begin{lemmabis}\labelbis{lemHp2}  For  $b>0,$ $0<\gamma<1,$ and $t_1<t_2<2t_1$ and
  $s>t_2-t_1,$ there is a constant $C_b$ so that
  \begin{equation}\label{lemHestp20} 
    \int\limits_0^{\infty}
   |L_bk^{b}_{t_2-t_1+s}(x,y)-L_bk^{b}_{s}(x,y)||\sqrt{x}-\sqrt{y}|^{\gamma}dy
\leq C_b (t_2-t_1)s^{\frac{\gamma}{2}-2}.
  \end{equation}
\end{lemmabis}
\begin{proof}[Proof of Lemma~\ref{lemH}] The derivation of~\eqref{lemHestp}
  from~\eqref{lemHestp2}  is quite easy:
  \begin{equation}
    \begin{split}
      \int\limits_{t_2-t_1}^{t_1}\int\limits_0^{\infty}
   |L_bk^{b}_{t_2-t_1+s}(x,y)-L_bk^{b}_{s}(x,y)||\sqrt{x}-\sqrt{y}|^{\gamma}dyds
&\leq C(t_2-t_1)\int\limits_{t_2-t_1}^{t_1}s^{\frac{\gamma}{2}-2}ds\\
&\leq C_b|t_2-t_1|^{\frac{\gamma}{2}}.
    \end{split}
  \end{equation}
\end{proof}

\begin{proof}[Proof of Lemma~\ref{lemHp2}]
  To prove~\eqref{lemHestp2} we need to apply Taylor's formula to estimate the difference
  $L_bk^b_{t_2-t_1+s}(x,y)-L_bk^b_{s}(x,y).$ To that end, we let
  $F(\tau,s,x,y)=L_bk^b_{\tau+s}(x,y);$ we denote the left hand side
  in~\eqref{lemHestp2}  as $I,$ which we can rewrite as
\begin{equation}
  I=\int\limits_0^{\infty}[F(\tau,s,x,y)-F(0,s,x,y)]
|\sqrt{x}-\sqrt{y}|^{\gamma}dy,
\end{equation}
here $\tau=t_2-t_1.$ From the mean value theorem, we get the estimate
\begin{equation}
|I|\leq
  \tau\int\limits_0^{\infty}
|\pa_{\tau}F(\xi,s,x,y)||\sqrt{y}-\sqrt{x}|^{\gamma} dy.
\end{equation}
Using the differential equation satisfied by $\psi_b$ we can show that
\begin{multline}
  \pa_{\tau}F(\xi,s,x,y)=\frac{y^{b-1}e^{-\left(\frac{x+y}{s+\xi}\right)}}{(s+\xi)^{b+2}}\Bigg\{
\psi_b\left(\frac{xy}{(s+\xi)^2}\right)\Bigg[\left(\frac{x+y}{s+\xi}-b\right)
\left(\frac{x+y}{s+\xi}-(b+1)\right)-\\
\left(\frac{x+y}{s+\xi}\right)
+\frac{4xy}{(s+\xi)^2}\Bigg]+
\frac{2xy}{(s+\xi)^2}\psi_b'\left(\frac{xy}{(s+\xi)^2}\right)\left[3-2\left(\frac{x+y}{s+\xi}\right)\right]\Bigg\}.
\end{multline}
Since $s\in [t_2-t_1,t_1]$ and $\xi\in [0,t_2-t_1],$ we see that $s<s+\xi<2s,$
and therefore
\begin{equation}
  \frac{xy}{4s^2} \leq \frac{xy}{(s+\xi)^2}\leq\frac{xy}{s^2},
\end{equation}
we can therefore split the $y$-integral into a compact part with $y\in[0,\frac{4s^2}{x}],$
$I^-$ and the remaining non-compact part $I^+.$ In the compact part we
estimate use the usual estimates for $\psi_b$ and  $\psi_b'.$ Setting $w=y/s,$ $\lambda=x/s,$
we obtain that
\begin{multline}
  |I^-|\leq 
s^{\frac{\gamma}{2}-2}\tau\int\limits_{0}^{\frac{4}{\lambda}}
w^{b-1}e^{-\frac{1}{2}(w+\lambda)}\times
\\
\left[\left(\lambda+w+1\right)^2\left(\frac{1}{\Gamma(b)}+w\lambda\right)+w\lambda+
w+\lambda\right]
\left|\sqrt{\lambda}-\sqrt{w}\right|^{\gamma}dw.
\end{multline}
If $\lambda$ is bounded then we easily see that this satisfies:
\begin{equation}
  |I^-|\leq 
C_b\tau s^{\frac{\gamma}{2}-2}.
\end{equation}
In the case that $\lambda$ is large, then we see that
\begin{equation}
  |I^-|\leq 
Cs^{\frac{\gamma}{2}-2}\tau e^{-\frac{\lambda}{2}},
\end{equation}
which therefore applies for $\lambda\in [0,\infty).$

To estimate $I^+,$ we use the second order asymptotic
expansions for $\psi_b$ and $\psi_b'$ to see that
\begin{multline}
  |\pa_{\tau}F(\xi,s,x,y)|=
\left(\frac{w}{\mu}\right)^{\frac{b}{2}-\frac{1}{4}}
\frac{e^{-(\sqrt{\mu}-\sqrt{w})^2}}{\sqrt{w}(s+\xi)^3}
\Big[(b-1)^2(\sqrt{w}-\sqrt{\mu})^4+\\ \frac{9}{4}(\sqrt{w}-\sqrt{\mu})^2+
O(1+\sqrt{w/\mu}+\sqrt{\mu/w})\Big],
\end{multline}
here
\begin{equation}
  w=\frac{y}{s+\xi}\quad\text{ and }\mu=\frac{x}{s+\xi}.
\end{equation}
From this expansion, and the fact that $s\leq s+\xi\leq 2s,$ it follows that
\begin{multline}
  |I^+|\leq 
C_b\frac{\tau}{s^3}
\int\limits_{\frac{4s^2}{x}}^{\infty}
\left(\frac{y}{x}\right)^{\frac{b}{2}-\frac{1}{4}} \sqrt{\frac{s}{y}}
e^{-\frac{x}{2s}\left(1-\sqrt{\frac{y}{x}}\right)^2}|\sqrt{y}-\sqrt{x}|^{\gamma}\times\\
\left[\left(\sqrt{\frac ys}-\sqrt{\frac xs}\right)^4+
\left(\sqrt{\frac ys}-\sqrt{\frac xs}\right)^2+O\left(1+\sqrt{\frac xy}+
\sqrt{\frac yx}\right)\right]dy
\end{multline}
To estimate this integral, we let $z=\sqrt{y/x}-1,$ and $\lambda=x/s,$ obtaining:
\begin{multline}
  |I^+|\leq 
C_b\frac{\tau x^{\frac{\gamma}{2}}\sqrt{\lambda}}{s^2}
\int\limits_{\frac{2}{\lambda}-1}^{\infty}
(z+1)^{b-\frac{1}{2}} 
e^{-\frac{\lambda}{2}z^2}|z|^{\gamma}\times\\
\left[\lambda^2z^4+\lambda z^2
+O\left(1+z+\frac{1}{1+z}\right)\right]dz.
\end{multline}
If $\lambda\to 0,$ then the integral behaves like $e^{-\frac{1}{4\lambda}}.$ As
$\lambda\to\infty,$ an application of Laplace's method shows that the
$z$-integral behaves like $\lambda^{-\frac{1+\gamma}{2}},$  which, in turn,
establishes~\eqref{lemHestp2}.
\end{proof}

\section{Off-diagonal and Large $t$ Behavior}
We close this section with estimates valid for $t,$ with positive real part,
which do not use an  assumption about the H\"older continuity of the data.
\begin{lemmabis}\labelbis{lrgt1db} For $0<b<B,$  $0<\phi<\frac{\pi}{2},$ and
  $j\in\bbN$ there is a constant $C_{j,B,\phi}$ so that if $t\in S_{\phi},$ then
  \begin{equation}\label{eqnA244}
    \int\limits_{0}^{\infty}|\pa_x^jk^b_t(x,y)|dy\leq \frac{C_{j,B,\phi}}{|t|^j},
  \end{equation}
and
\begin{equation}\label{eqnA245}
    \int\limits_{0}^{\infty}|x^{\frac{j}{2}}\pa_x^jk^b_t(x,y)|dy\leq
    \frac{C_{j,B}}{|t|^{{\frac{j}{2}}}}
  \end{equation}
\end{lemmabis}
\begin{proof} The proof of this lemma is easier than results proved above for
  small $t$ behavior. We observe that for $0<b<B$ we can write
  \begin{equation}
    k^b_t(x,y)=\frac{1}{y}F\left(\frac{x}{t},\frac{y}{t}\right),
  \end{equation}
where, for $z$ and $\zeta$ in the right half plane,
\begin{equation}
  F(\zeta,z)=z^{b}e^{-(z+\zeta)}\psi_b(z\zeta).
\end{equation}
This expression easily implies that
\begin{equation}
  \pa_x^jk^t_b(x,y)=\frac{1}{yt^{j}}\pa_{\zeta}^jF\left(\frac{x}{t},\frac{y}{t}\right).
\end{equation}
From the form of $F$ and the fact that $\pa_z\psi_b(z)=\psi_{b+1}(z),$we see
that a simple induction establishes:
\begin{equation}\label{eqnA.249}
  \pa_{\zeta}^jF(\zeta,z)=z^{b}e^{-(z+\zeta)}\sum_{l=0}^j
\left(\begin{matrix}j\\l\end{matrix}\right)(-1)^{j-l}z^l\psi_{b+l}(z\zeta).
\end{equation}

We let $t=|t|e^{i\theta},$ with $|\theta|<\frac{\pi}{2}-\phi,$ and set
$w=y/|t|,$ $\lambda =x/|t|.$ To complete the proof of~\eqref{eqnA244} it
suffices to show that there are constants $C_{l,B,\phi}$ so that
\begin{equation}
  \int\limits_{0}^{\infty}w^{b+l}
e^{-\cos\theta(w+\lambda)}|\psi_{b+l}(w\lambda e_{2\theta})|\frac{dw}{w}\leq C_{l,B,\phi}.
\end{equation}
When $l=0$ the integral is bounded in Lemma~\ref{lem9.1.3.00}, so we can assume
that $l\geq 1.$

We need to estimate
\begin{equation}
 I_{j,b+l}= \frac{1}{|t|^j}\int\limits_0^{\infty}
w^{b+l}e^{-\cos\theta(w+\lambda)}|\psi_{b+l}(w\lambda e_{2\theta})|\frac{dw}{w}.
\end{equation} 
We split the integral into the part from $[0,1/\lambda]$ and the rest; applying the
asymptotic formula we obtain that
\begin{equation}
  I_{j,b+l}\leq\frac{C_{b+l}}{|t|^j}\left[e^{-\cos\theta\lambda}\int\limits_0^{\frac{1}{\lambda}}
w^{b+l}e^{-\cos\theta w}\frac{dw}{w}+
\int\limits_{\frac{1}{\lambda}}^{\infty}
\left(\frac{w}{\lambda}\right)^{\frac{b+l}{2}-\frac
  14}e^{-\cos\theta(\sqrt{w}-\sqrt{\lambda})^2}
\frac{dw}{\sqrt{w}}\right].
\end{equation}
The first term in the brackets is bounded by $\Gamma(b+l)e^{-\cos\theta\lambda},$ and the
second term is rapidly decaying as $\lambda\to 0.$ To study the second term as
$\lambda\to\infty,$ we let
$z=\sqrt{w}-\sqrt{\lambda},$ to see that the second integral is bounded by
\begin{equation}
  C_{b+l}\int\limits_{\frac{1}{\sqrt{\lambda}}-\sqrt{\lambda}}^{\infty}
\left(1+\frac{z}{\sqrt{\lambda}}\right)^{b+l-\frac 12}e^{- \cos\theta z^2}dz
\end{equation}
As $b+l-\frac 12>0,$ it follows easily that this integral is bounded as
$\lambda\to\infty,$ which completes the proof of~\eqref{eqnA244}. 

The estimate in~\eqref{eqnA245} for $x/|t|<1$ follows immediately from this
formula. To prove~\eqref{eqnA245} for $x/|t|\geq 1$  requires more
careful consideration. Using~\eqref{eqnA.249} and the asymptotic expansions for
$\psi_{b+l}$ we see that
\begin{multline}
  |\pa_x^jk_t(x,y)|=\frac{(-1)^j}{t^j}\left(\frac{z}{\zeta}
\right)^{\frac b2-\frac 14}e^{-(\sqrt{z}-\sqrt{\zeta})^2}\times\\
\left[\sum\limits_{l=0}^j\left(\begin{matrix} j\\l\end{matrix}\right)(-1)^l
\left(\frac{z}{\zeta}\right)^{\frac{l}{2}}\left[
\sum_{k=0}^{[\frac j2]}\frac{(-1)^k \Gamma(b+l+k-\frac 12)}
{4^k(z\zeta)^{\frac k2}\Gamma(b+l-k-\frac 12)}+O\left(\frac{1}{(z\zeta)^{\frac
      j4}}\right)
\right]\right].
\end{multline}
We observe that the ratios of $\Gamma$-functions are polynomials in $l,$ which can be
expressed as
\begin{equation}
  \frac{\Gamma(b+l+k-\frac 12)}{\Gamma(b+l-k-\frac
    12)}=p_{k,0}(b)+\sum_{m=1}^{2k}p_{k,m}(b)l(l-1)\cdots
(l-m+1).
\end{equation}
The coefficients $\{p_{k,m}(b)\}$ are polynomials in $b.$
Putting this expression into the previous formula and using the fact that
\begin{equation}
  (1-u)^j=\sum\limits_{l=0}^j\left(\begin{matrix} j\\l\end{matrix}\right)(-1)^{l}u^{l},
\end{equation}
we see  there are
polynomials, $P_{j,k}(b,u)$ in $(b,u),$ so that
\begin{multline}
  \pa_x^jk_t(x,y)dy=\frac{(-1)^j}{t^j}\left(\frac{z}{\zeta}
\right)^{\frac b2-\frac 14}e^{-(\sqrt{z}-\sqrt{\zeta})^2}\times\\
\Bigg[\sum_{k=0}^{[\frac
    j2]}\frac{\left(1-\sqrt{\frac{z}{\zeta}}\right)^{j-2k}}
{(z\zeta)^{\frac k2}}P_{j,k}\left(b,\sqrt{\frac{z}{\zeta}}\right)+\\
\left[\sum\limits_{l=0}^j\left(\begin{matrix} j\\l\end{matrix}\right)(-1)^l
\left(\frac{z}{\zeta}\right)^{\frac{l}{2}}\right]\cdot O\left(\frac{1}{(z\zeta)^{\frac
    j4}}\right)\Bigg].
\end{multline}
Using this expression and the analysis from the previous case we easily show that
\begin{equation}
 \int\limits_{0}^{\infty}|x^{\frac j2}\pa_x^jk^b_t(x,y)|dy\leq
 \frac{C_{b,j,\phi}}
{|t|^{\frac  j2}},
\end{equation}
which completes the proof of the lemma.
\end{proof}

\begin{lemmabis}\labelbis{lrgt1de} For $j\in\bbN$ 
  and $0<\phi<\frac{\pi}{2}$ there is a constant $C_{j,\phi}$ so that if $t\in
  S_{\phi},$ then
  \begin{equation}
    \int\limits_{-\infty}^{\infty}|\pa_x^jk^e_t(x,y)|dy\leq \frac{C_{j,\phi}}{|t|^{\frac{j}{2}}}.
  \end{equation}
\end{lemmabis}
\begin{proof}
These estimates, which are classical, follow easily from homogeneity
considerations, and the formula
\begin{equation}
  \pa_x^jk^e_t(x,y)=\frac{1}{t^{\frac{j}{2}}}\sum_{l=0}^jc_{j,l}
\left(\frac{x-y}{2\sqrt{t}}\right)^lk^e_t(x,y).
\end{equation}
\end{proof}

We consider the off-diagonal behavior.
\begin{lemmabis}\labelbis{lem12.2.1} Let $b>0,$ $\eta>0$ and for $x\in\bbR_+$
  define the set
  \begin{equation}
    J_{x,\eta}=\{y\in\bbR_+:|\sqrt{x}-\sqrt{y}|\geq\eta\}.
  \end{equation}
  For $0\leq b<B,$ $0<\phi<\frac{\pi}{2},$ and $j\in\bbN_0$ there is a constant
  $C_{\eta, j,B,\phi}$ so that if $t =|t|e^{i\theta},$ with $|\theta|\leq
  \frac{\pi}{2}-\phi$, then
\begin{equation}
  \int\limits_{J_{x,\eta}}|\pa_x^jk^b_t(x,y)|dy\leq
  C_{\eta,j,B,\phi}\frac{e^{-\cos\theta\frac{\eta^2}{2|t|}}}{|t|^{j}}.
\end{equation}
\end{lemmabis}
For the Euclidean models we have
\begin{lemmabis}\labelbis{lem12.2.2} Let $\eta>0$ and for $x\in\bbR$ define the set
  \begin{equation}
    J_{x,\eta}=\{y\in\bbR:|x-y|\geq\eta\}.
  \end{equation}
For $j\in\bbN_0,$ $0<\phi<\frac{\pi}{2},$  there is a constant
  $C_{\eta, j,\phi}$ so that if $t =|t|e^{i\theta},$ with $|\theta|\leq
  \frac{\pi}{2}-\phi$, then 
\begin{equation}
  \int\limits_{J_{x,\eta}}|\pa_x^jk^e_t(x,y)|dy\leq
  C_{\eta,j,\phi}\frac{e^{-\cos\theta\frac{\eta^2}{8t}}}{|t|^{\frac j2}}.
\end{equation}
\end{lemmabis}

\begin{proof}[Proof of Lemma~\ref{lem12.2.1}] Recall that for $0<b,$ the kernel is given by
  \begin{equation}
    k^b_t(x,y)=\frac{1}{y}\left(\frac{y}{t}\right)^be^{-\frac{x+y}{t}}
\psi_b\left(\frac{xy}{t^2}\right),
  \end{equation}
where
\begin{equation}
  \psi_b(z)=\sum_{j=0}^{\infty}\frac{z^j}{j!\Gamma(j+b)}.
\end{equation}
Using a simple inductive argument, and the
fact that
\begin{equation}
  \pa_z^l\psi_b(z)=\psi_{b+l}(z),
\end{equation}
 we can show that there are constants
$\{c_{j,l}\}$ so that
\begin{equation}
  \pa_x^jk^b_t(x,y)=
\frac{1}{yt^j}\left(\frac{y}{t}\right)^be^{-\frac{x+y}{t}}\sum_{l=0}^{j}
c_{j,l}\left(\frac{y}{t}\right)^l\psi_{b+l}\left(\frac{xy}{t^2}\right).
\end{equation}
To prove the assertion of the lemma, it therefore suffices to prove it for each
function,
\begin{equation}
  \frac{1}{yt^j}\left(\frac{y}{t}\right)^be^{-\frac{x+y}{t}}
\left(\frac{y}{t}\right)^l\psi_{b+l}\left(\frac{xy}{t^2}\right)
\end{equation}
where $0\leq l\leq j.$

Letting $w=y/|t|$ and $\lambda=x/|t|,$ we see that we must estimate the integrals
\begin{equation}
I_{l,j}(x,t,\eta)=  \frac{1}{|t|^j}\int\limits_{|t|^{-1}J_{x,\eta}}
w^{b+l}e^{-\cos\theta(w+\lambda)}
|\psi_{b+l}(w\lambda e_{2\theta})|\frac{dw}{w}.
\end{equation}
There are two cases: if $x<\eta^2,$ then 
\begin{equation}
  |t|^{-1}J_{x,\eta}=\left[\frac{(\sqrt{x}+\eta)^2}{|t|},\infty\right).
\end{equation}
otherwise:
\begin{equation}
  |t|^{-1}J_{x,\eta}=\left[0,\frac{(\sqrt{x}-\eta)^2}{|t|}\right]
\bigcup\left[\frac{(\sqrt{x}+\eta)^2}{|t|},\infty\right).
\end{equation}
Without loss of generality, we can assume that $|t|<\eta^2.$

We first consider the case where $x<\eta^2.$ Here again there are two cases to
examine: if $|t|^2<x(\sqrt{x}+\eta)^2$ (``small $|t|$ case''), then we only need to
use the asymptotic expansion for $\psi_{b+l},$ otherwise (``large $|t|$ case'')
we also need to separately estimate the integral over
$\left[\frac{(\sqrt{x}+\eta)^2}{|t|},\lambda^{-1}\right).$ We begin with the
small $|t|$ case. The product $w\lambda>1$ and we can use the asymptotic
expansion
\begin{equation}
  \psi_{b+l}(z)\sim\frac{z^{\frac{1}{4}-\frac{b+l}{2}}}{\sqrt{4\pi}}e^{2\sqrt{z}}.
\end{equation}
There is a constant $C_{b,l}$ so that
\begin{equation}
\begin{split}
  I_{l,j}(x,t,\eta)\leq &
C_{b,l}\frac{1}{|t|^j}\int\limits_{\frac{(\sqrt{x}+\eta)^2}{|t|}}^{\infty}
\left(\frac{w}{\lambda}\right)^{\frac{b+l}{2}-\frac 14}e^{-\cos\theta(\sqrt{w}-\sqrt{\lambda})^2}
\frac{dw}{\sqrt{w}}\\
\leq&
C_{b,l}\frac{1}{|t|^j}\int\limits_{\frac{\eta}{\sqrt{|t|}}}^{\infty}
\left(1+\frac{z}{\sqrt{\lambda}}\right)^{b+l-\frac{1}{2}}
e^{-\cos\theta z^2}bz
\end{split}
\end{equation}
where we set $z=\sqrt{w}-\sqrt{\lambda}$ in the second line.
Since $\eta/\sqrt{|t|\lambda}\leq 2\eta^2/|t|,$ an elementary integration by parts
argument shows that
\begin{equation}
   I_{l,j}(x,t,\eta)\leq
C_{b,l,\theta}\left(\frac{\eta^2}{|t|}\right)^{b+l}\frac{e^{-\cos\theta\frac{\eta^2}{|t|}}}{|t|^j}
\end{equation}

For the large $|t|$ case we need to consider 
\begin{equation}
  I_{l,j}'(x,t,\eta) = \frac{1}{|t|^j}
\int\limits_{\frac{(\sqrt{x}+\eta)^2}{|t|}}^{\frac{1}{\lambda}}
w^{b+l}e^{-\cos\theta(w+\lambda)}
\psi_{b+l}(w\lambda e_{2\theta})\frac{dw}{w}
\end{equation}
In this case we approximate $\psi_{b+l}(z)$ by a constant to
obtain
\begin{equation}
   I_{l,j}'(x,t,\eta) \leq \frac{C_{b,l}e^{-\cos\theta\lambda}}{|t|^j}
\int\limits_{\frac{(\sqrt{x}+\eta)^2}{|t|}}^{\frac{1}{\lambda}}
w^{b+l}e^{- \cos\theta w}\frac{dw}{w}.
\end{equation}
If $b+l\geq 1,$ then this is estimated by 
\begin{equation}
  \frac{C_{b,l,\theta}}{|t|^j}\left(\frac{\eta^2}{|t|}\right)^{b+l-1}
e^{-\cos\theta\frac{\eta^2}{|t|}}
\leq  \frac{C_{b,l,\theta}}{|t|^j}e^{-\cos\theta\frac{\eta^2}{2|t|}}.
\end{equation}
If $0<b+l<1,$ then because $\eta^2/|t|>1,$ we have that

\begin{equation}
   I_{l,j}'(x,t,\eta) \leq \frac{C_{b,l,\theta}e^{-\cos\theta\frac{\eta^2}{|t|}}}{|t|^j}.
\end{equation}

The other part of $I_{l,j}(x,t,\eta)$ is bounded by
\begin{equation}
   I_{l,j}''(x,t,\eta) \leq
C_{b,l}\frac{1}{|t|^j}\int\limits_{\frac{1}{\lambda}}^{\infty}
\left(\frac{w}{\lambda}\right)^{\frac{b+l}{2}-\frac 14}
e^{-\cos\theta(\sqrt{w}-\sqrt{\lambda})^2}\frac{dw}{\sqrt{w}}.
\end{equation}
Estimating the integral shows that
\begin{equation}
   I_{l,j}''(x,t,\eta) \leq C_{b,l,\theta}\frac{\lambda^{-(b+l)}e^{-\frac{\cos\theta}{\lambda}}}{|t|^j}
\end{equation}
Since $1/\lambda>\frac{(\sqrt{x}+\eta)^2}{|t|},$ this is again easily seen to
satisfy
\begin{equation}
   I_{l,j}''(x,t,\eta) \leq
C_{b,l,\theta}\frac{e^{-\cos\theta\frac{\eta^2}{2|t|}}}{|t|^j}.
\end{equation}
This establishes the estimates
\begin{equation}
  I_{l,j}(x,t,\eta) \leq
C_{b,l,j,\theta}\frac{e^{-\cos\theta\frac{\eta^2}{2|t|}}}{|t|^j},\text{ when }x\leq \eta^2.
\end{equation}
The constants $C_{b,l,j,\theta}$ are uniformly bounded for $0<b<B,$ and
$|\theta|<\frac{\pi}{2}-\phi.$ 

We now consider $x\geq \eta^2;$ as before we assume that
$|t|<\eta^2,$ so that $1/\lambda<(\sqrt{x}+\eta)^2/|t|.$ We first estimate the
non-compact part of the integral:
\begin{equation}
\begin{split}
  I_{l,k,j}''(x,t,\eta)\leq &
C_{b,l}\frac{1}{|t|^j}\int\limits_{\frac{\eta}{\sqrt{|t|}}}^{\infty}
\left(1+\frac{z}{\sqrt{\lambda}}\right)^{b+l-\frac 12}
e^{-\cos\theta z^2}bz\\
\leq&
C_{b,l,\theta}\frac{e^{-\cos\theta\frac{\eta^2}{|t|}}}{|t|^j}
\end{split}
\end{equation}

This leaves only
\begin{equation}
   I_{l,j}'(x,t,\eta)=
\frac{1}{|t|^j}\int\limits_{0}^{\frac{(\sqrt{x}-\eta)^2}{|t|}}
w^{b+l}e^{-\cos\theta(w+\lambda)}
|\psi_{b+l}(w\lambda e_{2\theta})|\frac{dw}{w}.
\end{equation}
If $|t|^2<x(\sqrt{x}-\eta)^2,$ then we need to split this integral into two
parts: from $0$ to $1/\lambda$ and the rest. We first assume that there is just
one part. If $b+l\geq 1,$ then we have the estimate
\begin{equation}
   I_{l,j}'(x,t,\eta)\leq
\frac{C_{b,l}}{|t|^j}\int\limits_{0}^{\frac{(\sqrt{x}-\eta)^2}{|t|}}
w^{b+l}e^{-\cos\theta(w+\lambda)}
\frac{dw}{w}
\leq 
\frac{C_{b,l,\theta}e^{-\cos\theta\lambda}}{|t|^j}.
\end{equation}
In this case the fact that $x\geq\eta^2,$ shows that there is a constant
$C_{b,l,\theta}$ so that
\begin{equation}
   I_{l,j}'(x,t,\eta)\leq C_{b,l,\theta}\frac{e^{-\cos\theta\frac{\eta^2}{|t|}}}{|t|^j}.
\end{equation}

If $l=0$ and $b<1,$ then we need to use the approximation
\begin{equation}\label{eqn12.140}
  \psi_{b}(z)=\frac{1}{\Gamma(b)}+O(z)
\end{equation}
to see that
\begin{equation}
  I_{0,j}'(x,t,\eta)\leq
\frac{C_{b,0,\theta}e^{-\cos\theta\lambda}}{|t|^j}\int\limits_{0}^{\frac{(\sqrt{x}-\eta)^2}{|t|}}
w^{b-1}e^{-\cos\theta w}\left[\frac{1}{\Gamma(b)}+O(w\lambda)\right]dw,
\end{equation}
which again implies that 
\begin{equation}
   I_{l,j}'(x,t,\eta)\leq C_{b,0,\theta}\frac{\lambda e^{-\cos\theta\lambda}}{|t|^j}
\leq  C'_{b,0,\theta}\frac{e^{-\cos\theta\frac{\eta^2}{2|t|}}}{|t|^j}
\end{equation}
Here $C'_{b,0,\theta}$ is bounded as $b\to 0.$

The only case that remains is when $|t|^2<x(\sqrt{x}-\eta)^2,$ wherein
\begin{multline}
     I_{l,j}'(x,t,\eta)\leq
\frac{C_{b,l}}{|t|^j}\Bigg[\int\limits_{0}^{\frac{1}{\lambda}}
w^{b+l}e^{-\cos\theta(w+\lambda)}|\psi_{b+l}(w\lambda e_{2\theta})|
\frac{dw}{w}+\\
\int\limits_{\frac{1}{\lambda}}^{\frac{(\sqrt{x}-\eta)^2}{|t|}}
\left(\frac{w}{\lambda}\right)^{\frac{b+l}{2}-\frac 14}
e^{-\cos\theta(\sqrt{w}-\sqrt{\lambda})^2}\frac{dw}{\sqrt{w}}\Bigg].
\end{multline}
If $b+l\geq 1,$ then we can estimate $\psi_{b+l}$ by a constant to see that the
first term is bounded by 
\begin{equation}
\frac{C_{b,l,\theta}e^{-\cos\theta\lambda}}{|t|^j}\leq 
\frac{C_{b,l,\theta}e^{-\cos\theta\frac{\eta^2}{|t|}}}{|t|^j}
\end{equation}
If $l=0$ and $b<1,$ then, as before, we need to use~\eqref{eqn12.140} to see
that this term is bounded by
\begin{equation}
  \frac{C_{b,0,\theta}\lambda e^{-\cos\theta\lambda}}{|t|^j}
\leq \frac{C_{b,0,\theta}'e^{-\cos\theta\frac{\eta^2}{2|t|}}}{|t|^j},
\end{equation}
where again $C_{b,0,\theta}'$ is bounded for $b<B.$ This leaves only
\begin{equation}
\frac{C_{b,l}}{|t|^j}  \int\limits_{\frac{1}{\lambda}}^{\frac{(\sqrt{x}-\eta)^2}{|t|}}
\left(\frac{w}{\lambda}\right)^{\frac{b+l}{2}-\frac 14}
e^{-\cos\theta(\sqrt{w}-\sqrt{\lambda})^2}\frac{dw}{\sqrt{w}}=
\frac{C_{b,l}}{|t|^j}  \int\limits_{\frac{\eta}{\sqrt{|t|}}}^{\sqrt{\lambda}-\frac{1}{\sqrt{\lambda}}}
\left(1-\frac{z}{\sqrt{\lambda}}\right)^{b+l-\frac 12}
e^{-\cos\theta z^2}dz.
\end{equation}
If $b+l-\frac 12>0,$ then this bounded by
\begin{equation}
  \frac{C_{b,l}}{|t|^j}  \int\limits_{\frac{\eta}{\sqrt{|t|}}}^{\sqrt{\lambda}-\frac{1}{\sqrt{\lambda}}}
e^{-\cos\theta z^2}dz\leq  \frac{C_{b,l,\theta}e^{-\cos\theta\frac{\eta^2}{|t|}}}{|t|^j}.
\end{equation}

This leaves only the case $l=0,$ $b<\frac 12.$ To obtain a good estimate in
this case, as $\lambda\to\infty,$ we split the integral into two parts:
\begin{equation}
 \int\limits_{\frac{\eta}{\sqrt{|t|}}}^{\sqrt{\frac{2\lambda}{3}}}
\left(1-\frac{z}{\sqrt{\lambda}}\right)^{b-\frac 12}
e^{-\cos\theta z^2}dz+
\int\limits_{\sqrt{\frac{2\lambda}{3}}}^{\sqrt{\lambda}-\frac{1}{\sqrt{\lambda}}}
\left(1-\frac{z}{\sqrt{\lambda}}\right)^{b-\frac 12}
e^{- \cos\theta z^2}dz 
\end{equation}
The first term is bounded by $C_{b,0,\theta}e^{-\cos\theta\frac{\eta^2}{t}};$ and the second by 
$$C_{b,0,\theta}\sqrt{\lambda}e^{-\cos\theta\frac{2\lambda}{3}}\leq 
C_{b,0,\theta}e^{-\cos\theta\frac{\eta^2}{2|t|}}.$$
Altogether we have shown that
\begin{equation}
  I_{l,j}'(x,t,\eta)\leq \frac{C_{b,l,\theta}e^{-\cos\theta\frac{\eta^2}{2|t|}}}{|t|^{j}}
\end{equation}
\end{proof}

The proof of Lemma~\ref{lem12.2.2} is similar, but easier. It follows from
the formula, valid for $t\in S_0:$
\begin{equation}
  \pa_x^jk^e_t(x,y)=\frac{1}{t^{\frac
      j2}}\sum_{l=0}^jc_{j,l}\left(\frac{(x-y)}{2\sqrt{t}}\right)^l
k^e_t(x,y).
\end{equation}
The details of the proof are left to the reader.

{\bibliographystyle{siam} {\bibliography{alla-k}}}

\addcontentsline{toc}{chapter}{Index}
\printindex

\end{document}